\documentclass{article}
\pdfoutput=1

\usepackage[preprint]{neurips_2023}

\usepackage[utf8]{inputenc} 
\usepackage[T1]{fontenc}    
\usepackage[pagebackref=true]{hyperref}
\renewcommand*\backref[1]{\ifx#1\relax \else Cited on #1 \fi}
\usepackage{url}            
\usepackage{booktabs}       
\usepackage{amsmath,amsfonts,amssymb,amsthm}
\usepackage{nicefrac}       
\usepackage{microtype}      
\usepackage{graphicx}       
\usepackage{xcolor}
\usepackage[skip=0pt]{subfig}
\usepackage{alphalph}

\usepackage[ruled,vlined,linesnumbered]{algorithm2e}
\usepackage{diagbox}

\usepackage{enumitem}
\usepackage[skip=0pt]{caption}
\setlist[itemize]{topsep=0pt,itemsep=2pt,parsep=0pt,partopsep=0pt,leftmargin=3mm }
\setlist[enumerate]{topsep=0pt,itemsep=0pt,parsep=0pt,partopsep=0pt,leftmargin=0.8\labelwidth}
\graphicspath{{supplement/}} 

\title{Don't be so Monotone: Relaxing Stochastic\\Line Search in Over-Parameterized Models
}

\makeatother
\DeclareMathOperator*{\argmin}{argmin}

\DeclareMathOperator*{\E}{\mathbb{E}}

\DeclareMathOperator*{\Eik}{\mathbb{E}_{\mathit{i_k}}}

\DeclareMathOperator*{\grad}{\mathit{\nabla \!f}}

\DeclareMathOperator*{\gradik}{\mathit{\nabla\!\fik}}
\DeclareMathOperator*{\Lmax}{\mathit{L_{max}}}
\newcommand{\R}{\mathbb{R}}

\newcommand{\fik}{f_{i_k}}
\newcommand{\fikofwstar}{\fik(w^*)}
\newcommand{\fofwstar}{f(w^*)}

\newcommand{\Lik}{L_{i_k}}

\newcommand{\etamax}{\eta^{\text{max}}}

\newcommand{\etamin}{\eta^{\text{min}}}
\newcommand{\etaminn}{\bar{\eta}^{\text{min}}}
\newcommand{\minimum}[2]{\min \left\{ #1, #2 \right \} }
\newcommand{\maximum}[2]{\max \left\{ #1, #2 \right \} }

\makeatletter

\newtheorem{lemma}{Lemma}

\newtheorem{theorem}{Theorem}

\newcommand{\imgS}{.26}
\newcommand{\dir}{exp1/}
\newcommand{\model}{mlp}
\newcommand{\modelname}{mlp}

\newcommand{\mlp}{{\texttt{mnist|mlp}}}
\newcommand{\res}{{\texttt{cifar10|resnet34}}}
\newcommand{\dense}{{\texttt{cifar10|densenet121}}}
\newcommand{\ress}{{\texttt{cifar100|resnet34}}}
\newcommand{\denses}{{\texttt{cifar100|densenet121}}}
\newcommand{\fashion}{{\texttt{fashion|effb1}}}
\newcommand{\svhn}{{\texttt{svhn|wrn}}}
\newcommand{\wiki}{{\texttt{wiki2|encoder}}}

\newcommand{\mushrooms}{{\texttt{mushrooms}}}
\newcommand{\rcvone}{{\texttt{rcv1}}}
\newcommand{\ijcnn}{{\texttt{ijcnn}}}

\author{
	Leonardo Galli, Holger Rauhut \\
	RWTH Aachen University \\
	Aachen\\
	\texttt{\{galli, rauhut\}@mathc.rwth-aachen.de} \\
	\And
	Mark Schmidt \\
	University of British Columbia\\ Canada CIFAR AI Chair (Amii)\\
	\texttt{schmidtm@cs.ubc.ca} \\
}

\begin{document}

	\maketitle

	\begin{abstract}
		
Recent works have shown that line search methods can speed up Stochastic Gradient Descent (SGD) and Adam in modern over-parameterized settings. However, existing line searches may take steps that are smaller than necessary since they require a monotone decrease of the (mini-)batch objective function. We explore nonmonotone line search methods to relax this condition and possibly accept larger step sizes. Despite the lack of a monotonic decrease, we prove the same fast rates of convergence as in the monotone case. Our experiments show that nonmonotone methods improve the speed of convergence and generalization properties of SGD/Adam even beyond the previous monotone line searches. We propose a POlyak NOnmonotone Stochastic (PoNoS) method, obtained by combining a nonmonotone line search with a Polyak initial step size. Furthermore, we develop a new resetting technique that in the majority of the iterations reduces the amount of backtracks to zero while still maintaining a large initial step size. To the best of our knowledge, a first runtime comparison shows that the epoch-wise advantage of line-search-based methods gets reflected in the overall computational time.
	
	\end{abstract}
	
	\section{Introduction}
	
Stochastic Gradient Descent (SGD) \citep{robbins51a} is the workhorse for the whole Deep Learning (DL) activity today. Even though its simplicity and low memory requirements seem crucial for dealing with these huge models, the success of SGD is strongly connected to the choice of the learning rate. In the field of deterministic optimization \citep{nocedal06a}, this problem is addressed \citep{armijo66a} by employing a line search technique to select the step size. Following this approach, a recent branch of research has started to re-introduce the use of line search techniques for training DL models \citep{mahsereci17a,paquette18a,bollapragada18a,truong18a,vaswani19a}. These methods have been shown to speed up SGD \citep{vaswani19a} and Adam \citep{duchi11a,vaswani20a} in the over-parameterized regime. However, the existing stochastic line searches may take step sizes that are smaller than necessary since they require a monotone decrease in the (mini-)batch objective function. In particular, the highly non-linear non-convex landscapes of DL training losses suggest that it may be too restrictive to impose such a condition \citep{grippo86a}. In addition, the stochastic nature of SGD discourages imposing a monotone requirement on a function that is not directly the one we are minimizing \citep{ferris94a} (mini- vs full-batch function). Furthermore, the recent research on the edge of stability phenomenon \citep{cohen20a,nacson22a} shows that the best performance for training DL models with Gradient Descent (GD) is obtained when large step sizes yield a nonmonotone decrease of the loss function. In this paper, we propose the use of nonmonotone line search methods \citep{grippo86a,zhang04a} to possibly accept increases in the objective function and thus larger step sizes.  
\par In parallel with line search techniques, many other methods \citep{baydin17a,luo19a,mutschler20a,truong21a} have been recently developed to automatically select the step size. In \citet{asi19a,berrada20a,loizou21a,gower21a,gower22a,li22a}, the Polyak step size from \citet{polyak69a} has been adapted to SGD and successfully employed to train DL models. However, in this setting, the Polyak step size may become very large and it may lead to divergence or, in more favorable cases, induce the (mini-batch) function to decrease nonmonotonically. Starting from \citet{berrada20a}, the Polyak step size has been upper-bounded to avoid divergence and ``smoothed'' to prevent large fluctuations \citep{loizou21a}. However, this ``smoothing'' technique is a heuristic and, as we will clarify below, may reduce the Polyak step more than necessary. In this paper, we instead combine the Polyak initial step from \citet{loizou21a} with a nonmonotone line search. In fact, nonmonotone methods are well suited for accepting as often as possible a promising (nonmontone) step size, while still controlling its growth. In the deterministic setting, the same optimization recipe has been used in the seminal papers \citet{raydan97a} (for the \citet{barzilai88a} (BB) step size) and \citet{grippo86a} (for the unitary step in the case of Newton).
\par In the over-parameterized regime, the number of free parameters dominates the number of training samples. This provides modern DL models with the ability of exactly fitting the training data. In practice, over-parameterized models are able to reduce to zero the full-batch loss and consequently also all the individual losses. This is mathematically captured by the interpolation assumption, which allows SGD with non-diminishing step sizes to achieve GD-like convergence rates \citep{vaswani19a}. We develop the first rates of convergence for stochastic nonmonotone line search methods under interpolation. In the non-convex case, we prove a linear rate of convergence under the Polyak-Lojasiewicz (PL) condition. In fact in \citet{liu22a}, it has been shown that wide neural networks satisfy a local version of the PL condition, making this assumption more realistic for DL models than the Strong Growth Condition (SGC) \citep{schmidt13a}. 
\par When considering line search methods for training DL models one has to take into account that they require one additional forward step for each backtrack (internal iteration) and each mini-batch iteration requires an unknown amount (usually within 5) of backtracks. A single forward step may require up to one-third the runtime of SGD (see Section E.6 of the Appendix), thus, too many of them may become a computational burden. \citet{leroux12a} proposed a ``resetting'' technique that is able to (almost) always reduce the number of backtracking steps to zero by shrinking the initial step size \citep{vaswani19a}. In this paper, we will show that this technique may have a negative impact on the speed of convergence, even when only employed as a safeguard (as in the ``smoothing'' technique \citep{loizou21a}). As a third contribution, we develop a new resetting technique for the Polyak step size that reduces (almost) always the number of backtracks to zero, while still maintaining a large initial step size. To conclude, we will compare the runtime of different algorithms to show that line search methods equipped with our new technique outperform non-line-search-based algorithms.

	\section{Related Works}
	The first nonmonotone technique was proposed by \citet{grippo86a} to globalize the Newton method without enforcing a monotonic decrease of the objective function. After that, nonmonotone techniques have been employed to speed up various optimization methods by relaxing this monotone requirement. A few notable examples are the spectral gradient in \citet{raydan97a}, the spectral projected gradient in \citet{birgin00a}, sequential quadratic programming methods in \citet{zhou93a}, and the Polak-Ribière-Polyak conjugate gradient in \citet{zhou13a}.

In the pre-deep-learning era \citep{plagianakos02a}, shallow neural networks have been trained by combining the nonmonotone method \citep{grippo86a} with GD. In the context of speech recognition \citep{keskar15a}, one fully-connected neural network has been trained with a combination of a single-pass full-batch nonmonotone line search \citep{grippo86a} and SGD. The Fast Inertial Relaxation Engine (FIRE) algorithm is combined with the nonmonotone method \citep{zhang04a} in \citet{wang19a,wang21a} for solving sparse optimization methods and training logistic regression models, while a manually adapted learning rate is employed to train DL models. In \citet{krejic15a,krejic19,bellavia21a}, the nonmonotone line search by \citet{li99a} is used together with SGD to solve classical optimization problems. In concurrent work \citep{hafshejani23a}, the nonmonotone line search by \citet{grippo86a} has been adapted to SGD for training small-scale kernel models. This line search maintains a nonmonotone window $W$ (usually 10) of weights in memory and it computes the current mini-batch function value on all of them at every iteration. The method proposed in \citet{hafshejani23a} is computationally very expensive and impractical to train modern deep learning models due to the need to store previous weights. In this paper, we instead propose the first stochastic adaptation of the nonmonotone line search method by \citet{zhang04a}. Thanks to this line search, no computational overhead is introduced in computing the nonmonotone term which is in fact a linear combination of the previously computed mini-batch function values. To the best of our knowledge, we propose the first stochastic nonmonotone line search method to train modern convolutional neural networks and transformers \citep{vaswani17a}.

Our paper is the first to combine a line search technique with the Stochastic Polyak Step size (SPS) \citep{loizou21a}. In fact, the existing line search methods \citep{vaswani19a,paquette18a,mahsereci17a} do not address the problem of the selection of a suitable initial step size. In \citet{vaswani19a}, the authors focus on reducing the amount of backtracks and propose a ``resetting'' technique \eqref{eq:etamax_paper} that ends up also selecting the initial step size. In this paper, we consider these two problems separately and tackle them with two different solutions: a Polyak initial step size and a new resetting technique \eqref{eq:new_etak_paper} to reduce the amount of backtracks. Regarding the latter, a similar scheme was presented in \citet{grapiglia21a}, however we modify it to be combined with an independent initial step size (e.g., Polyak). This modification allows the original initial step size not to be altered by the resetting technique.

In this paper, we extend \citet{vaswani19a} by modifying the monotone line search proposed there with a nonmonotone line search. Our theoretical results extend the theorems presented in \citet{vaswani19a} to the case of stochastic nonmonotone line search methods. Previous results in a similar context were given by \citet{krejic15a,krejic19,bellavia21a} that assumed 1) the difference between the nonmonotone and monotone terms to be geometrically converging to 0 and 2) the batch-size to be geometrically increasing. In this paper, we replace both hypotheses with a weaker assumption (i.e., interpolation) \citep{meng20a} and we actually prove that the nonmonotone and monotone terms converge to the same value. In \citet{hafshejani23a}, related convergence theorems are proved under the interpolation assumption for the use of a different nonmonotone line search (i.e., \citet{grippo86a} instead of \citet{zhang04a}). However in \citet{hafshejani23a}, no explanation is given on how to use an asymptotic result (Lemma 4) for achieving their non-asymptotic convergence rates (Theorem 1 and 2). Addressing this problem represents one of the main challenges of adapting the existing nonmonotone theory \citep{grippo86a,dai02a,zhang04a} to the stochastic case.

In \citet{vaswani19a}, the rates provided in the non-convex case are developed under an SGC, while we here assume the more realistic PL condition \citep{liu22a}. In fact in \citet{liu22a}, not only wide networks are proven to satisfy PL but also networks with skip connections (e.g., ResNet \citep{he16a}, DenseNet \citep{huang17a}). For this result, we extend Theorem 3.6 of \citet{loizou21a} for stochastic nonmonotone line search methods by exploiting our proof technique. In \citet{loizou21a}, various rates are developed for SGD with a Polyak step size and it is possible to prove that they hold also for PoNoS. In fact, the step size yielded by PoNoS is by construction less or equal than Polyak's. However, the rates presented in this paper are more general since they are developed for nonmonotone line search methods and do not assume the use of any specific initial step size (e.g., they hold in combination with a BB step size).

	\section{Methods}\label{sec:method}
	Training machine learning models (e.g., neural networks) entails solving the finite sum problem $\min_{w\in \mathbb{R}^n} f(w) = \frac{1}{M}\sum_{i=1}^{M} f_i(w),$ where $w$ is the parameter vector and $f_i$ corresponds to a single instance of the $M$ points in the training set. We assume that $f$ is lower-bounded by some value $f^*$, that $f$ is $L$-smooth and that it either satisfies the PL condition, convexity or strong-convexity. Vanilla SGD can be described by the step $ w_{k+1} = w_k - \eta \gradik (w_k),$ where $i_k \in \{1,\dots, M\}$ is one instance randomly sampled at iteration $k$, $\gradik(w_k)$ is the gradient computed only w.r.t.\ the $i_k$-th instance and $\eta>0$ is the step size. The mini-batch version of this method modifies $i_k$ to be a randomly selected subset of instances, i.e., $i_k\subset \{1,\dots, M\}$, with $\gradik(w_k)$ being the averaged gradient on this subset and $|i_k| = b$ the mini-batch size. Through the whole paper we assume that each stochastic function $\fik$ and gradient $\gradik(w)$ evaluations are unbiased, i.e., $\Eik [\fik(w)] = f(w)$ and $\Eik [\gradik(w)] = \grad(w), \;\; \forall w \in \R^n$. Note that $\Eik$ represents the conditional expectation w.r.t. $w_k$, i.e., $\Eik [\cdot] = \E [\cdot| w_k]$. In other words, $\Eik$ is the expectation computed w.r.t. $\nu_k$, that is the random variable associated with the selection of the sample (or subset) at iteration $k$ (see \citet{bottou18a} for more details).
\par We interpret the over-parameterized setting in which we operate to imply the interpolation property, i.e., let $w^* \in \displaystyle\argmin_{w\in \R^n} f(w),$ then $w^* \in \displaystyle\argmin_{w\in \R^n} f_i(w)\;\forall 1\leq i\leq M$. This property is crucial for the convergence results of SGD-based methods because it can be combined either with the Lipschitz smoothness of $\fik$ or with a line search condition to achieve the bound $\Eik\| \gradik (w_k)\|^2 \leq a \left(f(w_k)-f(w^*)\right)$ with $a>0$ (see Lemma 4 of the Appendix for the proof). This bound on the variance of the gradient results in a r.h.s. which is independent of $i_k$ and, in the convex and strongly convex cases, it can be used to replace the SGC (see \citet{fan23a} for more details).

As previously stated, our algorithm does not employ a constant learning rate $\eta$, but the step size $\eta_k$ is instead determined at each iteration $k$ by a line search method. Given an initial step size $\eta_{k,0}$ and $\delta\in(0,1)$, the Stochastic Line Search (SLS) condition \citep{vaswani19a} select the smallest $l_k\in\mathbb{N}$ such that $\eta_k = \eta_{k,0} \delta^{l_k}$ satisfies the following condition 
\begin{equation}\label{eq:armijo_paper}
	\fik(w_k - \eta_k \gradik(w_k)) \leq \fik(w_k) -c \eta_k \| \gradik(w_k) \|^2,
\end{equation}
where $c\in(0,1)$ and $\|\cdot\|$ is the Euclidean norm. Each internal step of the line search is called backtrack since the procedure starts with the largest value $\eta_{k,0}$ and reduces (cuts) it until the condition is fulfilled. Note that \eqref{eq:armijo_paper} requires a monotone decrease of $\fik$. In this paper, we follow \citet{zhang04a} and propose to replace $\fik$ with a nonmonotone term $C_k$,
\begin{equation}\label{eq:zhang_paper}
	\begin{split}
		&\fik(w_k - \eta_k \gradik(w_k)) \leq C_k -c \eta_k \| \gradik(w_k) \|^2,\\
		C_k=\max &\left\{\tilde{C}_k; \fik(w_k) \right\}, \; \tilde{C}_k = \frac{\xi Q_k C_{k-1} + \fik(w_k)}{Q_{k+1}}, \; Q_{k+1} = \xi Q_k + 1,
	\end{split}
\end{equation}
where $\xi\in[0,1]$, $C_0=Q_0=0$ and $C_{-1} = f_{i_0}(w_0)$. The value $\tilde{C}_k$ is a linear combination of the previously computed function values $f_{i_0}(w_0), \dots, \fik(w_k)$ and it ranges between the strongly nonmonotone term $ \frac{1}{k}\sum_{j=0}^{k} f_{i_j}(w_j)$ (with $\xi=1$) and the monotone value $\fik(w_k)$ (with $\xi=0$). The maximum in $C_k$ is needed to ensure the use of a value that is always greater or equal than the monotone term $\fik(w_k)$. The activation of $\fik(w_k)$ instead of $\tilde{C}_k$ happens only in the initial iterations and it is rare in deterministic problems, but becomes more common in the stochastic case. The computation of $C_k$ does not introduce overhead since the linear combination is accumulated in $C_k$. 
\par Given $\gamma>1$, the initial step size employed in \citet{vaswani19a} is the following
\begin{equation}\label{eq:etamax_paper}
	\eta_{k,0} = \min \{\eta_{k-1}  \gamma^{b/M}, \etamax \}.
\end{equation}
In this paper, we propose to use the following modified version of SPS \citep{loizou21a}
\begin{equation}\label{eq:loizou_paper}
\eta_{k,0} = \min \left\{ \tilde{\eta}_{k,0}, \etamax \right\} \quad \text{ with }  \tilde{\eta}_{k,0}:=\frac{\fik(w_k) - \fik^*}{c_p||\gradik(w_k)||^2} \quad \text{ and } c_p\in(0,1).
\end{equation}
In \citet{loizou21a}, \eqref{eq:loizou_paper} is not directly employed, but equipped with the same resetting technique \eqref{eq:etamax_paper} to control the growth of $\tilde{\eta}_{k,0}$, i.e. $\eta_{k,0} = \min \left\{ \tilde{\eta}_{k,0}, \eta_{k-1} \gamma^{b/M}, \etamax \right\}.$ 
\par Given $l_{k-1}$ the amount of backtracks at iteration $k-1$, our newly introduced resetting technique stores this value for using it in the following iteration. In particular, to relieve the line search at iteration $k$ from the burden of cutting the new step size $l_{k-1}$ times, the new initial step size is pre-scaled by $\delta^{l_{k-1}}$. In fact, $l_{k-1}$ is a good estimate for $l_k$ (see Section E.3 of the Appendix). However, to allow the original $\eta_{k,0}$ to be eventually used non-scaled, we reduce $l_{k-1}$ by one. We thus redefine $\eta_k$ as
\begin{equation}\label{eq:new_etak_paper}
	\eta_k =   \eta_{k,0} \delta^{\bar{l}_k} \delta^{l_k},\quad \text { with } \; \bar{l}_{k}:= \max \{\bar{l}_{k-1} + l_{k-1} -1, 0\}.
\end{equation}
In Section \ref{sec:experiments}, our experiments show that thanks to \eqref{eq:new_etak_paper} we can save many backtracks and reduce them to zero in the majority of the iterations. At the same time, the original $\eta_{k,0}$ is not modified and the resulting step size $\eta_k$ is always a scaled version of $\eta_{k,0}$. Moreover, the presence of the $-1$ in \eqref{eq:new_etak_paper} keeps pushing $\delta^{\bar{l}_k}$ towards zero, so that the initial step size is not cut more than necessary. Note that \eqref{eq:new_etak_paper} can be combined with any initial step size and it is not limited to the case of \eqref{eq:loizou_paper}.

	\section{Rates of Convergence}\label{sec:theory}
	We present the first rates of convergence for nonmonotone stochastic line search methods under interpolation (see Section B of the Appendix for the proofs). The three main theorems are given under strong-convexity, convexity and the PL condition. Our results do not prove convergence only for PoNoS, but more generally for methods employing \eqref{eq:zhang_paper} and any bounded initial step size, i.e.
\begin{equation}\label{eq:generic_eta_short}
	\eta_{k,0} \in [\etaminn, \etamax] , \quad \text{ with } \etamax>\bar{\eta}^{\text{min}}> 0.
\end{equation}
Many step size rules in the literature fulfill the above requirement \citep{berrada20a,loizou21a,liang19a} since when applied to non-convex DL models the step size needs to be upper bounded to achieve convergence \citep{bottou18a,berrada20a}.

In Theorem \ref{thm:strongly_convex_short} below, we show a linear rate of convergence in the case of a strongly convex function $f$, with each $\fik$ being convex. We are able to recover the same speed of convergence of Theorem 1 in \citet{vaswani19a}, despite the presence of the nonmonotone term $C_k$ instead of $\fik(w_k)$. The challenges of the new theorem originate from studying the speed of convergence of $C_k$ to $f(w^*)$ and from combining the two interconnected sequences $\{w_k\}$ and $\{C_k\}$. It turns out that if $\xi$ is small enough (i.e., such that $b<(1-\etamin\mu)$), then the presence of the nonmonotone term does not alter the speed of convergence of $\{w_k\}$ to $w^*$, not even in terms of constants.

\begin{theorem}\label{thm:strongly_convex_short}
	Let $C_k$ and $\eta_k$ be defined in \eqref{eq:zhang_paper}, with $\eta_{k,0}$ defined in \eqref{eq:generic_eta_short}. We assume interpolation, $\fik$ convex, $f$ $\mu$-strongly convex and $\fik$ $\Lik$-Lipschitz smooth. Assuming $c> \frac{1}{2}$ and $\xi< \frac{1}{\left(1 + \frac{\etamax}{\etamin\left(2c - 1\right)}\right)}$, we have
	\begin{equation*}
		\begin{split}
			\E \left[\| w_{k+1} - w^* \|^2 + a ( C_k - \fofwstar) \right] & \leq d^{k} \left(\| w_0 - w^* \|^2 + a \left(f(w_0) - \fofwstar\right)\right),
		\end{split}
	\end{equation*}
	where $d:=\maximum{(1-\etamin\mu)}{b}\in(0,1)$, $b:=\left(1+\frac{\etamax}{ac}\right)\xi\in(0,1)$, $a:= \etamin\left(2-\frac{1}{c}\right)>0$ with $\etamin:=\min\{ \frac{2\delta(1-c)}{\Lmax},\etaminn \}$.
\end{theorem}

Comparing the above result with the corresponding deterministic rate (Theorem 3.1 from \citet{zhang04a}), we notice that the constants of the two rates are different. In particular, the proof technique employed in Theorem 3.1 from \citet{zhang04a} cannot be reused in Theorem \ref{thm:strongly_convex_short} since a few bounds are not valid in the stochastic case (e.g., $\|\grad (w_{k+1})\| \leq \| \grad (w_k) \|$ or $C_{k+1} \leq C_k - \| \gradik (w_k) \|$ ). The only common aspect in these proofs is the distinction between the two different cases defining $C_k$ (i.e., being $\tilde{C}_k$ or $\fik(w_k)$). Moreover, in both theorems $\xi$ needs to be small enough to allow the final constant to be less than 1.
\par In Theorem \ref{thm:convex_short} we show a $O(\frac{1}{k})$ rate of convergence in the case of a convex function $f$. Interestingly, in both Theorems \ref{thm:strongly_convex_short} and \ref{thm:convex_short} (and also in the corresponding monotone versions from \citet{vaswani19a}), $c$ is required to be larger than $\frac{1}{2}$. The same lower bound is also required for achieving Q-super-linear convergence for the Newton method in the deterministic case, but it is often considered too large in practice and the default value (also in the case of first-order methods) is $0.1$ or smaller \citep{nocedal06a}. In this paper, we numerically tried both values and found out that $c=0.1$ is too small since it may indeed lead to divergence (see Section E.1 of the Appendix).
\begin{theorem}\label{thm:convex_short}
	Let $C_k$ and $\eta_k$ be defined in \eqref{eq:zhang_paper}, with $\eta_{k,0}$ defined in \eqref{eq:generic_eta_short}. We assume interpolation, $f$ convex and $\fik$ $\Lik$-Lipschitz smooth. Given a constant $a_1$ such that $0 < a_1 < \left(2 - \frac{1}{c}\right)$ and assuming $c> \frac{1}{2}$ and $\xi< \frac{a_1}{2}$, we have
	\begin{equation*}
		\begin{split}
			\E \left[f(\bar{w}_k) - \fofwstar\right] \leq \frac{d_1}{k}\left(\frac{1}{\etamin}\| w_0 - w^* \|^2 + a_1\left(f(w_0) -f(w^*) \right)\right),\\
		\end{split}
	\end{equation*}
	where $\bar{w}_k\!=\!\frac{1}{k}\sum_{j=0}^{k} w_j$, $d_1\!:=\!\frac{c}{c(2 - a_1) -1}\!>\!0$, $b_1\!:=\!\left(1\!+\!\frac{1}{a_1c}\right) \xi \in(0,1], \etamin\!:=\!\min\{ \frac{2\delta(1-c)}{\Lmax},\etaminn \}$.
\end{theorem}
In Theorem \ref{thm:pl_short} below we prove a linear convergence rate in the case of $f$ satisfying a PL condition. We say that a function $f:\mathbb{R}^n\to \mathbb{R}$ satisfies the PL condition if there exists $\mu>0$ such that, $\forall w \in \mathbb{R}^n: \| \grad(w)\|^2 \geq 2\mu (f(w)-\fofwstar)$. The proof of this theorem extends Theorem 3.6 of \citet{loizou21a} to the use of a stochastic nonmonotone line search. Again here, the presence of a nonmonotone term does not modify the constants controlling the speed of convergence ($a_2$ below can be chosen arbitrarily small) as long as $\xi$ is small enough. The conditions on $c$ and on $\etamax$ are the same as those of Theorem 3.6 of \citet{loizou21a} (see the proof).

\begin{theorem}\label{thm:pl_short}
	Let $C_k$ and $\eta_k$ be defined in \eqref{eq:zhang_paper}, with $\eta_{k,0}$ defined in \eqref{eq:generic_eta_short}. We assume interpolation, the PL condition on $f$ and that $\fik$ are $\Lik$-Lipschitz smooth. Given $0<a_2:=\frac{4\mu c(1-c) -\Lmax}{4\delta c (1-c)}+\frac{1}{2\etamax}$ and assuming $\frac{2\delta(1-c)}{\Lmax}<\etaminn,\etamax<\frac{2\delta c(1-c)}{\Lmax -4\mu c(1-c)}$, $\frac{\Lmax}{4\mu}<c<1$ and $\xi<\frac{a_2c}{a_2c + \Lmax}$, we have
	\begin{equation*}
		\begin{split}
			\E \left[f(w_{k+1}) - \fofwstar + a_2\etamax (C_k - \fofwstar) \right] 
			&\leq d_2^k \left(1+a_2\etamax\right) \left(f(w_0) - \fofwstar\right)
		\end{split}
	\end{equation*}
	where $d_2\!:=\!\minimum{\nu}{b_2}\!\in\!(0,1),\,\nu\!:=\!\etamax (\frac{\Lmax -4\mu c(1-c)}{2\delta c(1-c)}+a_2)\!\in\!(0,1), \, b_2\!:=\!\left(1\!+\!\frac{\Lmax}{a_2c}\right) \xi\!\in\!(0,1)$.
\end{theorem}
Theorem \ref{thm:pl_short} is particularly meaningful for modern neural networks because of the recent work of \citet{liu22a}. More precisely, their Theorem 8 and Theorem 13 prove that, respectively given a wide-enough fully-connected or convolutional/skip-connected neural network, the corresponding function satisfies a local version of the PL condition. As a consequence, nonmonotone line search methods achieve linear convergence for training of non-convex neural networks if the assumptions of both Theorem 3 above and Theorem 8 from \citet{liu22a} hold and if the initial point $w_0$ is close enough to $w^*$. 
\setcounter{lemma}{0}
\setcounter{theorem}{0}

	\section{Experiments}\label{sec:experiments}
	In this section, we benchmark PoNoS against state-of-the-art methods for training different over-parametrized models. We focus on various deep learning architectures for multi-class image classification problems, while in Section \ref{sec:convex_and_trans} we consider transformers for language modelling and kernel models for binary classification. In the absence of a better estimate, we assume $\fik(w^*) =0$ for all the experiments. Indeed the final loss is in many problems very close to $0$.
\par Concerning datasets and architectures for the image classification task, we follow the setup by \citet{luo19a}, later employed also in \citet{vaswani19a,loizou21a}. In particular, we focus on the datasets MNIST, FashionMNIST, CIFAR10, CIFAR100 and SVHN, addressed with the architectures MLP \citep{luo19a}, EfficientNet-b1 \citep{tan19a}, ResNet-34 \citep{he16a}, DenseNet-121 \citep{huang17a} and WideResNet \citep{zagoruyko16a}. Because of space limitation, we will not show all the experiments here, but report on the first three datasets combined respectively with the first three architectures. We refer to the Appendix for the complete benchmark (see Figures 5-16) and the implementation details (Section C). In Section C of the Appendix, we also discuss the sensitivity of the results to the selection of hyper-parameters.
\par In Figure \ref{fig:exp1_short} we compare SGD \citep{robbins51a}, Adam \citep{kingma15a}, SLS \citep{vaswani19a}, SPS \citep{loizou21a} and PoNoS. We present train loss (log scale), test accuracy and average step size within the epoch (log scale). Note that the learning rate of SGD and Adam has been chosen through a grid-search as the one achieving the best performance on each problem. Based on Figure \ref{fig:exp1_short}, we can make the following observations:
\begin{itemize}
\item PoNoS achieves the best performances both in terms of train loss and test accuracy. In particular, it often terminates by reaching a loss of $10^{-6}$ and it always achieves the highest test accuracy. 
\item SLS does not always achieve the same best accuracy as PoNoS. In terms of training loss, SLS is particularly slow on \mlp$ $ and \fashion, while competitive on the others. On these two problems, both SLS and SPS employ a step size that is remarkably smaller than that of PoNoS. On the same problems, SPS behaves similarly as SLS because the original Polyak step size is larger than the ``smoothing/resetting'' value \eqref{eq:etamax_paper}, i.e., $\tilde{\eta}_{k,0}> \eta_{k-1}\gamma^{b/M}$ and thus the step size is controlled by \eqref{eq:etamax_paper}. On the other hand, the proposed nonmonotone line search selects a larger step size, achieving faster convergence and better generalization \citep{nacson22a}.
\item SLS encounters some cancellation errors (e.g., around epoch 120 in \res) that lead the step size to be reduced drastically, sometimes reaching values below $10^-7$. These events are caused by the fact that $\fik(w_k)$ and $\fik(w_{k+1})$ became numerically identical. Thanks to the nonmonotone term $C_k$ replacing $\fik(w_k)$, PoNoS avoids cancellations errors by design.
\item SGD, Adam and SPS never achieve the best test accuracy nor the best training loss. In accordance with the results in \citet{vaswani19a}, the line-search-based algorithms outperform the others.
\end{itemize}
Because of space limitation, we defer the comparison between PoNoS and its monotone counterpart to the Appendix (Section E.2). There we show that PoNoS outperforms the other line search techniques, including a tractable stochastic adaptation of the method by \citet{grippo86a}.

\renewcommand{\dir}{exp1/}
\renewcommand{\imgS}{0.35}
\begin{figure}[!ht] 
	\begin{minipage}{0.95\textwidth}
		\subfloat{\includegraphics[width=\imgS\linewidth]{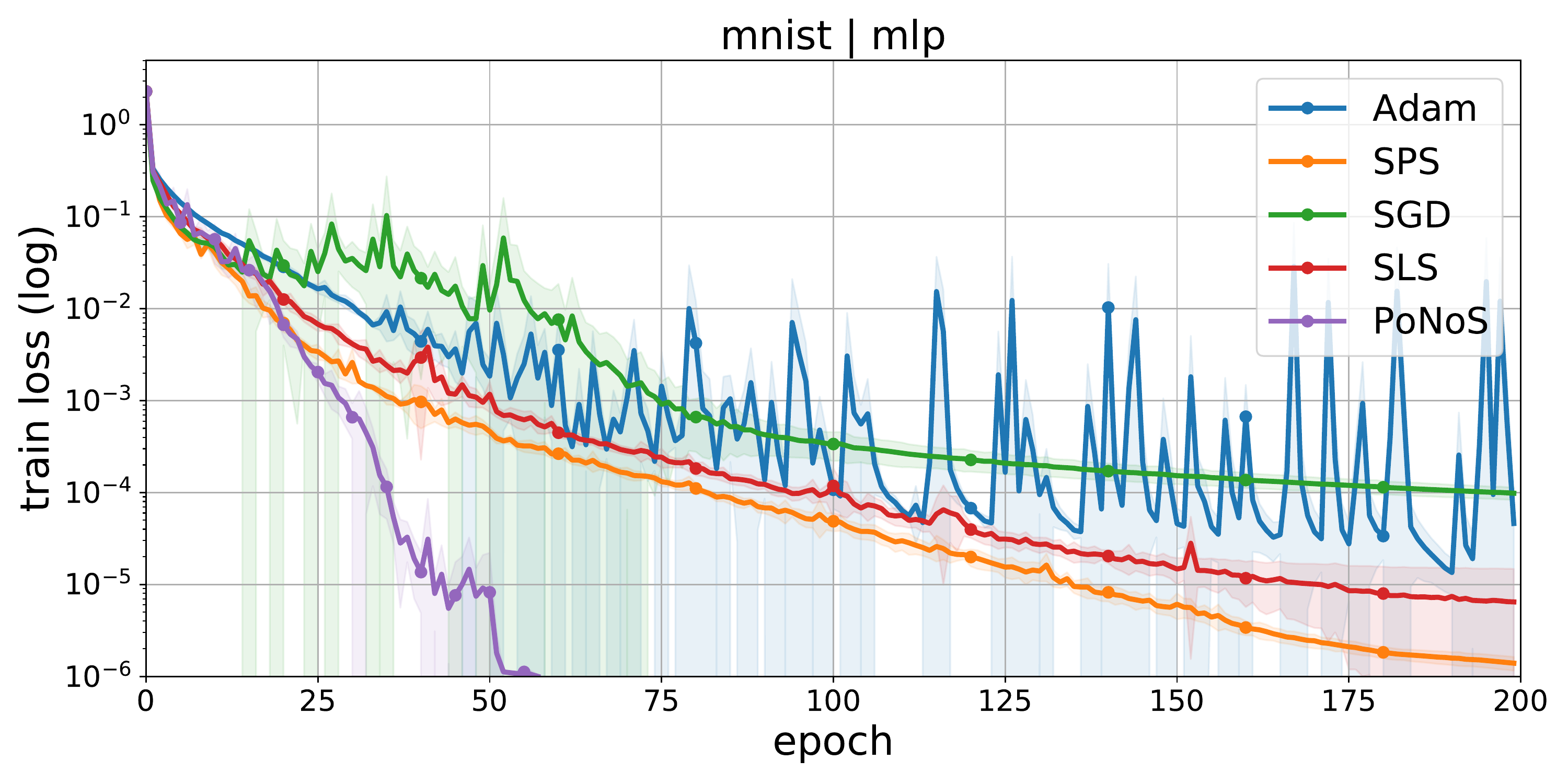}}
		\subfloat{\includegraphics[width=\imgS\linewidth]{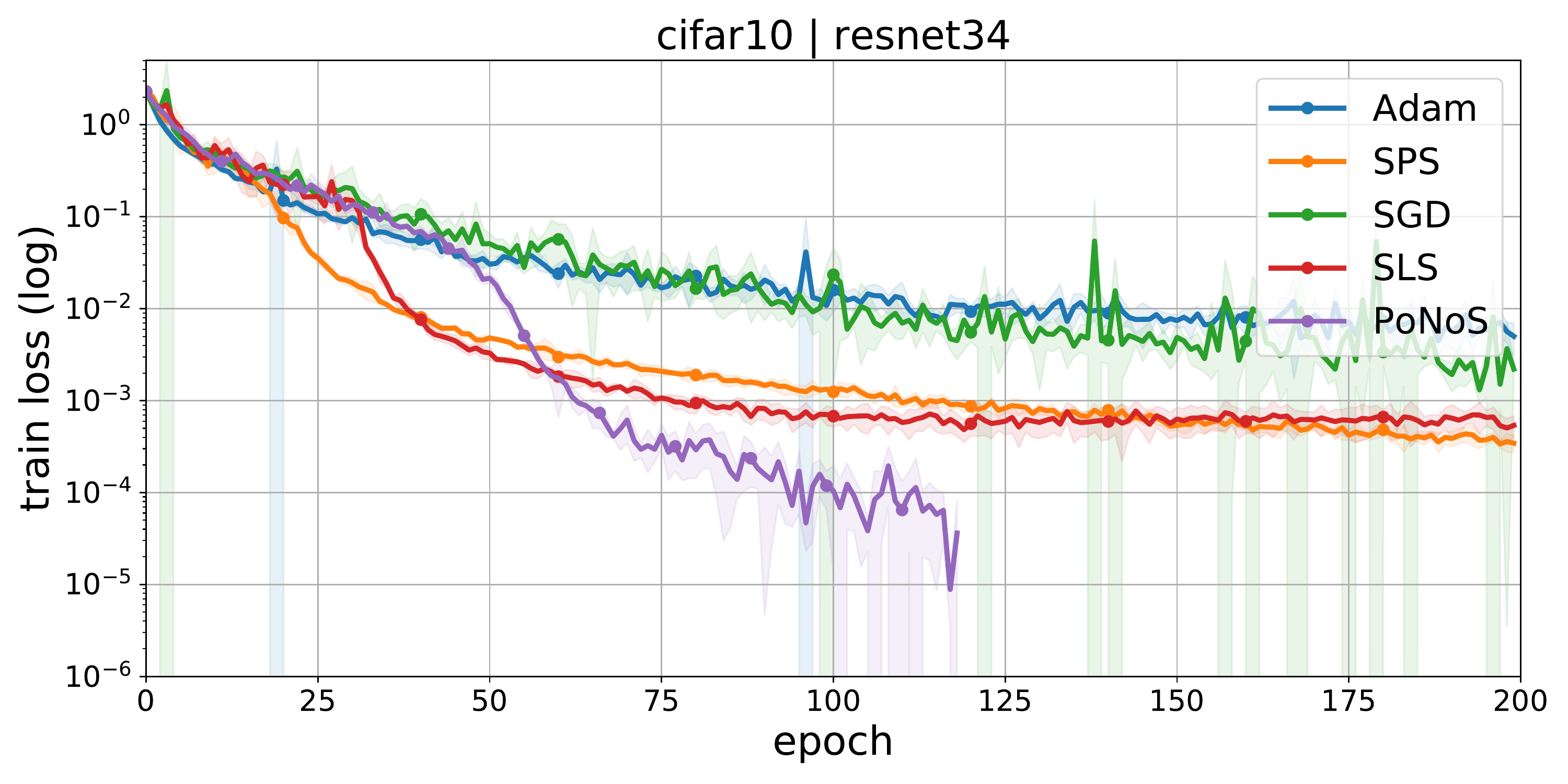}}
		\subfloat{\includegraphics[width=\imgS\linewidth]{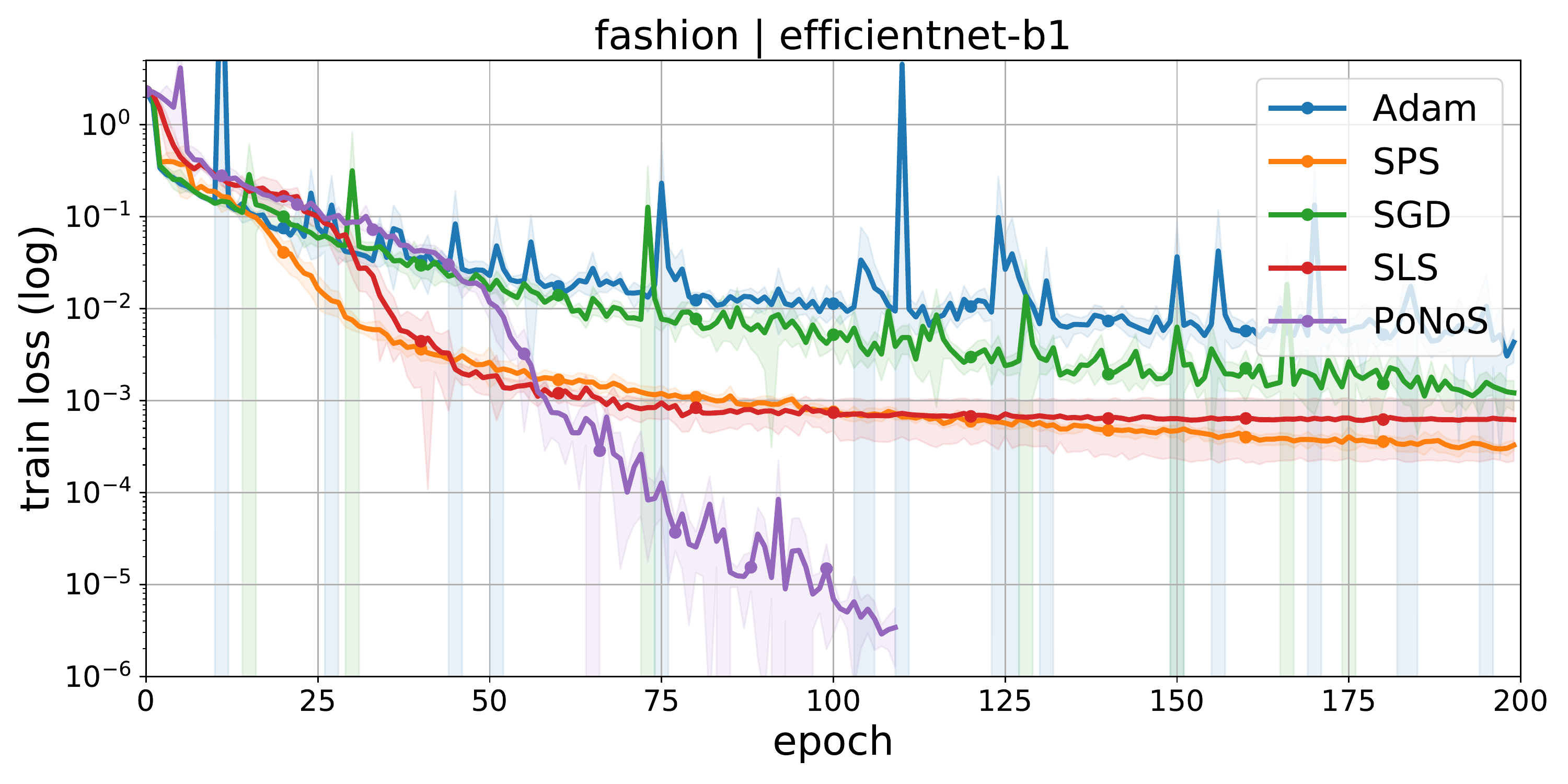}}\\
		\subfloat{\includegraphics[width=\imgS\linewidth]{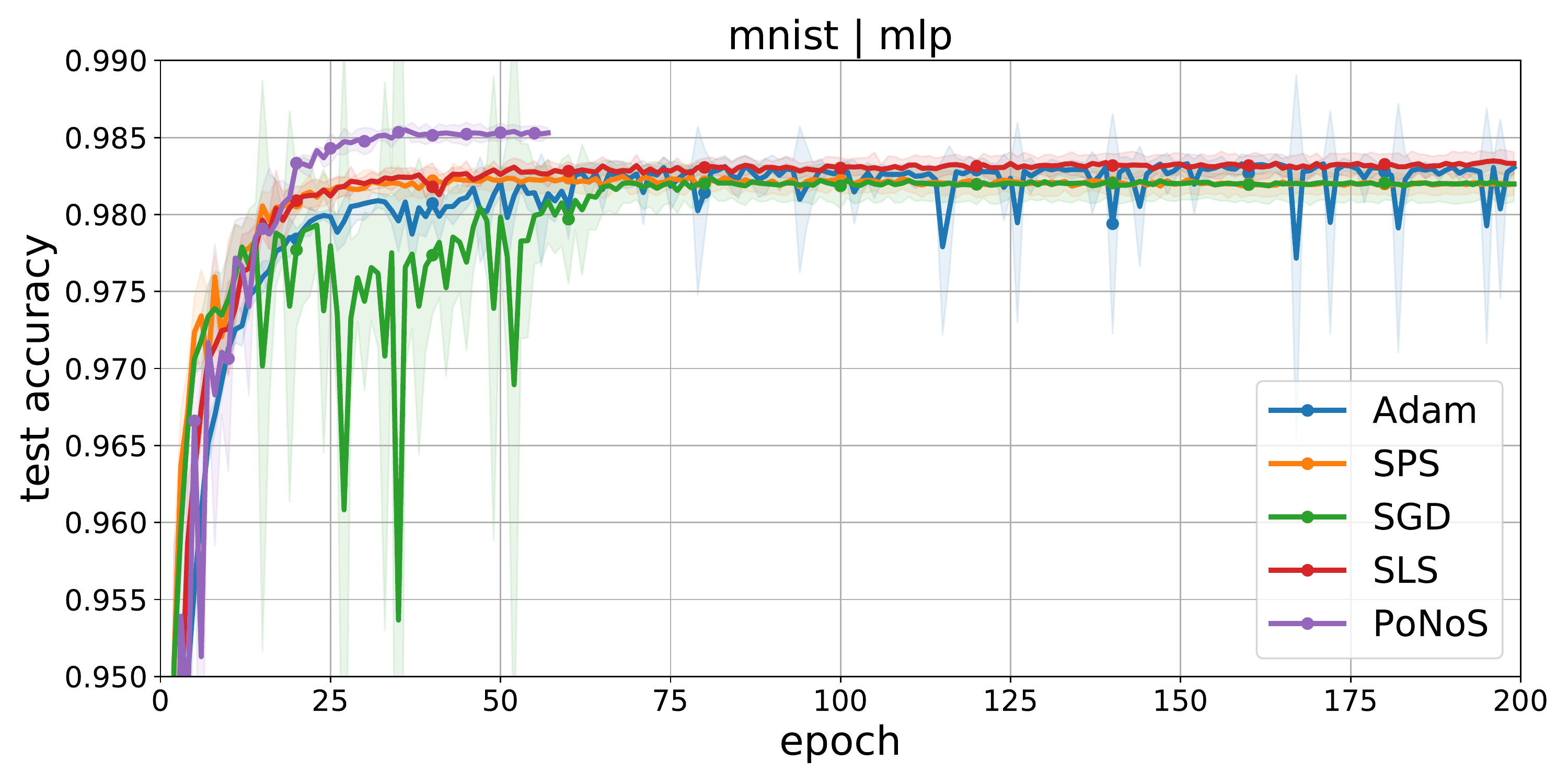}}
		\subfloat{\includegraphics[width=\imgS\linewidth]{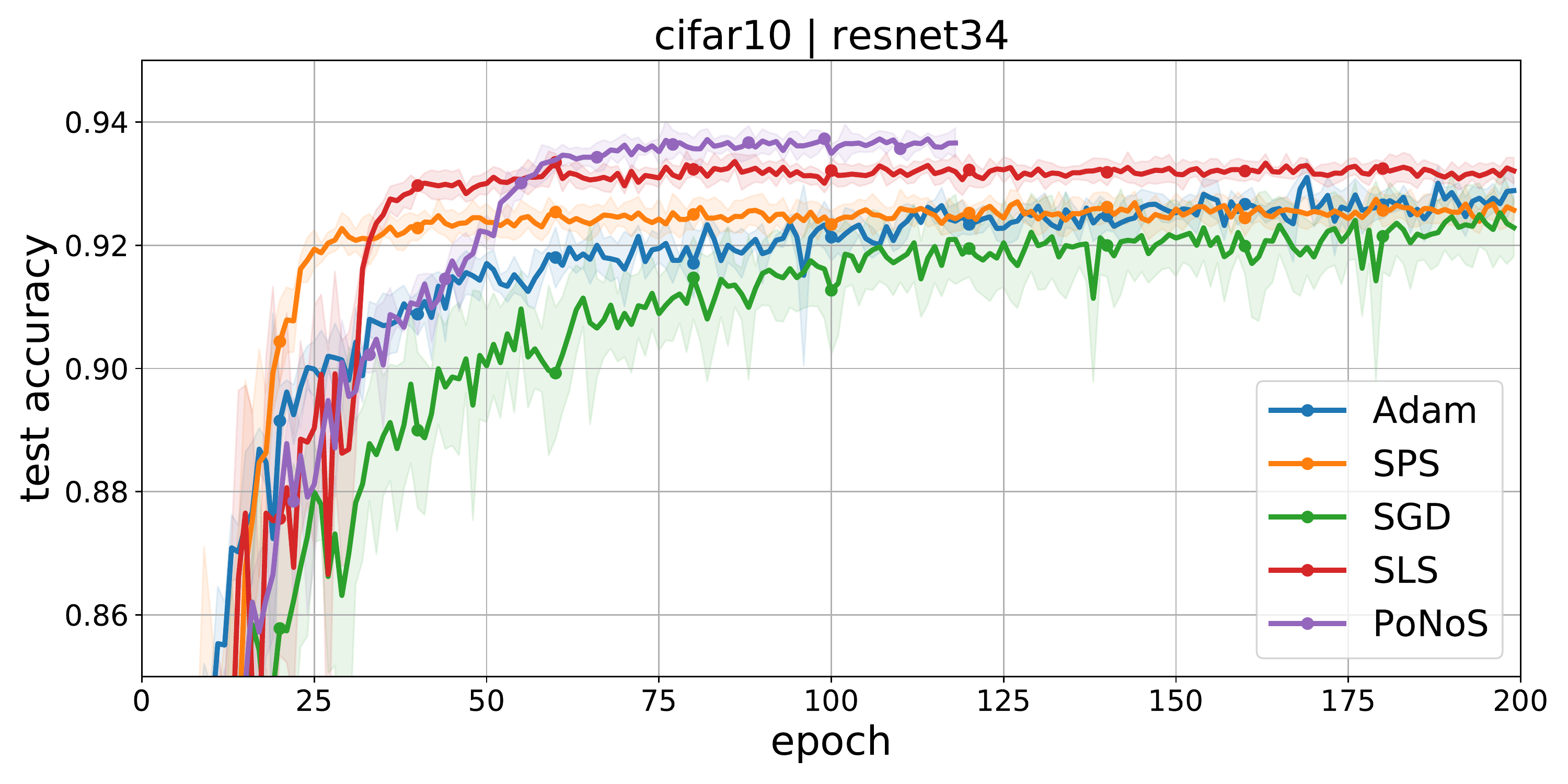}}
		\subfloat{\includegraphics[width=\imgS\linewidth]{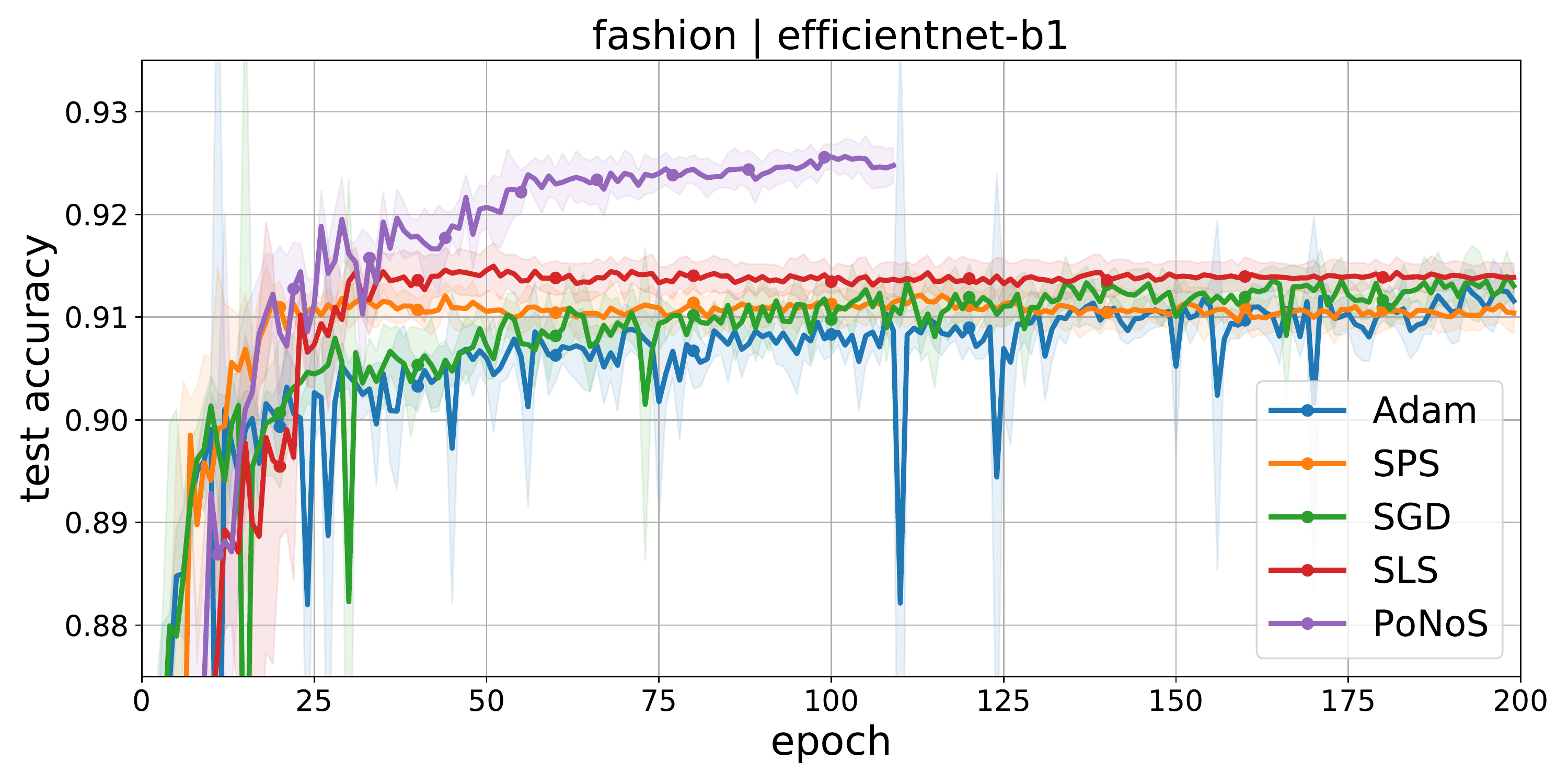}}\\
		\subfloat{\includegraphics[width=\imgS\linewidth]{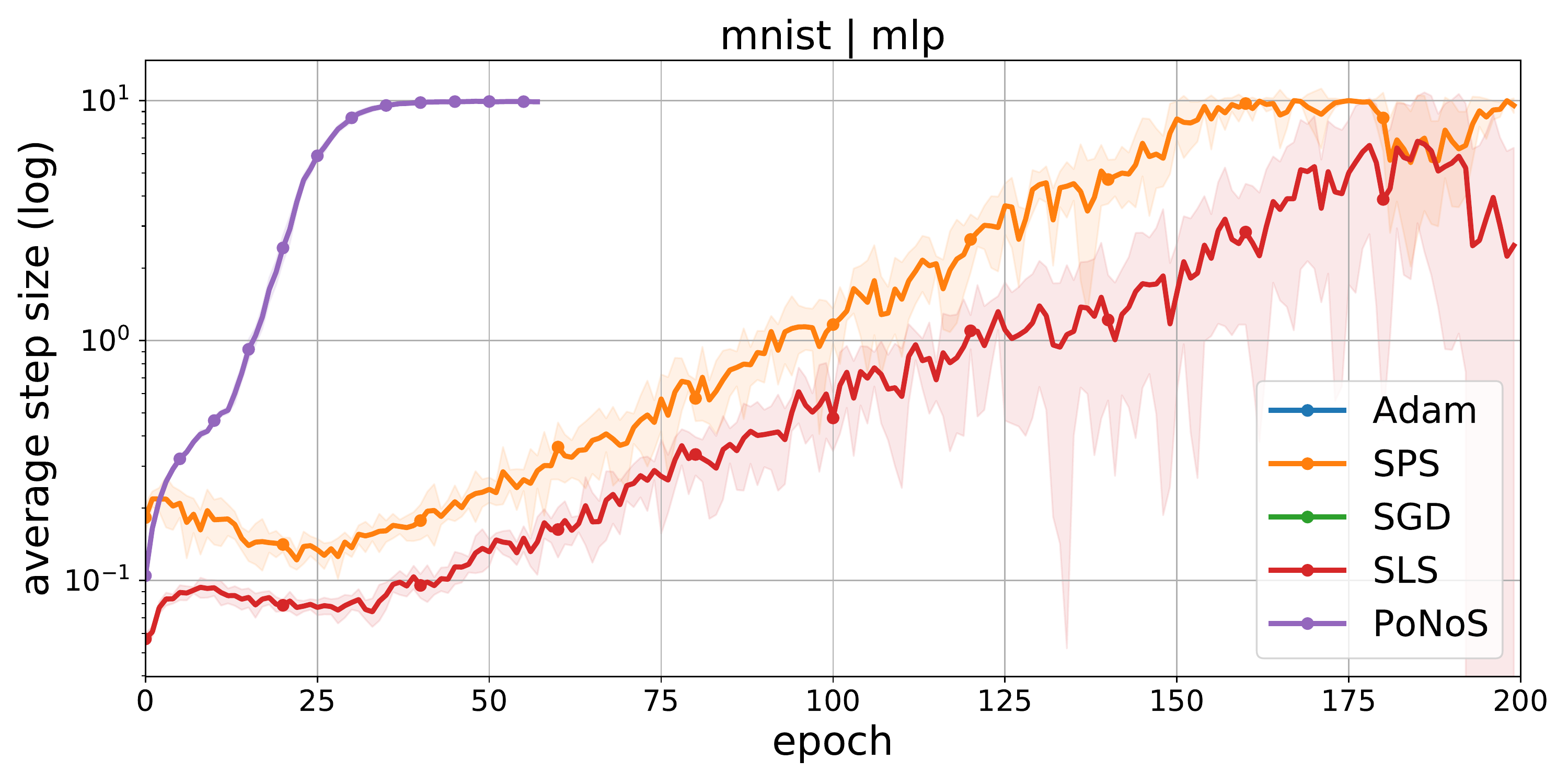}}
		\subfloat{\includegraphics[width=\imgS\linewidth]{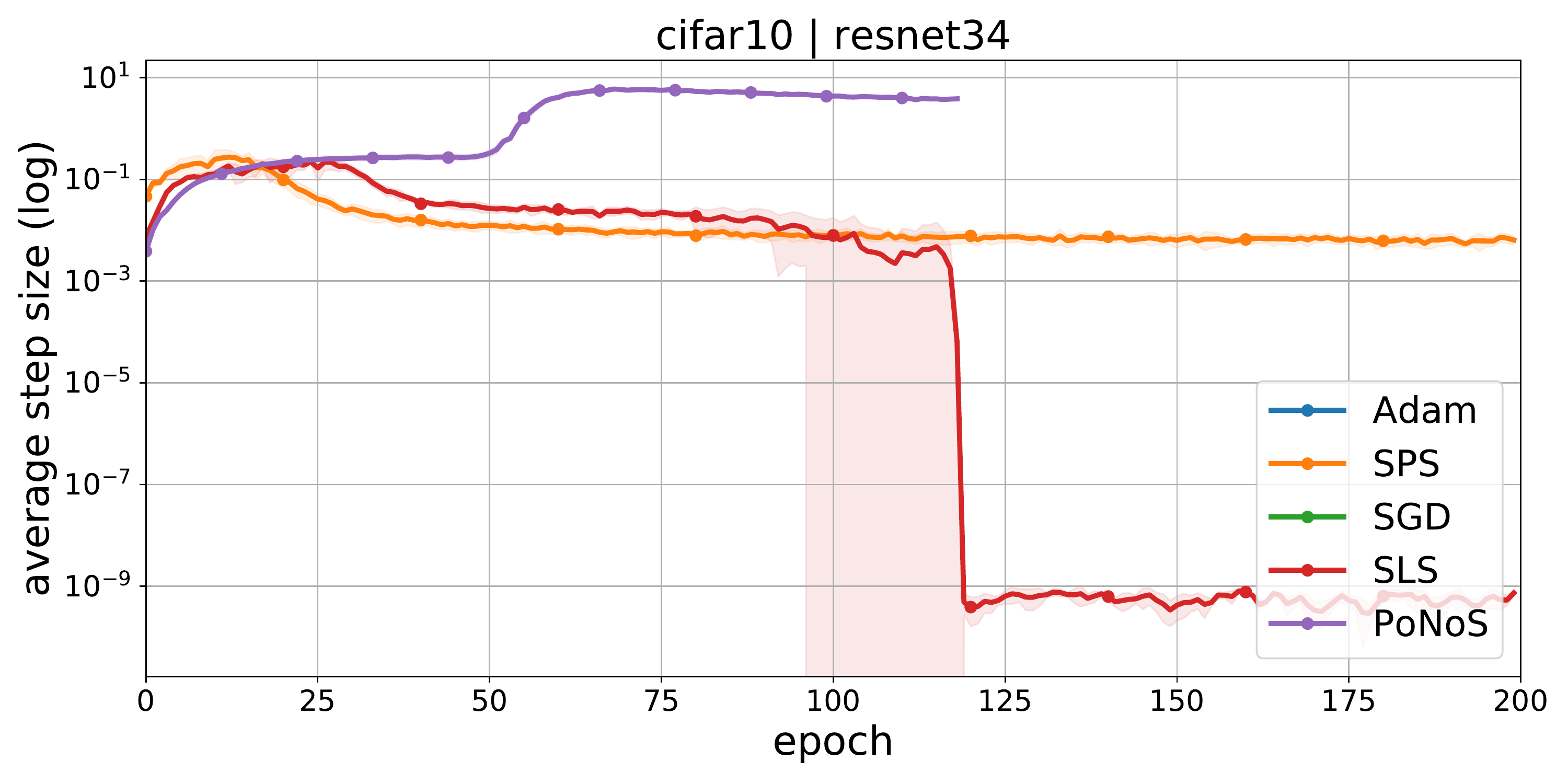}}
		\subfloat{\includegraphics[width=\imgS\linewidth]{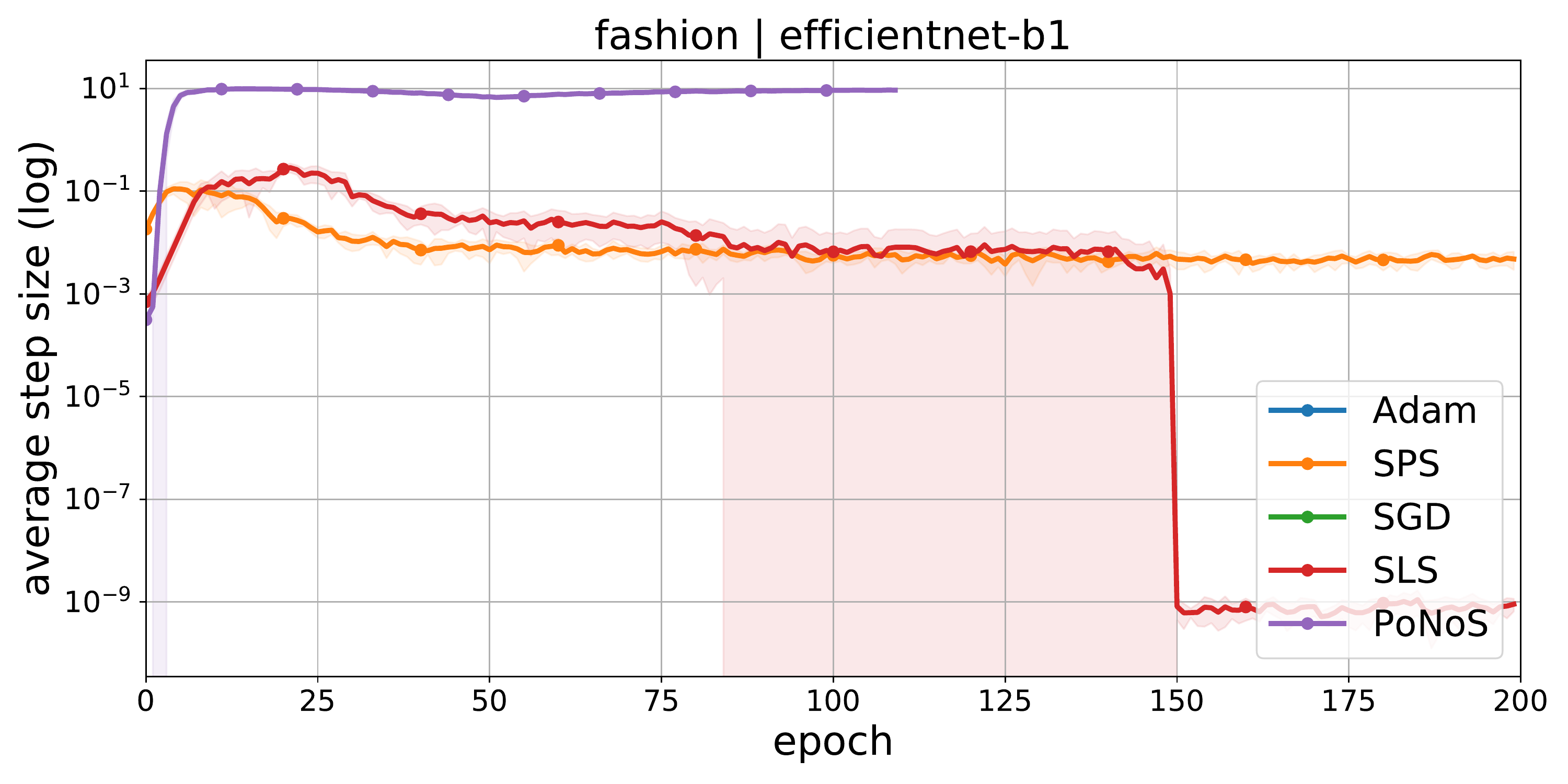}}
		\caption{Comparison between the proposed method (PoNoS) and the-state-of-the-art. Each column focus on a dataset. First row: train loss. Second row: test accuracy. Third row: step size.}\label{fig:exp1_short}
	\end{minipage}
\end{figure}

	\subsection{A New Resetting Technique}
	\par In this subsection, we compare different initial step sizes and resetting techniques. We fix the line search to be \eqref{eq:zhang_paper}, but similar comments can be made on \eqref{eq:armijo_paper} (see the Section D.2 of the Appendix). In Figure \ref{fig:reset_short} we compare PoNoS and PoNoS\_reset0 (without \eqref{eq:new_etak_paper}) with zhang\_reset2 (initial step size \eqref{eq:etamax_paper}) and: 
\begin{itemize}
	\item zhang\_reset3: $\eta_{k,0}=\eta_{k-1} \frac{||\nabla f_{i_{k-1}}(w_{k-1})||^2}{||\gradik(w_k)||^2}$, adapted to SGD from \citet{nocedal06a},
	\item zhang\_reset4: $\eta_{k,0}= \frac{2\left(f_{i_{k-1}}(w_{k-1}) - f_{i_{k-1}}(w_k)\right)}{||\nabla f_{i_{k-1}}(w_{k-1})||}.$ adapted to SGD from  \citet{nocedal06a},
	\item zhang\_every2: same as PoNoS\_reset0, but a new step is computed only every 2 iterations. 
\end{itemize}
In Figure \ref{fig:reset_short}, we report train loss (log scale), the total amount of backtracks per epoch and the average step size within the epoch (log scale). From Figure \ref{fig:reset_short}, we can make the following observations:

\begin{itemize}
	\item PoNoS and PoNoS\_reset0 achieve very similar performances. In fact, the two algorithms yield step sizes that are almost always overlapping. An important difference between PoNoS\_reset0 and PoNoS can be noticed in the amount of backtracks that the two algorithms require. The plots show a sum of $\frac{M}{b}$ elements, with $\frac{M}{b}=391$ or $469$ depending on the problem and PoNoS's line hits exactly this value. This means that PoNoS employs a median of 1 backtrack per iteration for the first 5-25 epochs, while PoNoS\_reset0 needs more backtracks in this stage (around 1500-3000 per-epoch, see Section D.2 of the Appendix). After this initial phase, both PoNoS\_reset0 and PoNoS reduce the amount of backtracks until it reaches a (almost) constant value of 0.
	\item zhang\_every2 does not achieve the same good performance as PoNoS or PoNoS\_reset0. The common belief that step sizes can be used in many subsequent iterations does not find confirmation here. In fact, zhang\_every2 shows that we cannot skip the application of a line search if we want to maintain the same good performances.
	\item All the other initial step sizes achieve poor performances on both train loss and test accuracy. In many cases, the algorithms yield step sizes (also before the line search) that are too small w.r.t. \eqref{eq:loizou_paper}. 
\end{itemize}

\renewcommand{\dir}{reset/}
\renewcommand{\imgS}{0.35}
\begin{figure}[!ht] 
	\begin{minipage}{0.95\textwidth}
		\subfloat{\includegraphics[width=\imgS\linewidth]{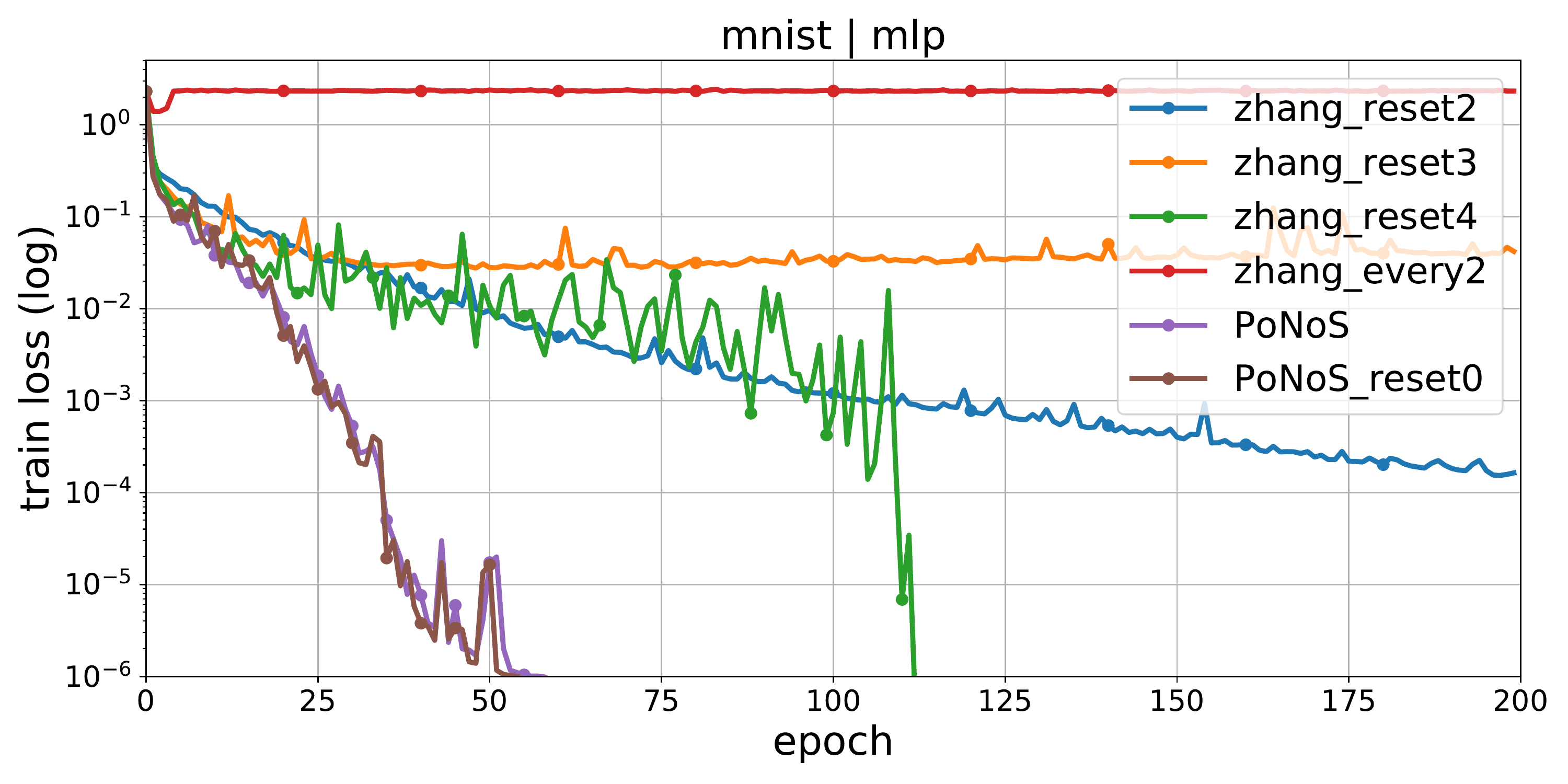}}
		\subfloat{\includegraphics[width=\imgS\linewidth]{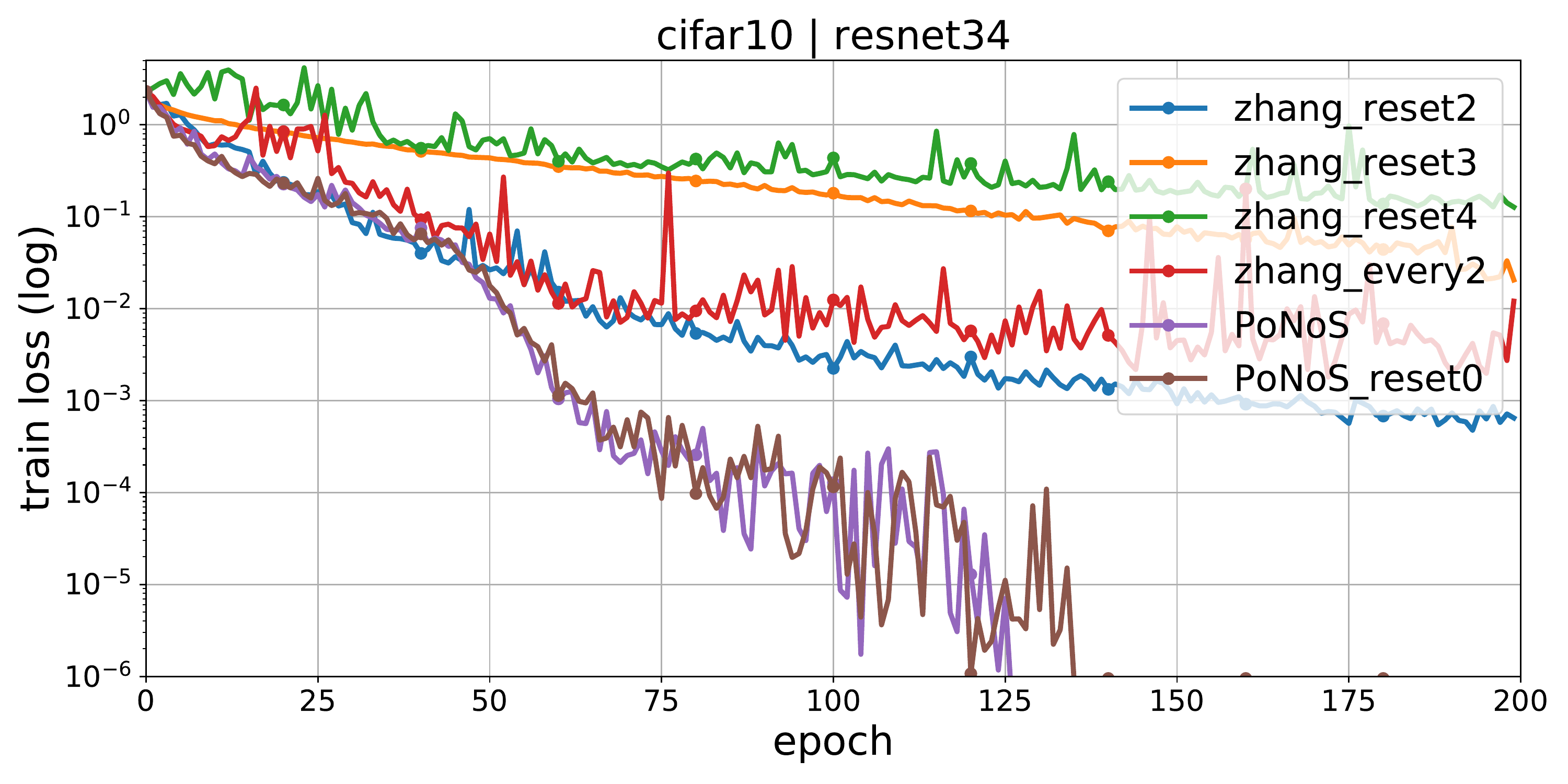}}
		\subfloat{\includegraphics[width=\imgS\linewidth]{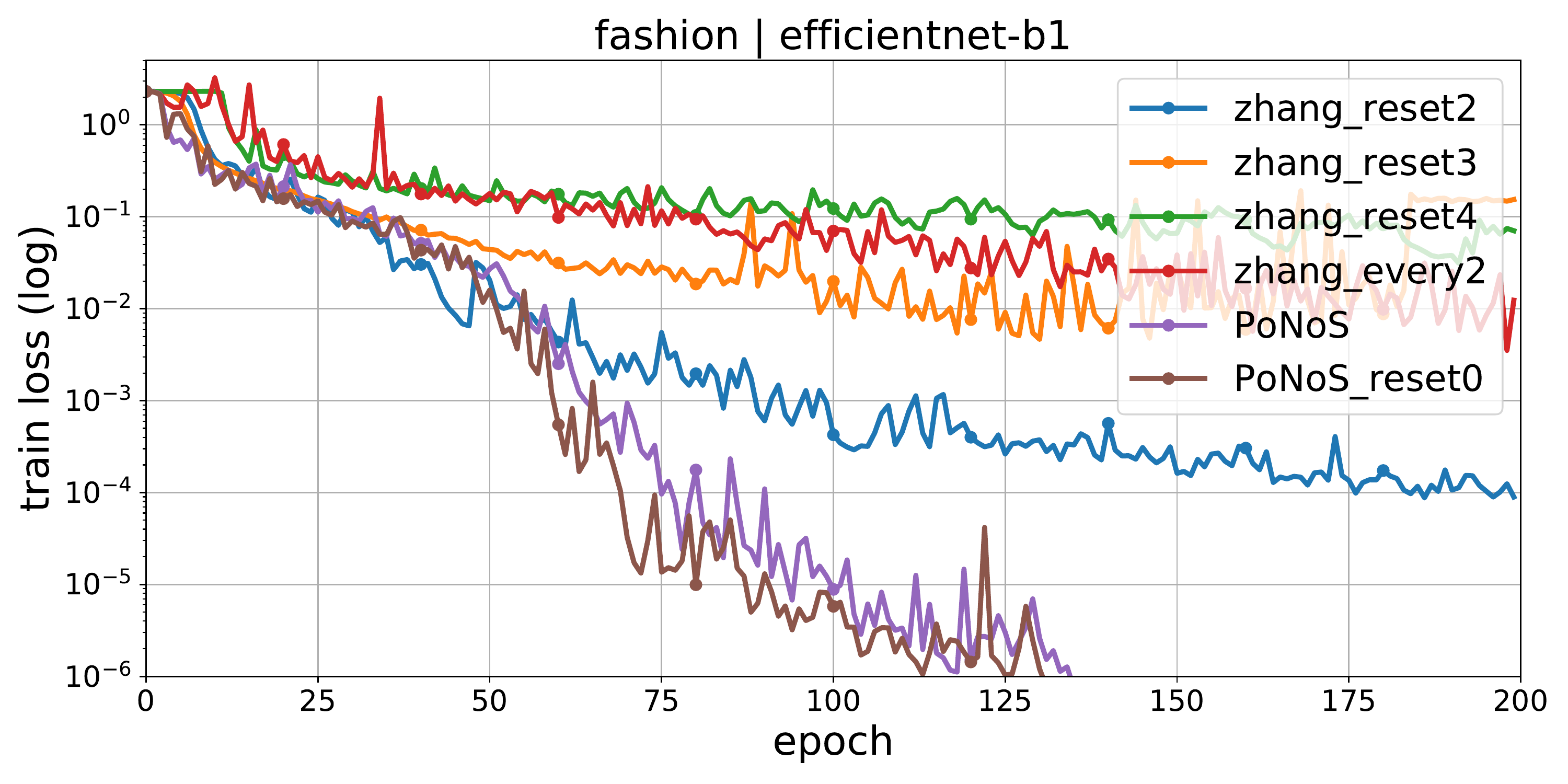}}\\
		\subfloat{\includegraphics[width=\imgS\linewidth]{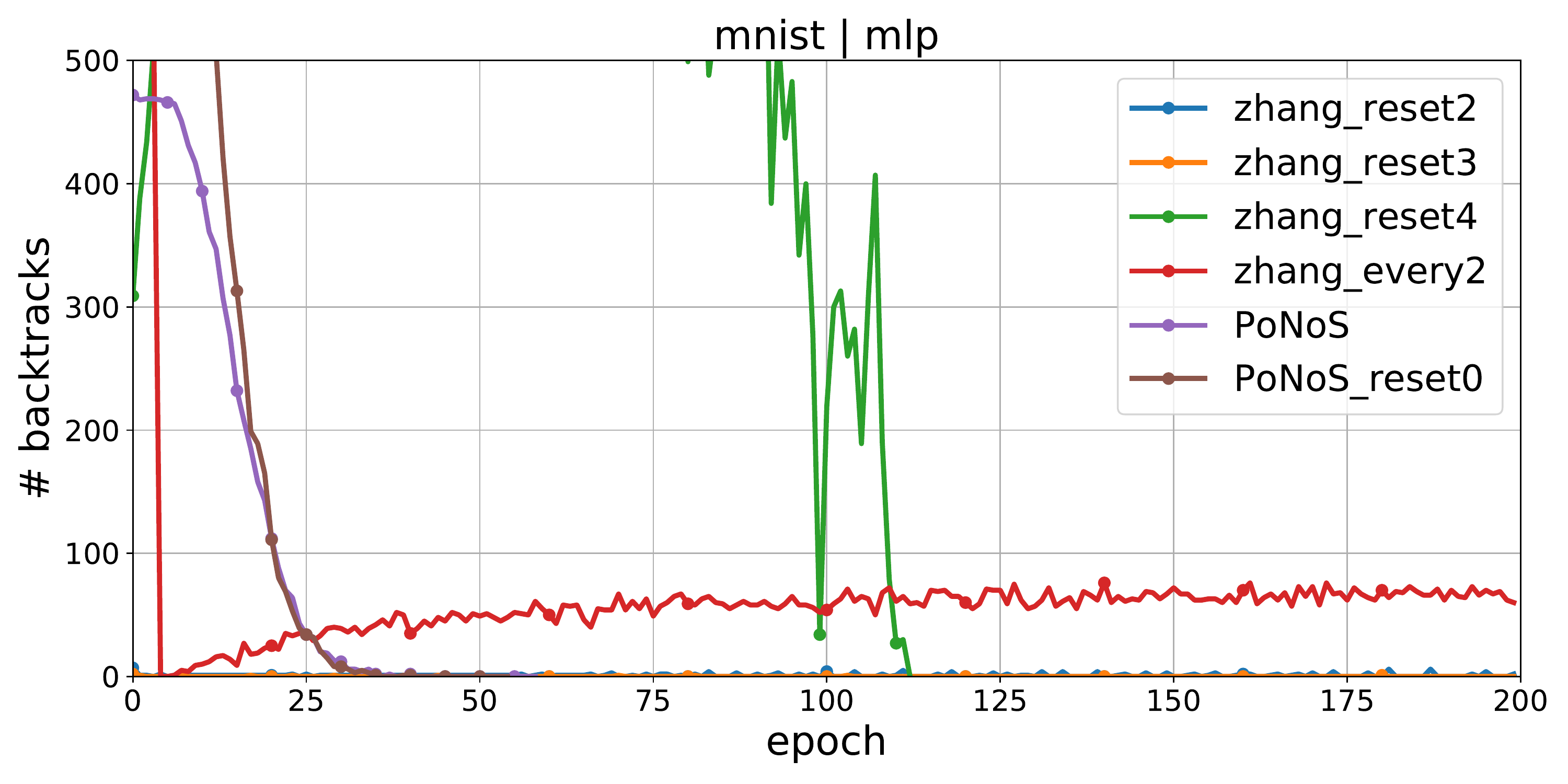}}
		\subfloat{\includegraphics[width=\imgS\linewidth]{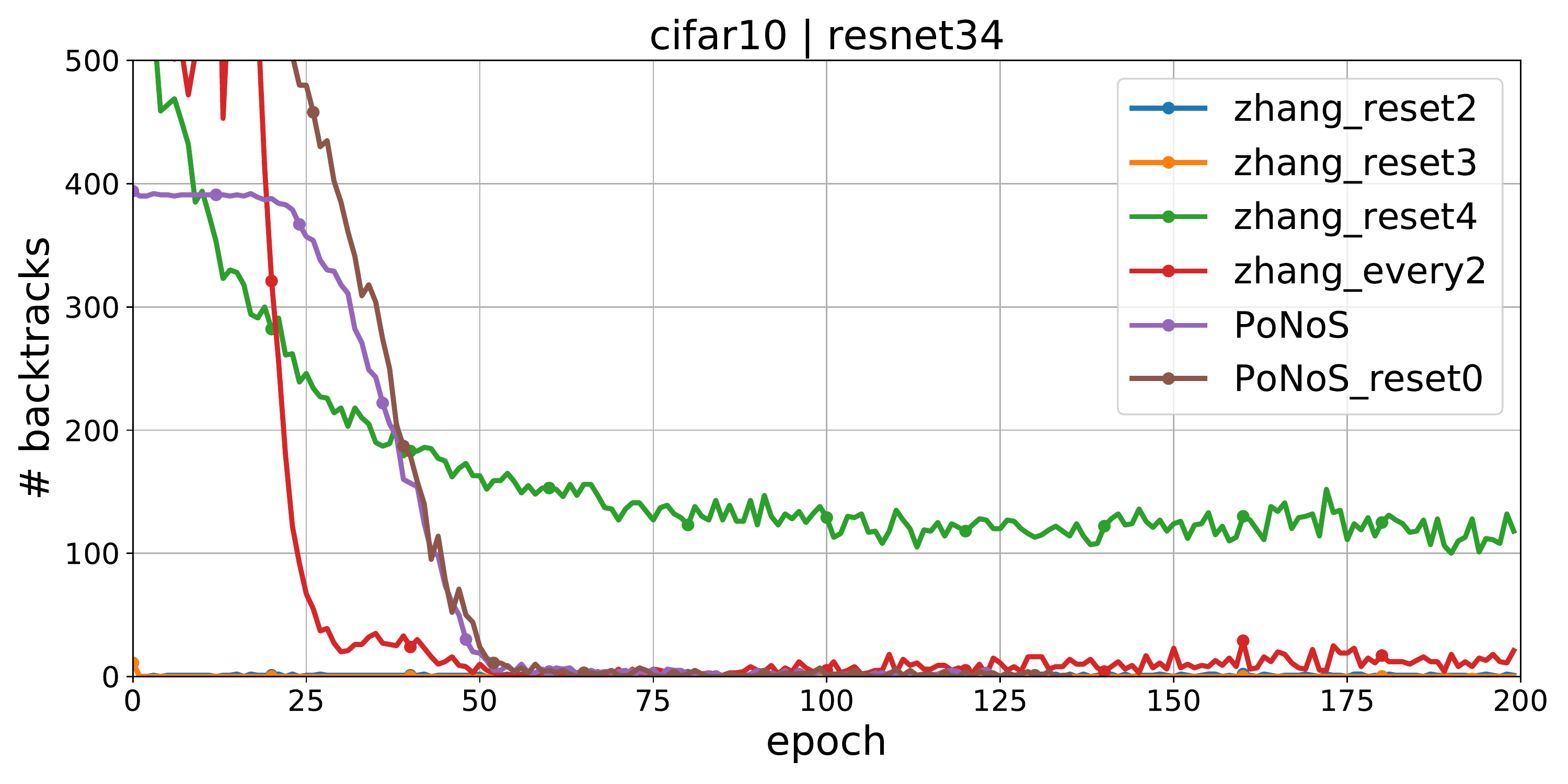}}
		\subfloat{\includegraphics[width=\imgS\linewidth]{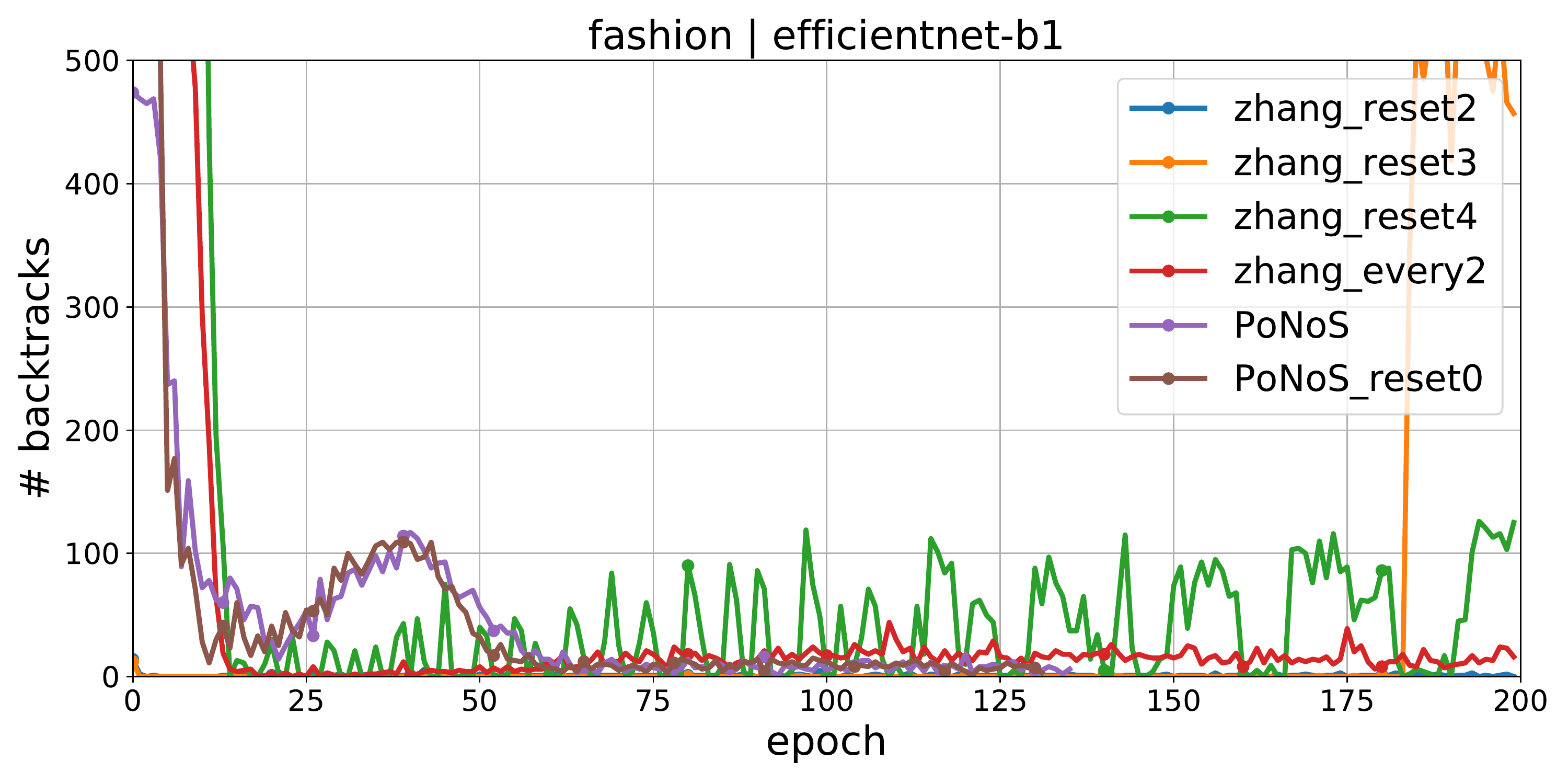}}\\
		\subfloat{\includegraphics[width=\imgS\linewidth]{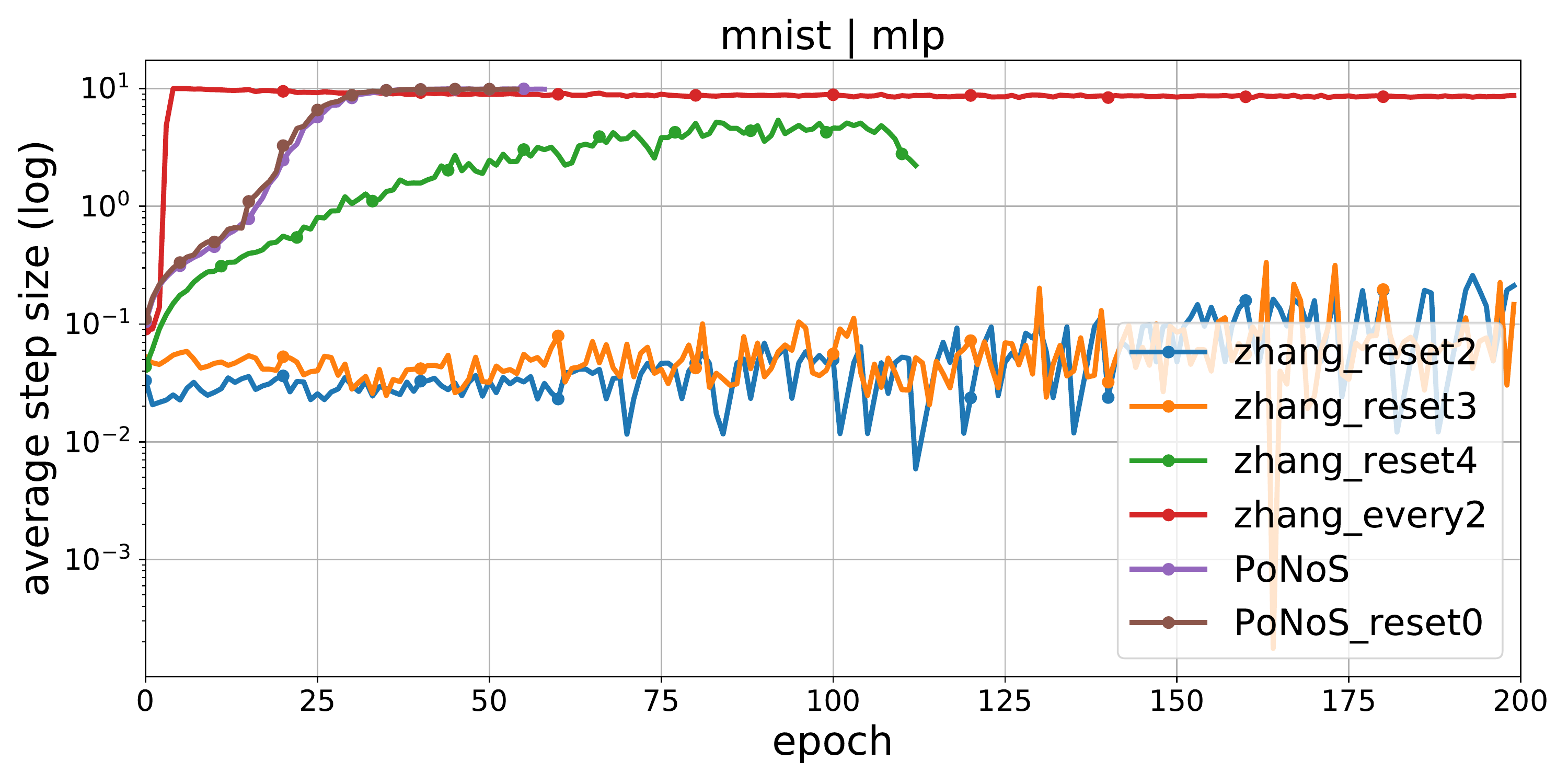}}
		\subfloat{\includegraphics[width=\imgS\linewidth]{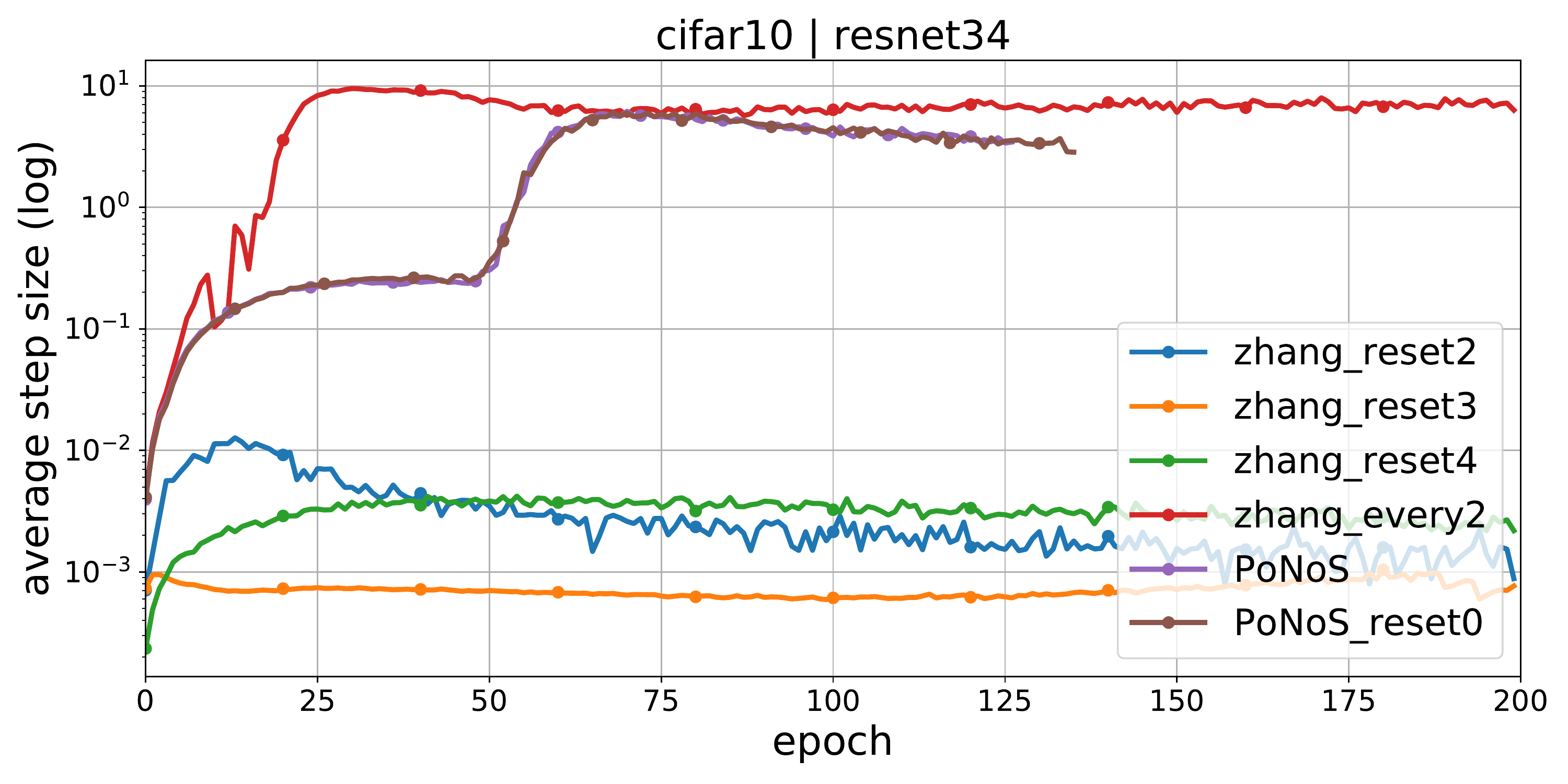}}
		\subfloat{\includegraphics[width=\imgS\linewidth]{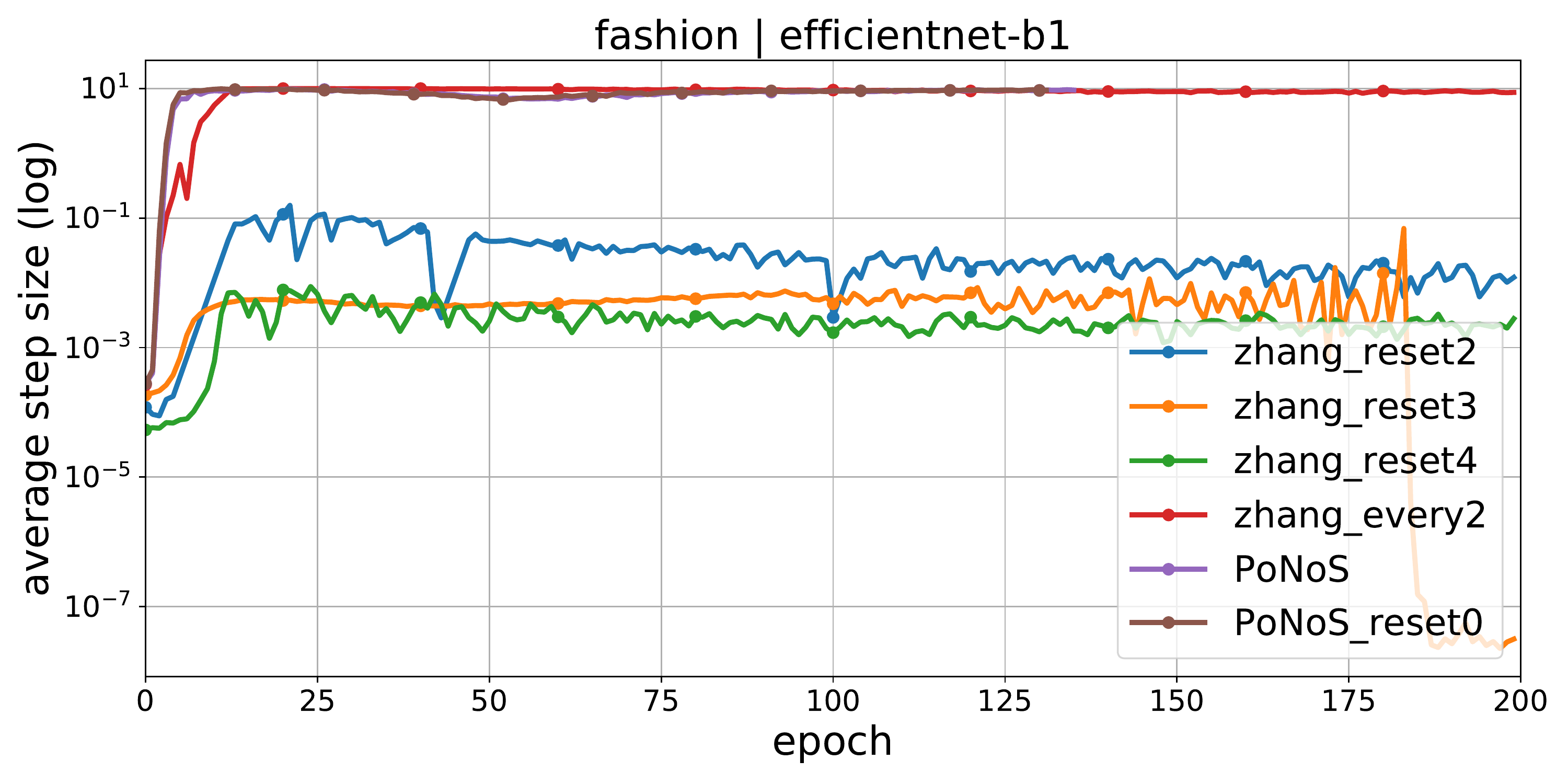}}
		\caption{Comparison between different initial step sizes and resetting techniques. Each column focus on a dataset. First row: train loss. Second row: $\#$ backtracks. Third row: step size.}\label{fig:reset_short}
	\end{minipage}
\end{figure}
	
	\subsection{Time Comparison}
	In this subsection, we show runtime comparisons corresponding to Figure \ref{fig:exp1_short}. In the first row of Figure \ref{fig:time_short}, we plot the train loss as in the first row of Figure \ref{fig:exp1_short}. However, the $x$-axis of Figure \ref{fig:time_short} measures the cumulative epoch time of the average of 5 different runs of the same algorithm with different seeds. In the second row of Figure \ref{fig:time_short}, we report the runtime per-epoch (with shaded error bars) on the $y$-axis and epochs on the $x$-axis. From Figures \ref{fig:exp1_short} and \ref{fig:time_short}, it is clear that PoNoS is not only faster than the other methods in terms of epochs but also in terms of total computational time. In particular, PoNoS is faster than SGD, despite the fact the second achieves the lowest per-epoch time. Again from the second row of Figure \ref{fig:time_short}, we can observe that PoNoS's per-epoch time makes a transition from the phase of a median of 1 backtrack (first 5-25 epochs) to the phase of a median of 0 backtracks where its time is actually overlapping with that of SLS (always a median of 0 backtracks). In the first case, PoNoS requires less than twice the time of SGD, and in the second, this time is lower than $1.5$ that of SGD. Given these measures, PoNoS becomes faster than SGD/Adam in terms of per-epoch time as soon as a grid-search (or any other hyper-parameter optimization) is employed to select the best-performing learning rate.
\par To conclude, let us recall that any algorithm based on a stochastic line search always requires one additional forward pass if compared with SGD. In fact, both $\fik(w_k)$ and $\fik(w_{k+1})$ are computed at each $k$ and each backtrack requires one additional forward pass. On the other hand, if we consider that one backward pass costs roughly two forward passes and that SGD needs one forward and one backward pass, any additional forward pass costs roughly one-third of SGD. These rough calculations have been verified in Section E.6 of the Appendix, where we profiled the single iteration of PoNoS. Following these calculations and referring to the two phases of Figure \ref{fig:time_short}, one iteration of PoNoS only costs $\frac{5}{3}$ that of SGD in the first phase and $\frac{4}{3}$ in the second. 

\renewcommand{\dir}{time/}
\renewcommand{\imgS}{0.35}
\begin{figure}[!ht] 
	\hspace{-1mm}
	\begin{minipage}{0.95\textwidth}
		\subfloat{\includegraphics[width=\imgS\linewidth]{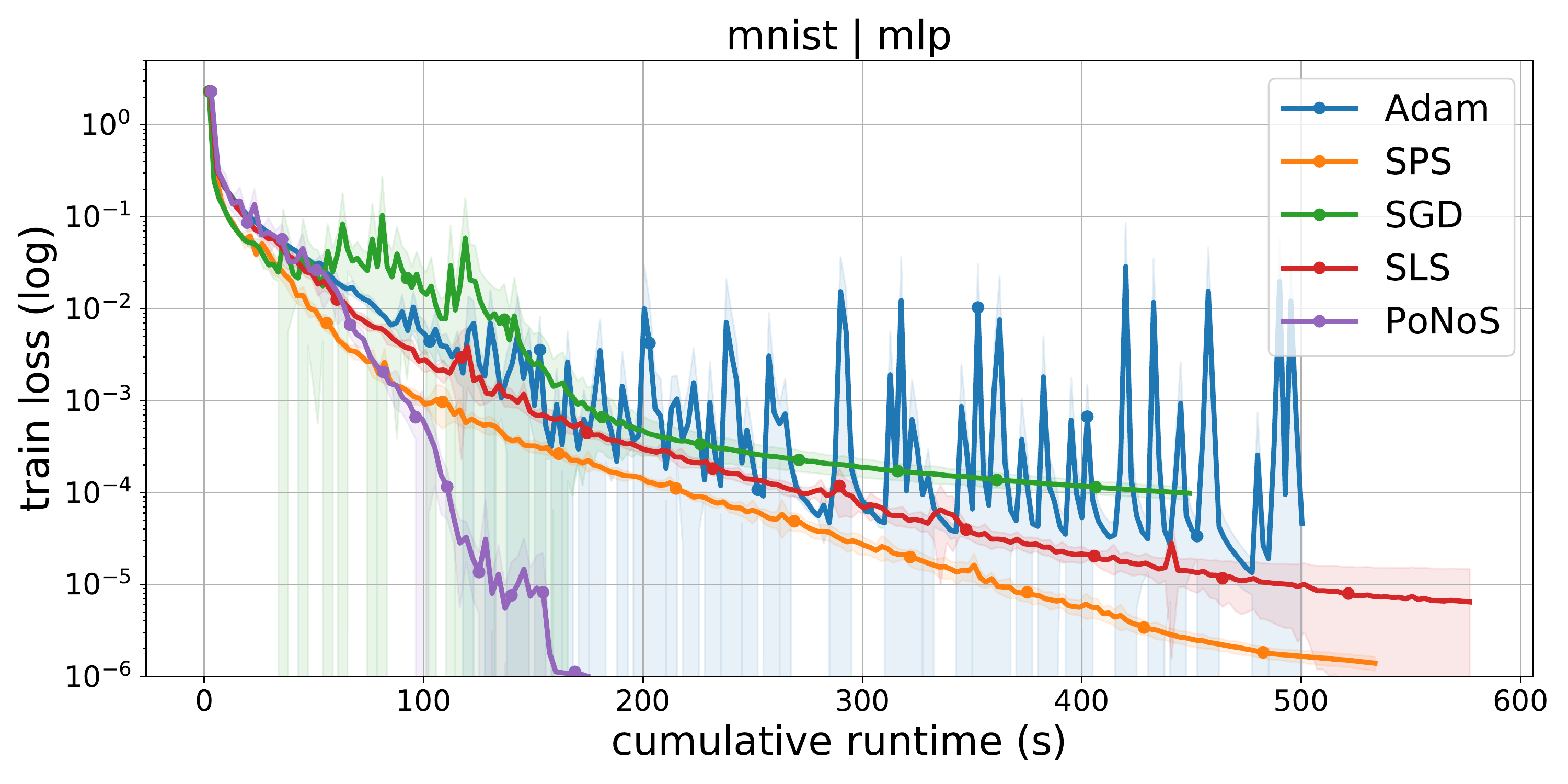}}
		\subfloat{\includegraphics[width=\imgS\linewidth]{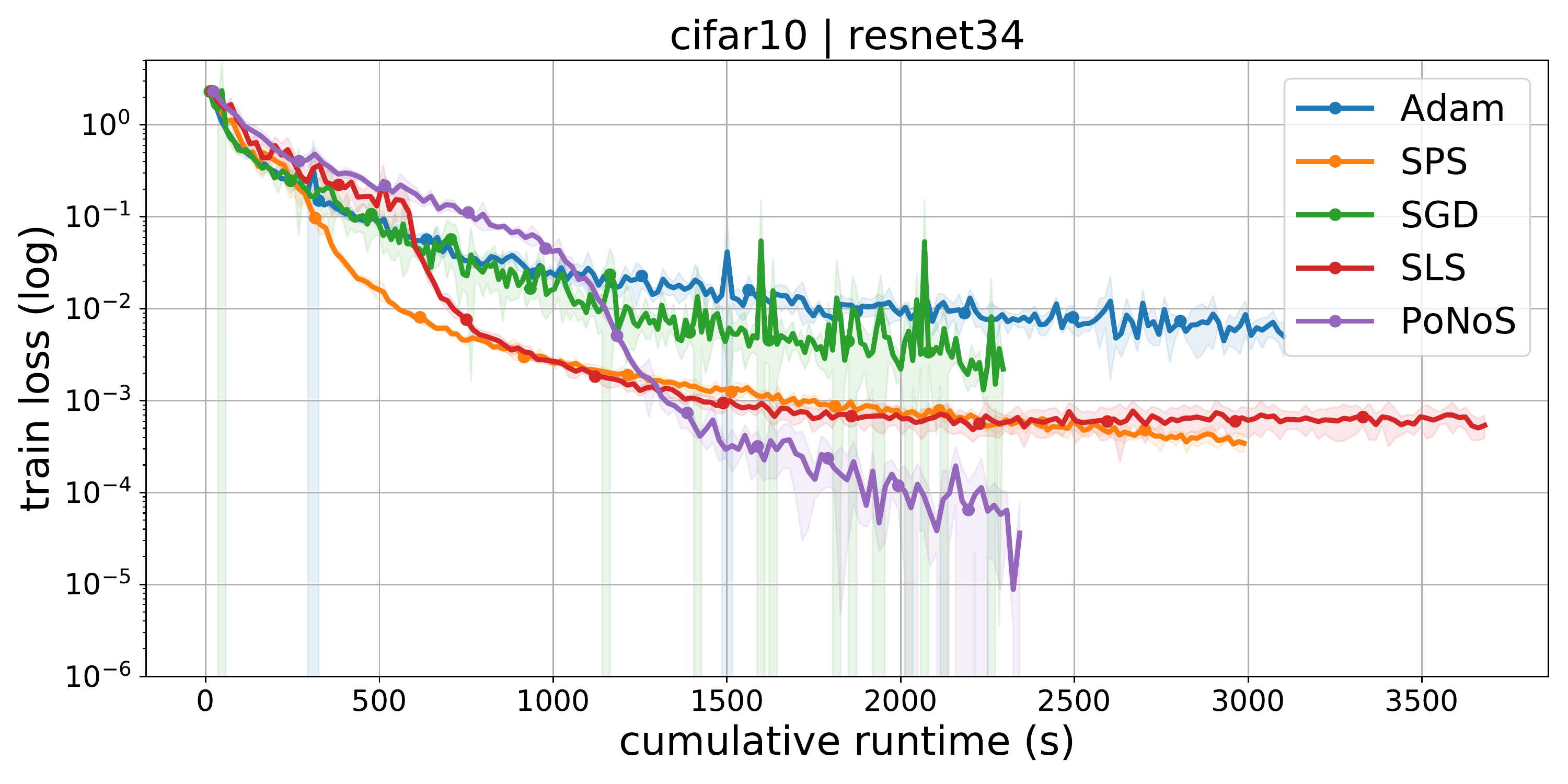}}
		\subfloat{\includegraphics[width=\imgS\linewidth]{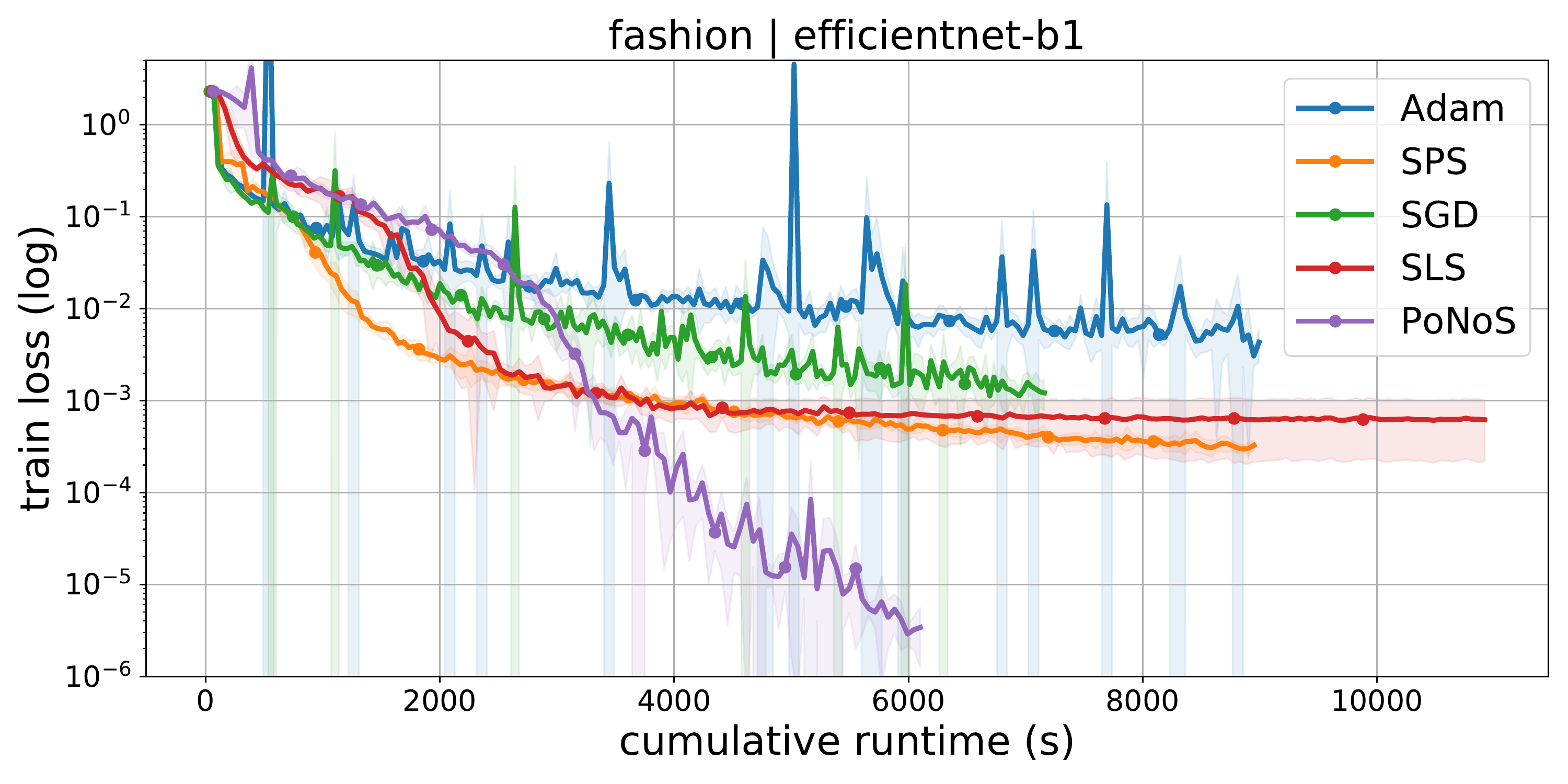}}\\
		\renewcommand{\dir}{exp1/}
		\subfloat{\includegraphics[width=\imgS\linewidth]{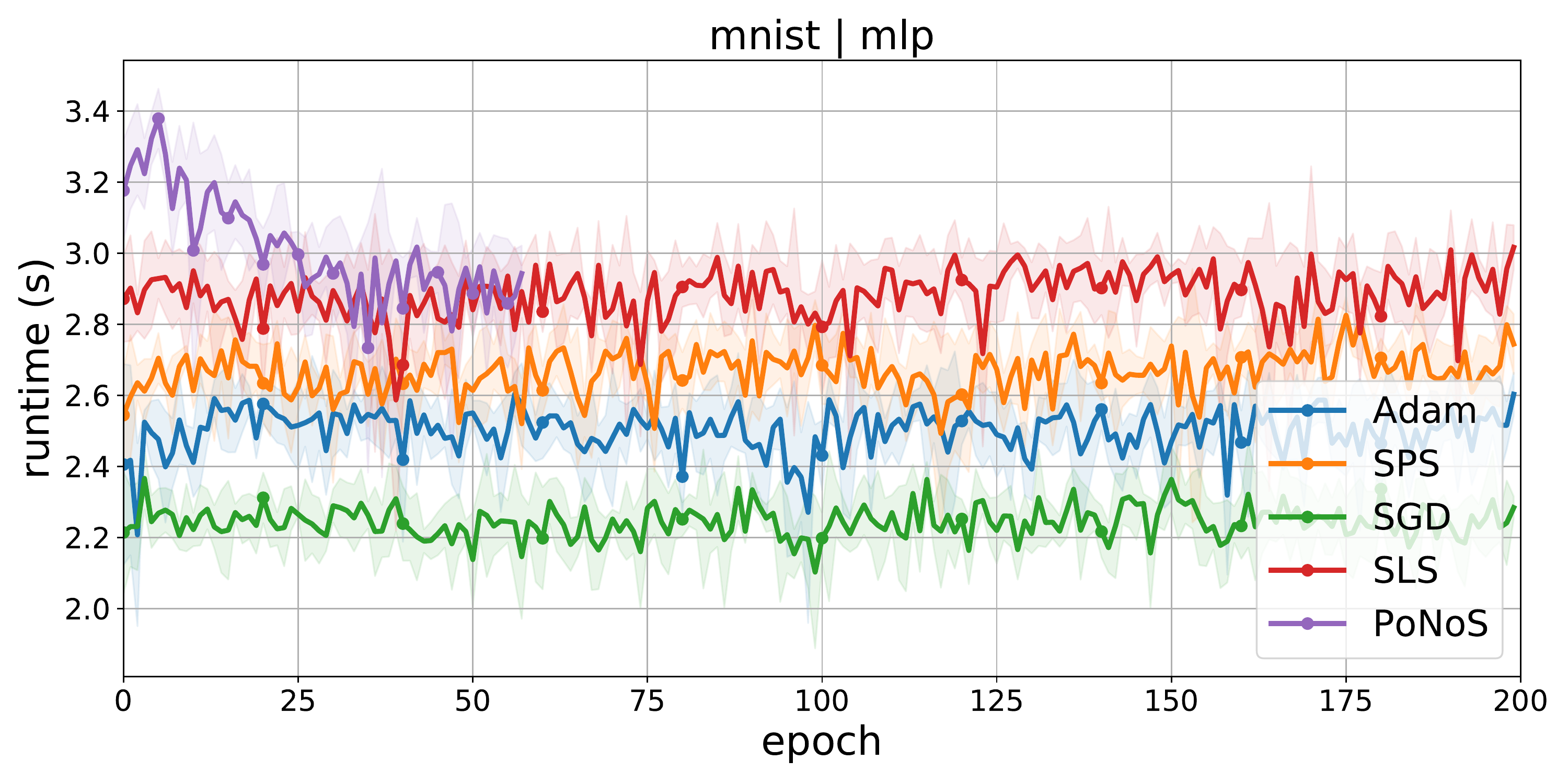}}
		\subfloat{\includegraphics[width=\imgS\linewidth]{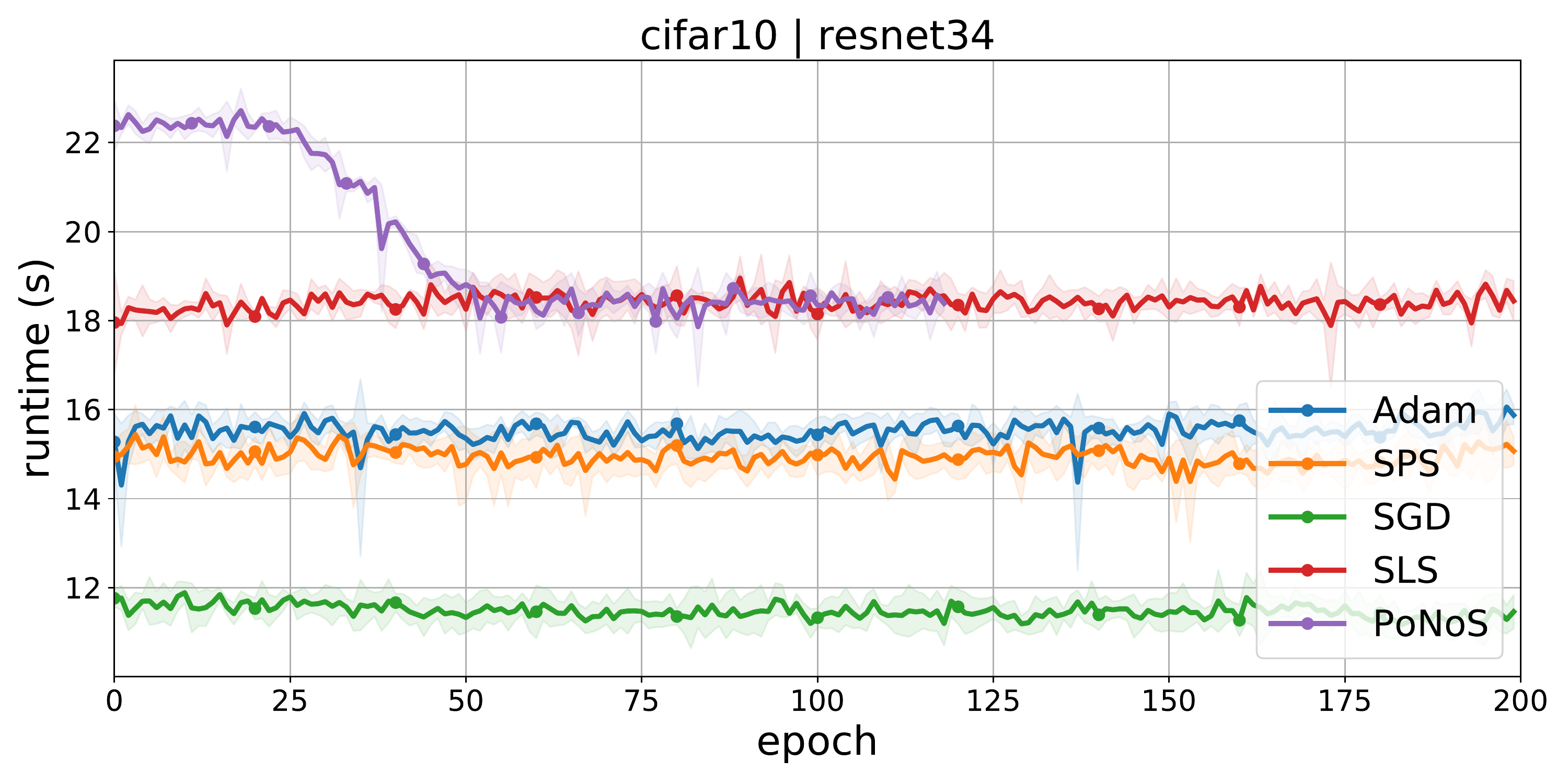}}
		\subfloat{\includegraphics[width=\imgS\linewidth]{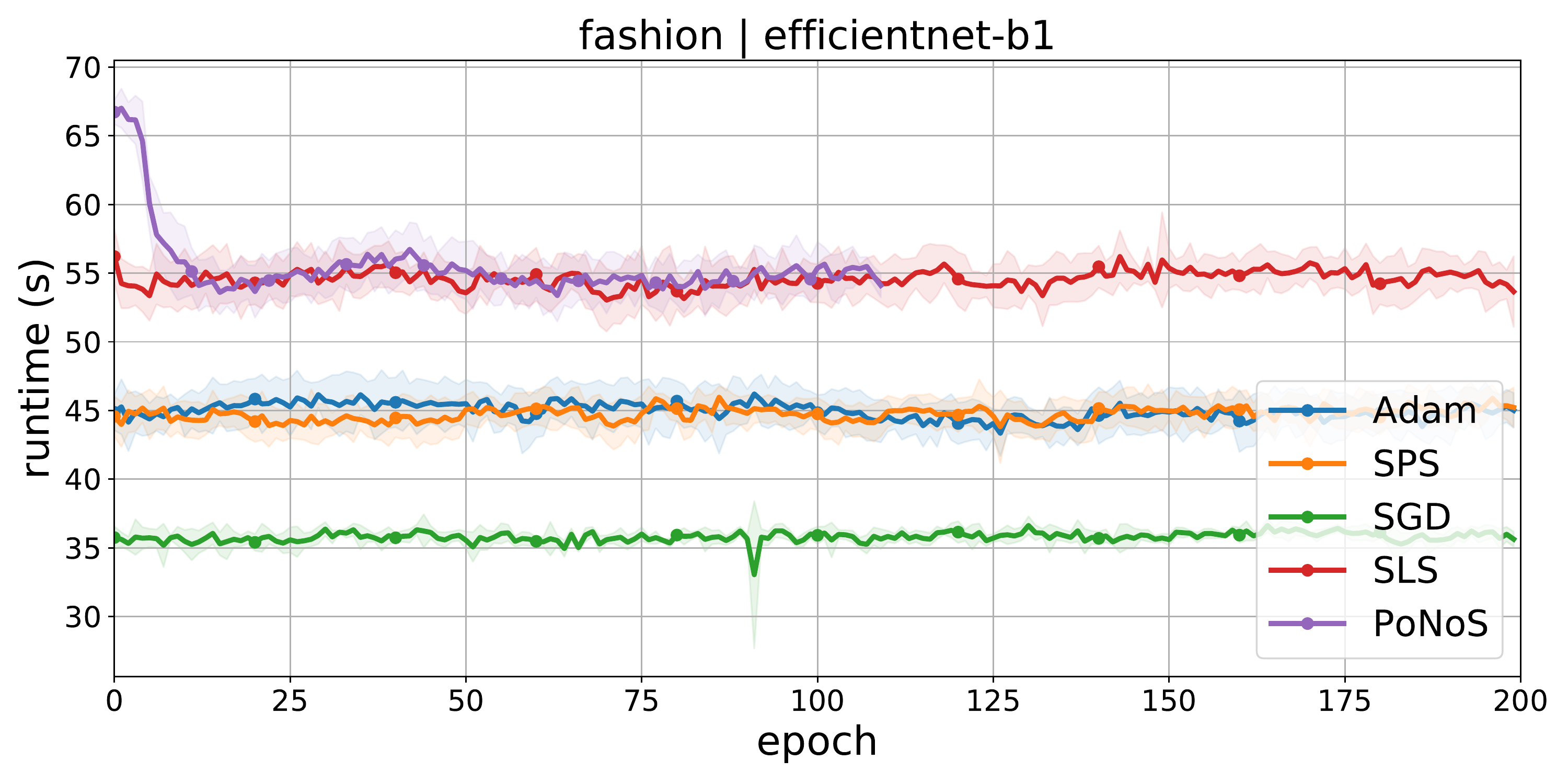}}\\
		\caption{Time comparison (s) between the proposed method (PoNoS) and the-state-of-the-art. Each column focus on a dataset. First row: train loss vs runtime. Second row: epoch runtime.}\label{fig:time_short}
	\end{minipage}
\end{figure}
	
	\subsection{Experiments for Convex Losses and for Transformers}\label{sec:convex_and_trans}
	As a last benchmark, we take into account a set of convex problems from \citet{vaswani19a,loizou21a} and the transformers \citep{vaswani17a} trained from scratch in \citet{kunstner23a}. In Figure \ref{fig:convex_trans} we show one convex problem (first column) and two transformers (last two column). We leave the fine-tuning of transformers to future works. Our experiments take into account binary classification problems addressed with a RBF kernel model without regularization. We show the results achieved on the dataset mushrooms, while leaving those on ijcnn, rcv1 and w8a to Section D.4 of the Appendix. Given the high amount of iterations, the smoothed version of the train loss will be reported. The test accuracy is also only reported in Section D.4 since (almost) all the methods achieve the best test accuracy in all the problems within the first few iterations. From the left subplot of Figure \ref{fig:convex_trans}, we can observe that PoNoS obtains very good performances also in this setting (see Appendix). In Figure \ref{fig:convex_trans}, PoNoS achieves a very low train loss ($10^{-4}$) within the first 200 iterations. Only SLS is able to catch up, but this takes 6 times the iterations of PoNoS. On this problem, SLS is the only method reaching the value ($10^{-6}$). The methods SLS and SPS behave very similarly on all the datasets (see the Appendix) since in both cases \eqref{eq:etamax_paper} controls the step size. As clearly shown by the comparison with PoNoS, this choice is suboptimal and the Polyak step size is faster. Because of \eqref{eq:etamax_paper}, both SLS and SPS encounter slow convergence issues in many of problems of this setting. As in Figure \ref{fig:exp1_short}, SGD and Adam are always slower than PoNoS. 

We consider training transformers on language modeling datasets. In particular, we train a Transformer Encoder \citep{vaswani17a} on PTB \citep{marcus93a} and a Transformer-XL \citep{dai19a} on Wikitext2 \citep{merity17a}. In contrast to the case of convolutional neural networks, the most popular method for training transformers is Adam and not SGD \citep{pan22a,kunstner23a}. For this reason, we use a preconditioned version of PoNoS, SLS, and SPS (respectively PoNoS\_prec, SLS\_prec, and SPS\_prec, see Section D.5 of the Appendix for details). In fact, Adam can be considered a preconditioned version of SGD with momentum \citep{vaswani20a}. As in \citet{kunstner23a}, we focus on the training procedure and defer the generalization properties to the Appendix. From Figure \ref{fig:convex_trans}, we can observe that the best algorithms in this setting are preconditioned-based and not SGD-based, in accordance with the literature. Moreover, PoNoS\_prec achieves similar performances as Adam's. In particular from the central subplot of Figure \ref{fig:convex_trans}, we observe that PoNoS\_prec is the only algorithm as fast as Adam, while all the others have difficulties achieving loss below 1. Taking the right subplot of Figure \ref{fig:convex_trans} into account, we can notice that PoNoS\_prec is slower than Adam on this problem. On the other hand, there is not one order of difference between the two final losses (i.e., $\sim 1.5$ points). Moreover, we should keep in mind that Adam's learning rate has been fine-tuned separately for each problem, while PoNoS\_prec has been used off-the-shelf.

\renewcommand{\imgS}{0.35}
\begin{figure}[!h] 
	\hspace{-1mm}
	\begin{minipage}{0.95\textwidth}
		\subfloat{\includegraphics[width=\imgS\linewidth]{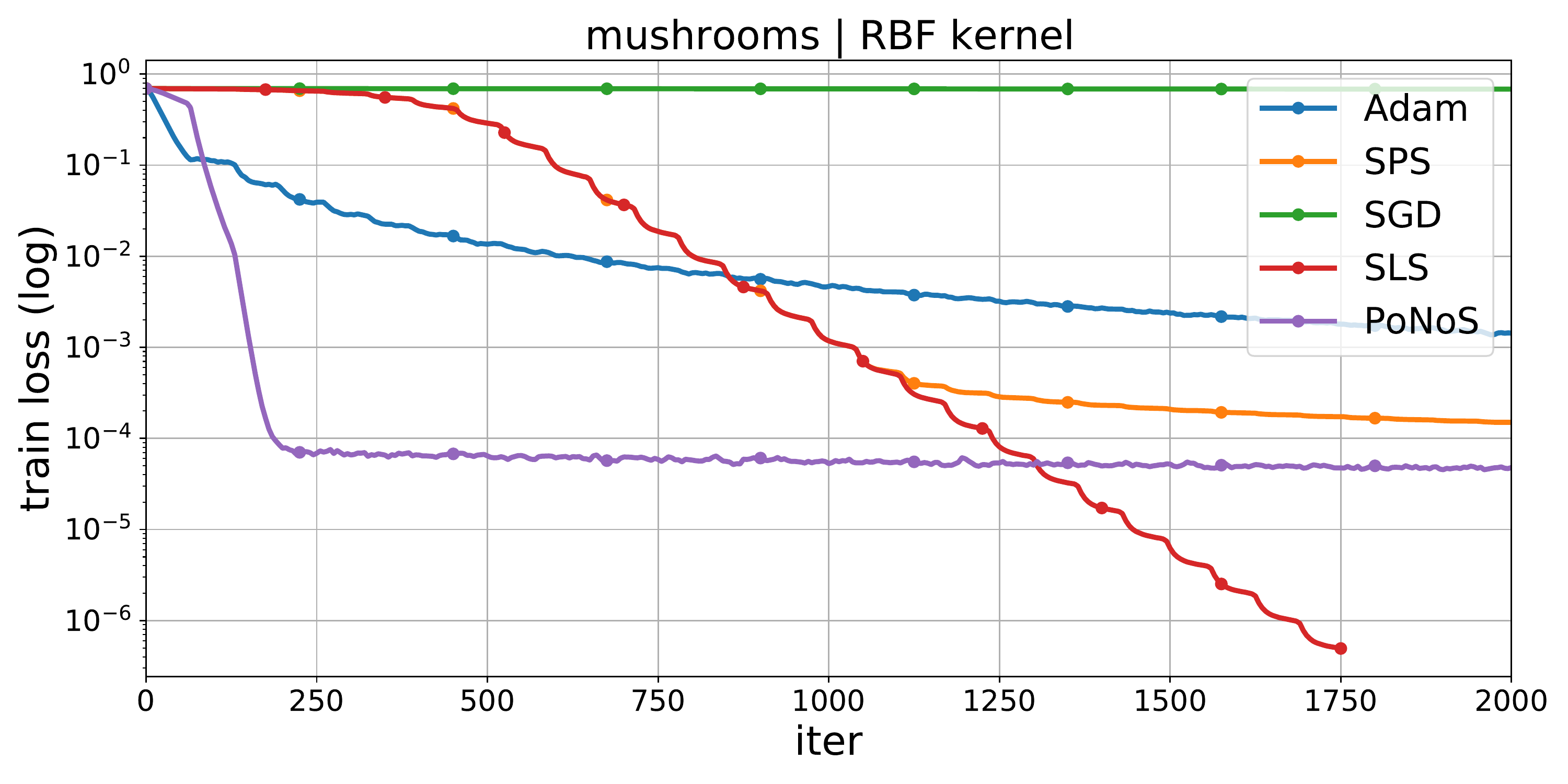}}
		\subfloat{\includegraphics[width=\imgS\linewidth]{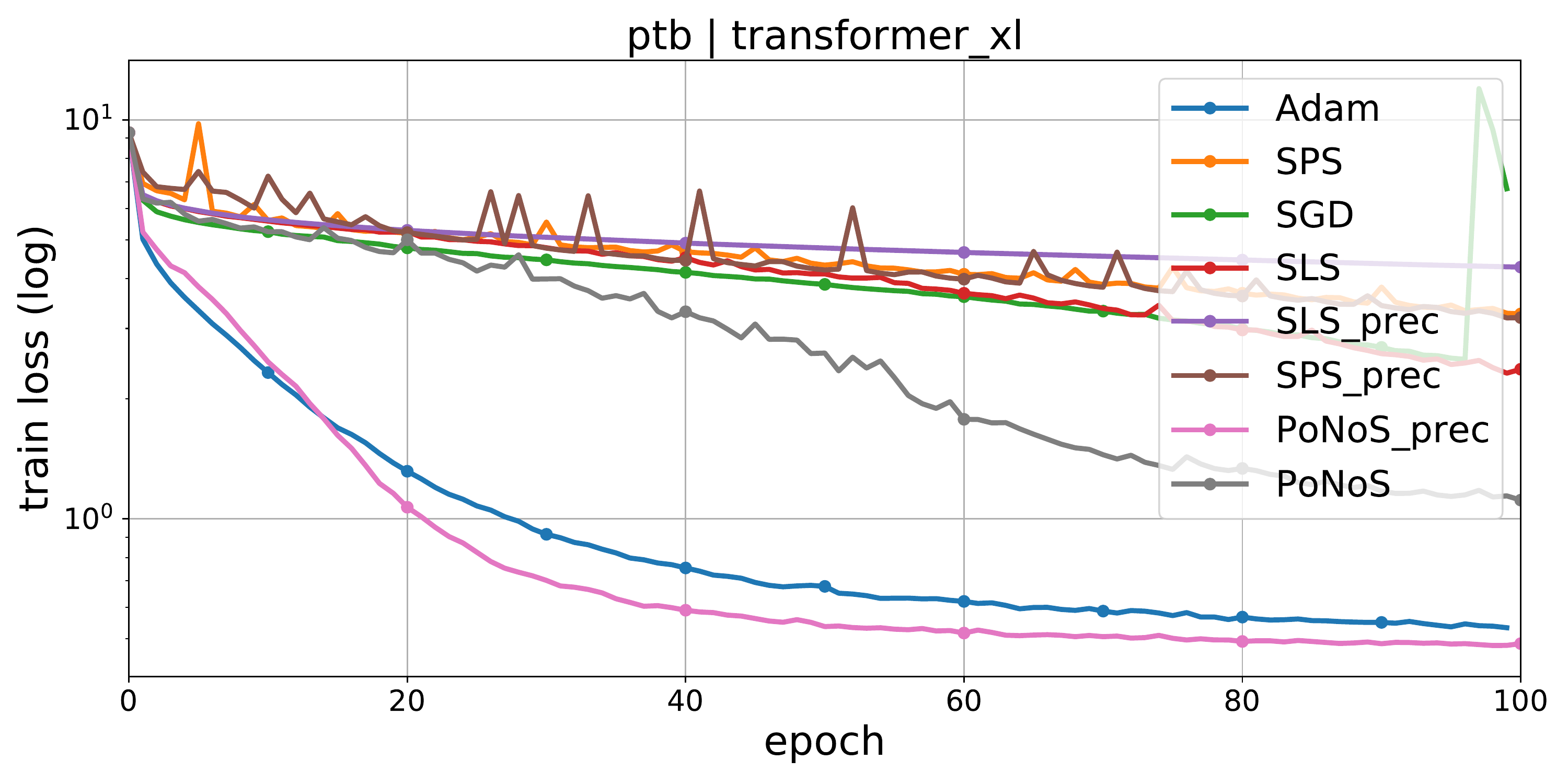}}
		\subfloat{\includegraphics[width=\imgS\linewidth]{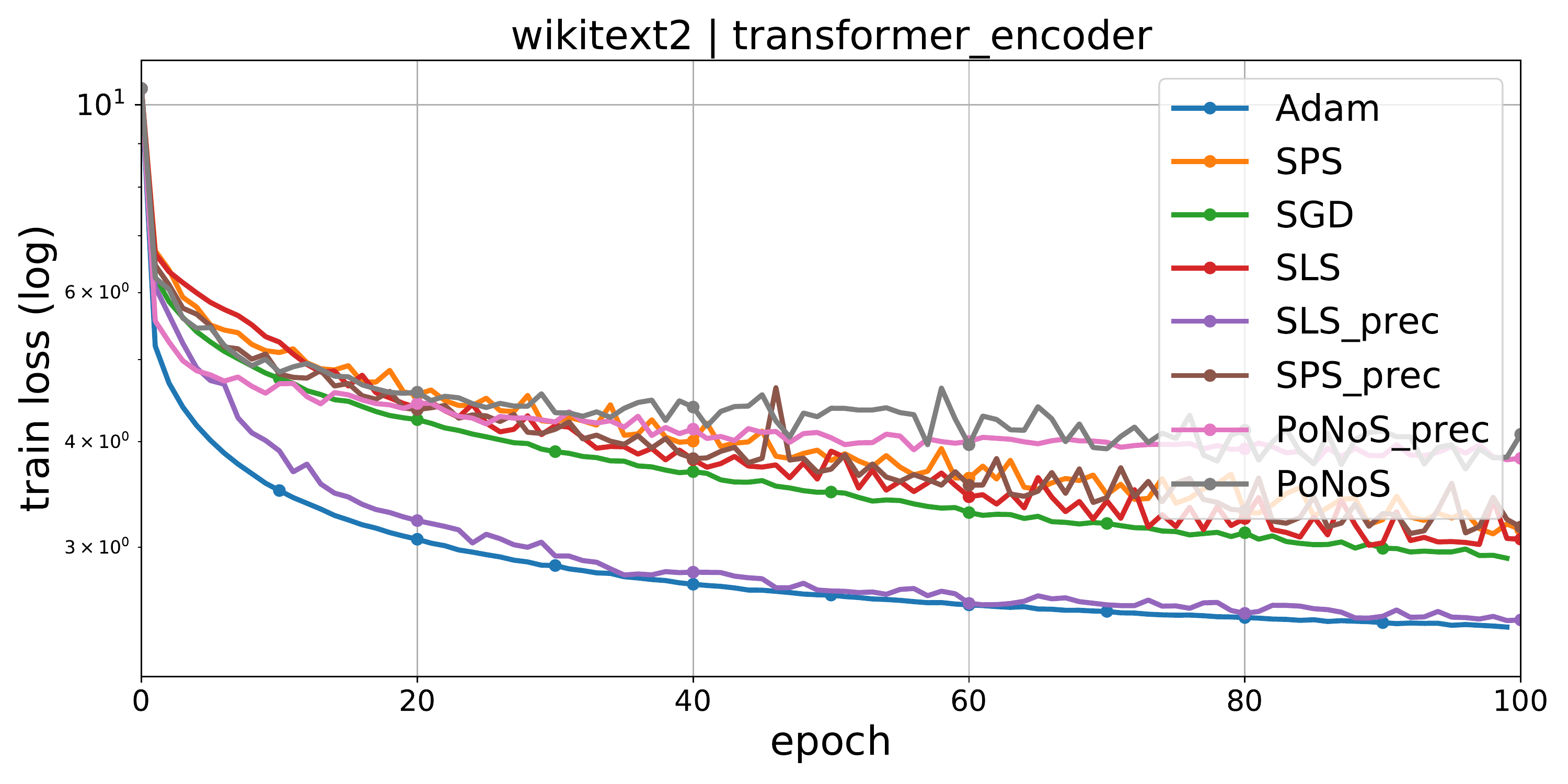}}
		\caption{Train loss comparison between the new method (PoNoS) and the state-of-the-art on convex kernel models (first column) and transformers (last two columns).}\label{fig:convex_trans}
	\end{minipage}
\end{figure}

	\section{Conclusion}\label{sec:conclusion}
	In this work, we showed that modern DL models can be efficiently trained by nonmonotone line search methods. More precisely, nonmonotone techniques have been shown to outperform the monotone line searches existing in the literature. A stochastic Polyak step size with resetting has been employed as the initial step size for the nonmonotone line search, showing that the combined method is faster than the version without line search. Moreover, we presented the first runtime comparison between line-search-based methods and SGD/Adam. The results show that the new line search is overall computationally faster than the state-of-the-art. A new resetting technique is developed to reduce the amount of backtracks to almost zero on average, while still maintaining a large initial step size.
To conclude, the similar behavior of SLS\_prec and Adam on the rightmost subplot of Figure \ref{fig:convex_trans} suggests that other initial step sizes might also be suited for training transformers. We leave such exploration (e.g., a stochastic BB like in \citet{tan16a,liang19a}) to future works.

We proved three convergence rate results for stochastic nonmonotone line search methods under interpolation and under either strong-convexity, convexity, or the PL condition. Our theory matches its monotone counterpart despite the use of a nonmonotone term. In the future, we plan to explore the conditions of the theorems in \citet{liu22a} and to study a bridge between their local PL results and our global PL assumption. To conclude, it is worth mentioning that nonmonotone line search methods might be connected to the edge of stability phenomenon described in \citet{cohen20a}, because of the very similar behavior they induce in the decrease of (deterministic) objective functions. However, a rigorous study remains for future investigations.
	
	\section*{Acknowledgments}
	The work was conducted within the KI-Starter project
	“Robustness and Generalization in Training Deep Neural Networks” funded by the Ministry of Culture and Science Nordrhein-Westfalen, Germany and partially supported by the Canada CIFAR AI Chair Program.
	
	\bibliographystyle{plainnat}
	\bibliography{biblio}   
	\newgeometry{left=1.50cm, right=1.50cm, top=2.00cm, bottom=2.00cm}
	\appendix
	{\huge\textbf{Appendix}}
	\section*{Content}
	\addcontentsline{toc}{section}{\protect\numberline{}Appendix Content}
	\begin{itemize}
		\item[] \hyperref[sec:supp_alg]{\textbf{Appendix A} The Algorithm}\\
		\item[] \hyperref[sec:supp_rates]{\textbf{Appendix B} Convergence Rates}
		\begin{itemize}
			\item[] \hyperref[sec:supp_strongly_convex]{\textbf{Appendix B.1} Rate of Convergence for Strongly Convex Functions}
			\item[] \hyperref[sec:supp_convex]{\textbf{Appendix B.2} Rate of Convergence for Convex Functions}
			\item[] \hyperref[sec:supp_pl]{\textbf{Appendix B.3} Rate of Convergence for Functions Satisfying the PL Condition}
			\item[] \hyperref[sec:supp_lemmas]{\textbf{Appendix B.4} Common Lemmas}
			\item[] \hyperref[sec:supp_polyak]{\textbf{Appendix B.5} The Polyak Step Size is Bounded}\\
		\end{itemize}
		\item[] \hyperref[sec:supp_rates]{\textbf{Appendix C} Experimental details}\\
		\item[] \hyperref[sec:supp_plots]{\textbf{Appendix D} Plots Completing the Figures in the Main Paper}
		\begin{itemize}
			\item[] \hyperref[sec:supp_exp1]{\textbf{Appendix D.1} Comparison between PoNoS and the state-of-the-art}
			\item[] \hyperref[sec:supp_reset]{\textbf{Appendix D.2} A New Resetting Technique}
			\item[] \hyperref[sec:supp_time]{\textbf{Appendix D.3} Time Comparison}
			\item[] \hyperref[sec:supp_convex_short]{\textbf{Appendix D.4} Experiments on Convex Losses}
			\item[] \hyperref[sec:supp_trans]{\textbf{Appendix D.5} Experiments on Transformers}\\
		\end{itemize}
		\item[] \hyperref[sec:supp_plots]{\textbf{Appendix E} Additional Plots}
		\begin{itemize}
			\item[] \hyperref[sec:supp_c]{\textbf{Appendix E.1} Study on the Choice of $c$: Theory (0.5) vs Practice (0.1)}
			\item[] \hyperref[sec:supp_line_search]{\textbf{Appendix E.2} Study on the Line Search Choice: Various Nonmonotone Adaptations}
			\item[] \hyperref[sec:supp_f_eval]{\textbf{Appendix E.3} Zoom in on the Amount of Backtracks}
			\item[] \hyperref[sec:supp_etamax]{\textbf{Appendix E.4} Study on the Choice of $\etamax$}
			\item[] \hyperref[sec:supp_c_p]{\textbf{Appendix E.5} Study on the Choice of $c_p$: Doubling the Legacy Value}
			\item[] \hyperref[sec:profiling]{\textbf{Appendix E.6} Profiling PoNoS}\\[25\baselineskip]
		\end{itemize}
	\end{itemize}
	\section{The Algorithm}\label{sec:supp_alg}
In this section, we give the details of our proposed algorithm PoNoS.\\ 
Training machine learning models (e.g., neural networks) entails solving the following \textbf{finite sum problem:}
\begin{equation} \label{eq:problem}
	\min_{w\in \mathbb{R}^n} f(w) = \frac{1}{M}\sum_{i=1}^{M} f_i(w), 
\end{equation}
where $w$ is the parameter vector and $f_i$ corresponds to a single instance of the $M$ points in the training set.\\
Given an initial step size $\eta_{k,0}$ and $\delta\in(0,1)$, the \textbf{Stochastic (Amijo) Line Search (SLS)} \citep{vaswani19a} select the smallest $l_k\in\mathbb{N}$ such that $\eta_k = \eta_{k,0} \delta^{l_k}$ satisfies the following condition: 
\begin{equation}\label{eq:armijo}
	\fik(w_k - \eta_k \gradik(w_k)) \leq \fik(w_k) -c \eta_k \| \gradik(w_k) \|^2,
\end{equation}
where $c\in(0,1)$ and $\|\cdot\|$ is the Euclidean norm.\\
The newly proposed \textbf{Stochastic Zhang \& Hager line search} adapted from \citet{zhang04a} is 
\begin{equation}\label{eq:zhang}
	\begin{split}
		&\fik(w_k - \eta_k \gradik(w_k)) \leq C_k -c \eta_k \| \gradik(w_k) \|^2,\\
		C_k=\max &\left\{\tilde{C}_k; \fik(w_k) \right\}, \; \tilde{C}_k = \frac{\xi Q_k C_{k-1} + \fik(w_k)}{Q_{k+1}}, \; Q_{k+1} = \xi Q_k + 1,
	\end{split}
\end{equation}
where $\xi\in[0,1]$, $C_0=Q_0=0$ and $C_{-1} = f_{i_0}(w_0)$.\\
Given $\gamma>1$, the \textbf{resetting/smoothing technique} employed in \citet{vaswani19a} is
\begin{equation}\label{eq:etamax}
	\eta_{k,0} = \eta_{k-1}  \gamma^{b/M},
\end{equation}
where $\gamma>1$ and $b$ is the mini-batch size.\\
Given the step size
\begin{equation}\label{eq:polyak}
	\tilde{\eta}_{k,0}:=\frac{\fik(w_k) - \fik^*}{c_p||\gradik(w_k)||^2} \quad \text{and } c_p\in(0,1),
\end{equation}
we recall \textbf{Stochastic Polyak Step (SPS) size} from \citet{loizou21a}
\begin{equation}\label{eq:loizou}
	\eta_{k,0} = \min \left\{ \tilde{\eta}_{k,0}, \eta_{k-1} \gamma^{b/M}, \etamax \right\} \quad \text{ with $\etamax>0$ and $\tilde{\eta}_{k,0}$ defined in \eqref{eq:polyak},}  
\end{equation}
and its \textbf{non-smoothed version} (used in Algorithm \ref{alg})
\begin{equation}\label{eq:upolyak}
	\eta_{k,0} = \min \left\{ \tilde{\eta}_{k,0}, \etamax \right\} \quad \text{ with $\tilde{\eta}_{k,0}$ defined in \eqref{eq:polyak}.}  
\end{equation}
To employ our new \textbf{resetting technique}, we redefine $\eta_k$ as
\begin{equation}\label{eq:new_etak}
	\eta_k =   \eta_{k,0} \delta^{\bar{l}_k} \delta^{l_k},\quad \text { with } \bar{l}_{k}:= \max \{\bar{l}_{k-1} + l_{k-1} -1, 0\}.
\end{equation}

\begin{algorithm}[H]\label{alg}
	\caption{The POlyak NOnmonotone Stochastic (PoNoS) line search method}
	
	\KwIn{$D=\{(x_i,y_i)\}_{i=1}^M, w_0 \in \R^n, \etamax >0$, $c\in(0,1)$, $c_p\in(0,1), \delta\in(0,1), \xi\in[0,1]$, $b$ mini-batch size, $Q_0=0$, $k=0$}
	
	\For{$epoch = 0, 1, 2, \dots, max\_epoch$}{
		
		\For{$i = 0, 1, 2, \dots, \frac{M}{b}$}{
			
			sample $i_k \subset \{1, \dots, M\}: |i_k| = b$
			
			$\eta_{k,0} = \eqref{eq:upolyak}$
			
			$l_k=0$
			
			\If{k=0}{$C_{-1}=f_{i_0}(w_0)$}
			
			$\tilde{C}_k = \frac{\xi Q_k C_{k-1} + \fik(w_k)}{\xi Q_k + 1}$
			
			$C_k = \max \left\{\tilde{C}_k; \fik(w_k) \right\}$ 
			
			\While{$\fik(w_k - \eta_k \gradik(w_k)) > C_k -c \eta_k \| \gradik(w_k) \|^2$}{
				
				$l_k = l_k+1$
				
				$\eta_k =   \eta_{k,0} \delta^{\bar{l}_k} \delta^{l_k}$ \label{step:etak}
			}
			
			$w_{k+1} = w_k - \eta_k \gradik(w_k)$ \label{step:wk}
			
			$\bar{l}_{k+1}:= \max \{\bar{l}_k + l_k -1, 0\}$
			
			$Q_{k+1} = \xi Q_k + 1$
			
			$k = k+1$
		}
	}
\end{algorithm}
	\section{Convergence Rates}\label{sec:supp_rates}
	
Our results do not prove convergence only for PoNoS, but more in general for methods employing \eqref{eq:zhang} as a line search and a bounded initial step size, i.e., 
\begin{equation}\label{eq:generic_eta}
\eta_{k,0} \in [\etaminn, \etamax] , \quad \text{ with } \etamax>\bar{\eta}^{\text{min}}> 0.
\end{equation}
The next lemma provides the possible range of $\eta_k$ as a consequence of the line search technique. Thanks to the fact that we always have $C_k \geq \fik(w_k)$, Lemma \ref{lemma:lbeta} recovers the monotone range (see Lemma 1 in \citet{vaswani19a}). We say that $f$ is $L$-Lipschitz smooth when $f$ is continuously differentiable with Lipschitz continuous gradient, i.e.  
\begin{equation*}
	\| \grad(x) -\grad(y) \| \leq L \| x -y \| \quad \forall x,y \in \mathbb{R}^n, \quad \text{ with } L>0.
\end{equation*}
\begin{lemma}\label{lemma:lbeta} 
	Let $\fik$ be $\Lik$-Liptschitz smooth. The range of the step size $\eta_k$ returned by \eqref{eq:zhang} and with $\eta_{k,0}$ defined in \eqref{eq:generic_eta} is either
	\begin{equation}\label{eq:eta_k_lb}
		\eta_k \in 
		\begin{cases}
		[\etaminn , \etamax] \qquad &\text{if } l_k=0,\\
		[\etamin, \etamax]  \qquad &\text{if } l_k>0,
		\end{cases}
	\end{equation}
	where $\etamin :=\min\left\{ \frac{2\delta(1-c)}{\Lmax},\etaminn \right\}$ 
	and $\Lmax= \max_{i} L_i$. 
\end{lemma}
\begin{proof}
Let us denote $g_k:= \gradik(w_k)$. Applying Lemma \ref{lemma:LC1_bound} below on $\fik$, with $y=w_k -\eta_k g_k$ and $x=w_k$ we have
\begin{equation*}
	\begin{split}
		\fik (w_k -\eta_k g_k ) & \leq \fik(w_k) + g_k^T(w_k -\eta_k g_k - w_k) + \frac{\eta_k^2\Lik}{2} \| g_k\|^2\\
		& = \fik(w_k) - \left(\eta_k - \frac{\eta_k^2 \Lik}{2} \right) \|g_k\|^2,\\
	\end{split}
\end{equation*}
which can be rewritten as 
\begin{equation}\label{eq:p_eta}
	\fik (w_k -\eta_k g_k ) \leq p_k(\eta_k), \quad \text{with } p_k(\eta):=  \fik(w_k) - \left(\eta - \frac{\eta^2 \Lik}{2} \right) \|g_k\|^2.
\end{equation}
Note that \eqref{eq:p_eta} is valid for any $\eta$.
Let us rewrite \eqref{eq:zhang} as 
\begin{equation*}
	\fik(w_k - \eta_k g_k) \leq q_k(\eta_k), \quad \text{with } q_k(\eta) :=  C_k -c \eta \| g_k \|^2.
\end{equation*}
Now, the backtracking procedure in \eqref{eq:zhang} admits two possible output:\\
Case 1: $l_k=0$. In this case, we have $\eta_k = \eta_{k,0}$ and thus directly $\eta_k \in [\etaminn , \etamax]$.\\
Case 2: $l_k>0$. In this case, we have $\eta_k < \eta_{k,0}$ with $\fik(w_k - \frac{\eta_k}{\delta} g_k) > q_k(\frac{\eta_k}{\delta})$. Then, we have that $q_k(\frac{\eta_k}{\delta})\leq p_k(\frac{\eta_k}{\delta})$ because $q_k(\frac{\eta_k}{\delta})> p_k(\frac{\eta_k}{\delta})$ would lead to a contradiction. In fact
\begin{equation*}
	\fik \left(w_k - \frac{\eta_k}{\delta} g_k \right) > q_k\left(\frac{\eta_k}{\delta}\right) >  p_k\left(\frac{\eta_k}{\delta}\right)\geq \fik \left(w_k - \frac{\eta_k}{\delta} g_k\right)
\end{equation*}
is false. Thus, it has to be $q_k(\frac{\eta_k}{\delta})\leq p_k(\frac{\eta_k}{\delta})$, from which we get that 
\begin{equation*}
	\fik(w_k) -c \frac{\eta_k}{\delta} \| g_k \|^2 \leq C_k -c \frac{\eta_k}{\delta} \| g_k \|^2 \leq \fik(w_k) - \left(\frac{\eta_k}{\delta} - \frac{\eta_k^2 \Lik}{2\delta^2} \right) \|g_k\|^2
\end{equation*}
and consequently 
\begin{equation*}
	-c \leq - \left(1 - \frac{\eta_k \Lik}{2\delta} \right) \Leftrightarrow \eta_k \geq \frac{2\delta(1-c)}{\Lik},
\end{equation*}
which leads to \eqref{eq:eta_k_lb}.
\end{proof}

One of the challenges of convergence theorems for nonmonotone line search methods is to prove that the sequence of the nonmonotone terms $\{C_k\}$ converges to $f(w^*)$. In order to achieve this, in Lemma \ref{lemma:induction_short} below we prove that $C_k$ and $C_{k-1}$ are lower-bounded by $\fik(w^*)$. Before that, we establish the following auxiliary result.
\begin{lemma}\label{lemma:qk_useful}
From the definition of $Q_k$ in \eqref{eq:zhang}, it follows
\begin{equation}\label{eq:qk}
	1\leq \xi Q_k + 1\leq \frac{1}{1-\xi}.
\end{equation}
and
\begin{equation}\label{eq:frac_qk_ub1}
	\frac{\xi Q_k}{\xi Q_k +1} = \leq  \xi
\end{equation}
\end{lemma}
\begin{proof}
From the definition of $Q_{k+1}$ we have 
\begin{equation*}
1\leq \xi Q_k + 1=:Q_{k+1} = 1+ \sum_{j=0}^{k}\xi^{j+1} \leq \sum_{j=0}^{\infty} \xi^j = \frac{1}{1-\xi}.
\end{equation*}
which implies \eqref{eq:qk}. Thus we have 
\begin{equation*}
\frac{\xi Q_k}{\xi Q_k +1} = \frac{\xi Q_k + 1 -1}{\xi Q_k +1} = 1 - \frac{1}{\xi Q_k +1}\leq 1- \frac{1}{\frac{1}{1-\xi}} = \xi.
\end{equation*}
which implies \eqref{eq:frac_qk_ub1} and concludes the proof.
\end{proof}
\noindent We say that $f$ satisfies interpolations if given $w^* \in \displaystyle\argmin_{w\in \R^n} f(w),$ then $w^* \in \displaystyle\argmin_{w\in \R^n} f_i(w)\;\forall 1\leq i\leq M$. 
\begin{lemma}\label{lemma:induction_short}
	Let $C_k$ be defined in \eqref{eq:zhang}. Assuming interpolation, the following bounds hold for all $k\in \mathbb{N}$,
	\begin{equation}\label{eq:induction1}
		C_k - \fik(w^*) \geq 0 \quad\forall k
	\end{equation}
	and
	\begin{equation}\label{eq:induction2}
		C_{k-1} - \fik(w^*) \geq 0 \quad\forall k. 
	\end{equation}
\end{lemma}
\begin{proof}
We will prove both statements by induction, starting with \eqref{eq:induction1}. For $k=0$, interpolation yields $f_{i_0} (w_0) \geq f_{i_0} (w^*)$. Assuming now that the statement is valid for $k-1\in \mathbb{N}_0$ (i.e., $C_{k-1} - f_{k-1}(w^*) \geq 0$), let us prove that it is valid also for $k$. If $\tilde{C}_k> \fik(w_k)$, we have
\begin{equation*}
	C_k - \fik(w^*) = \frac{\xi Q_k }{\xi Q_k +1} C_{k-1} + \frac{1}{\xi Q_k +1}\fik(w_k)- \fik(w^*)\geq \frac{\xi Q_k }{\xi Q_k +1} f_{i_{k-1}}(w^*) + \frac{1}{\xi Q_k +1}\fik(w^*)- \fik(w^*) = 0,
\end{equation*} 
where we used the induction hypothesis and the fact that $f_{i_{k-1}}(w^*) = \fik(w^*)$, thanks to interpolation. If $\tilde{C}_k \leq \fik(w_k)$, from interpolation we have
\begin{equation*}
C_k - \fik(w^*) = \fik(w_k) - \fik(w^*) \geq 0.
\end{equation*} 
The last two inequalities prove \eqref{eq:induction1}.
We can follow a similar path for \eqref{eq:induction2}. For $k=0$, from the definition of $C_{-1}$, we have $C_{-1} - f_{i_0} (w^*) = f_{i_0} (w_0)- f_{i_0} (w^*) \geq 0$. Assume now that the statement is valid for $k-1\in\mathbb{N}_0$ (i.e., $C_{k-2} - f_{k-1}(w^*) \geq 0$). Then, if $\tilde{C}_{k-1} > f_{i_{k-1}}(w_{k-1})$, we have 
\begin{equation*}
	\begin{split}
		C_{k-1} - \fik(w^*) &= 
		\frac{\xi Q_{k-1} }{\xi Q_{k-1} +1} C_{k-2} + \frac{1}{\xi Q_{k-1} +1}f_{i_{k-1}}(w_k)- \fik(w^*) \\
		&\geq \frac{\xi Q_{k-1}}{\xi Q_{k-1} +1} f_{i_{k-1}}(w^*) + \frac{1}{\xi Q_{k-1} +1}f_{i_{k-1}}(w^*)- \fik(w^*) = 0,
	\end{split}
\end{equation*} 
where we used again the induction hypothesis and the fact that $f_{i_{k-1}}(w^*) = \fik(w^*)$, thanks to interpolation. If $\tilde{C}_{k-1} \leq f_{i_{k-1}}(w_{k-1})$, from interpolation we have
\begin{equation*}
C_{k-1} - \fik(w^*) = f_{i_{k-1}}(w_k) - f_{i_{k-1}}(w^*) \geq 0,
\end{equation*} 
which concludes the proof.
\end{proof}
\noindent The following Lemma shows the importance of the interpolation property. A similar result can be obtained by replacing the $\Lik$-smoothness assumption with the line search condition \eqref{eq:armijo} or \eqref{eq:zhang} (see the proof of Theorem \ref{thm:strongly_convex} below).
\begin{lemma}
	We assume interpolation and that $\fik$ are $\Lik$-smooth. Then, we obtain
	\begin{equation*}
		\Eik\| \gradik (w_k)\|^2 \leq \Lmax \left(f(w_k)-f(w^*)\right),
	\end{equation*}
where $\Lmax= \max_{i} L_i$.
\end{lemma}
\begin{proof}
Let $w^* \in \displaystyle\argmin_{w\in \R^n} f(w)$ and $w_{i_k}^* \in \displaystyle\argmin_{w\in \R^n} \fik(w)$. From interpolation we have that $w^* \in \displaystyle\argmin_{w\in \R^n} \fik(w)$, which means that $\fik(w^*)= \fik(w_{i_k}^*)$. Thus, interpolation and $\Lik$-smoothness of $\fik$ brings to
\begin{equation*}
	\| \gradik (w_k)\|^2 \leq \Lik \left(\fik(w_k)-\fik(w_{i_k}^*)\right) =  \Lik \left(\fik(w_k)-\fik(w^*)\right).
\end{equation*}
Now, by applying the conditional expectation $\Eik$ on the above inequality, we obtain
\begin{equation*}
	\Eik\| \gradik (w_k)\|^2 \leq \Lmax \left(f(w_k)-f(w^*)\right).
\end{equation*}
\end{proof}

\subsection{Rate of Convergence for Strongly Convex Functions}\label{sec:supp_strongly_convex}
In this subsection, we prove a linear rate of convergence in the case of a strongly convex function f.
\begin{theorem}\label{thm:strongly_convex}
	Let $C_k$ and $\eta_k$ be defined in \eqref{eq:zhang}, with $\eta_{k,0}$ defined in \eqref{eq:generic_eta}. We assume interpolation, $\fik$ convex, $f$ $\mu$-strongly convex and $\fik$ $\Lik$-Lipschitz smooth. If $c> \frac{1}{2}$ and $\xi< \frac{1}{\left(1 + \frac{\etamax}{\etamin\left(2c - 1\right)}\right)}$, we have
	\begin{equation*}
		\begin{split}
			\E \left[\| w_{k+1} - w^* \|^2 + a ( C_k - \fofwstar) \right] & \leq d^{k} \left(\| w_0 - w^* \|^2 + a \left(f(w_0) - \fofwstar\right)\right),
		\end{split}
	\end{equation*}
	where $d:=\maximum{(1-\etamin\mu)}{b}\in(0,1)$, $b:=\left(1+\frac{\etamax}{ac}\right)\xi\in(0,1)$, $a:= \etamin\left(2-\frac{1}{c}\right)>0$ with $\etamin$ as defined in Lemma \ref{lemma:lbeta}.
\end{theorem}
\begin{proof}
From \eqref{eq:zhang} and interpolation we obtain
\begin{equation}\label{eq:initial_strongly_convex}
	\begin{split}
		\| w_{k+1} - w^* \|^2 &= \| w_k -\eta_k \gradik (w_k) - w^* \|^2\\
		& = \| w_k - w^* \|^2 - 2\eta_k \langle \gradik (w_k), w_k -w^* \rangle + \eta_k^2 \| \gradik (w_k) \|^2 \\
		&\leq \| w_k - w^* \|^2 - 2\eta_k \langle \gradik (w_k), w_k -w^* \rangle + \frac{\eta_k}{c} \left( C_k - \fik(w_{k+1}) \right)\\
		&\leq \| w_k - w^* \|^2 - 2\eta_k \langle \gradik (w_k), w_k -w^* \rangle + \frac{\eta_k}{c} \left( C_k - \fikofwstar \right).\\
	\end{split}
\end{equation}
Let us distinguish 2 cases: either 1) $ \fik(w_k) \geq \tilde{C}_k$ and then $C_k= \fik(w_k)$, or 2) $ \fik(w_k) < \tilde{C}_k$  and then $C_k= \tilde{C}_k$. Let us first analyze case 1). With $c> \frac{1}{2}$ and $a= \etamin\left(2 - \frac{1}{c}\right)$, from \eqref{eq:initial_strongly_convex} we have 
\begin{equation*}
\begin{split}
	\| w_{k+1} - w^* \|^2 + a (C_k - \fikofwstar) &\leq \| w_k - w^* \|^2 - 2\eta_k \langle \gradik (w_k), w_k -w^* \rangle + \left(a + \frac{\eta_k}{c}\right) \left( C_k - \fikofwstar \right)\\
	&= \| w_k - w^* \|^2 - 2\eta_k \langle \gradik (w_k), w_k -w^* \rangle  + \frac{\eta_k}{c} \left( \fik(w_k) - \fikofwstar \right)\\
	&\quad + a \left( \fik(w_k) - \fikofwstar \right)\\
	&\leq  \| w_k - w^* \|^2 + \etamin \left[-2\langle \gradik (w_k), w_k -w^* \rangle + \frac{1}{c}\left(\fik(w_k) - \fikofwstar \right)\right]\\
	&\quad + a \left( \fik(w_k) - \fikofwstar \right)\\
	&=  \| w_k - w^* \|^2 + \etamin \left[-2\langle \gradik (w_k), w_k -w^* \rangle + \left(\frac{a}{\etamin} + \frac{1}{c}\right)\left(\fik(w_k) - \fikofwstar \right)\right]\\
	&\leq  \| w_k - w^* \|^2 + 2\etamin \left[-\langle \gradik (w_k), w_k -w^* \rangle + \left(\fik(w_k) - \fikofwstar \right)\right]\\
	&\leq  \| w_k - w^* \|^2 + 2\etamin \left[-\langle \gradik (w_k), w_k -w^* \rangle + \left(\fik(w_k) - \fikofwstar \right)\right]\\
	&\quad + a b (C_{k-1} - \fikofwstar),
\end{split}
\end{equation*}
	where the second inequality follows from $c> \frac{1}{2}$ and from the fact that the term between square brackets is negative, since $\fik$ is a convex function. The third inequality follows from the definition of $a$ and the last inequality follows from \eqref{eq:induction2}. In the following we are going to show that the same bound can also be achieved in case 2).\\
\noindent Let us now analyze case 2). From \eqref{eq:initial_strongly_convex}, \eqref{eq:qk} and \eqref{eq:frac_qk_ub1}, and again from $c> \frac{1}{2}$, convexity of $\fik$ and the definition of $a$, we have 
\begin{equation*}
	\begin{split}
		\| w_{k+1} - w^* \|^2 + a (C_k - \fikofwstar) &\leq \| w_k - w^* \|^2 - 2\eta_k \langle \gradik (w_k), w_k -w^* \rangle + \left(a + \frac{\eta_k}{c}\right) \left( \tilde{C}_k - \fikofwstar \right)\\
		&= \| w_k - w^* \|^2 - 2\eta_k \langle \gradik (w_k), w_k -w^* \rangle + \left(a + \frac{\eta_k}{c}\right) \frac{1}{\xi Q_k +1}\left(\fik(w_k) - \fikofwstar \right)\\
		&\quad +\left(a + \frac{\eta_k}{c}\right) \frac{\xi Q_k}{\xi Q_k +1} \left( C_{k-1} - \fikofwstar \right)\\
		&\leq \| w_k - w^* \|^2 - 2\eta_k \langle \gradik (w_k), w_k -w^* \rangle + 2\etamin \left(\fik(w_k) - \fikofwstar \right)\\
		&\quad + \left(a + \frac{\eta_k}{c}\right) \xi \left( C_{k-1} - \fikofwstar \right)\\
		&\leq \| w_k - w^* \|^2 + 2\etamin \left[-\langle \gradik (w_k), w_k -w^* \rangle + \left(\fik(w_k) - \fikofwstar \right)\right]\\
		&\quad + \left(a + \frac{\eta_k}{c}\right) \xi \left( C_{k-1} - \fikofwstar \right)\\
		&\leq \| w_k - w^* \|^2 + 2\etamin \left[-\langle \gradik (w_k), w_k -w^* \rangle + \left(\fik(w_k) - \fikofwstar \right)\right]\\
		&\quad + a\left(1 + \frac{\etamax}{ac}\right) \xi \left( C_{k-1} - \fikofwstar \right),\\
	\end{split}
\end{equation*}
where the fourth inequality follows from \eqref{eq:eta_k_lb}. By defining $b = \left(1 + \frac{\etamax}{ac}\right) \xi$ we can conclude that the same bound holds in both cases 1) and 2). Let us now show that $b<1$. Under the assumption that $c>\frac{1}{2}$ and $\xi< \frac{1}{\left(1 + \frac{\etamax}{\etamin\left(2c - 1\right)}\right)}$ we have
\begin{equation*}
	\begin{split}
		b=\left(1 + \frac{\etamax}{ac}\right) \xi&= \left(1 + \frac{\etamax}{\etamin\left(2 - \frac{1}{c}\right)c}\right) \xi = \left(1 + \frac{\etamax}{\etamin\left(2c - 1\right)}\right) \xi < 1.
	\end{split}
\end{equation*}
Now, by taking expectation w.r.t. $i_k$ on the common bound achieved in both cases 1) and 2), using \\$\Eik [\gradik(w_k)] = \grad(w_k)$, and applying strong convexity of $f$ we obtain
\begin{equation*}
	\begin{split}
		\Eik \left[\| w_{k+1} - w^* \|^2\right] + a (\Eik\!\left[C_k\right] - \fofwstar ) & \leq \| w_k - w^* \|^2 +  2 \etamin \left[-\langle \grad (w_k), w_k -w^* \rangle + f(w_k) - \fofwstar \right]\\
		& \quad + a b \left( C_{k-1} - \fofwstar \right)\\
		&\leq \| w_k - w^* \|^2 - 2 \etamin\frac{\mu}{2} \| w_k - w^* \|^2 + a b  \left(C_{k-1} - \fofwstar\right)\\
		&= (1- \etamin\mu) \| w_k - w^* \|^2 + a b \left(C_{k-1} - \fofwstar\right)\\
		&\leq  d \left( \| w_k - w^* \|^2 + a \left(C_{k-1} - \fofwstar\right)\right),
	\end{split}
\end{equation*}
where $d: = \maximum{(1- \etamin\mu)}{b}$ and in the first inequality we used the fact that $C_{k-1}$ does not depend on $i_k$. Taking the total expectation gives
\begin{equation*}
	\begin{split}
		\E \left[\| w_{k+1} - w^* \|^2 + a ( C_k - \fofwstar) \right] & \leq d \E \left[ \left( \| w_k - w^* \|^2 + a \left(C_{k-1} - \fofwstar\right)\right) \right].
	\end{split}
\end{equation*}
At this point we can use the above inequality recursively, resulting in
\begin{equation*}
\begin{split}
\E \left[\| w_{k+1} - w^* \|^2 + a ( C_k - \fofwstar) \right] & \leq d^{k} \E \left[ \| w_0 - w^* \|^2 + a \left(C_{-1} - \fofwstar\right)\right]\\
& = d^{k} \left(\| w_0 - w^* \|^2 + a \left(f(w_0) - \fofwstar\right)\right),
\end{split}
\end{equation*}
where in the last equality we have used that $C_{-1} = f_{i_0}(w_0)$.
\end{proof}

\subsection{Rate of Convergence for Convex Functions}\label{sec:supp_convex}
In this subsection, we prove a $O(\frac{1}{k})$ rate of convergence in the case of a convex function.
\begin{theorem}\label{thm:convex}
	Let $C_k$ and $\eta_k$ be defined in \eqref{eq:zhang}, with $\eta_{k,0}$ defined in \eqref{eq:generic_eta}. We assume interpolation, $f$ convex and $\fik$ $\Lik$-Lipschitz smooth. Given a constant $a_1$ such that $0 < a_1 < \left(2 - \frac{1}{c}\right)$, if $c> \frac{1}{2}$ and $\xi< \frac{a_1}{2}$, we have
	\begin{equation*}
		\begin{split}
			\E \left[f(\bar{w}_k) - \fofwstar\right] \leq \frac{d_1}{k}\left(\frac{1}{\etamin}\| w_0 - w^* \|^2 + a_1\left(f(w_0) -f(w^*) \right)\right),\\
		\end{split}
	\end{equation*}
	where $\bar{w}_k = \frac{1}{k}\sum_{j=0}^{k} w_j$ and $d_1 := \frac{c}{c(2 - a_1) -1}>0$. 
\end{theorem}
\begin{proof}
	From \eqref{eq:zhang} and interpolation we obtain
	\begin{equation}\label{eq:initial_convex}
		\begin{split}
			\| w_{k+1} - w^* \|^2 &= \| w_k -\eta_k \gradik (w_k) - w^* \|^2\\
			& = \| w_k - w^* \|^2 - 2\eta_k \langle \gradik (w_k), w_k -w^* \rangle + \eta_k^2 \| \gradik (w_k) \|^2 \\
			&\leq \| w_k - w^* \|^2 - 2\eta_k \langle \gradik (w_k), w_k -w^* \rangle + \frac{\eta_k}{c} \left( C_k - \fik(w_{k+1}) \right)\\
			&\leq \| w_k - w^* \|^2 - 2\eta_k \langle \gradik (w_k), w_k -w^* \rangle + \frac{\eta_k}{c} \left( C_k - \fikofwstar \right).\\
		\end{split}
	\end{equation}

\noindent Let us distinguish 2 cases: either 1) $ \fik(w_k) \geq \tilde{C}_k$ and then $C_k= \fik(w_k)$, or 2) $ \fik(w_k) < \tilde{C}_k$  and then $C_k= \tilde{C}_k$. Let us first analyze case 1). From \eqref{eq:initial_convex} we have 
\begin{equation*}
	\begin{split}
		\| w_{k+1} - w^* \|^2 + \eta_k a_1 (C_k - \fikofwstar) &\leq \| w_k - w^* \|^2 - 2\eta_k \langle \gradik (w_k), w_k -w^* \rangle + \left(\eta_k a_1 + \frac{\eta_k}{c}\right) \left( C_k - \fikofwstar \right)\\
		&=  \| w_k - w^* \|^2 + \eta_k \left[-2\langle \gradik (w_k), w_k -w^* \rangle + \left(a_1 + \frac{1}{c}\right)\left(\fik(w_k) - \fikofwstar \right)\right]\\
		&\leq  \| w_k - w^* \|^2 + \eta_k \left[-2\langle \gradik (w_k), w_k -w^* \rangle + \left(a_1 + \frac{1}{c}\right)\left(\fik(w_k) - \fikofwstar \right)\right]\\
		&\quad +\eta_k a_1 b_1 (C_{k-1} - \fikofwstar),\\
	\end{split}
\end{equation*}
where the second inequality follows from \eqref{eq:induction2} and $b_1 := \left(1 + \frac{1}{a_1c}\right) \xi >0$. The above bound will now be proven also for case 2). From \eqref{eq:initial_convex} we have 
\begin{equation*}
	\begin{split}
		\| w_{k+1} - w^* \|^2 + \eta_k a_1 (C_k - \fikofwstar) &\leq \| w_k - w^* \|^2 - 2\eta_k \langle \gradik (w_k), w_k -w^* \rangle + \left(\eta_k a_1 + \frac{\eta_k}{c}\right) \left( \tilde{C}_k - \fikofwstar \right)\\
		&= \| w_k - w^* \|^2 - 2\eta_k \langle \gradik (w_k), w_k -w^* \rangle + \left(\eta_k a_1 + \frac{\eta_k}{c}\right) \frac{1}{\xi Q_k +1}\left(\fik(w_k) - \fikofwstar \right)\\
		&\quad +\left(\eta_k a_1 + \frac{\eta_k}{c}\right) \frac{\xi Q_k}{\xi Q_k +1} \left( C_{k-1} - \fikofwstar \right)\\
		&\leq \| w_k - w^* \|^2 + \eta_k \left[-2\langle \gradik (w_k), w_k -w^* \rangle + \left(a_1 + \frac{1}{c}\right)\left(\fik(w_k) - \fikofwstar \right)\right]\\
		&\quad + \eta_k a_1 \left(1 + \frac{1}{a_1c}\right) \xi \left( C_{k-1} - \fikofwstar \right),\\
	\end{split}
\end{equation*}
where the second inequality follows from \eqref{eq:qk} and \eqref{eq:frac_qk_ub1}. By defining $b_1 := \left(1 + \frac{1}{a_1c}\right) \xi$ we can conclude that the same bound can be achieved in both cases 1) and 2). Let us now show that $b_1<1$. From the definition of $a_1$, we have $\frac{1}{c}<2 -a_1,$ thus, under the assumption $\xi< \frac{a_1}{2}$ it holds that
\begin{equation*}
	\begin{split}
		b_1=\left(1 + \frac{1}{a_1c}\right) \xi< \left(1 + \frac{2-a_1}{a_1}\right) \xi = \frac{2}{a_1} \xi < 1.
	\end{split}
\end{equation*}
From $b_1<1$, rearranging and dividing the bound found for both cases 1) and 2) by $2\eta_k$, we have the following
\begin{equation*}
	\begin{split}
		\langle \gradik (w_k), w_k -w^* \rangle &\leq \frac{1}{2\eta_k} \left[\| w_k - w^* \|^2 - \| w_{k+1} - w^* \|^2\right] + \left(\frac{a_1}{2} + \frac{1}{2c}\right) \left(\fik(w_k) - \fikofwstar\right)\\
		&\quad + \frac{a_1}{2} \left[\left( C_{k-1} - \fikofwstar \right) - \left(C_k - \fikofwstar\right) \right].\\
	\end{split}
\end{equation*}
Now, by taking expectation w.r.t. $i_k$, we obtain
\begin{equation*}
	\begin{split}
		\langle \grad (w_k), w_k -w^* \rangle &\leq \Eik \left[\frac{1}{2\eta_k} \left(\| w_k - w^* \|^2 - \| w_{k+1} - w^* \|^2\right)\right] +\left(\frac{a_1}{2} + \frac{1}{2c}\right) \left(f(w_k) - \fofwstar\right)\\
		&\quad + \frac{a_1}{2} \left( \left( C_{k-1} - \fofwstar \right) - \left(\Eik[C_k] - \fofwstar\right) \right).\\
	\end{split}
\end{equation*}
By convexity of $f$ we have $\left(f(w_k) - \fofwstar\right)\leq \langle \grad (w_k), w_k -w^* \rangle$, thus, 
\begin{equation*}
	\begin{split}
		f(w_k) - \fofwstar &\leq \Eik \left[\frac{1}{2\eta_k} \left(\| w_k - w^* \|^2 - \| w_{k+1} - w^* \|^2\right)\right] +\left(\frac{a_1}{2} + \frac{1}{2c}\right) \left(f(w_k) - \fofwstar\right)\\
		&\quad + \frac{a_1}{2} \left( \left( C_{k-1} - \fofwstar \right) - \left(\Eik[C_k] - \fofwstar\right) \right).\\
	\end{split}
\end{equation*}
Moreover, noticing that $1 - \left(\frac{a_1}{2} + \frac{1}{2c}\right) > 0$ by the definition of $a_1$, we obtain 
\begin{equation*}
	\begin{split}
		f(w_k) - \fofwstar &\leq \Eik \left[\frac{d_1}{\eta_k} \left(\| w_k - w^* \|^2 - \| w_{k+1} - w^* \|^2\right)\right]\\
		&\quad + d_1a_1\left( \left( C_{k-1} - \fofwstar \right) - \left(\Eik[C_k] - \fofwstar\right) \right),\\
	\end{split}
\end{equation*}
where $d_1 := \frac{1}{2}\frac{1}{1 - \left(\frac{a_1}{2} + \frac{1}{2c}\right)}= \frac{c}{c(2 - a_1) -1}$. Taking the total expectation and summing from $k$ to $0$
\begin{equation*}
	\begin{split}
		\E \left[\sum_{j=0}^{k} f(w_j) - \fofwstar\right] &\leq \E \left[\sum_{j=0}^{k} \frac{d_1}{\eta_j} \left(\| w_j - w^* \|^2 - \| w_{j+1} - w^* \|^2\right)\right]\\
		&\quad + d_1a_1 \E\left[\sum_{j=0}^{k} \left( C_{j-1} - f(w^*) \right) - \left(C_j - \fofwstar\right) \right].\\
	\end{split}
\end{equation*}
Recalling that $\bar{w}_k = \frac{1}{k}\sum_{j=0}^{k} w_j$, it follows from Jensen's inequality that 
\begin{equation*}
\E \left[f(\bar{w}_k) - \fofwstar\right] \leq \E \left[\frac{1}{k}\sum_{j=0}^{k}f(w_j) - \fofwstar\right] = \frac{1}{k} \E \left[\sum_{j=0}^{k} (f(w_j) - \fofwstar)\right].
\end{equation*}
Now, putting together the last two inequalities and defining $\Delta_j:=\| w_j - w^* \|^2$ and $\Gamma_j:=C_j - \fofwstar$, we have 
\begin{equation*}
	\begin{split}
		\E \left[f(\bar{w}_k) - \fofwstar\right] &\leq \frac{1}{k} d_1\E \left[\sum_{j=0}^{k} \frac{1}{\eta_j} \left(\Delta_j - \Delta_{j+1}\right)\right] + \frac{1}{k} d_1a_1 \E\left[\sum_{j=0}^{k} \Gamma_{j-1} - \Gamma_j \right]\\
		&\leq \frac{1}{k} \frac{d_1}{\etamin} \E \left[\Delta_0 - \Delta_k\right] + \frac{1}{k} d_1a_1 \E\left[\Gamma_{-1} - \Gamma_k \right]\\
		&\leq \frac{1}{k} \frac{d_1}{\etamin} \Delta_0 + \frac{1}{k} d_1a_1 \Gamma_{-1}=\frac{d_1}{k}\left(\frac{1}{\etamin}\| w_0 - w^* \|^2 + a_1\left(f(w_0) -f(w^*) \right)\right),\\
	\end{split}
\end{equation*}
where the second inequality follows from the fact that $\eta_k\geq \etamin$ and from simplification of the telescopic sum and the last inequality follows from  $\Delta_k,\Gamma_k>0$ (because of \eqref{eq:induction1} in case of $\Gamma_k$).
\end{proof}

\subsection{Rate of Convergence for Functions Satisfying the PL Condition}\label{sec:supp_pl}
In this subsection, we prove a linear convergence rate in the case of $f$ satisfying a PL condition. We say that a function $f:\mathbb{R}^n\to \mathbb{R}$ satisfies the PL condition if there exists $\mu>0$ such that, $\forall w \in \mathbb{R}^n: \| \grad(w)\|^2 \geq 2\mu (f(w)-\fofwstar)$. From $\fik$ being $\Lik$-Lipschitz smooth $\forall i_k$, it follows that $f$ is also Lipschitz smooth. Let us call $L$ the Lipschitz smoothness constant of $f$ and let us note that $L \leq \frac{1}{M}\sum_{i=1}^{M}L_i\leq \Lmax$,
\begin{theorem}\label{thm:pl}
	Let $C_k$ and $\eta_k$ be defined in \eqref{eq:zhang}, with $\eta_{k,0}$ defined in \eqref{eq:generic_eta}. We assume interpolation, the PL condition on $f$ and that $\fik$ are $\Lik$-Lipschitz smooth. Given $0<a_2:=\frac{4\mu c(1-c) -\Lmax}{4 \delta c (1-c)}+\frac{1}{2\etamax}$ and assuming $\frac{2\delta(1-c)}{\Lmax}<\etaminn, \etamax<\frac{2\delta c(1-c)}{\Lmax -4\mu c(1-c)}$, $\frac{\Lmax}{4\mu}<c<1$ and $\xi<\frac{a_2c}{a_2c + \Lmax}$, we have
	\begin{equation*}
		\begin{split}
			\E \left[f(w_{k+1}) - \fofwstar + a_2\etamax (C_k - \fofwstar) \right] 
			&\leq d_2^k \left(1+a_2\etamax\right) \left(f(w_0) - \fofwstar\right)
		\end{split}
	\end{equation*}
	where $d_2\!:=\!\minimum{\nu}{b_2}\!\in\!(0,1),\,\nu\!:=\!\etamax (\frac{\Lmax -4\mu c(1-c)}{2\delta c(1-c)}+a_2)\!\in\!(0,1), \, b_2\!:=\!\left(1\!+\!\frac{\Lmax}{a_2c}\right) \xi\!\in\!(0,1)$.
\end{theorem}
\begin{proof}
	From the smoothness of $f$ we obtain 
	\begin{equation*}
		\begin{split}
			f(w_{k+1}) &\leq f(w_k) + \langle \grad (w_k), w_{k+1} - w_k \rangle + \frac{L}{2} \| w_{k+1} - w_k \|^2\\
			&= f(w_k) + \eta_k\langle \grad (w_k), \gradik(w_k) \rangle + \frac{L\eta_k^2}{2} \| \gradik(w_k) \|^2
		\end{split}
	\end{equation*}
	We then rearrange, sum $a_2 (C_k - \fikofwstar)$ on both sides and use \eqref{eq:zhang}, to obtain
	\begin{equation}\label{eq:initial_pl}
		\begin{split}
			\frac{f(w_{k+1}) - f(w_k)}{\eta_k} + a_2 (C_k - \fikofwstar) &\leq -\langle \grad (w_k), \gradik(w_k) \rangle + \frac{L\eta_k}{2} \| \gradik(w_k) \|^2\\
			&\quad + a_2 (C_k - \fikofwstar)\\
			&\leq -\langle \grad (w_k), \gradik(w_k) \rangle + (\frac{L}{2c}+a_2) \left( C_k - \fikofwstar \right),
		\end{split}
	\end{equation}
	Let us now distinguish 2 cases: either 1) $ \fik(w_k) \geq \tilde{C}_k$ and then $C_k= \fik(w_k)$, or 2) $ \fik(w_k) < \tilde{C}_k$  and then $C_k= \tilde{C}_k$. Let us first analyze case 1). Assuming $a_2,b_2>0$, from \eqref{eq:initial_pl} we have 
	\begin{equation*}
		\begin{split}
			\frac{f(w_{k+1}) - f(w_k)}{\eta_k} + a_2 (C_k - \fikofwstar) &\leq -\langle \grad (w_k), \gradik(w_k) \rangle + (\frac{L}{2c}+ a_2) \left( \fik(w_k) - \fikofwstar \right)\\
			&\leq -\langle \grad (w_k), \gradik(w_k) \rangle + (\frac{L}{2c}+ a_2) \left( \fik(w_k) - \fikofwstar \right)\\
			& + a_2b_2\left( C_{k-1} - \fikofwstar \right)
		\end{split}
	\end{equation*}
	where the last inequality follows from \eqref{eq:induction2}. The above bound will be now proven also for case 2). From \eqref{eq:initial_pl} we have 
	\begin{equation*}
		\begin{split}
			\frac{f(w_{k+1}) - f(w_k)}{\eta_k} + a_2 (C_k - \fikofwstar) &\leq -\langle \grad (w_k), \gradik(w_k) \rangle + (\frac{L}{2c}+ a_2) \left( \tilde{C}_k - \fikofwstar \right)\\
			&= -\langle \grad (w_k), \gradik(w_k) \rangle + (\frac{L}{2c}+ a_2) \frac{1}{\xi Q_k +1}\left(\fik(w_k) - \fikofwstar \right)\\
			&\quad +(\frac{L}{2c}+ a_2) \frac{\xi Q_k}{\xi Q_k +1} \left( C_{k-1} - \fikofwstar \right)\\
			&\leq -\langle \grad (w_k), \gradik(w_k) \rangle + (\frac{L}{2c}+ a_2)\left(\fik(w_k) - \fikofwstar \right)\\
			&\quad +(\frac{\Lmax}{2c}+ a_2) \xi \left( C_{k-1} - \fikofwstar \right)\\
		\end{split}
	\end{equation*}
	where the second inequality follows from \eqref{eq:qk} and \eqref{eq:frac_qk_ub1} and $L\leq \Lmax$. By defining $b_2 := \left(1 + \frac{\Lmax}{a_2c}\right) \xi$ we can conclude that the same bound can be achieved in both cases 1) and 2). Now, by taking expectation w.r.t. $i_k$ on the common bound achieved in both cases 1) and 2), and applying PL condition, we have
	\begin{equation*}
		\begin{split}
			\Eik \left[\frac{f(w_{k+1}) - f(w_k)}{\eta_k} + a_2 (C_k - \fofwstar) \right] 
			&\leq -\langle \grad (w_k), \grad(w_k) \rangle + (\frac{L}{2c}+ a_2)\left(f(w_k) - \fofwstar \right)\\
			&\quad + a_2b_2 \left( C_{k-1} - \fofwstar \right)\\
			&\leq(\frac{L}{2c}+ a_2-2\mu)\left(f(w_k) - \fofwstar \right)\\
			&\quad + a_2b_2 \left( C_{k-1} - \fofwstar \right),\\
		\end{split}
	\end{equation*}
	where in the first inequality we used the fact that $C_{k-1}$ does not depend on $i_k$ and $\Eik \fik(w^*) = f(w^*) = \Eik f(w^*) $. Thus,
	\begin{equation*}
		\begin{split}
			\Eik \left[\frac{f(w_{k+1}) - \fofwstar}{\eta_k} + a_2 (C_k - \fofwstar) \right] 
			&\leq \Eik \left[\frac{f(w_k) - \fofwstar}{\eta_k} \right]+ (\frac{L}{2c}+ a_2-2\mu)\left(f(w_k) - \fofwstar \right)\\
			&\quad + a_2b_2 \left( C_{k-1} - \fofwstar \right)\\
			&\leq (\frac{1}{\etamin}+\frac{L}{2c}+ a_2-2\mu)\left(f(w_k) - \fofwstar \right)\\
			&\quad + a_2b_2 \left( C_{k-1} - \fofwstar \right)
		\end{split}
	\end{equation*}
	Using $\eta_k \leq \etamax$, $L \leq \Lmax$ and taking the total expectation we obtain
	\begin{equation}\label{eq:final_pl}
		\begin{split}
			\E \left[f(w_{k+1}) - \fofwstar + a_2\etamax (C_k - \fofwstar) \right] 
			&\leq \etamax (\frac{1}{\etamin}+\frac{\Lmax}{2c}+ a_2-2\mu)\E \left[f(w_k) - \fofwstar \right]\\
			&\quad + a_2\etamax b_2 \E\left[ C_{k-1} - \fofwstar \right]
		\end{split}
	\end{equation}
	Defining $\nu :=\etamax (\frac{1}{\etamin}+\frac{\Lmax}{2c}+ a_2-2\mu)$, let us now show that $0<\nu<1$. From the assumption $\etaminn>\frac{2\delta(1-c)}{\Lmax}$, we have $\etamin = \minimum{\frac{2\delta(1-c)}{\Lmax}}{\etaminn} =\frac{2\delta(1-c)}{\Lmax}$, and then 
	\begin{equation*}
		\begin{split}
			\nu=\etamax\left( \frac{\Lmax}{2\delta(1-c)}+\frac{\Lmax}{2c}-2\mu + a_2\right) = \etamax \left(\frac{\Lmax}{2\delta c(1-c)} - 2\mu + a_2\right) = \etamax\left(\frac{\Lmax -4\mu c(1-c)}{2\delta c(1-c)} + a_2\right).
		\end{split}
	\end{equation*}
	Let $a_2:=\frac{4\mu c(1-c) -\Lmax}{4 \delta c (1-c)}+\frac{1}{2\etamax}$, under the assumption $\etamax<\frac{2\delta c(1-c)}{\Lmax -4\mu c(1-c)}$, we have
	\begin{equation*}
		\begin{split}
			a_2 = \frac{4\mu c(1-c) -\Lmax}{4\delta c (1-c)}+\frac{1}{2\etamax}> \frac{4\mu c(1-c) -\Lmax}{4\delta c (1-c)} + \frac{\Lmax -4\mu c(1-c)}{4\delta c(1-c)}=0, 
		\end{split}
	\end{equation*}
	and thus $a_2>0$. Let us now use substitute $a_2$ in the definition of $\nu$, to obtain
	\begin{equation}\label{eq:nu}
		\begin{split}
			\nu= \etamax\left(\frac{\Lmax -4\mu c(1-c)}{2\delta c(1-c)} + a_2\right) = \etamax\left(\frac{\Lmax -4\mu c(1-c)}{4\delta c(1-c)}\right) + \frac{1}{2} 
		\end{split}
	\end{equation}
	Again from $\etamax<\frac{2\delta c(1-c)}{\Lmax -4\mu c(1-c)}$ and \eqref{eq:nu}, we obtain $\nu< \frac{1}{2}+\frac{1}{2} = 1$. Moreover, $\Lmax -4\mu c(1-c)>0$ because it is a quadratic polynomial in $c$, whose $\Delta<0$ since of $\mu<\Lmax$. Thus, together with $\etamax>0$ and $4\delta c(1-c)>0$, from \eqref{eq:nu} we achieve $\nu >0$.
	Regarding $b_2$, we have $b_2>0$ because $a_2>0$. Moreover, by assuming $\xi<\frac{a_2c}{a_2c + \Lmax}$, it also follows that 
	\begin{equation*}
		\begin{split}
			b_2 := \left(1 + \frac{\Lmax}{a_2c}\right) \xi = \frac{a_2c + \Lmax}{a_2c}\xi < 1.
		\end{split}
	\end{equation*}
	At this point, we need to ensure that $\frac{2\delta(1-c)}{\Lmax}  < \etaminn < \etamax<\frac{2\delta c(1-c)}{\Lmax -4\mu c(1-c)}$ or equivalently 
	\begin{equation*}
		\begin{split}
			0 &< \frac{1}{\Lmax} \left(- 1 + \frac{c\Lmax}{\Lmax -4\mu c(1-c)} \right)\\
			& = \frac{1}{\Lmax} \left(\frac{ c\Lmax - \Lmax +4\mu c(1-c)}{\Lmax -4\mu c(1-c)} \right)\\
			& = \frac{1}{\Lmax} \left(\frac{ -4\mu c^2 + (4\mu +\Lmax)c - \Lmax }{ 4\mu c^2 -4\mu c + \Lmax } \right).
		\end{split}
	\end{equation*}
	In particular, we can solve the inequality for $c$ and find that the numerator is positive for $\frac{\Lmax}{4\mu}<c<1$ and that the denominator is always positive since $\mu<\Lmax$ (as above).\\
	To conclude, we can use \eqref{eq:final_pl} recursively from $k$ to $0$ and obtain 
	\begin{equation*}
		\begin{split}
			\E \left[f(w_{k+1}) - \fofwstar + a_2\etamax (C_k - \fofwstar) \right] 
			&\leq d_2^k \E \left[f(w_0) - \fofwstar  + a_2\etamax \left(C_{-1} - \fofwstar\right) \right]\\
			&= d_2^k \left(1+a_2\etamax\right) \left(f(w_0) - \fofwstar\right)
		\end{split}
	\end{equation*}
	where $d_2:= \minimum{\nu}{b_2} < 1$ and in the last equality we have used that $C_{-1} = f_{i_0}(w_0)$.
\end{proof}

	\subsection{Common Lemmas}\label{sec:supp_lemmas}
	
	\begin{lemma}\label{lemma:LC1_bound}
		Let $f$ be $L$-Lipschitz smooth. Then
		\begin{equation}\label{eq:LC1_bound}
			f(y) \leq f(x) + \grad(x)^T (y-x) +\frac{L}{2} \| y-x\|^2.
		\end{equation}
	\end{lemma}
	\begin{proof}
		From the mean value theorem and differentiability of $f$ we have 
		\begin{equation*}
			\begin{split}
				f (y) & = f(x) + \int_{0}^{1} \grad((1-t)x + t y)^T\! (y-x) \; dt\\
				& = f(x) + \int_{0}^{1} \grad((1-t)x + t y)^T\! (y-x) - \grad(x)^T(y-x) \; dt + \grad(x)^T(y-x)\\ 
				&\leq f(x) + \int_{0}^{1} \| \grad((1-t)x + t y) - \grad(x)\|\! \cdot\! \|y-x\| \; dt + \grad(x)^T(y-x)\\  
				& \leq \fik(x) + \int_{0}^{1} L \| t(y-x) \|\! \cdot\! \|y-x\| \; dt + \grad(x)^T(y-x)\\
				& = f(x) + L\|y-x\|^2 \cdot \frac{t^2}{2} \Big|_0^1 + \grad(x)^T(y-x)\\
				& = f(x) + \grad(x)^T(y-x) + \frac{L}{2} \|y-x\|^2,
			\end{split}
		\end{equation*}
		where the second inequality follows from the Lipschitz continuity of $\grad$.
	\end{proof}

\begin{lemma}\label{lemma:dai02_32}
	Let $f$ be $L$-Lipschitz smooth and strongly convex. Then
	\begin{equation}\label{eq:dai02_32}
		f(x) - f(x^*)\leq \frac{1}{2\mu} \| \grad (x) \|^2.
	\end{equation}
\end{lemma}
\begin{proof}
	From the strong convexity of $f$ we have 
	\begin{equation*}
		f(y) \geq r(x,y) :=  f(x) + \grad(x)^T(y-x) + \frac{\mu}{2} \| y-x \|^2.
	\end{equation*}
	We now minimize both sides of the above inequality w.r.t. $\!y$. In particular, we differentiate $r(x,y)$ w.r.t. $\!y$ 
	and obtain 
	\begin{equation*}
		\frac{\partial \,r(x,y)}{\partial y} = \grad(x) + \mu (y-x) = 0 \Leftrightarrow y^* = x - \frac{1}{\mu} \grad(x)
	\end{equation*}
	which means that 
	\begin{equation*}
		r(x,y^*) = f(x) - \frac{1}{\mu} \| \grad(x) \|^2 + \frac{1}{2\mu} \| \grad(x) \|^" = f(x) - \frac{1}{2\mu} \| \grad(x) \|^2.
	\end{equation*}
	Thus, since  $\min_y f(y) \geq \min_y r(x,y)$,
	we can conclude that
	\begin{equation}\label{eq:pl}
		f(x^*) \geq f(x) - \frac{1}{2\mu} \| \grad(x) \|^2.
	\end{equation}
\end{proof}
	
	\begin{lemma}\label{lemma:lipschitz}
		Let $f$ be $L$-Lipschitz smooth. Then
		\begin{equation}\label{eq:lipschitz}
			\frac{1}{2L} \| \grad(x)\|^2 \leq f(x) - f(x^*).
		\end{equation}
	\end{lemma}
	\begin{proof}
		We can repeat the same argument of Lemma \ref{lemma:dai02_32} in the following inequality (obtained applying Lemma \ref{lemma:LC1_bound})
		\begin{equation*}
			f(y) \leq f(x) + \grad(x)^T(y-x) + \frac{L}{2} \| y-x \|^2,
		\end{equation*}
		and get that 
		\begin{equation*}
			f(x^*) \leq f(x) - \frac{1}{2L} \| \grad(x) \|^2,
		\end{equation*}
	which concludes the proof.
	\end{proof}

	\subsection{The Polyak Step Size is Bounded}\label{sec:supp_polyak}
	In this subsection, we show that the Polyak step size \eqref{eq:upolyak} is bounded. In particular, this step is capped at $\etamax>0$, so it is bounded from above by $\etamax$. By definition of \eqref{eq:upolyak}, interpolation and Lemma \ref{lemma:lipschitz} applied on $\fik$ we get 
\begin{equation*}
\eta_{k,0}\geq\frac{\fik(w_k) - \fik^*}{c_p||\gradik(w_k)||^2} \geq \frac{\frac{1}{2\Lik}||\gradik(w_k)||^2}{c_p||\gradik(w_k)||^2}\geq \frac{1}{2c_p\Lmax}.
\end{equation*}

	\section{Experimental Details}\label{sec:supp_experiments}
	The \texttt{PyTorch} \citep{paszke19a} code to reproduce our results can be found at \url{https://github.com/leonardogalli91/PoNoS}. Also, PoNoS is there available as a \texttt{torch.optim.Opimizer}. Experiments are conducted on a machine with an NVIDIA A100 PCIe GPU with 40 GB of memory.\\
Problems (dataset, model):
\begin{enumerate}
	\item MNIST, MLP (1 hidden-layer multi-layer perceptron of width 1000) \citep{lecun98a,luo19a};
	\item CIFAR10, ResNet-34 \citep{krizhevsky09a,he16a};
	\item CIFAR10, DenseNet-121 \citep{krizhevsky09a,huang17a};
	\item CIFAR100, ResNet-34 \citep{krizhevsky09a,he16a};
	\item CIFAR100, DenseNet-121 \citep{krizhevsky09a,huang17a};
	\item Fashion MNIST, EFFicientnet-B1 \citep{xiao17a,tan19a};
	\item SVHN, WideResNet \citep{netzer11a,zagoruyko16a};
	\item mushrooms, RBF-kernel model \citep{chang11a};  
	\item rcv1, RBF-kernel model \citep{chang11a};  
	\item ijcnn, RBF-kernel model \citep{chang11a};
	\item w8a, RBF-kernel model \citep{chang11a}; 
	\item PTB, Transformer Encoder \citep{marcus93a,vaswani17a};
	\item Wikitext2, Transformer-XL \citep{dai19a,merity17a}.
\end{enumerate}
All the datasets can be freely obtained respectively through the \texttt{pytorch} package (image classification), or from the LIBSVM repository \citep{chang11a} (binary classification), or from the websites \url{http://www.fit.vutbr.cz/~imikolov/rnnlm/simple-examples.tgz} (PTB) and \url{https://s3.amazonaws.com/research.metamind.io/wikitext/wikitext-2-v1.zip} (Wikitext2). In Table \ref{tab:data}, we report a few information concerning the problems above and the hyper-parameters employed by the algorithms on that problem. In order by column, we report the number of parameters of the model $n$, the number of instances of the train set $M$, the batch size employed by the algorithms $b$, the number of iterations ($\#$ mini-batches) in each epoch $\frac{M}{b}$, the number of instances in the test set, the max amount of epochs and in brackets the corresponding max number of iterations. The combinations of dataset/model have been replaced by the corresponding number in the above listing. 
\begin{table}
\centering
\begin{tabular}{|c|c|c|c|c|c|c|}
	\hline
\diagbox{problem}{data}	& $n$ & $M$ & $b$ & $\frac{M}{b}$ & test size & max epochs (corr. iter.) \\
	\hline
1	& 535818 & 60000 & 128 & 469 & 10000 & 200 (93800)\\
	\hline
2   & 21282122 & 50000 & 128 & 391 & 10000 & 200 (78200)\\
	\hline
3	& 6956298 & 50000 & 128 & 391 & 10000 & 200 (78200)\\
	\hline
4	& 21328292 & 50000 & 128 & 391 & 10000 & 200 (78200)\\
	\hline
5	& 7048548 & 50000 & 128 & 391 & 10000 & 200 (78200)\\
	\hline
6	& 6525418 & 60000 & 128 & 469 & 10000 & 200 (93800)\\
	\hline
7	& 369498 & 73257 & 128 & 573  & 26032 & 200 (114600)\\
	\hline
8	& 112 & 6499 & 100 & 65 & 1625 & 35 (2275)\\
	\hline
9	& 22 & 39992 & 100 & 400 & 9998 & 35 (14000)\\
	\hline
10	& 47236 & 16194 & 100 & 162 & 4048 & 35 (5670)\\
	\hline
11	& 300 & 39799 & 100 & 398 & 9950 & 35 (13930)\\
	\hline
12	& 13828478 & 59712 & 64 & 933 & / & 100 (93300)\\
	\hline
13	& 30725904 & 7296 & 64 & 114 & / & 100 (11400)\\
	\hline
\end{tabular}
\caption{Information on the problems and on some of the hyper-parameters of the algorithms. In order by column: number of parameters of the model $n$, number of instances of the train set $M$, batch size $b$, number of iterations ($\#$ mini-batches) in each epoch $\frac{M}{b}$, number of instances in the test set, max amount of epochs (corresponding max amount of iterations).}\label{tab:data}
\end{table}

\noindent The numerical results report various measures:
\begin{itemize}
	\item train loss: the full-batch loss on the training set;
	\item test accuracy: accuracy on the test set;
	\item average step size: average of all the mini-batch step sizes within the epoch;
	\item initial step size: average of all the mini-batch initial step sizes within the epoch;
	\item gradient norm: average of all the mini-batch gradient norms within the epoch;
	\item $\#$ backtracks: the total number of backtracking steps required within the epoch;
	\item runtime: wall clock time of each epoch.
\end{itemize}
Many of the plots in this paper have been created by averaging 5 runs that differ from each other only on the random seed randomizing the algorithms. The shaded error bars in the plots correspond to the standard deviation from the mean of the 5 runs.
\par If not specified differently, the following is the setting of hyper-parameters used for Algorithm \ref{alg}
\begin{equation*}
\delta = 0.5,\quad \xi=1, \quad \etamax=10,\quad c=0.5,\quad c_p=0.1.
\end{equation*}
The performance of PoNoS is not sensitive to hyper-parameters. In fact, the same values work across experiments and there is no need to fine-tune them. Most of PoNoS's hyperparameters are set to standard values, while others are either inherited by recent papers or fixed by the theory:
\begin{itemize}
	\item $\delta=0.5$, classical cut of the step \citep{nocedal06a}. SLS employs an unusual value of $0.9$ and this choice is connected to the use of their resetting technique \eqref{eq:etamax}. We checked the results of SLS with $\delta=0.5$ and they indeed turned out to not be as good as with $\delta=0.9$.
	\item $\xi=1$, the fully nonmonotone version of \citet{zhang04a}.
	\item $\etamax=10$, very classical value \citep{vaswani19a,loizou21a}. We conducted an ablation study in Section \ref{sec:supp_etamax}, which shows that larger values of $\etamax$ do not have a remarkable impact on the results. These results show that PoNoS is more robust than SLS and SPS to these changes.
	\item $c=0.5$, suggested by the theory. Both our theory and that of \citet{vaswani19a} suggest employing $0.5$ for $c$, rather than the classical $0.1$ or lower \citep{nocedal06a}. The numerical results in Section \ref{sec:supp_c} support this choice. In particular, they show in Figure \ref{fig:convex} that $c=0.1$ might bring PoNoS and its monotone counterpart to diverge.
	\item $c_p=0.1$, half of the inherited value \citep{loizou21a}. In \citet{loizou21a}, the value $0.2$ was suggested for SPS, however in our case, the initial step size is not the final step since the backtracking procedure might reduce this value to its half (or less). For this reason, we decided to employ a step that is initially double that of SPS. The results show that PoNoS|0.1 is consistently better than PoNoS|0.2 (see Section \ref{sec:supp_c_p}). Thus, we also checked whether SPS|0.1 would be consistently better than the original SPS, but this is not the case.
\end{itemize}
All the other optimizers were used with their default hyper-parameters and without any weight decay. The implementation of SGD and Adam is provided by \texttt{pytorch} \citep{paszke19a} and their learning rates were selected through a separate grid search on each of the problems. These values are collected in Table \ref{tab:grid}, where again the names are replaced by the above numbers. For the classification problems (1-11) the grid search was $\{10^{-1}, 10^{-2}, 10^{-3}, 10^{-4}\}$, while for training the language models we employed the larger grid $\{5, 2.5, 1\}\times \{10^{-1}, 10^{-2}, 10^{-3}, 10^{-4}\}$.
\begin{table}
\hspace{-3mm}
\begin{tabular}{|c|c|c|c|c|c|c|c|c|c|c|c|c|c|}
	\hline
	\diagbox{method}{problem} & 1 & 2 & 3 & 4 & 5 & 6 & 7 & 8 & 9 & 10 & 11 & 12 & 13 \\
	\hline
	SGD	& $0.1$ & $0.1$ & $0.1$ & $0.1$ & $0.1$ & $0.1$ & $0.1$ & $0.1$  & $0.1$  & $0.1$  & $0.1$  & 0.5 & 0.25 \\
	\hline
	Adam & $10^{-4}$ & $10^{-3}$ & $10^{-3}$ & $10^{-3}$ & $10^{-3}$ & $10^{-3}$ & $10^{-3}$ & $0.1$  & $0.1$ & $0.1$ & $0.1$ & $2.5\!\cdot\!10^{-4}$ & $2.5\!\cdot\!10^{-4}$ \\
	\hline
\end{tabular}
\caption{Learning rates for SGD and Adam obtained through a grid-search procedure.}\label{tab:grid}
\end{table}
	\section{Plots Completing the Figures in the Main Paper}\label{sec:supp_plots}
	\subsection{Comparison between PoNoS and the state-of-the-art}\label{sec:supp_exp1}
	See Figure \ref{fig:exp1} for the complete results relative to Figure 1 of the main paper. From Figure \ref{fig:exp1} we can additionally observe that on some problems (e.g., \ress$ $ and \denses) the step size yielded by the Polyak formula \eqref{eq:loizou} is very big and grows very fast. The algorithm SPS controls this step by reducing it to \eqref{eq:etamax}. This step does not depend on the Polyak rule and it does not employ the local information provided by $\fik(w_k)$ and $\gradik(w_k)$. Instead, it grows exponentially controlled by \eqref{eq:etamax}. On the other hand, the step size of PoNoS is always a scaled version of \eqref{eq:loizou}.

\renewcommand{\dir}{exp1/}
\renewcommand{\imgS}{0.30}
\begin{figure}[!h] 
	\hspace{-9mm}
	\begin{minipage}{0.9\textwidth}
		\renewcommand{\model}{mlp}
		\renewcommand{\modelname}{mlp}
		\subfloat{\includegraphics[width=\imgS\linewidth]{\dir/\model/train_loss}}
		\subfloat{\includegraphics[width=\imgS\linewidth]{\dir/\model/val_acc_focus}}
		\subfloat{\includegraphics[width=\imgS\linewidth]{\dir/\model/agv_step_size}}
		\subfloat{\includegraphics[width=\imgS\linewidth]{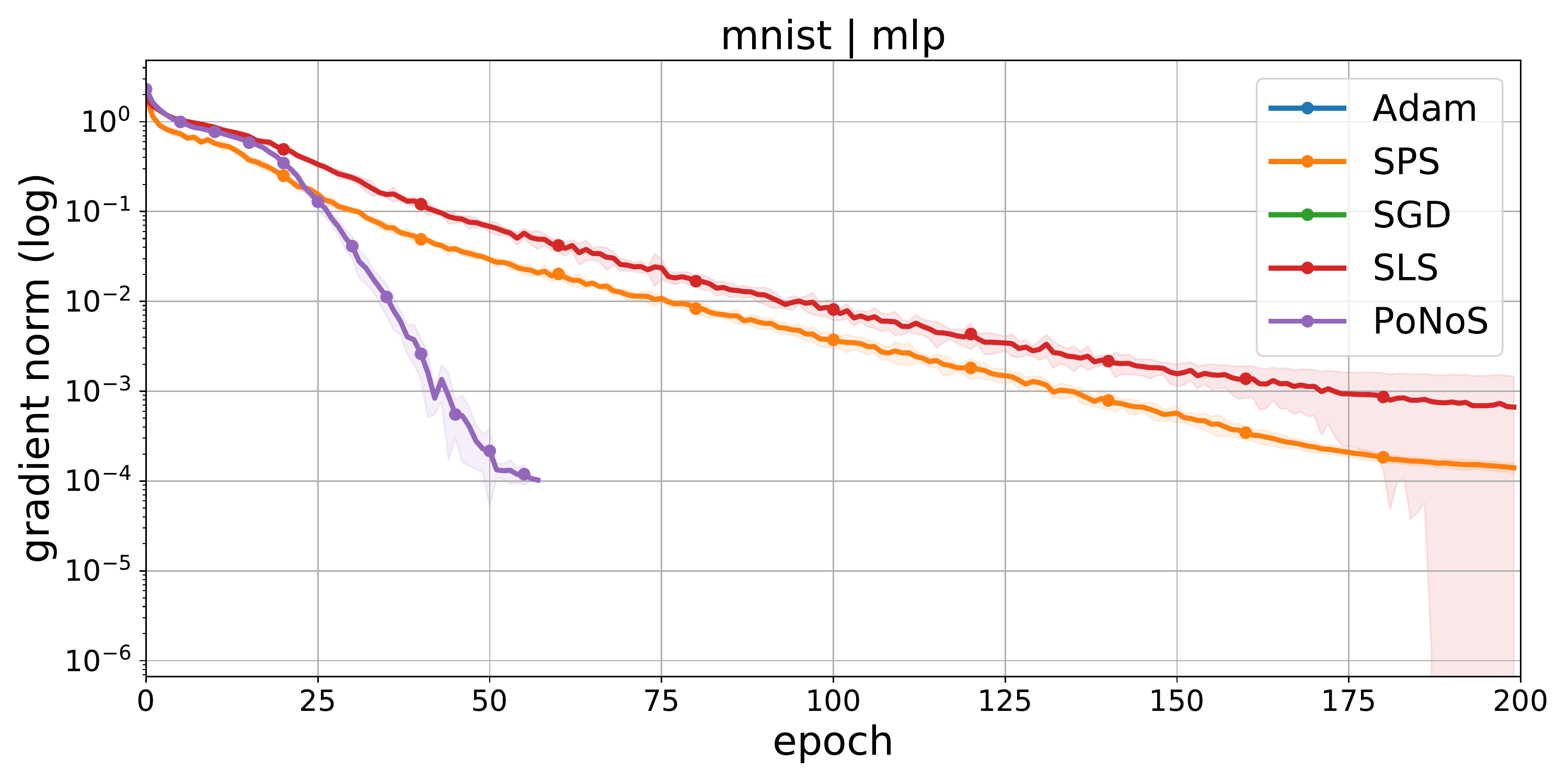}}\\
		\renewcommand{\model}{cifar10_resnet}
		\renewcommand{\modelname}{cifar10\_resnet}
		\subfloat{\includegraphics[width=\imgS\linewidth]{\dir/\model/train_loss}}
		\subfloat{\includegraphics[width=\imgS\linewidth]{\dir/\model/val_acc_focus}}
		\subfloat{\includegraphics[width=\imgS\linewidth]{\dir/\model/agv_step_size}}
		\subfloat{\includegraphics[width=\imgS\linewidth]{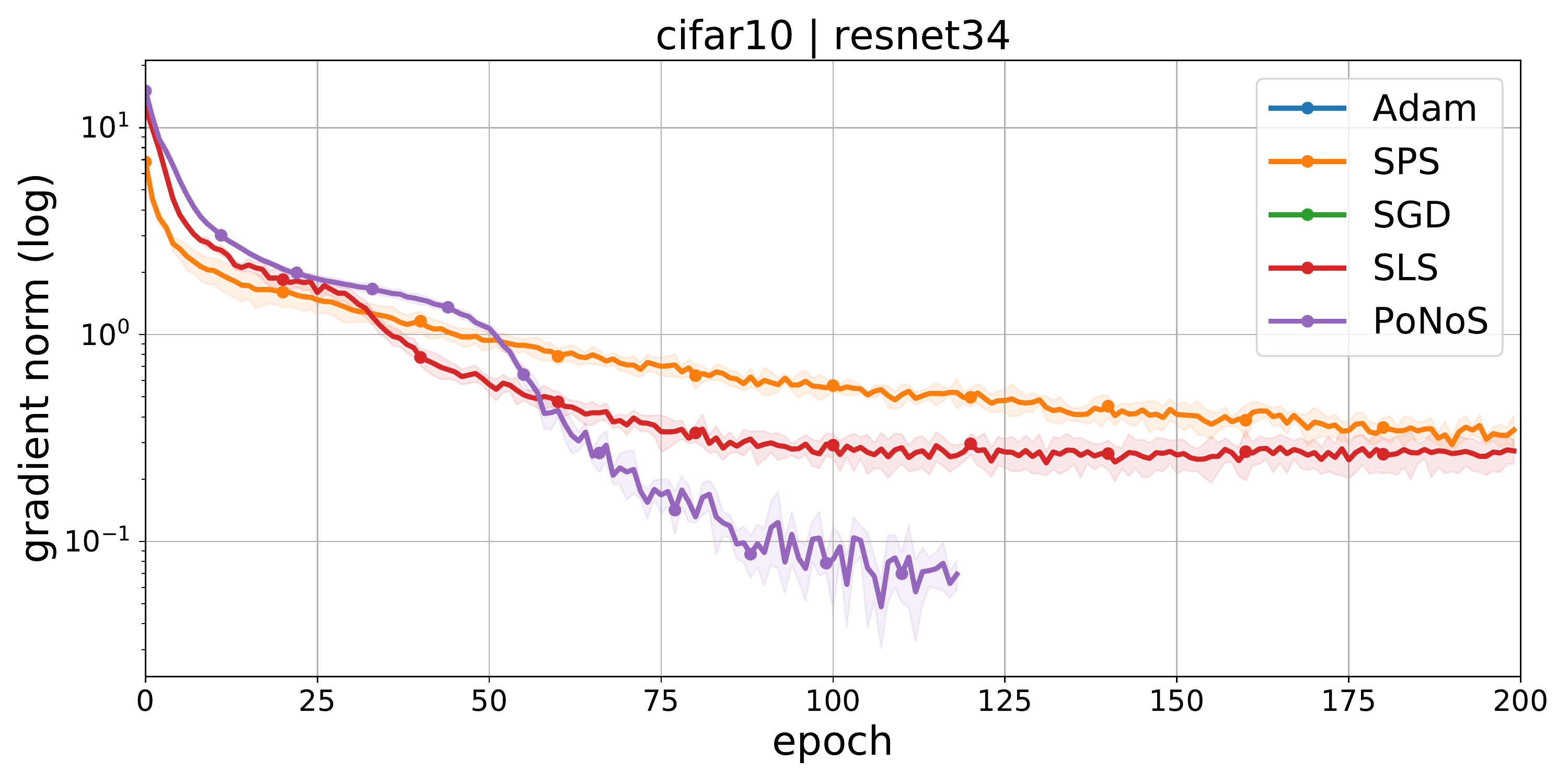}}\\
		\renewcommand{\model}{cifar10_densenet}
		\renewcommand{\modelname}{cifar10\_densenet}
		\subfloat{\includegraphics[width=\imgS\linewidth]{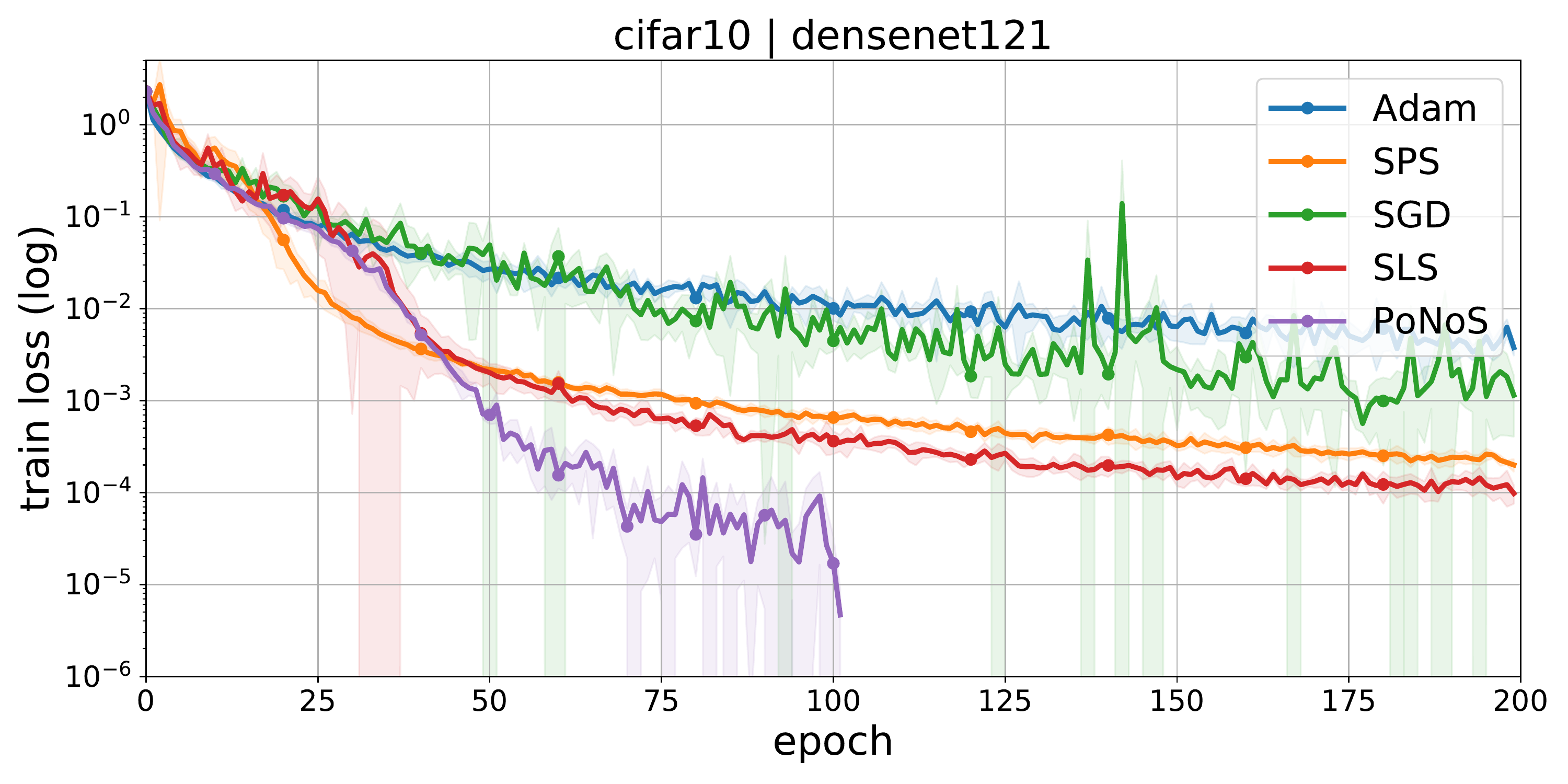}}
		\subfloat{\includegraphics[width=\imgS\linewidth]{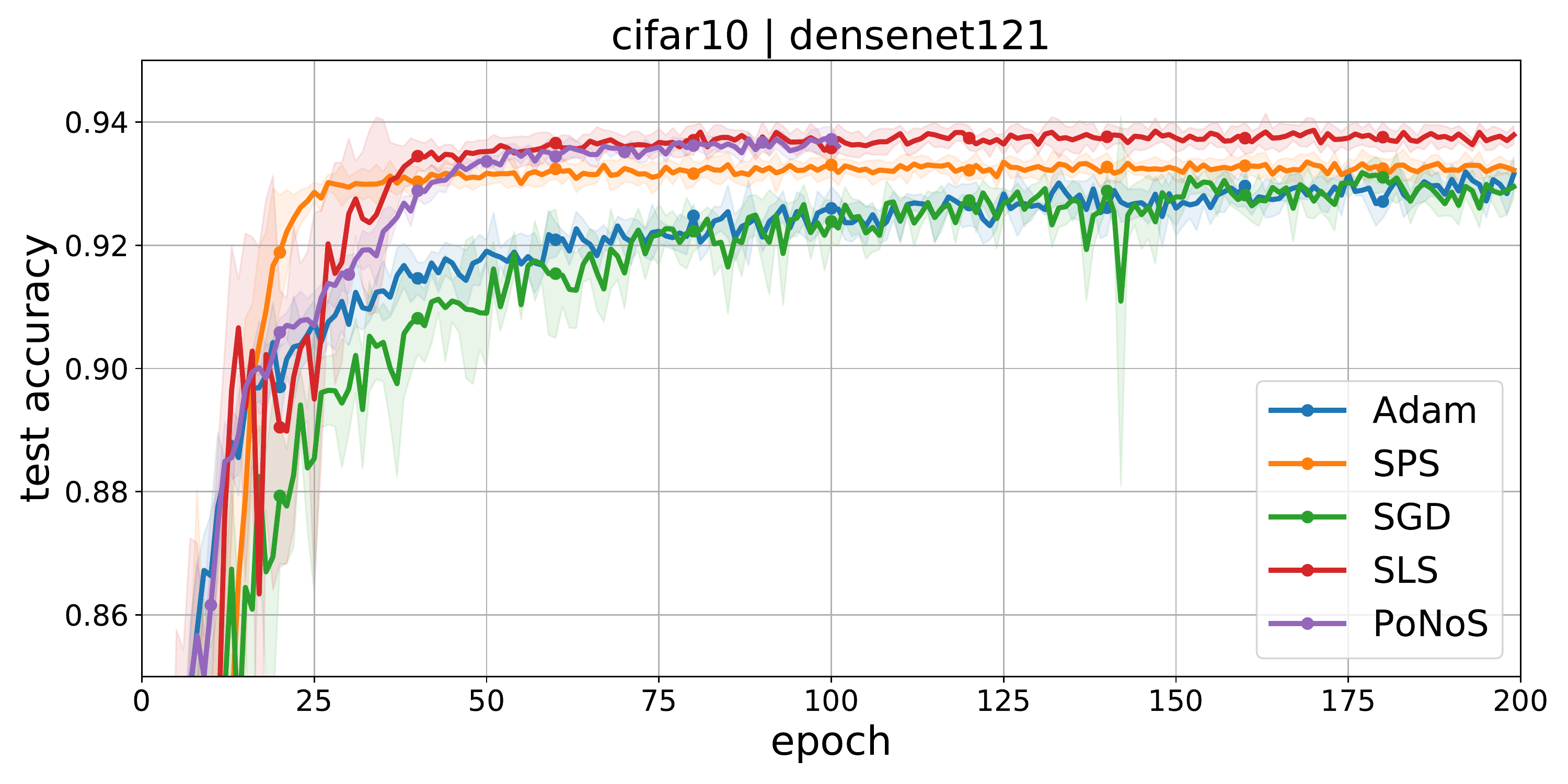}}
		\subfloat{\includegraphics[width=\imgS\linewidth]{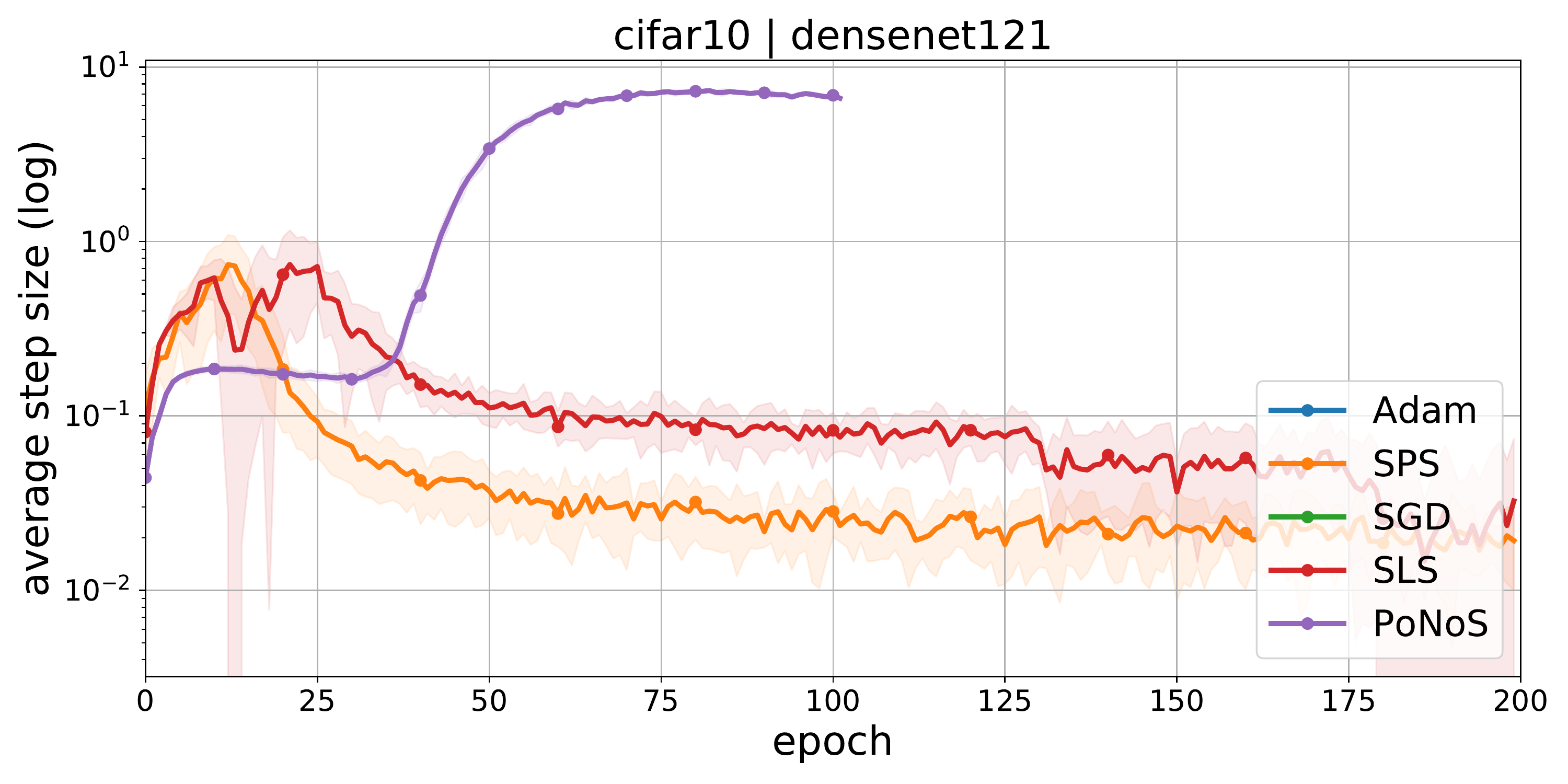}}
		\subfloat{\includegraphics[width=\imgS\linewidth]{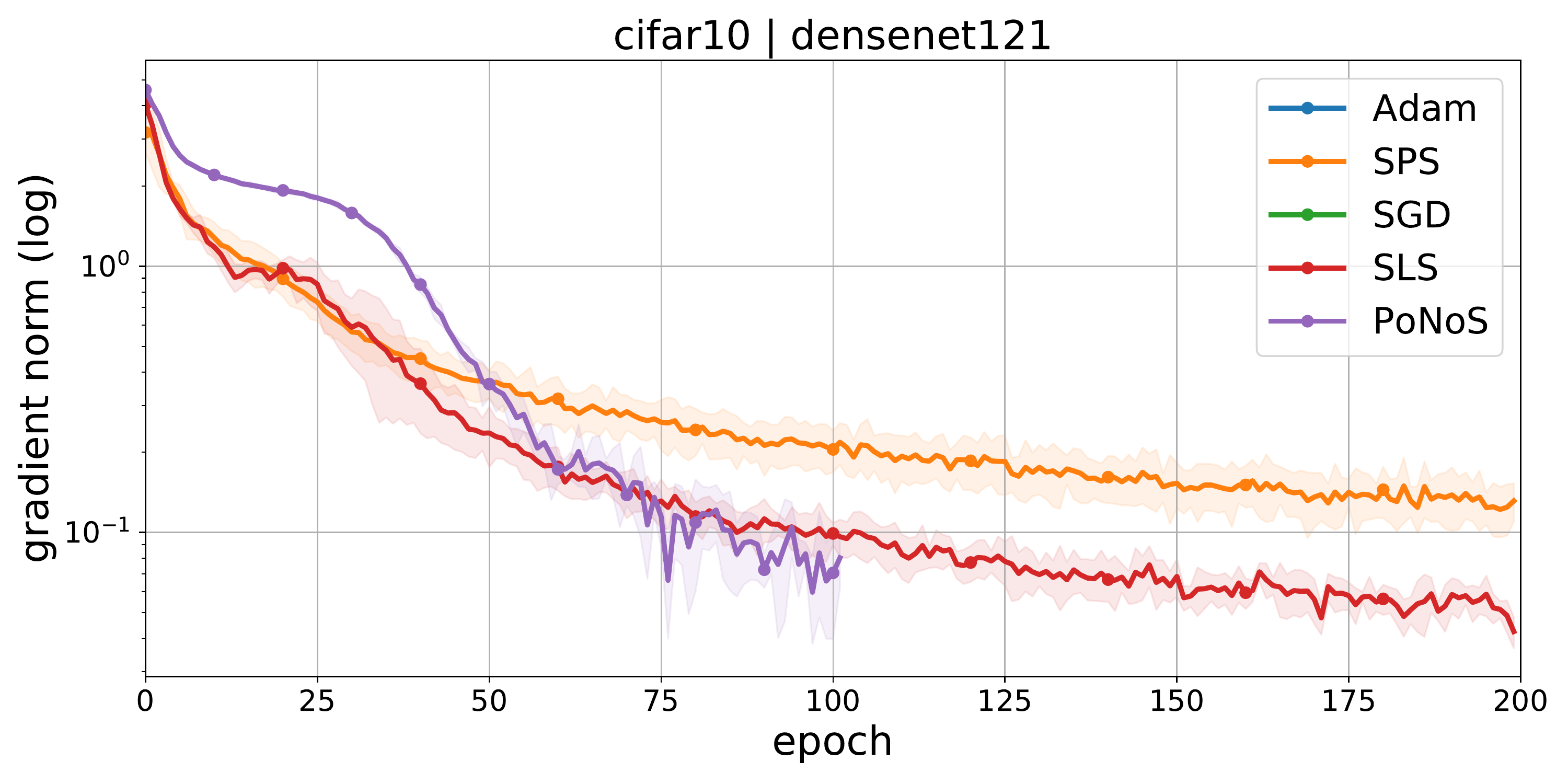}}\\
		\renewcommand{\model}{cifar100_resnet}
		\renewcommand{\modelname}{cifar100\_resnet}
		\subfloat{\includegraphics[width=\imgS\linewidth]{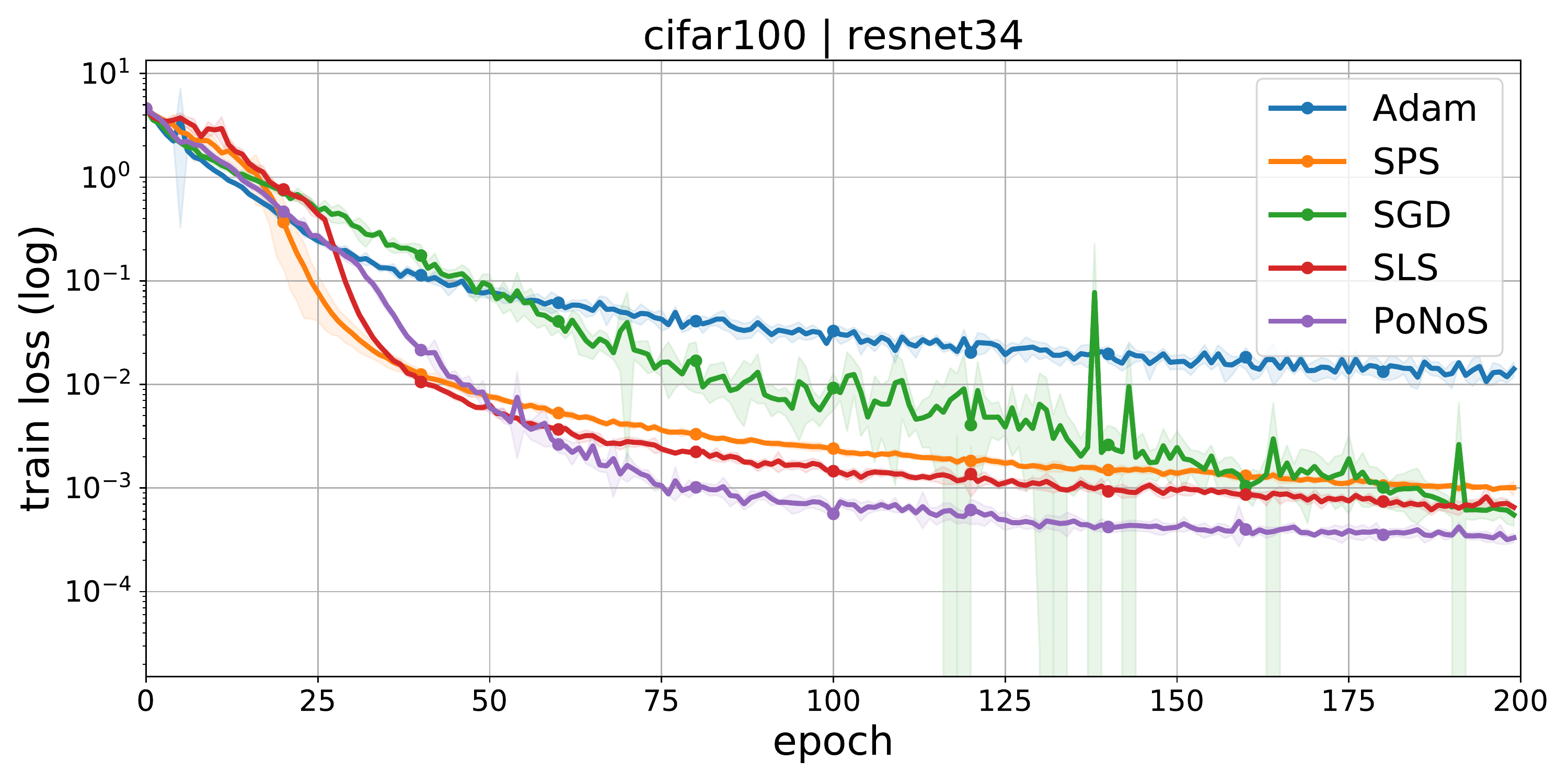}}
		\subfloat{\includegraphics[width=\imgS\linewidth]{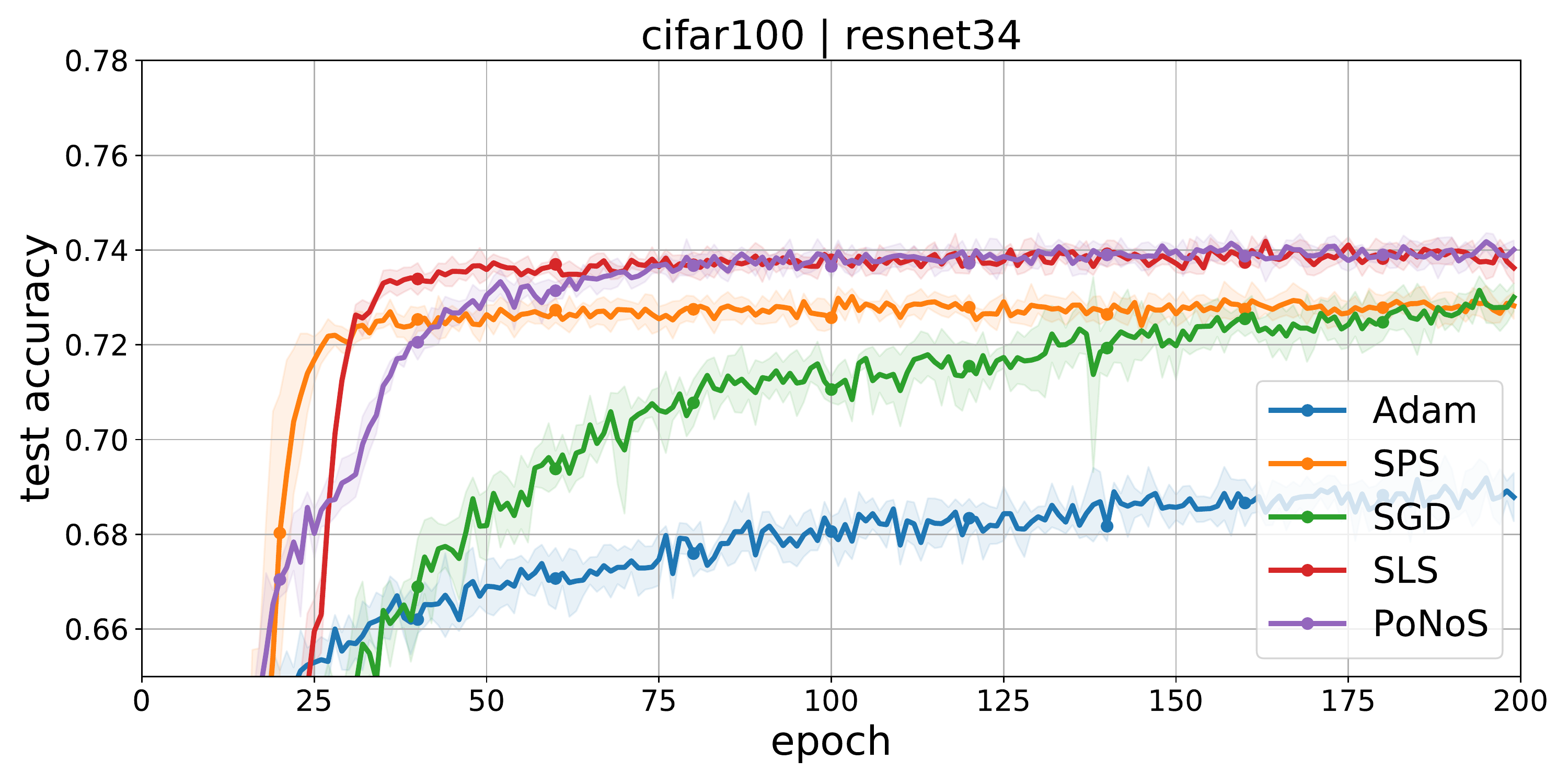}}
		\subfloat{\includegraphics[width=\imgS\linewidth]{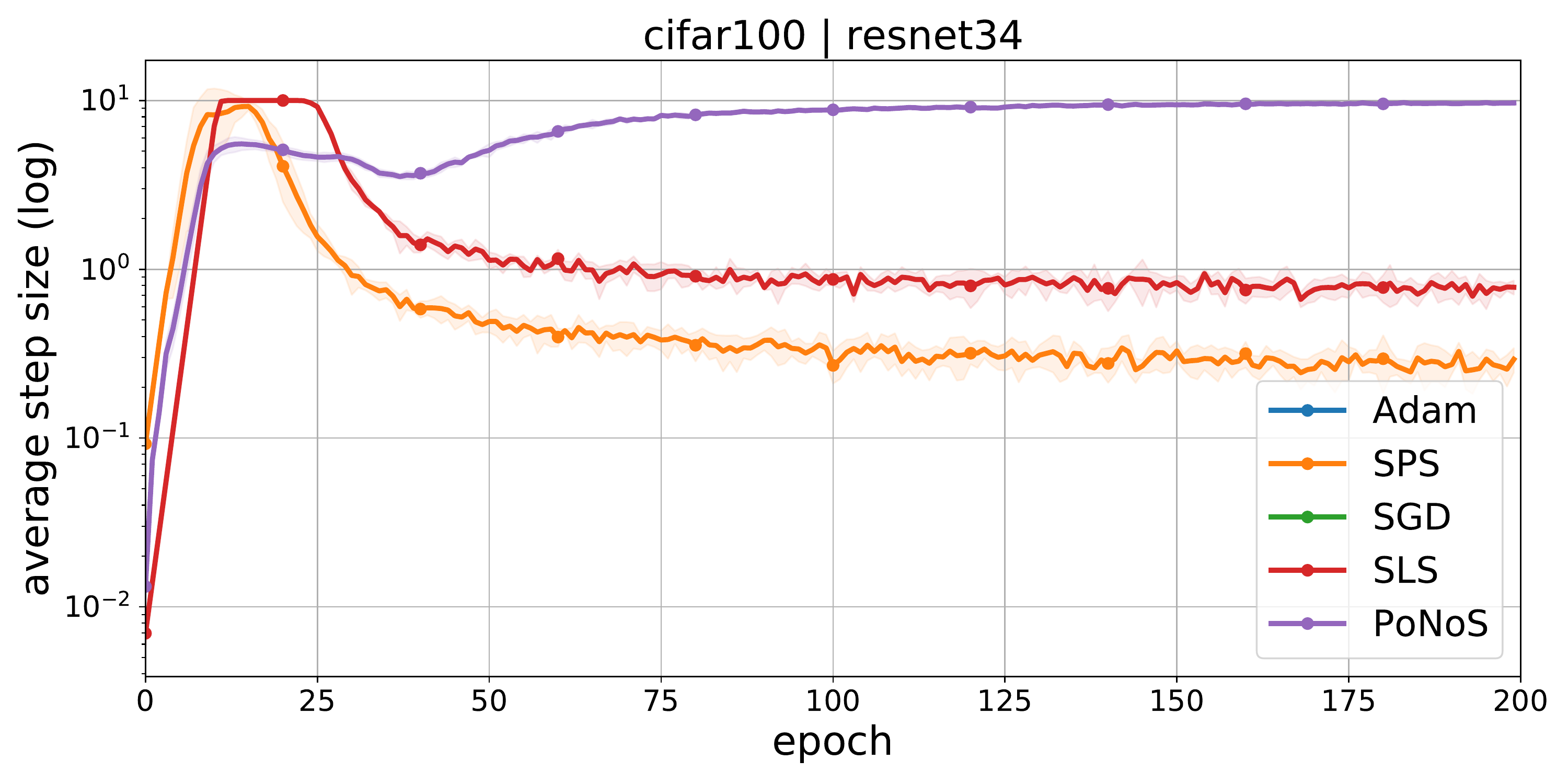}}
		\subfloat{\includegraphics[width=\imgS\linewidth]{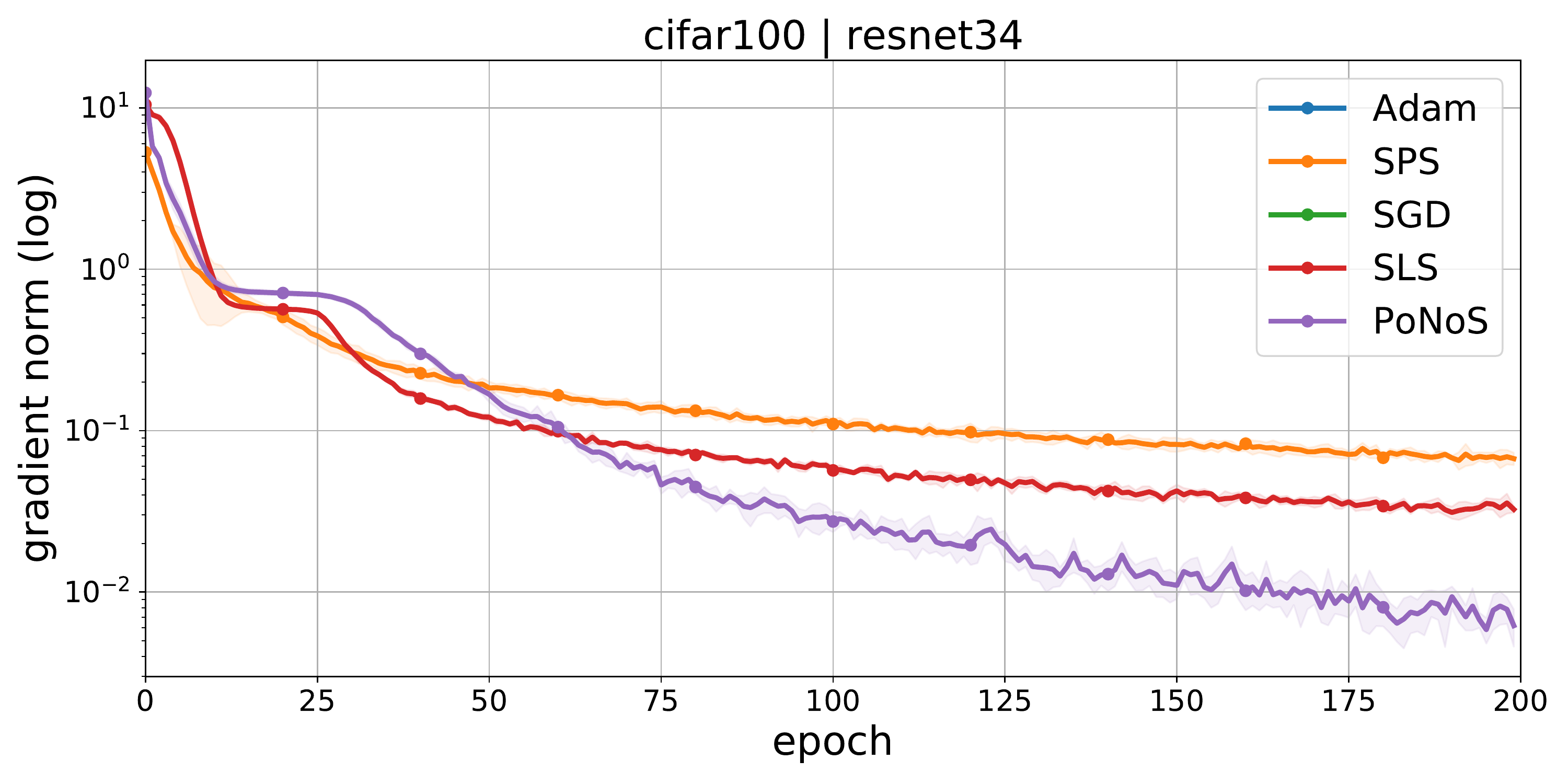}}\\
		\renewcommand{\model}{cifar100_densenet}
		\renewcommand{\modelname}{cifar100\_densenet}
		\subfloat{\includegraphics[width=\imgS\linewidth]{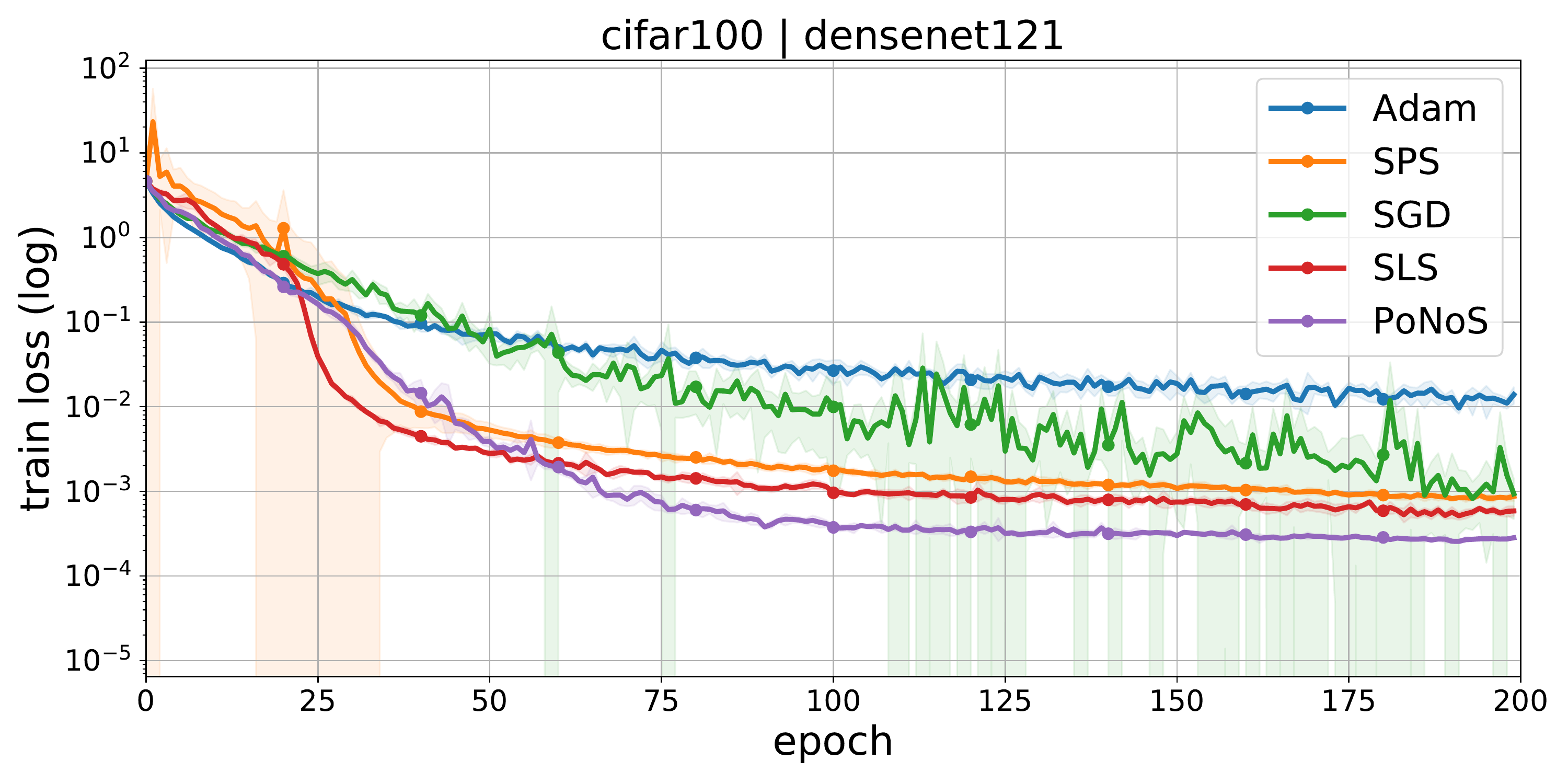}}
		\subfloat{\includegraphics[width=\imgS\linewidth]{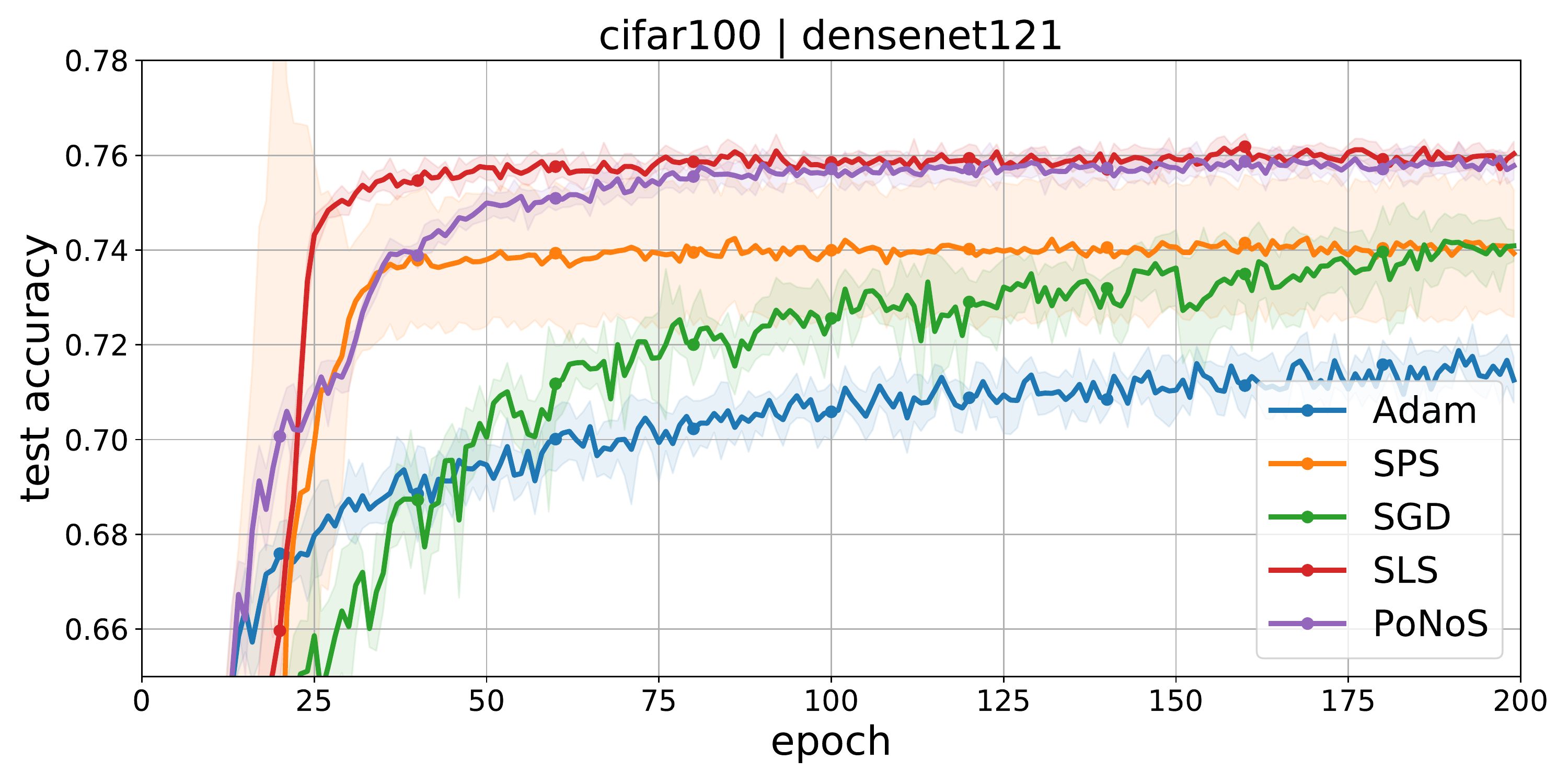}}
		\subfloat{\includegraphics[width=\imgS\linewidth]{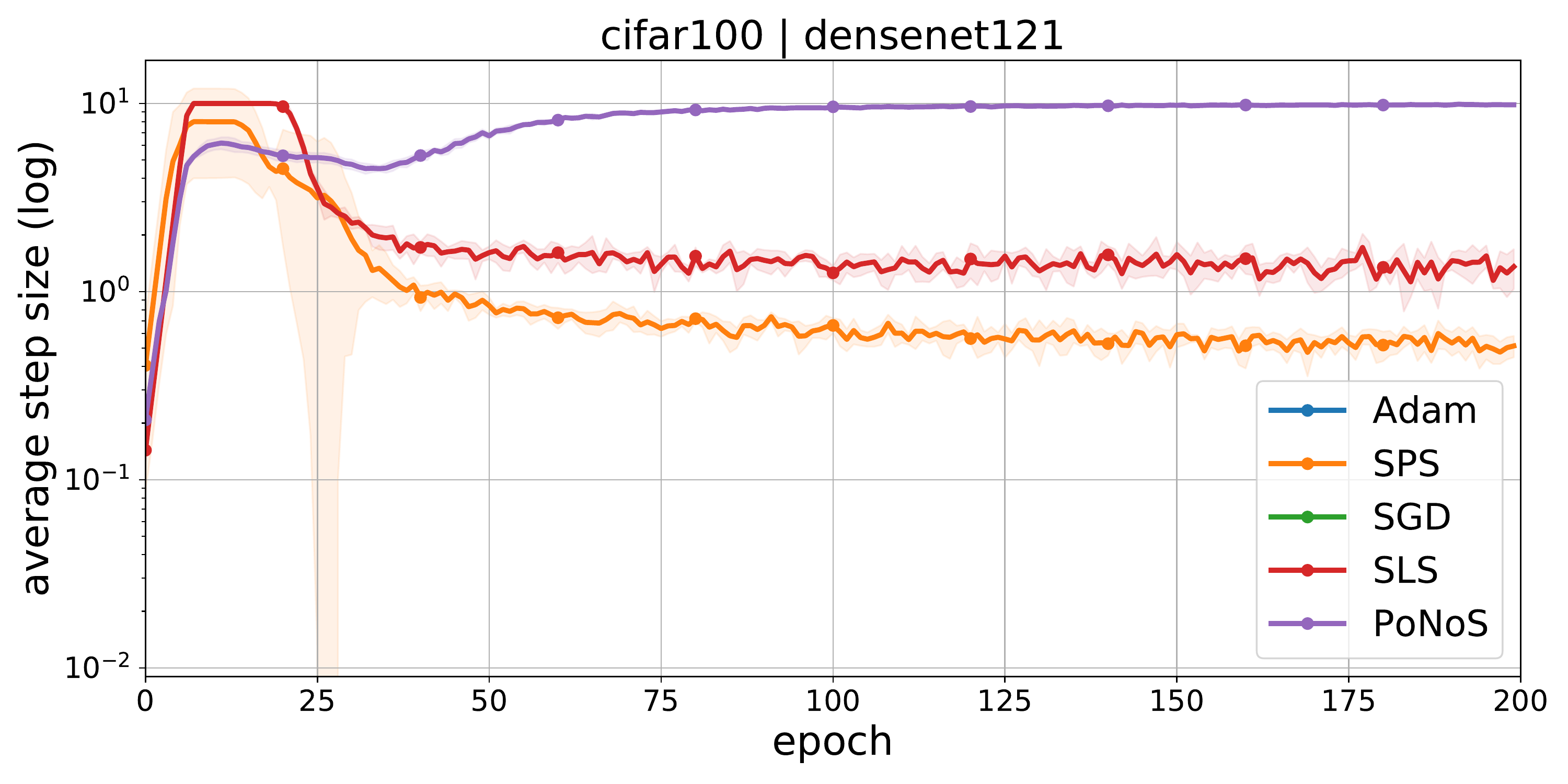}}
		\subfloat{\includegraphics[width=\imgS\linewidth]{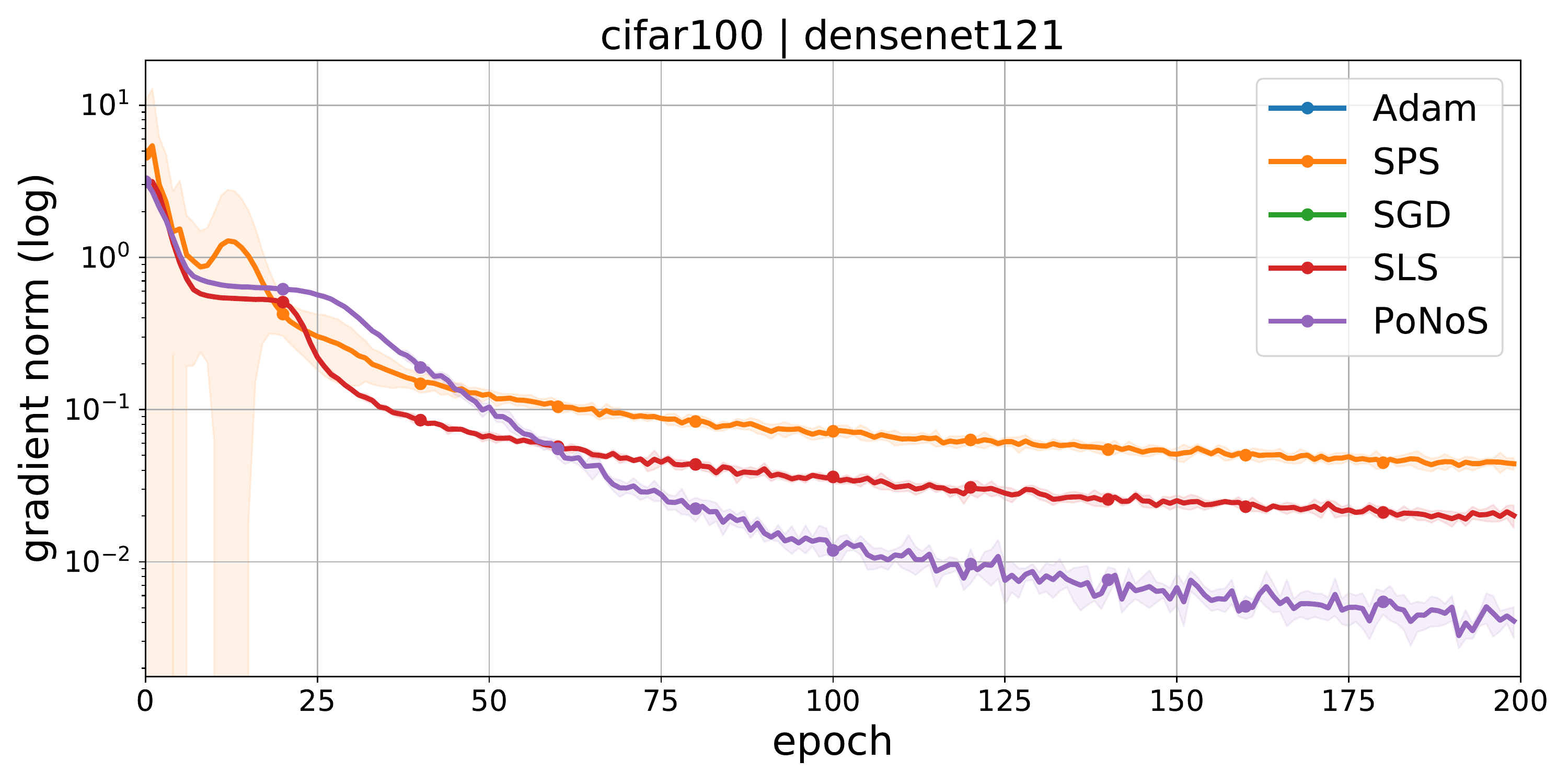}}\\
		\renewcommand{\model}{fashion_effb1}
		\renewcommand{\modelname}{fashion\_effb1}
		\subfloat{\includegraphics[width=\imgS\linewidth]{\dir/\model/train_loss}}
		\subfloat{\includegraphics[width=\imgS\linewidth]{\dir/\model/val_acc_focus}}
		\subfloat{\includegraphics[width=\imgS\linewidth]{\dir/\model/agv_step_size}}
		\subfloat{\includegraphics[width=\imgS\linewidth]{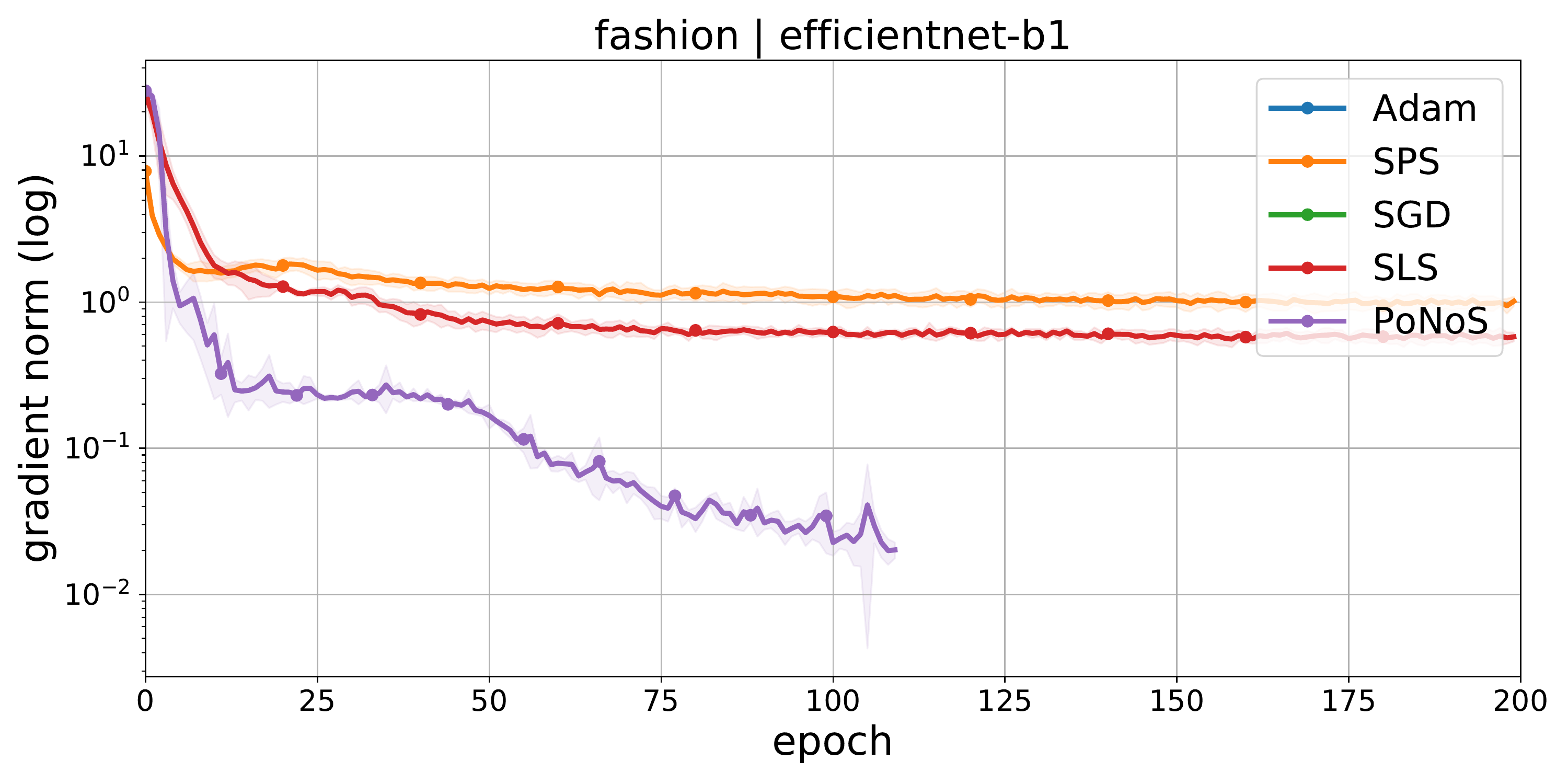}}\\
		\renewcommand{\model}{svhn_wrn}
		\renewcommand{\modelname}{svhn\_wrn}
		\subfloat{\includegraphics[width=\imgS\linewidth]{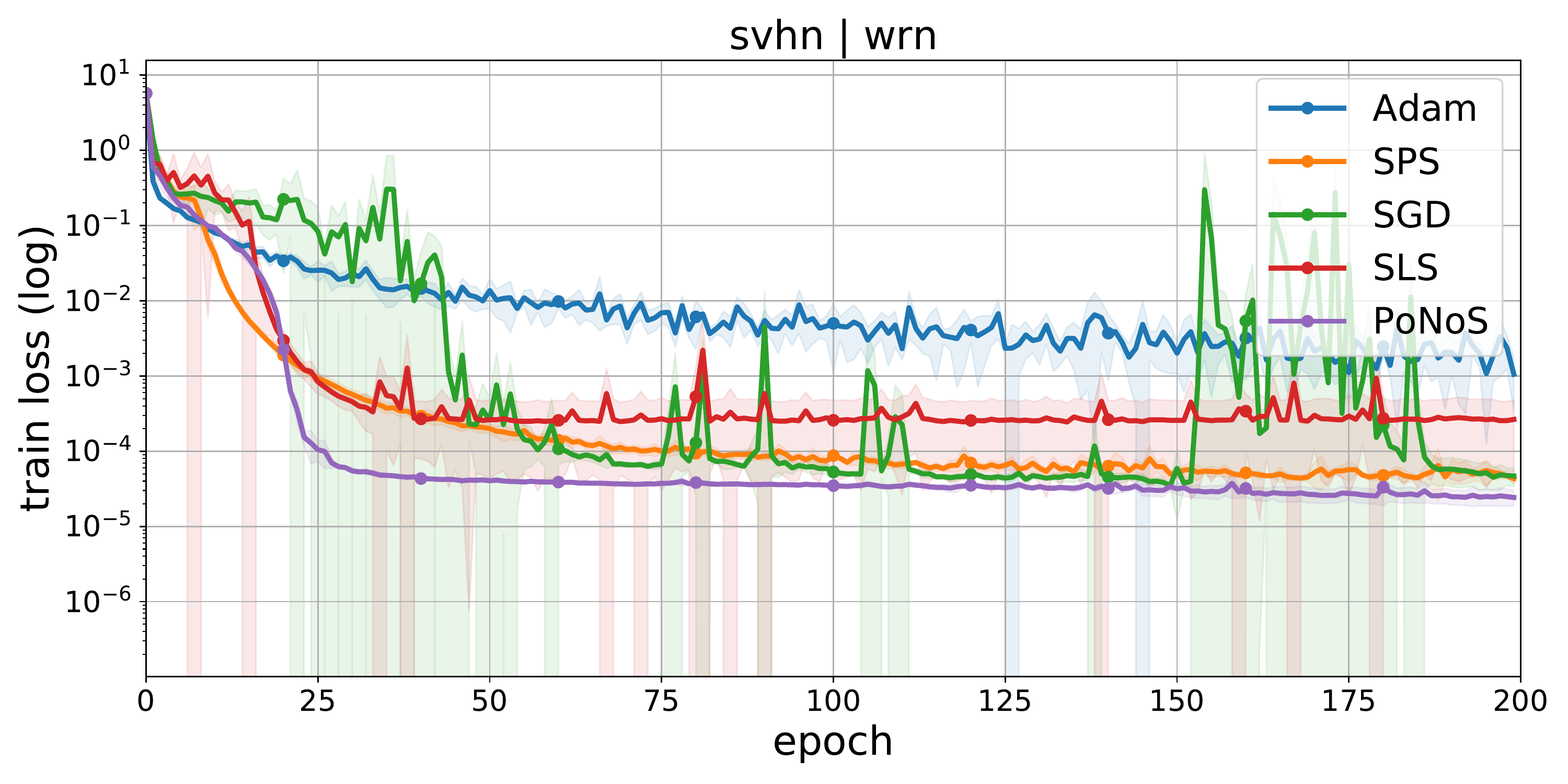}}
		\subfloat{\includegraphics[width=\imgS\linewidth]{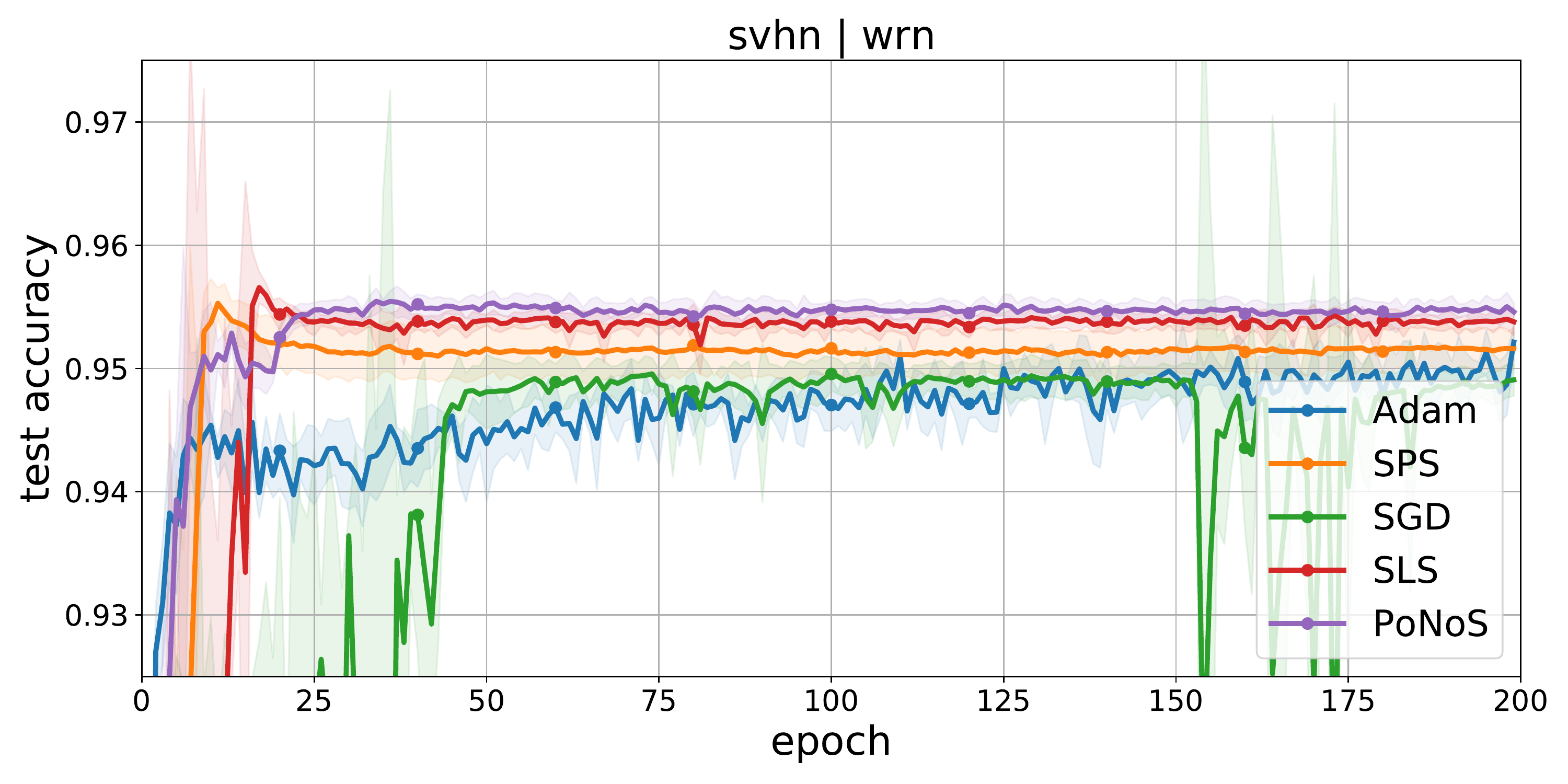}}
		\subfloat{\includegraphics[width=\imgS\linewidth]{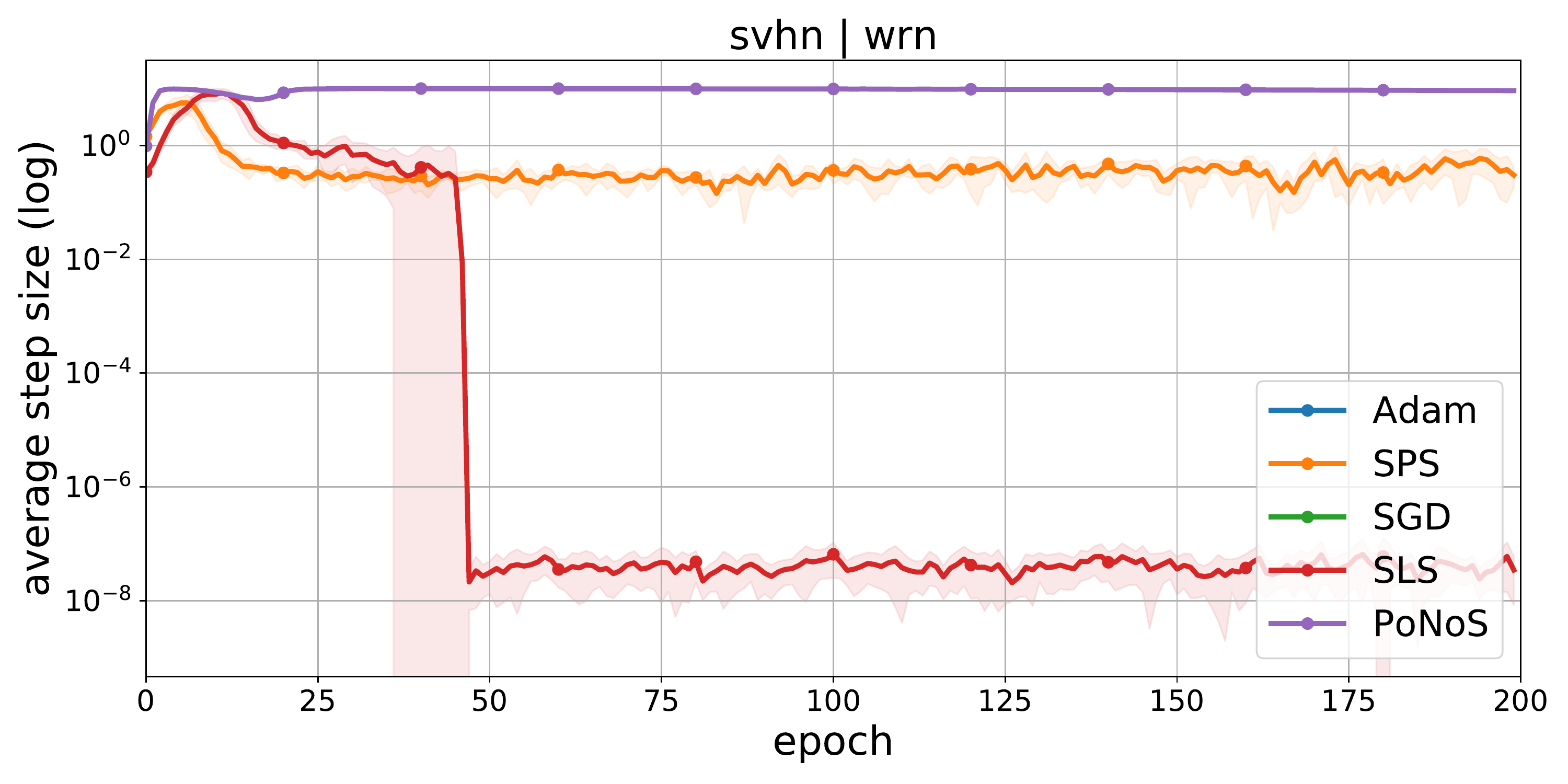}}
		\subfloat{\includegraphics[width=\imgS\linewidth]{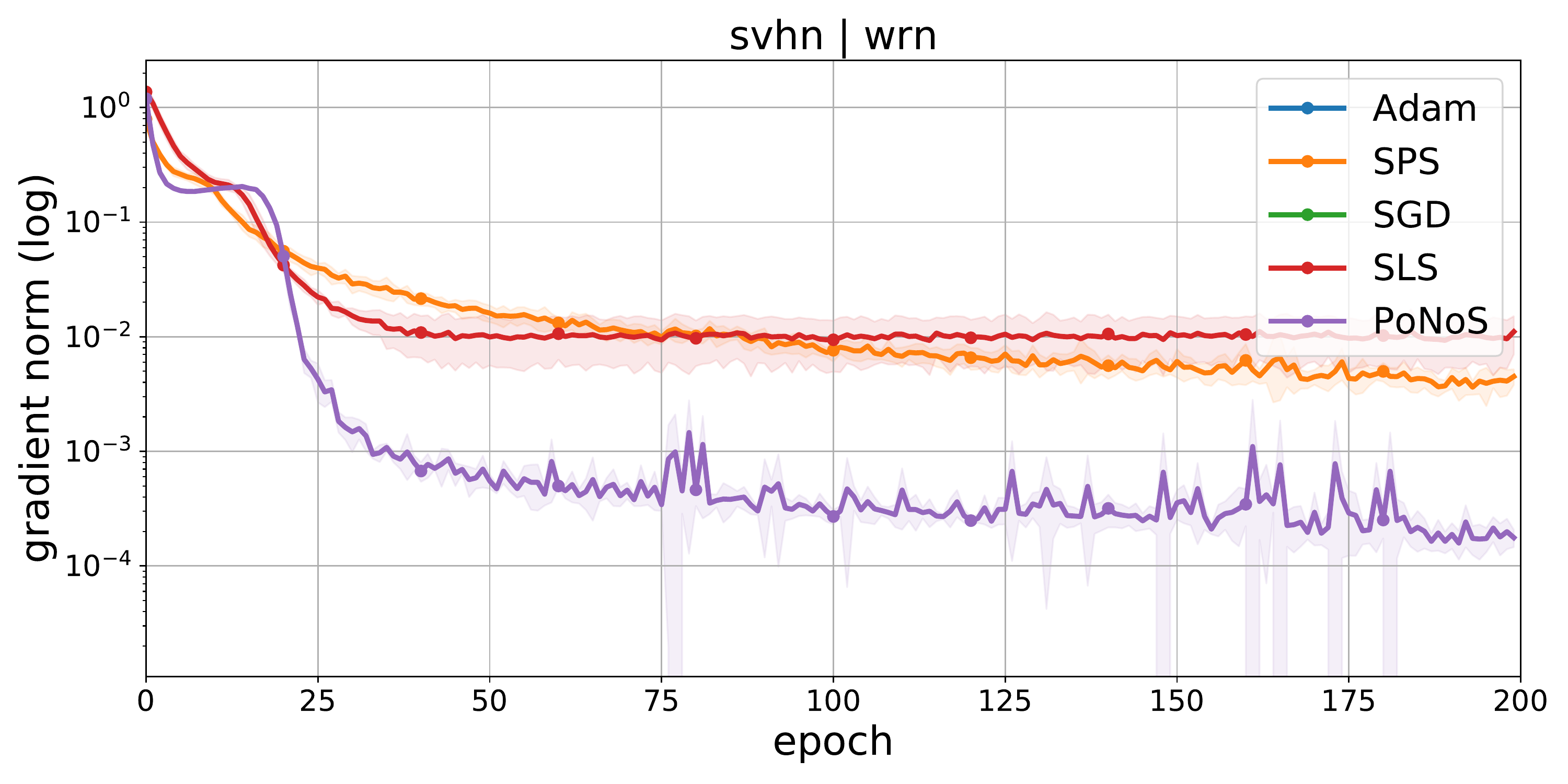}}
		\caption{Comparison between the proposed method (PoNoS) and the-state-of-the-art. Each row focus on a dataset/model combination. First column: train loss. Second column: test accuracy. Third column: average step size of the epoch. Fourth column: average gradient norm of the epoch.}\label{fig:exp1}
	\end{minipage}
\end{figure}

	\subsection{A New Resetting Technique}\label{sec:supp_reset}
	See Figure \ref{fig:reset} for the complete results relative to Figure 2 of the main paper. From Figure \ref{fig:reset}, we can additionally observe that the step size yielded by PoNoS is generally small in the initial phase, it grows in the intermediate phase and reaches $\etamax$ towards the end. This behavior can be also noticed in the amount of backtracks of PoNoS\_reset0 in the rightmost column of Figure \ref{fig:reset}. The number of backtracks is conspicuous at the beginning, rare in the intermediate phase and (almost) zero in the local phase. On the other hand, the resetting technique introduced in PoNoS maintains the amount of backtracks to be limited by 1 per iteration (on average), while still yielding the same step size. When PoNoS\_reset0 starts to reduce the amount of backtracks, also PoNoS does the same. 
\par The step size behavior described above and the corresponding amount of backtracks are in accordance with the intuition and the theory. Intuitively, a small step size allows the algorithm to proceed more cautiously in the global phase. A larger step size allows the algorithm to converge faster once it gets closer to the local phase. In fact, for achieving local Q-superlinear convergence, the theory predicts that the line search should always accept a new step size in the  local phase \citep{nocedal06a}.

\renewcommand{\dir}{reset/}
\renewcommand{\imgS}{0.30}
\begin{figure}[!h] 
\hspace{-9mm}
	\begin{minipage}{0.9\textwidth}
		\renewcommand{\model}{mlp}
		\renewcommand{\modelname}{mlp}
		\subfloat{\includegraphics[width=\imgS\linewidth]{\dir/\model/train_loss}}
		\subfloat{\includegraphics[width=\imgS\linewidth]{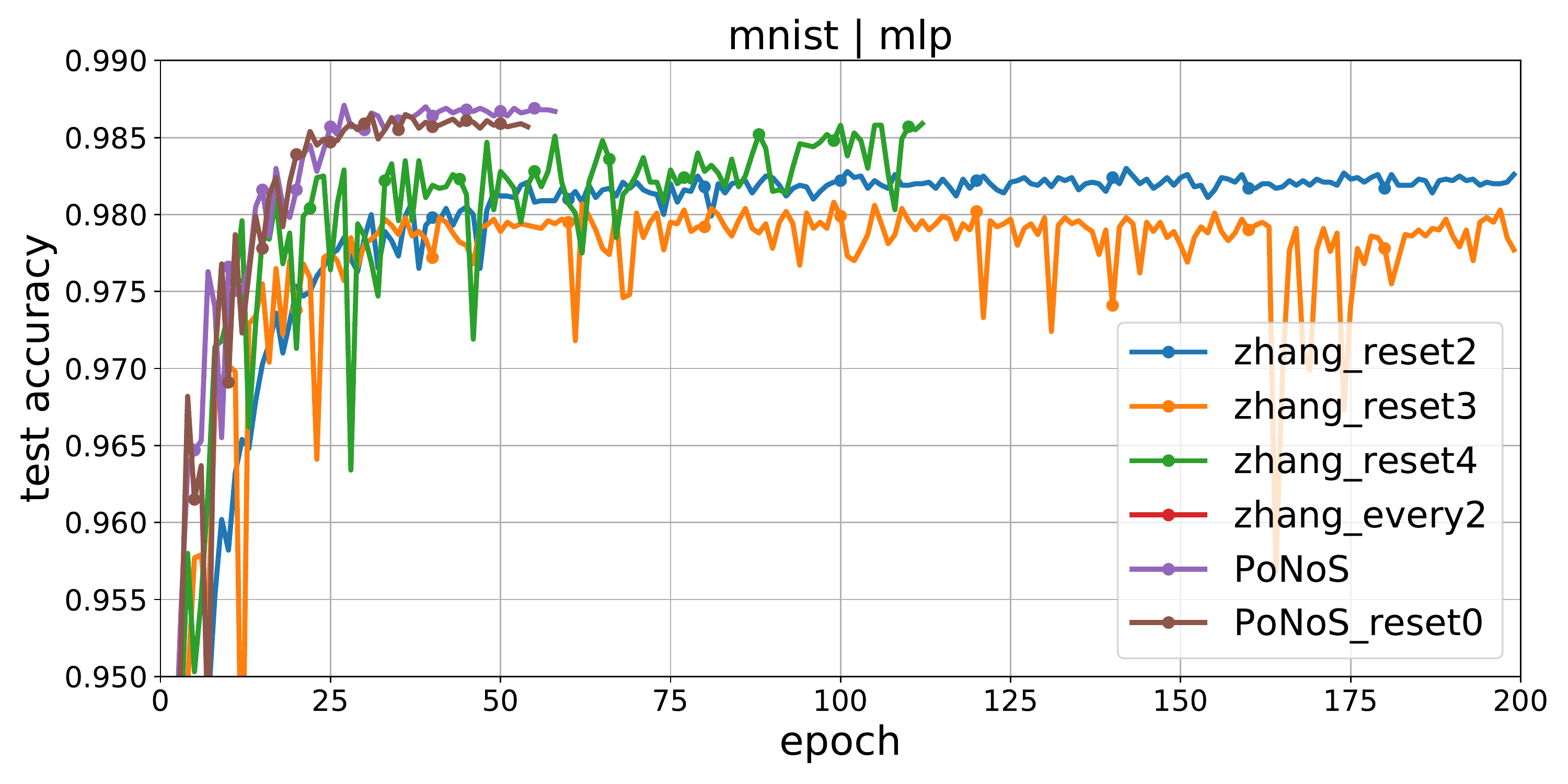}}
		\subfloat{\includegraphics[width=\imgS\linewidth]{\dir/\model/agv_step_size}}
		\subfloat{\includegraphics[width=\imgS\linewidth]{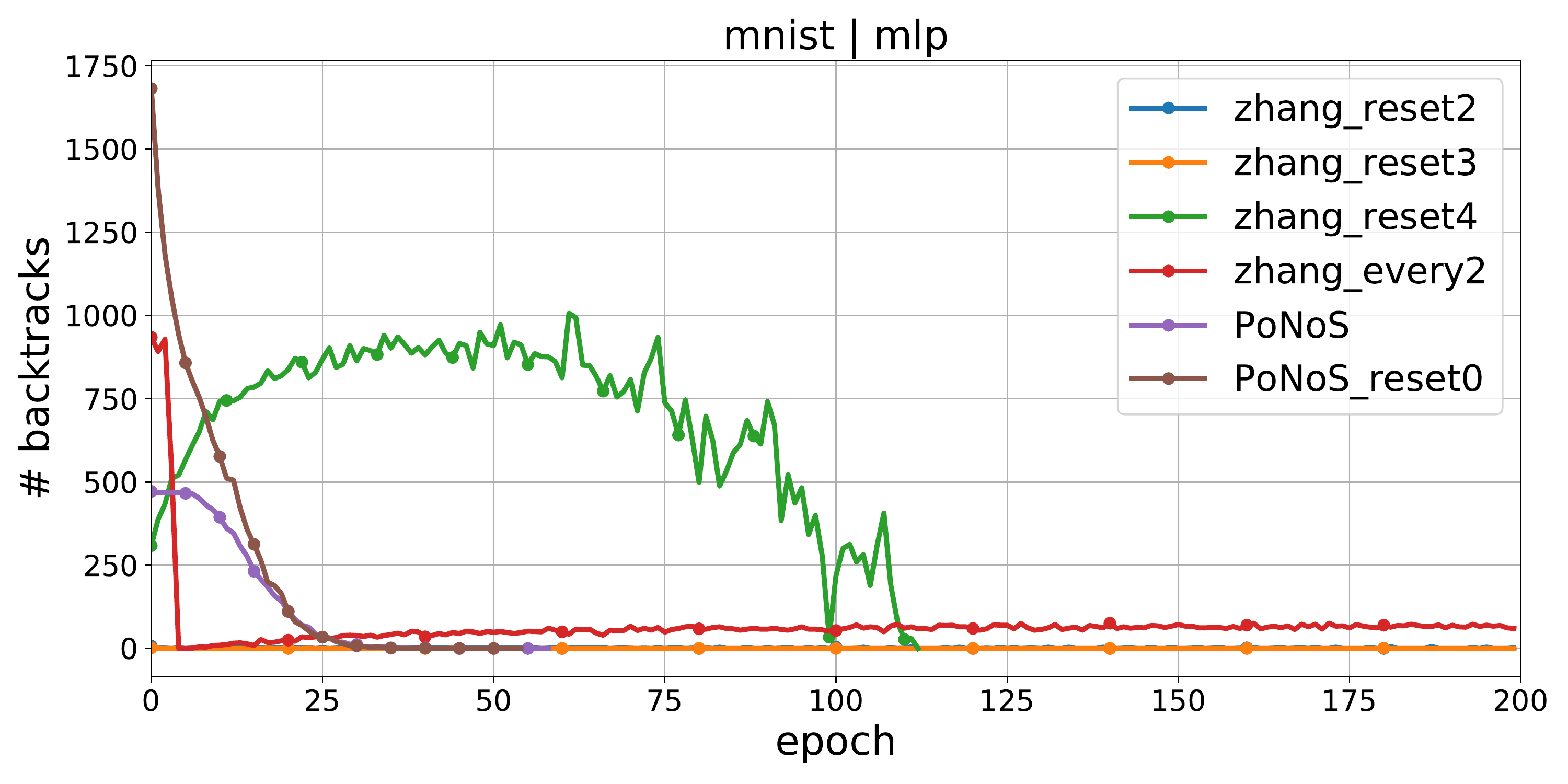}}\\
		\renewcommand{\model}{cifar10_resnet}
		\renewcommand{\modelname}{cifar10\_resnet}
		\subfloat{\includegraphics[width=\imgS\linewidth]{\dir/\model/train_loss}}
		\subfloat{\includegraphics[width=\imgS\linewidth]{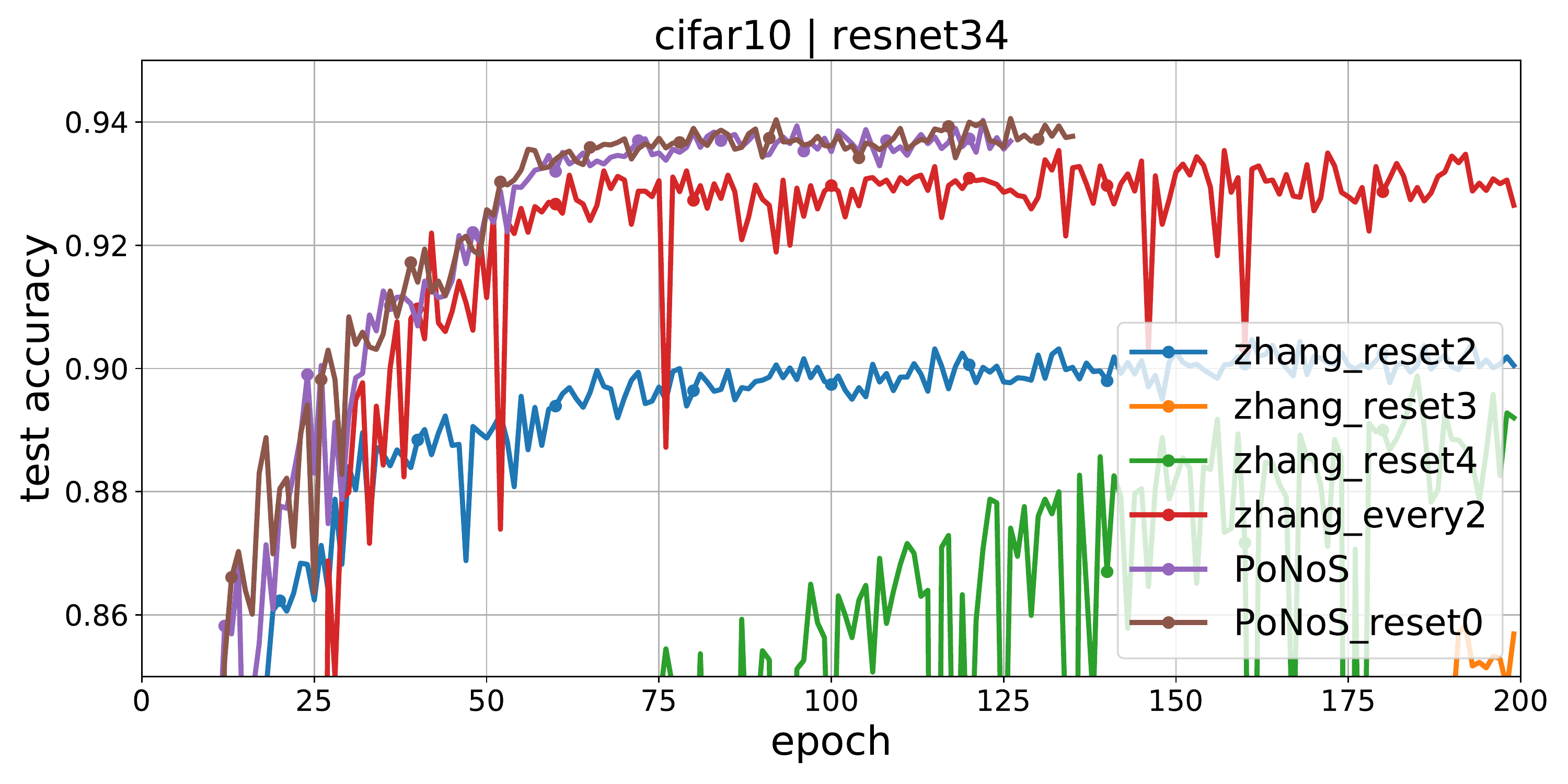}}
		\subfloat{\includegraphics[width=\imgS\linewidth]{\dir/\model/agv_step_size}}
		\subfloat{\includegraphics[width=\imgS\linewidth]{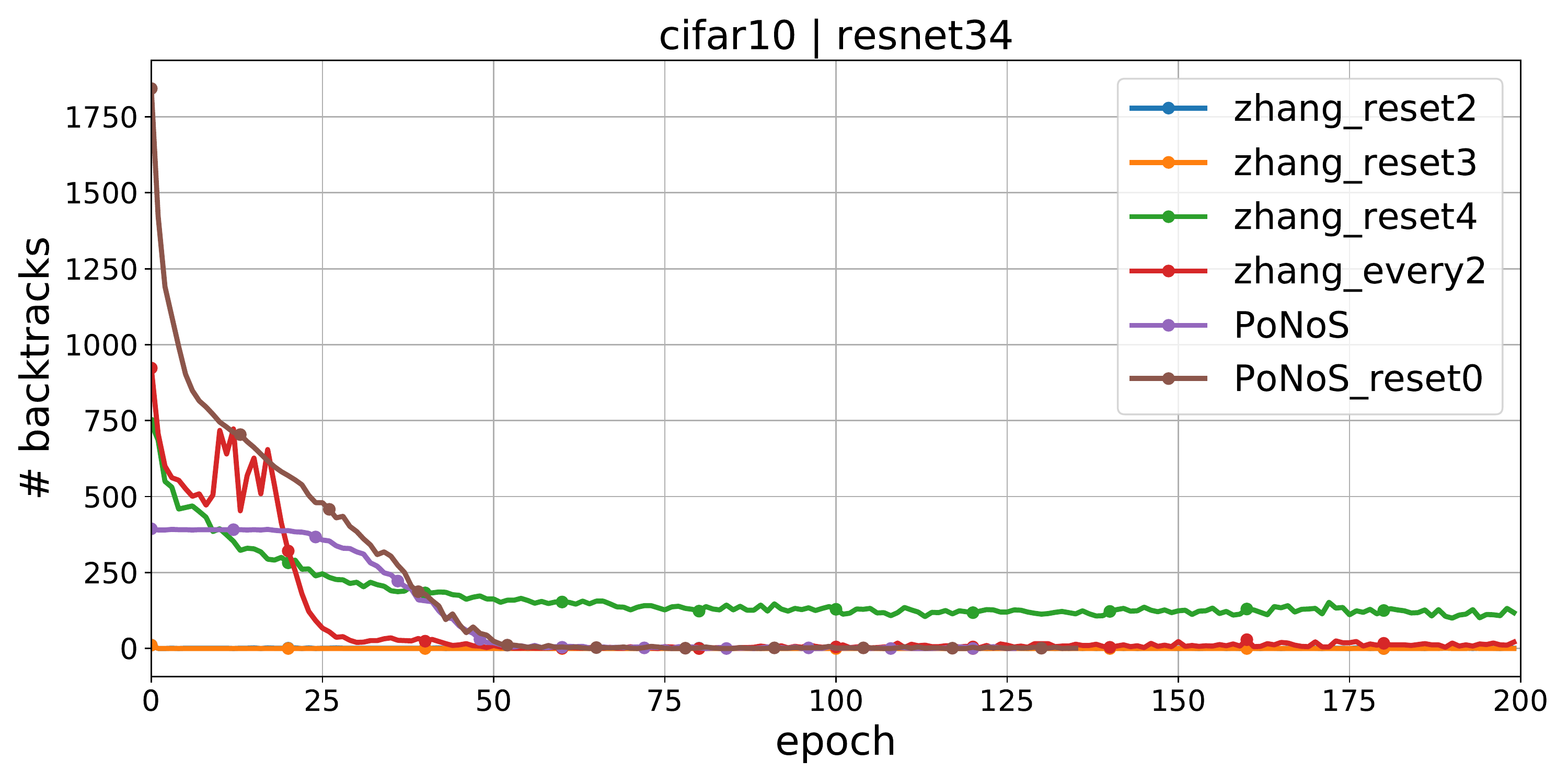}}\\
		\renewcommand{\model}{cifar10_densenet}
		\renewcommand{\modelname}{cifar10\_densenet}
		\subfloat{\includegraphics[width=\imgS\linewidth]{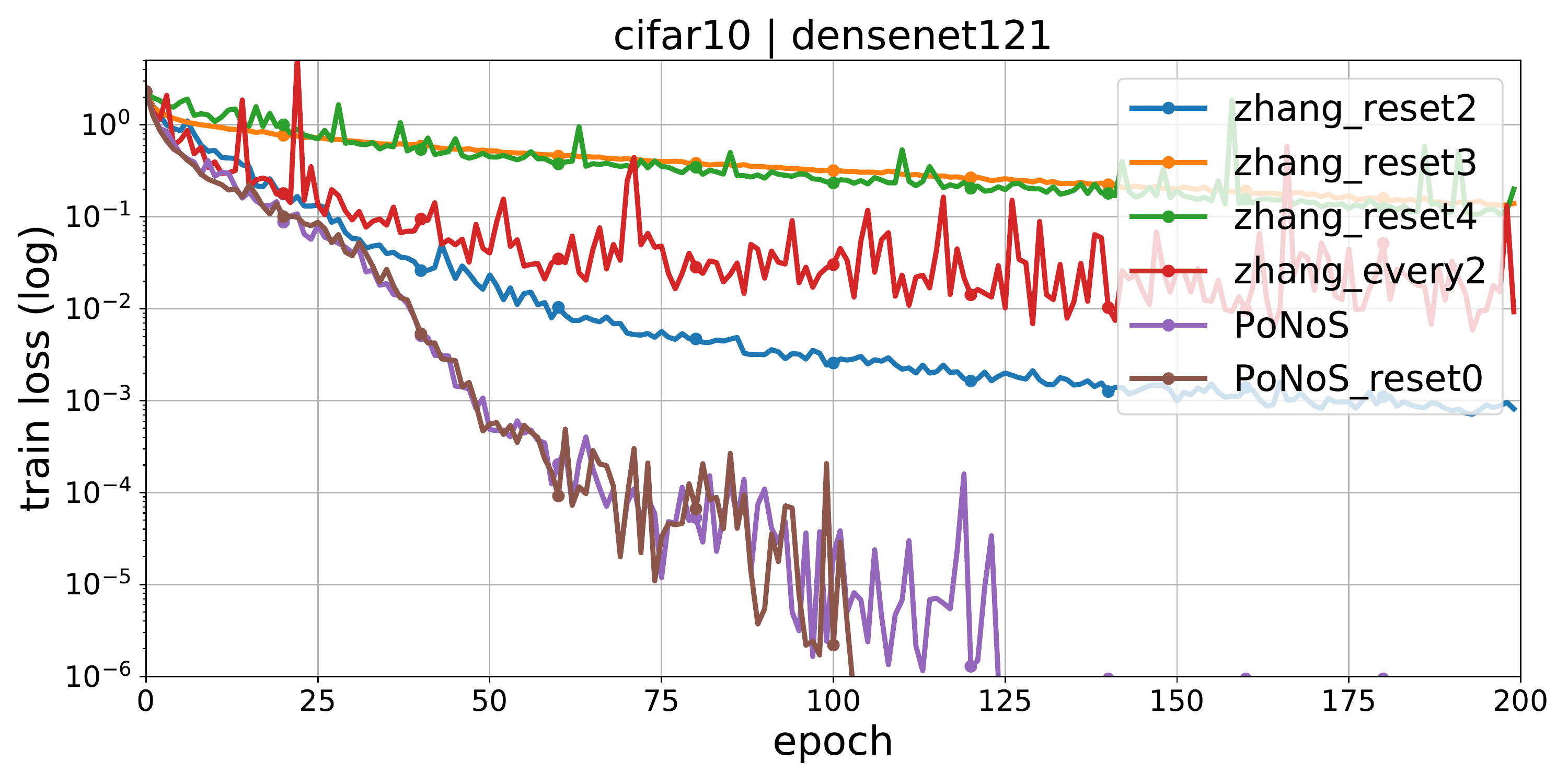}}
		\subfloat{\includegraphics[width=\imgS\linewidth]{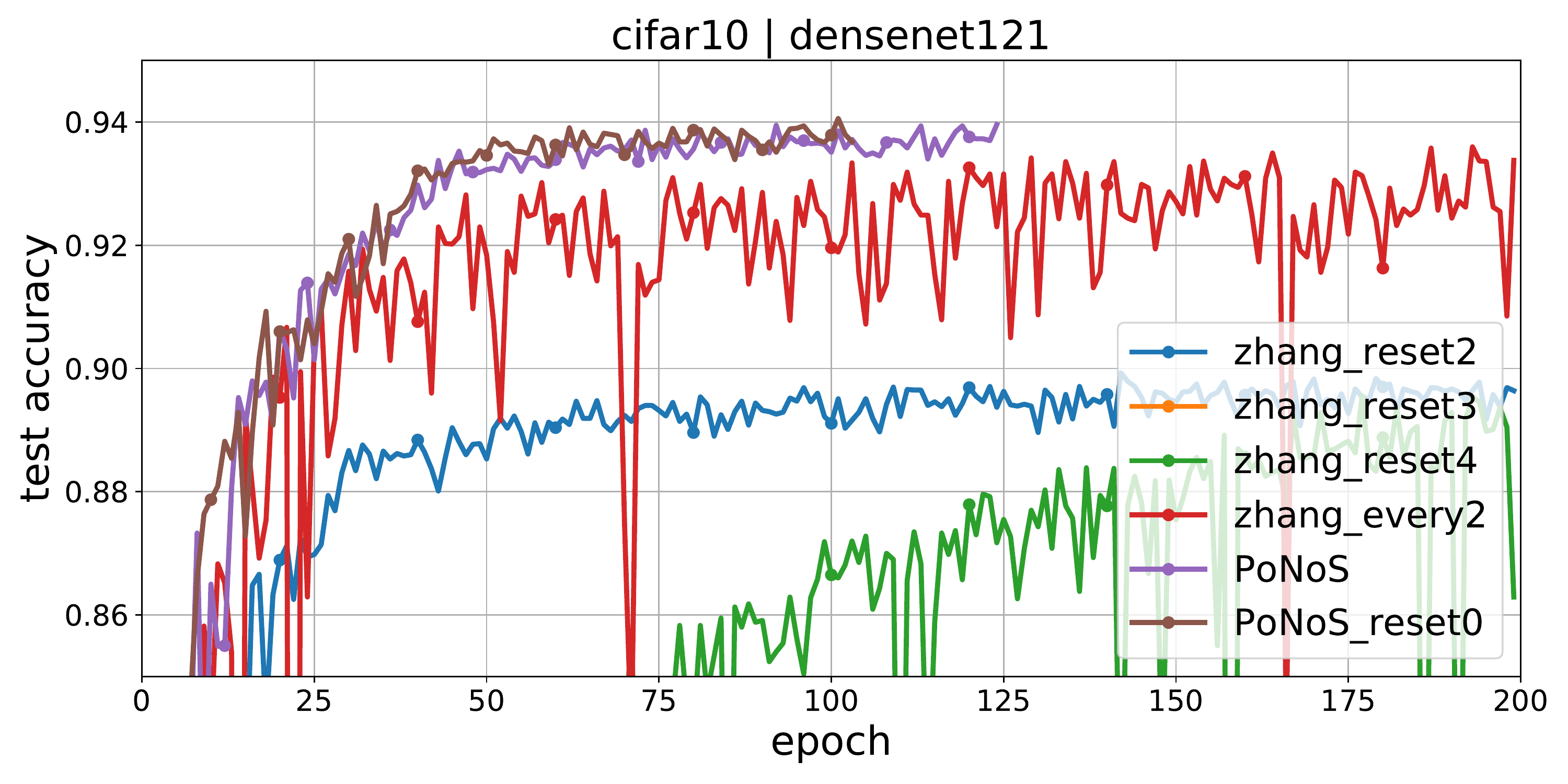}}
		\subfloat{\includegraphics[width=\imgS\linewidth]{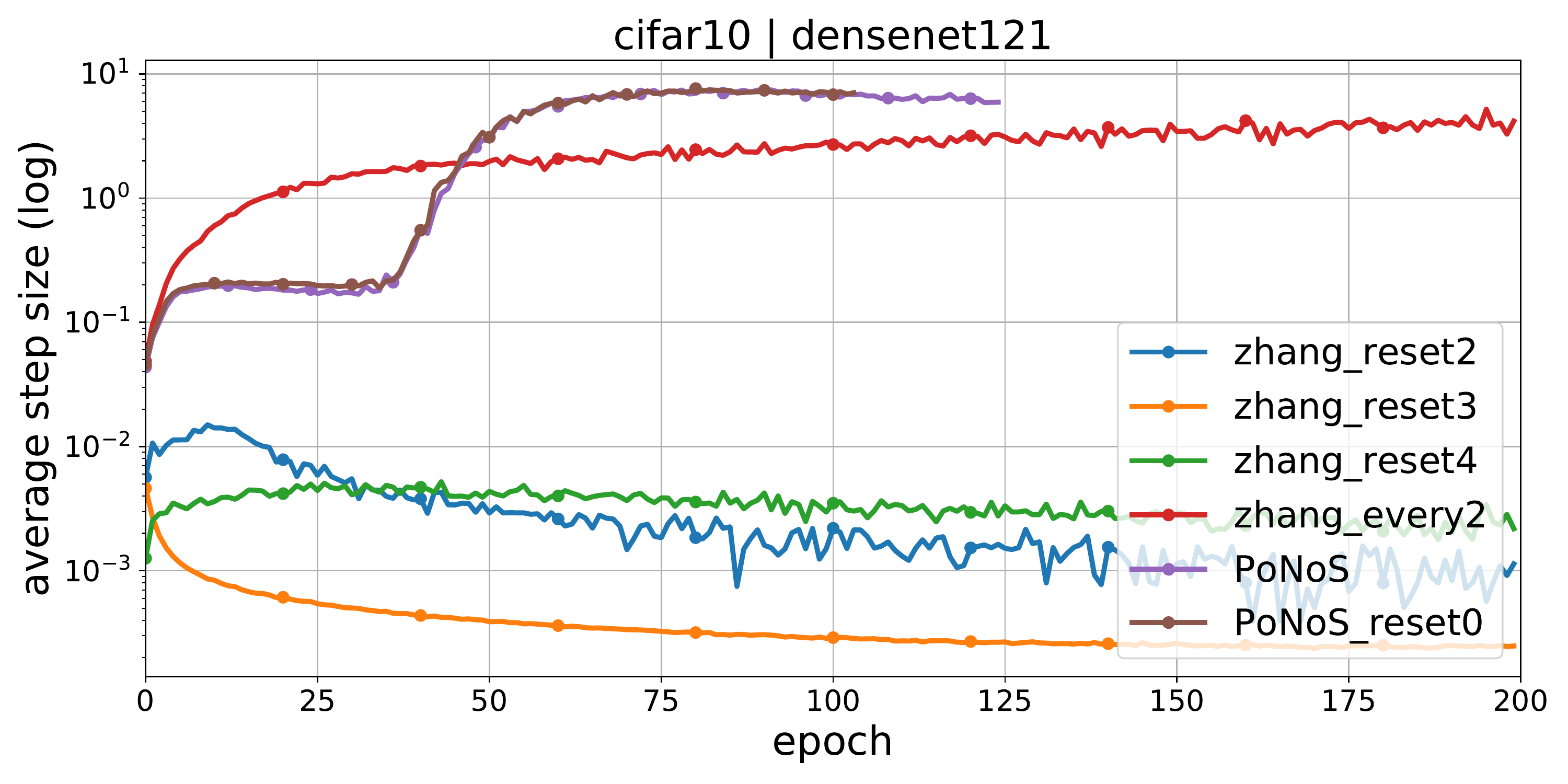}}
		\subfloat{\includegraphics[width=\imgS\linewidth]{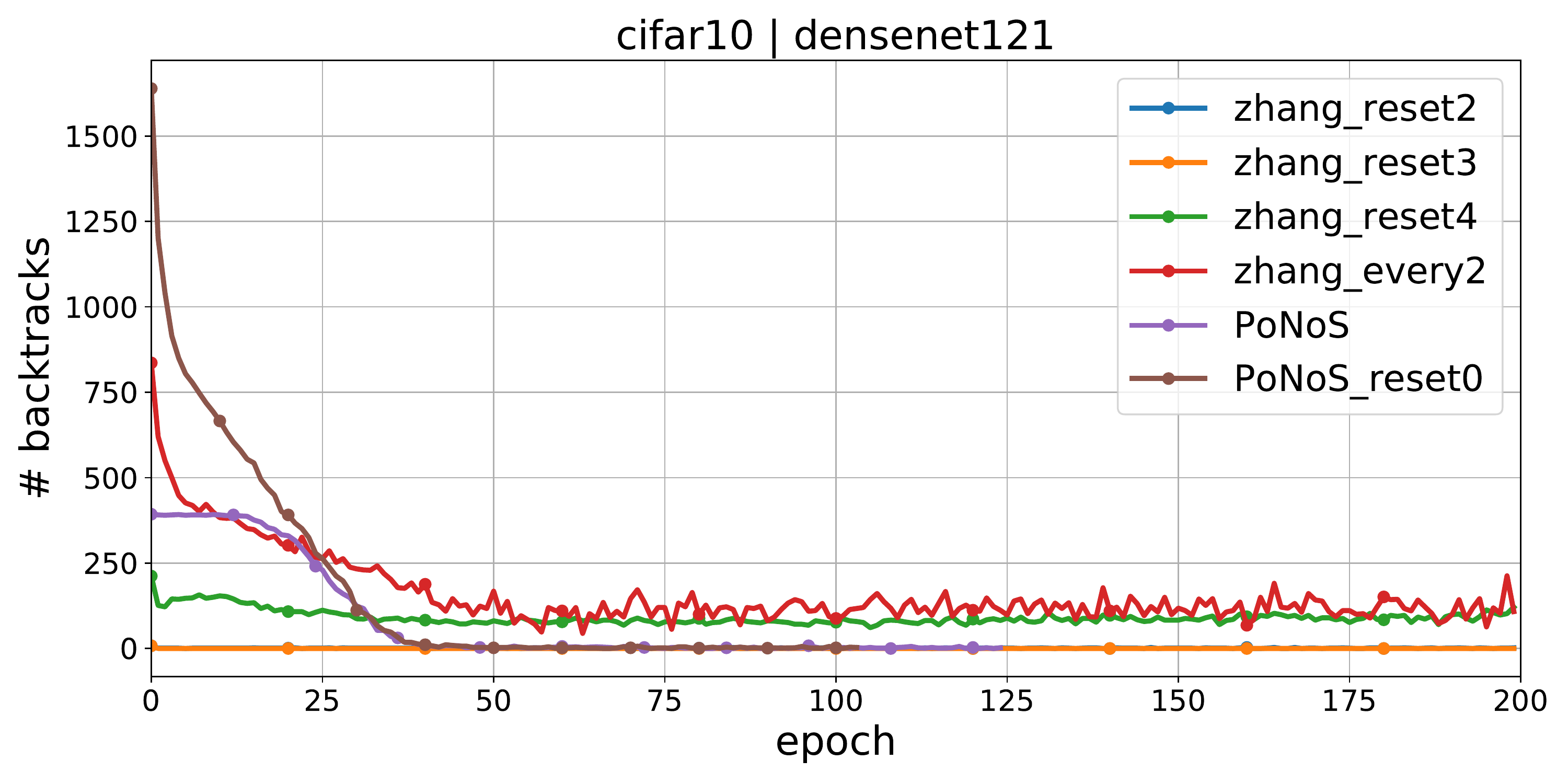}}\\
		\renewcommand{\model}{cifar100_resnet}
		\renewcommand{\modelname}{cifar100\_resnet}
		\subfloat{\includegraphics[width=\imgS\linewidth]{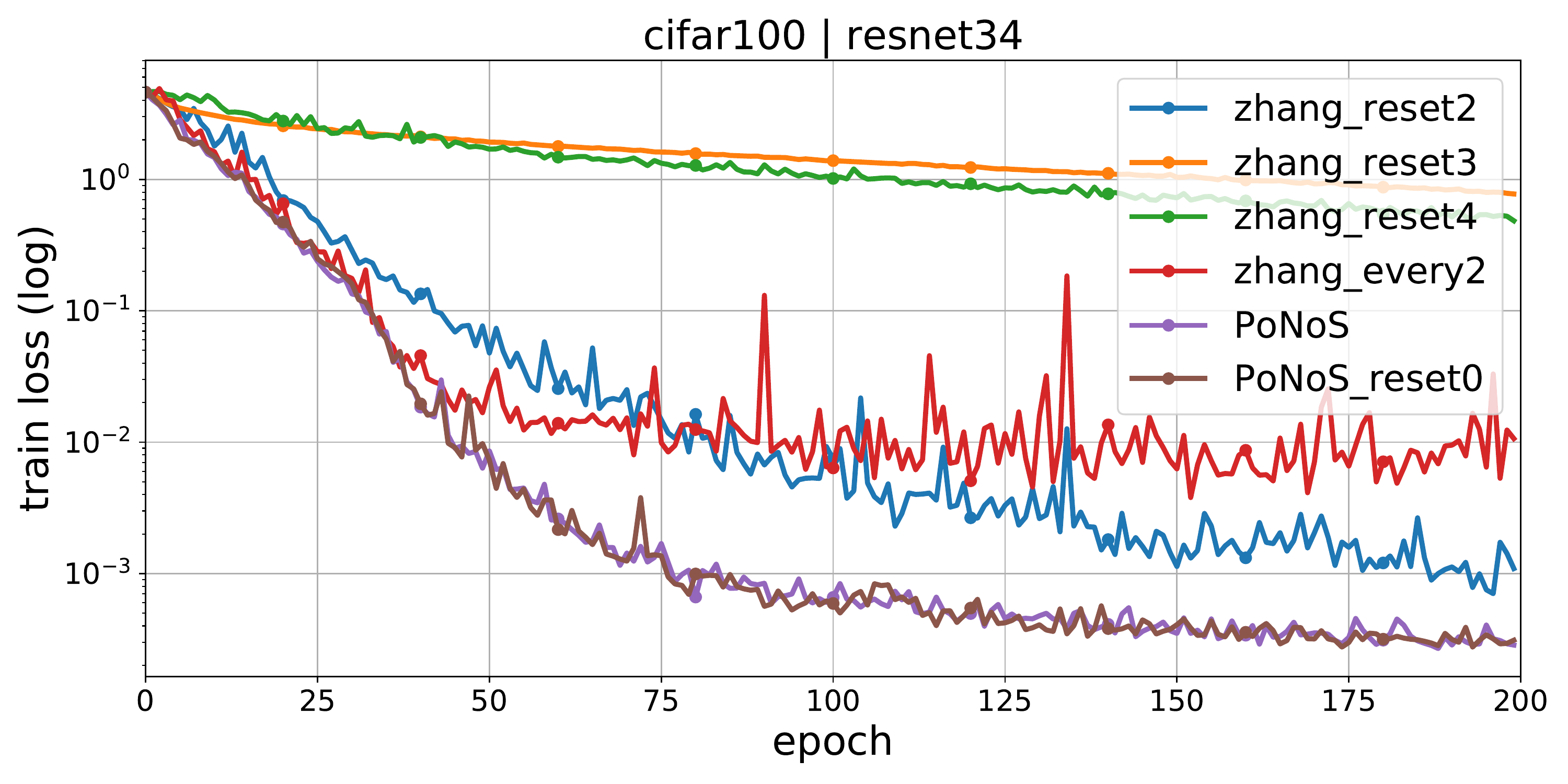}}
		\subfloat{\includegraphics[width=\imgS\linewidth]{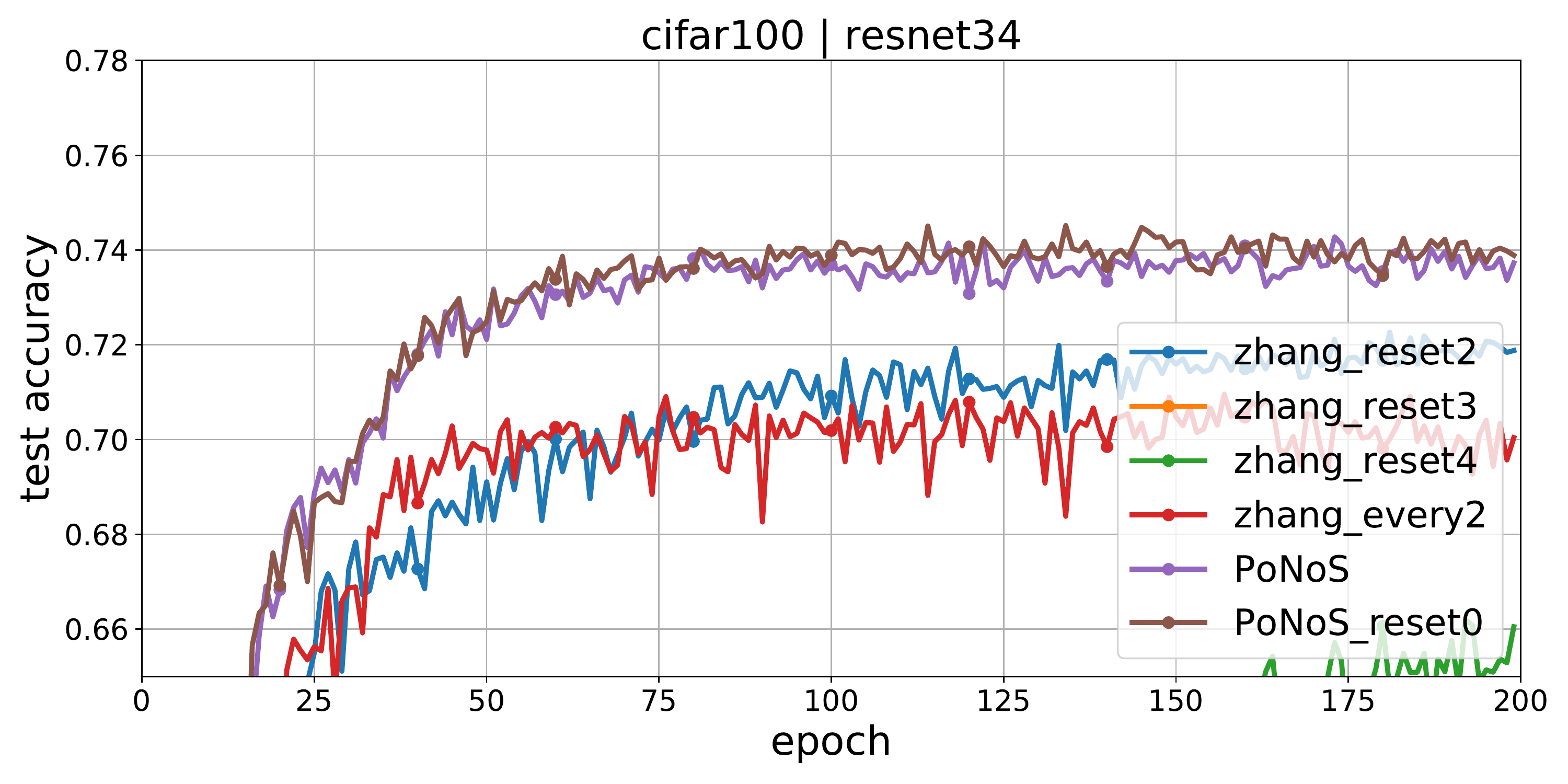}}
		\subfloat{\includegraphics[width=\imgS\linewidth]{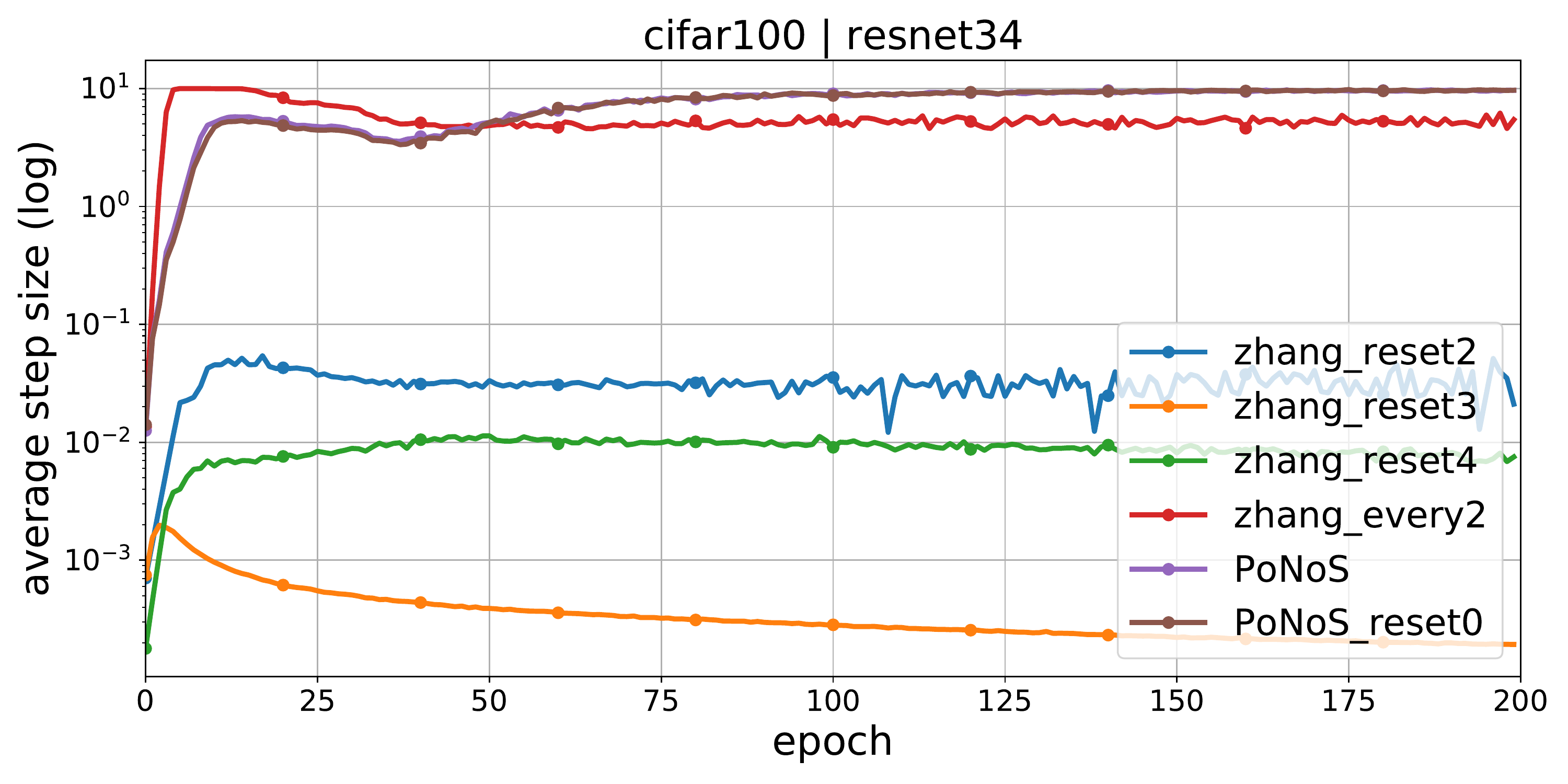}}
		\subfloat{\includegraphics[width=\imgS\linewidth]{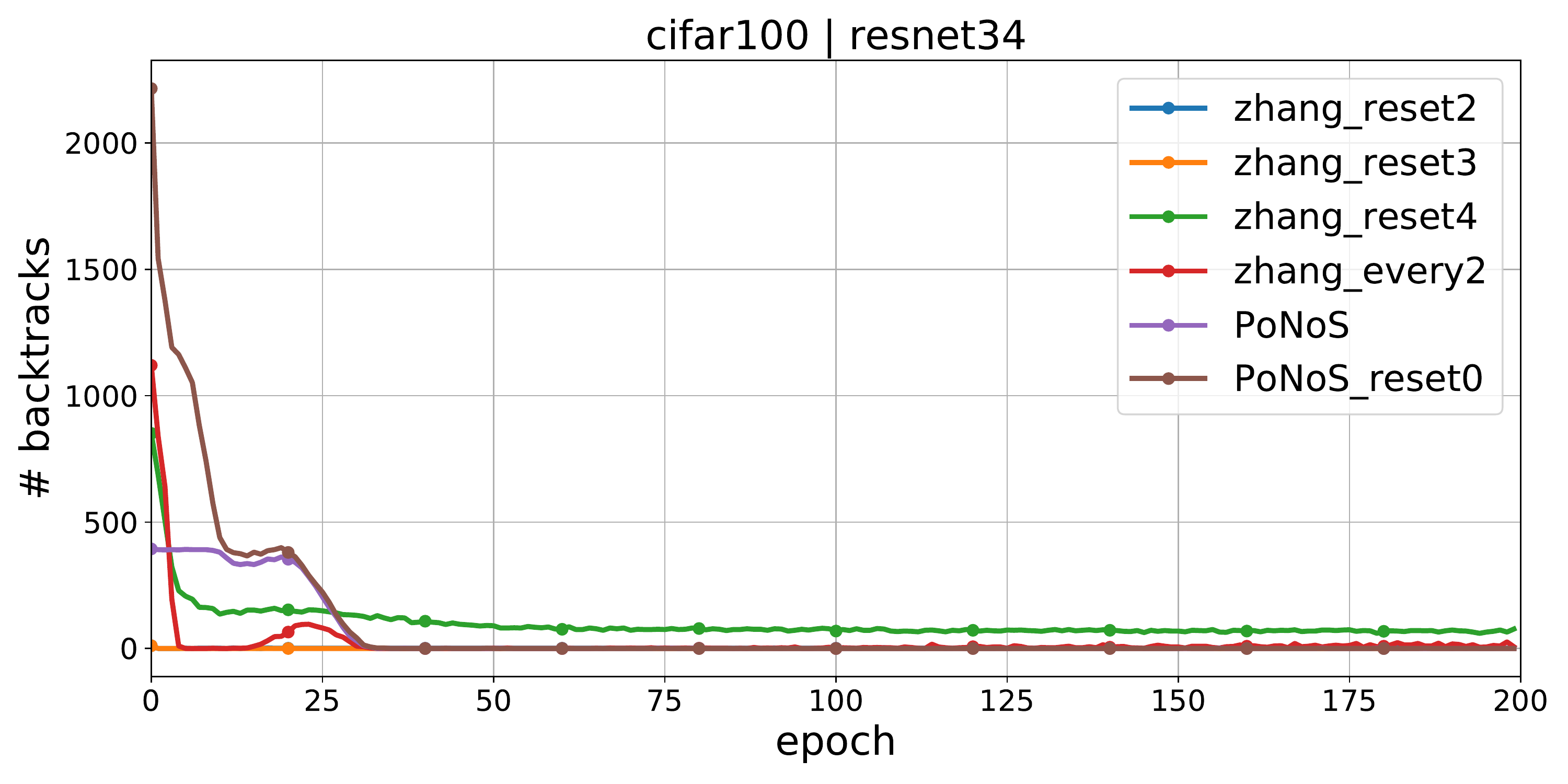}}\\
		\renewcommand{\model}{cifar100_densenet}
		\renewcommand{\modelname}{cifar100\_densenet}
		\subfloat{\includegraphics[width=\imgS\linewidth]{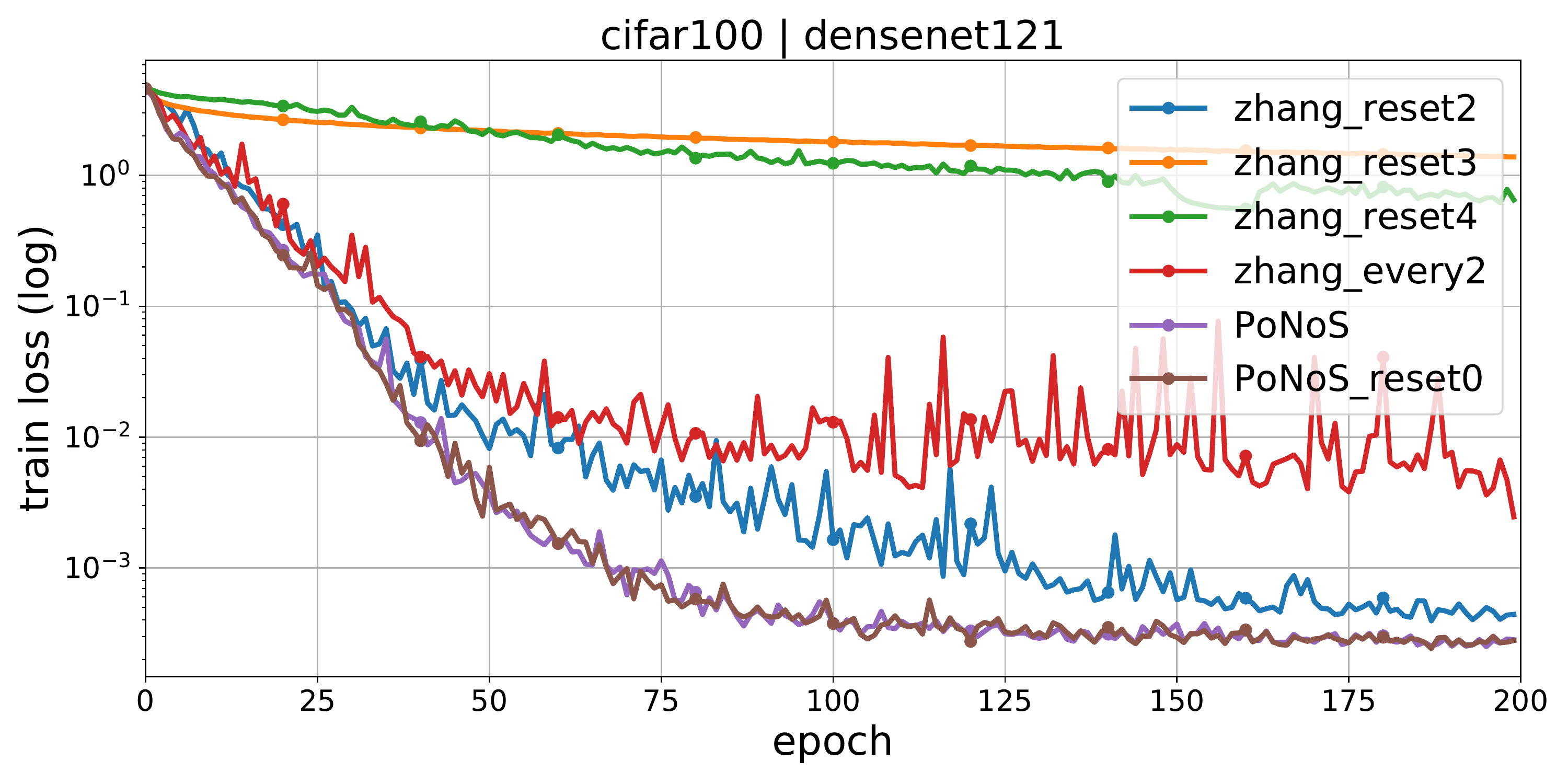}}
		\subfloat{\includegraphics[width=\imgS\linewidth]{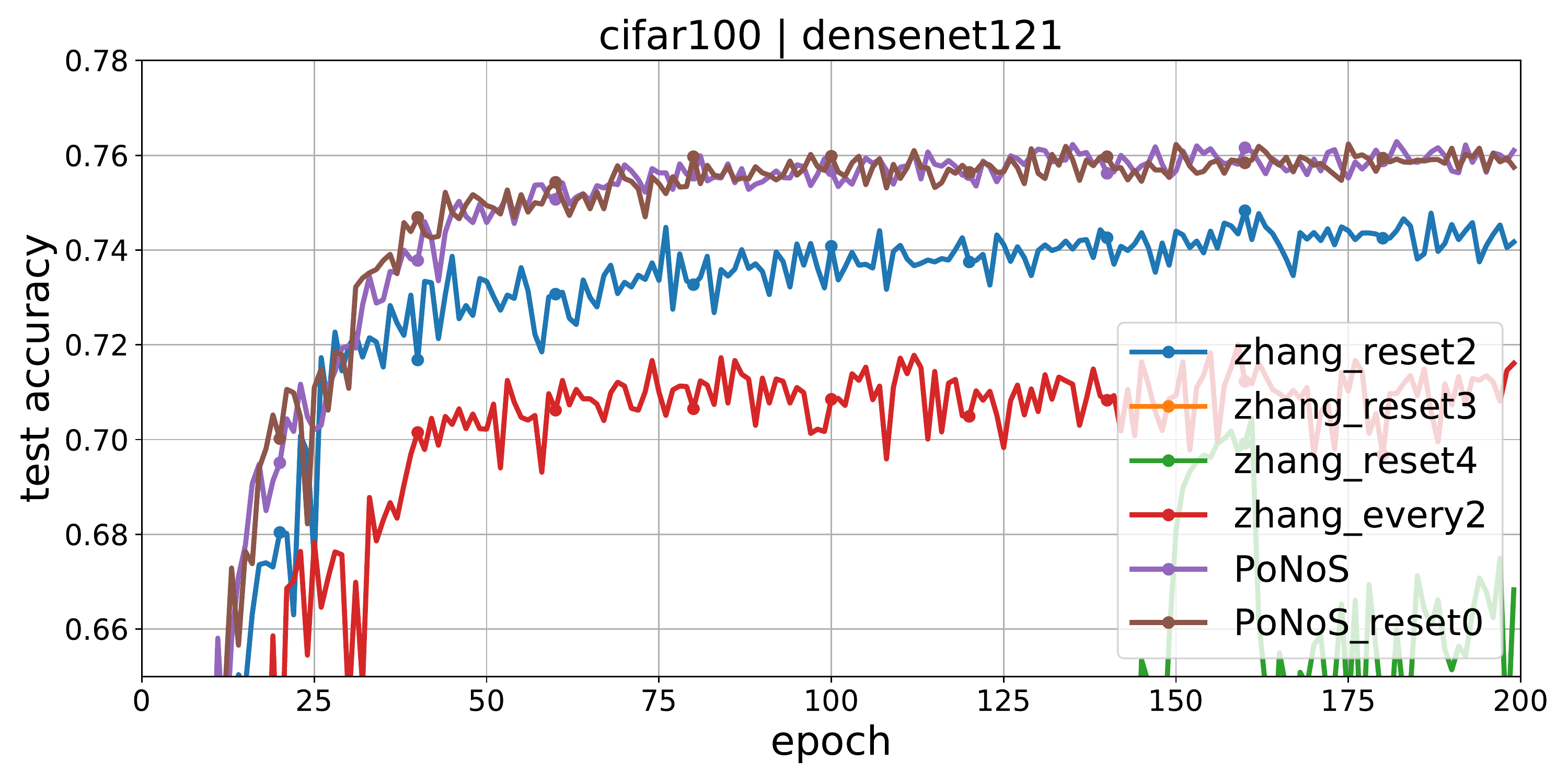}}
		\subfloat{\includegraphics[width=\imgS\linewidth]{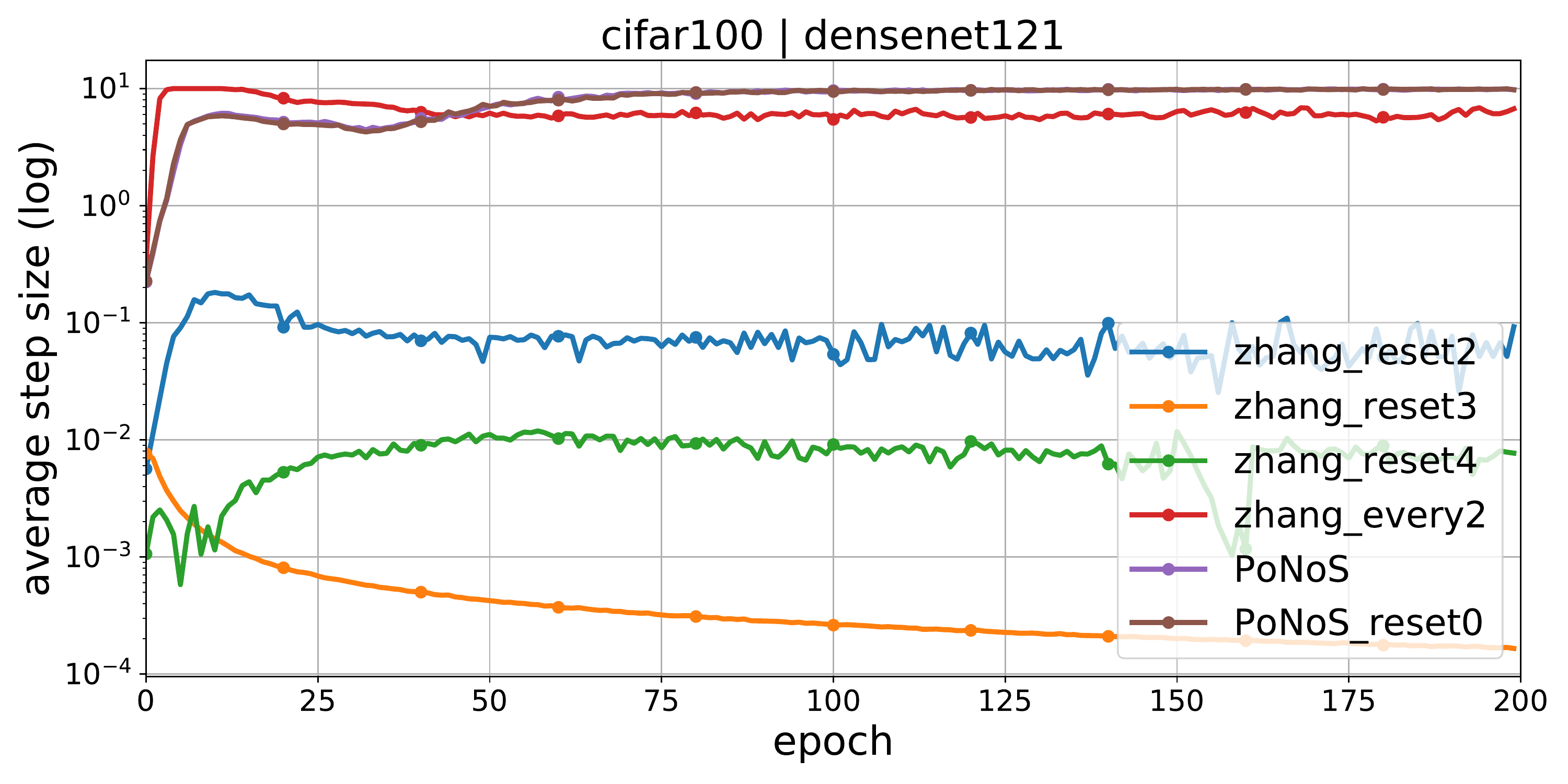}}
		\subfloat{\includegraphics[width=\imgS\linewidth]{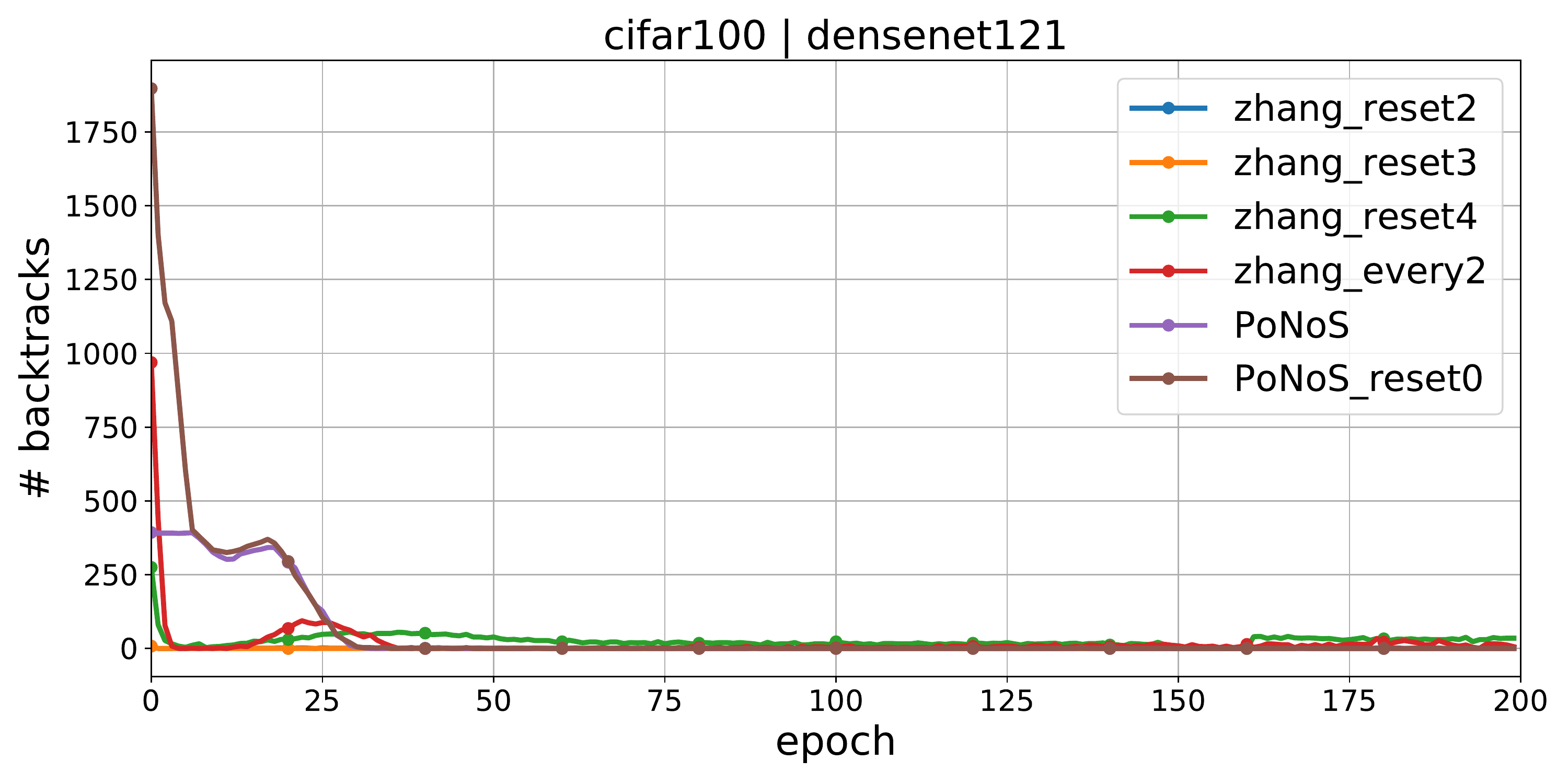}}\\
		\renewcommand{\model}{fashion_effb1}
		\renewcommand{\modelname}{fashion\_effb1}
		\subfloat{\includegraphics[width=\imgS\linewidth]{\dir/\model/train_loss}}
		\subfloat{\includegraphics[width=\imgS\linewidth]{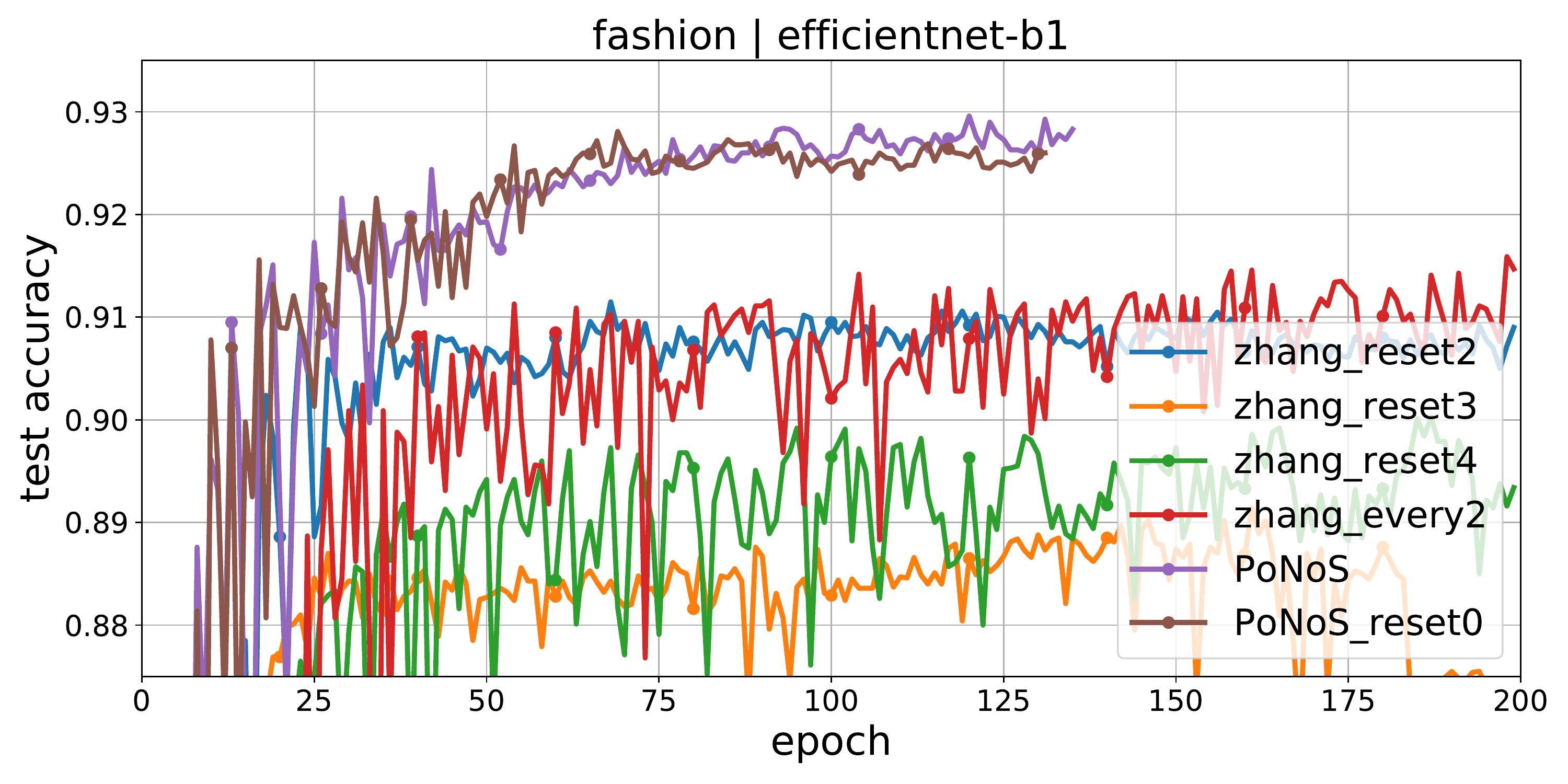}}
		\subfloat{\includegraphics[width=\imgS\linewidth]{\dir/\model/agv_step_size}}
		\subfloat{\includegraphics[width=\imgS\linewidth]{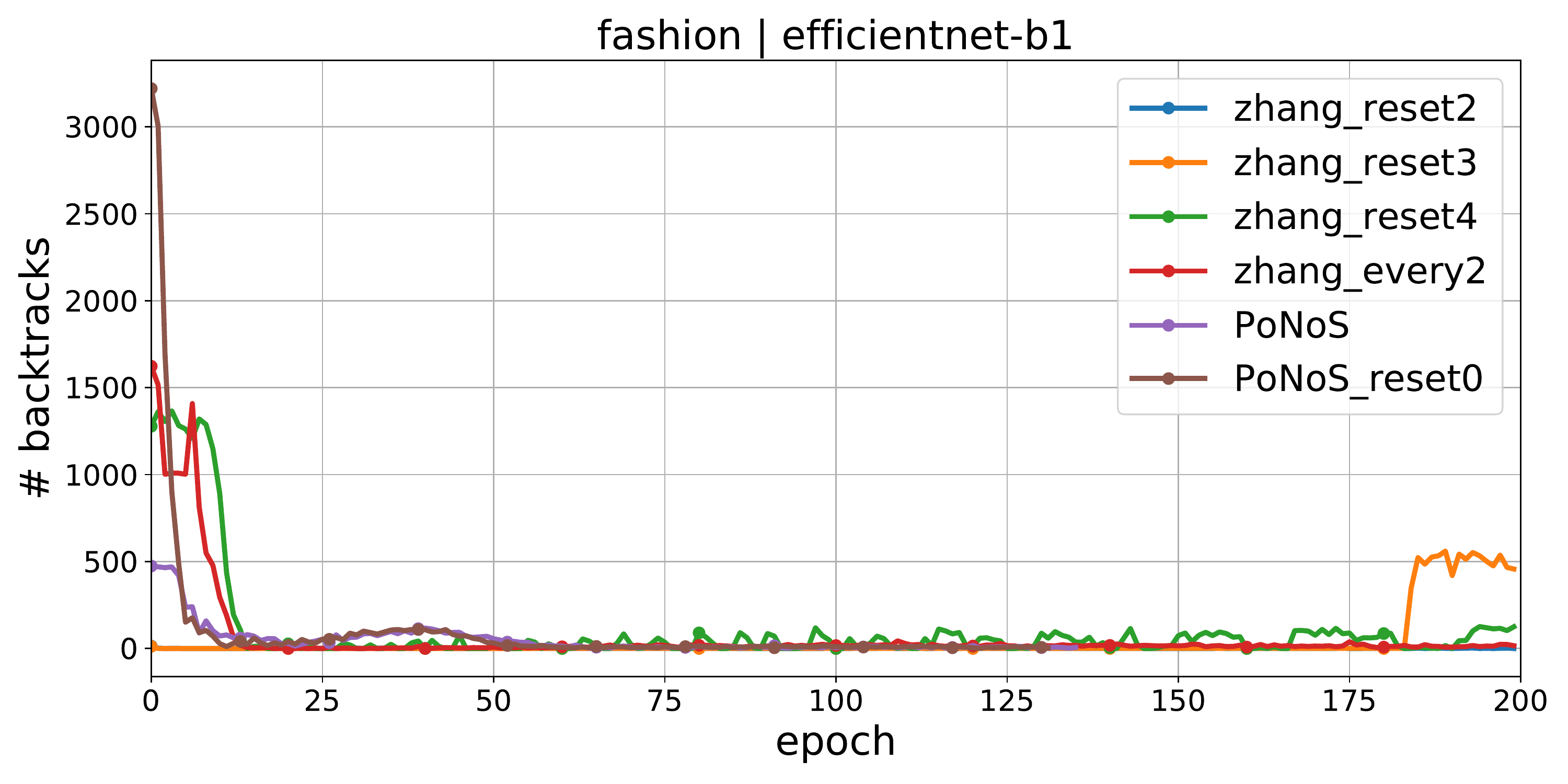}}\\
		\renewcommand{\model}{svhn_wrn}
		\renewcommand{\modelname}{svhn\_wrn}
		\subfloat{\includegraphics[width=\imgS\linewidth]{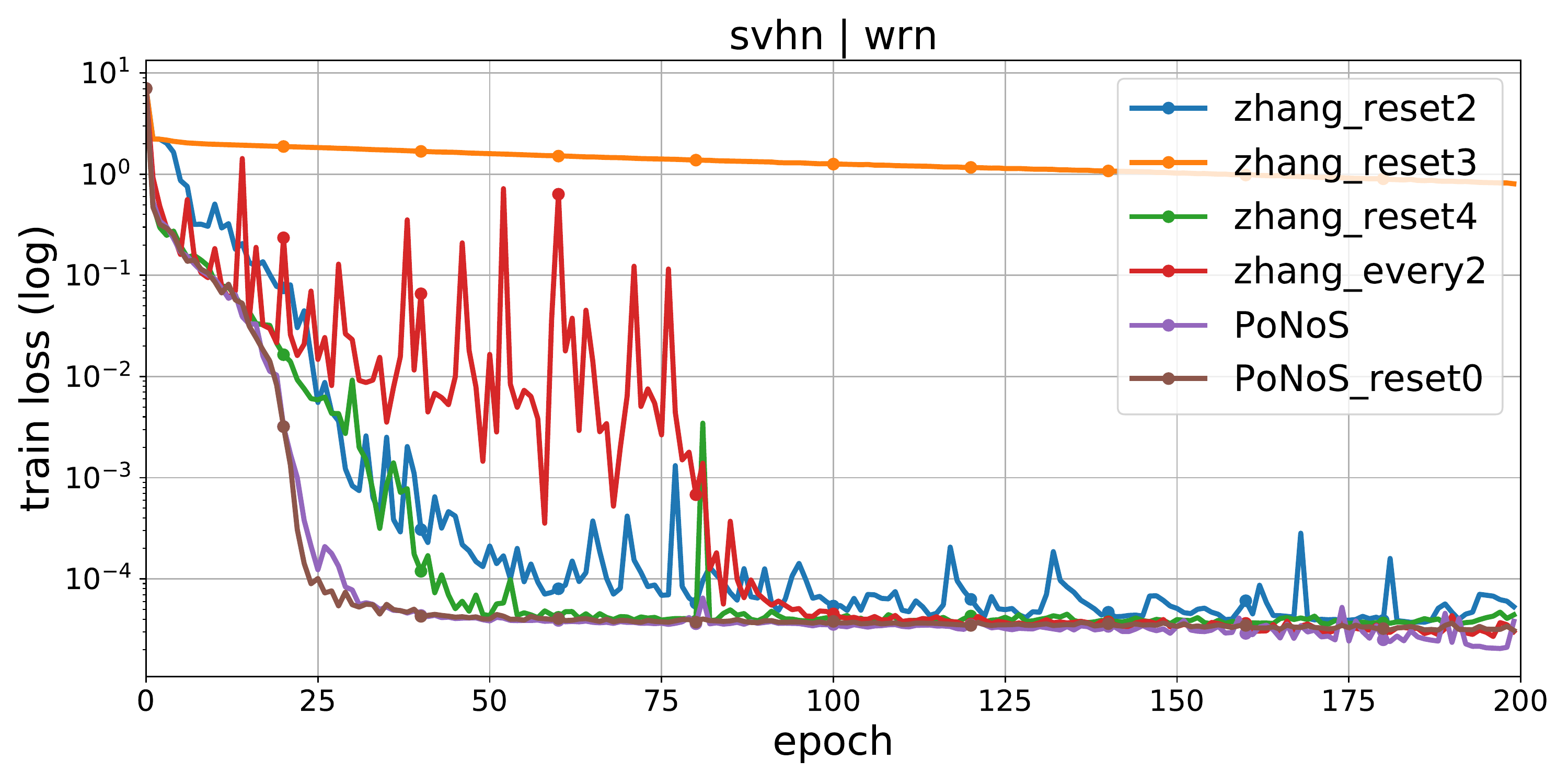}}
		\subfloat{\includegraphics[width=\imgS\linewidth]{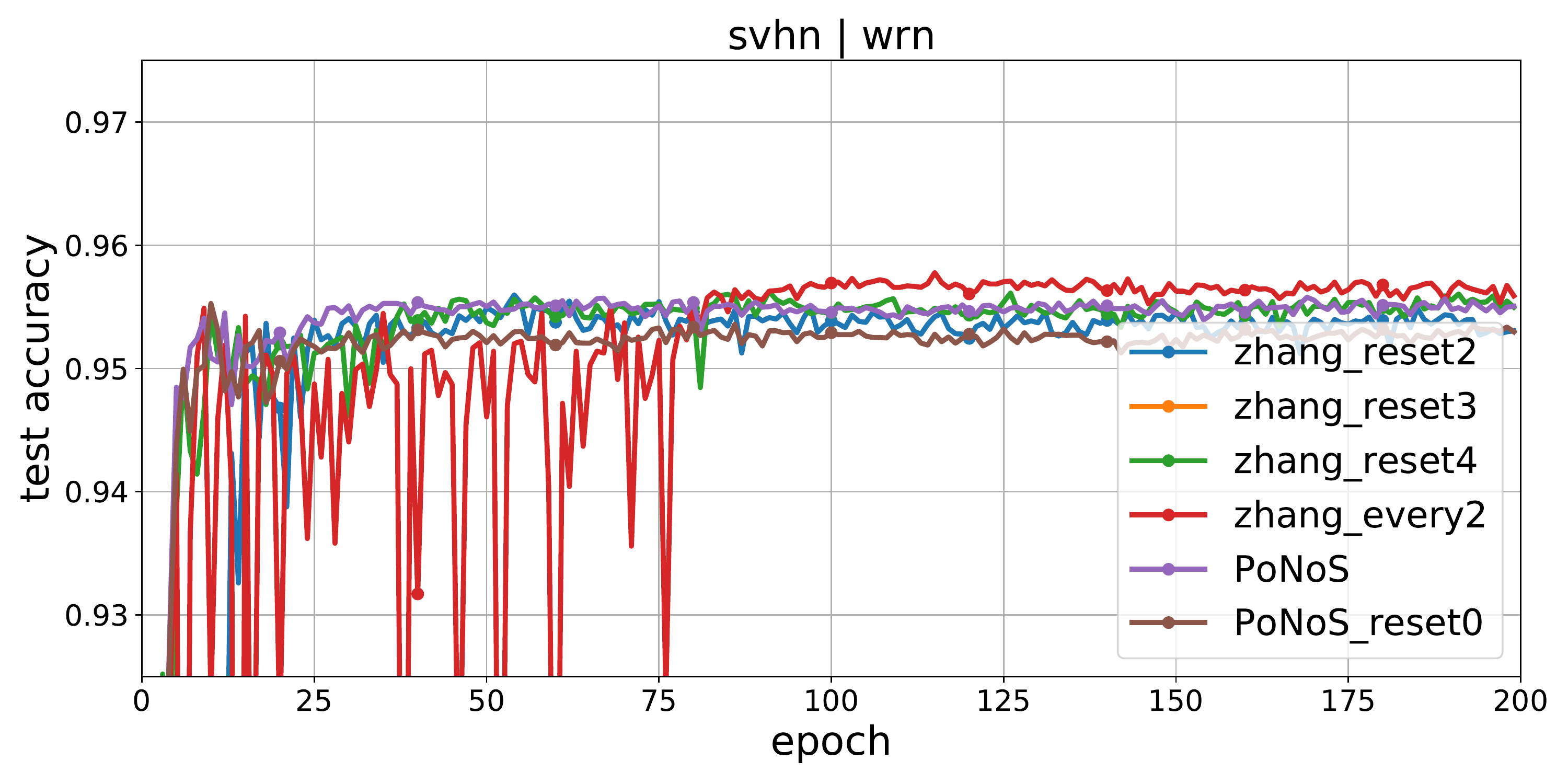}}
		\subfloat{\includegraphics[width=\imgS\linewidth]{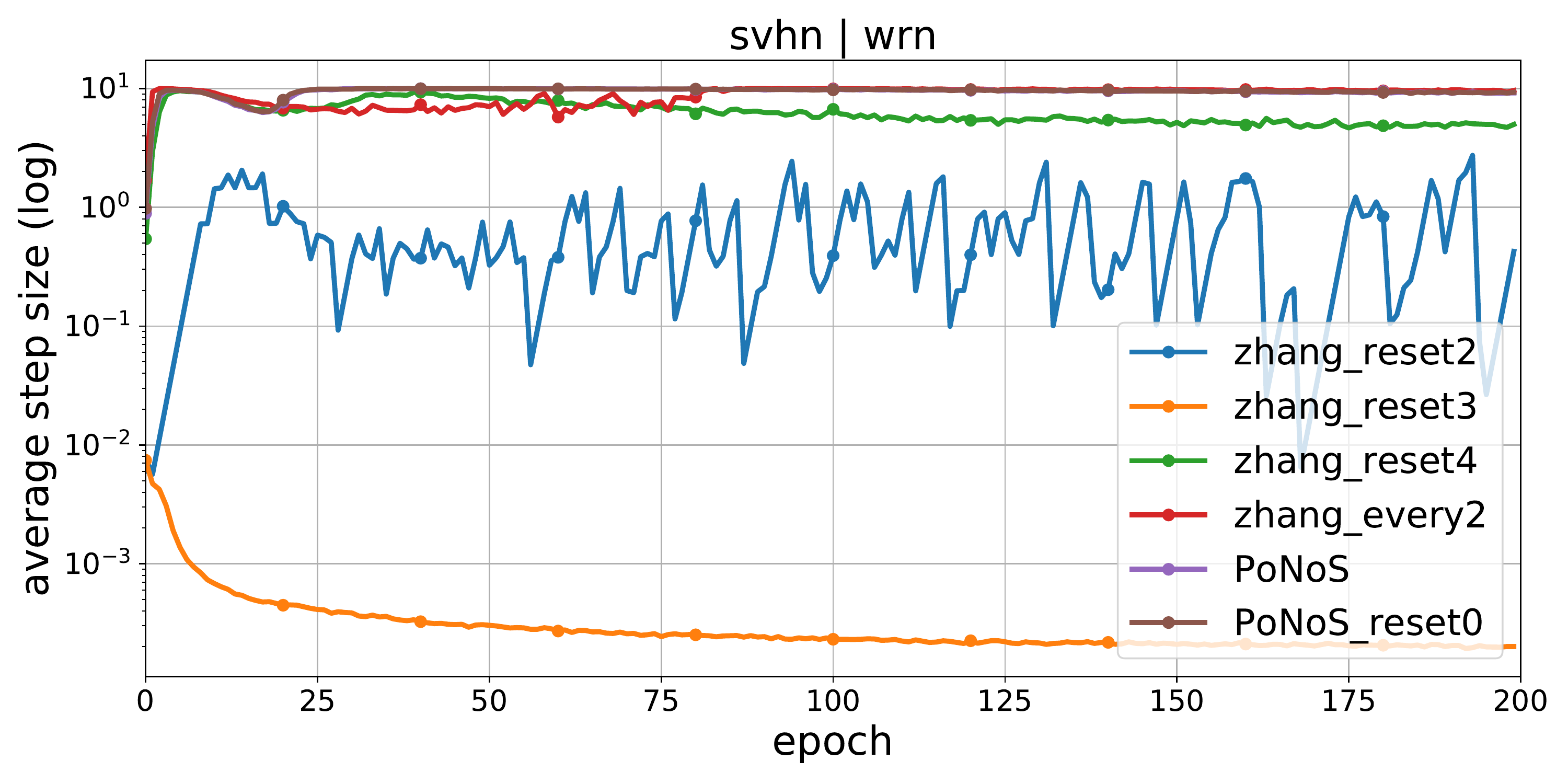}}
		\subfloat{\includegraphics[width=\imgS\linewidth]{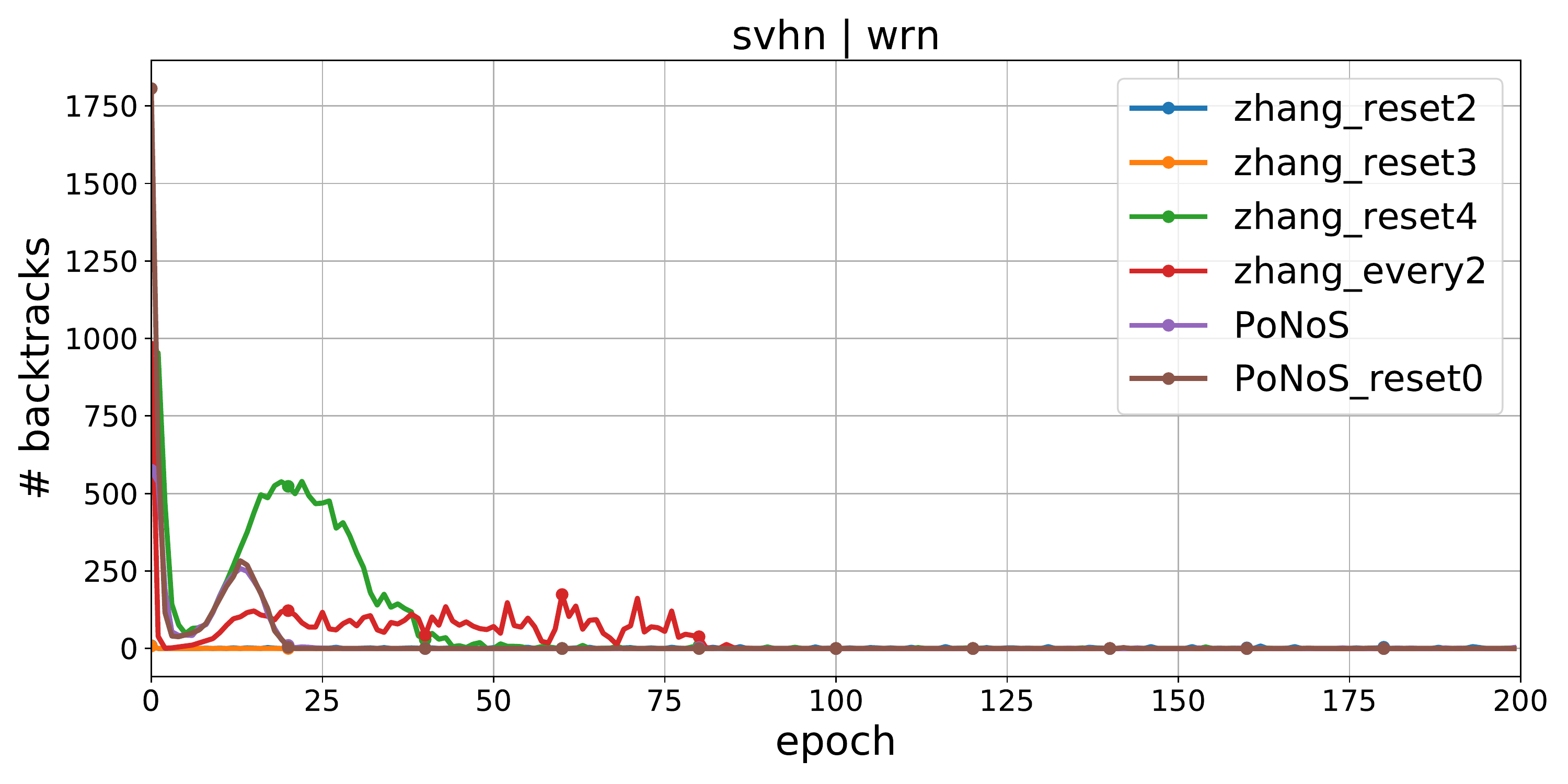}}
		\caption{Comparison between different initial step sizes and resetting techniques. Each row focus on a dataset/model combination. First column: train loss. Second column: test accuracy. Third column: average step size of the epoch. Fourth column: cumulative number of backtracks in the epoch.}\label{fig:reset}
	\end{minipage}
\end{figure}
	\subsection{Time Comparison}\label{sec:supp_time}
	See Figure \ref{fig:time} for the complete results relative to Figure 3 of the main paper. All the time measures are obtained by using the \texttt{time.time()} command from the \texttt{time python} package. The cumulative time in the $x$-axis of the first and second columns of Figure \ref{fig:time} has been computed as an average of 5 different runs. More precisely, the time of each epoch have been computed separately for the 5 runs and then averaged, so that the same average time has been assigned to each of the 5 runs. At this point, these averaged per-epoch times have been cumulated along the epochs. The shaded error bars in the plots correspond to the standard deviation from the mean of the loss/accuracy and not to the runtime. The runtime standard deviation is instead reported in the third column of Figure \ref{fig:time}.

\renewcommand{\dir}{time/}
\renewcommand{\imgS}{0.37}
\begin{figure}[!h] 
	\hspace{-3mm}
	\begin{minipage}{0.9\textwidth}
		\renewcommand{\model}{mlp}
		\renewcommand{\modelname}{mlp}
		\subfloat{\includegraphics[width=\imgS\linewidth]{\dir/\model/train_loss}}
		\subfloat{\includegraphics[width=\imgS\linewidth]{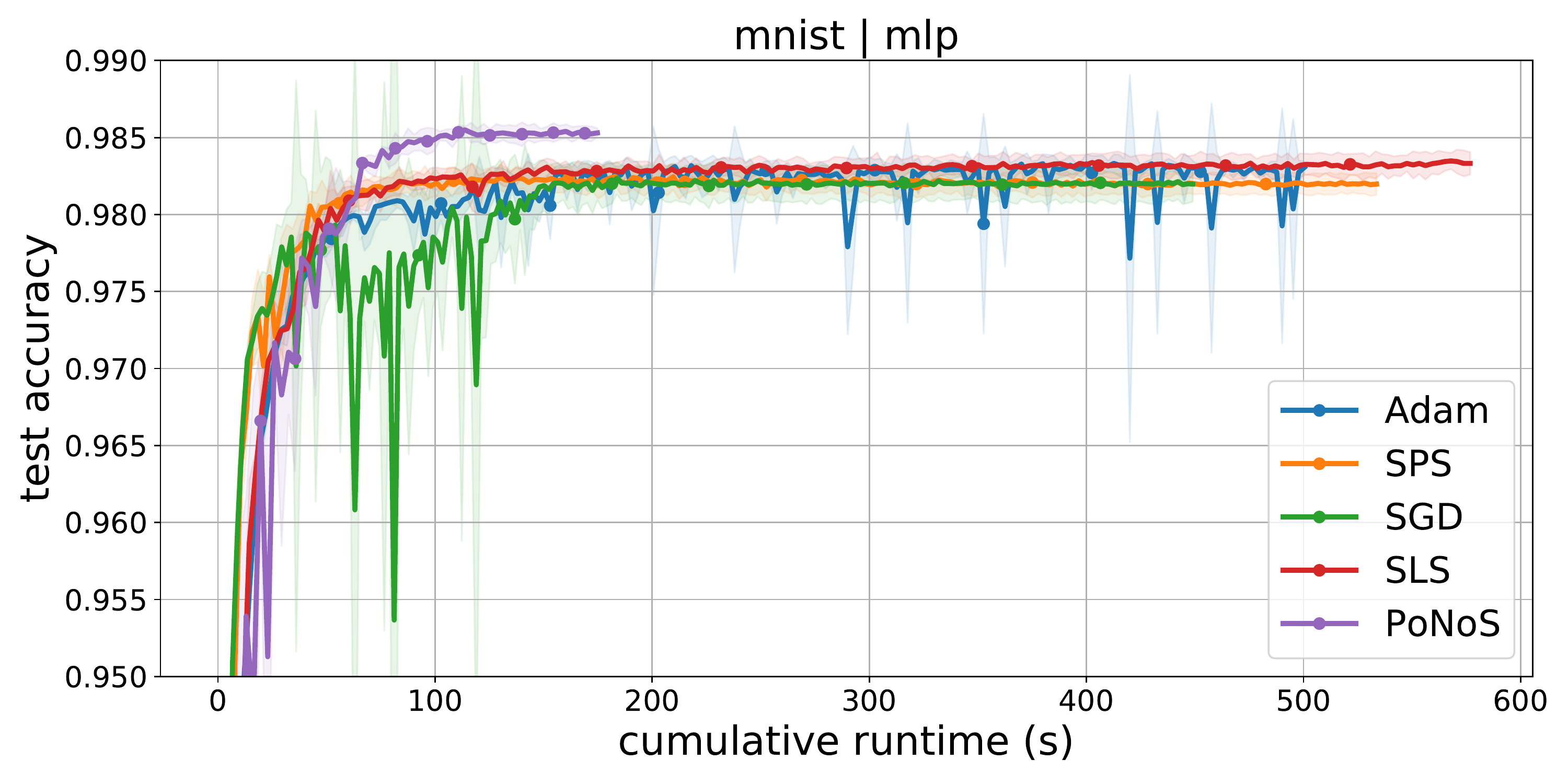}}
		\subfloat{\includegraphics[width=\imgS\linewidth]{exp1/\model/train_epoch_time}}\\
		\renewcommand{\model}{cifar10_resnet}
		\renewcommand{\modelname}{cifar10\_resnet}
		\subfloat{\includegraphics[width=\imgS\linewidth]{\dir/\model/train_loss}}
		\subfloat{\includegraphics[width=\imgS\linewidth]{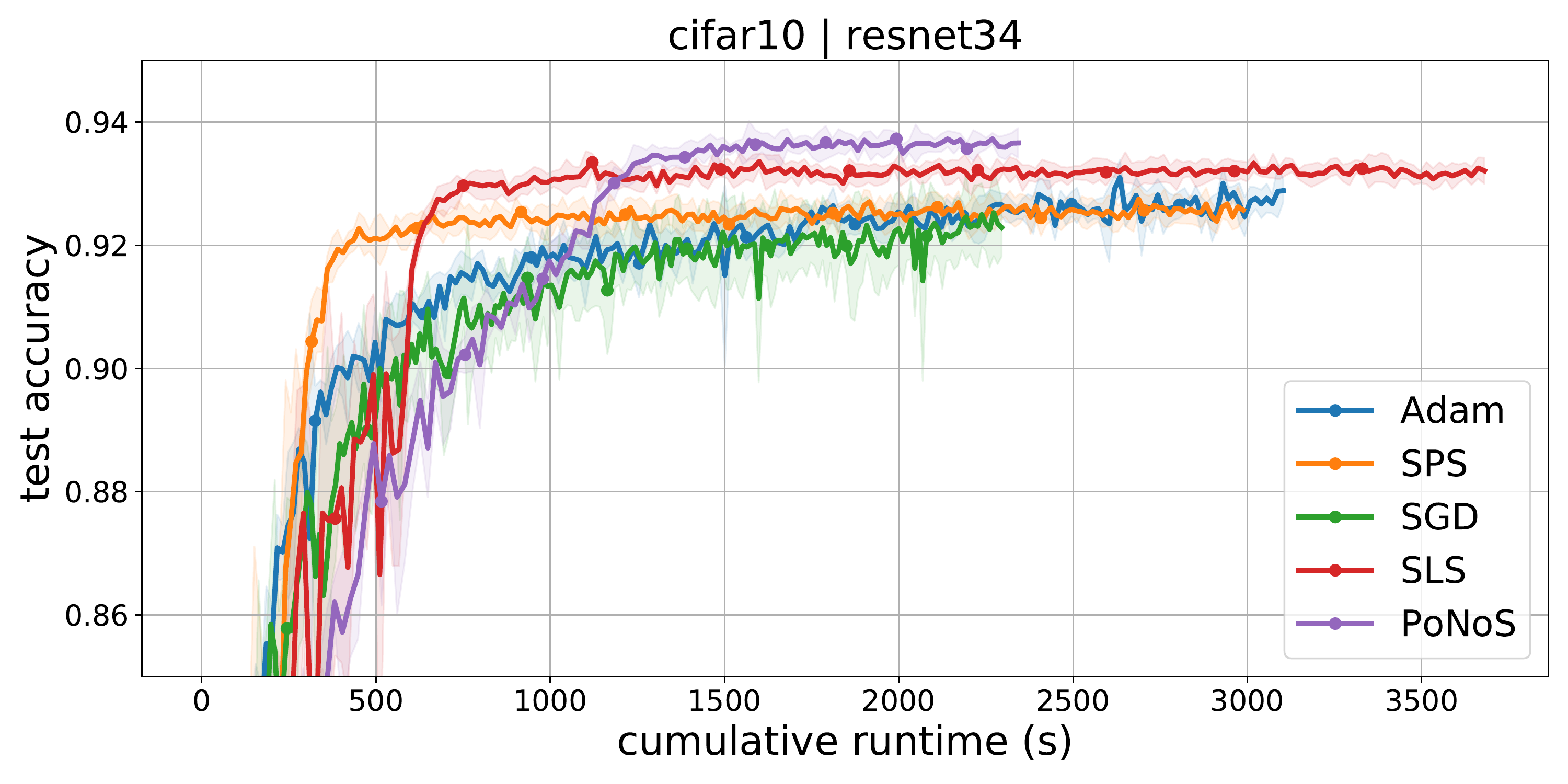}}
		\subfloat{\includegraphics[width=\imgS\linewidth]{exp1/\model/train_epoch_time}}\\
		\renewcommand{\model}{cifar10_densenet}
		\renewcommand{\modelname}{cifar10\_densenet}
		\subfloat{\includegraphics[width=\imgS\linewidth]{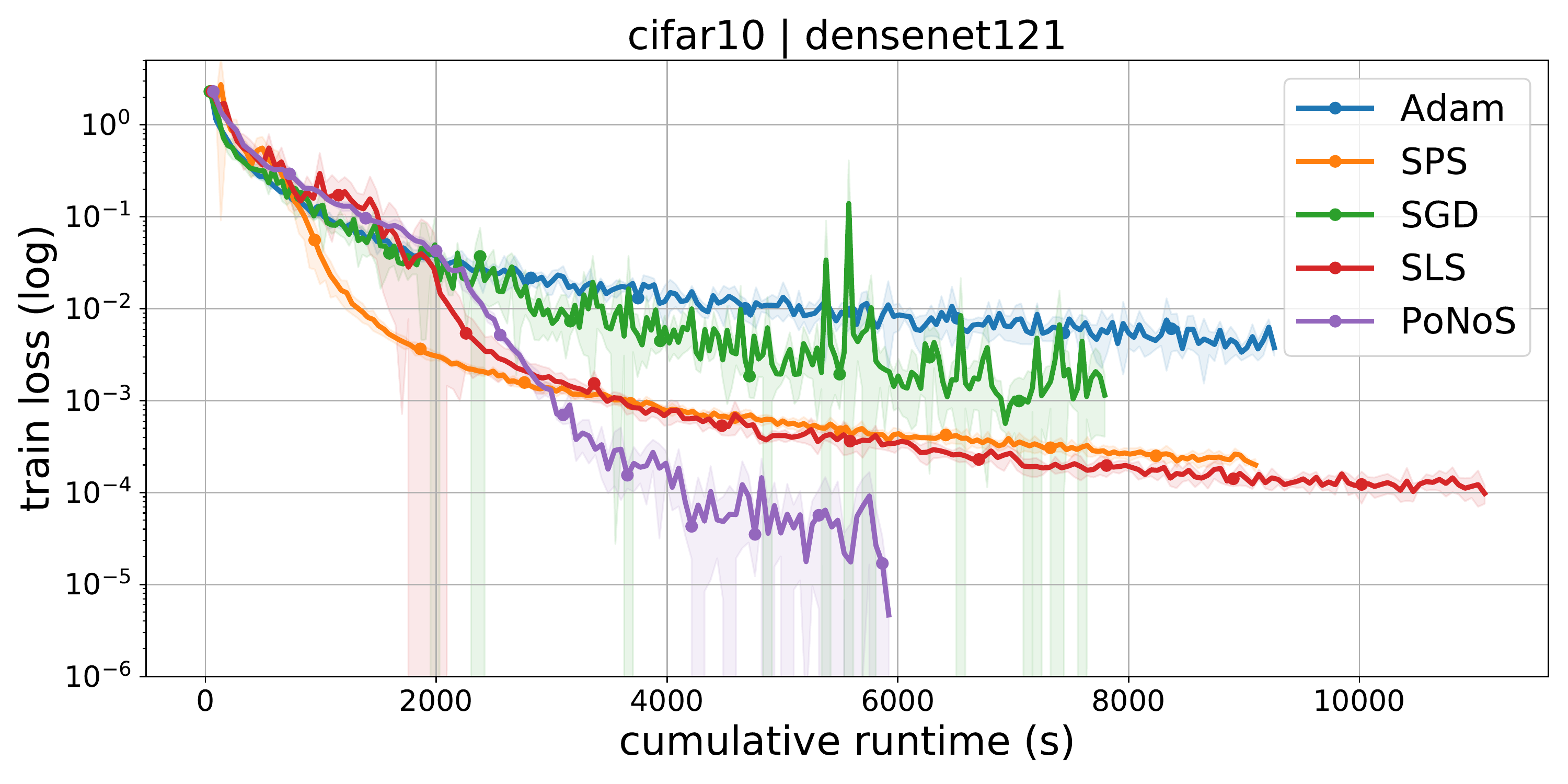}}
		\subfloat{\includegraphics[width=\imgS\linewidth]{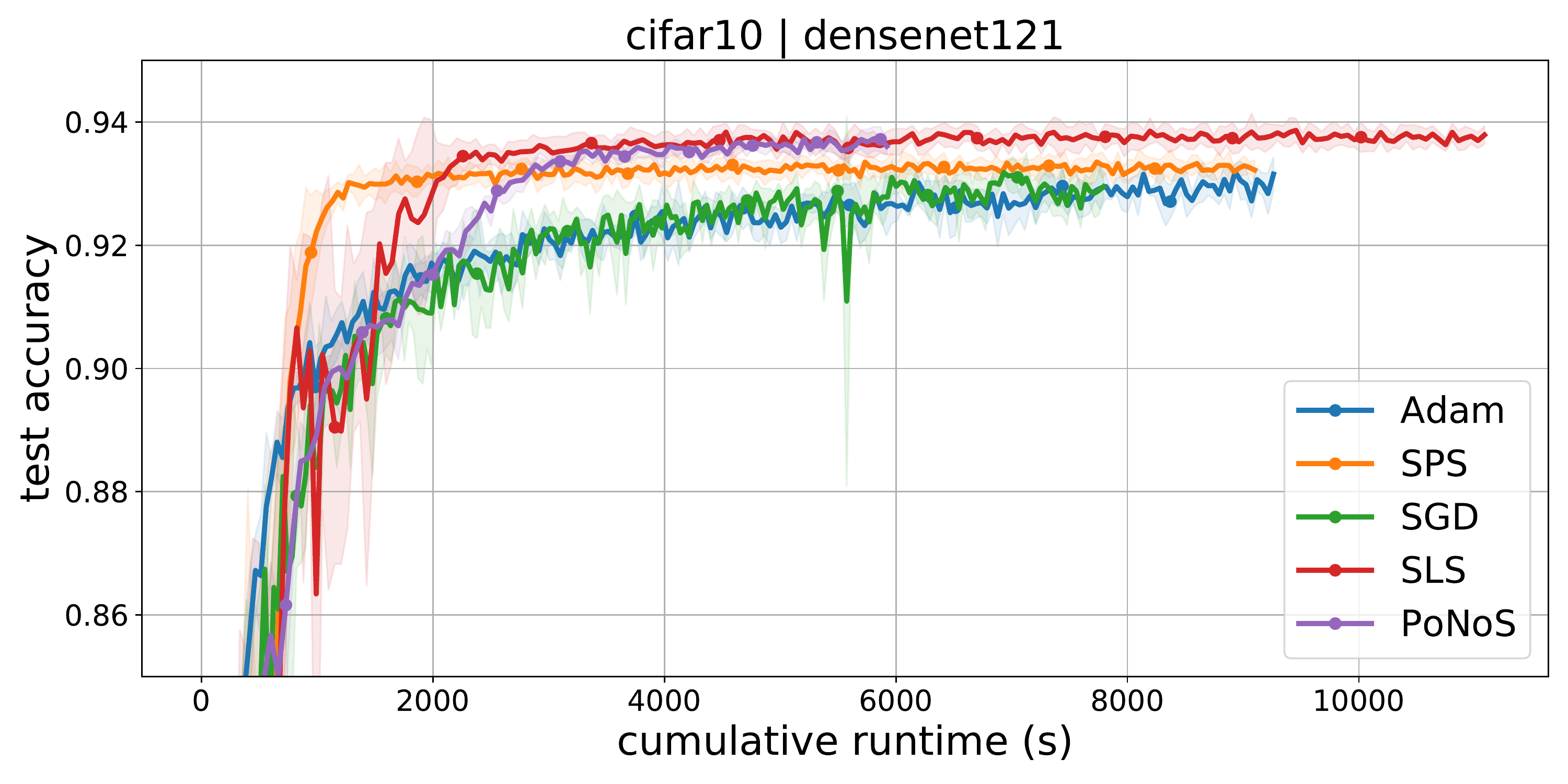}}
		\subfloat{\includegraphics[width=\imgS\linewidth]{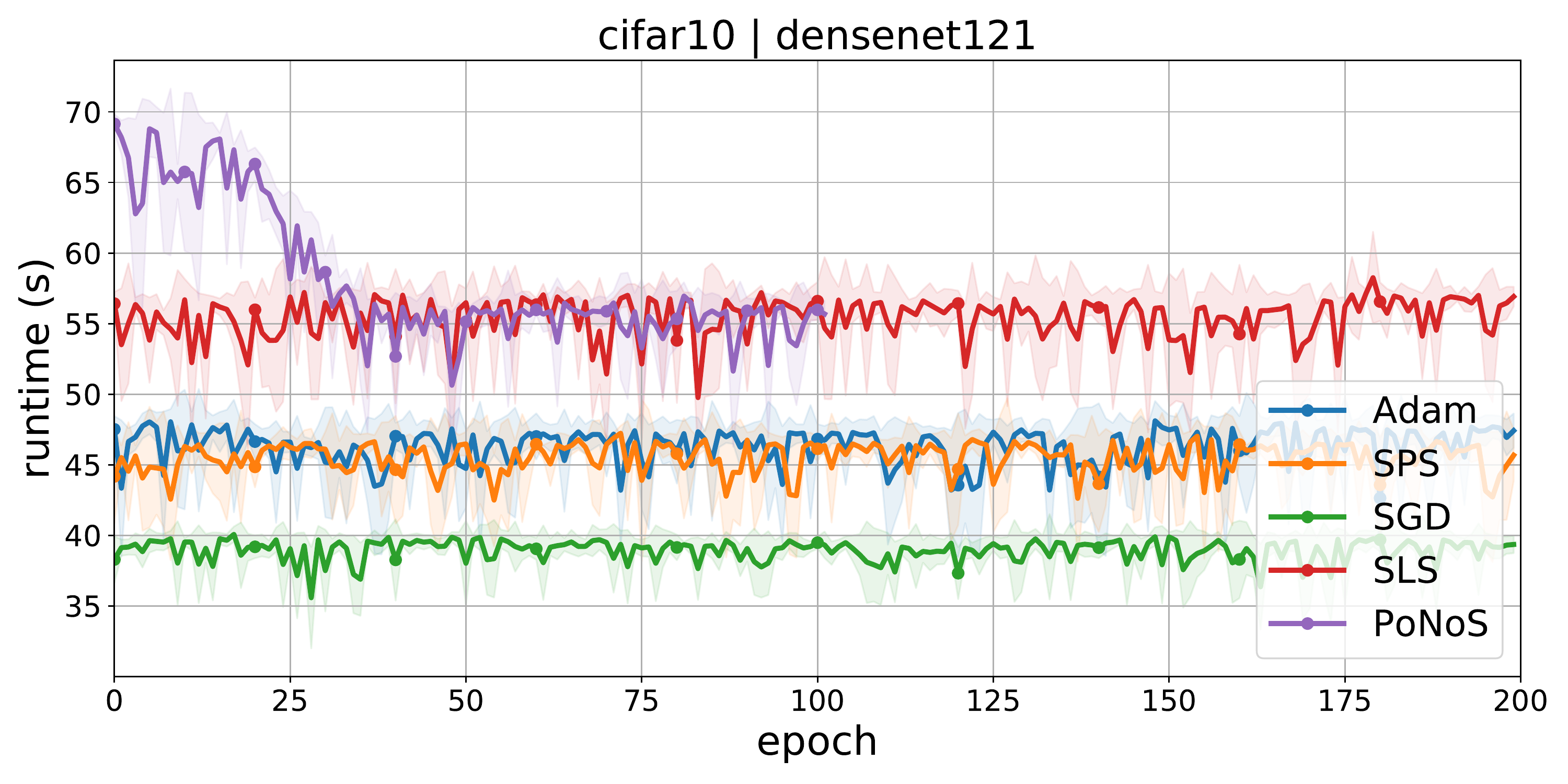}}\\
		\renewcommand{\model}{cifar100_resnet}
		\renewcommand{\modelname}{cifar100\_resnet}
		\subfloat{\includegraphics[width=\imgS\linewidth]{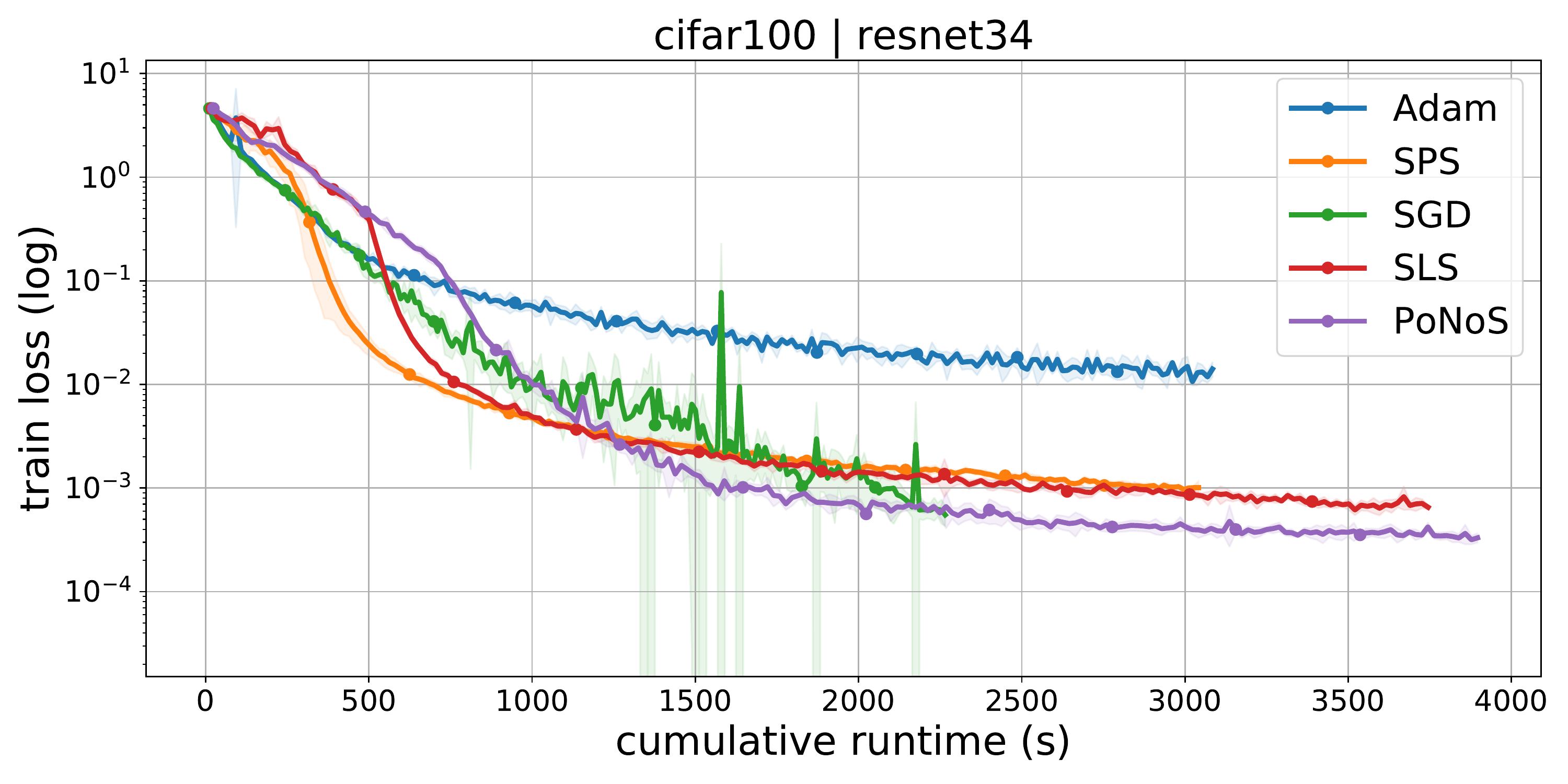}}
		\subfloat{\includegraphics[width=\imgS\linewidth]{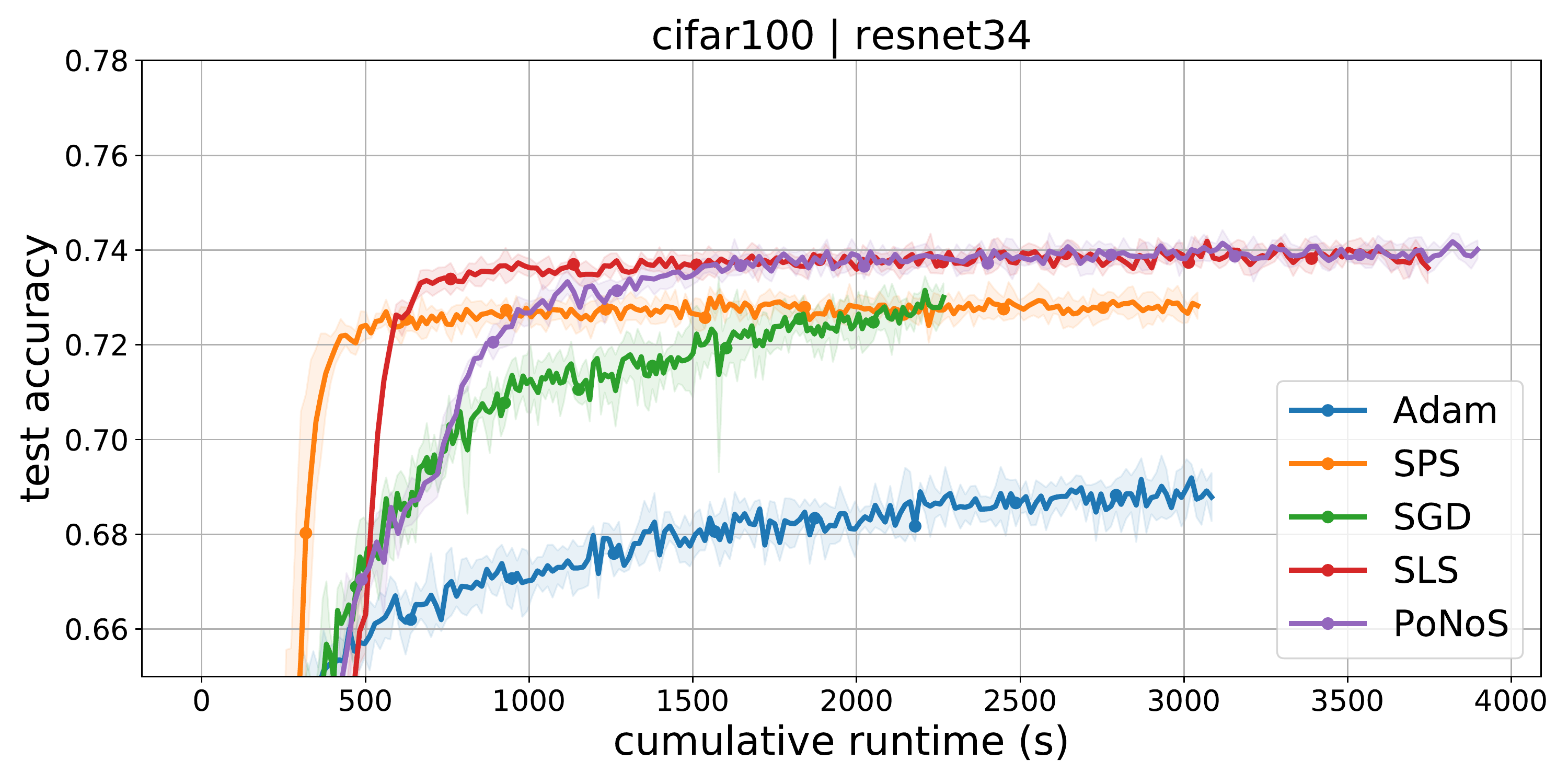}}
		\subfloat{\includegraphics[width=\imgS\linewidth]{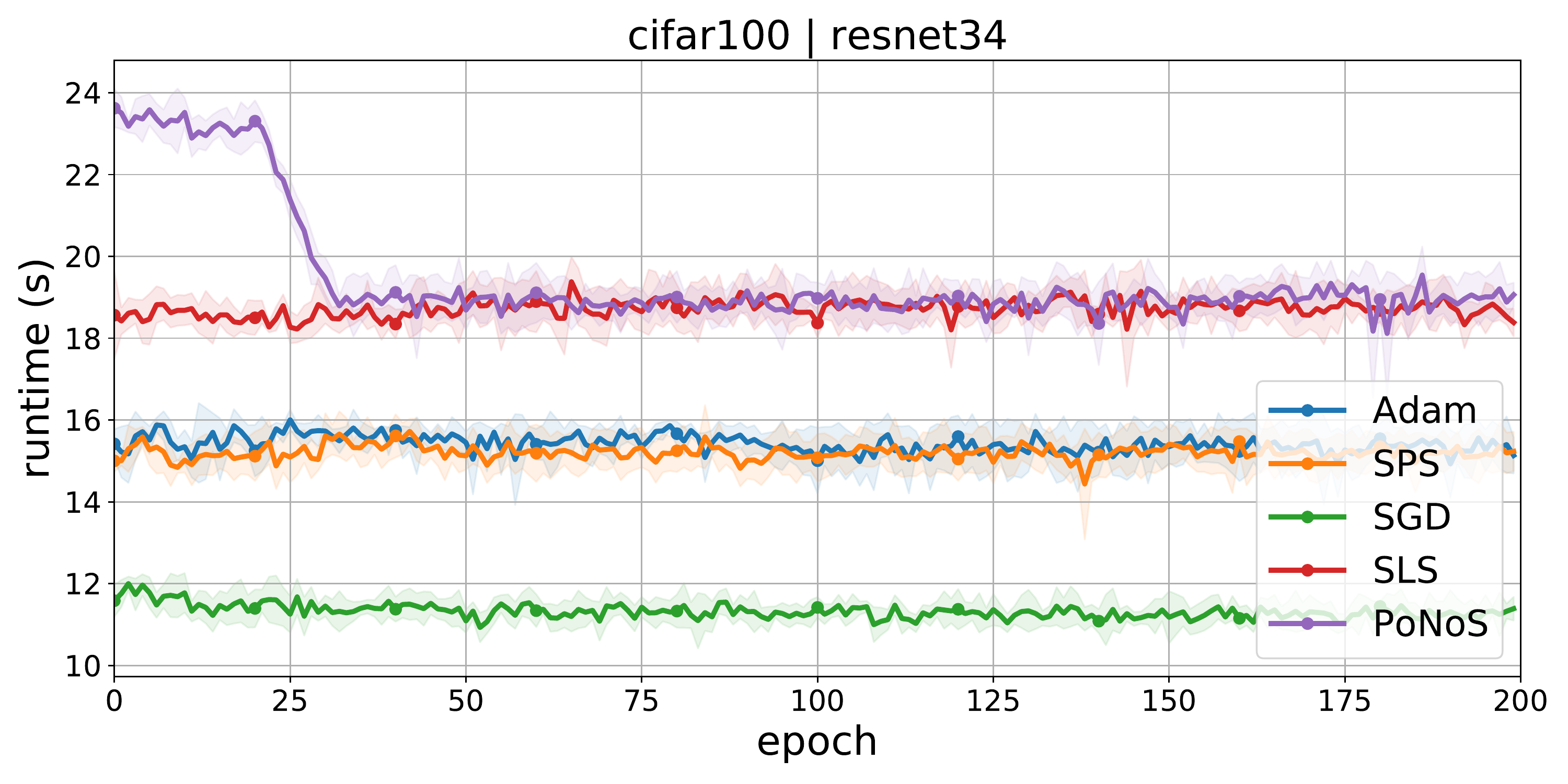}}\\
		\renewcommand{\model}{cifar100_densenet}
		\renewcommand{\modelname}{cifar100\_densenet}
		\subfloat{\includegraphics[width=\imgS\linewidth]{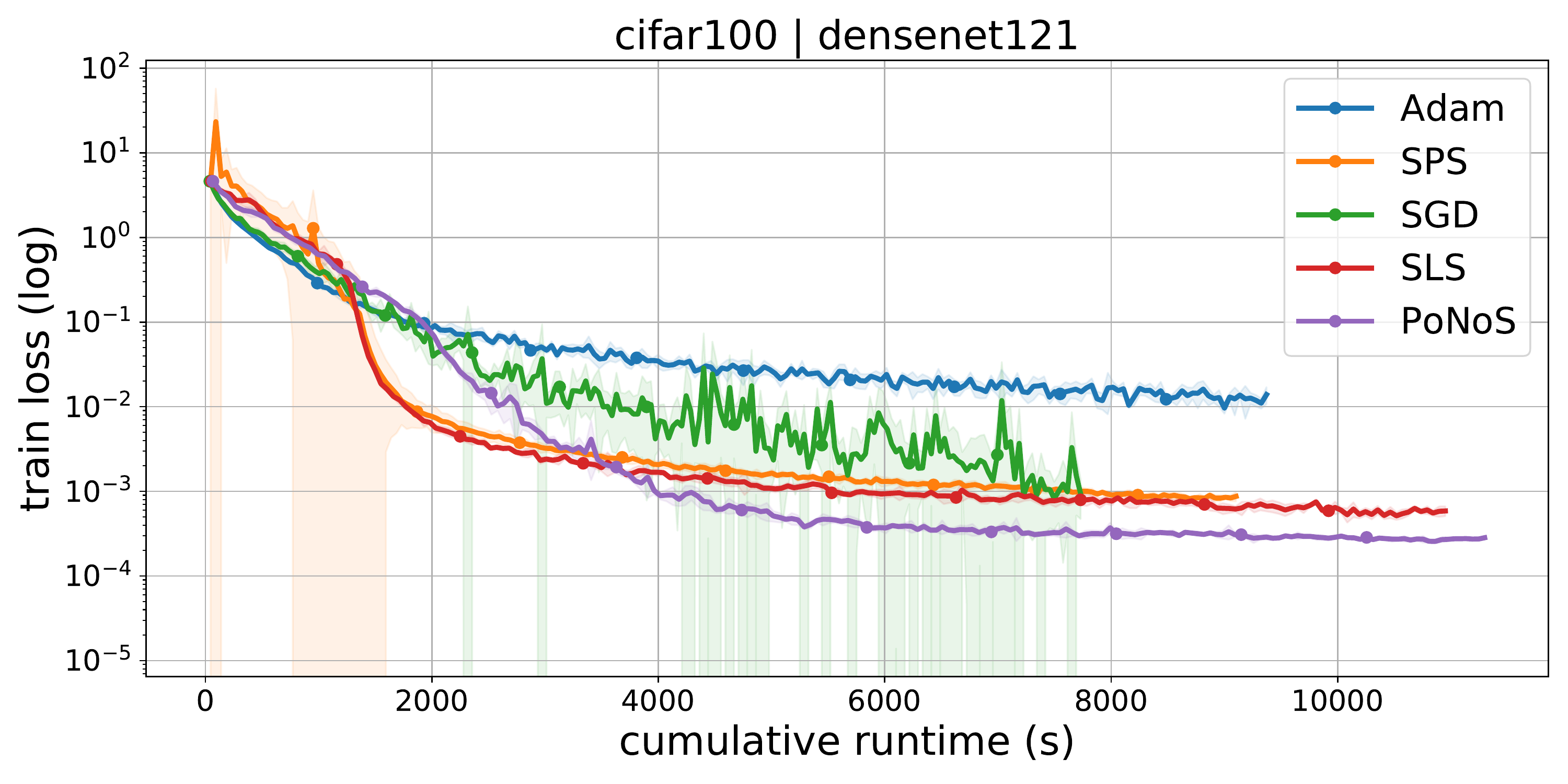}}
		\subfloat{\includegraphics[width=\imgS\linewidth]{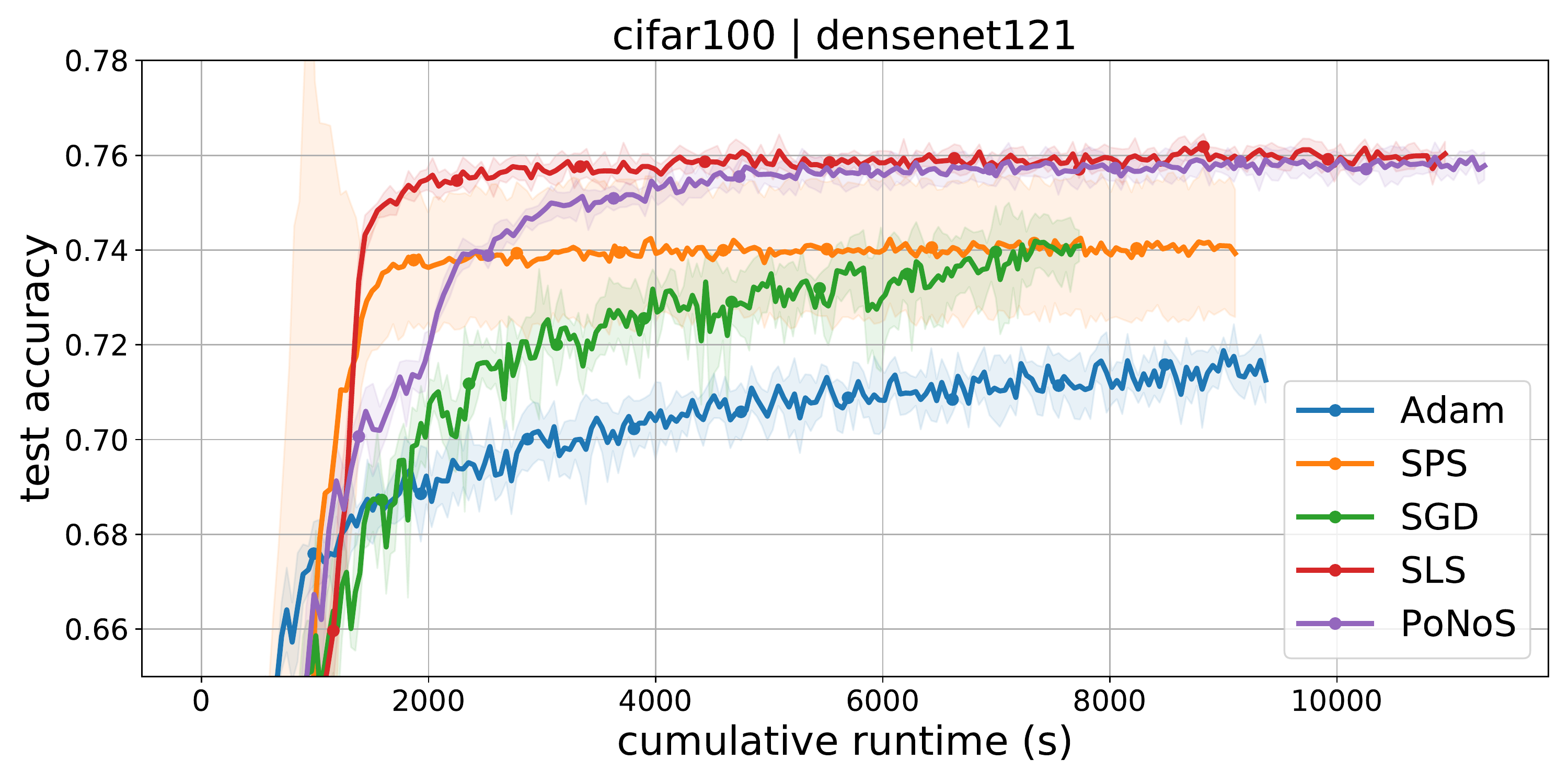}}
		\subfloat{\includegraphics[width=\imgS\linewidth]{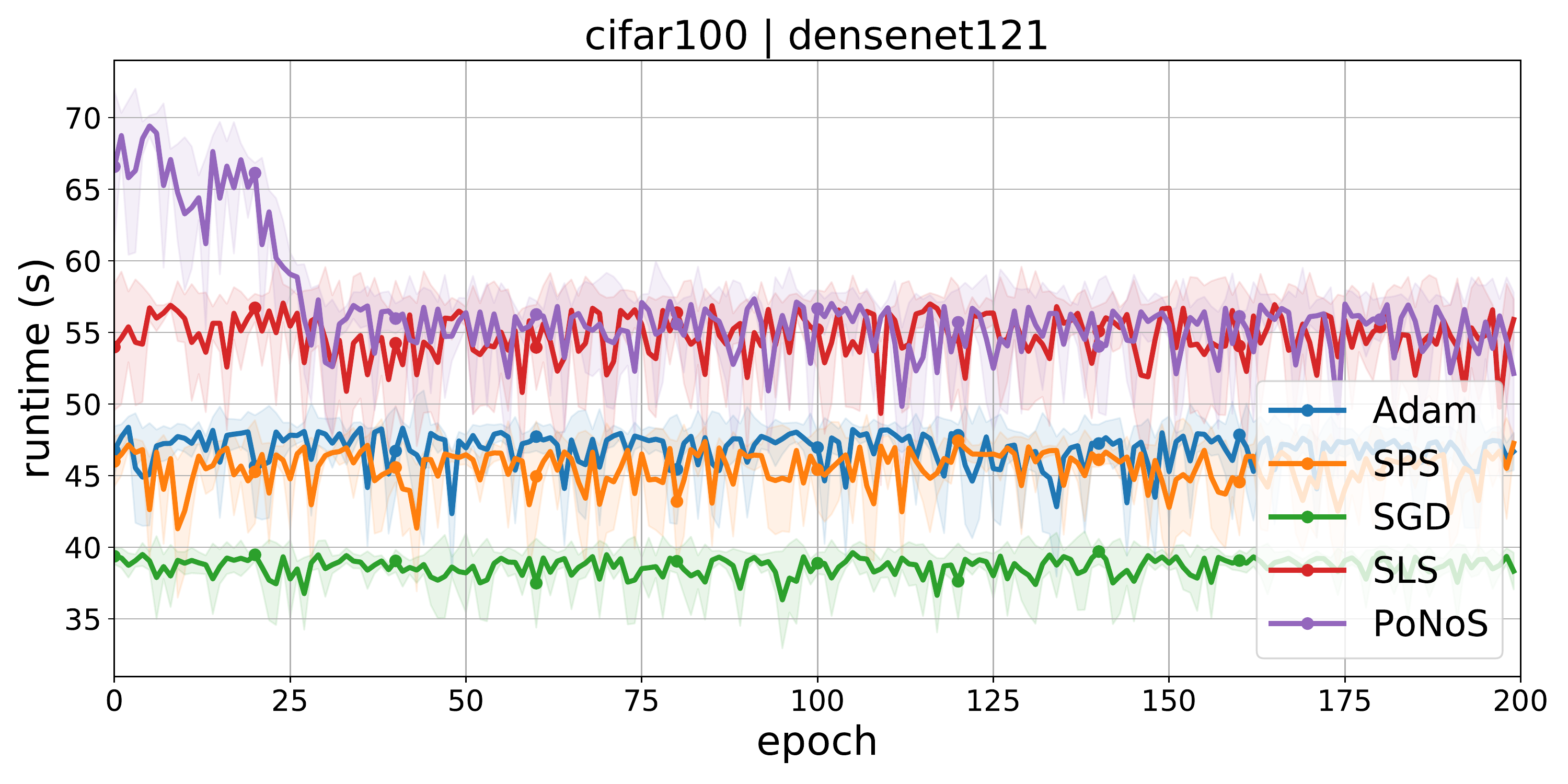}}\\
		\renewcommand{\model}{fashion_effb1}
		\renewcommand{\modelname}{fashion\_effb1}
		\subfloat{\includegraphics[width=\imgS\linewidth]{\dir/\model/train_loss}}
		\subfloat{\includegraphics[width=\imgS\linewidth]{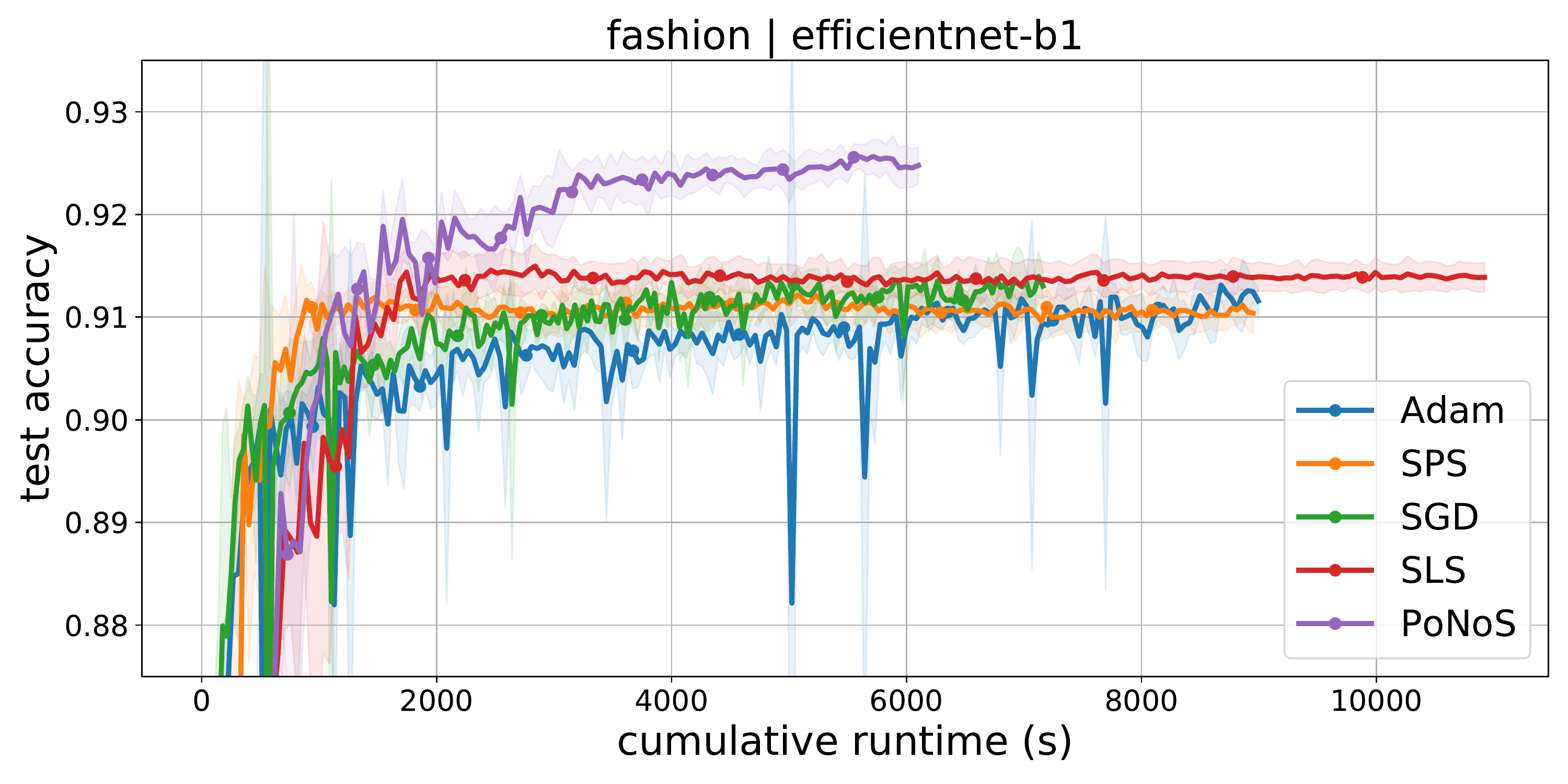}}
		\subfloat{\includegraphics[width=\imgS\linewidth]{exp1/\model/train_epoch_time}}\\
		\renewcommand{\model}{svhn_wrn}
		\renewcommand{\modelname}{svhn\_wrn}
		\subfloat{\includegraphics[width=\imgS\linewidth]{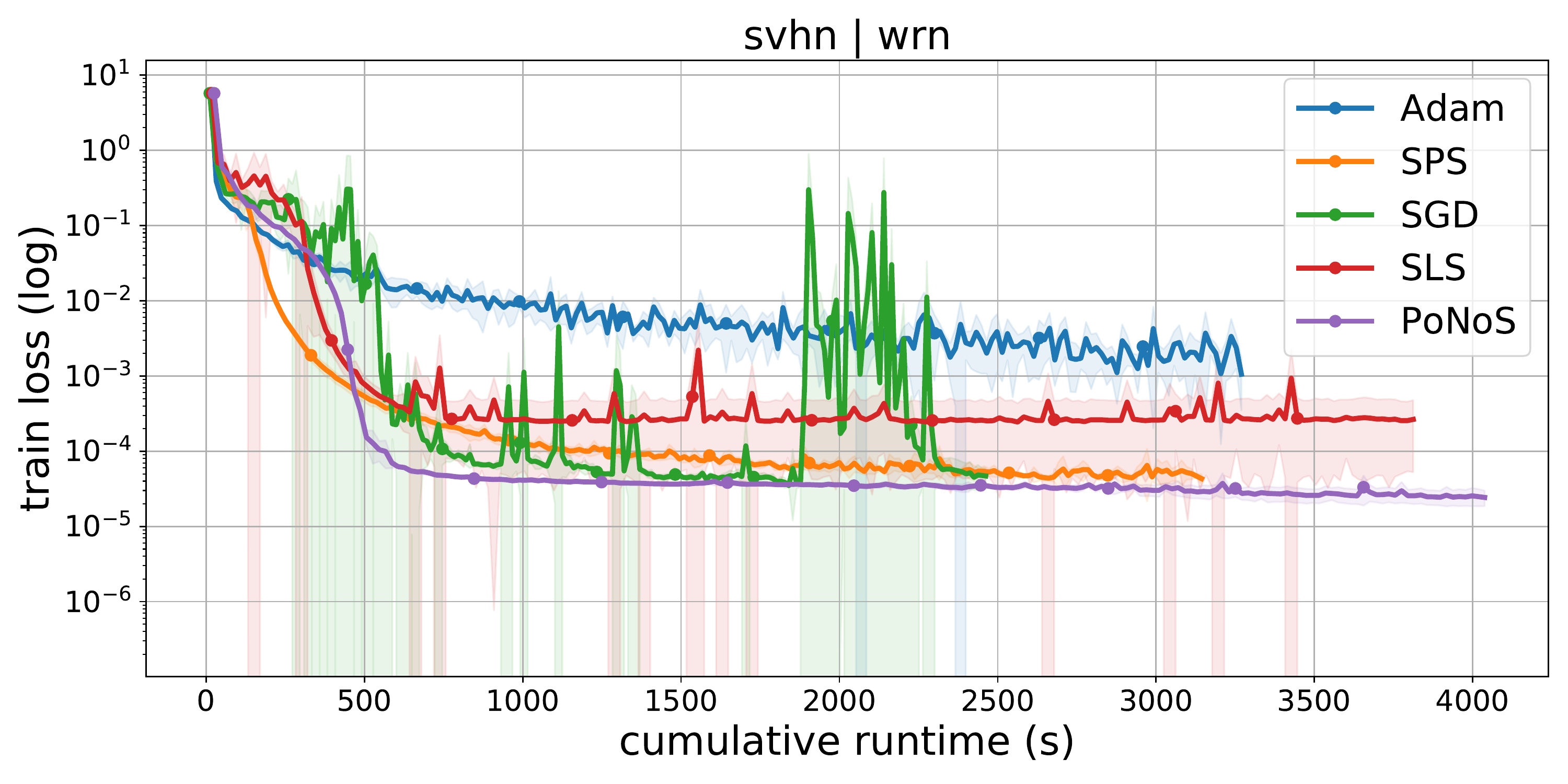}}
		\subfloat{\includegraphics[width=\imgS\linewidth]{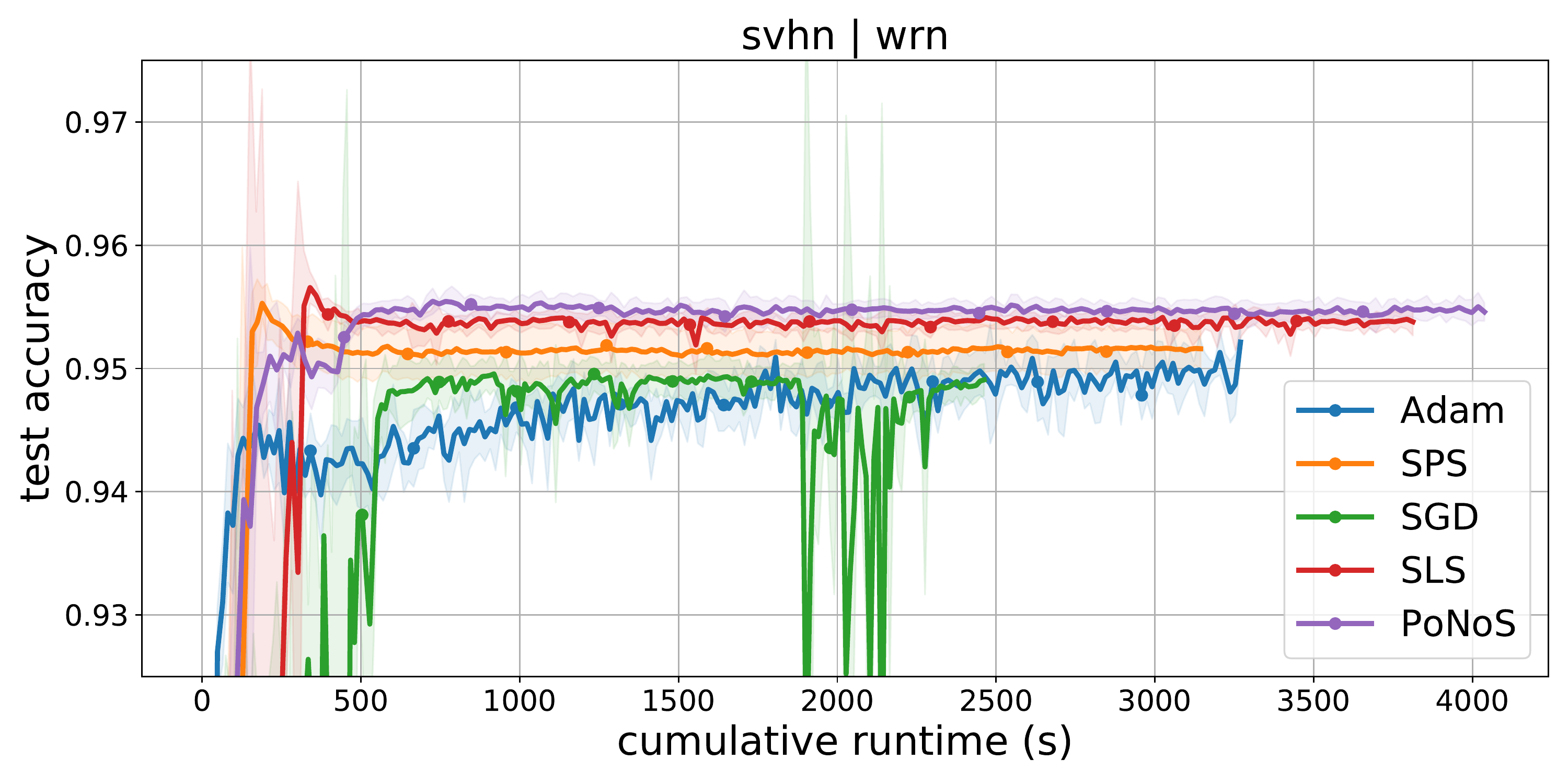}}
		\subfloat{\includegraphics[width=\imgS\linewidth]{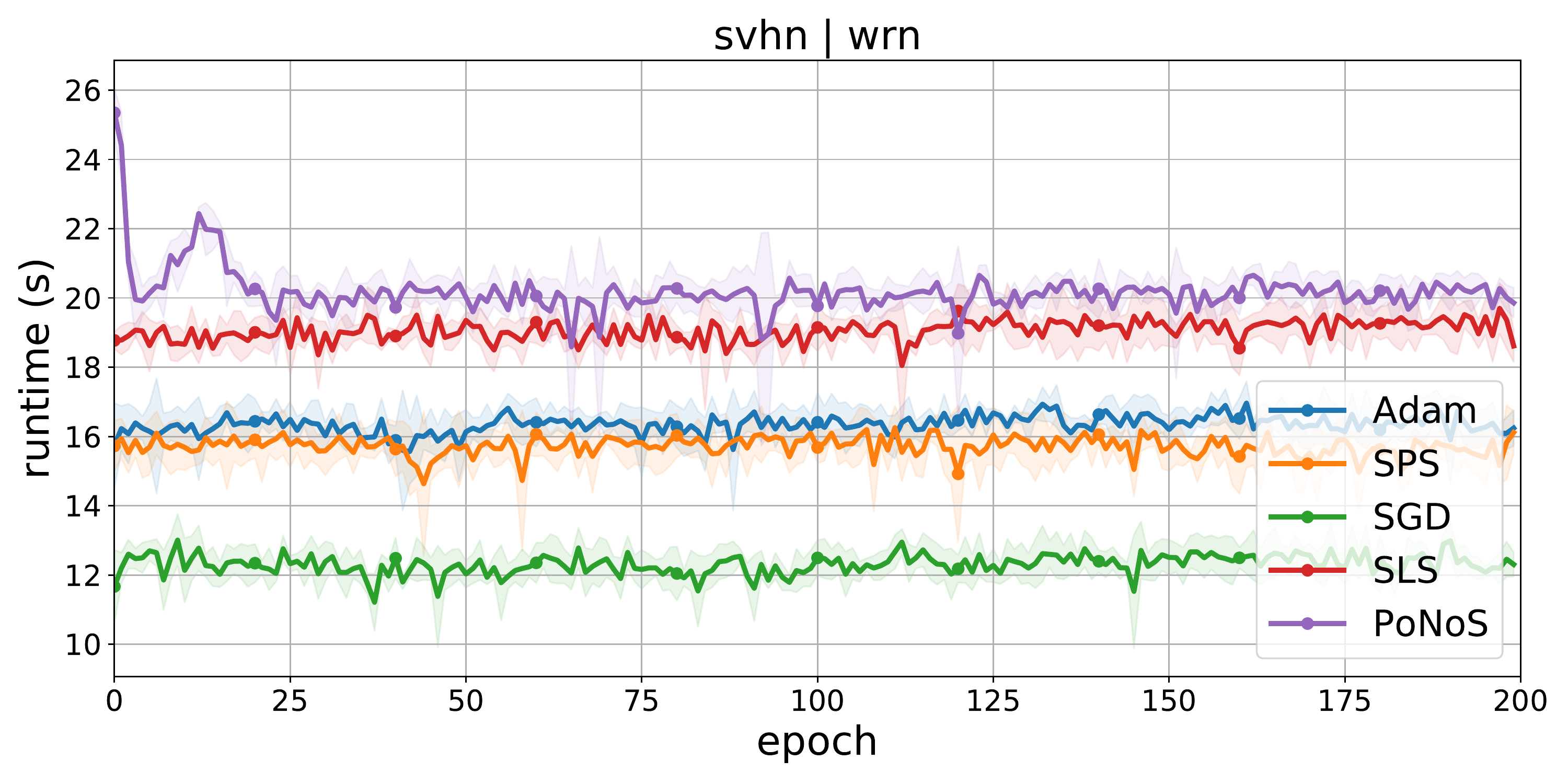}}
		\caption{Time comparison (s) between the proposed method (PoNoS) and the-state-of-the-art. Each row focus on a dataset/model combination. First column: train loss vs cumulative runtime. Second column: test accuracy vs cumulative runtime. Third column: per-epoch runtime.}\label{fig:time}
	\end{minipage}
\end{figure}
	\subsection{Experiments on Convex Losses}\label{sec:supp_convex_short}
	See Figure \ref{fig:convex_short} for the complete results relative to the convex experiments of Figure 4 of the main paper. Only the training sets available in the LIBSVM library were used for these datasets. The $80\%$ split of the data was used as a training set and $20\%$ split as the test set. For the RBF kernel bandwidth, we employed the parameters suggested in \citet{vaswani19a}, that is $\{0.5, 0.25, 0.05, 20\}$ respectively for mushrooms, rcv1, ijcnn and w8a. We did not use any bias parameter in these experiments. Furthermore, given the convexity of these problems, $\etamax$ is not needed and it has been set to $\infty$.

In Figure \ref{fig:convex_short}, the three measures are reported by iterations and not by epochs. To avoid large fluctuations in the plots, train loss and step size have been smoothed using an exponential moving average ($\beta=0.9$). On the other hand, the test accuracy is only computed at the beginning of each epoch and the same value is reported till the next epoch. In Figure \ref{fig:convex_short}, we can make additional observations which are not visible in Figure 4 of the main paper:
\begin{itemize}
\item PoNoS achieves the lowest loss value on \rcvone$ $ and \ijcnn$ $. Regarding the test accuracy, PoNoS achieves always the highest score, apart from \rcvone$ $ on which it loses $~0.5$ points w.r.t. SLS and SPS. 
\item SLS and SPS behave very similarly on all the datasets. They achieve the best accuracy on all the problems, but they are both very slow on \mushrooms, \rcvone$ $ and \ijcnn$ $ in terms of loss. This behavior is due to \eqref{eq:etamax}, as it is clear from the third column of Figure \ref{fig:convex_short}. The step size yielded by \eqref{eq:etamax} is often too small and it slowly grows exponentially through the whole optimization process. This choice is suboptimal if compared with the step size yielded by PoNoS.
\end{itemize}

\renewcommand{\dir}{convex_short}
\renewcommand{\imgS}{0.33}
\begin{figure}[!h] 
	\begin{minipage}{0.9\textwidth}
		\centering
		\renewcommand{\model}{mushrooms}
		\renewcommand{\modelname}{mushrooms}
		\subfloat{\includegraphics[width=\imgS\linewidth]{\dir/\model/smooth_loss}}
		\subfloat{\includegraphics[width=\imgS\linewidth]{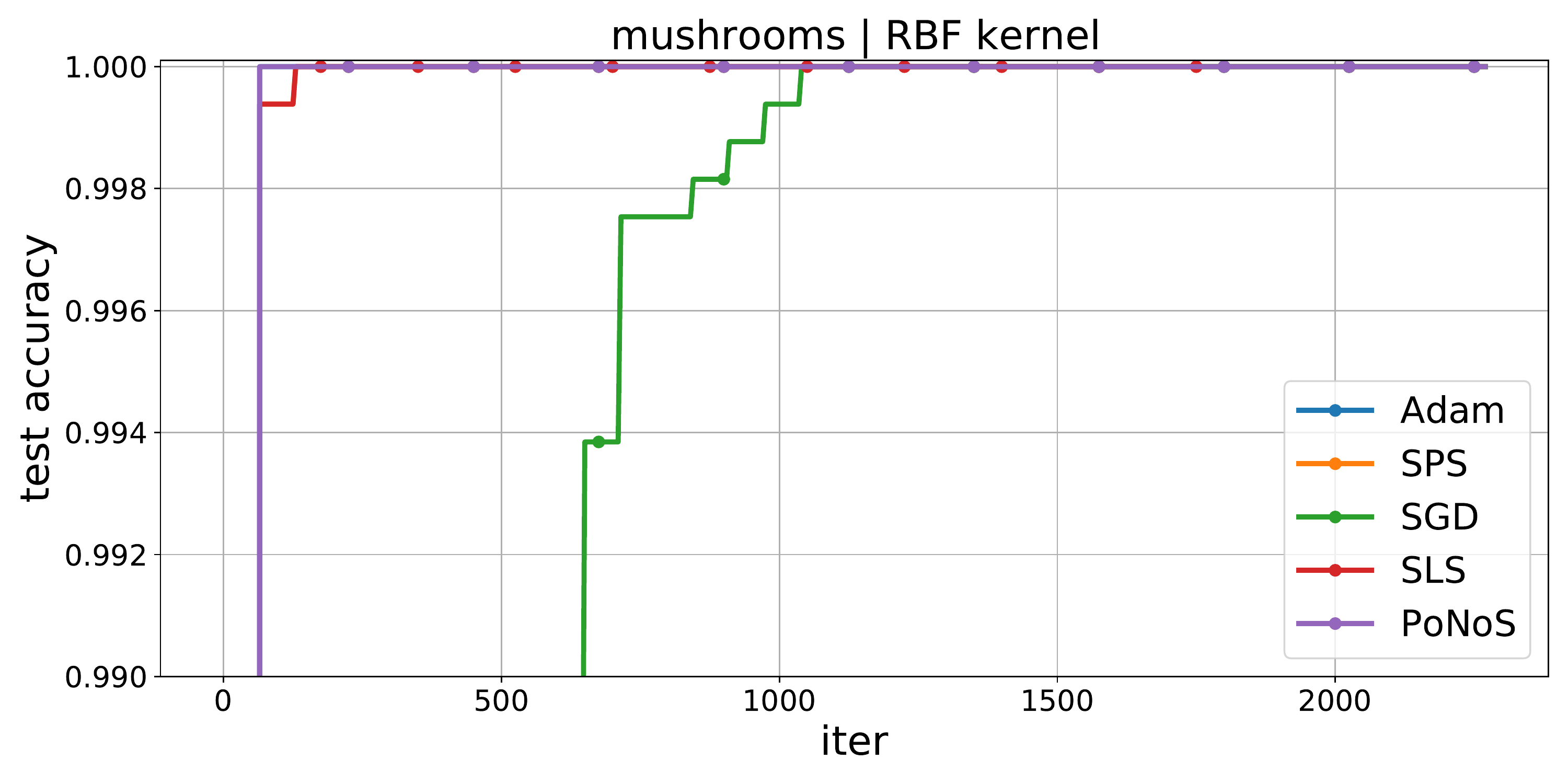}}
		\subfloat{\includegraphics[width=\imgS\linewidth]{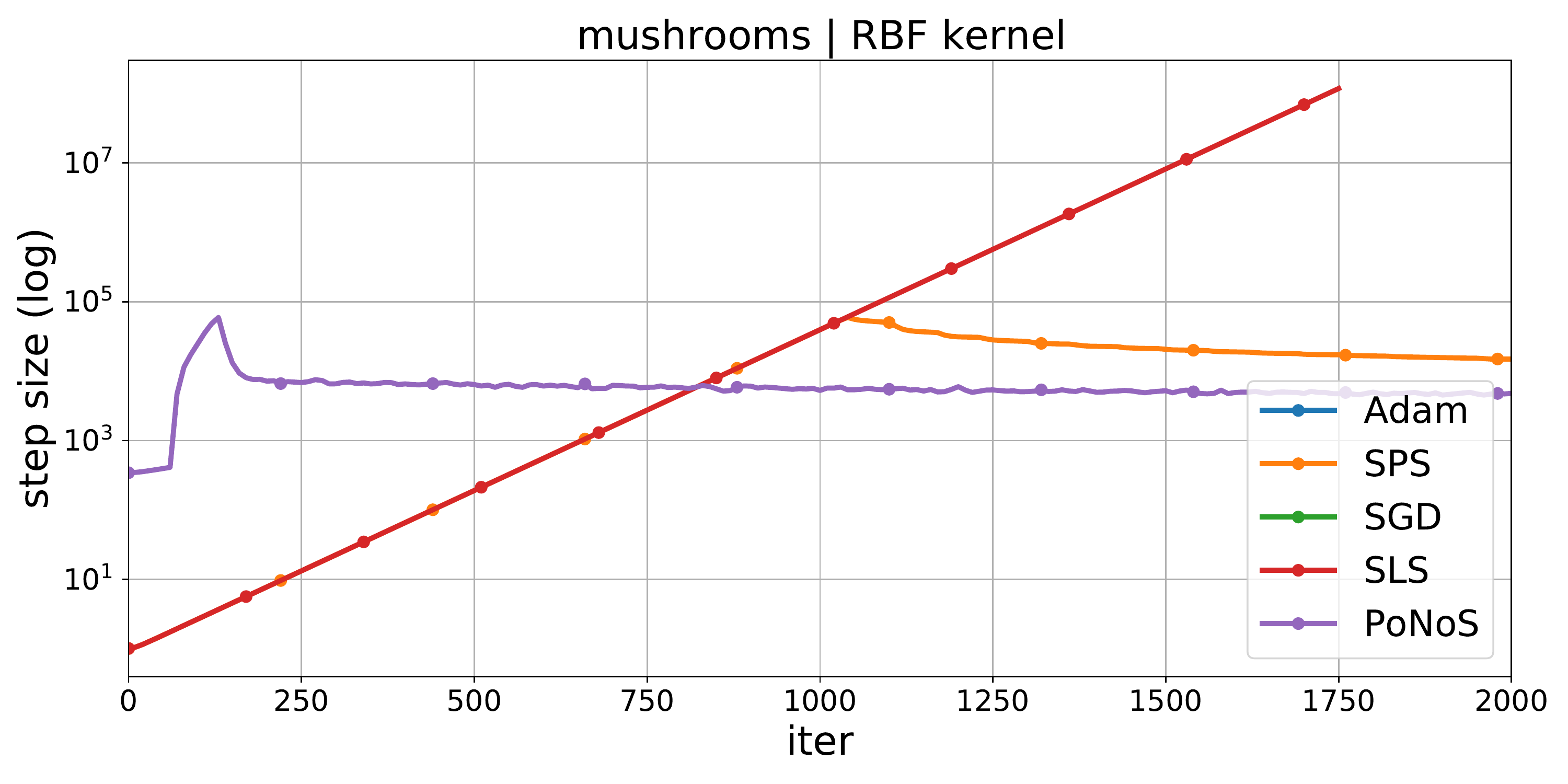}}\\
		\renewcommand{\model}{rcv1}
		\renewcommand{\modelname}{rcv1}
		\subfloat{\includegraphics[width=\imgS\linewidth]{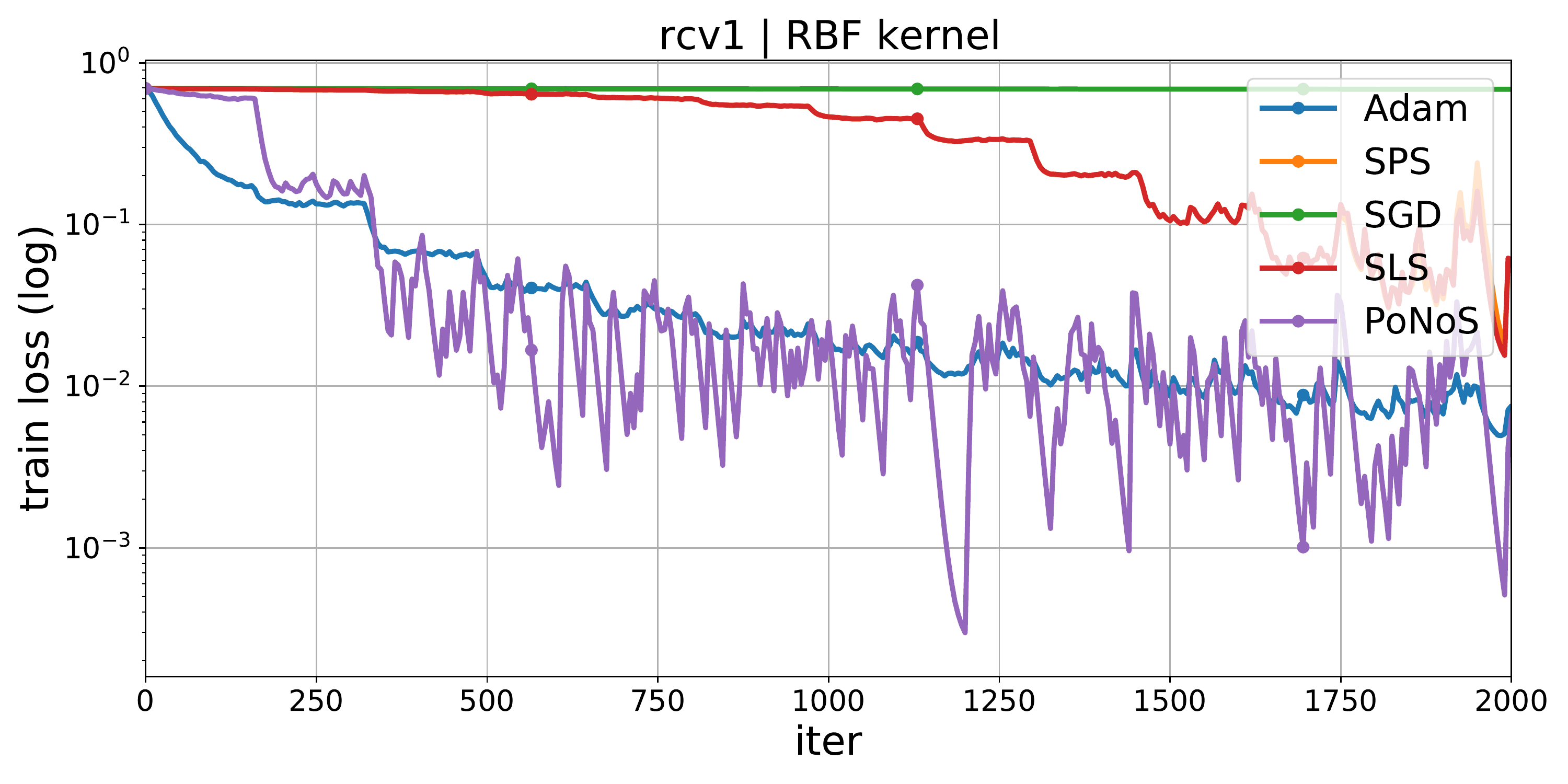}}
		\subfloat{\includegraphics[width=\imgS\linewidth]{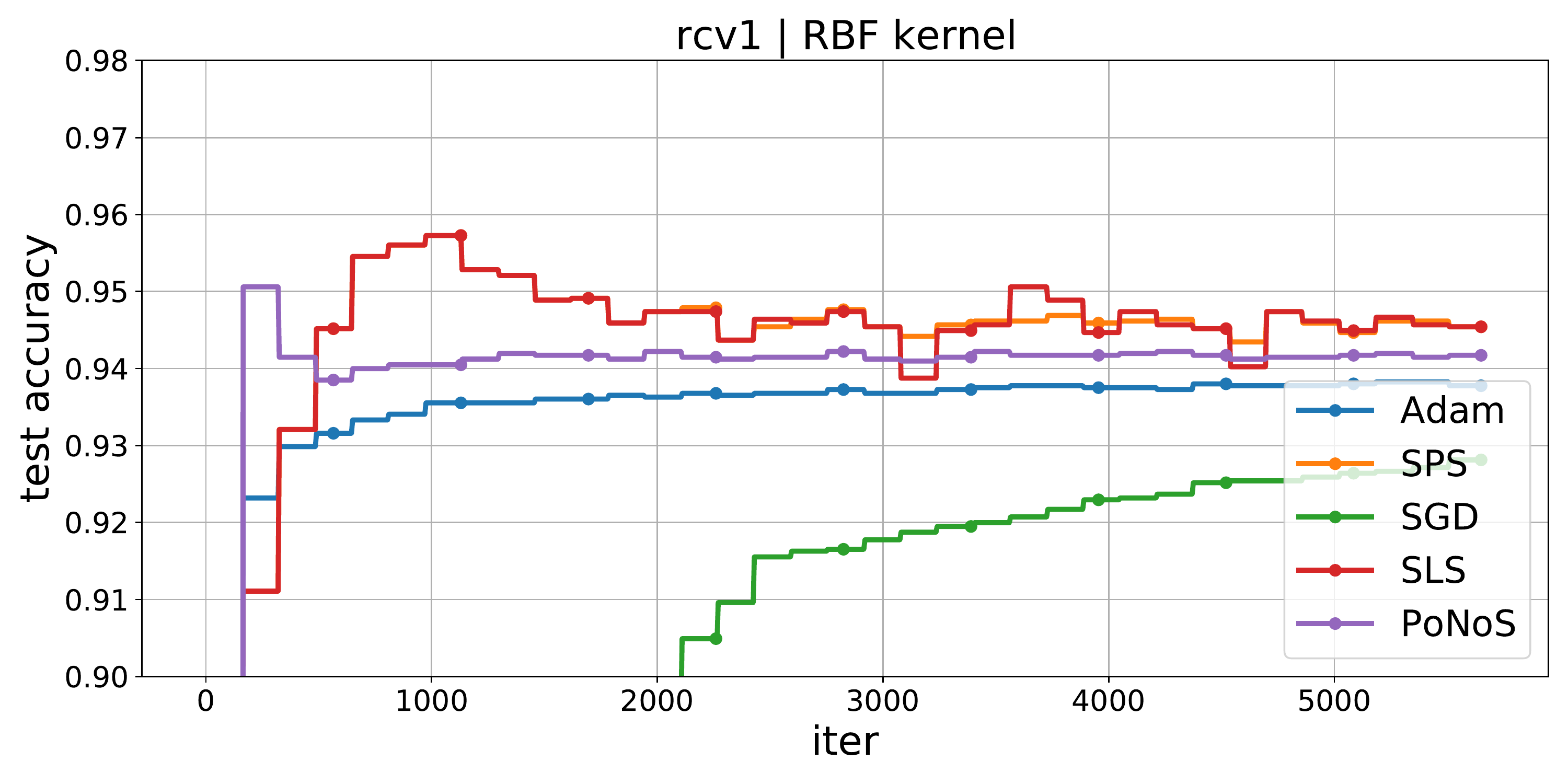}}
		\subfloat{\includegraphics[width=\imgS\linewidth]{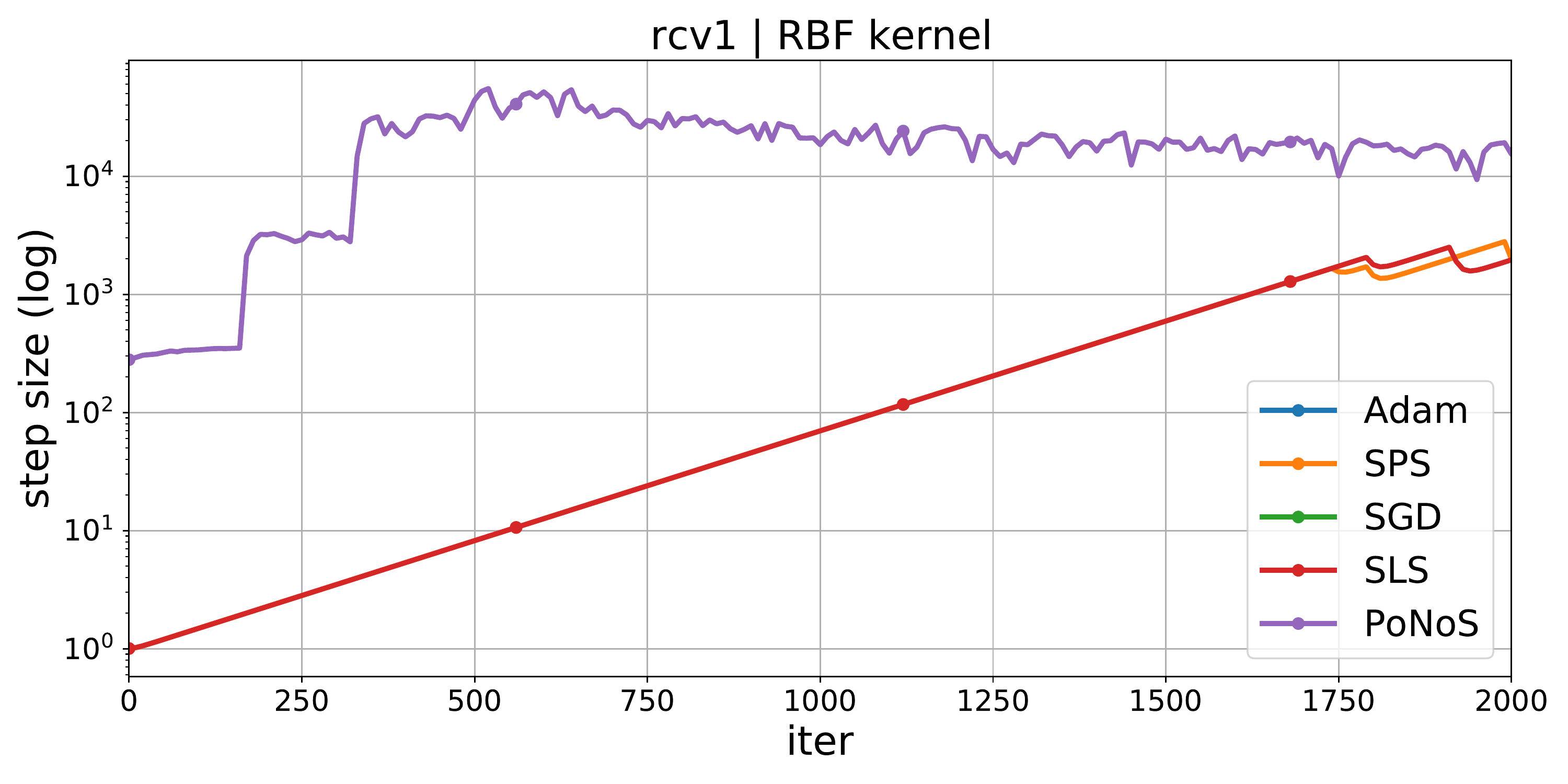}}\\
		\renewcommand{\model}{ijcnn}
		\renewcommand{\modelname}{ijcnn}
		\subfloat{\includegraphics[width=\imgS\linewidth]{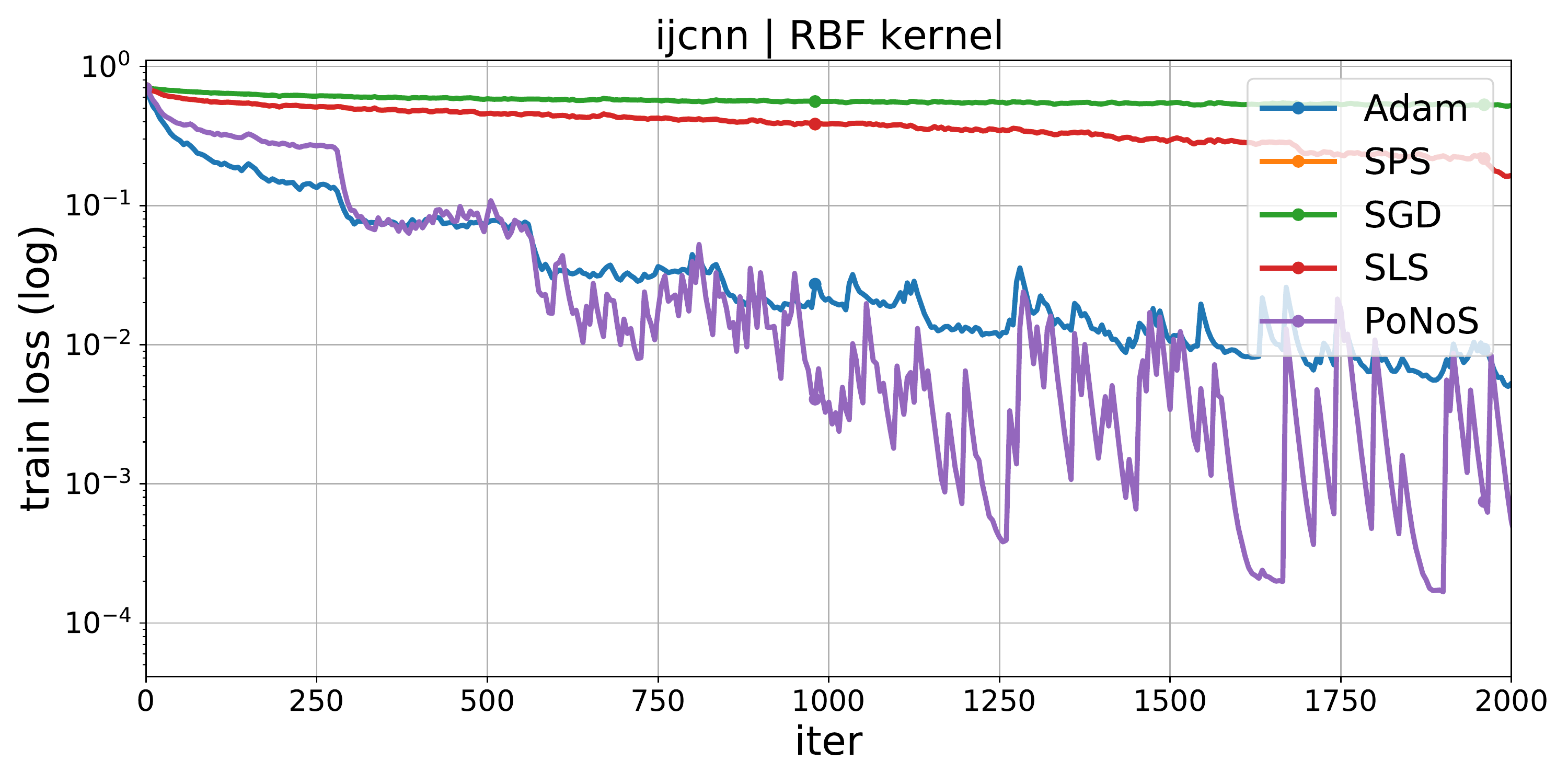}}
		\subfloat{\includegraphics[width=\imgS\linewidth]{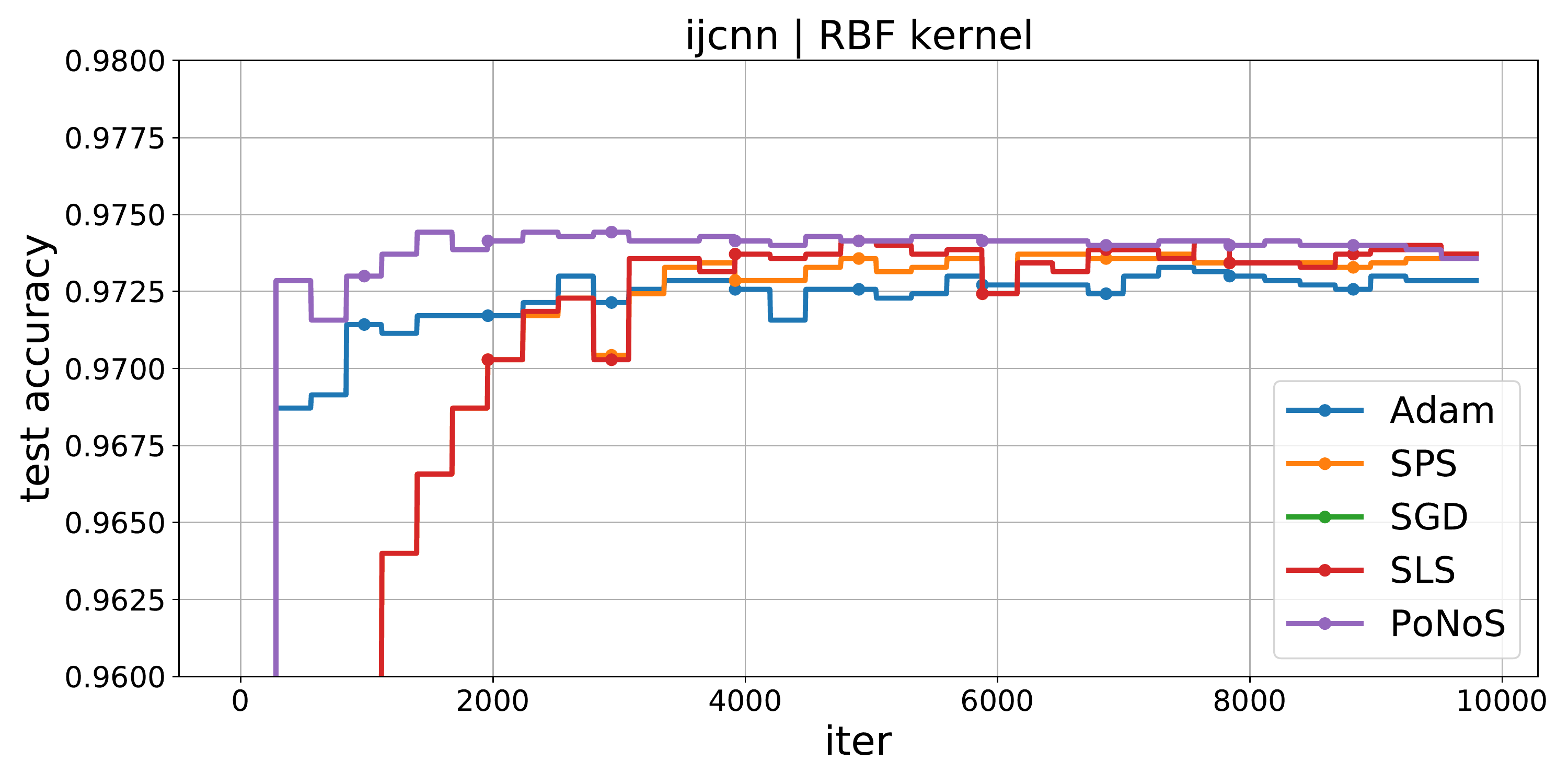}}
		\subfloat{\includegraphics[width=\imgS\linewidth]{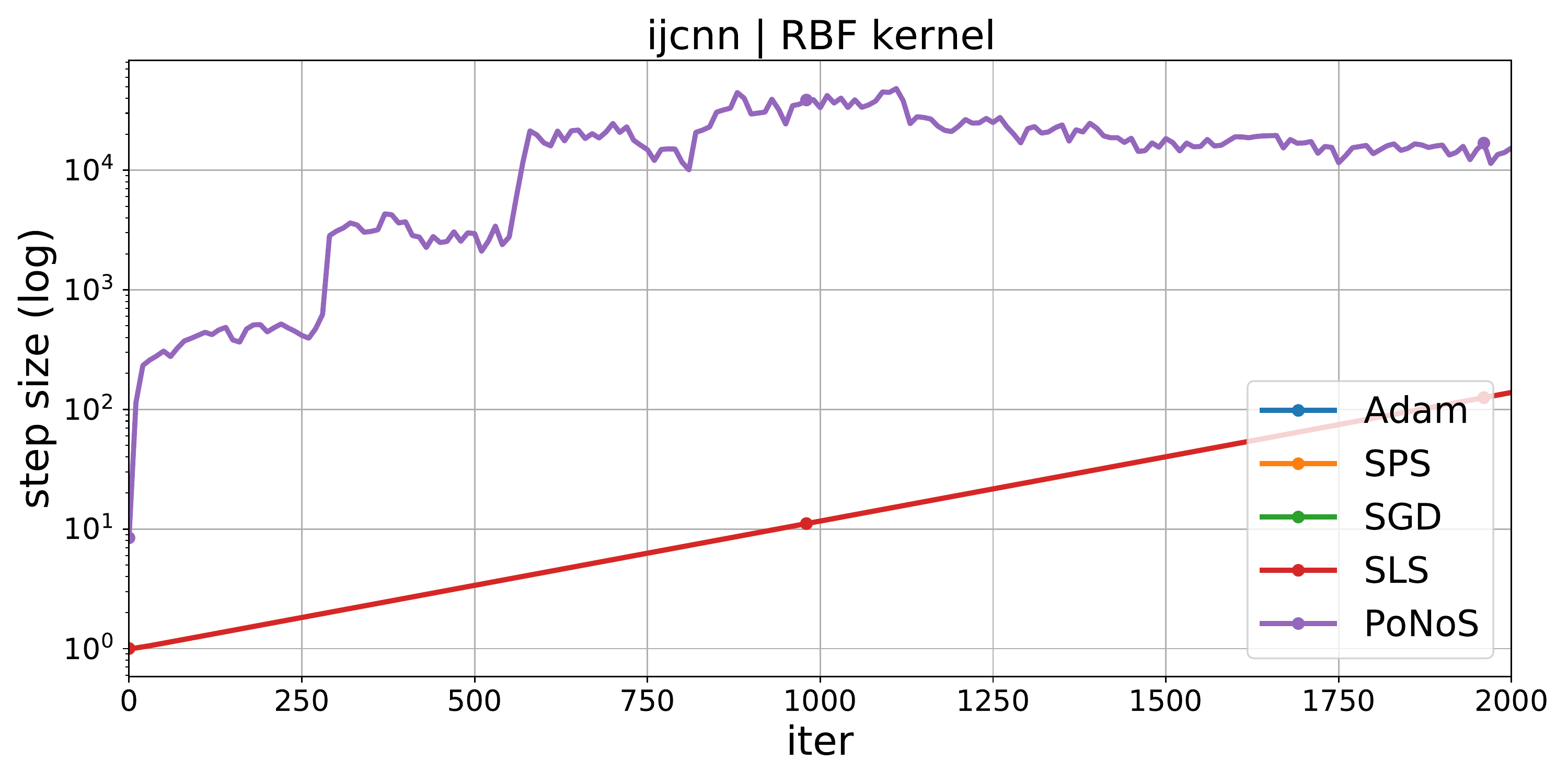}}\\
		\renewcommand{\model}{w8a}
		\renewcommand{\modelname}{w8a}
		\subfloat{\includegraphics[width=\imgS\linewidth]{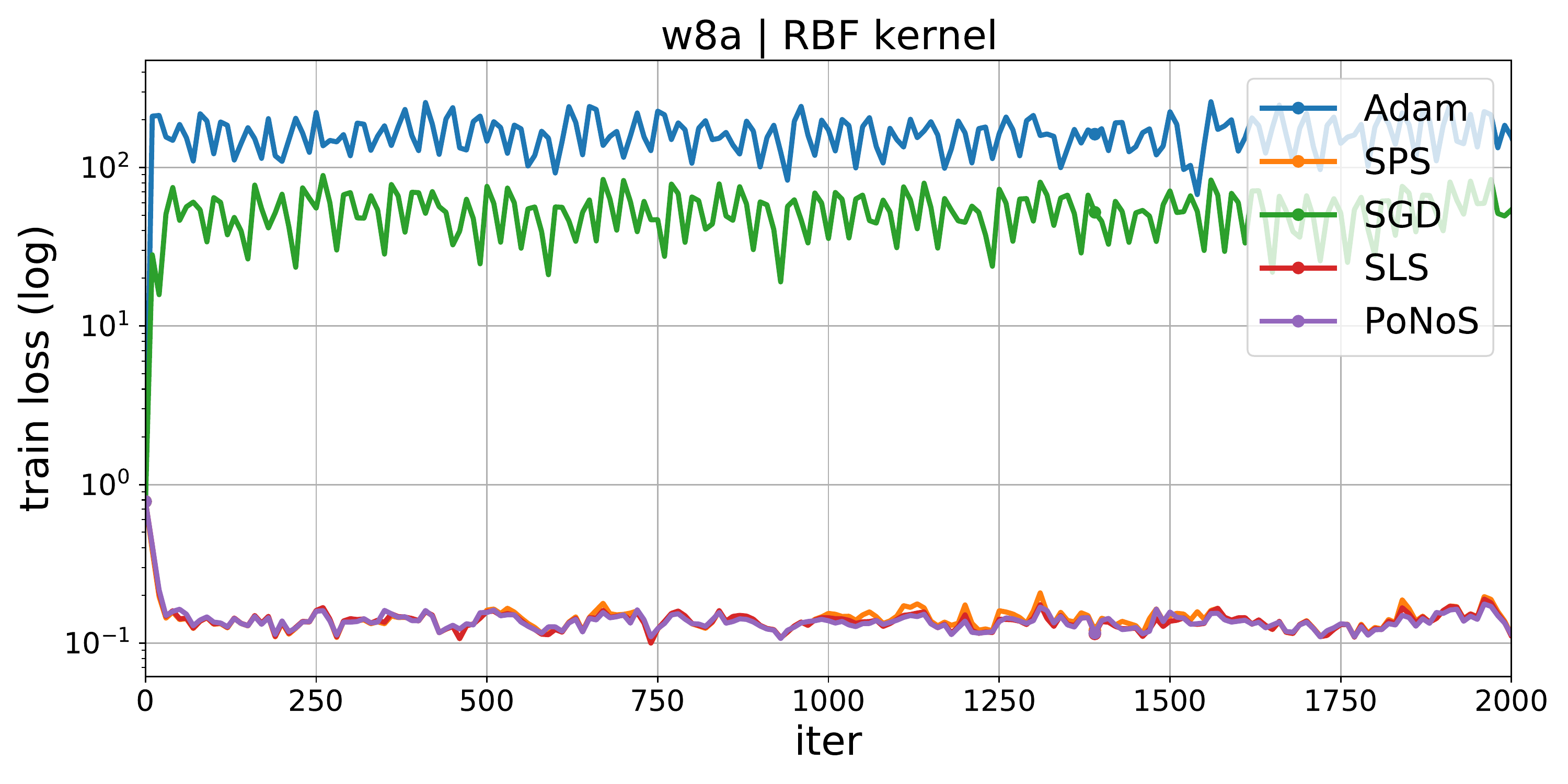}}
		\subfloat{\includegraphics[width=\imgS\linewidth]{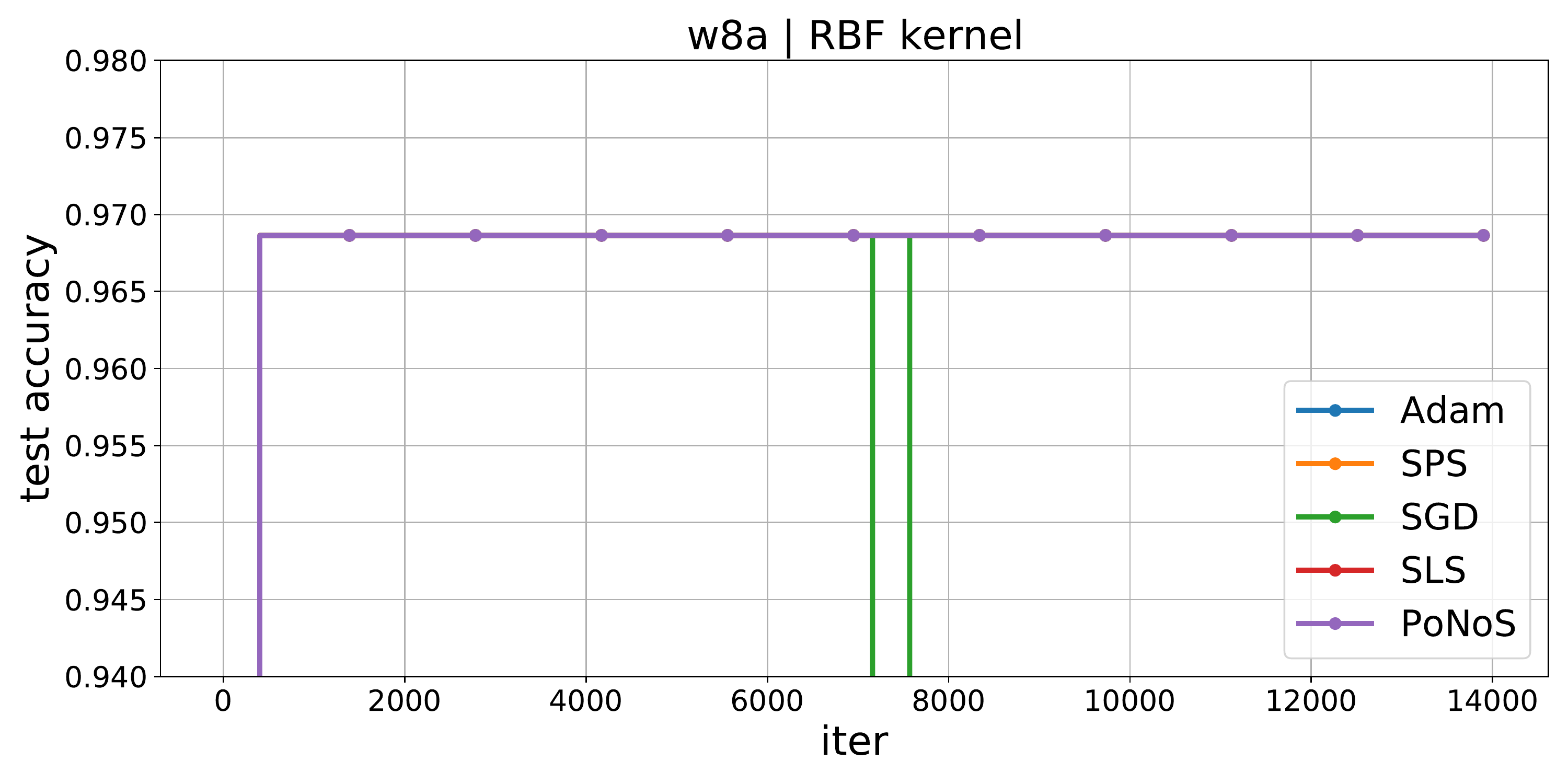}}
		\subfloat{\includegraphics[width=\imgS\linewidth]{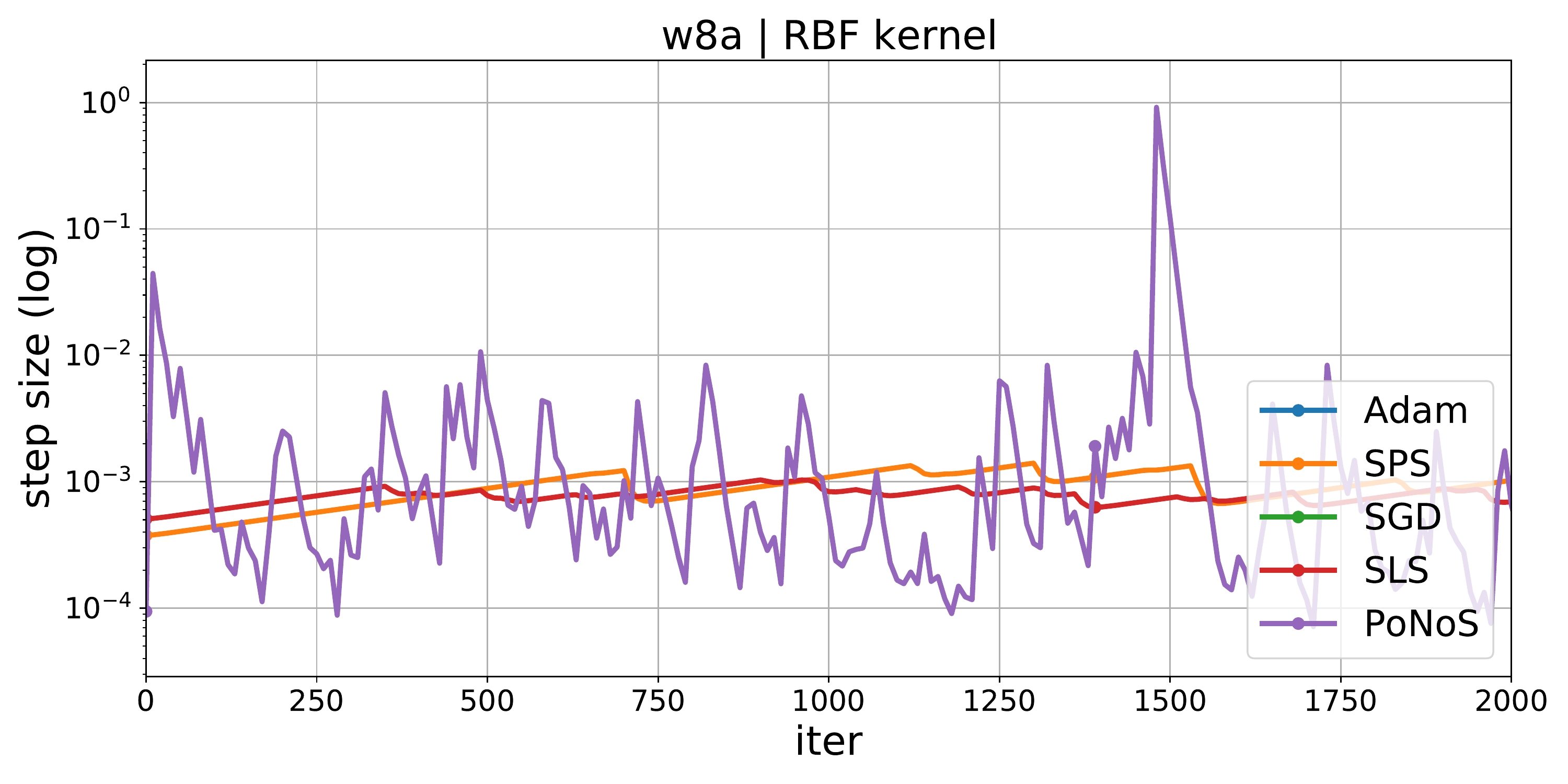}}
		\caption{Comparison between the proposed method (PoNoS) and the state-of-the-art on convex kernel models for binary classification. Each row focus on a different dataset. First column: exponentially averaged train loss. Second column: test accuracy. Third column: exponentially averaged step size.}\label{fig:convex_short}
	\end{minipage}
\end{figure}
	\subsection{Experiments on Transformers}\label{sec:supp_trans}
	See Figure \ref{fig:trans} for the complete results relative to experiments on transformers of Figure 4 of the main paper. In these experiments, we followed the setup by \citet{kunstner23a}. The word-level language modeling has been addressed with sequences of 35 or 128 tokens respectively for PTB and Wikitext2. In case of PTB, we employed a simple transformer model whose architecture consists of an 200-dimensional embedding layer, 2 transformer layers (2-head self attention, layer normalization, linear$(200\times200)$-ReLU-linear$(200\times200)$, layer normalization) followed by a linear layer. The data processing and the implementation of the Transformer-XL follow \citet{dai19a}. In the case of Wikitext2, the hyperparameters are set as in the ENWIK8 base experiment of \citet{dai19a}, except with the modifications of \citet{zhang20a}, using 6 layers and a target length of 128.

Since Adam is commonly known to achieve better performances than SGD on these networks \citep{kunstner23a}, we develop a preconditioned version of PoNoS, SLS and SPS respectively called PoNoS\_prec, SLS\_prec, and SPS\_prec. These algorithms differ from the originals in three aspects. First, they all employ a direction which is Adam without momentum ($\beta_1=0$). More precisely, the mini-batch gradient $\gradik(w_k)$ in Step \ref{step:wk} of Algorithm \ref{alg} is replaced with $d_k$ as computed below
\begin{equation*}
	\begin{split}
		g_{k} &= \gradik(w_k) \\
		v_{k} &= \beta_2 \cdot v_{k-1} + (1-\beta_2)\cdot g_{k}^2 \\
		\hat v_{k} &=  v_k /(1-\beta^k_2) \\
		d_k &= - g_{k}/(\sqrt{\hat v_{k} }+\epsilon),
	\end{split}
\end{equation*}
where all the operations on vectors are to be considered component-wise, $\beta_2\in(0,1)$ and $\epsilon>0$ is a small constant. We use the default values from Adam for $\beta_2$ and $\epsilon$, that is $\beta_2=0.9$ and $\epsilon=10^{-8}$.
The second difference concerns the line search, as PoNoS\_prec and SLS\_prec exploit a condition that reflects the use of the above preconditioned direction, i.e., 
\begin{equation*}
	\fik(w_k + \eta_k d_k) \leq R_k + c \cdot \eta_k \langle d_k, \gradik(w_k)\rangle = R_k - c \cdot \eta_k  \sum_{j=0}^{n-1}\frac{\left(\gradik(w_k)\right)_j^2}{\left(\sqrt{\hat{v}_{k} }+\epsilon\right)_j},
\end{equation*}
where $(\cdot)_j$ is referring to the $j$-th component of the vector in brackets, and $R_k$ is either $C_k$ in case of PoNoS\_prec or $\fik(w_k)$ in case of SLS\_prec.
The third difference concerns the Polyak step size, which is also computed taking into account the above directional derivative $\langle d_k, \gradik(w_k)\rangle$. In particular, PoNoS\_prec and SPS\_prec replace \eqref{eq:polyak} with the following
\begin{equation*}
	\tilde{\eta}_{k,0}:=\frac{\fik(w_k) - \fik^*}{-c_p\langle d_k, \gradik(w_k)\rangle}.
\end{equation*}

As in \cite{kunstner23a} we focus on the training procedure, however, in the second column of Figure \ref{fig:trans} we report the perplexity of the language model on the test. With respect to Figure 4 of the main paper, in Figure \ref{fig:trans} we also show the gradient norm and the step size yielded by the different methods. From these plots, it is possible to notice that the range in which the step size varies is reduced if compared with that of Figure \ref{fig:exp1}. This behavior depends on the dynamics of the loss and of the norm of the gradient, which are also reduced if compared with those of Figure \ref{fig:exp1}. The reduced ranges of loss and gradient norm might be an issue for Polyak-based algorithms since they rely on these measures. Furthermore, the good results of SLS\_prec on \wiki$ $ suggest that other initial step sizes might also be suited for training transformers. We leave such exploration to future works.

\renewcommand{\dir}{trans}
\renewcommand{\imgS}{0.30}
\begin{figure}[!h] 
	\hspace{-9mm}
	\begin{minipage}{0.9\textwidth}
		\renewcommand{\model}{trans_enc}
		\renewcommand{\modelname}{trans\_enc}
		\subfloat{\includegraphics[width=\imgS\linewidth]{\dir/\model/train_loss}}
		\subfloat{\includegraphics[width=\imgS\linewidth]{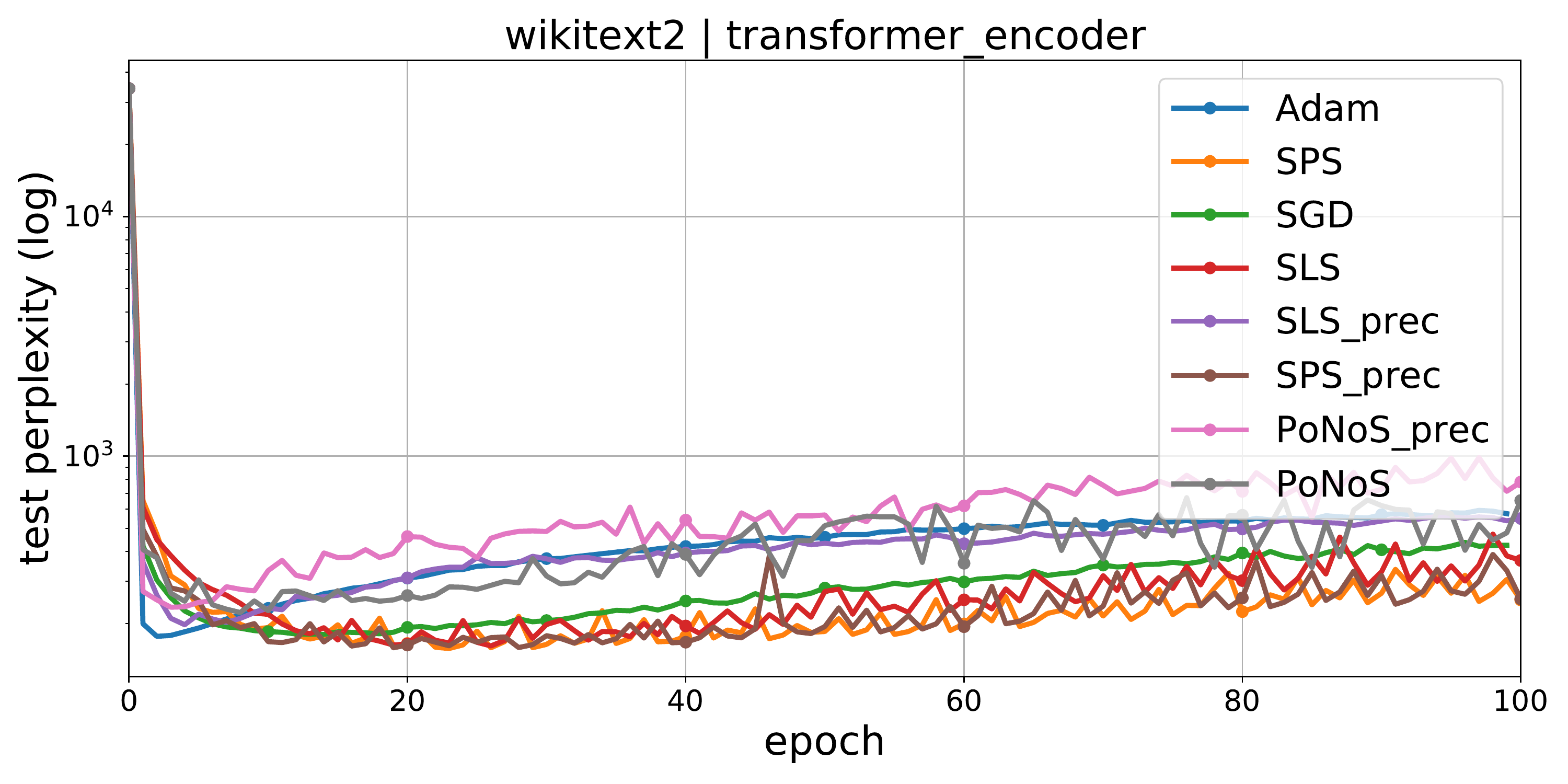}}
		\subfloat{\includegraphics[width=\imgS\linewidth]{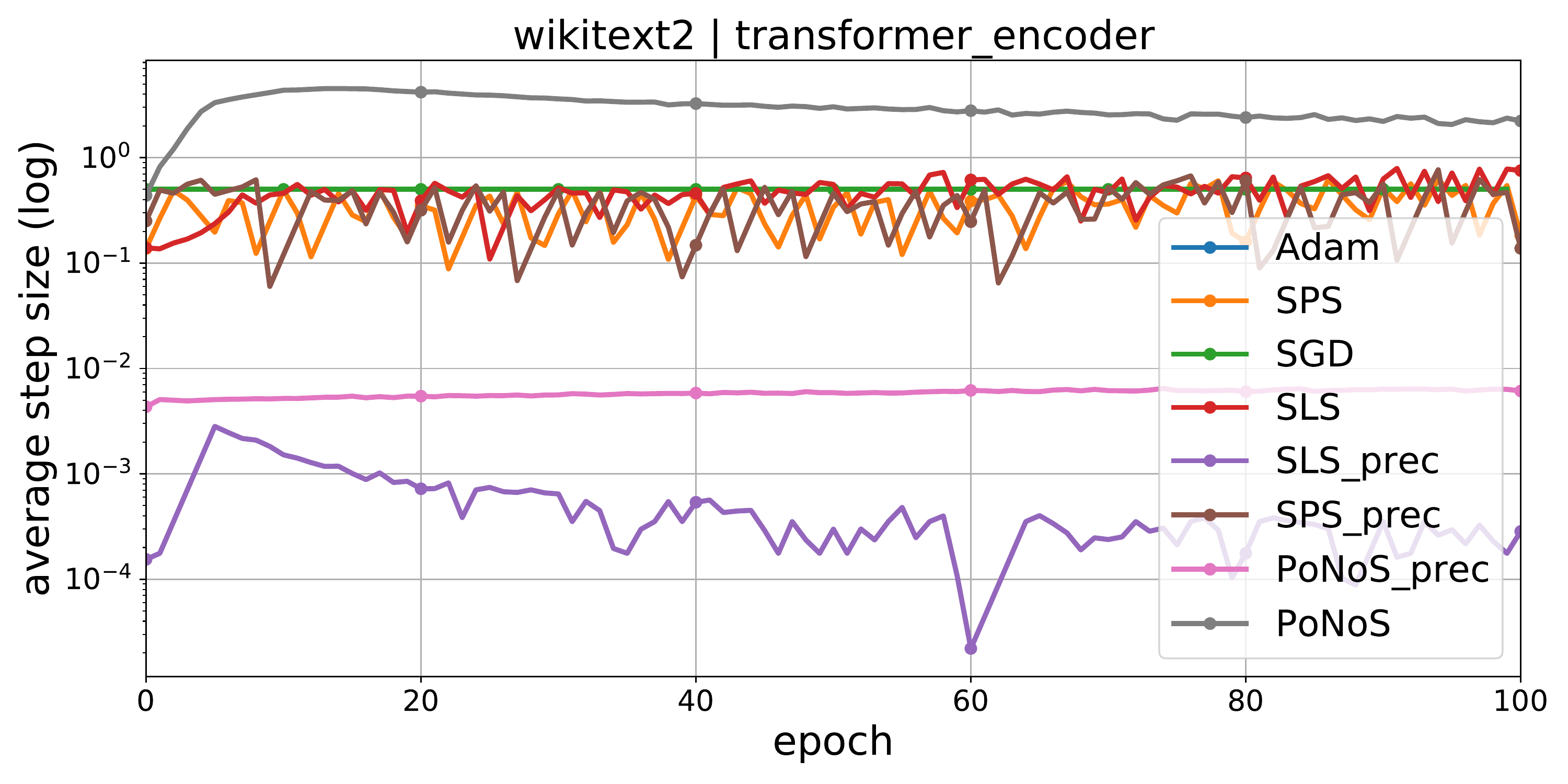}}
		\subfloat{\includegraphics[width=\imgS\linewidth]{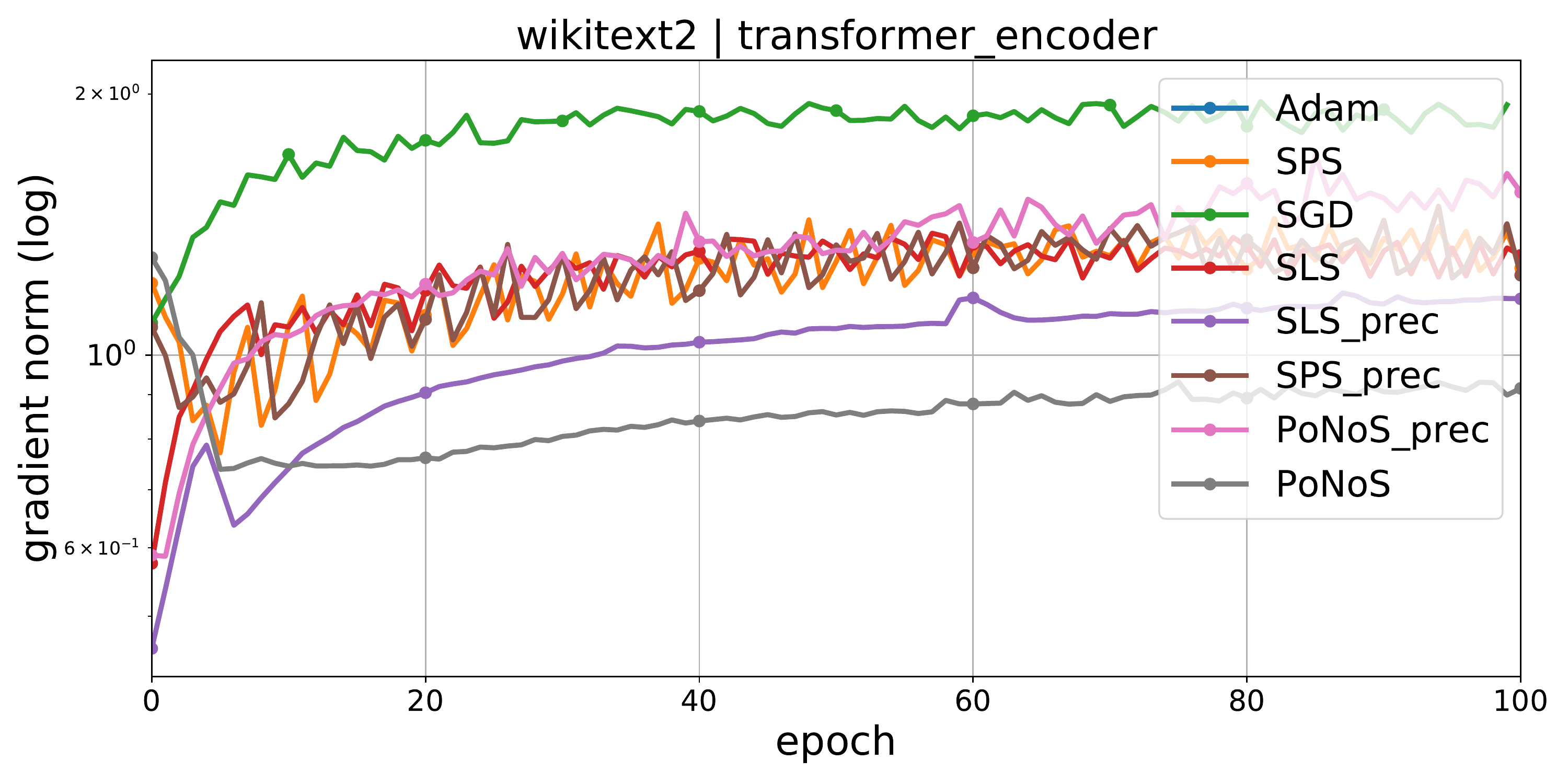}}\\
		\renewcommand{\model}{trans_xl}
		\renewcommand{\modelname}{trans\_xl}
		\subfloat{\includegraphics[width=\imgS\linewidth]{\dir/\model/train_loss}}
		\subfloat{\includegraphics[width=\imgS\linewidth]{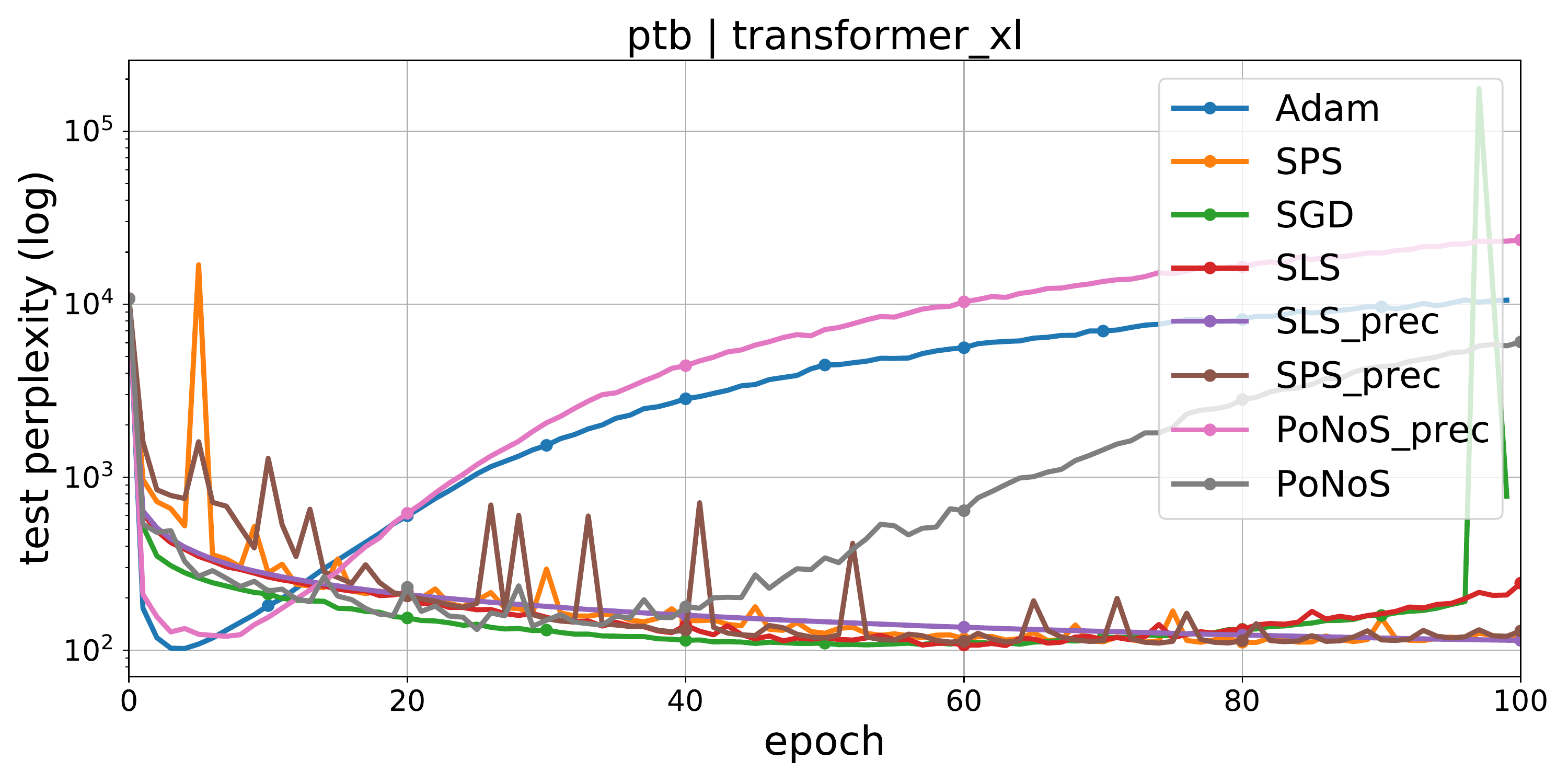}}
		\subfloat{\includegraphics[width=\imgS\linewidth]{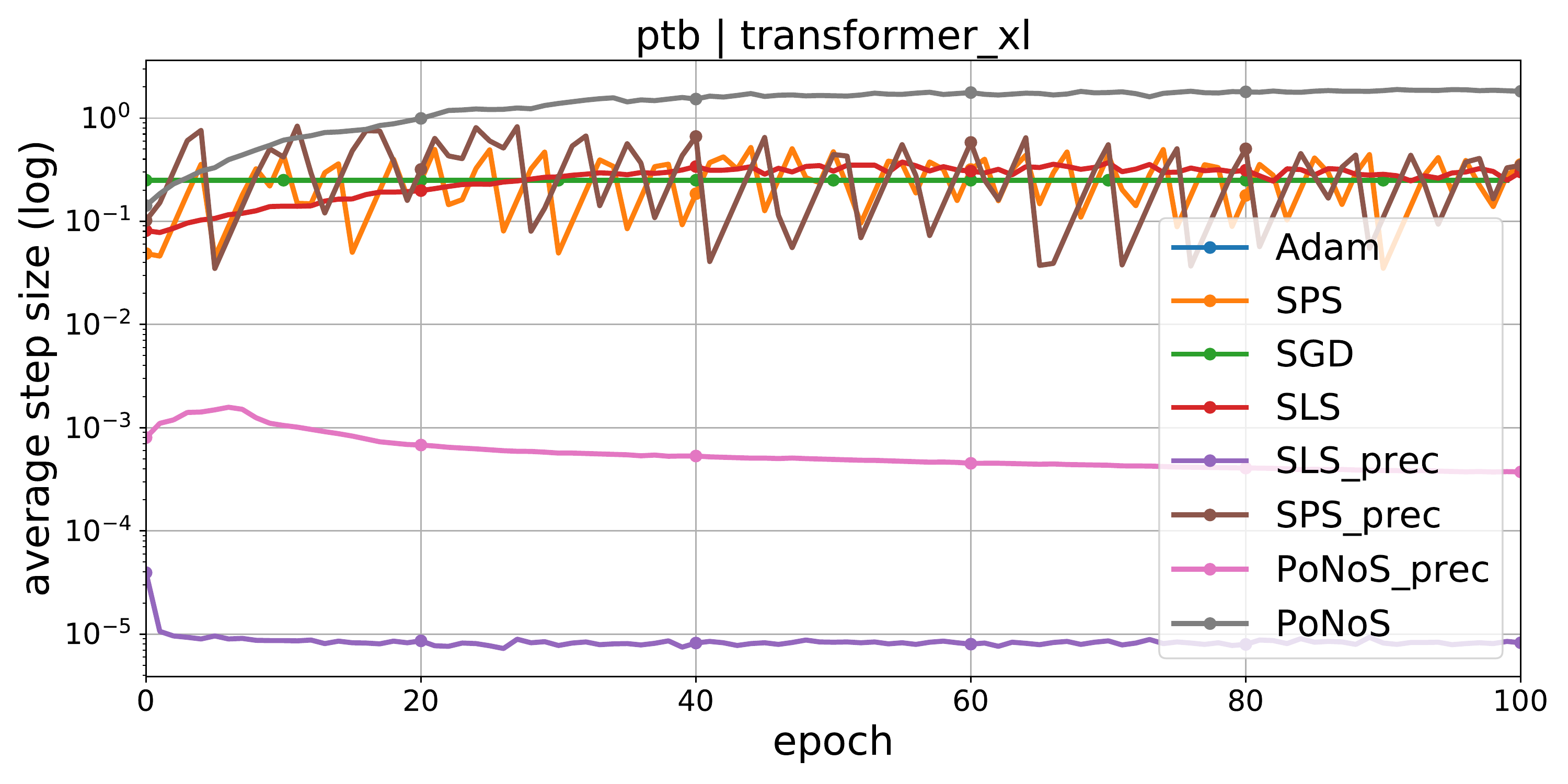}}
		\subfloat{\includegraphics[width=\imgS\linewidth]{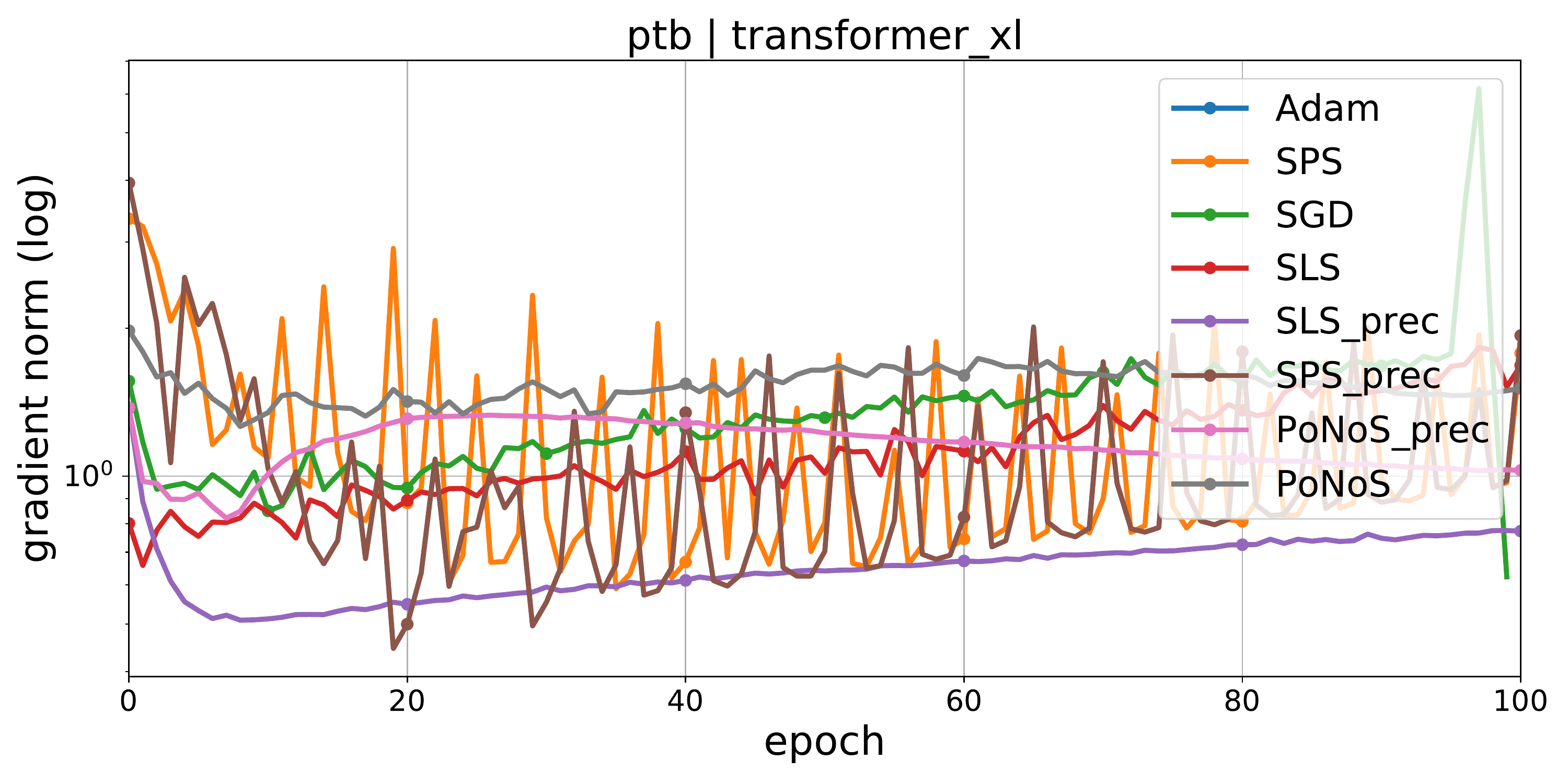}}\\
		\caption{Comparison between the proposed method (PoNoS) and the state-of-the-art on the training of transformers for language modeling tasks. Each row focus on a dataset/model combination. First column: train loss. Second column: train perplexity. Third column: average step size. Fourth column: average gradient norm.}\label{fig:trans}
	\end{minipage}
\end{figure}

	\section{Additional Plots}\label{sec:supp_add_plots}
	\subsection{Study on the Choice of $c$: Theory (0.5) vs Practice (0.1)}\label{sec:supp_c}
	In these experiments, we consider the constant $c$ in \eqref{eq:zhang} and \eqref{eq:armijo}. As described in Section 4 of the main paper, this value is required to be larger than $\frac{1}{2}$ in both Theorems 1 and 2 (and also in the corresponding monotone versions from \citet{vaswani19a}). This value is often considered too large in
practice and the default choice is $0.1$ (also for SLS \citep{vaswani19a}) or smaller \citep{nocedal06a}. The constant $c$ controls the weight of the sufficient decrease in the line search conditions and a smaller value of $c$ corresponds to a looser line search.
In this subsection, we numerically try both $c=0.5$ and $c=0.1$. 
In Figure \ref{fig:theoretically_fixed_c} and \ref{fig:convex} we compare 
\begin{itemize}
	\item monotone|0.1: the monotone stochastic Armijo line search \eqref{eq:armijo} with $c=0.1$. 
	\item monotone|0.5: the monotone stochastic Armijo line search \eqref{eq:armijo} with $c=0.5$.
	\item zhang|0.1: the nonmonotone Zhang and Hager line search \eqref{eq:zhang} with $c=0.1$.
	\item zhang|0.5: the nonmonotone Zhang and Hager line search \eqref{eq:zhang} with $c=0.5$. This setting corresponds to PoNoS.
\end{itemize}
For all the above algorithms the initial step size is \eqref{eq:loizou}. From Figure \ref{fig:theoretically_fixed_c} we can observe that:
\begin{itemize}
	\item zhang|0.5 (PoNoS) achieves the best performances both in terms of train loss and test accuracy. It is the only algorithm able to reduce the train loss below the threshold of $10^{-6}$ on the problems \res$ $ and \ress. Apart from the case of \svhn$ $(where it loses less than $0.5$ points w.r.t. zhang|0.1), it always achieves the best test accuracy.
	\item monotone|0.1 and zhang|0.1 are generally competitive in terms of training loss, but they achieve very poor generalization skills on \ress $ $ and \denses. In particular, their step size on these two problems is growing very rapidly. Already in the first epoch, the average step size is greater than 5, thanks to the fact that both algorithms are reducing the gradient norm rapidly below 1.
	\item The step size yielded by zhang|0.5 starts substantially lower than those of monotone|0.1 and zhang|0.1. Afterwards, the step size increases, then stabilizes, sometimes slightly decreases and finally increases again. This behavior is accomplished thanks to the combination between a larger $c$ and a nonmonotone line search. Also the gradient norm is affected by this choice, since it does not decrease as suddenly as for monotone|0.1 and zhang|0.1. 
	\item monotone|0.5 never achieves the best test accuracy nor the best training loss.  This line search is too strict and yields step sizes that are very small in comparison to those yielded by zhang|0.5.
\end{itemize}	
\par In conclusion, PoNoS (zhang|0.5) employs a large constant $c$ ($0.5$), but it combines that with a nonmonotone line search to achieve the best middle way between strictness and tolerance.

\renewcommand{\dir}{theoretically_fixed_c/}
\renewcommand{\imgS}{0.30}
\begin{figure}[!h] 
	\hspace{-9mm}
	\begin{minipage}{0.9\textwidth}
		\renewcommand{\model}{mlp}
		\renewcommand{\modelname}{mlp}
		\subfloat{\includegraphics[width=\imgS\linewidth]{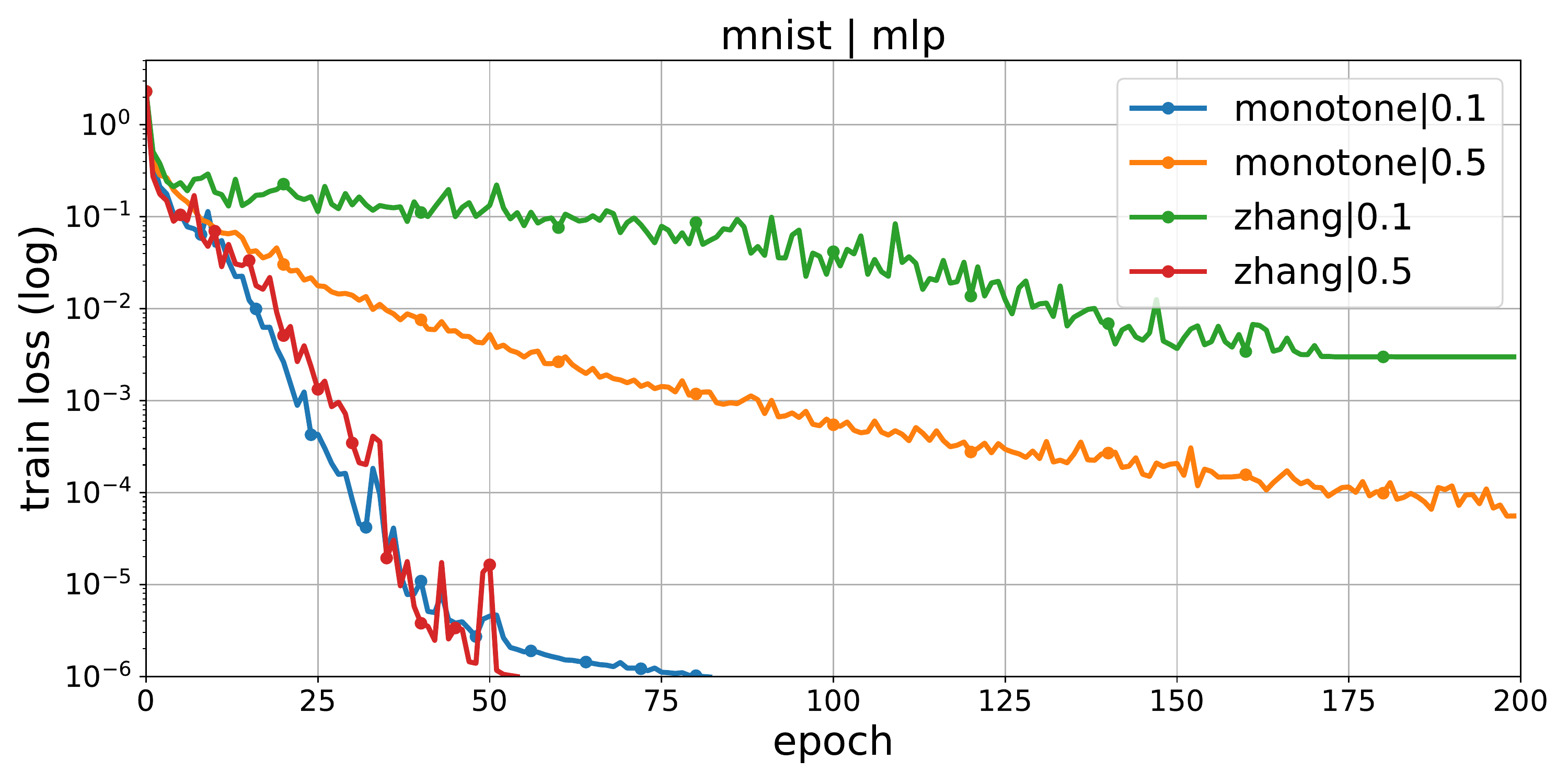}}
		\subfloat{\includegraphics[width=\imgS\linewidth]{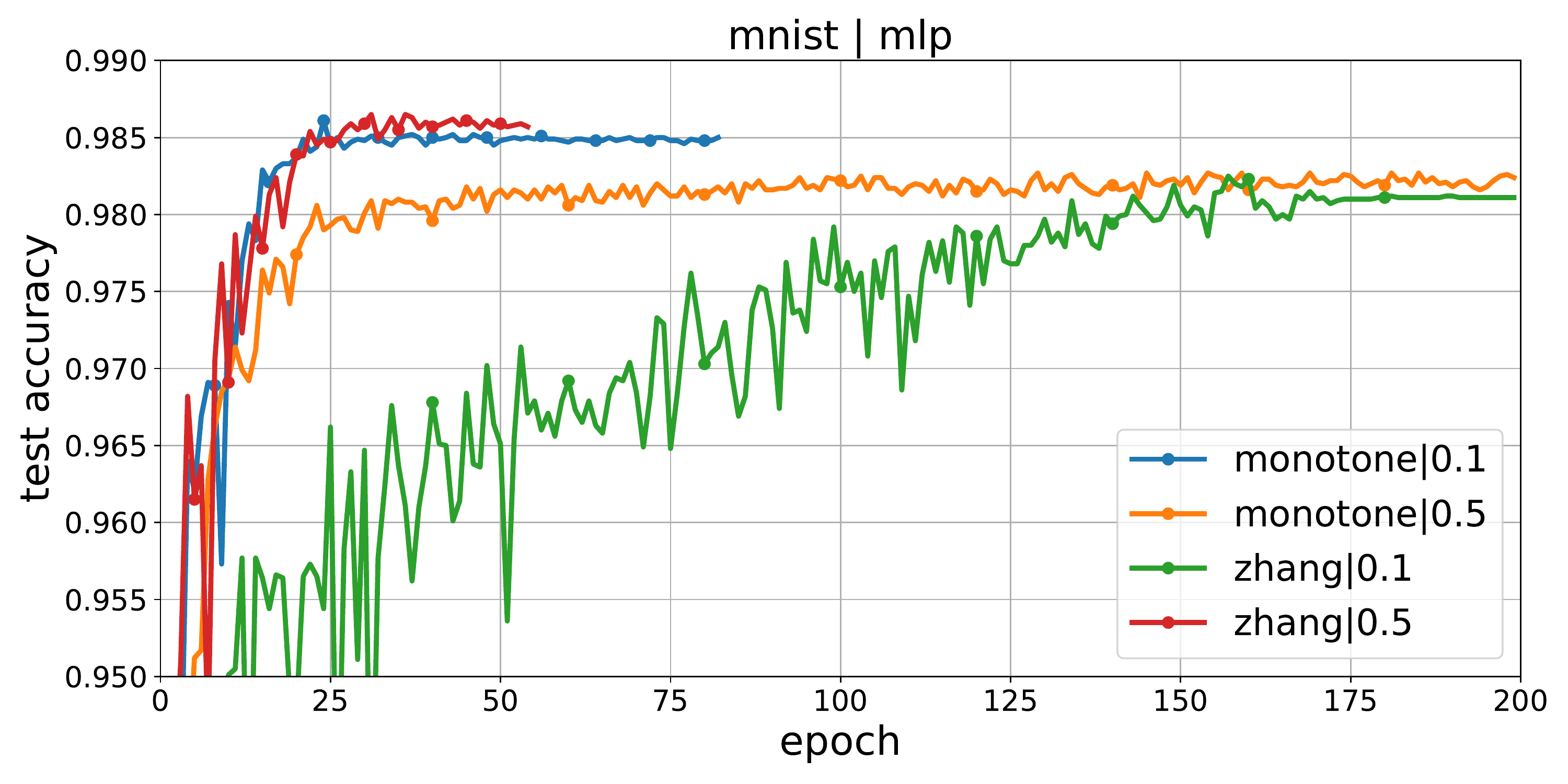}}
		\subfloat{\includegraphics[width=\imgS\linewidth]{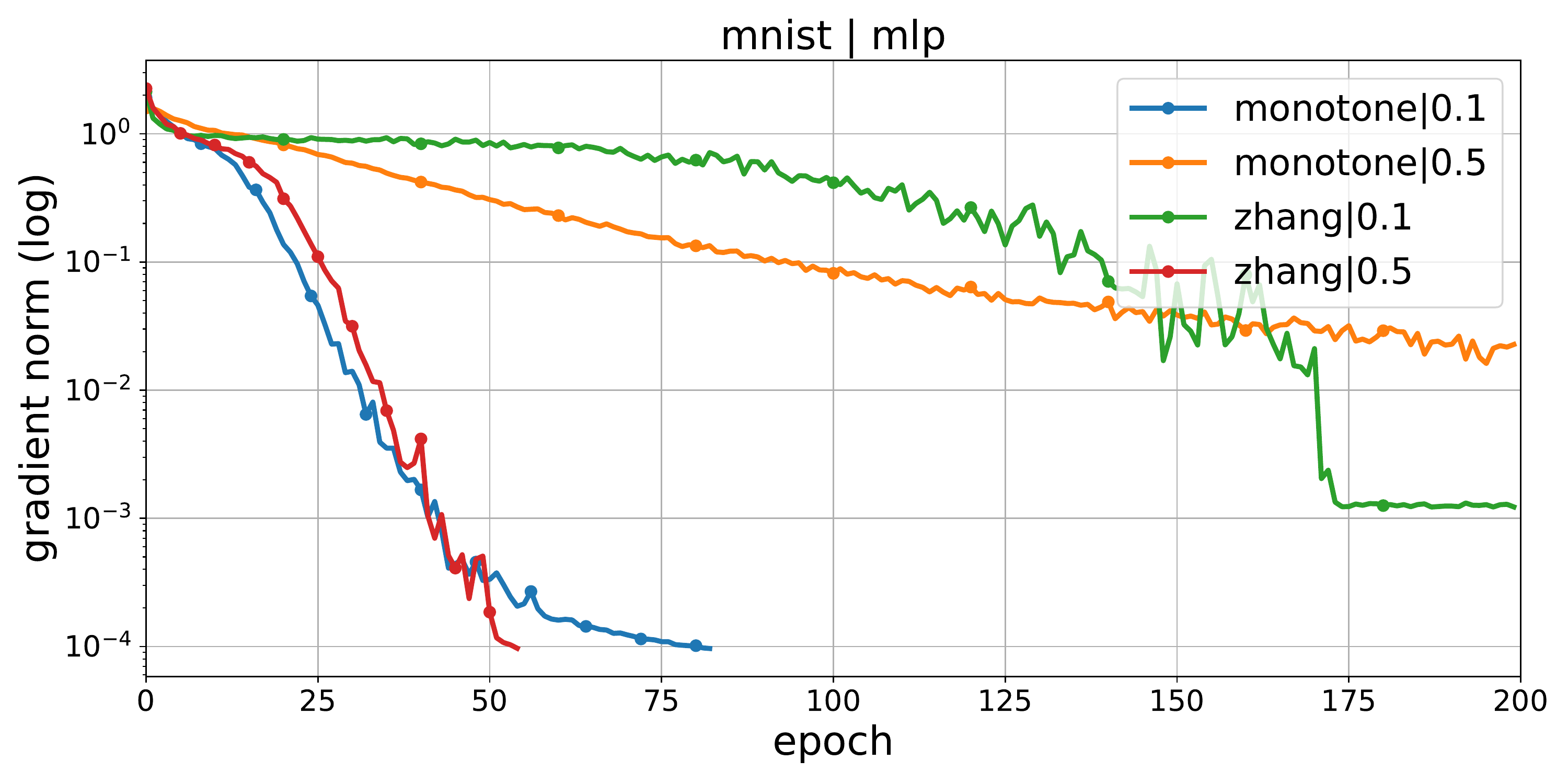}}
		\subfloat{\includegraphics[width=\imgS\linewidth]{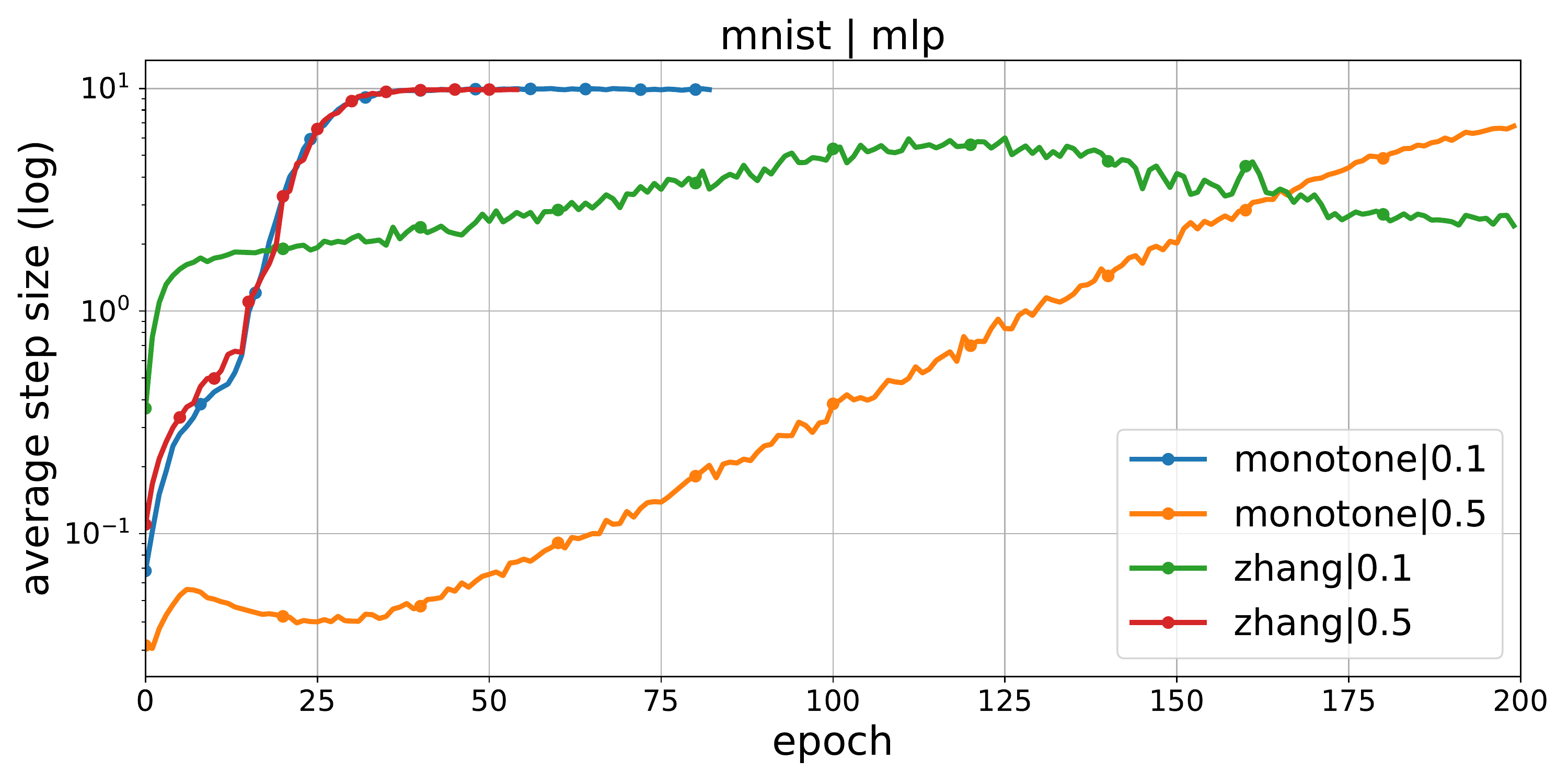}}\\
		\renewcommand{\model}{cifar10_resnet}
		\renewcommand{\modelname}{cifar10\_resnet}
		\subfloat{\includegraphics[width=\imgS\linewidth]{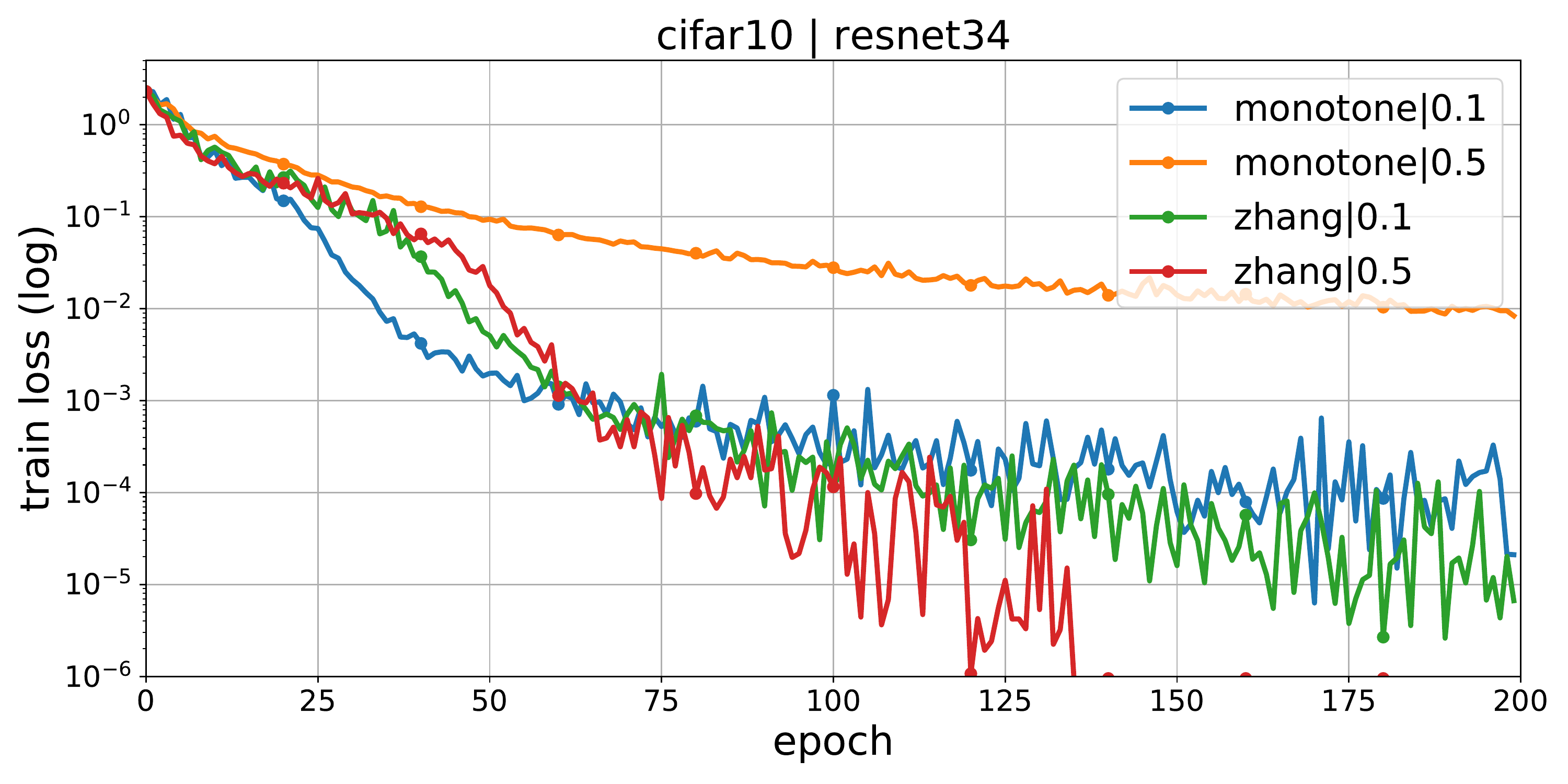}}
		\subfloat{\includegraphics[width=\imgS\linewidth]{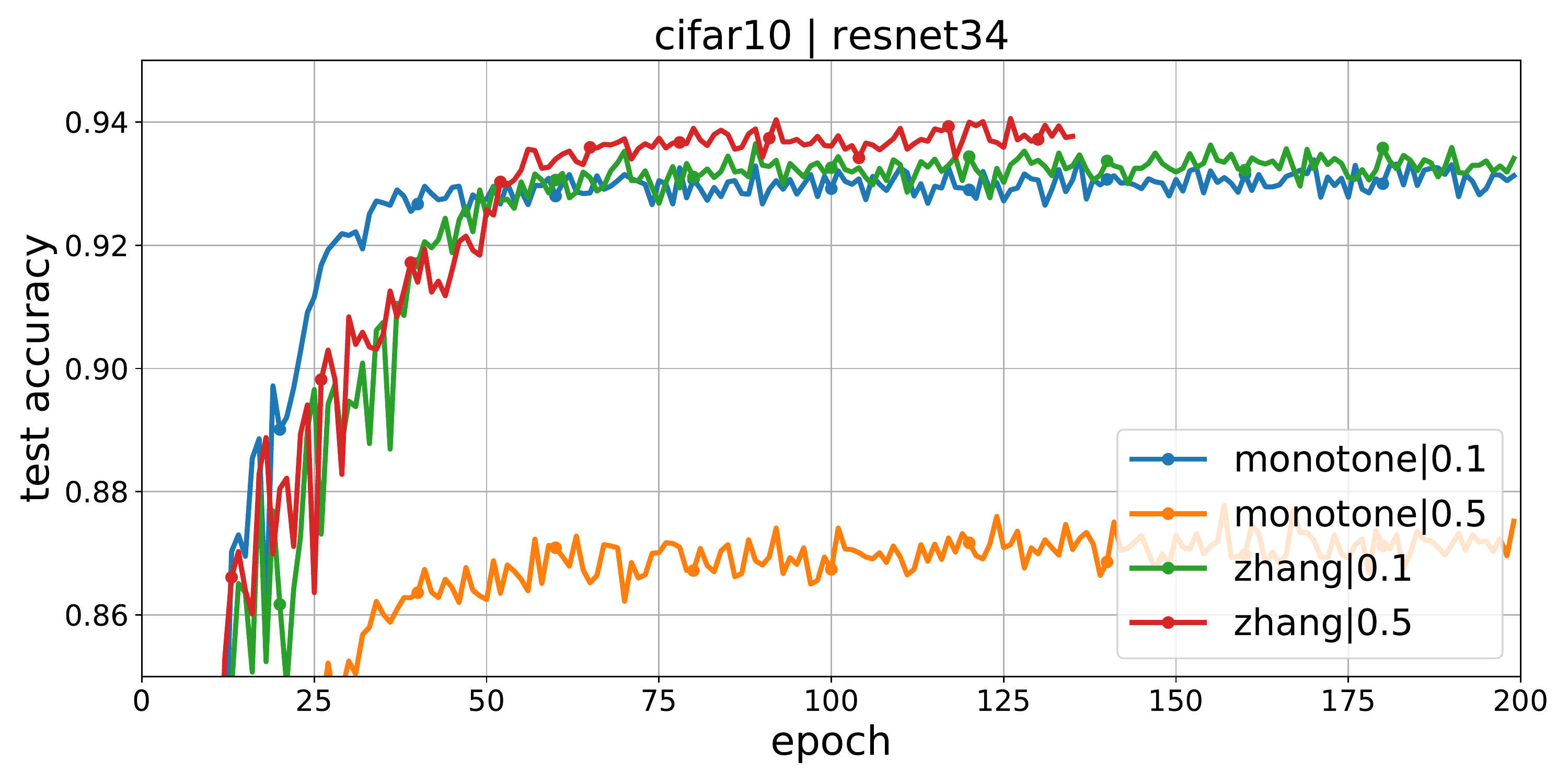}}
		\subfloat{\includegraphics[width=\imgS\linewidth]{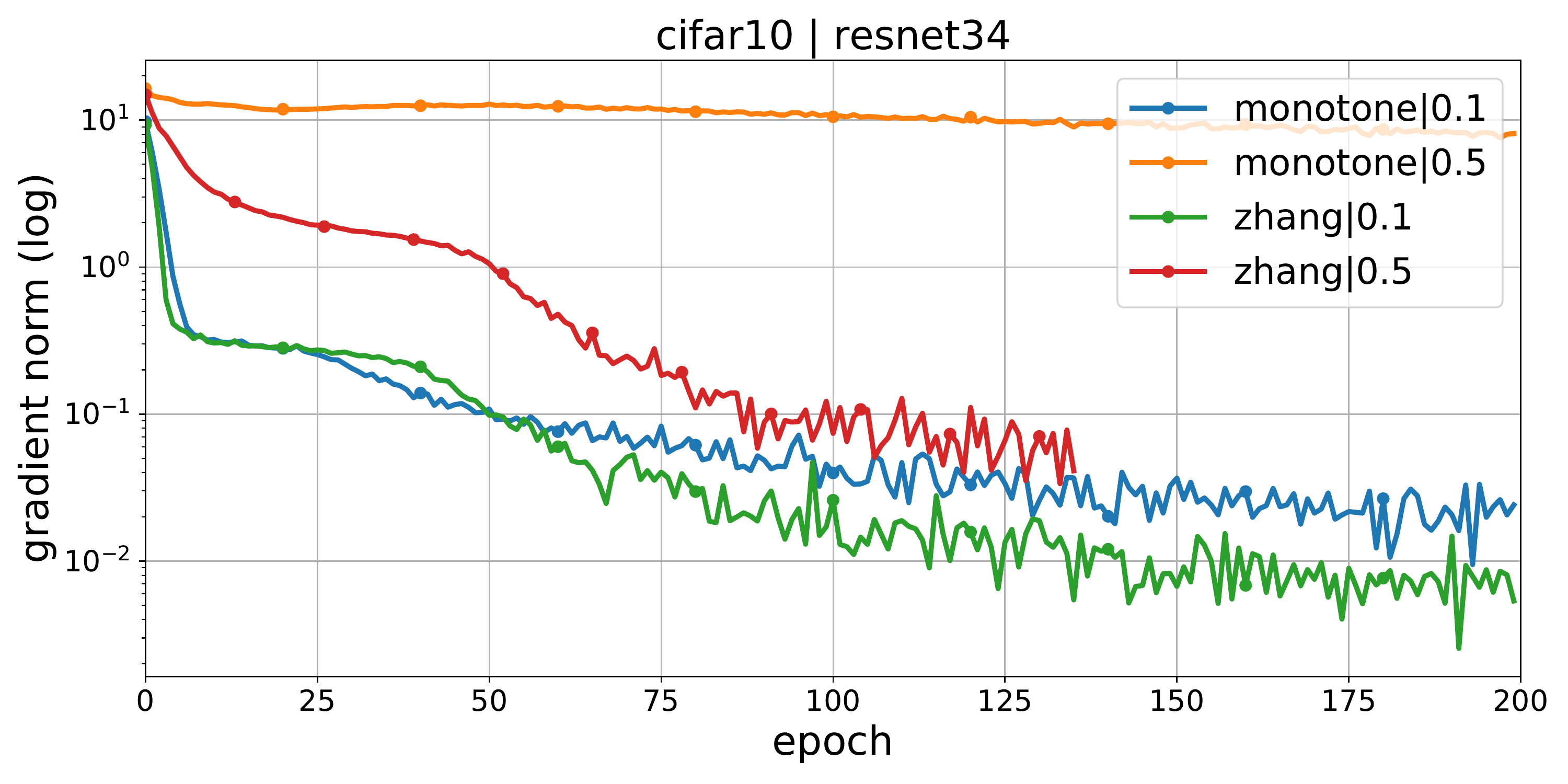}}
		\subfloat{\includegraphics[width=\imgS\linewidth]{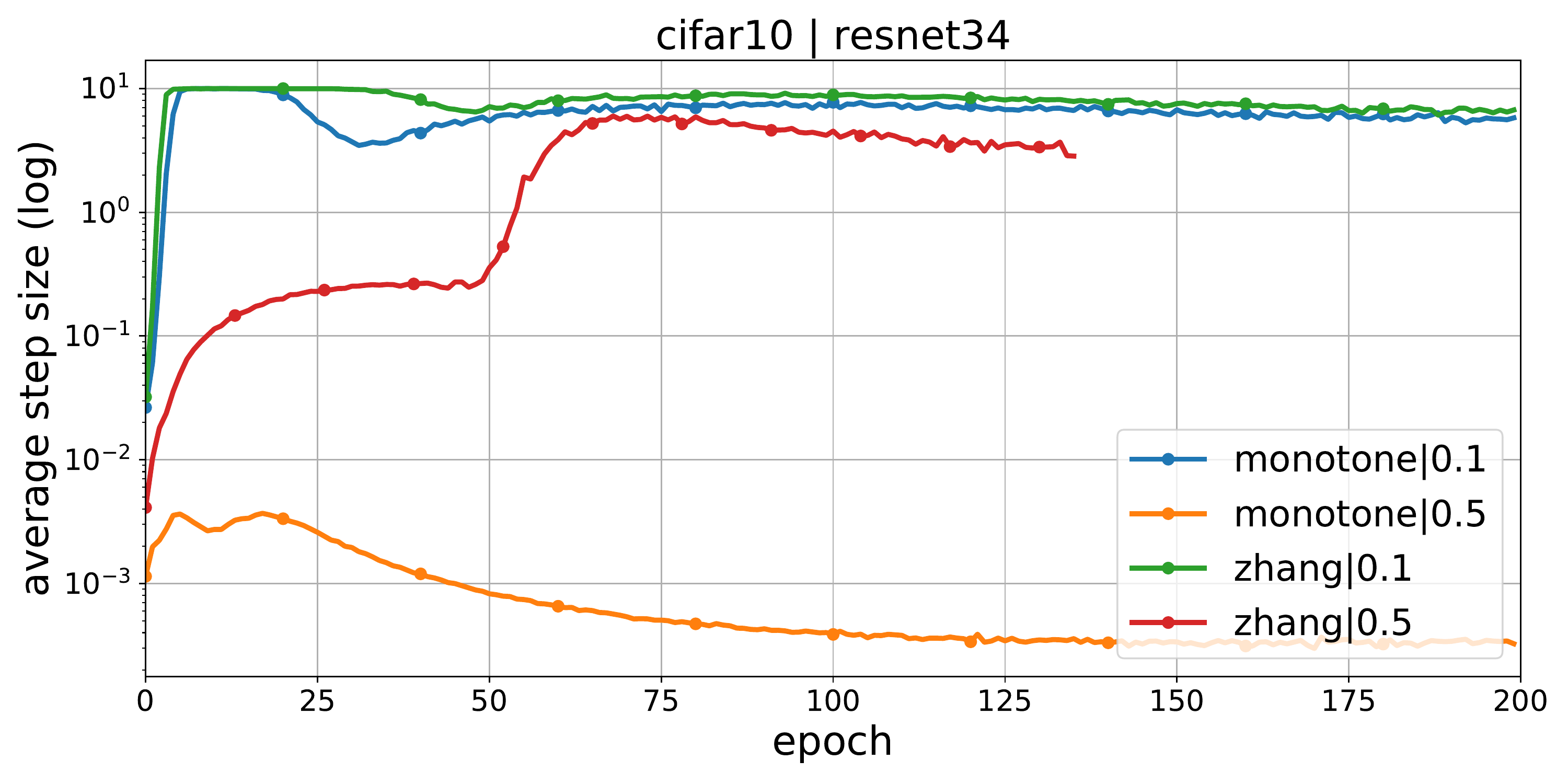}}\\
		\renewcommand{\model}{cifar10_densenet}
		\renewcommand{\modelname}{cifar10\_densenet}
		\subfloat{\includegraphics[width=\imgS\linewidth]{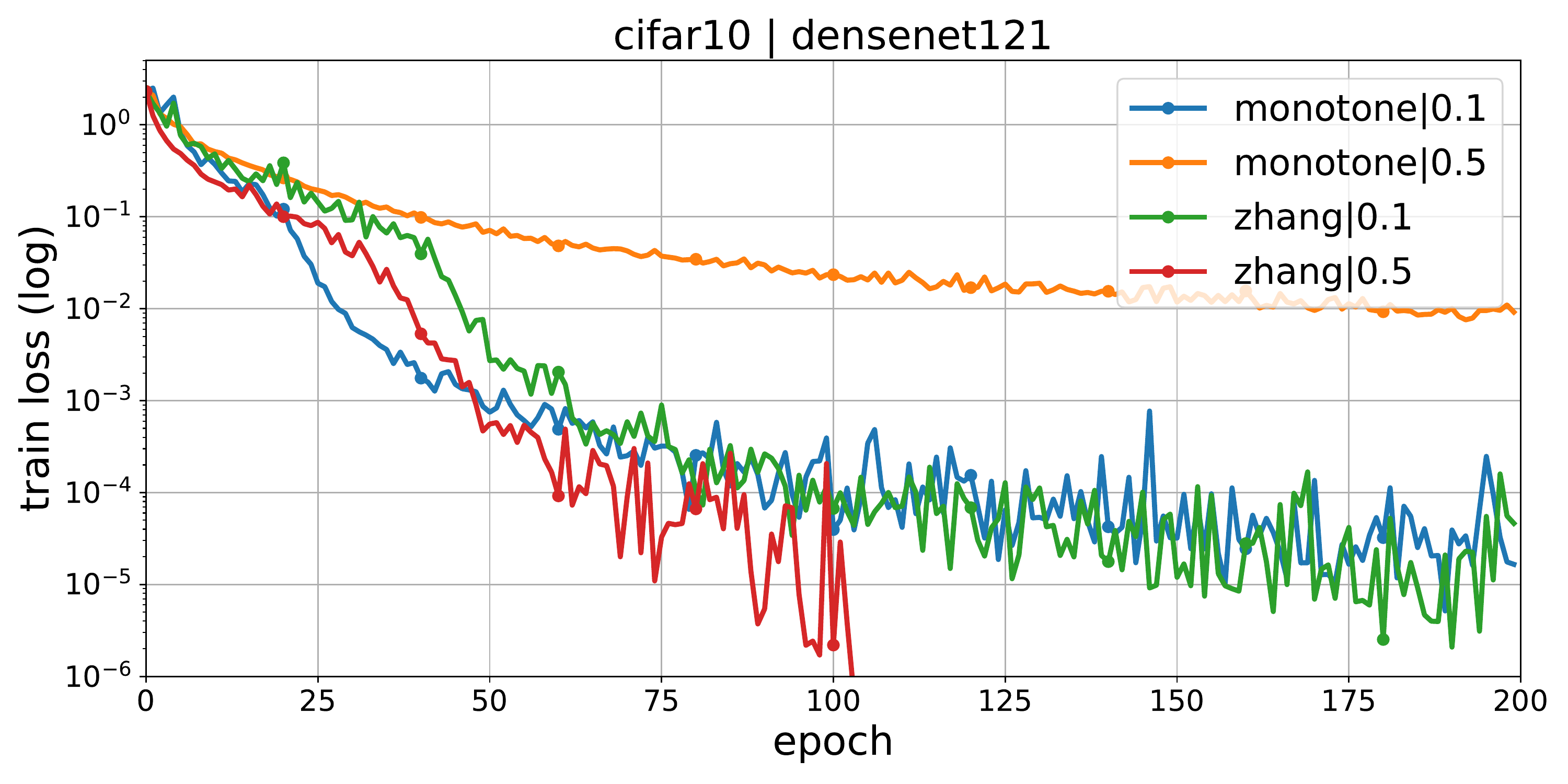}}
		\subfloat{\includegraphics[width=\imgS\linewidth]{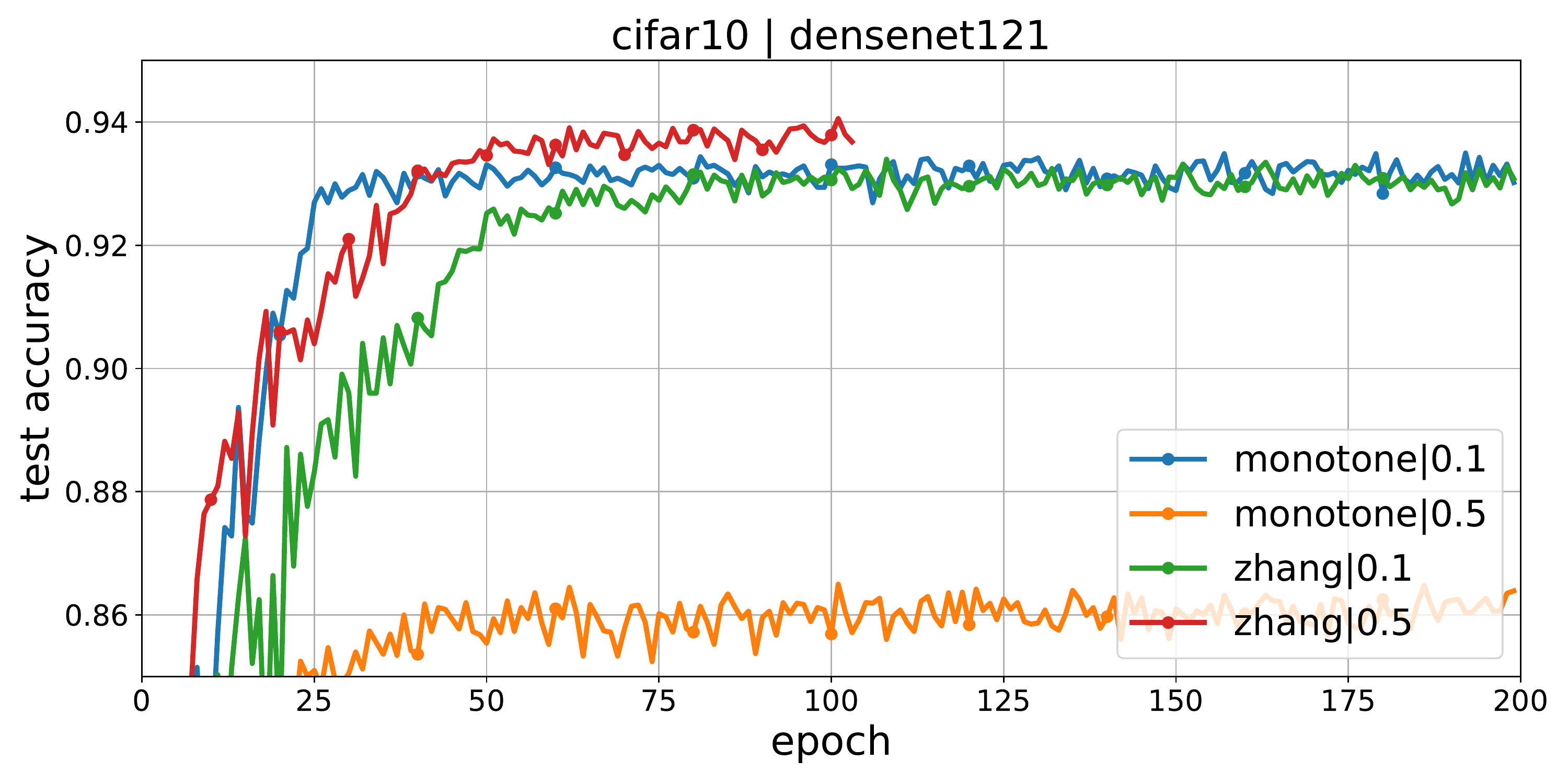}}
		\subfloat{\includegraphics[width=\imgS\linewidth]{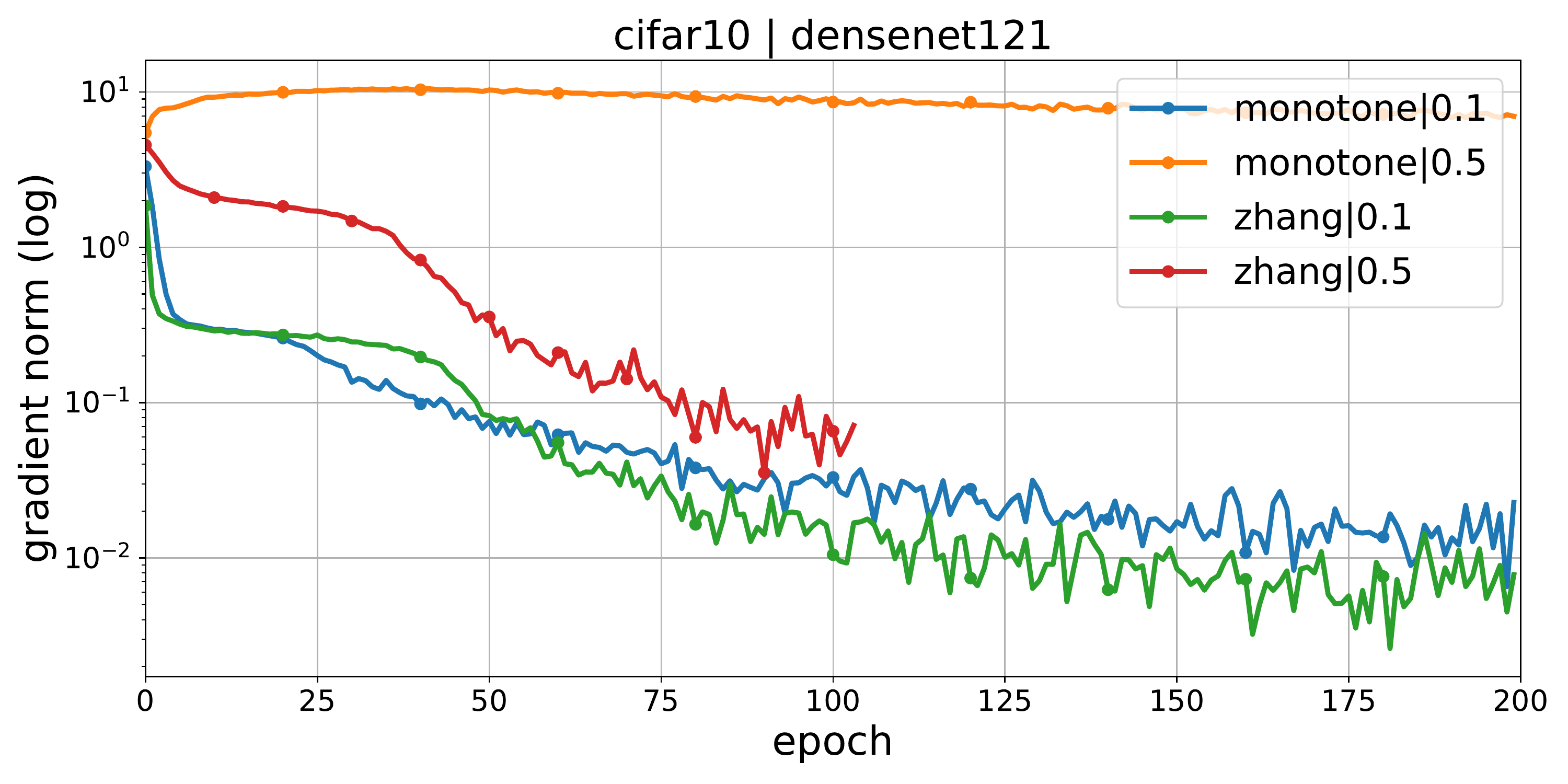}}
		\subfloat{\includegraphics[width=\imgS\linewidth]{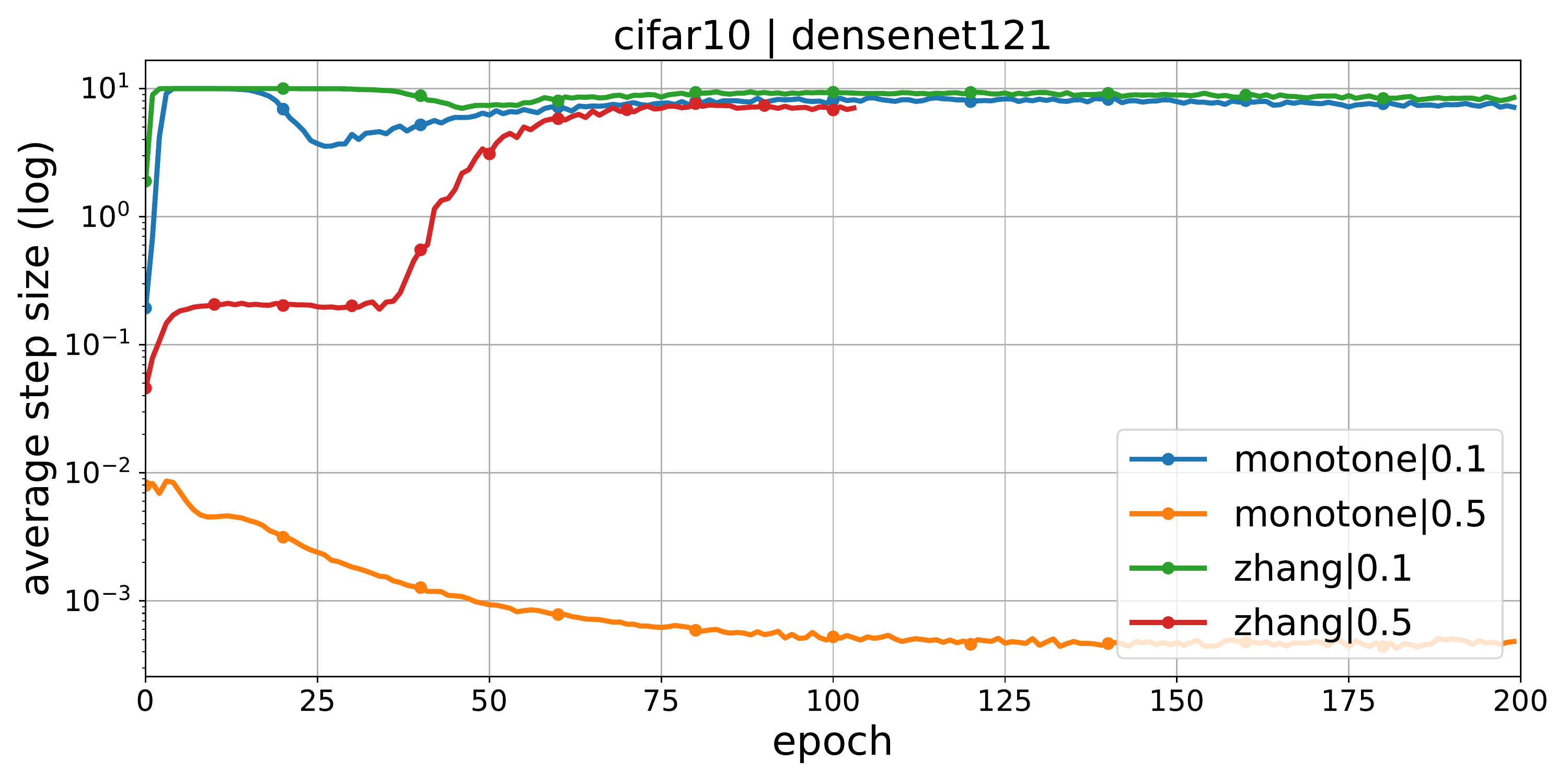}}\\
		\renewcommand{\model}{cifar100_resnet}
		\renewcommand{\modelname}{cifar100\_resnet}
		\subfloat{\includegraphics[width=\imgS\linewidth]{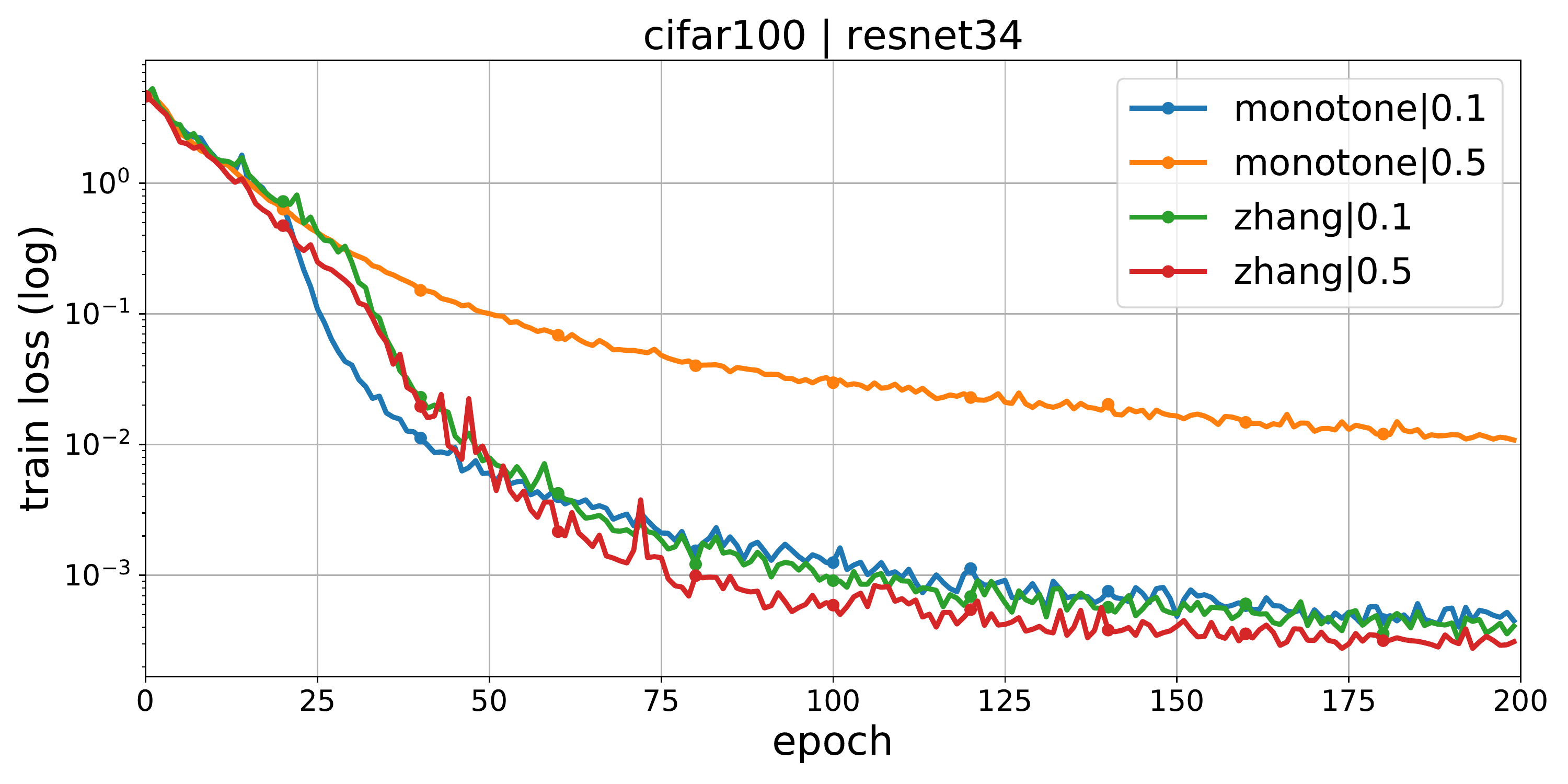}}
		\subfloat{\includegraphics[width=\imgS\linewidth]{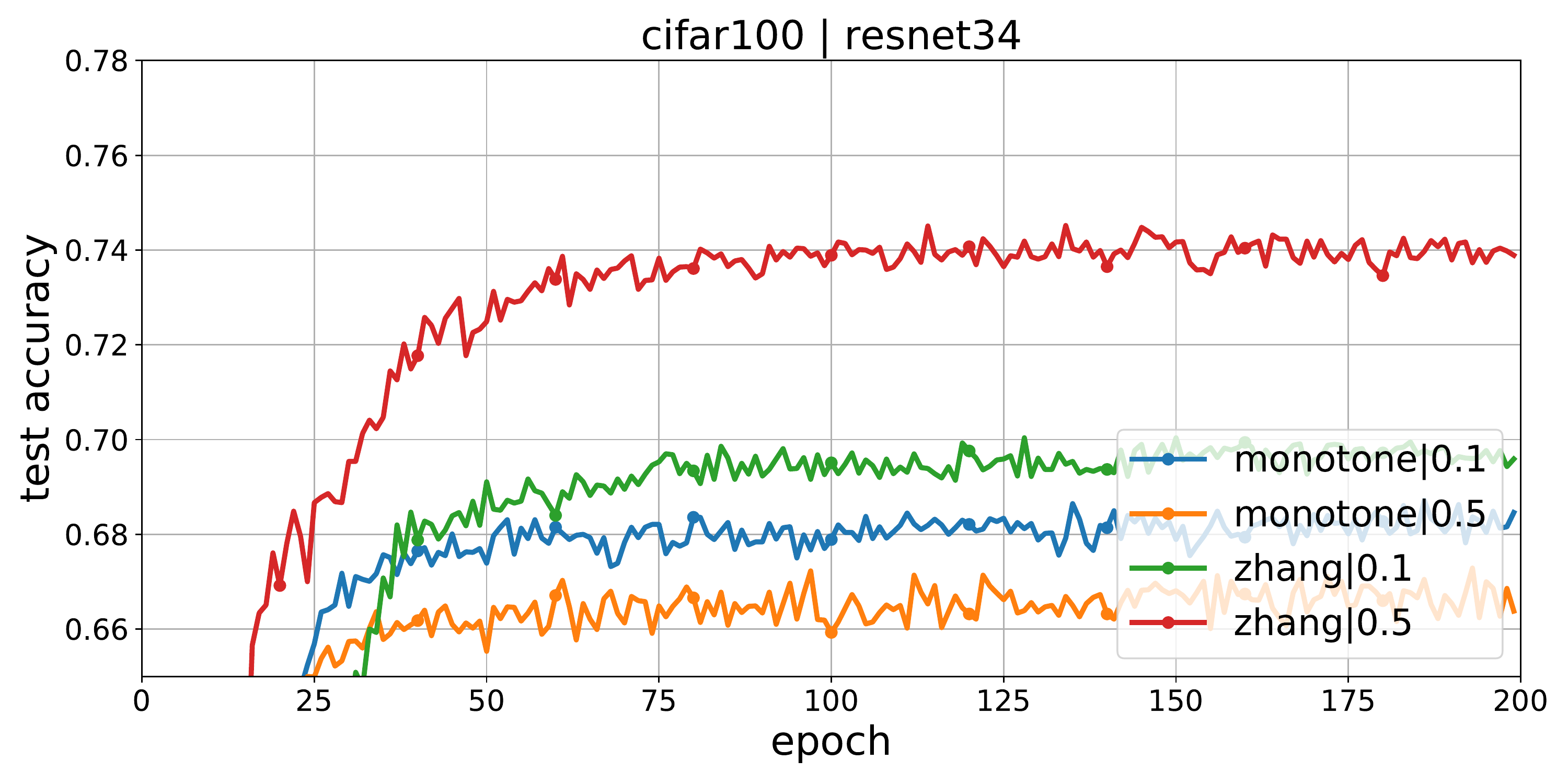}}
		\subfloat{\includegraphics[width=\imgS\linewidth]{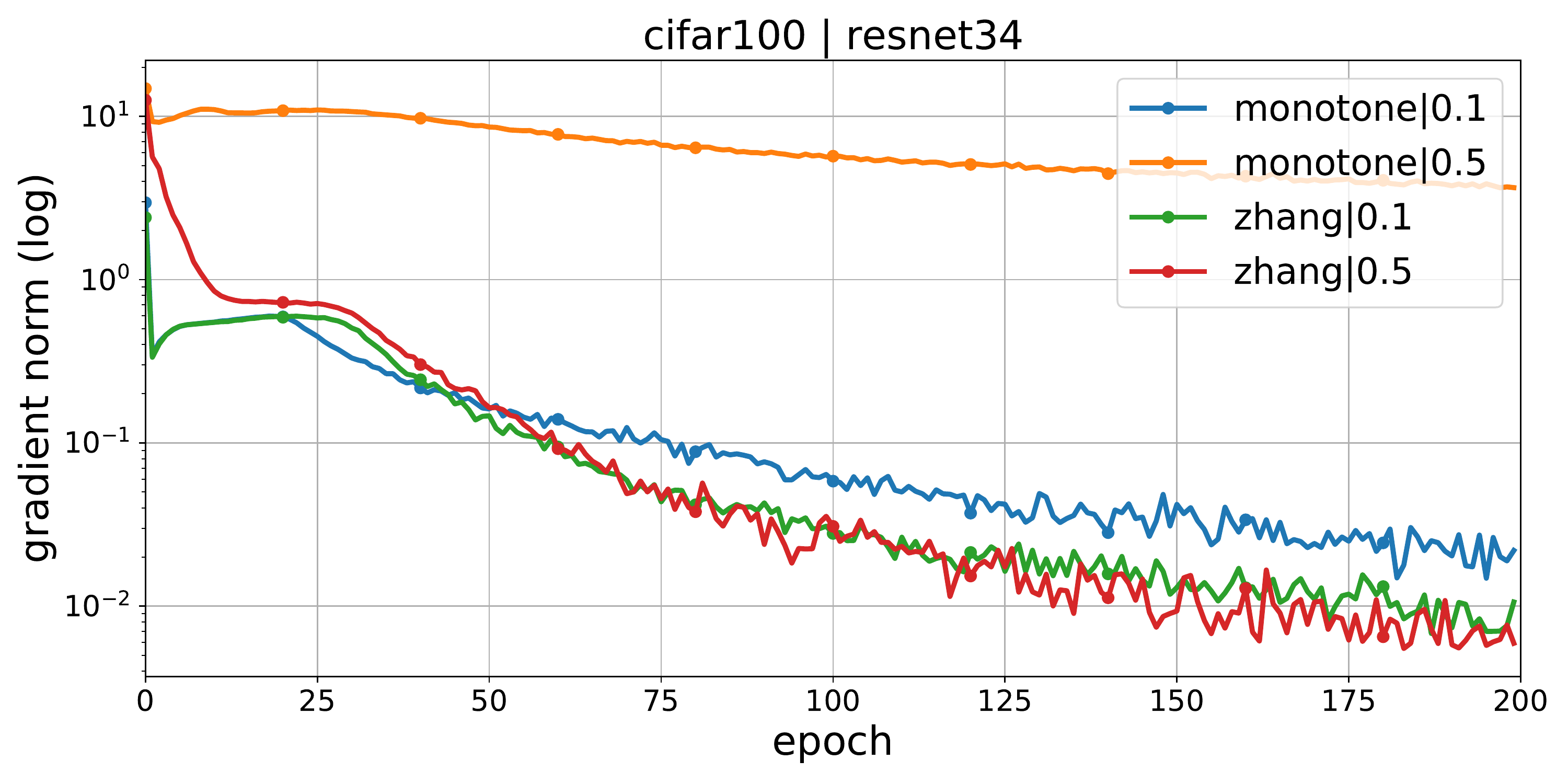}}
		\subfloat{\includegraphics[width=\imgS\linewidth]{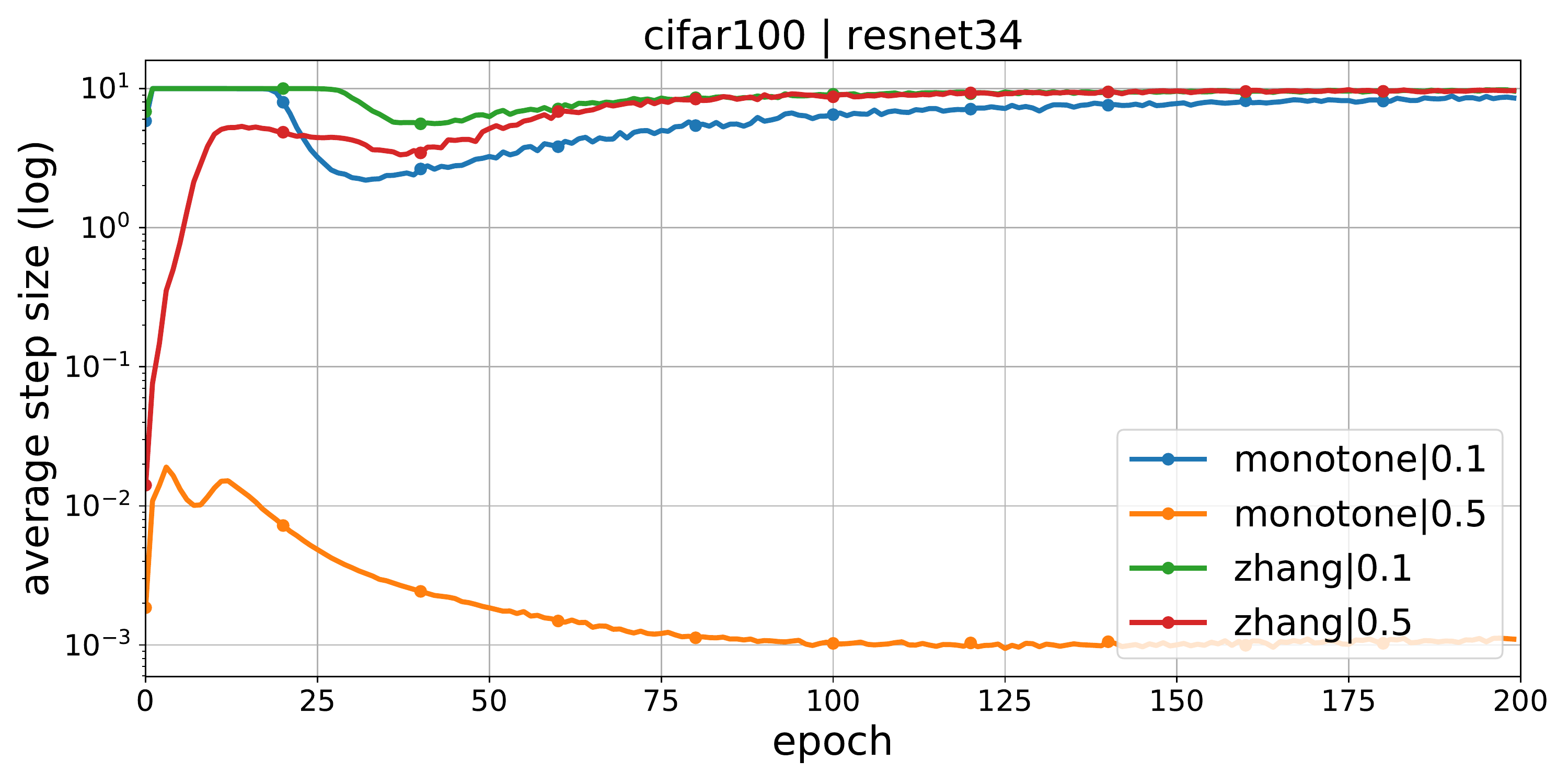}}\\
		\renewcommand{\model}{cifar100_densenet}
		\renewcommand{\modelname}{cifar100\_densenet}
		\subfloat{\includegraphics[width=\imgS\linewidth]{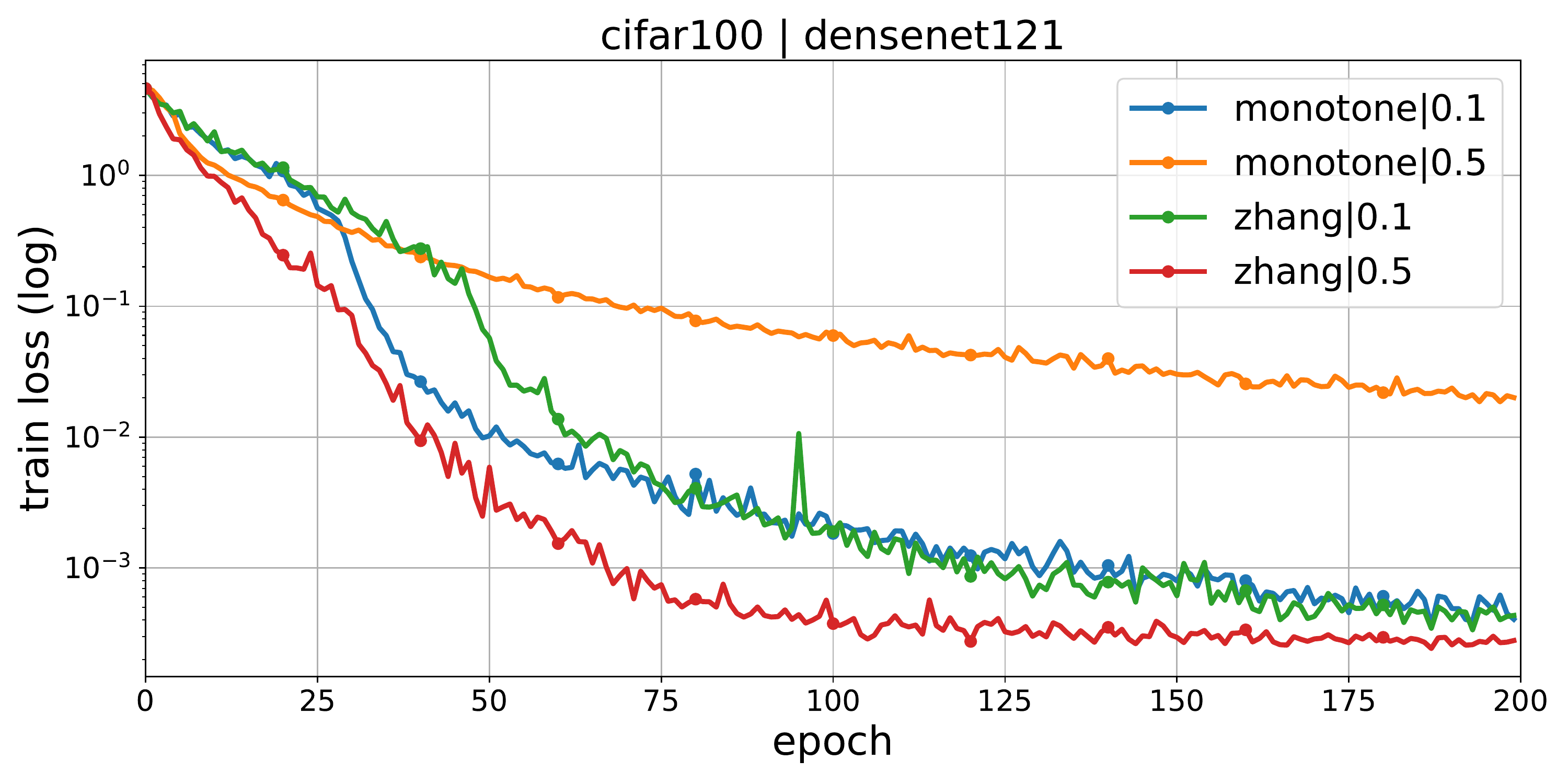}}
		\subfloat{\includegraphics[width=\imgS\linewidth]{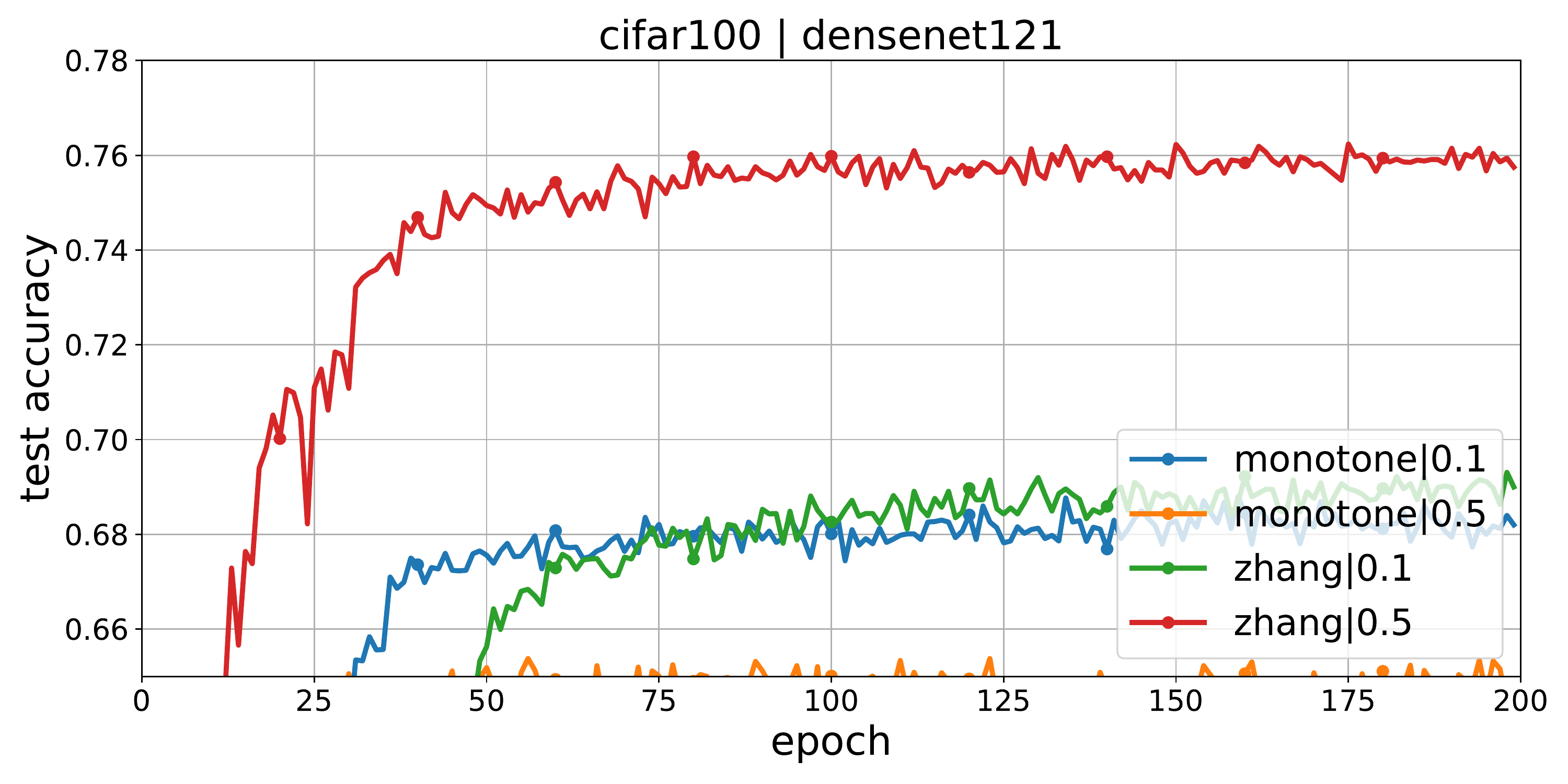}}
		\subfloat{\includegraphics[width=\imgS\linewidth]{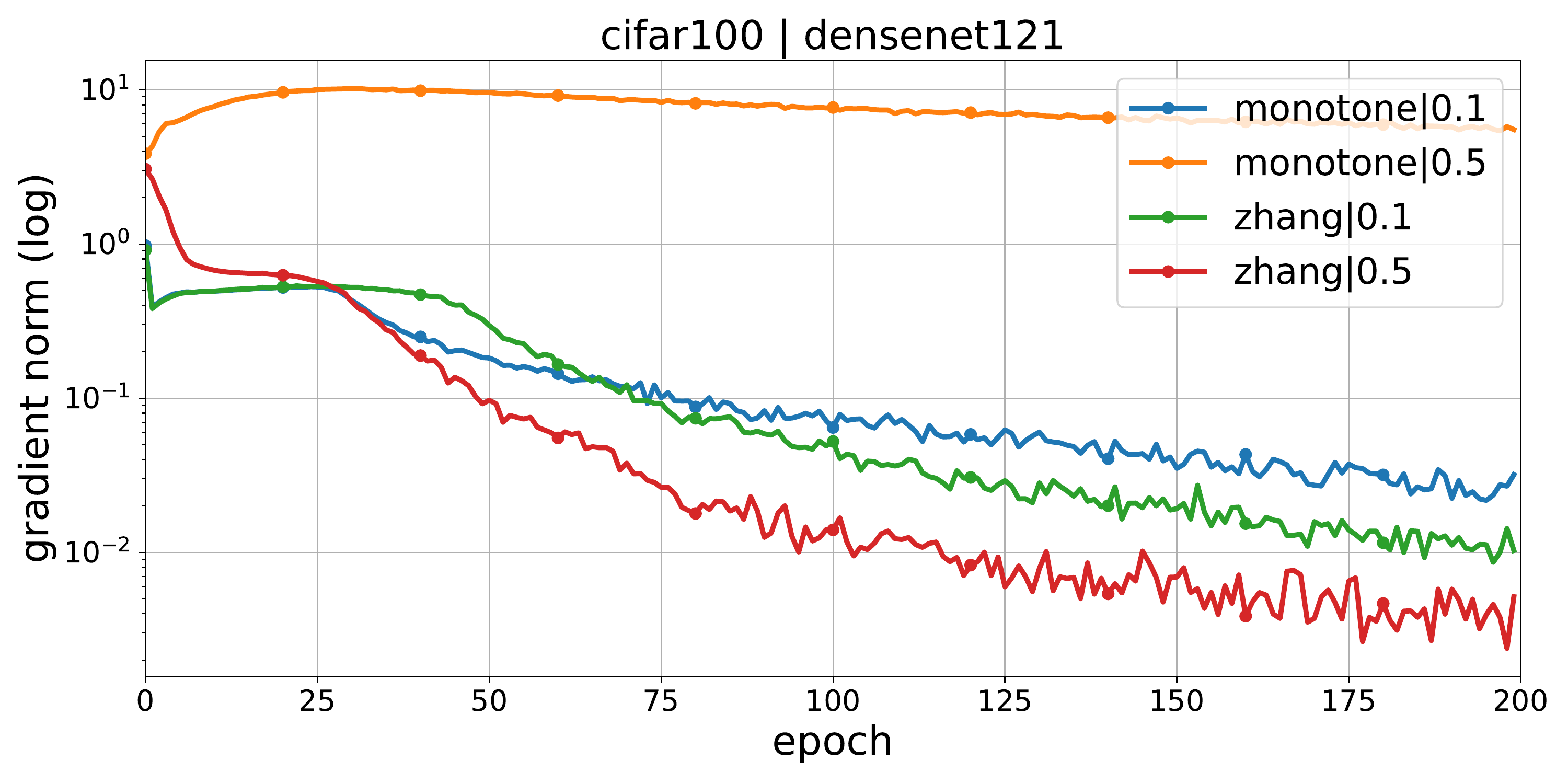}}
		\subfloat{\includegraphics[width=\imgS\linewidth]{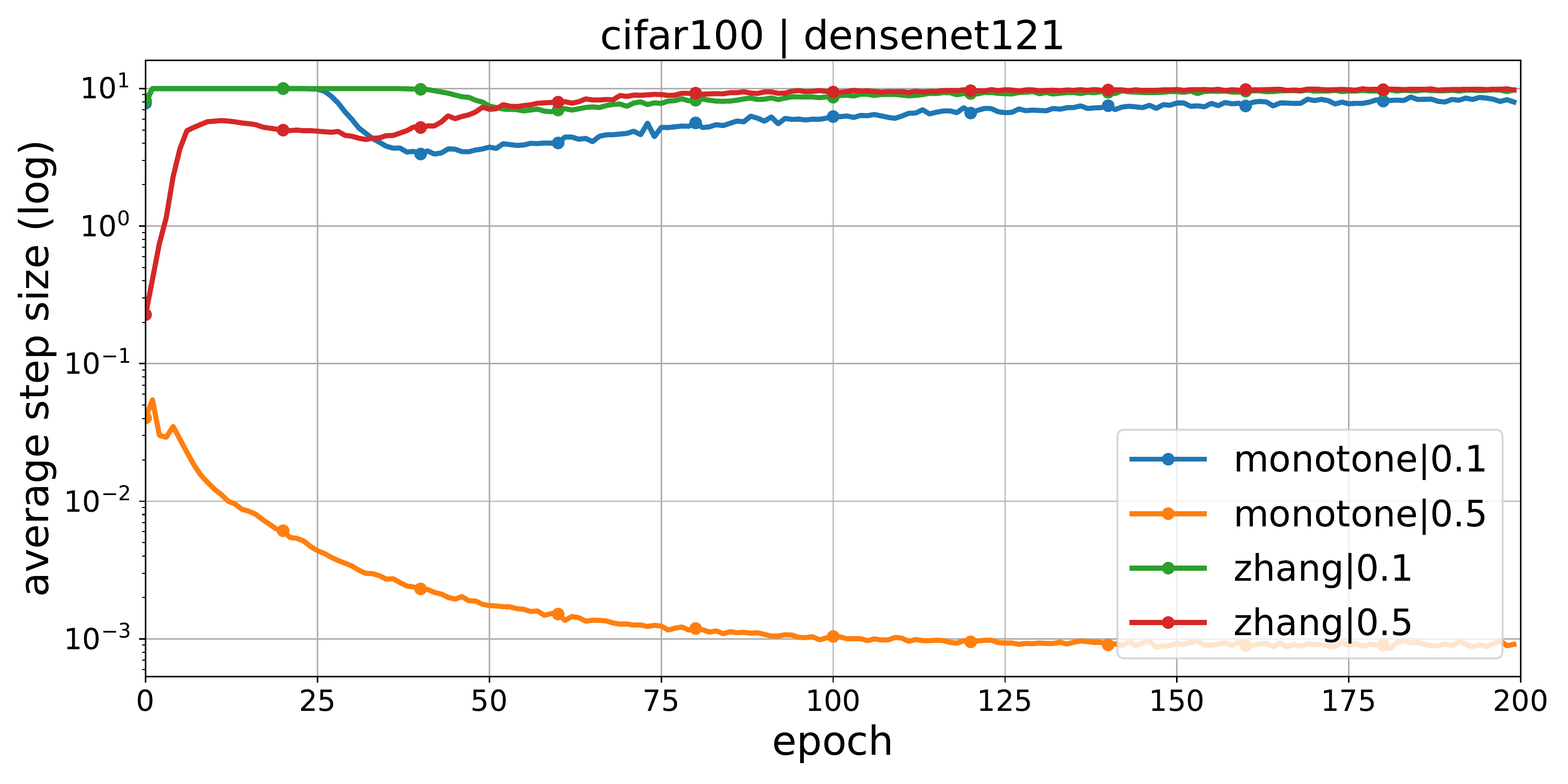}}\\
		\renewcommand{\model}{fashion_effb1}
		\renewcommand{\modelname}{fashion\_effb1}
		\subfloat{\includegraphics[width=\imgS\linewidth]{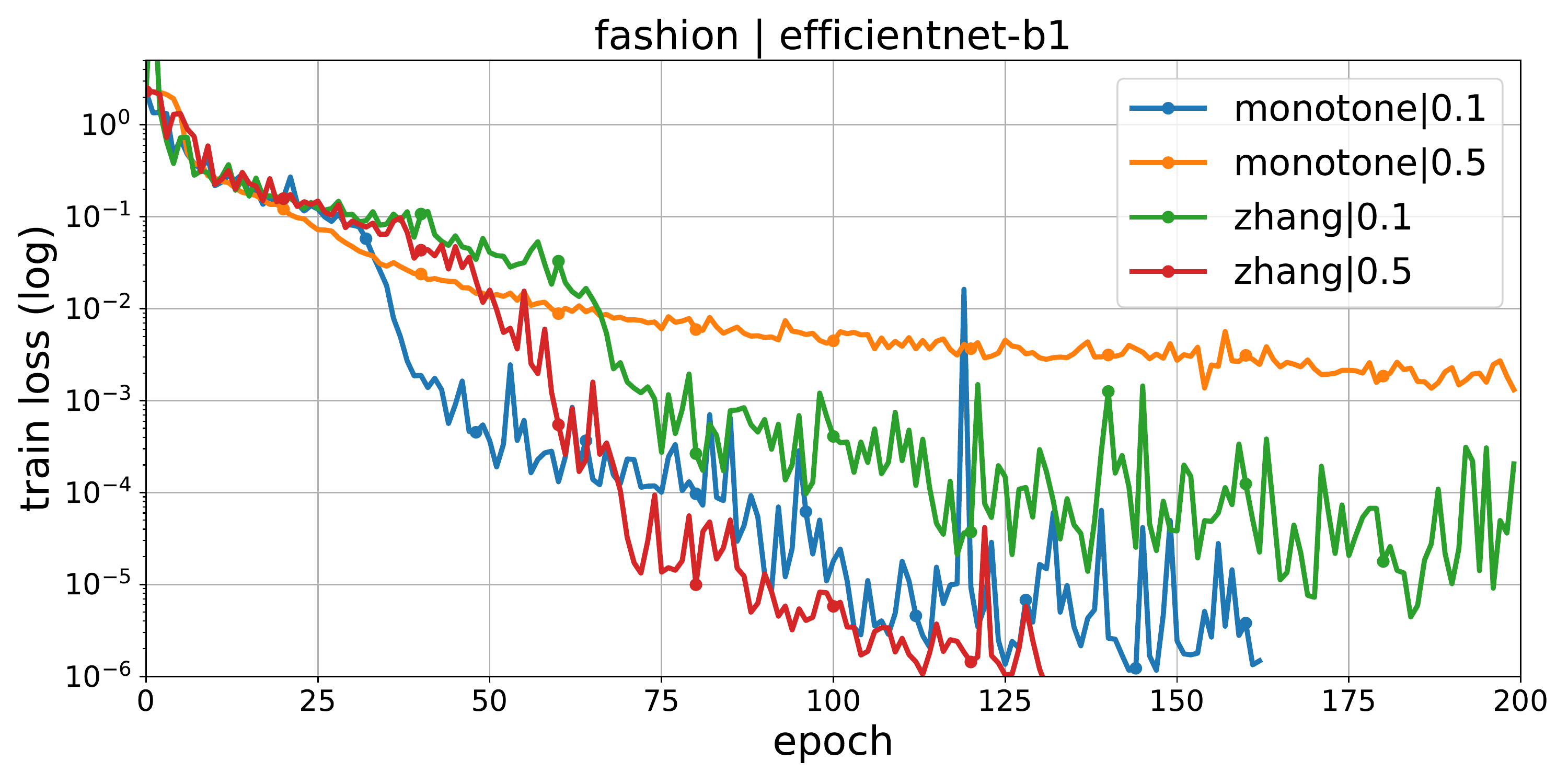}}
		\subfloat{\includegraphics[width=\imgS\linewidth]{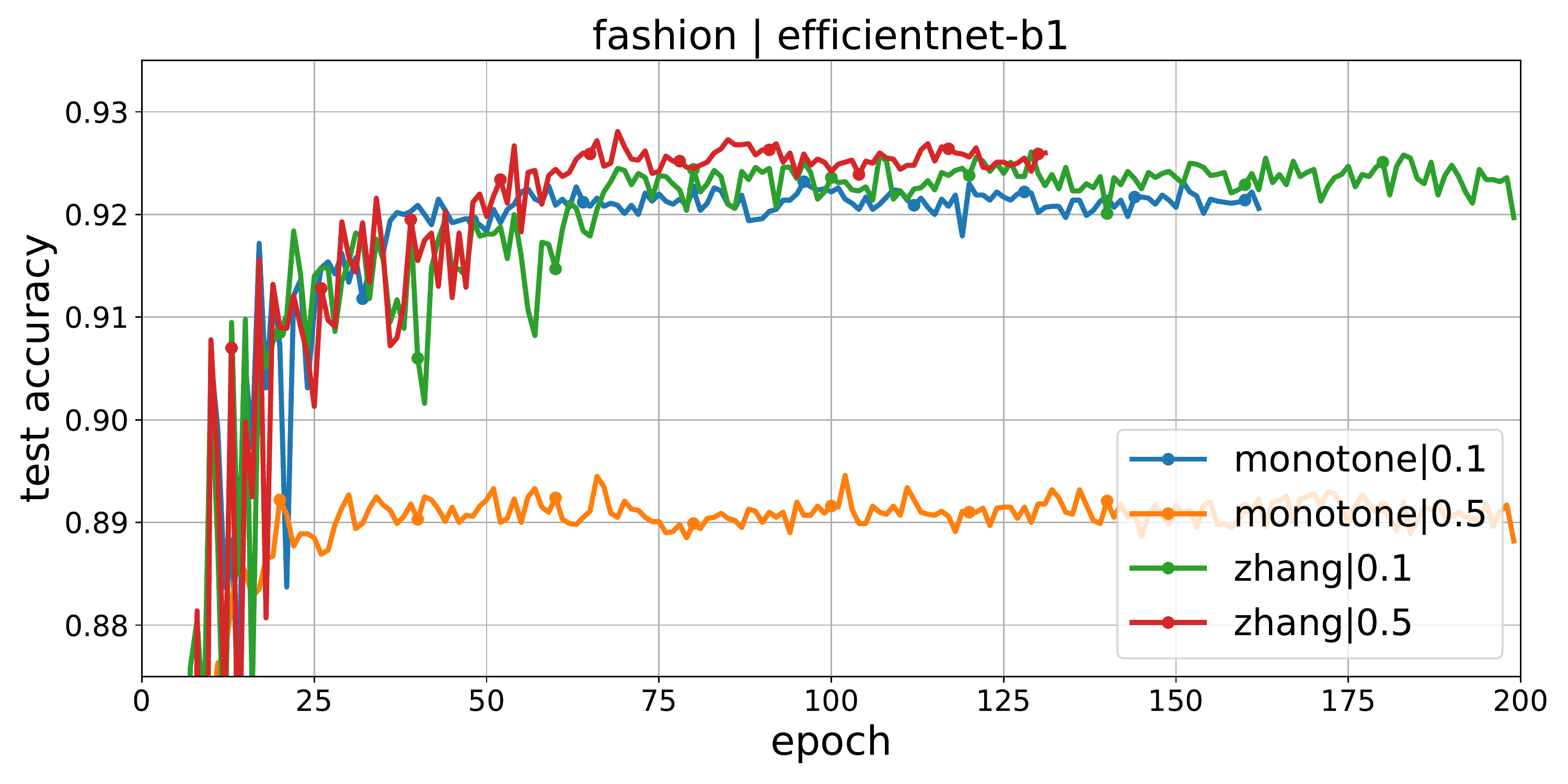}}
		\subfloat{\includegraphics[width=\imgS\linewidth]{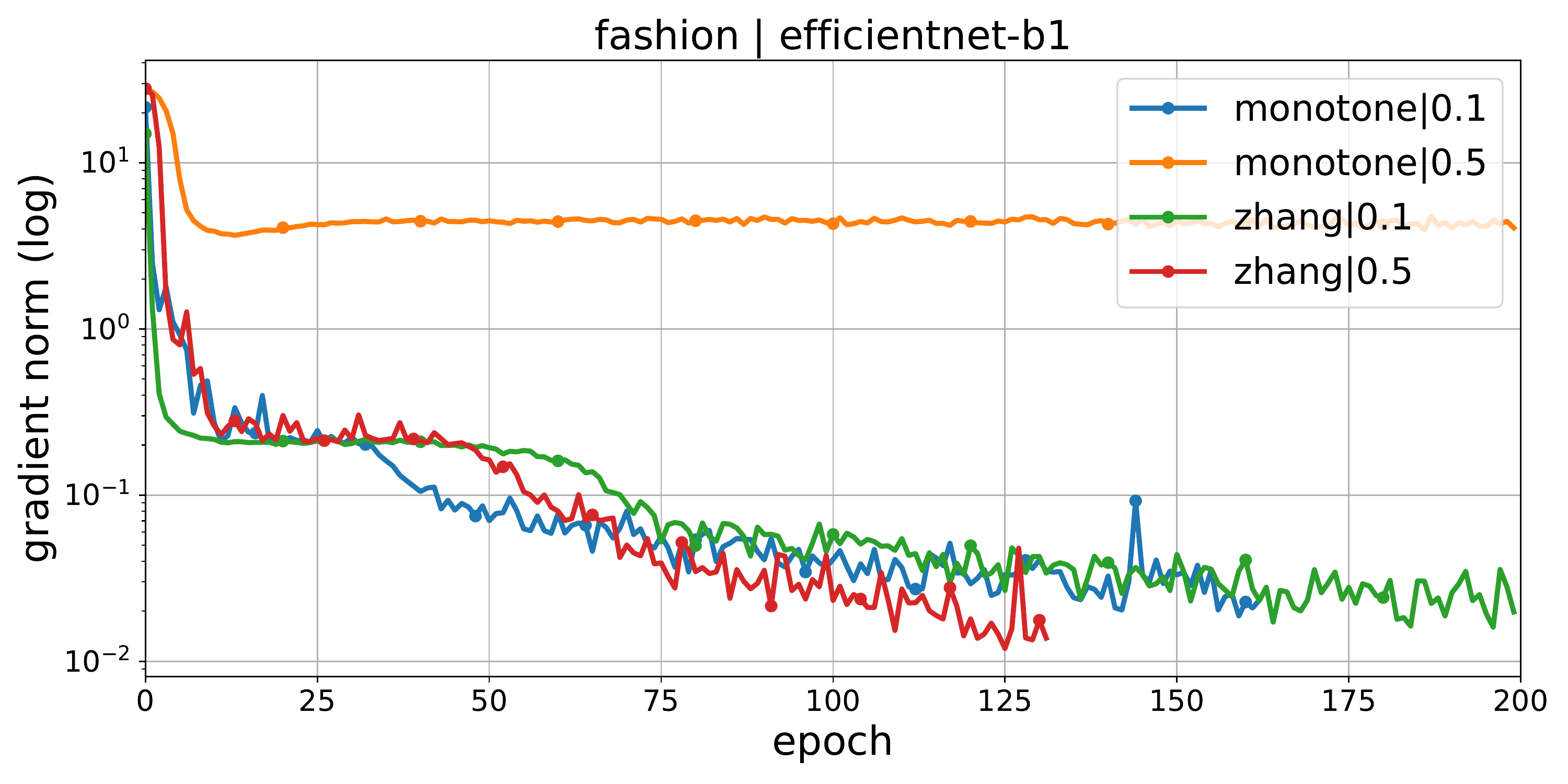}}
		\subfloat{\includegraphics[width=\imgS\linewidth]{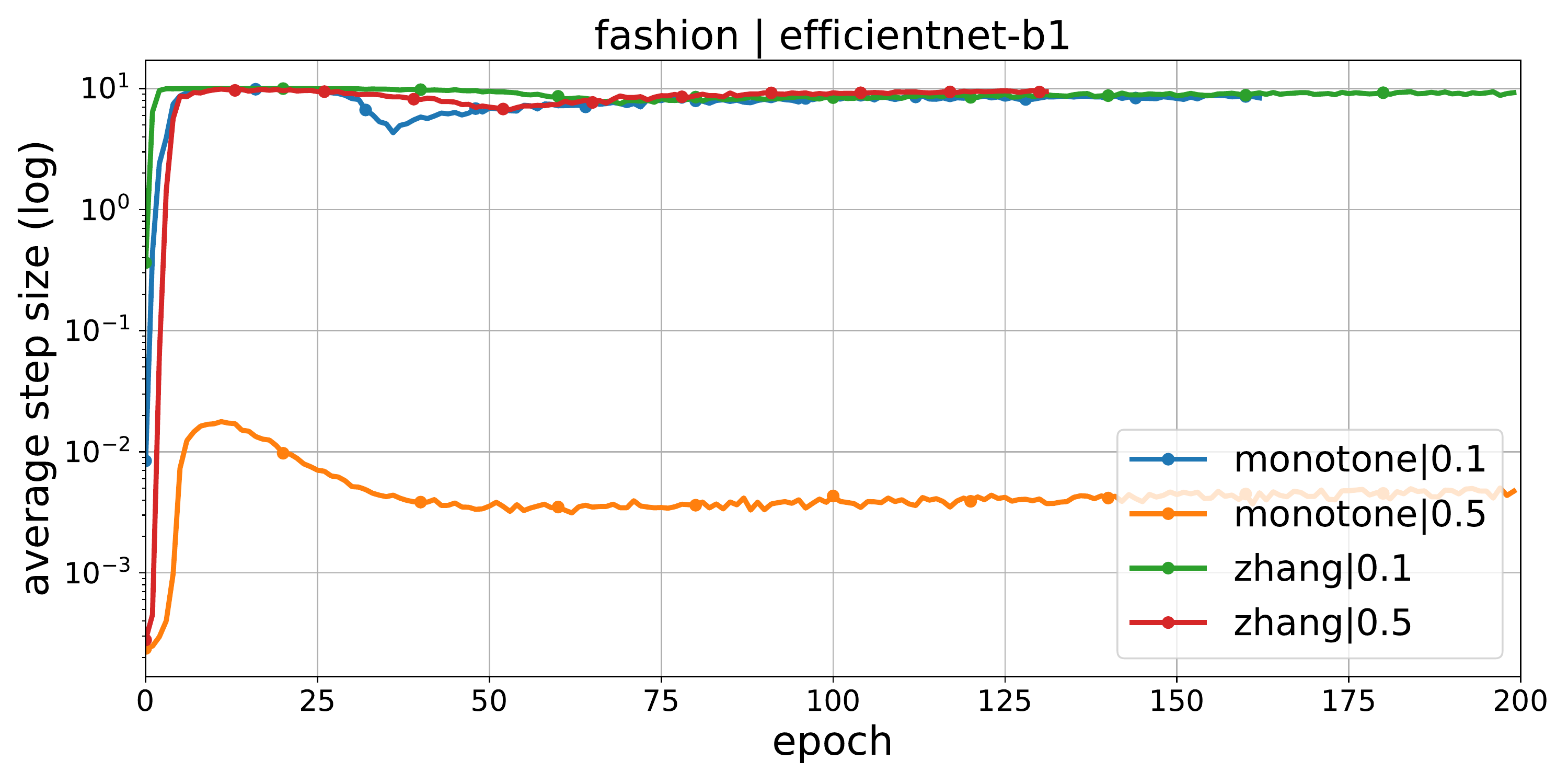}}\\
		\renewcommand{\model}{svhn_wrn}
		\renewcommand{\modelname}{svhn\_wrn}
		\subfloat{\includegraphics[width=\imgS\linewidth]{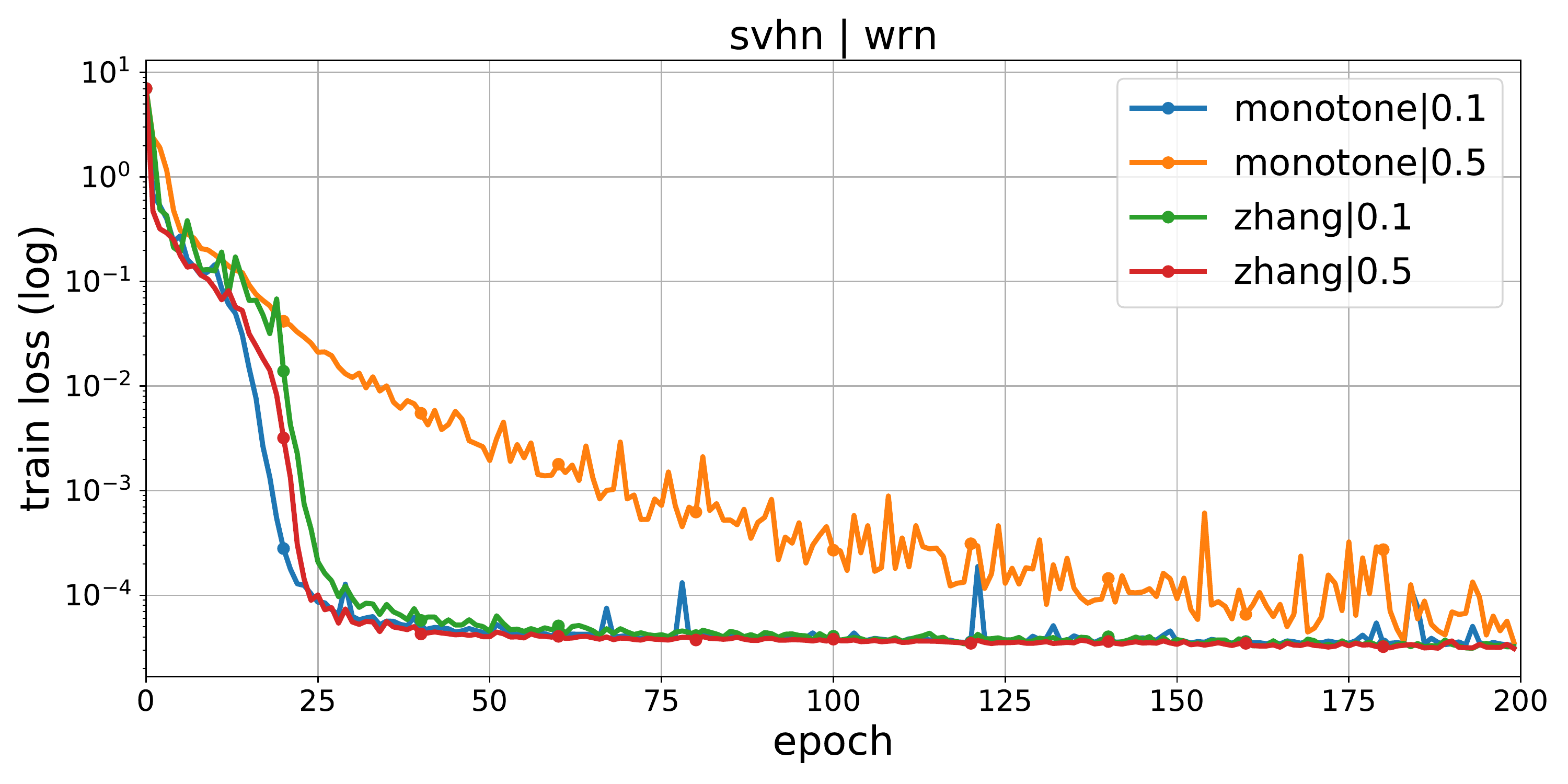}}
		\subfloat{\includegraphics[width=\imgS\linewidth]{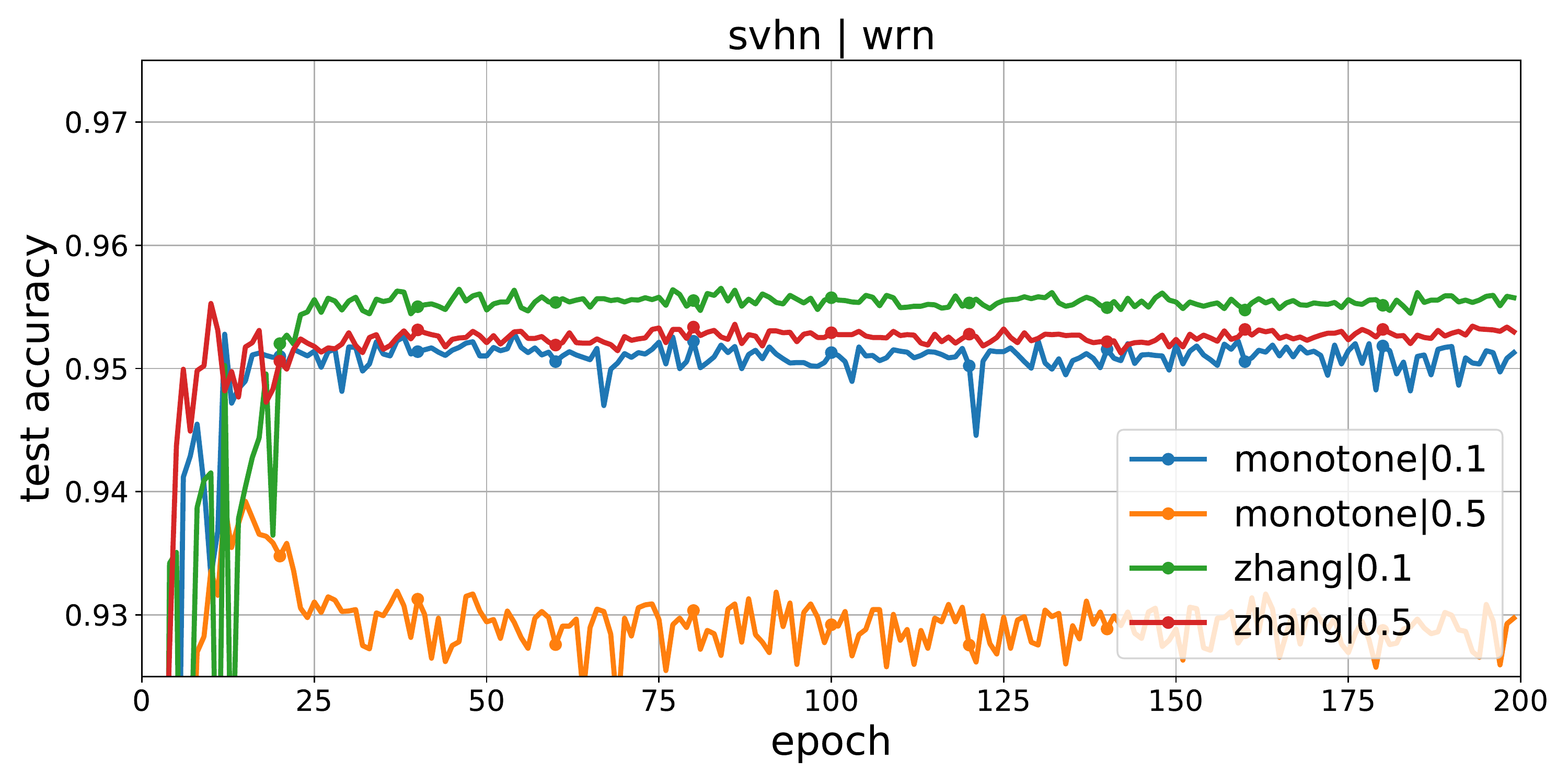}}
		\subfloat{\includegraphics[width=\imgS\linewidth]{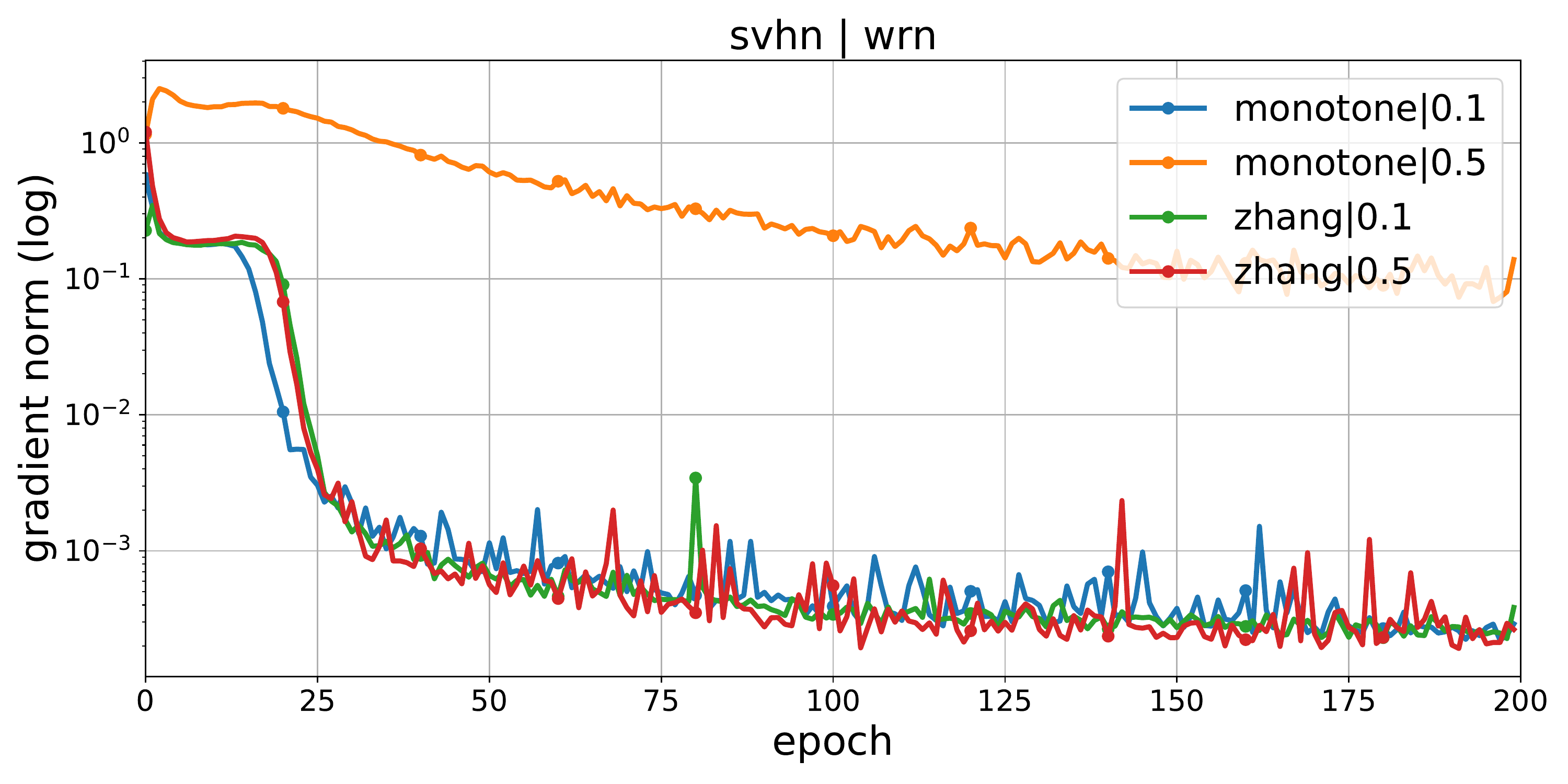}}
		\subfloat{\includegraphics[width=\imgS\linewidth]{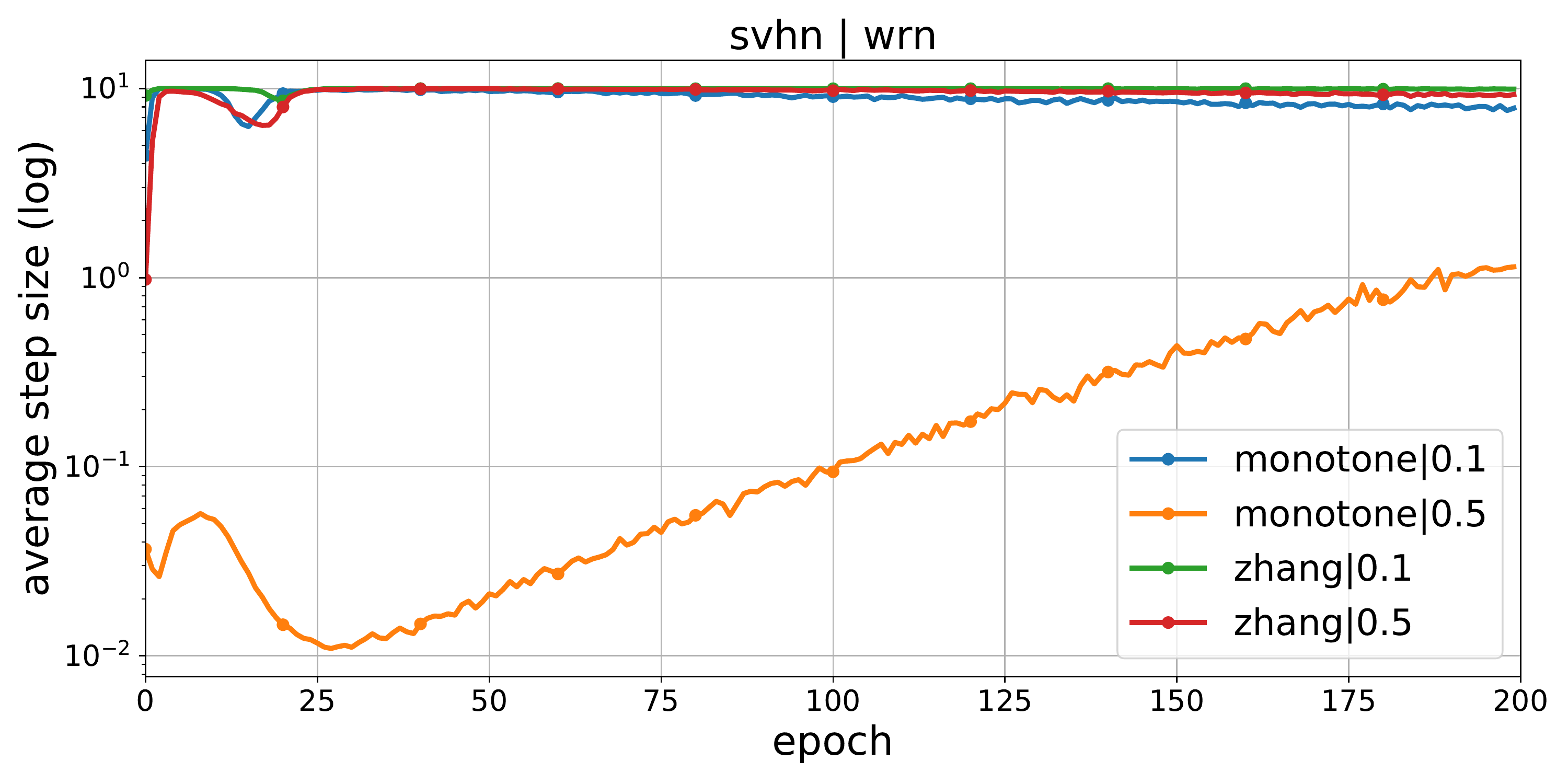}}
		\caption{Comparison between the use of the constant $c=0.1$ and $c=0.5$ on both monotone and nonmonotone algorithms. Each row focus on a dataset/model combination. First column: train loss. Second column: test accuracy. Third column: average step size of the epoch. Fourth column: average gradient norm of the epoch.}\label{fig:theoretically_fixed_c}
	\end{minipage}
\end{figure}
	\par In Figure \ref{fig:convex} we propose the same comparison as in Figure \ref{fig:theoretically_fixed_c}, but in the case of the convex experiments of Figure \ref{fig:convex_short}. From Figure \ref{fig:convex} we can observe that zhang|0.1 and monotone|0.1 do not obtain good performances. On both \rcvone$ $ and \ijcnn, these algorithms do not converge, suggesting that a large constant $c$ is sometimes required to achieve convergence. 
\renewcommand{\dir}{convex}
\renewcommand{\imgS}{0.33}
\begin{figure}[!h] 
	\begin{minipage}{0.9\textwidth}
		\centering
		\renewcommand{\model}{mushrooms}
		\renewcommand{\modelname}{mushrooms}
		\subfloat{\includegraphics[width=\imgS\linewidth]{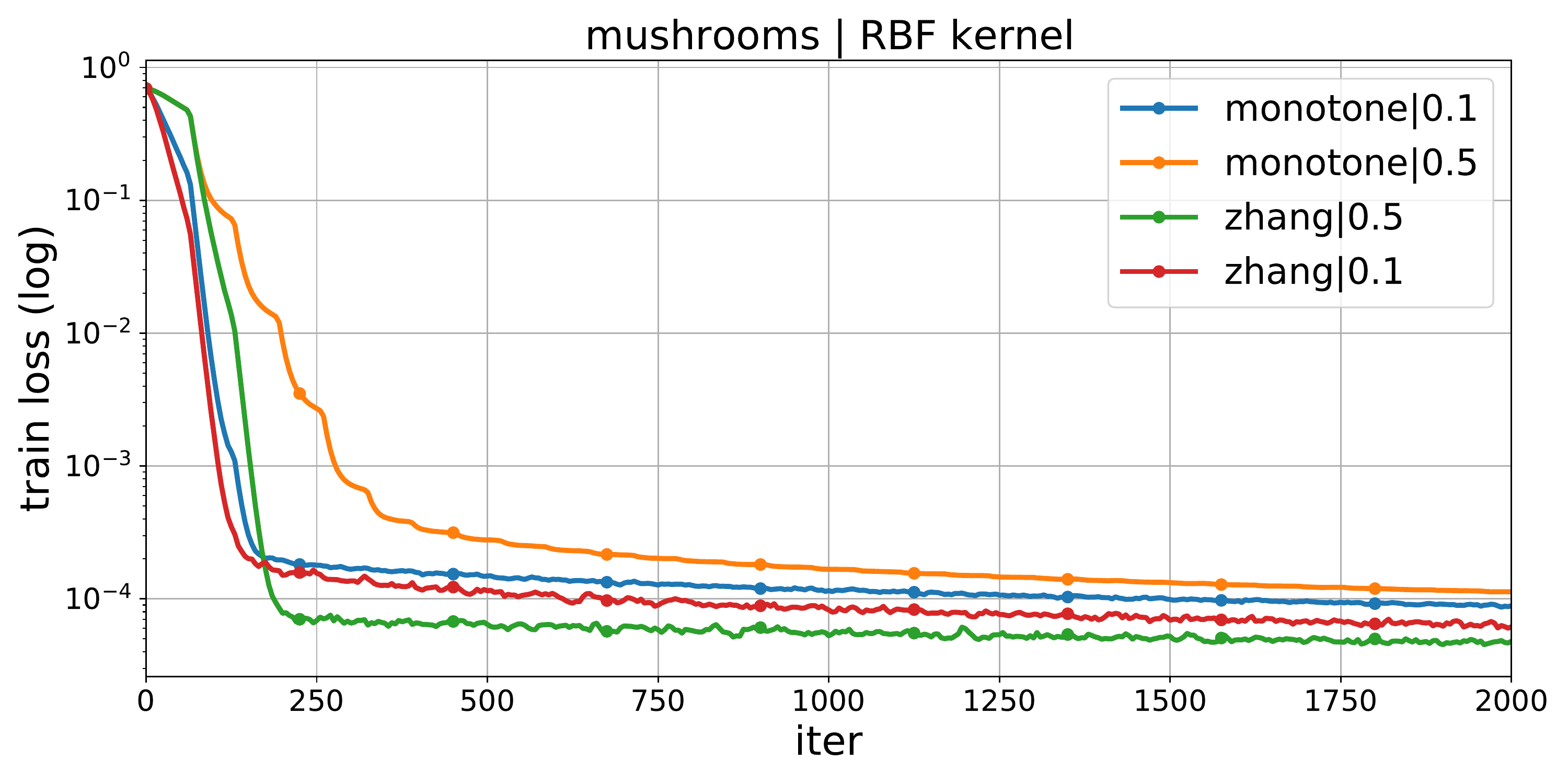}}
		\subfloat{\includegraphics[width=\imgS\linewidth]{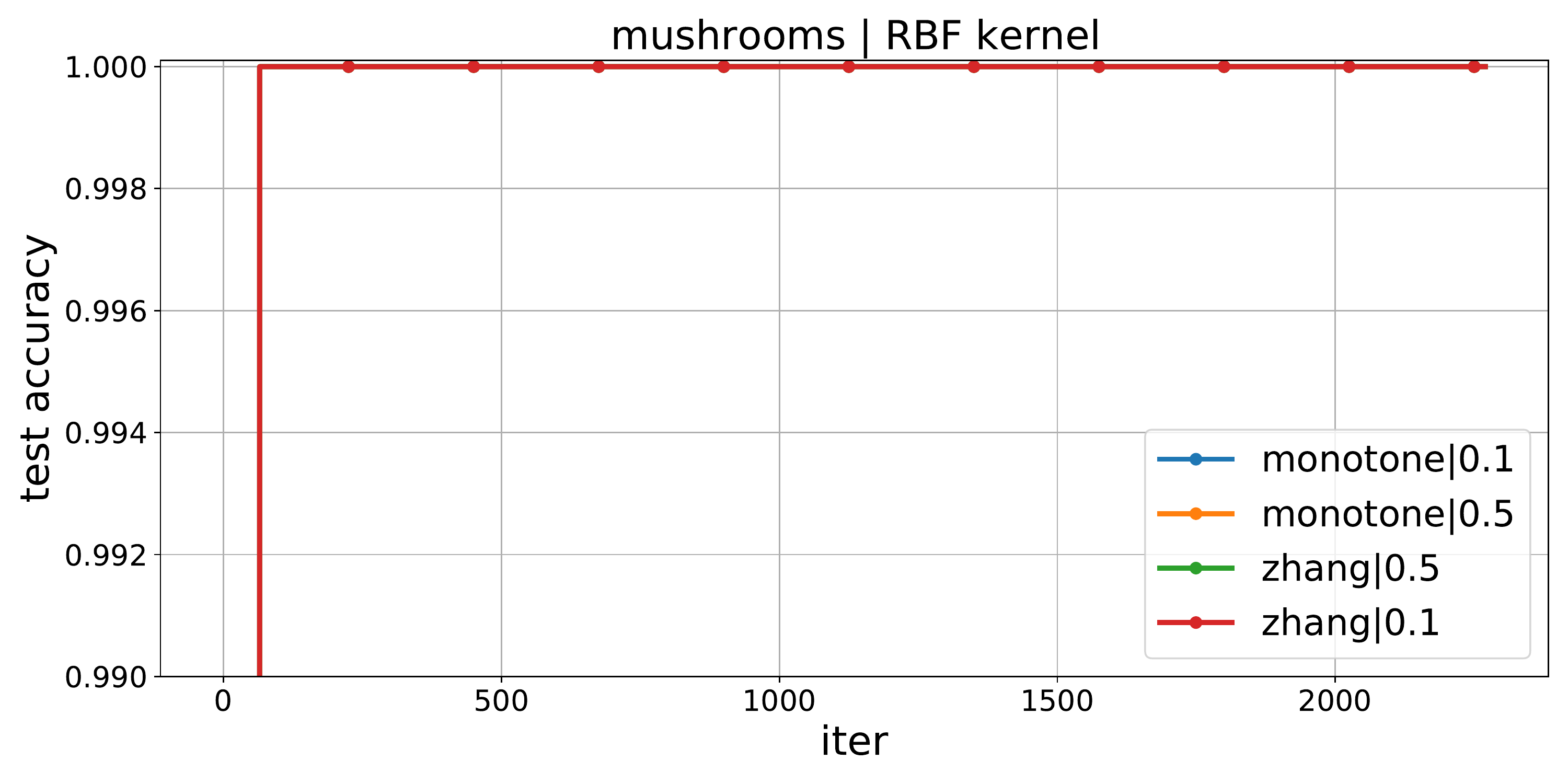}}
		\subfloat{\includegraphics[width=\imgS\linewidth]{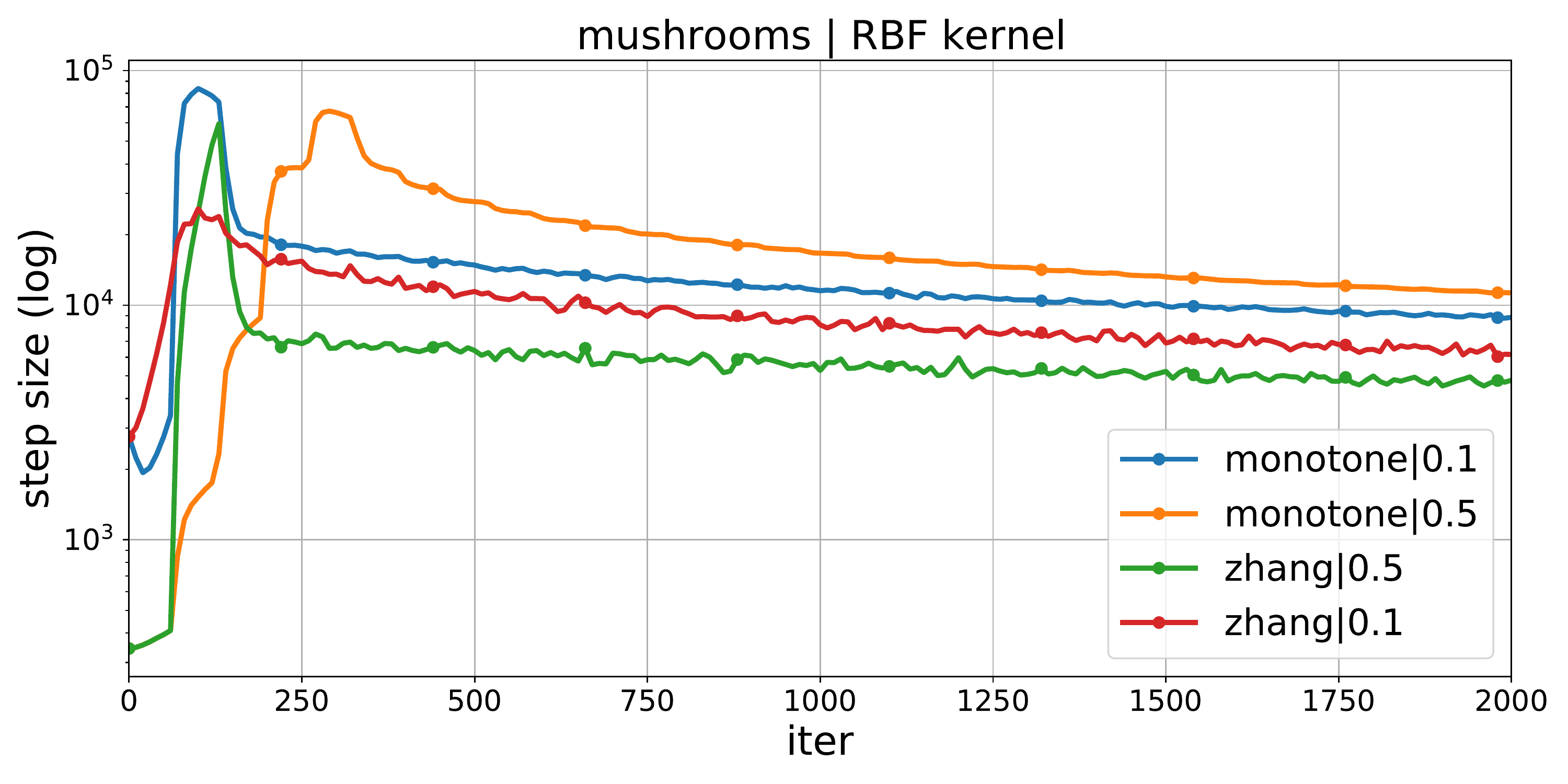}}\\
		\renewcommand{\model}{rcv1}
		\renewcommand{\modelname}{rcv1}
		\subfloat{\includegraphics[width=\imgS\linewidth]{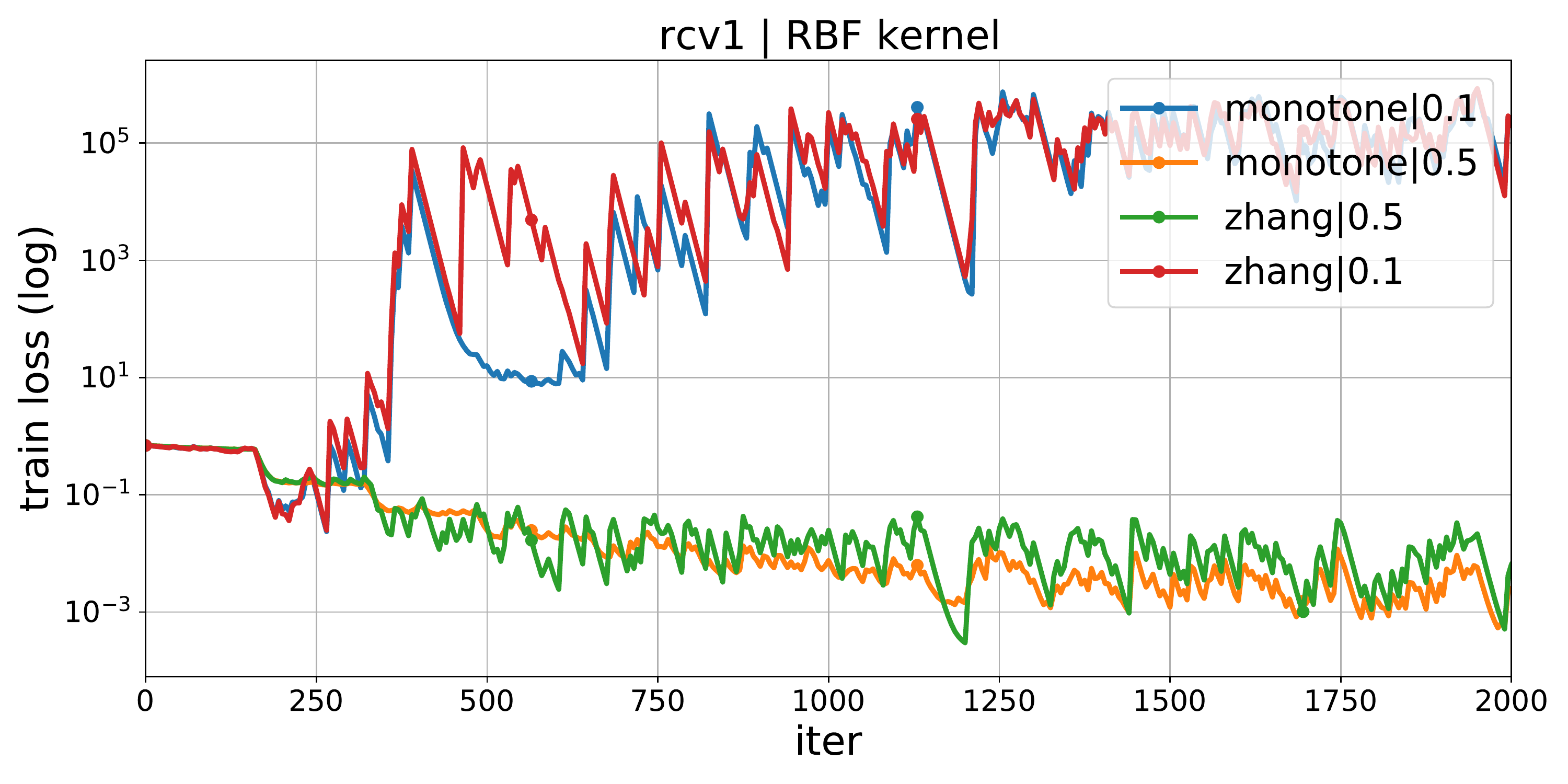}}
		\subfloat{\includegraphics[width=\imgS\linewidth]{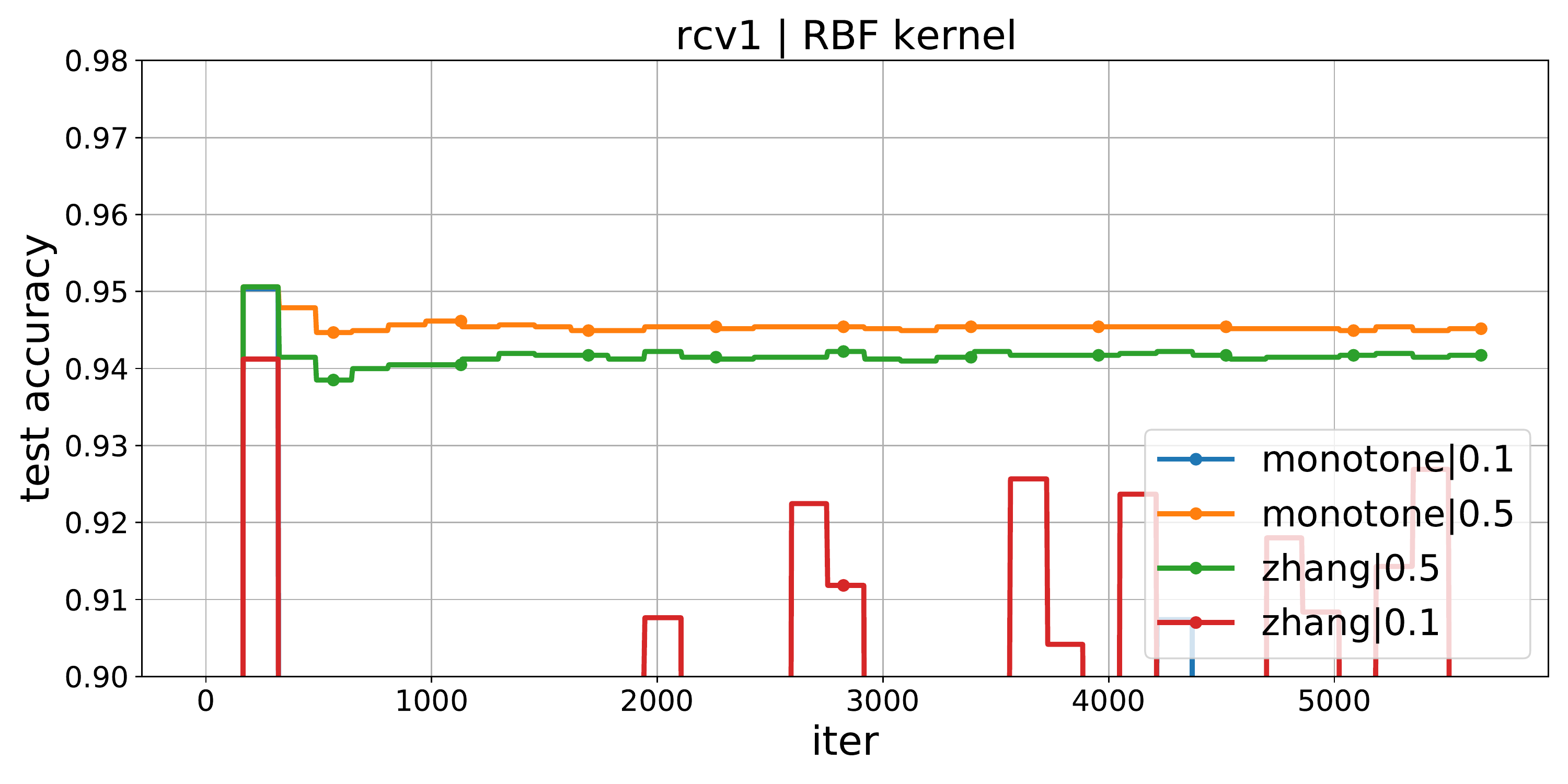}}
		\subfloat{\includegraphics[width=\imgS\linewidth]{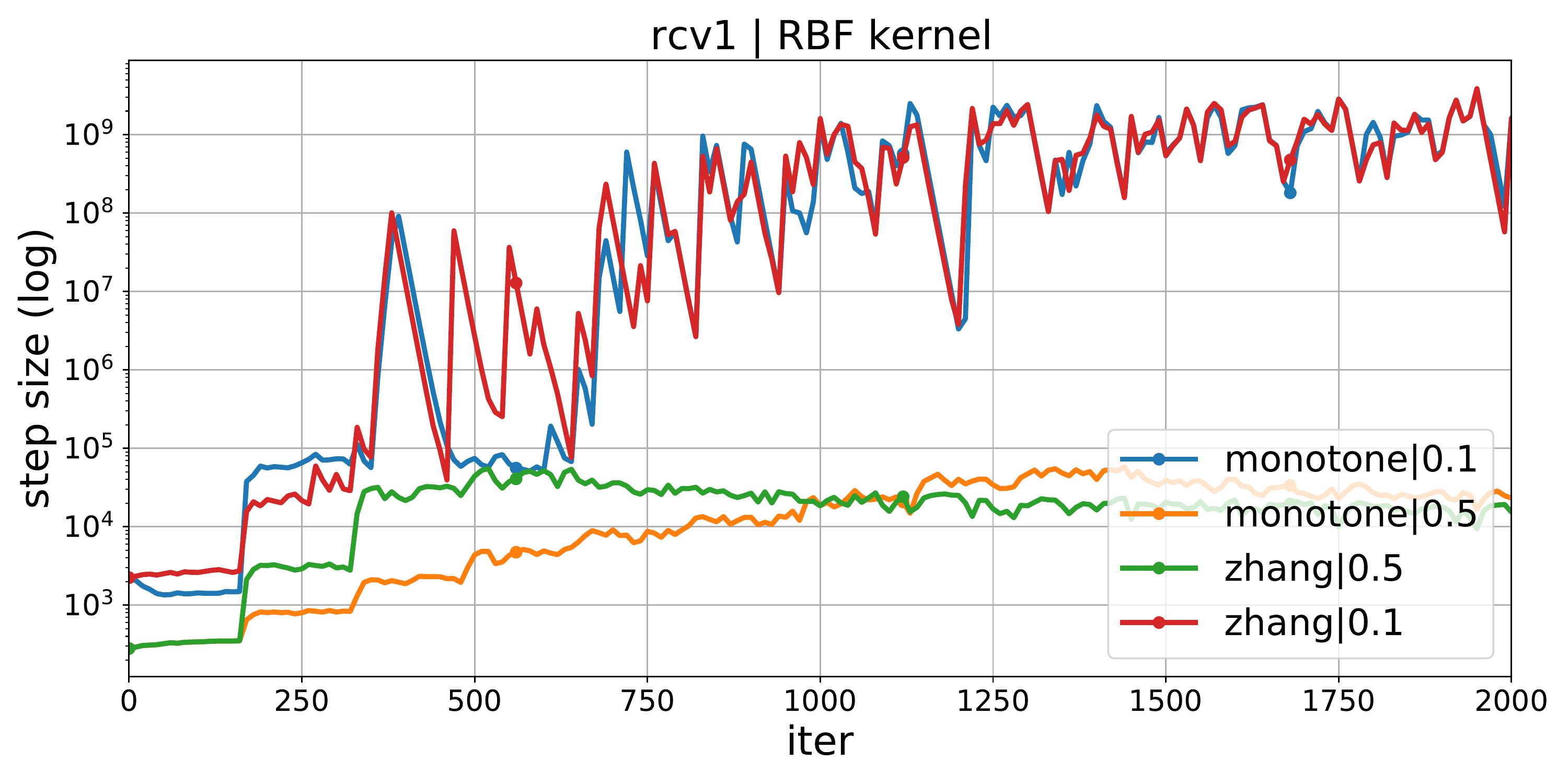}}\\
		\renewcommand{\model}{ijcnn}
		\renewcommand{\modelname}{ijcnn}
		\subfloat{\includegraphics[width=\imgS\linewidth]{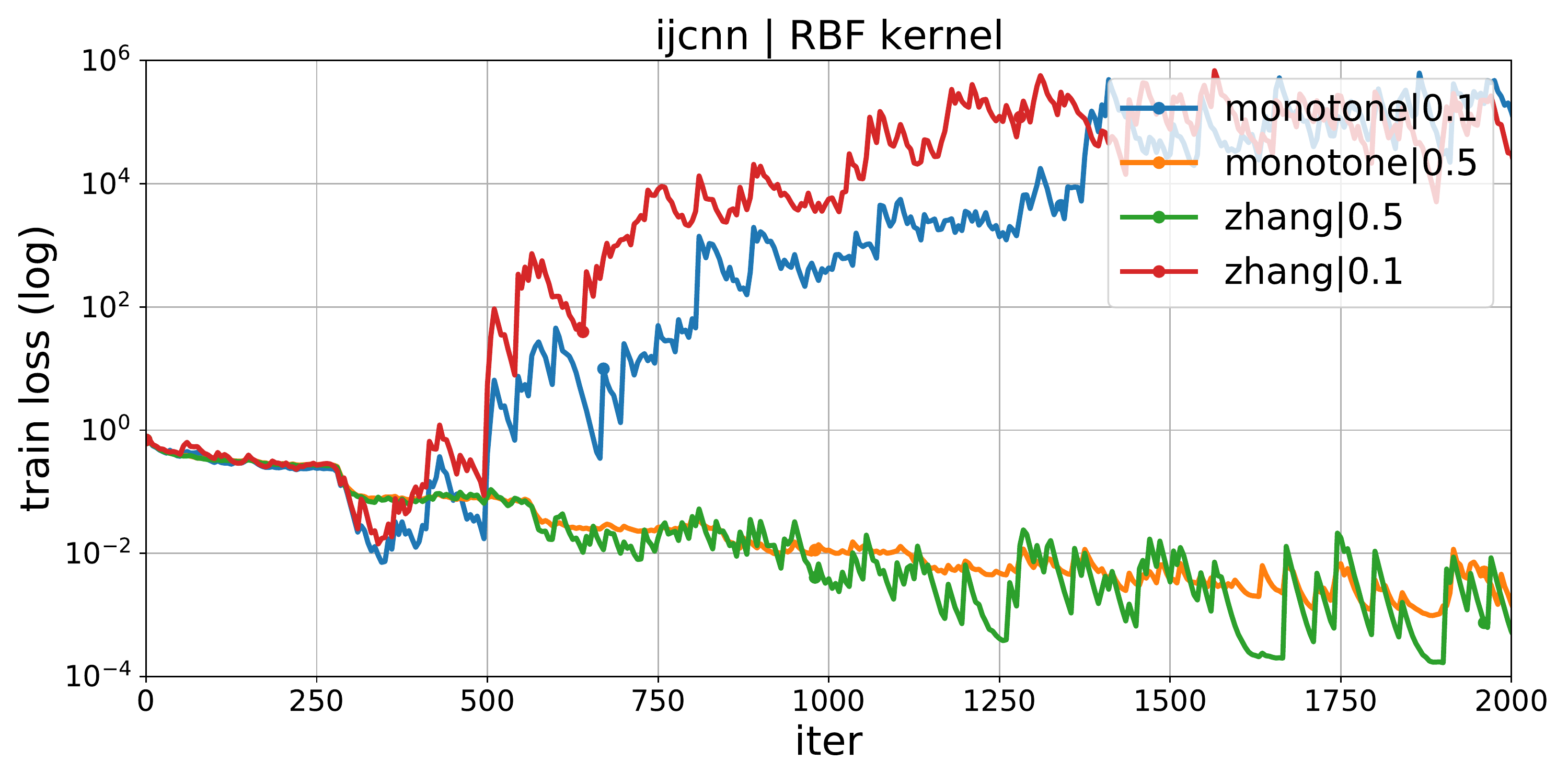}}
		\subfloat{\includegraphics[width=\imgS\linewidth]{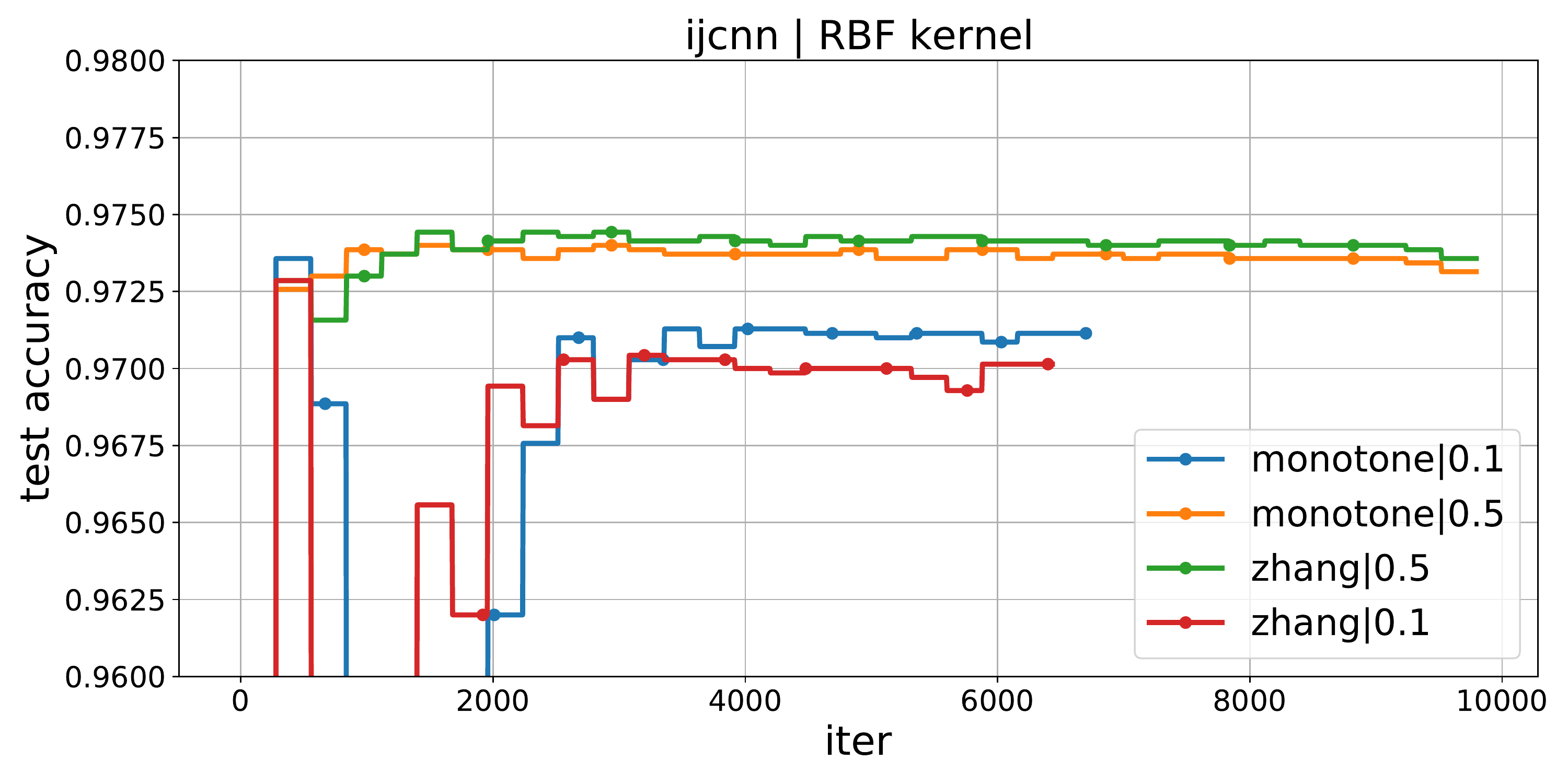}}
		\subfloat{\includegraphics[width=\imgS\linewidth]{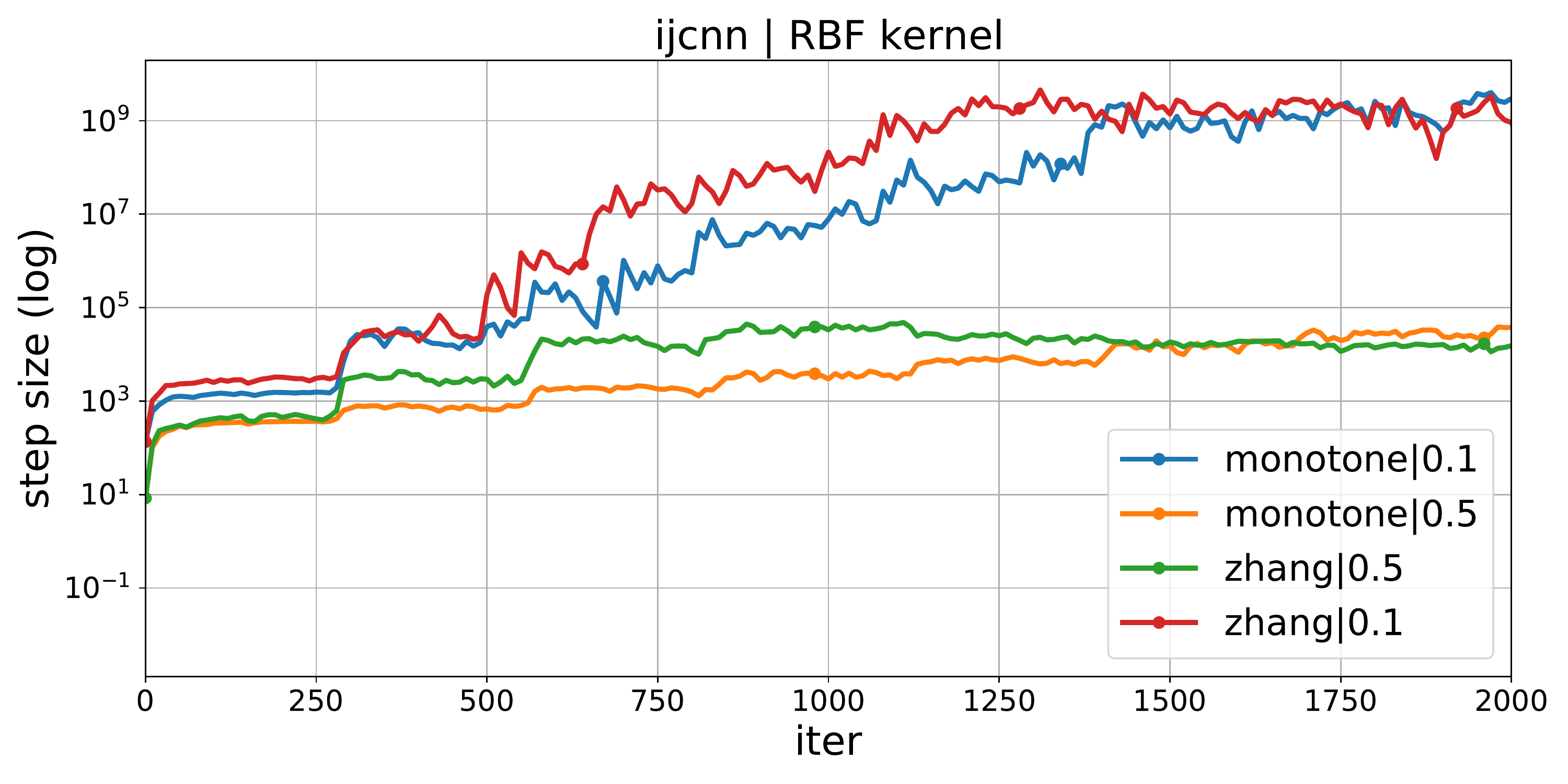}}\\
		\renewcommand{\model}{w8a}
		\renewcommand{\modelname}{w8a}
		\subfloat{\includegraphics[width=\imgS\linewidth]{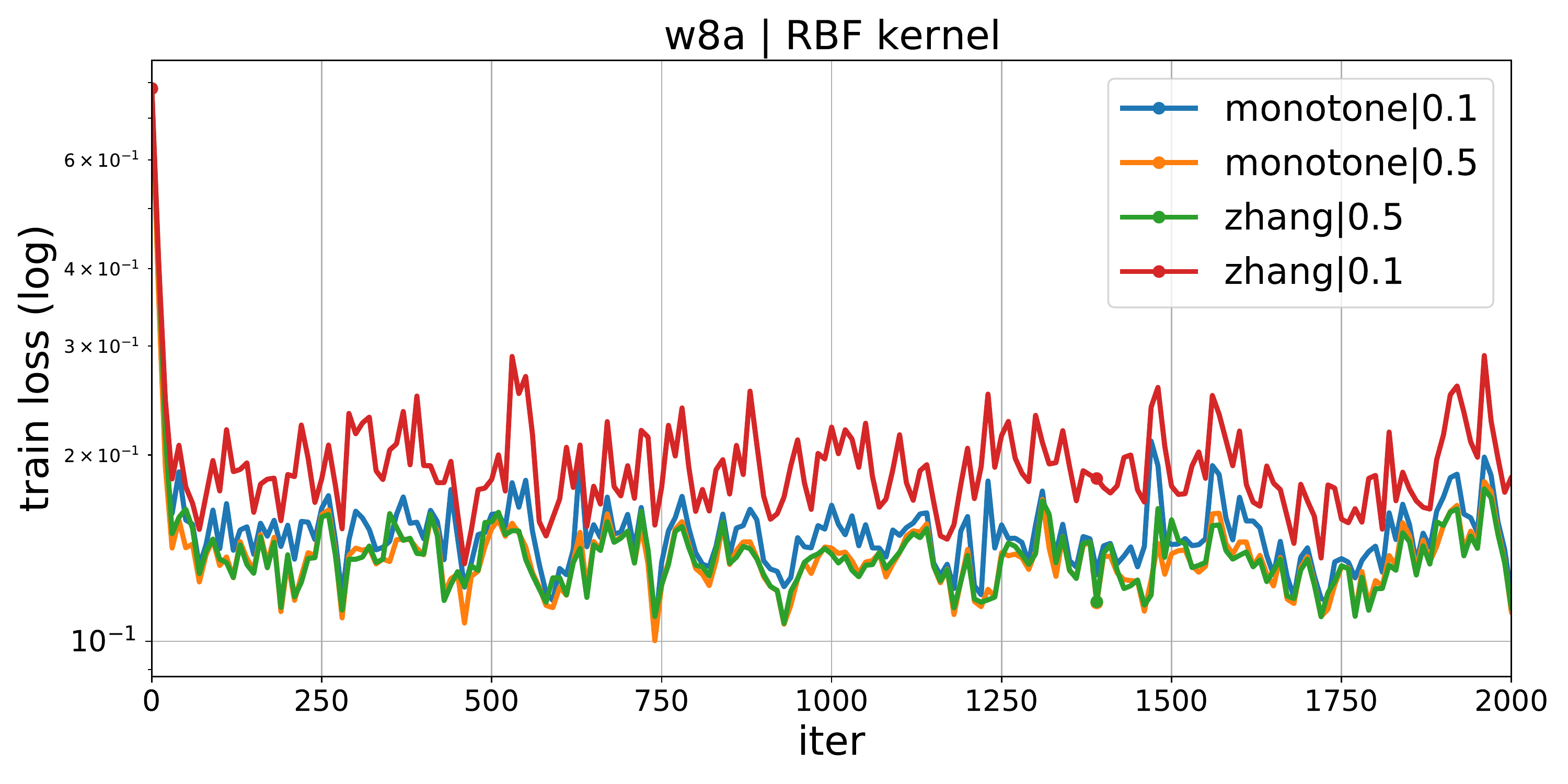}}
		\subfloat{\includegraphics[width=\imgS\linewidth]{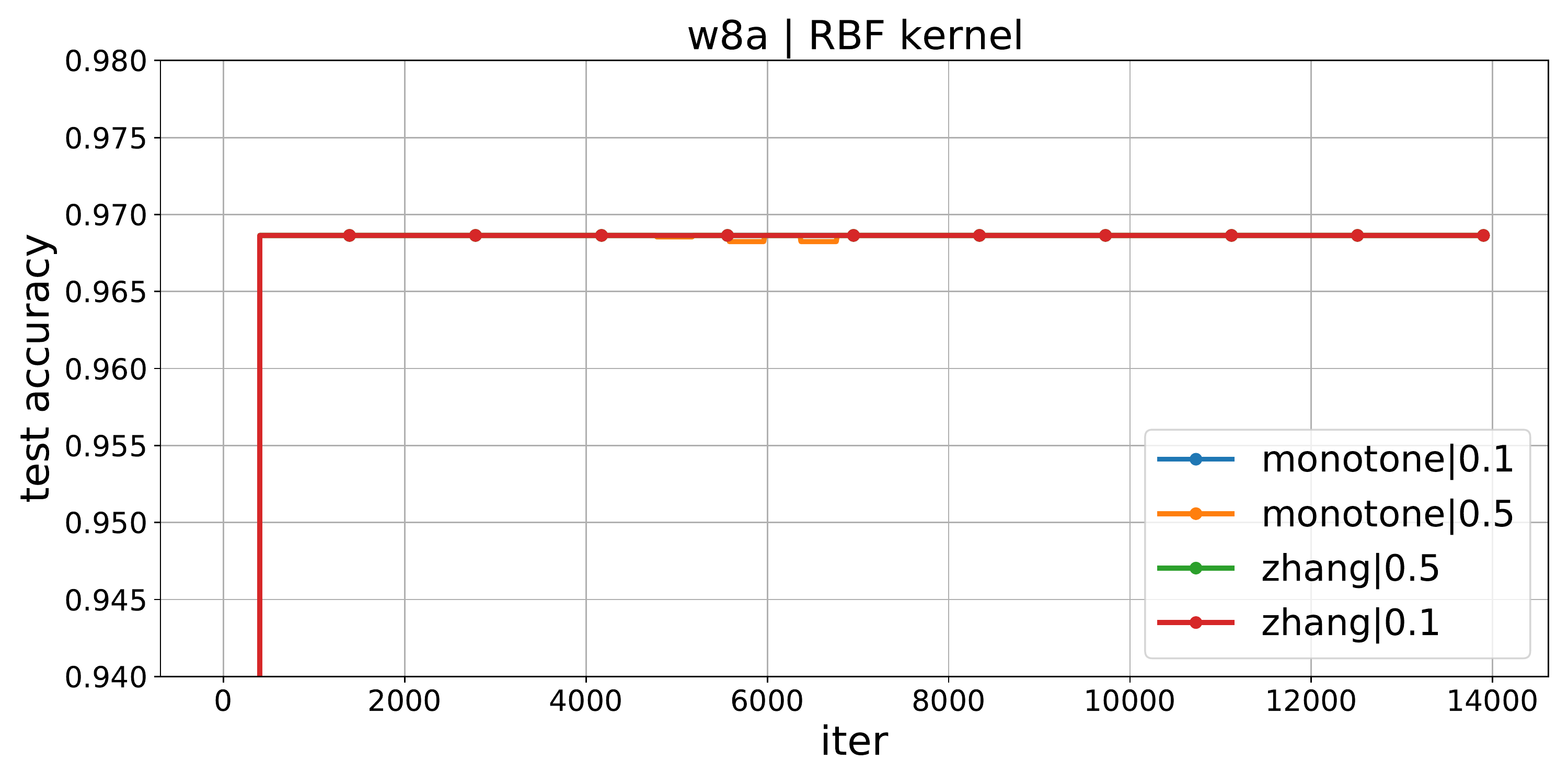}}
		\subfloat{\includegraphics[width=\imgS\linewidth]{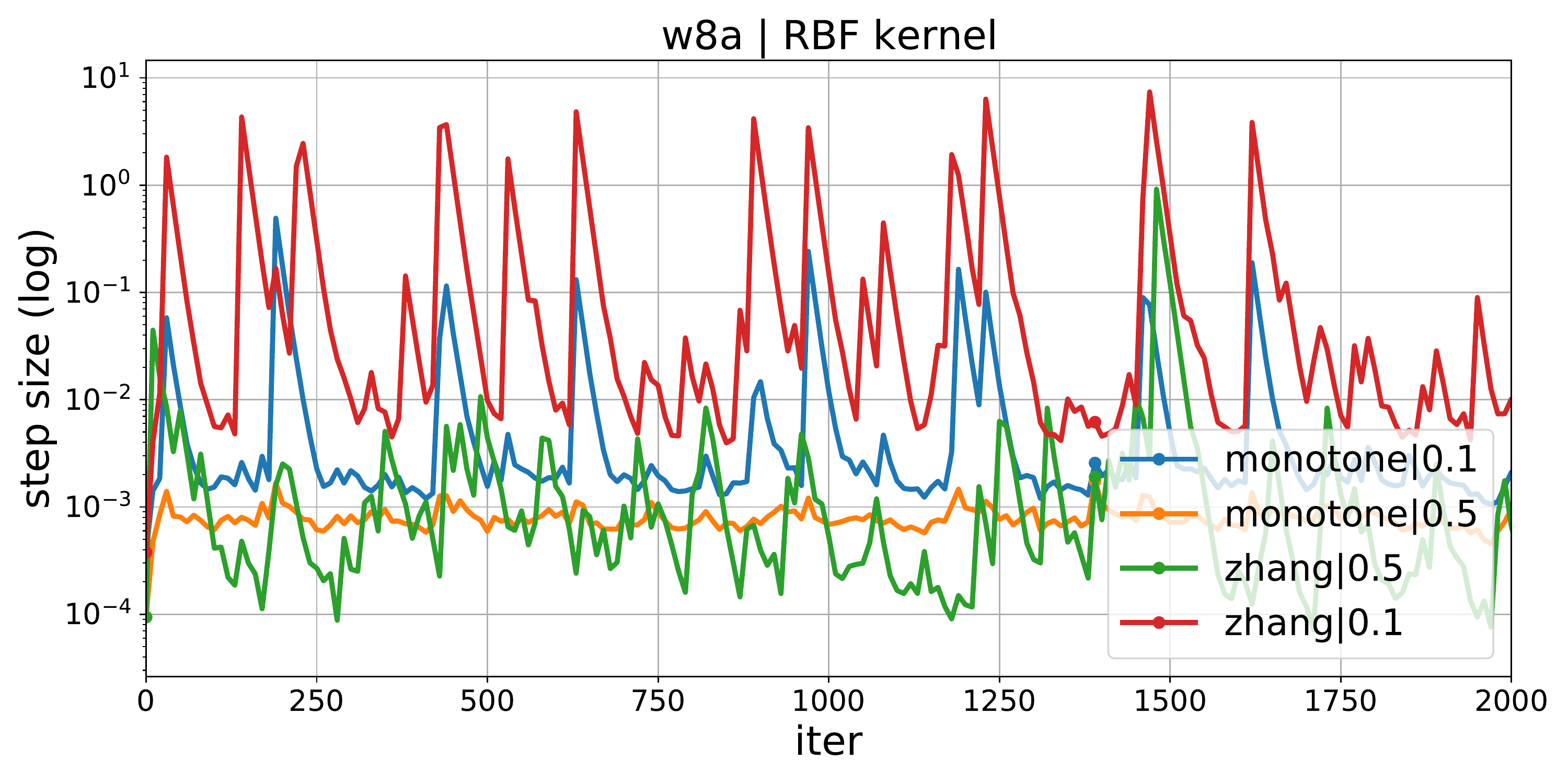}}
		\caption{Comparison between the use of the constant $c=0.1$ and $c=0.5$ on both monotone and nonmonotone algorithms on convex binary classification problems. Each row focus on a different dataset. First column: exponentially averaged train loss. Second column: test accuracy. Third column: exponentially averaged step size.}\label{fig:convex}
	\end{minipage}
\end{figure}
	\subsection{Study on the Line Search Choice: Various Nonmonotone Adaptations}\label{sec:supp_line_search}
	In this subsection, we propose a comparison between various line search conditions. For the first time in this paper, we adapt different nonmonotone techniques to SGD. A straightforward adaptation of the nonmonotone line search from \citet{grippo86a} is
\begin{equation}\label{eq:grippo_naive}
\fik(w_k - \eta_k \gradik(w_k)) \leq \max_{0\leq j \leq W-1} \fik(w_{k-j}) -c \eta_k \| \gradik(w_k) \|^2.
\end{equation}
The nonmonotone term in \eqref{eq:grippo_naive} can be computed in two ways. Either by keeping in memory $W$ (usually $10$ or $20$) previous vector weights $\{w_{k-W}, \dots, w_k\}$ and computing the current mini-batch function $\fik$ on all of them. Or by computing all the following $W$ mini-batch functions $\{ \fik(\cdot), \dots, f_{i_{k+W}} (\cdot)\}$ on $w_k$. Both options are very expensive in the case of large neural networks and they should be avoided. The condition \eqref{eq:grippo_naive} is the same proposed in \citet{hafshejani23a}.
\par To reduce the cost of computing \eqref{eq:grippo_naive}, we here propose two other stochastic adaptations of the nonmonotone line search from \citet{grippo86a}. We name  \textit{cross-batch Grippo} the following line search condition
\begin{equation}\label{eq:nonmonotone}
	f_{i_k}(w_k - \eta_k \nabla f_{i_k}(w_k)) \leq f_{i_{k-\tilde{r}}} (w_{\tilde{r}}) -c \eta_k \| \nabla f_{i_k}(w_k) \|^2,
\end{equation}
where $\tilde{r}$ is a short notation for $\tilde{r}(k, i_k)$, which is the (say, largest) iteration index such that\\ $f_{i_{k-\tilde{r}(k, i_k)}}(w_{\tilde{r}(k, i_k)}) = \displaystyle\max_{0\leq j \leq W-1} f_{i_{k-j}}(w_{k-j})$. The computation of \eqref{eq:nonmonotone} does not introduce overhead since it directly uses the function values computed in the previous iterations. However, the fact that \eqref{eq:nonmonotone} computes the maximum over different mini-batch functions $\{f_{i_{k-W}} (\cdot), \dots, \fik(\cdot)\}$ complicates the convergence analysis. We conjecture that under the interpolation assumption alone, it is not possible to achieve convergence if \eqref{eq:nonmonotone} replaces \eqref{eq:zhang}, not even in the strongly convex case.\\
\noindent We call \textit{single-batch Grippo} the following line search condition
\begin{equation}\label{eq:true_nonmonotone}
	\fik(w_k - \eta_k \gradik(w_k)) \leq \fik(w_{r(k, i_k)}) -c \eta_k \| \gradik(w_k) \|^2,
\end{equation}
where $r(k, i_k)$ is the (say, largest) iteration index such that $\fik (w_{r(k, i_k)}) = \displaystyle\max_{0\leq j \leq W-1} \fik(w_{k-j\frac{M}{b}})$. This line search is similar to \eqref{eq:grippo_naive} since it focuses on $\fik$. However, it also reduces the cost of computing the nonmonotone term. In fact, given a certain set of indexes $i_k$, it exploits the function values computed in the previous epochs, without re-computing $\fik$. This requires saving a matrix of $M\!\times\!W$ floating-point numbers. On the other hand, to use this computational trick, the nonmonotone term can only be computed starting from the second epoch. Moreover, \eqref{eq:true_nonmonotone} is not computationally as cheap as \eqref{eq:nonmonotone} or \eqref{eq:zhang}. In particular, at every mini-batch iteration \eqref{eq:nonmonotone} requires the extraction of the $b\times W$ values corresponding to the indexes in $i_k$ from the above-mentioned matrix and to compute the maximum over these values.\\
\noindent In Figure \ref{fig:line_search} we compare 
\begin{itemize}
	\item monotone: the monotone stochastic Armijo line search \eqref{eq:armijo}.
	\item cross\_batch\_grippo: the cross-batch nonmonotone Grippo's line search  \eqref{eq:nonmonotone} with $W=\frac{M}{b}$.  
	\item single\_batch\_grippo: the single-batch nonmonotone Grippo's line search \eqref{eq:true_nonmonotone} with $W=10$.
	\item zhang: the nonmonotone \citet{zhang04a} line search adapted to the stochastic case \eqref{eq:zhang}. This setting corresponds to PoNoS.
\end{itemize}
For all the above algorithms the initial step size is \eqref{eq:loizou}. From Figure \ref{fig:line_search} we can observe that:
\begin{itemize}
	\item zhang (PoNoS) achieves the best performances both in terms of train loss and test accuracy. Also in this case, it is the only algorithm able to reduce the train loss below the threshold of $10^{-6}$ on the problems \res, \dense $ $ and \fashion. Moreover, it always achieves the highest test accuracy. In \fashion$ $it is not as fast as cross\_batch\_grippo and single\_batch\_grippo.
	\item single\_batch\_grippo achieves similar performances as zhang. However, it does not always reach the same test accuracy (e.g., it loses $\sim$1 point on \ress). On \fashion it is the fastest algorithm in terms of train loss. On the other hand, it almost never reduces the amount of backtracks below 100 per epoch.
	\item The step sizes yielded by both Grippo's adaptations are often growing faster than those of zhang, especially in the initial phase of the optimization procedure. On the other extreme, the monotone line search is too strict and yields step sizes that are very small in comparison to those of zhang. In fact, monotone is achieving very poor results both in terms of train loss and test accuracy.
	\item cross\_batch\_grippo is never as fast as single\_batch\_grippo in terms of train loss and it loses $\sim$5 points of accuracy on \ress$ $ and \denses$ $ w.r.t. zhang. In accordance with the conjecture above, cross\_batch\_grippo does not converge on \mlp.
\end{itemize}	
\par In conclusion, zhang (PoNoS) is the best-performing line search technique among those compared in Figure \ref{fig:line_search}. In terms of step size, zhang provides the right middle way between the very permissive Grippo's conditions and the very strict monotone one. Moreover, even if zhang and single\_batch\_grippo behave similarly, the second is computationally more expensive and it is never able to reduce the backtracks to (almost) always zero. 
\renewcommand{\dir}{line_search/}
\renewcommand{\imgS}{0.30}
\begin{figure}[!h] 
\hspace{-9mm}
\begin{minipage}{0.9\textwidth}
\renewcommand{\model}{mlp}
\renewcommand{\modelname}{mlp}
\subfloat{\includegraphics[width=\imgS\linewidth]{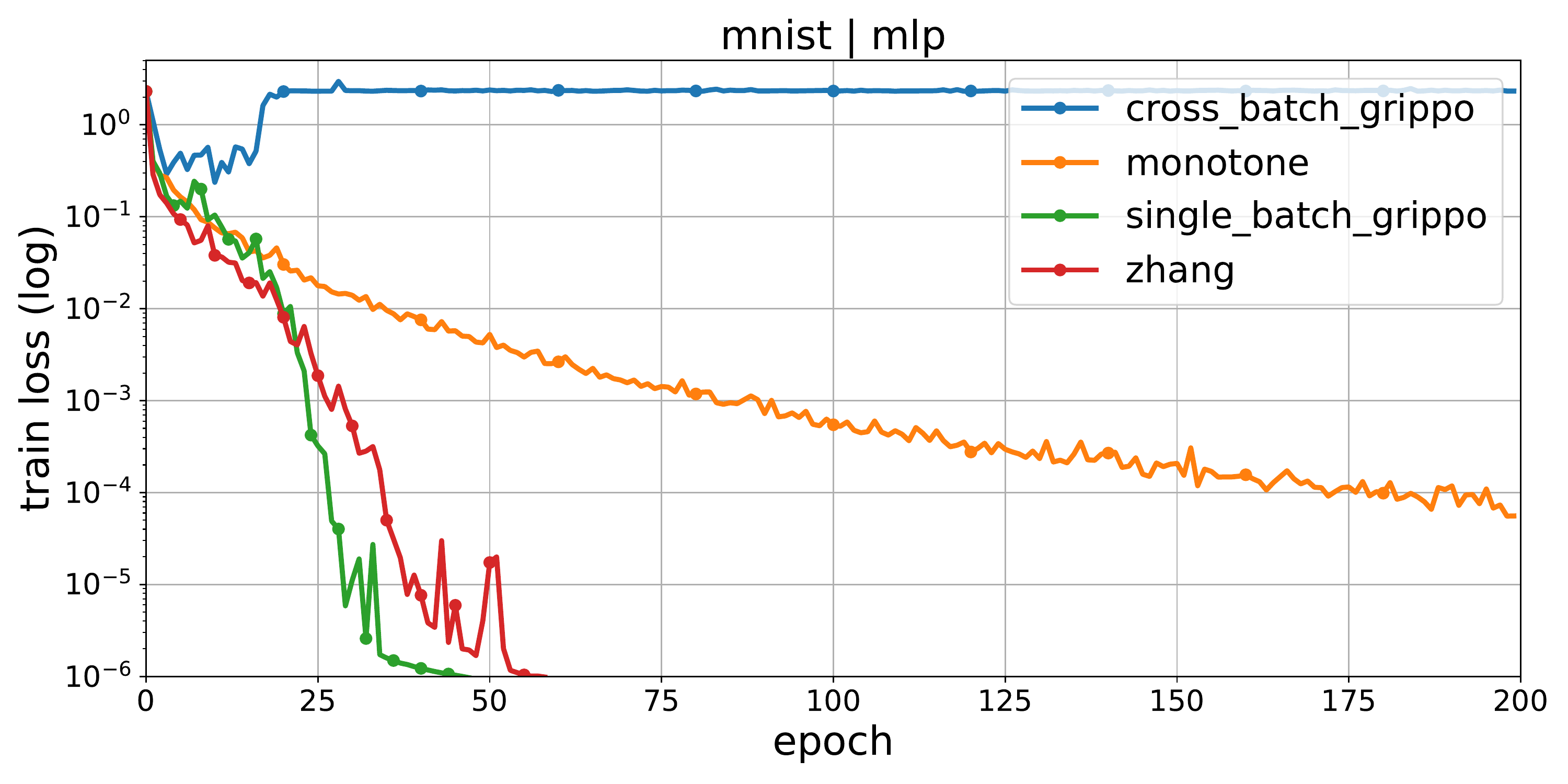}}
\subfloat{\includegraphics[width=\imgS\linewidth]{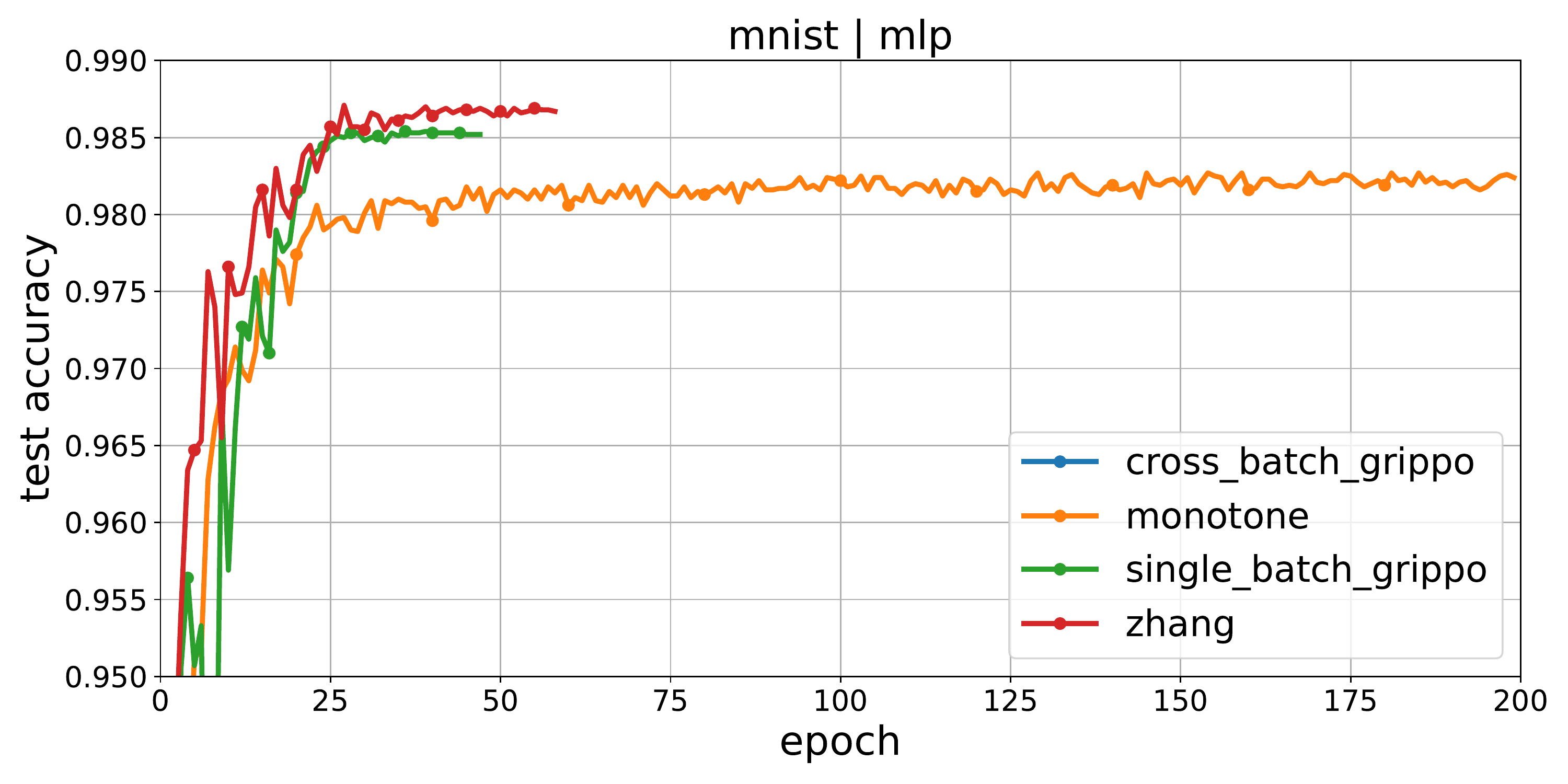}}
\subfloat{\includegraphics[width=\imgS\linewidth]{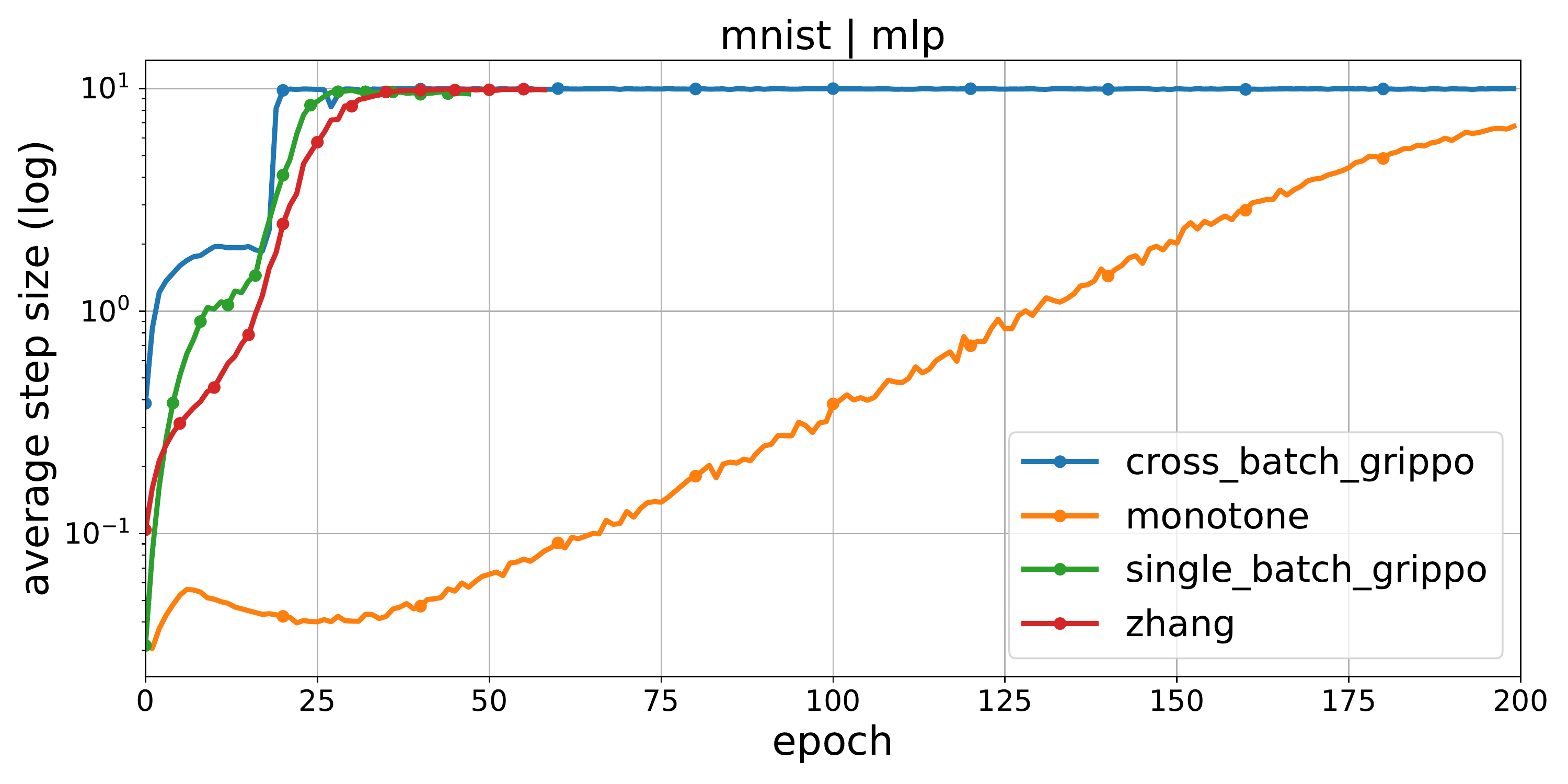}}
\subfloat{\includegraphics[width=\imgS\linewidth]{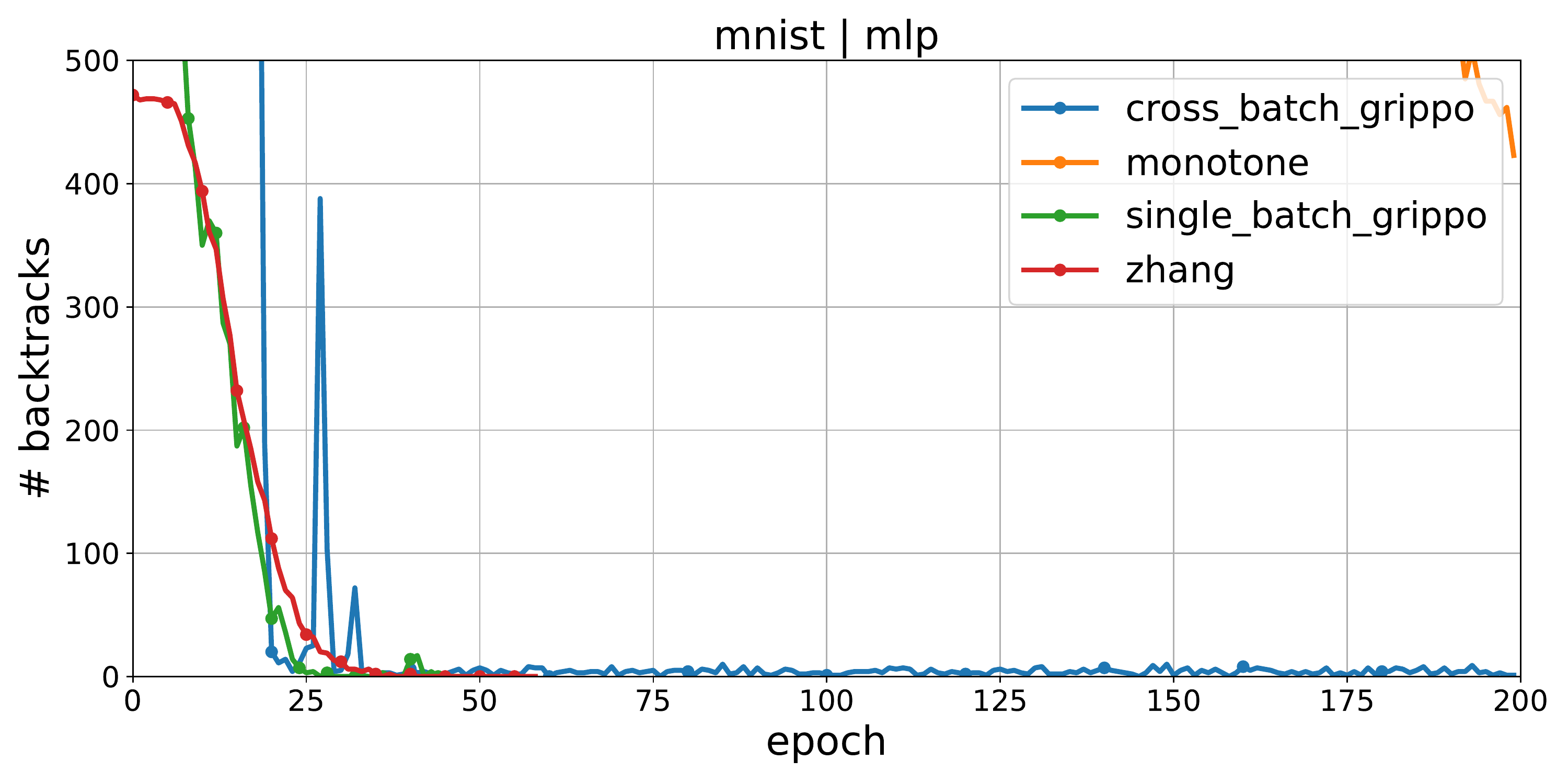}}\\
\renewcommand{\model}{cifar10_resnet}
\renewcommand{\modelname}{cifar10\_resnet}
\subfloat{\includegraphics[width=\imgS\linewidth]{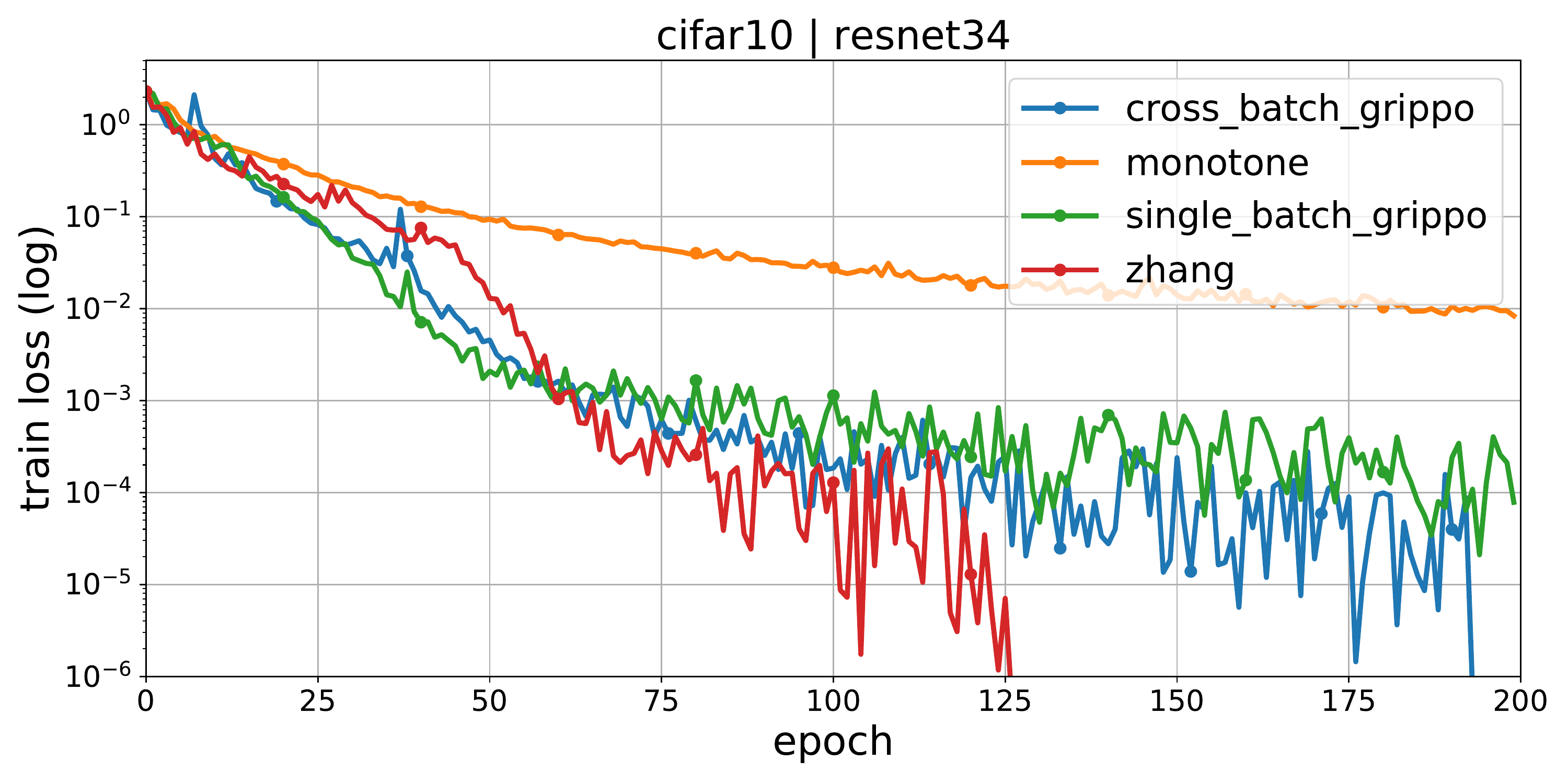}}
\subfloat{\includegraphics[width=\imgS\linewidth]{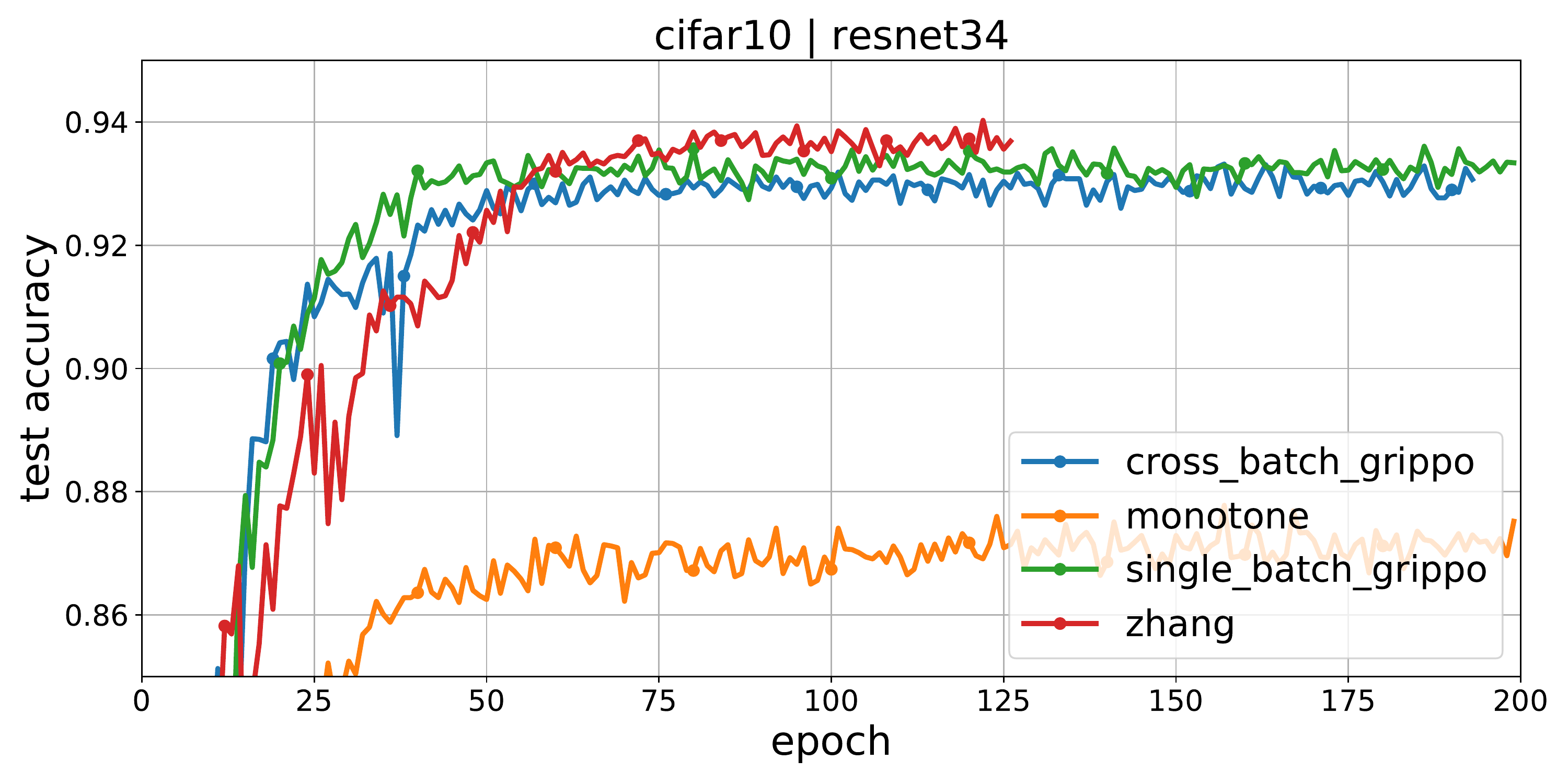}}
\subfloat{\includegraphics[width=\imgS\linewidth]{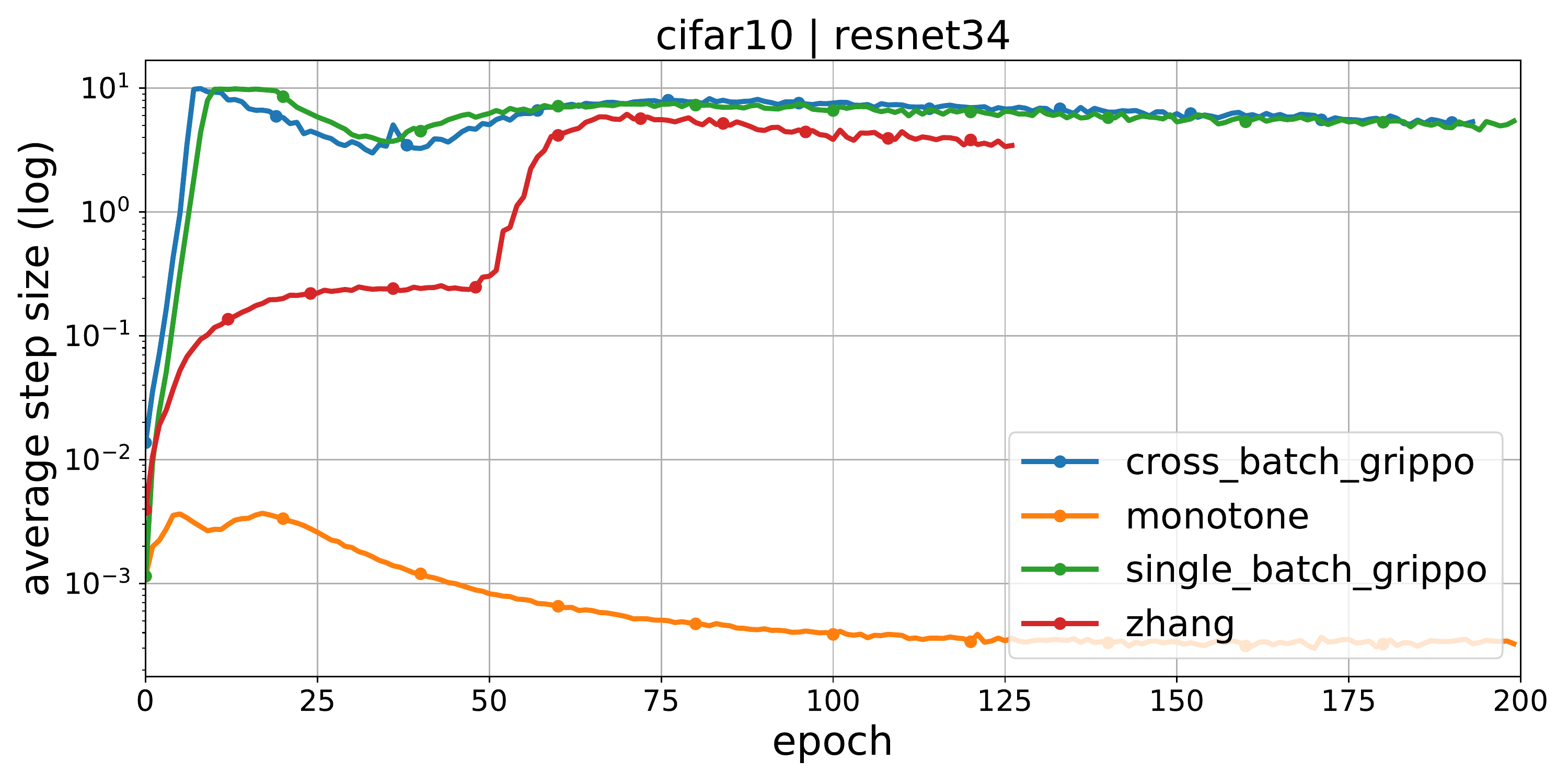}}
\subfloat{\includegraphics[width=\imgS\linewidth]{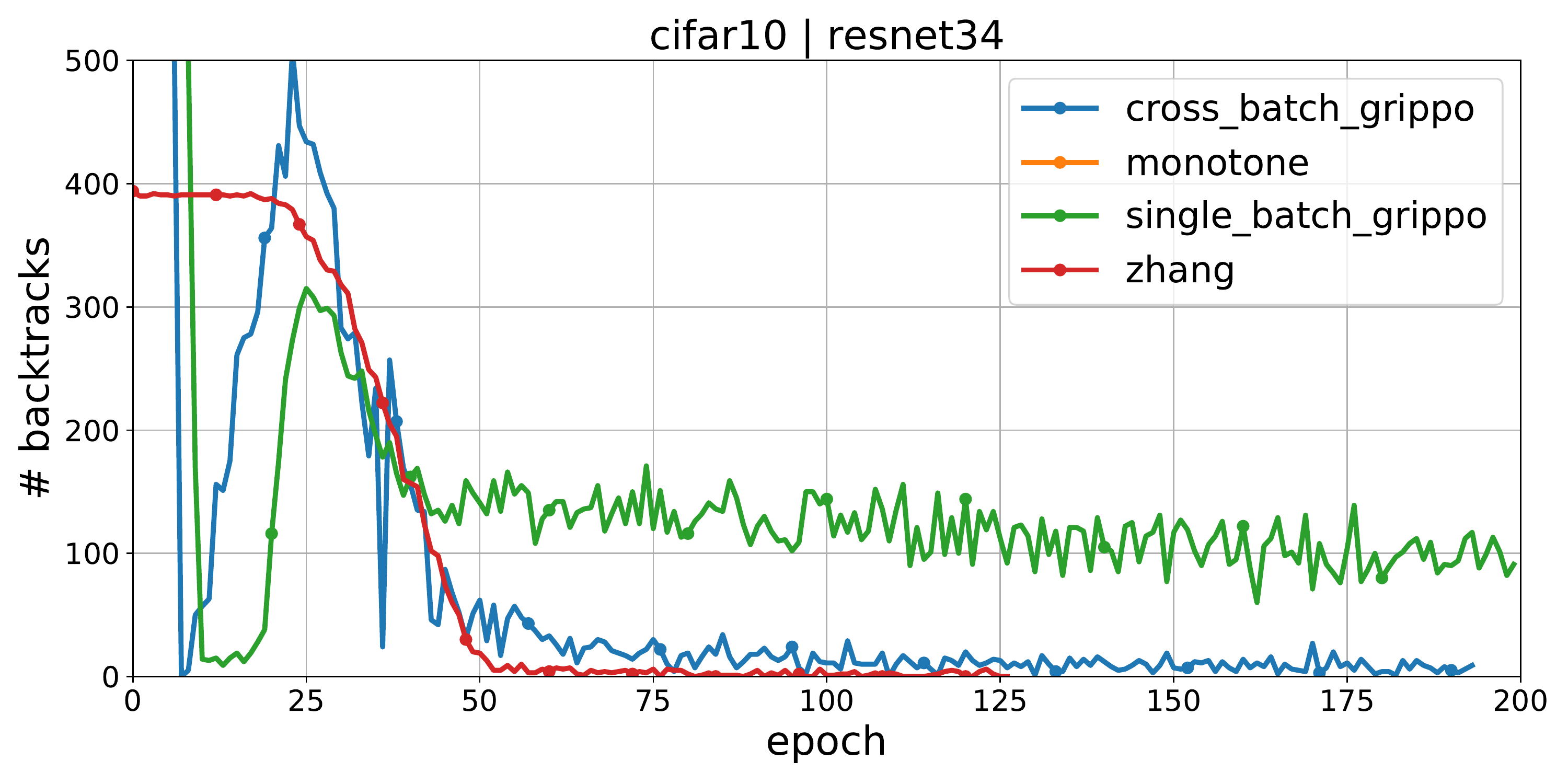}}\\
\renewcommand{\model}{cifar10_densenet}
\renewcommand{\modelname}{cifar10\_densenet}
\subfloat{\includegraphics[width=\imgS\linewidth]{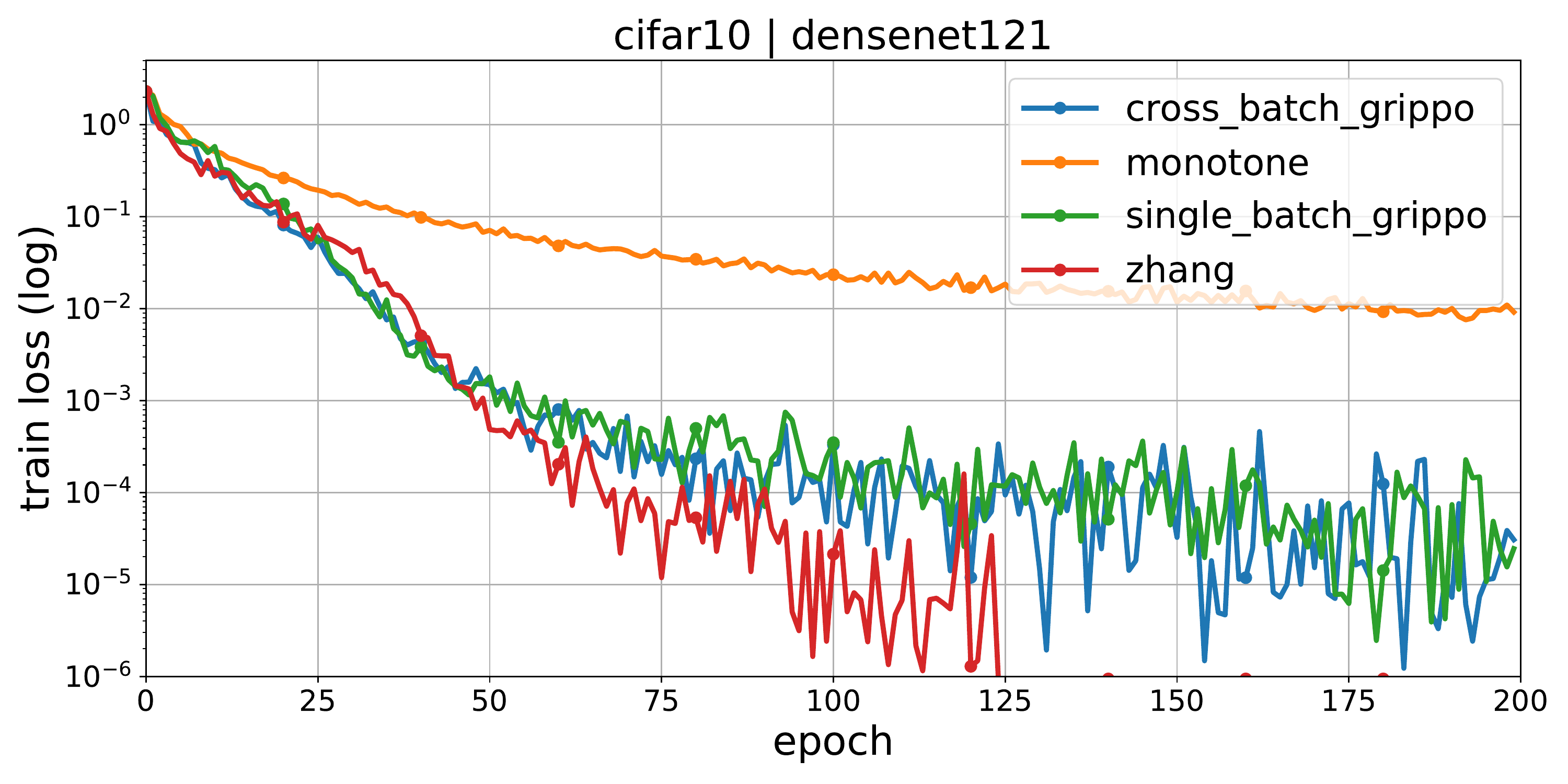}}
\subfloat{\includegraphics[width=\imgS\linewidth]{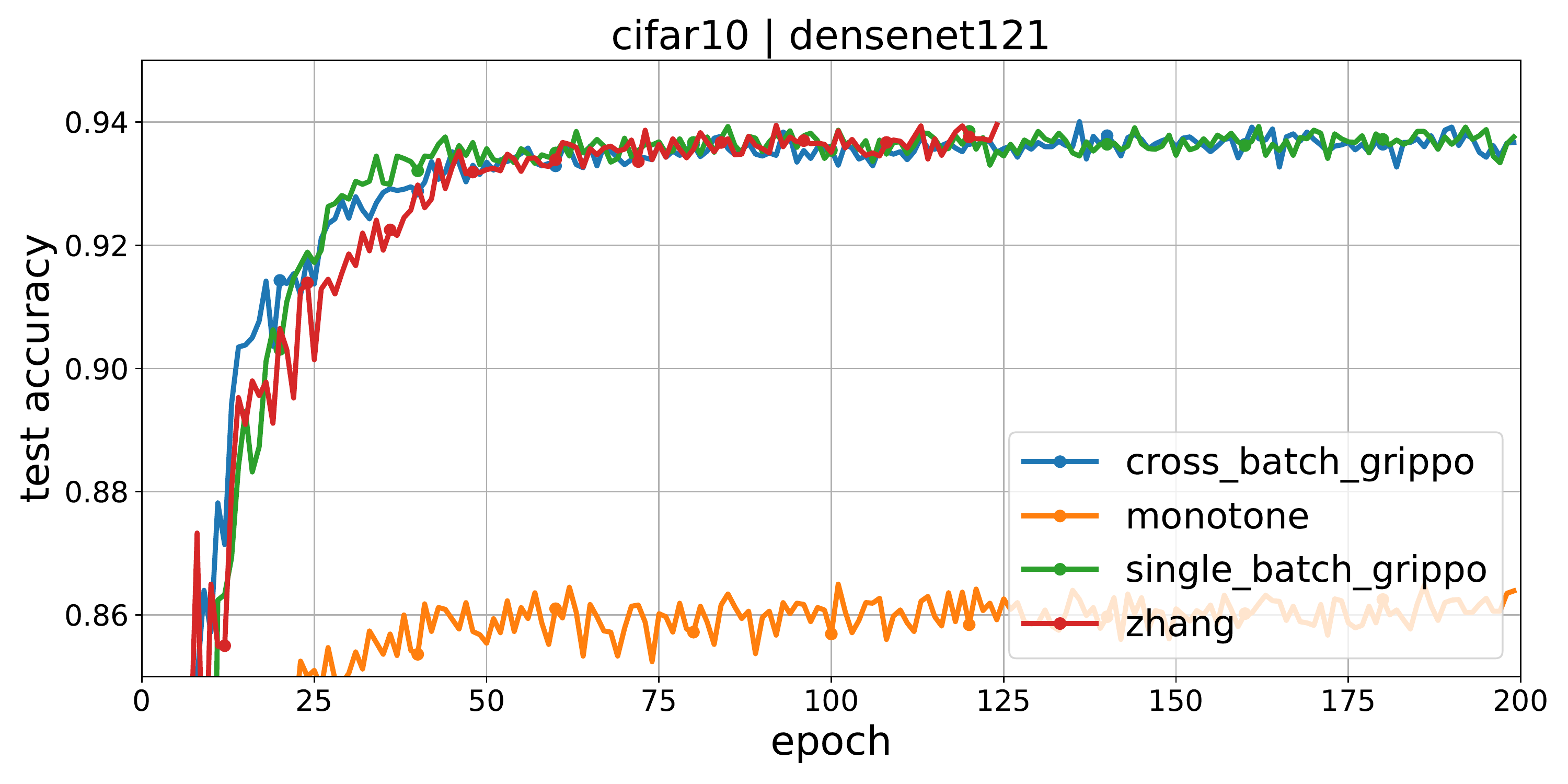}}
\subfloat{\includegraphics[width=\imgS\linewidth]{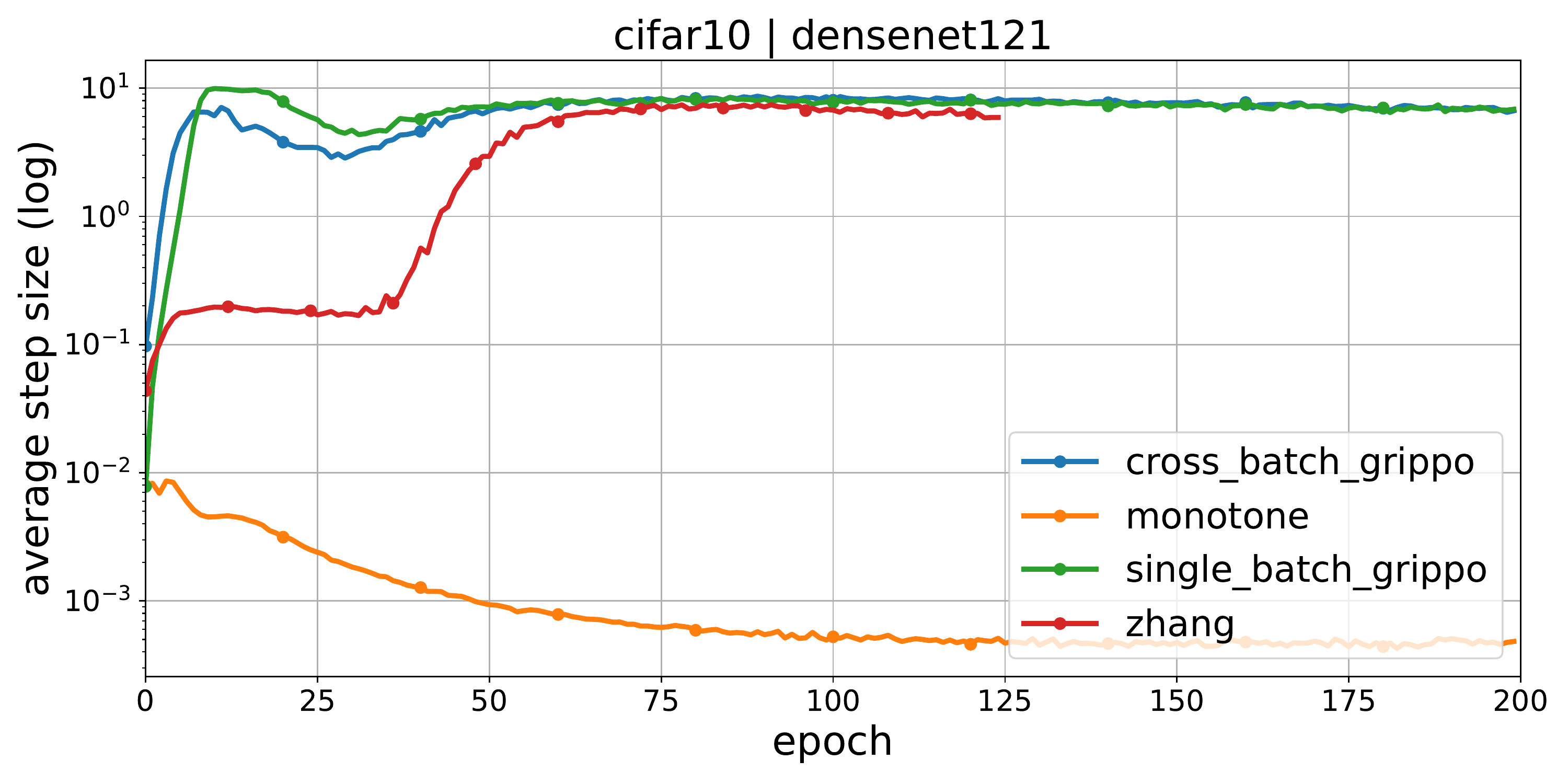}}
\subfloat{\includegraphics[width=\imgS\linewidth]{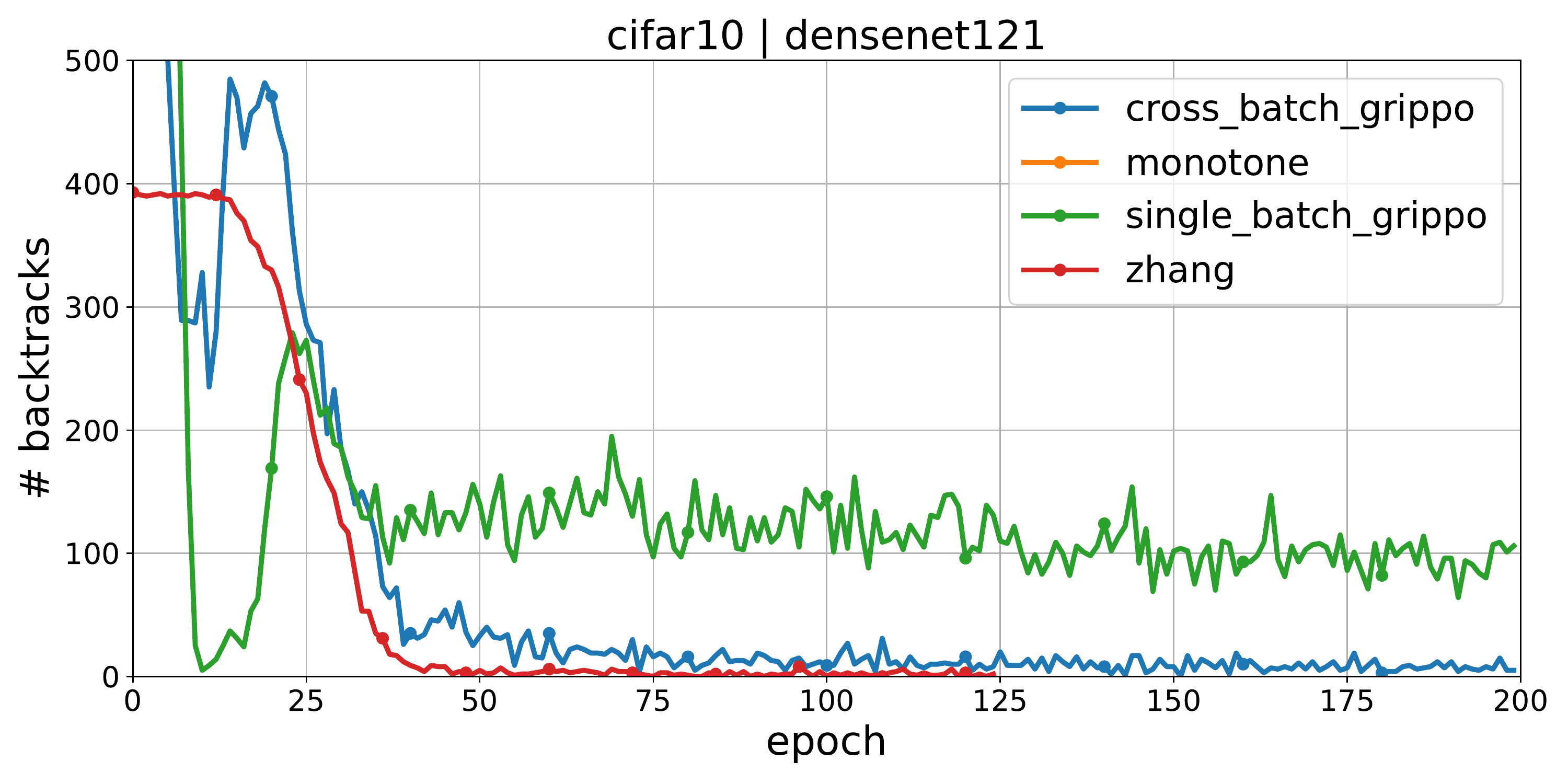}}\\
\renewcommand{\model}{cifar100_resnet}
\renewcommand{\modelname}{cifar100\_resnet}
\subfloat{\includegraphics[width=\imgS\linewidth]{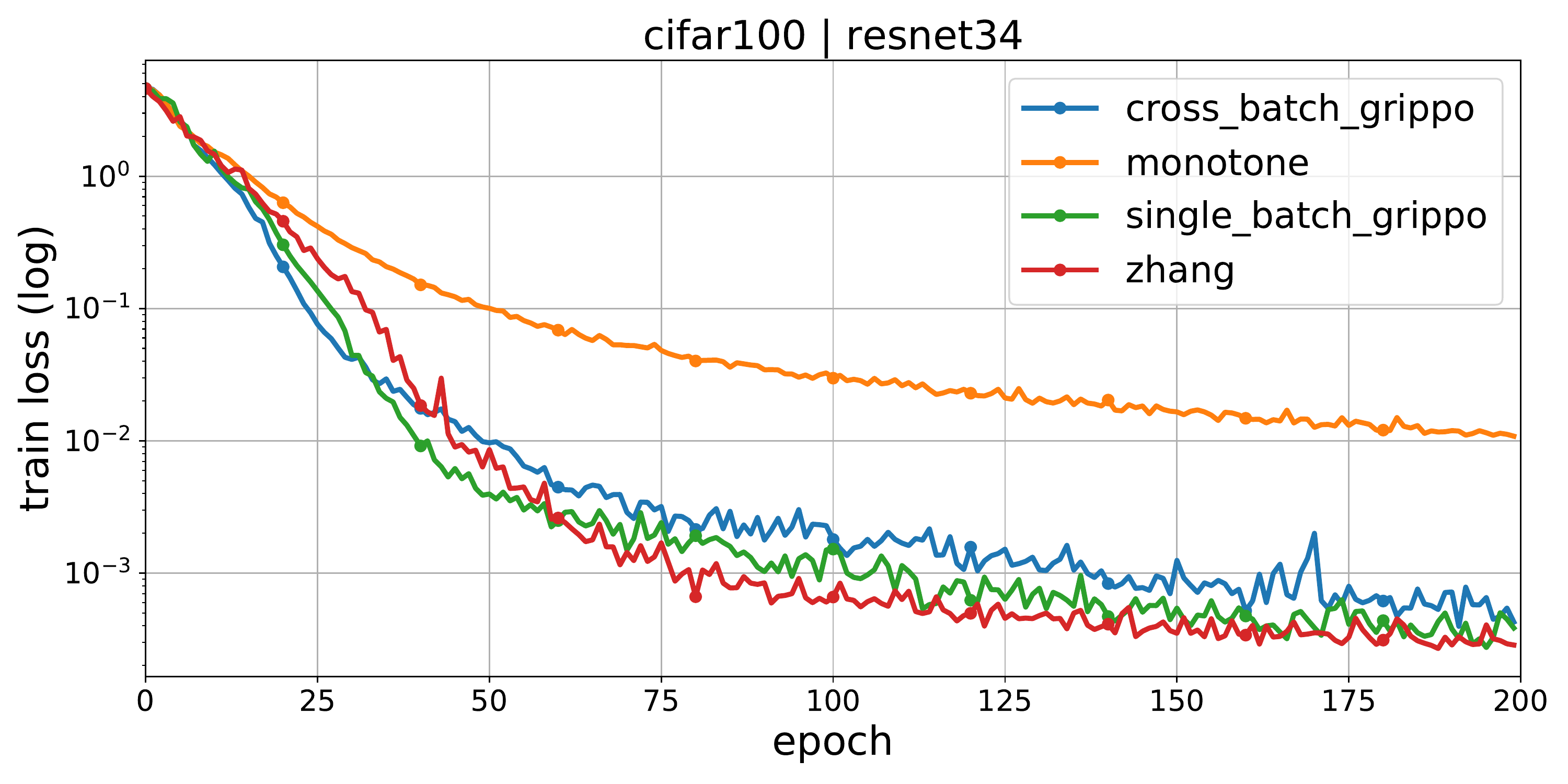}}
\subfloat{\includegraphics[width=\imgS\linewidth]{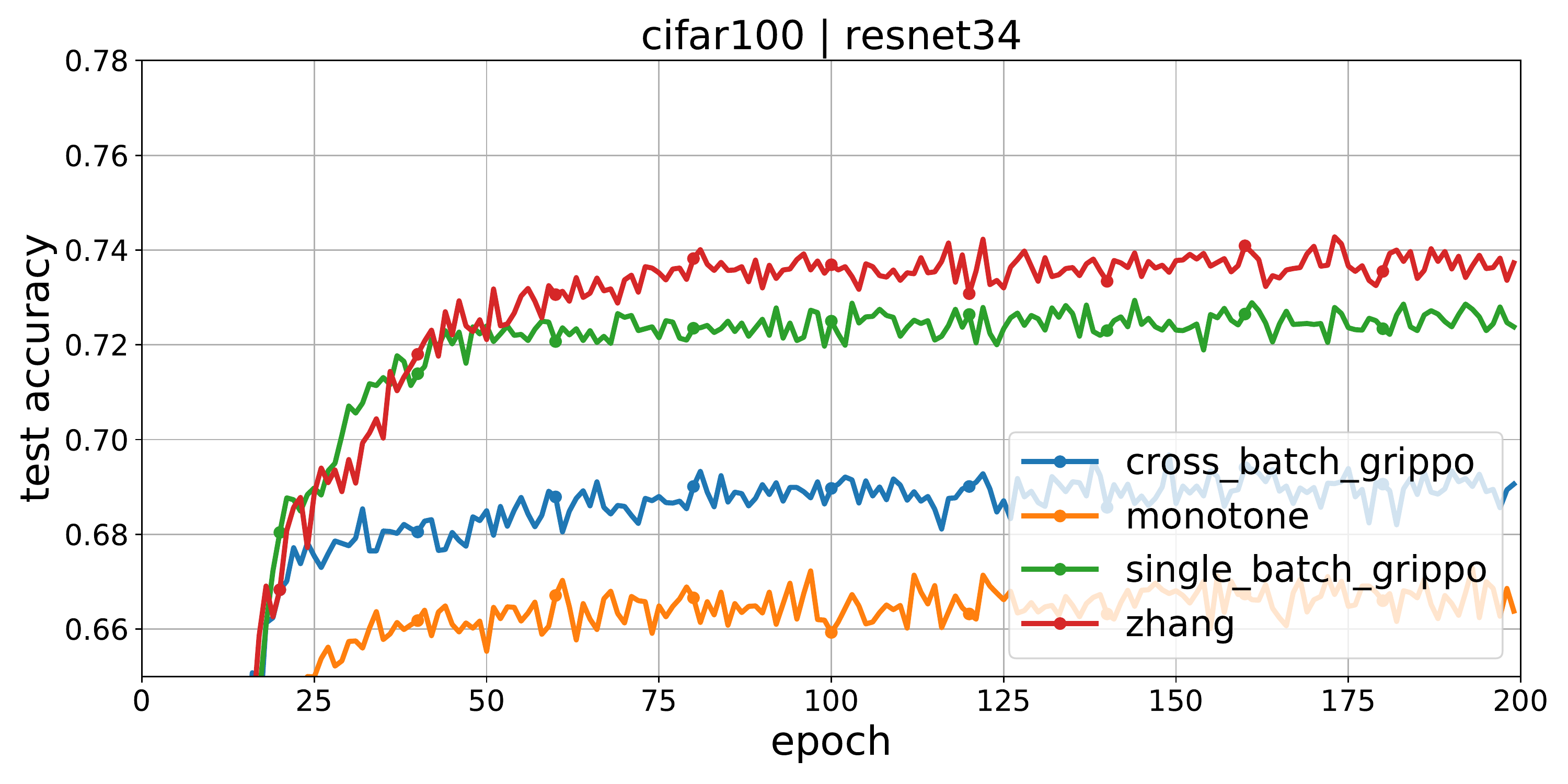}}
\subfloat{\includegraphics[width=\imgS\linewidth]{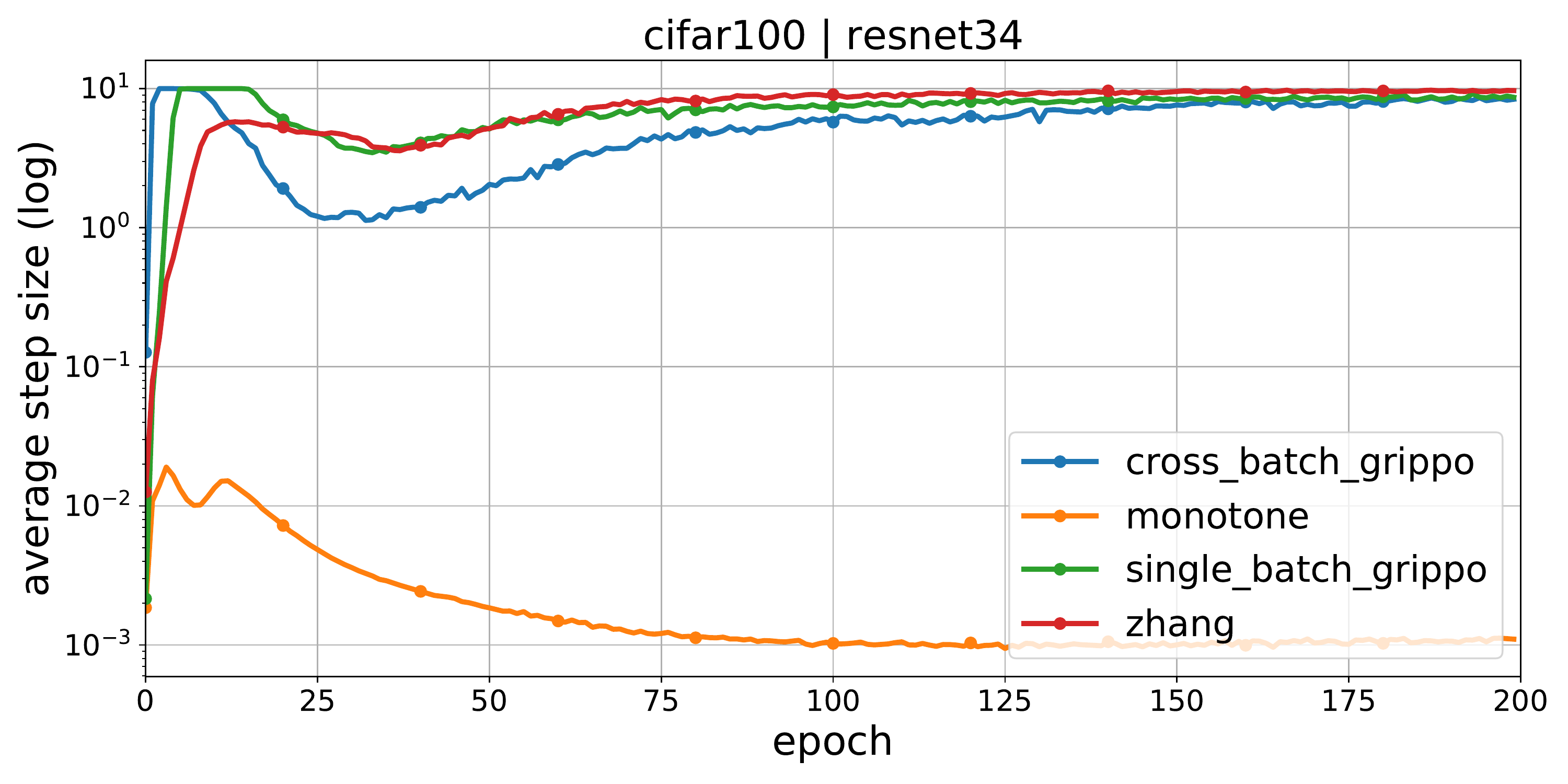}}
\subfloat{\includegraphics[width=\imgS\linewidth]{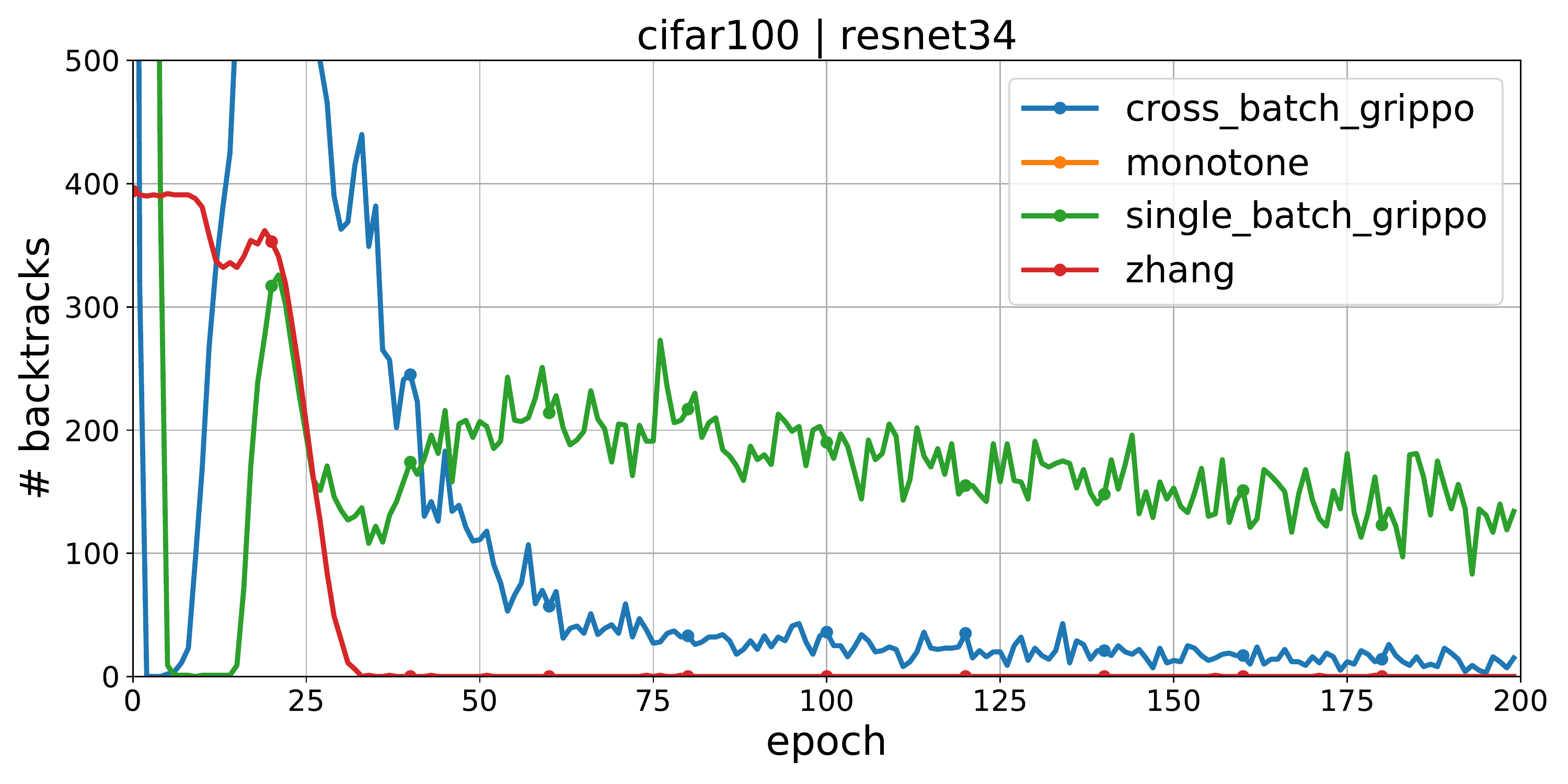}}\\
\renewcommand{\model}{cifar100_densenet}
\renewcommand{\modelname}{cifar100\_densenet}
\subfloat{\includegraphics[width=\imgS\linewidth]{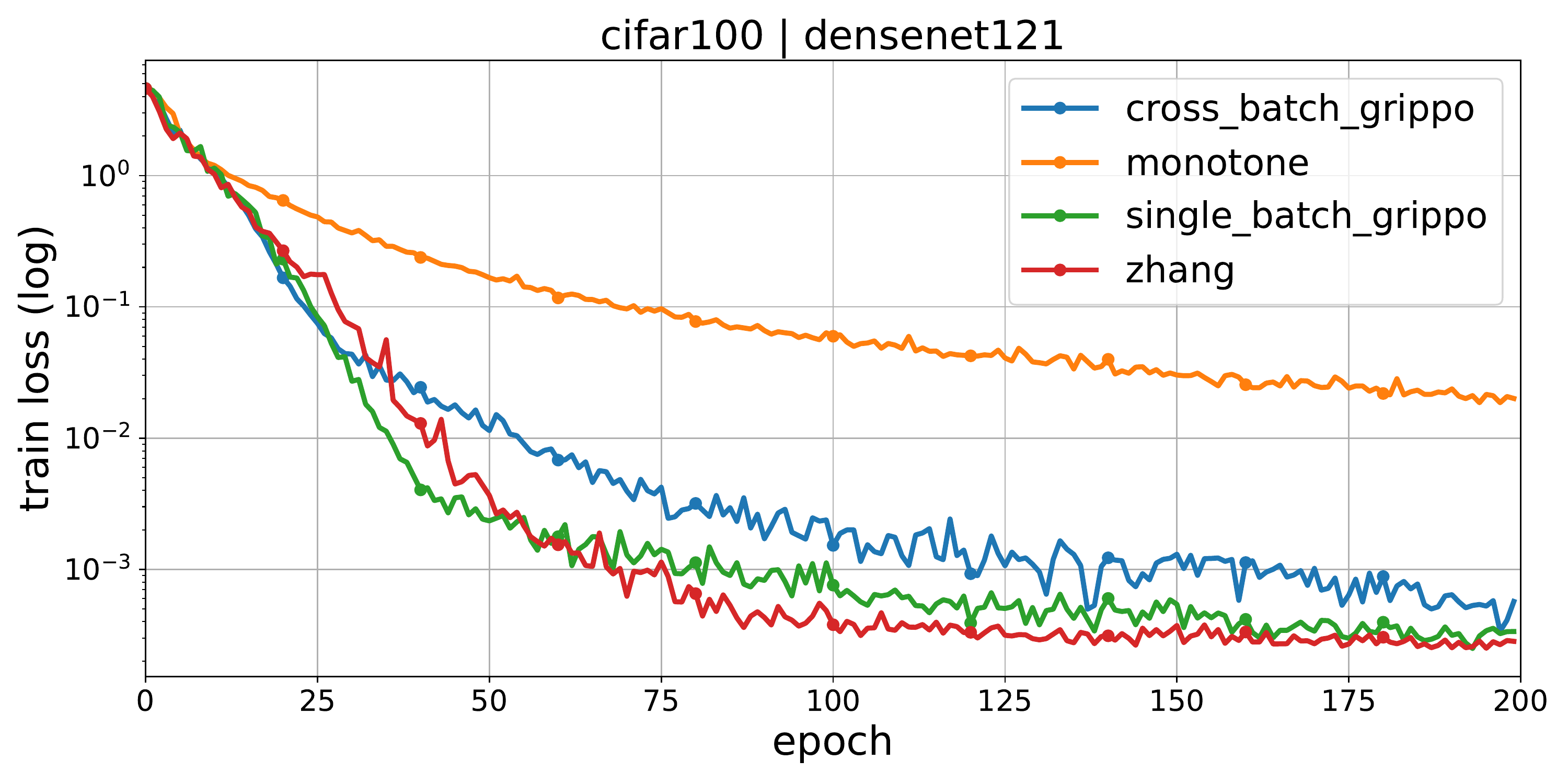}}
\subfloat{\includegraphics[width=\imgS\linewidth]{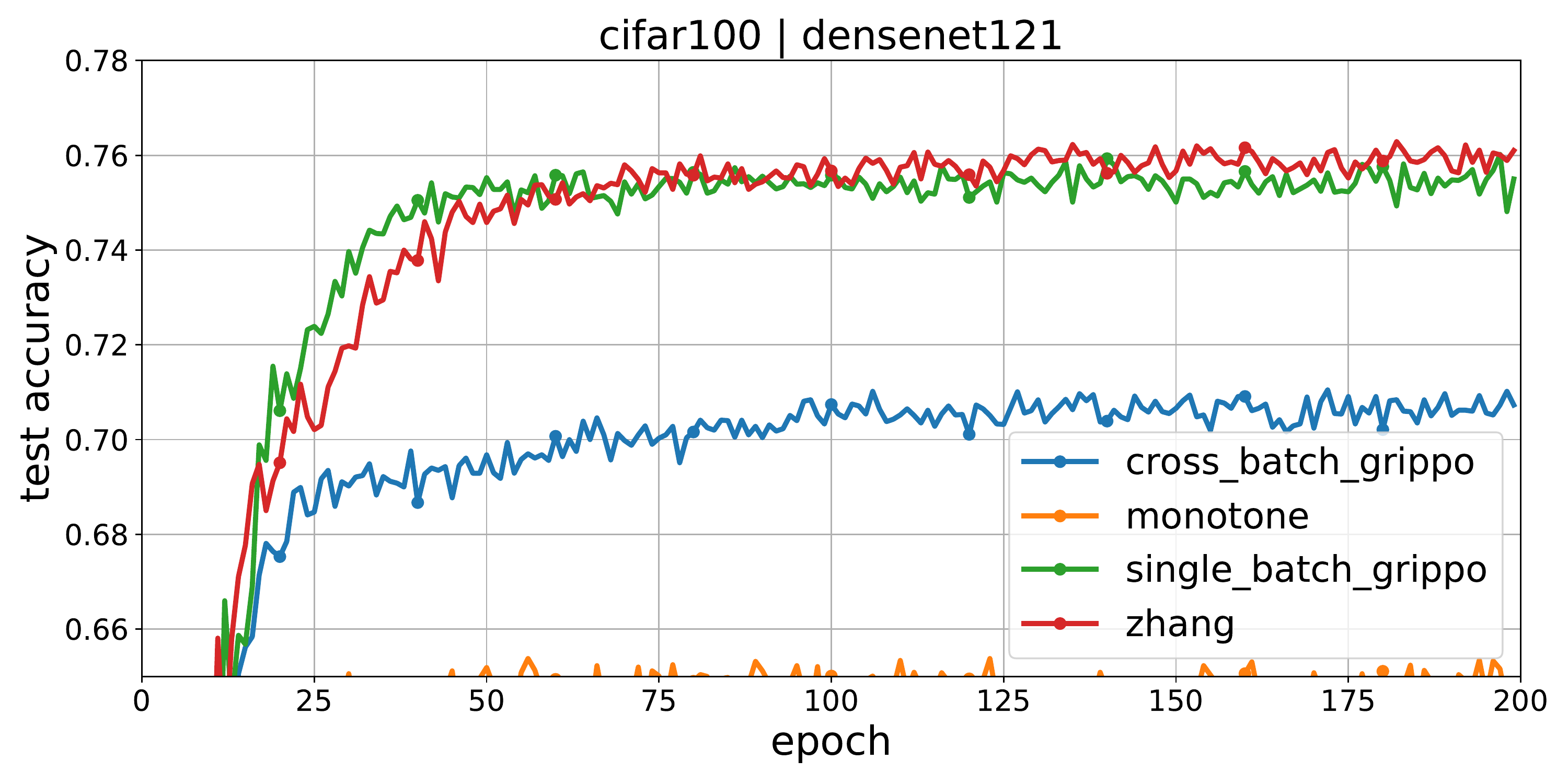}}
\subfloat{\includegraphics[width=\imgS\linewidth]{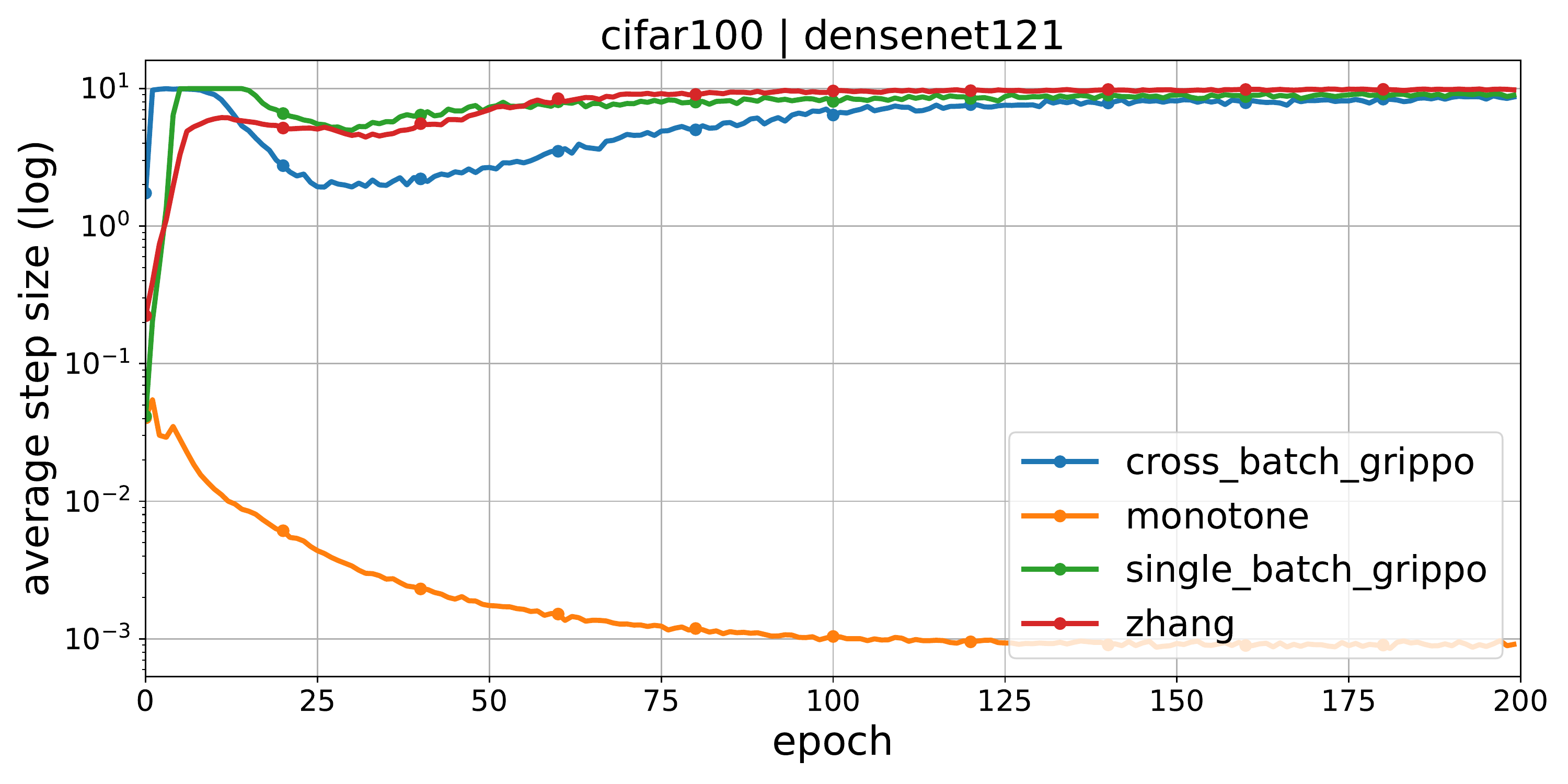}}
\subfloat{\includegraphics[width=\imgS\linewidth]{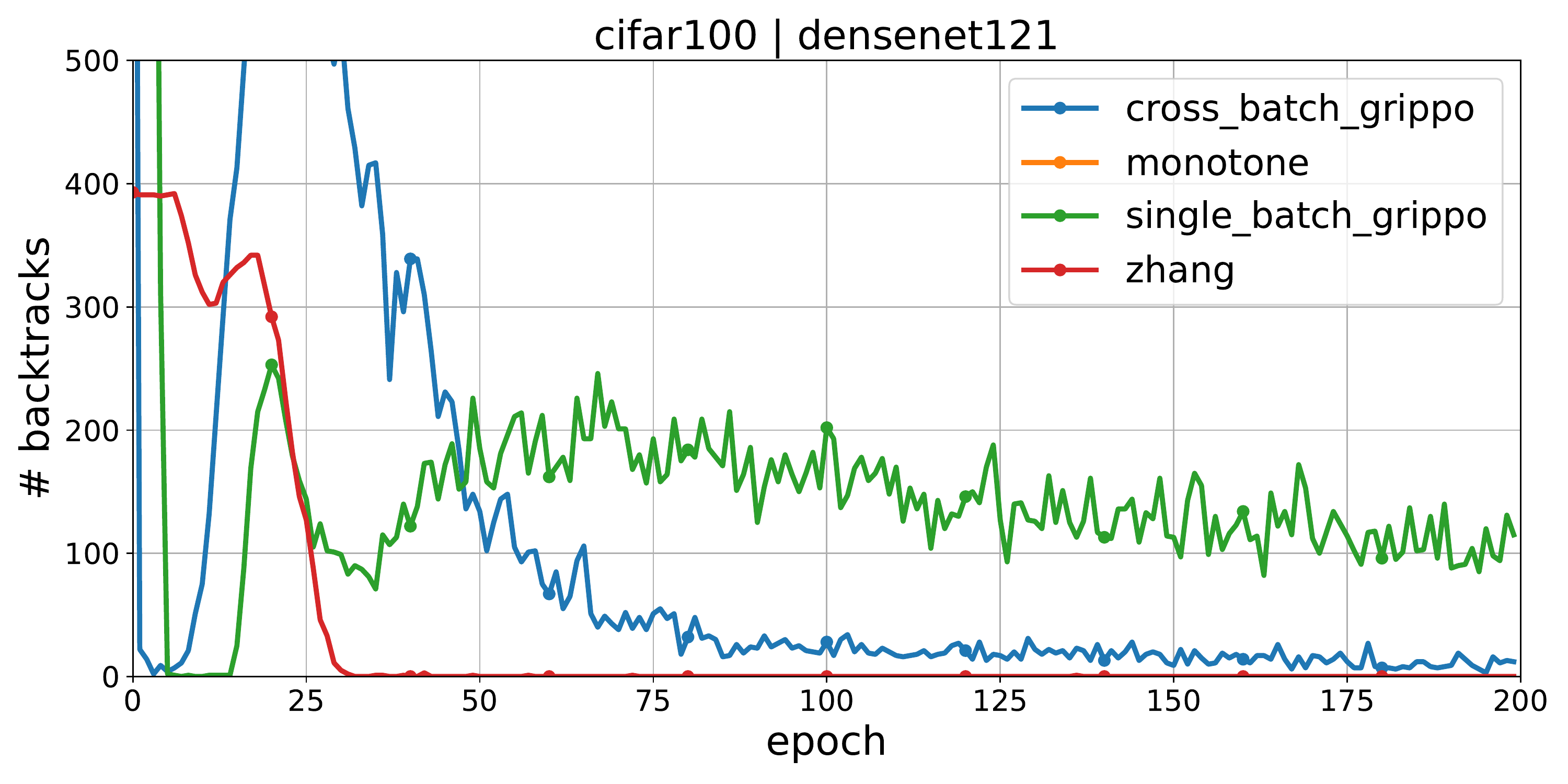}}\\
\renewcommand{\model}{fashion_effb1}
\renewcommand{\modelname}{fashion\_effb1}
\subfloat{\includegraphics[width=\imgS\linewidth]{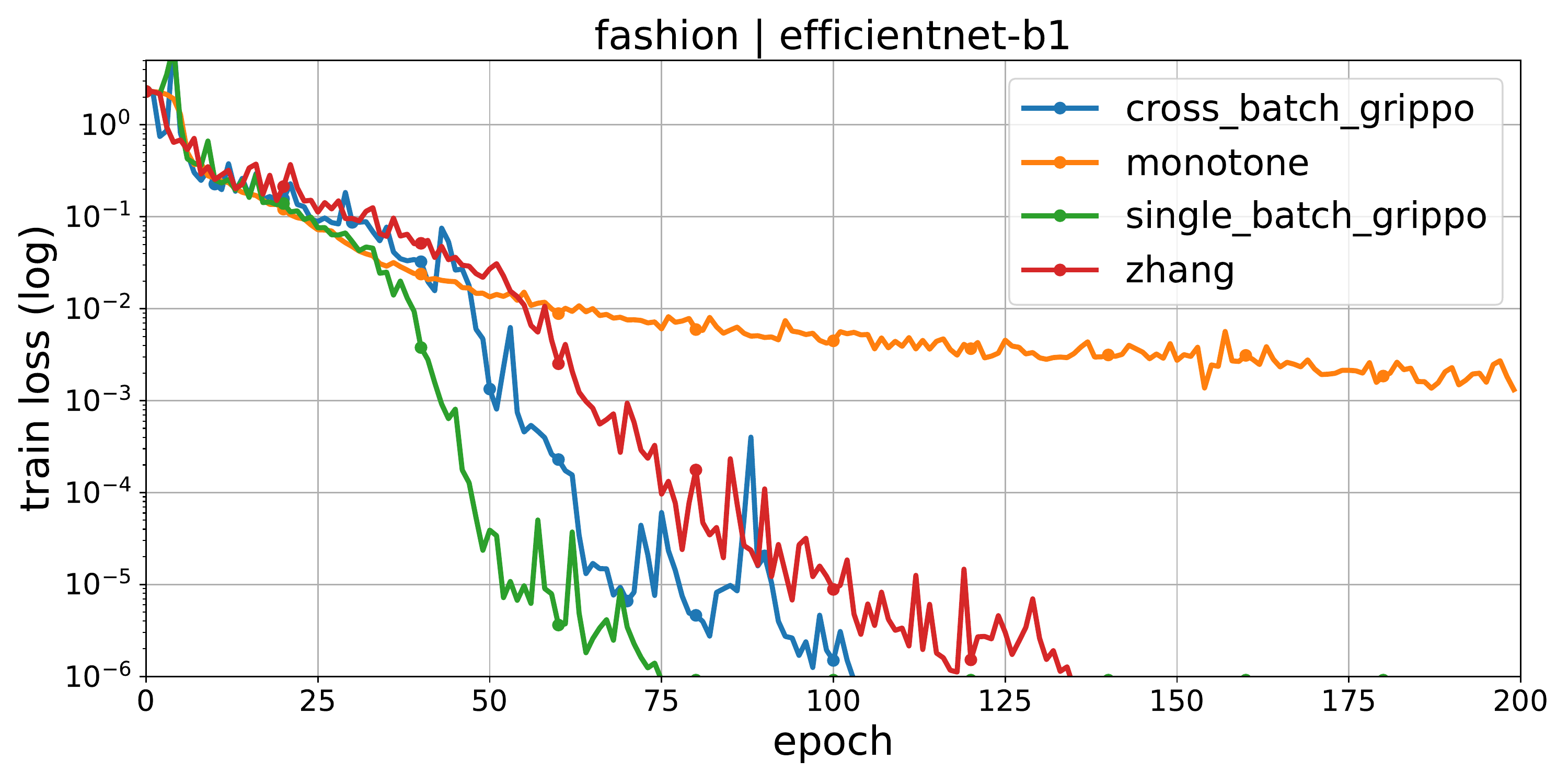}}
\subfloat{\includegraphics[width=\imgS\linewidth]{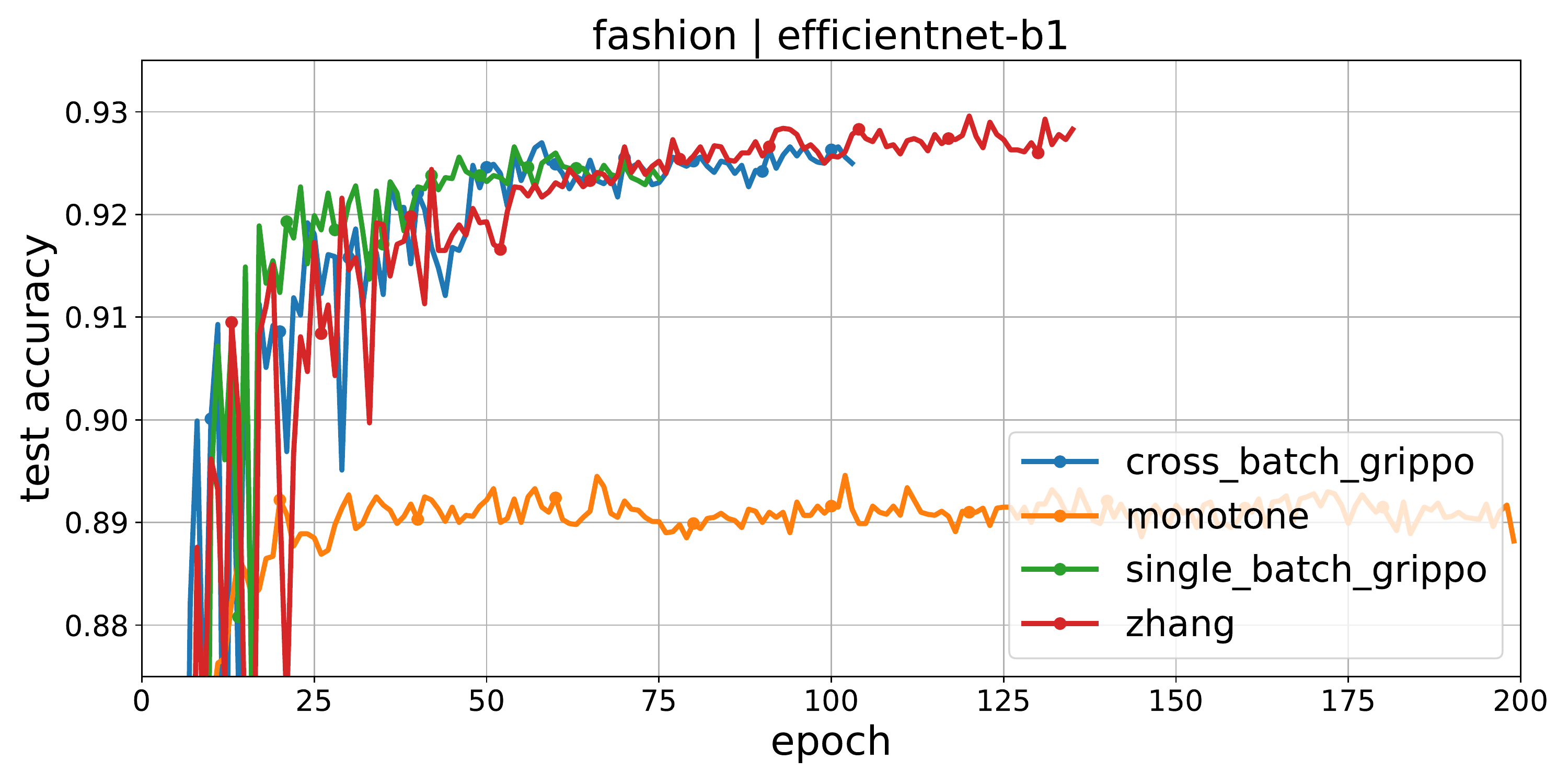}}
\subfloat{\includegraphics[width=\imgS\linewidth]{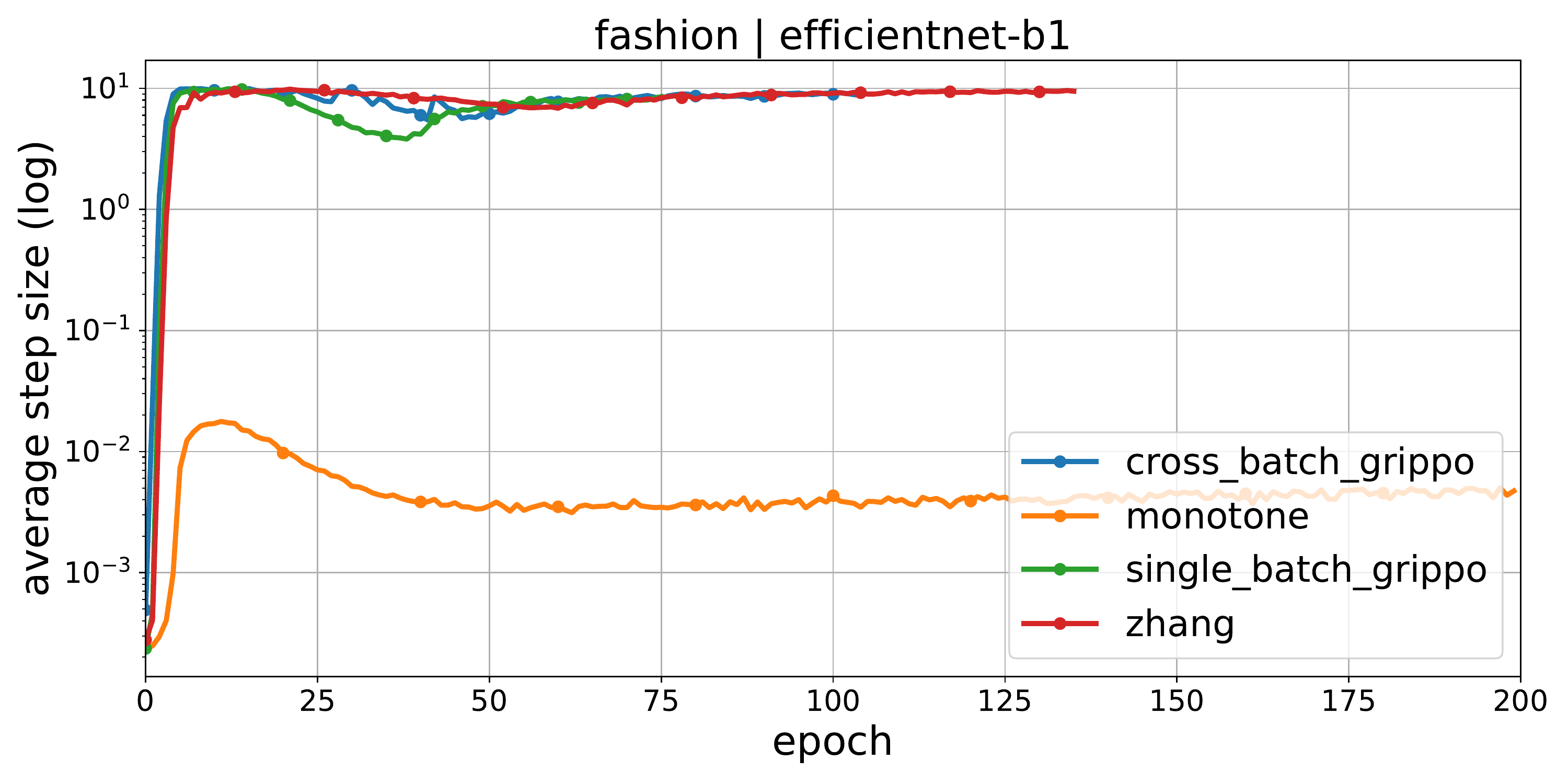}}
\subfloat{\includegraphics[width=\imgS\linewidth]{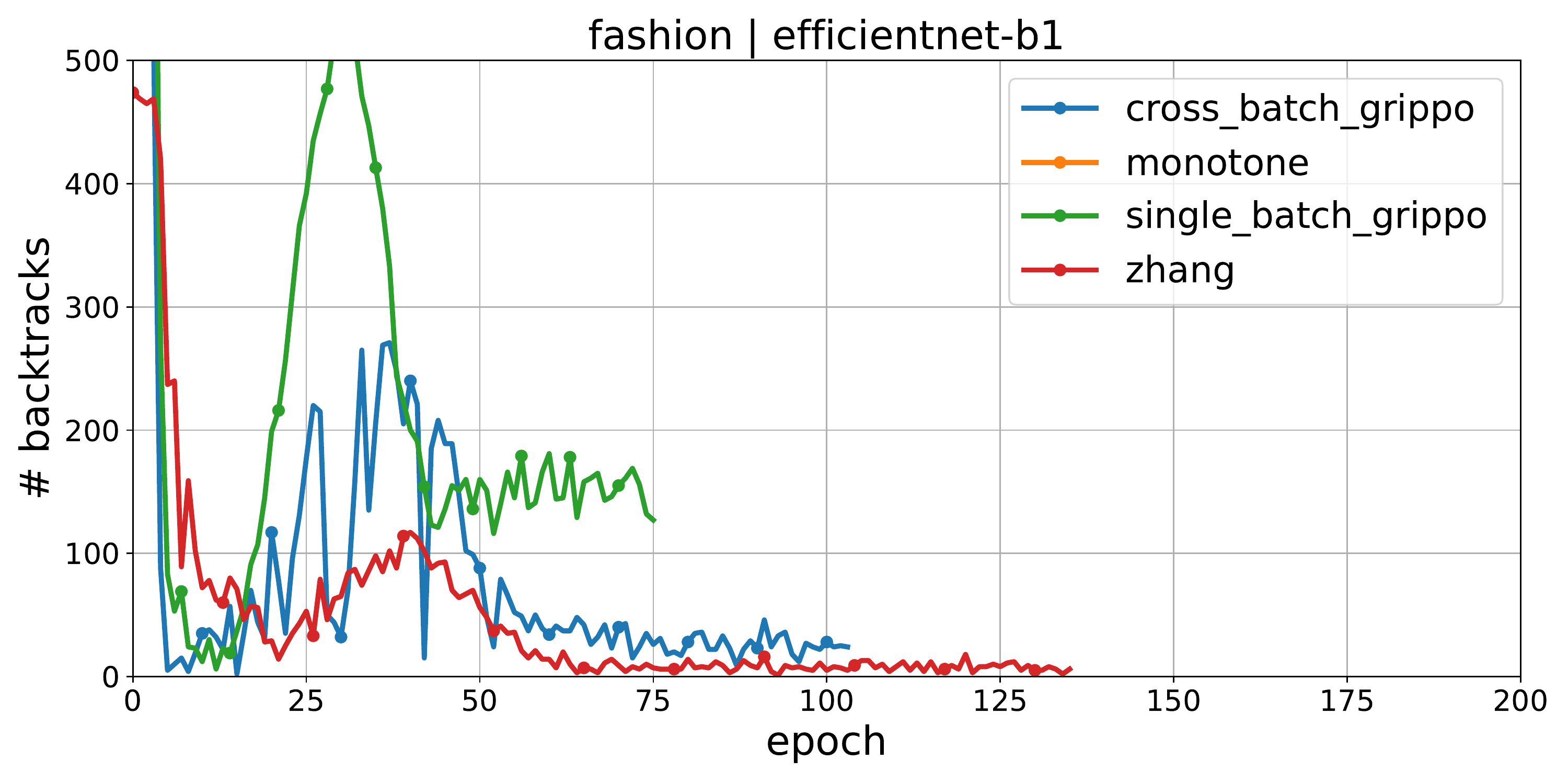}}\\
\renewcommand{\model}{svhn_wrn}
\renewcommand{\modelname}{svhn\_wrn}
\subfloat{\includegraphics[width=\imgS\linewidth]{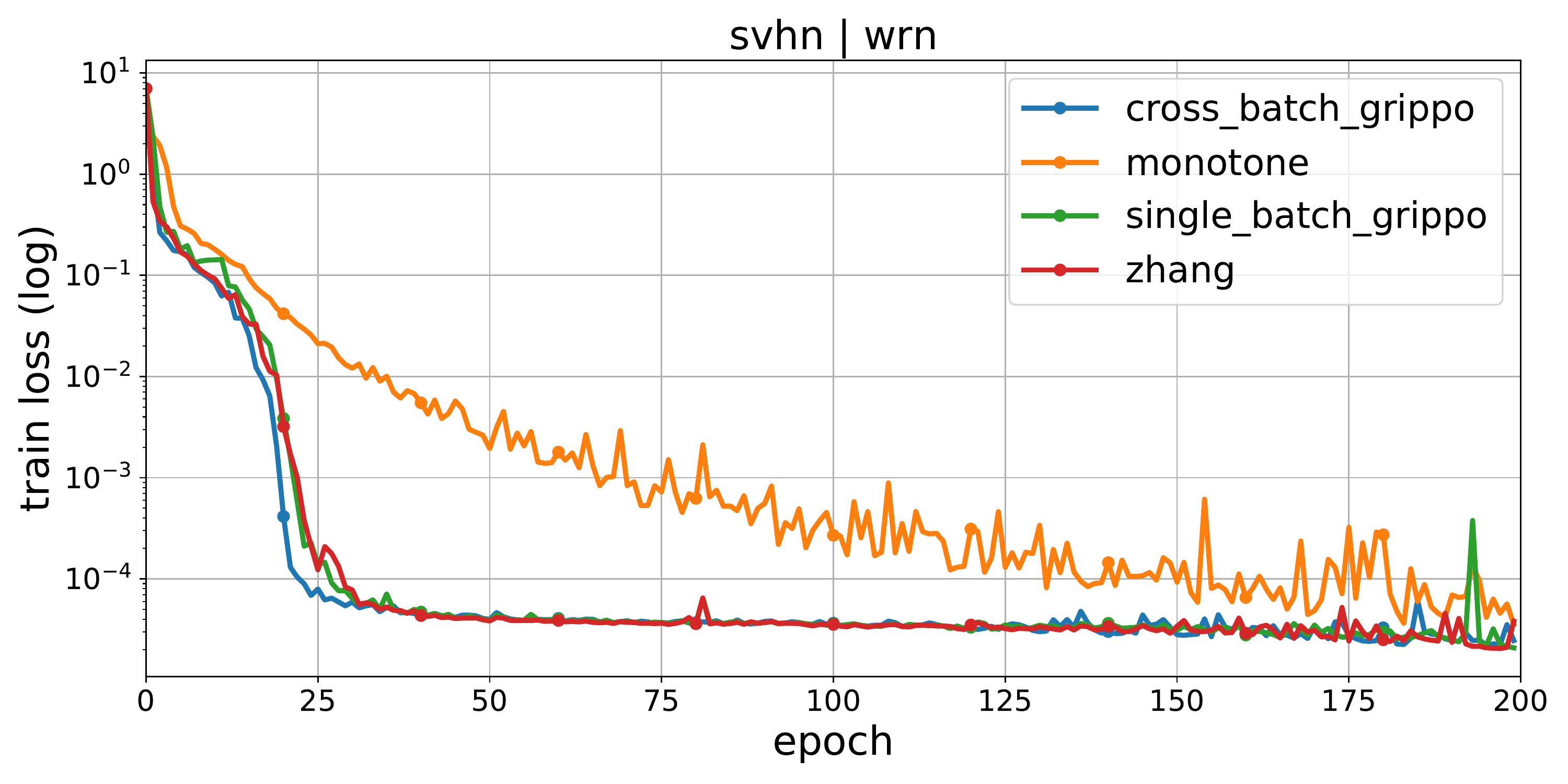}}
\subfloat{\includegraphics[width=\imgS\linewidth]{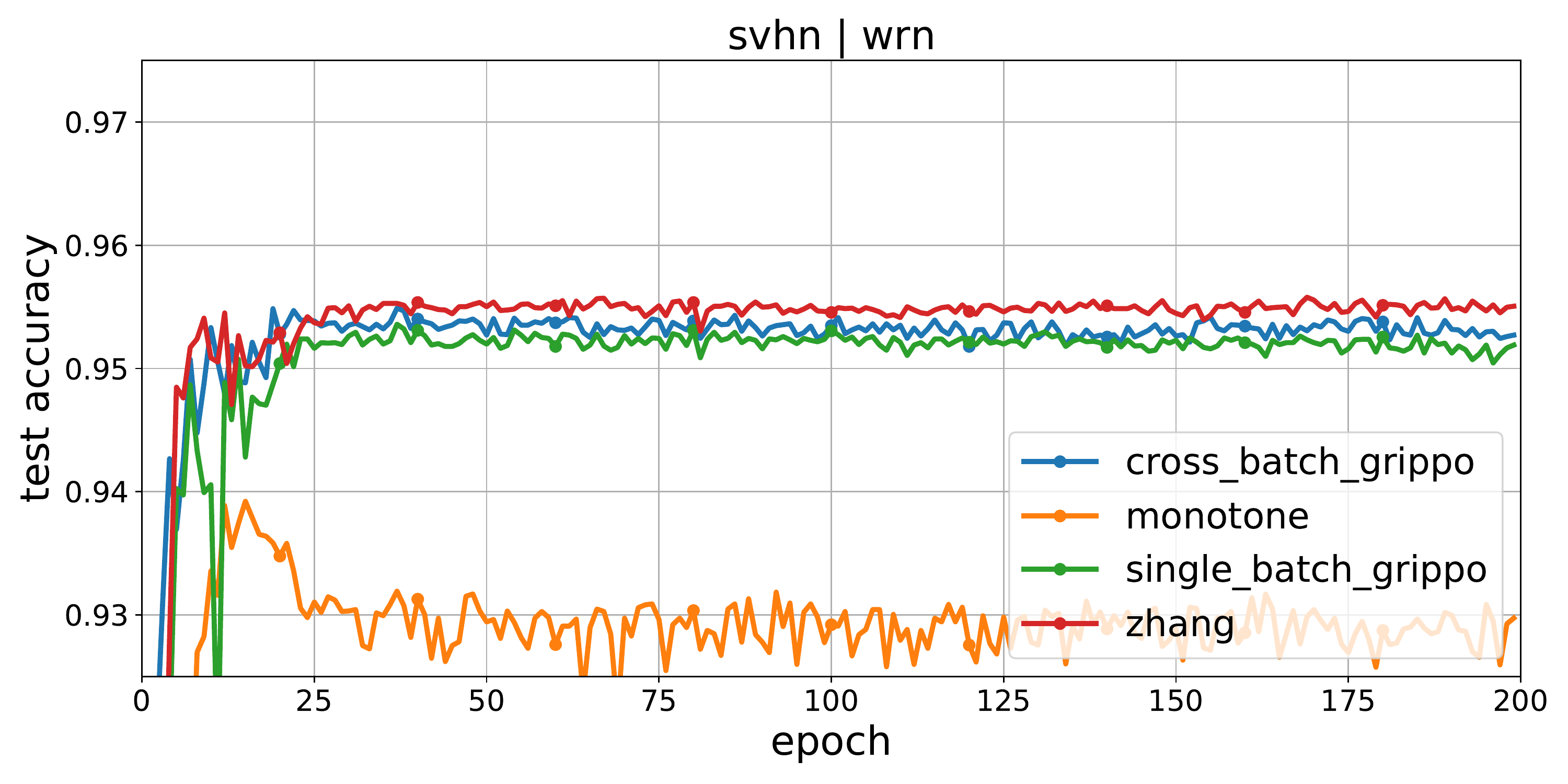}}
\subfloat{\includegraphics[width=\imgS\linewidth]{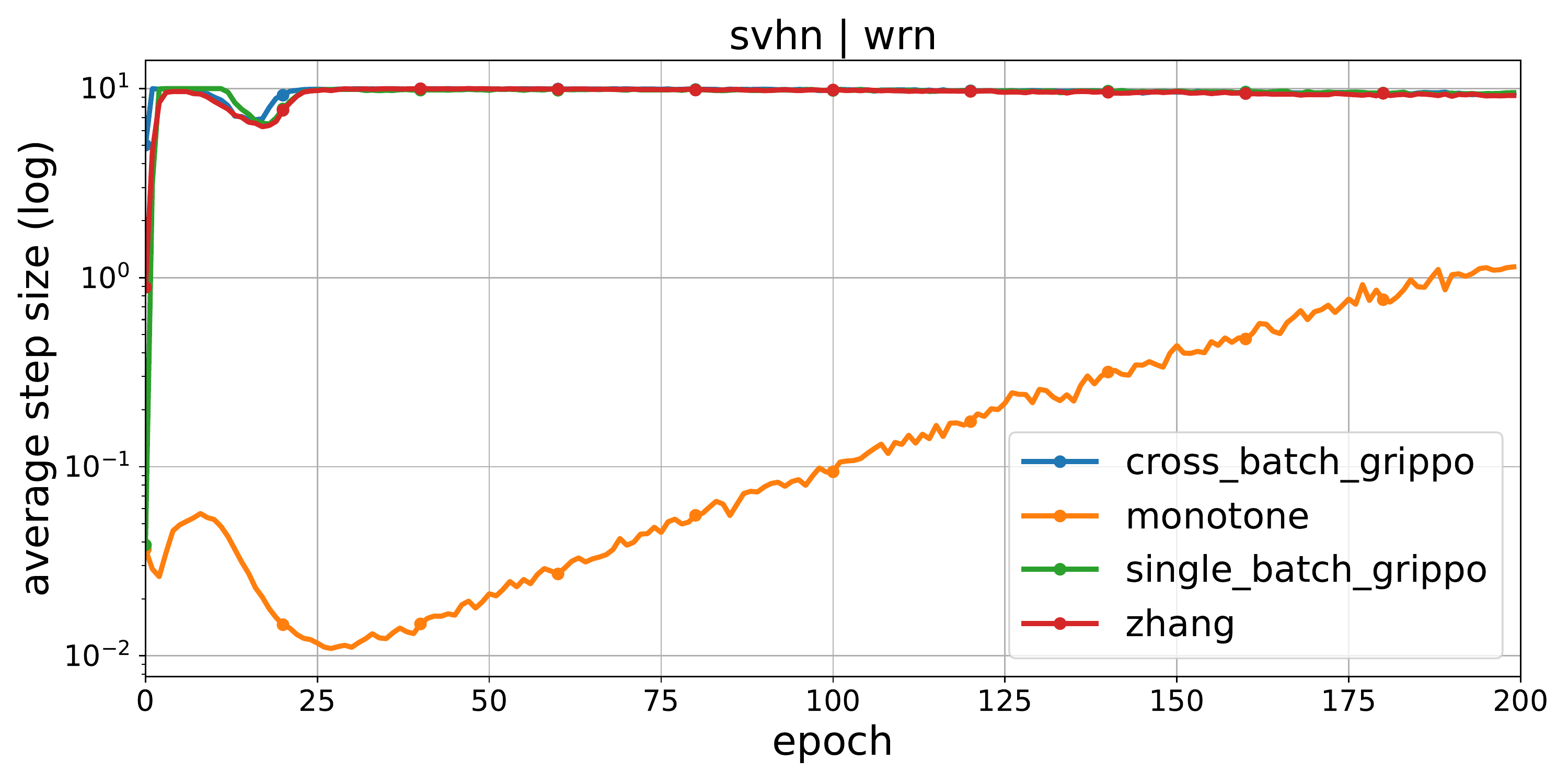}}
\subfloat{\includegraphics[width=\imgS\linewidth]{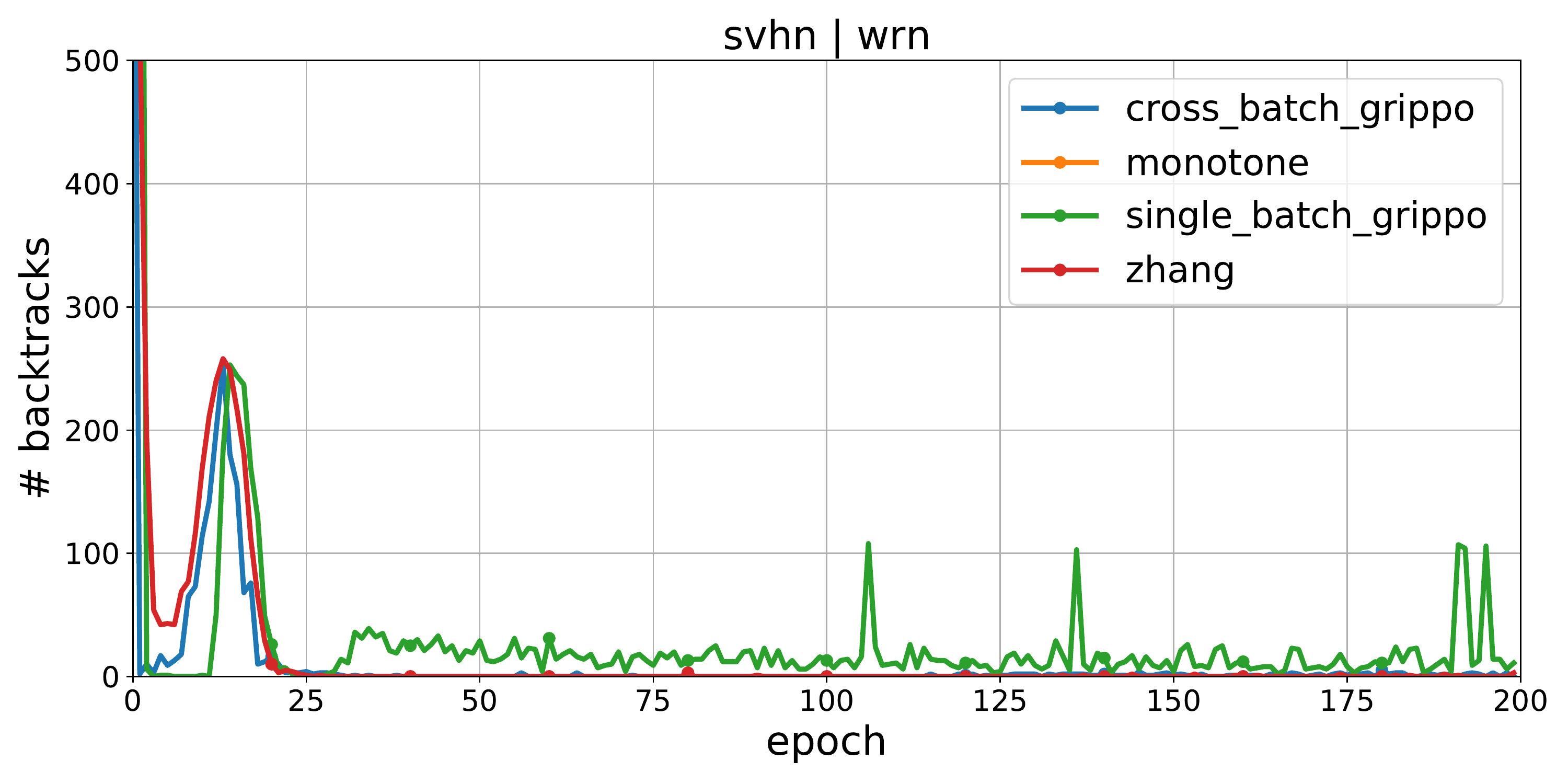}}
\caption{Comparison between different monotone and nonmonotone line search conditions. Each row focus on a dataset/model combination. First column: train loss. Second column: test accuracy. Third column: average step size of the epoch. Fourth column: cumulative number of backtracks in the epoch.}\label{fig:line_search}
	\end{minipage}
\end{figure}

	\subsection{Zoom in on the Amount of Backtracks}\label{sec:supp_f_eval}
	In this subsection, we zoom in on the amount of backtracks employed by PoNoS and PoNoS\_reset0. Instead of showing the cumulative number of backtracks in each epoch (as in the fourth column of Figure \ref{fig:reset}), Figure \ref{fig:f_eval} reports the amount of backtracks in each iteration. To help visualizing the whole optimization procedure, we average the number of backtracks over 10 consecutive iterations. Respectively on the leftmost and rightmost column of Figure \ref{fig:f_eval} we report the first 20000 iterations and 100 iterations (out of 78200-114600). From these columns we can observe:
\begin{itemize}
	\item PoNoS\_reset0 employs a stable amount of backtracks across iterations.
	\item PoNoS reduces the number of backtracks to 1 on average already after the first iteration. Afterwards, this value is first stable for the first 1000-10000 iterations (5-25 epochs) and it then reaches a median of 0 after a transition phase.
\end{itemize}
In the second column of Figure \ref{fig:f_eval}, we report the average difference between two consecutive amount of backtracks in the first 20000 iterations. In particular, this value is higher in the initial phase, while (almost) always 0 in a later stage. Focusing on PoNoS\_reset0 this difference is almost always below 1. Regarding the newly proposed resetting technique \eqref{eq:new_etak}, we can conclude that the value $l_{k-1}$ is a good estimate for $l_{k}$.
\renewcommand{\dir}{f_eval/}
\renewcommand{\imgS}{0.37}
\begin{figure}[!h] 
	\hspace{-3mm}
	\begin{minipage}{0.9\textwidth}
		\renewcommand{\model}{mlp}
		\renewcommand{\modelname}{mlp}
		\subfloat{\includegraphics[width=\imgS\linewidth]{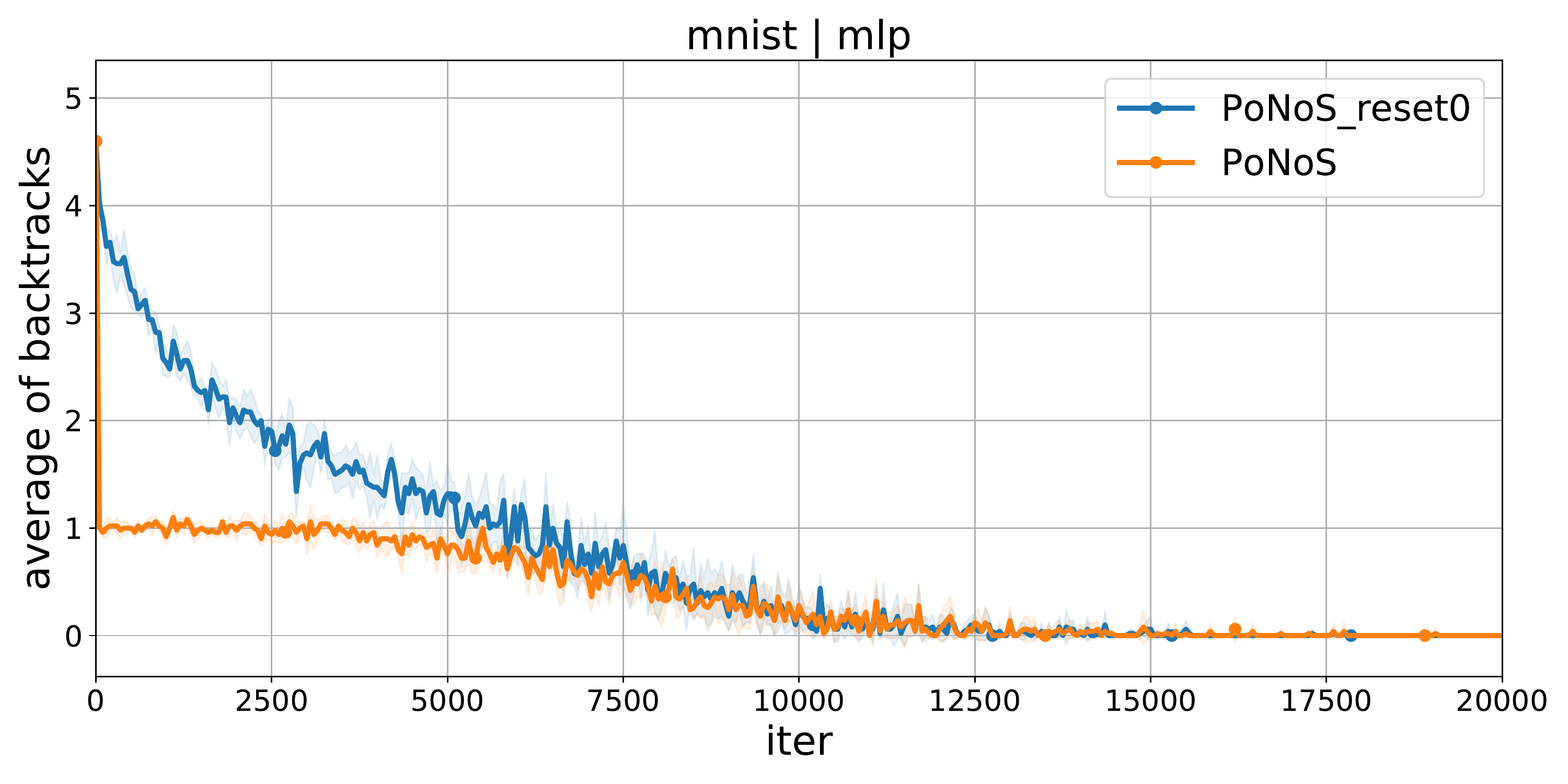}}
		\subfloat{\includegraphics[width=\imgS\linewidth]{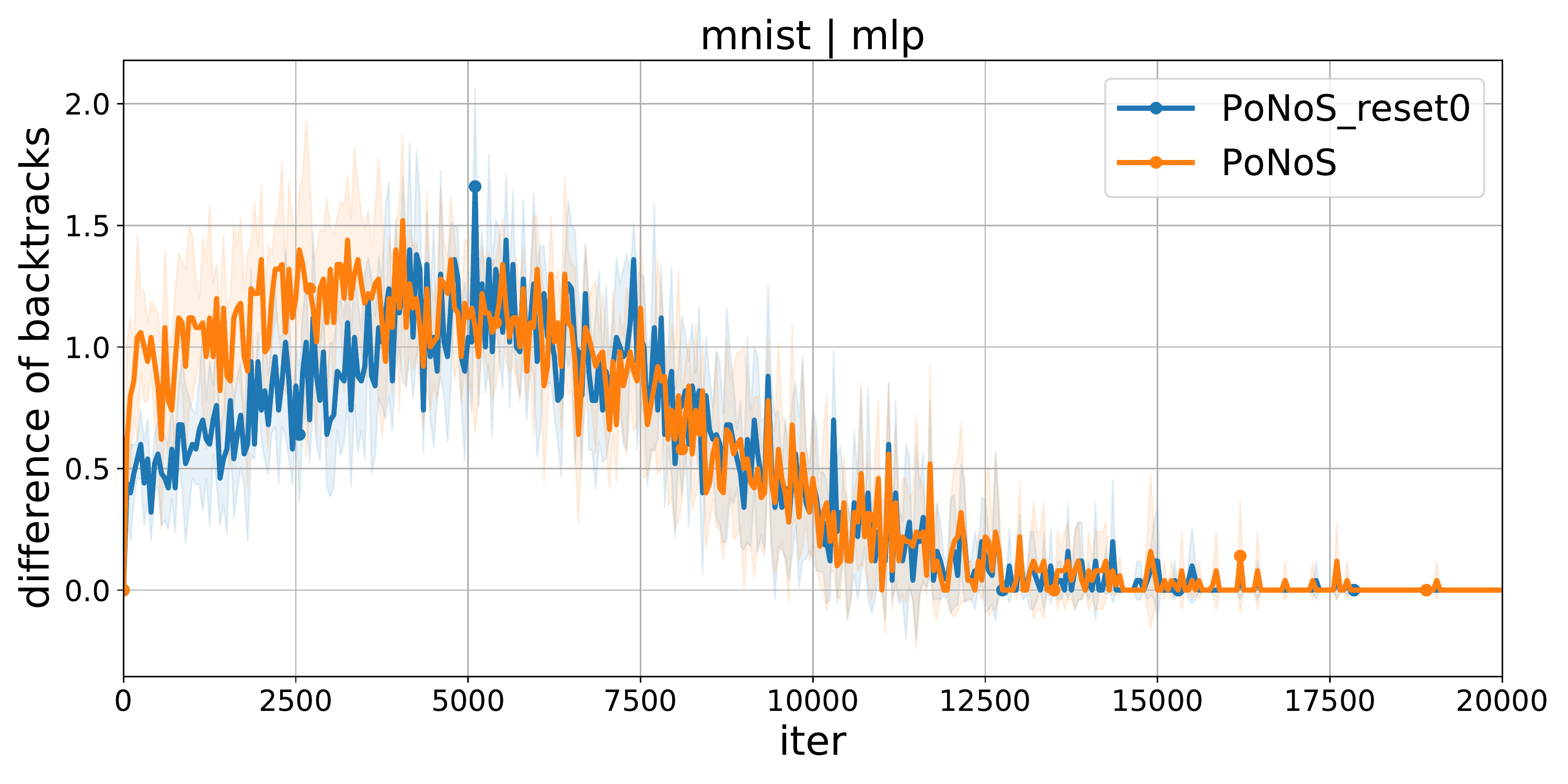}}
		\subfloat{\includegraphics[width=\imgS\linewidth]{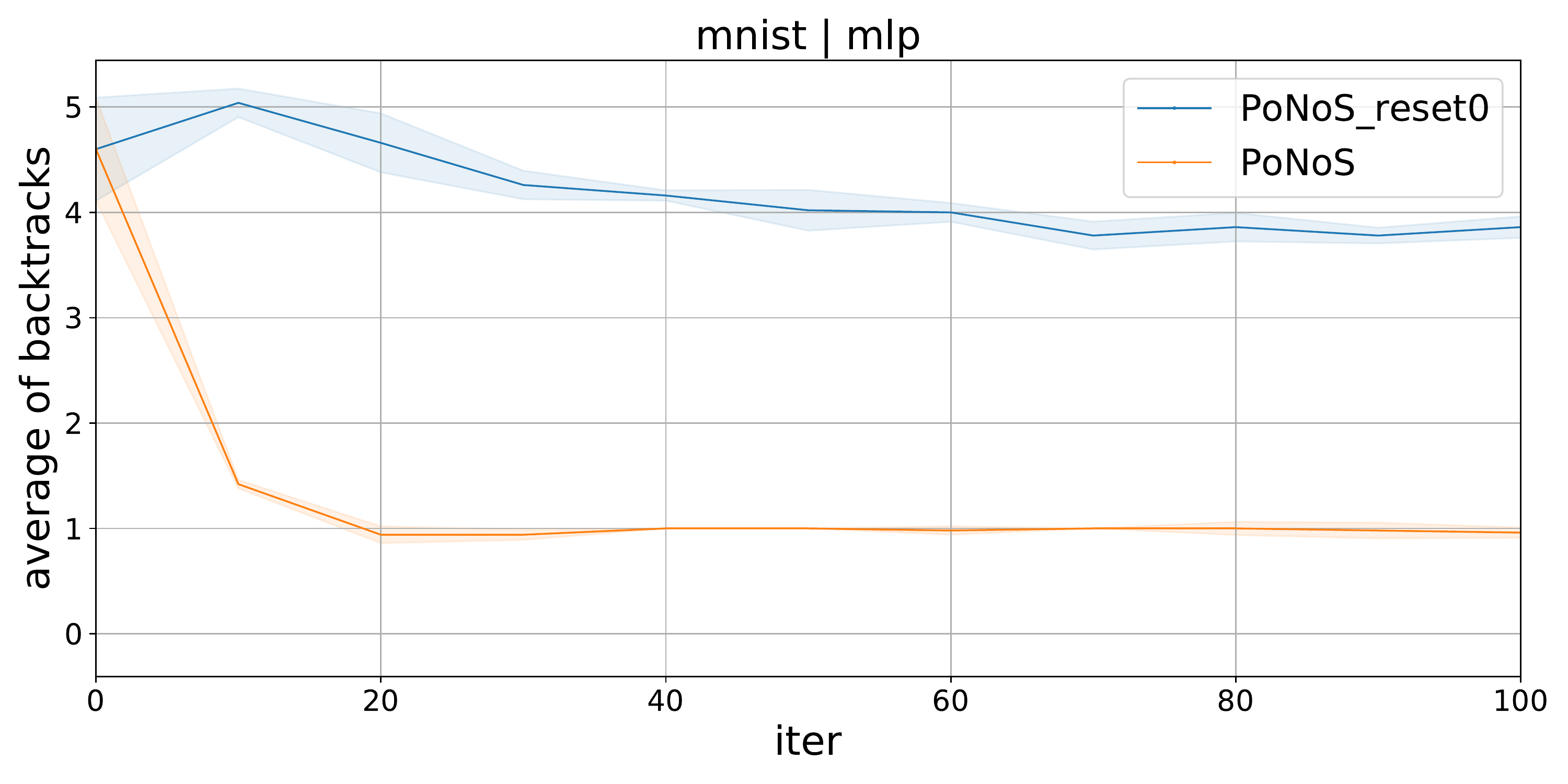}}\\
		\renewcommand{\model}{cifar10_resnet}
		\renewcommand{\modelname}{cifar10\_resnet}
		\subfloat{\includegraphics[width=\imgS\linewidth]{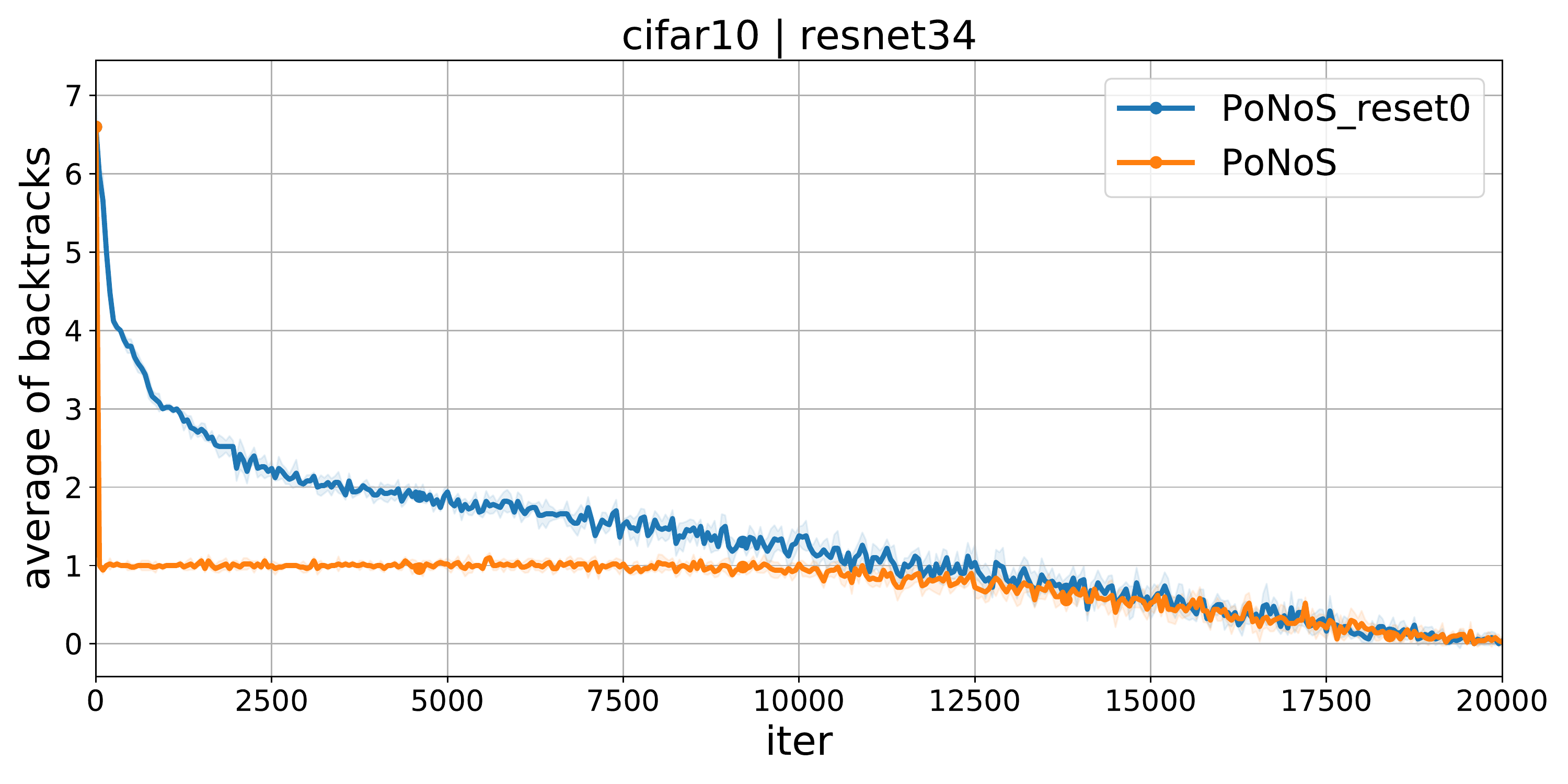}}
		\subfloat{\includegraphics[width=\imgS\linewidth]{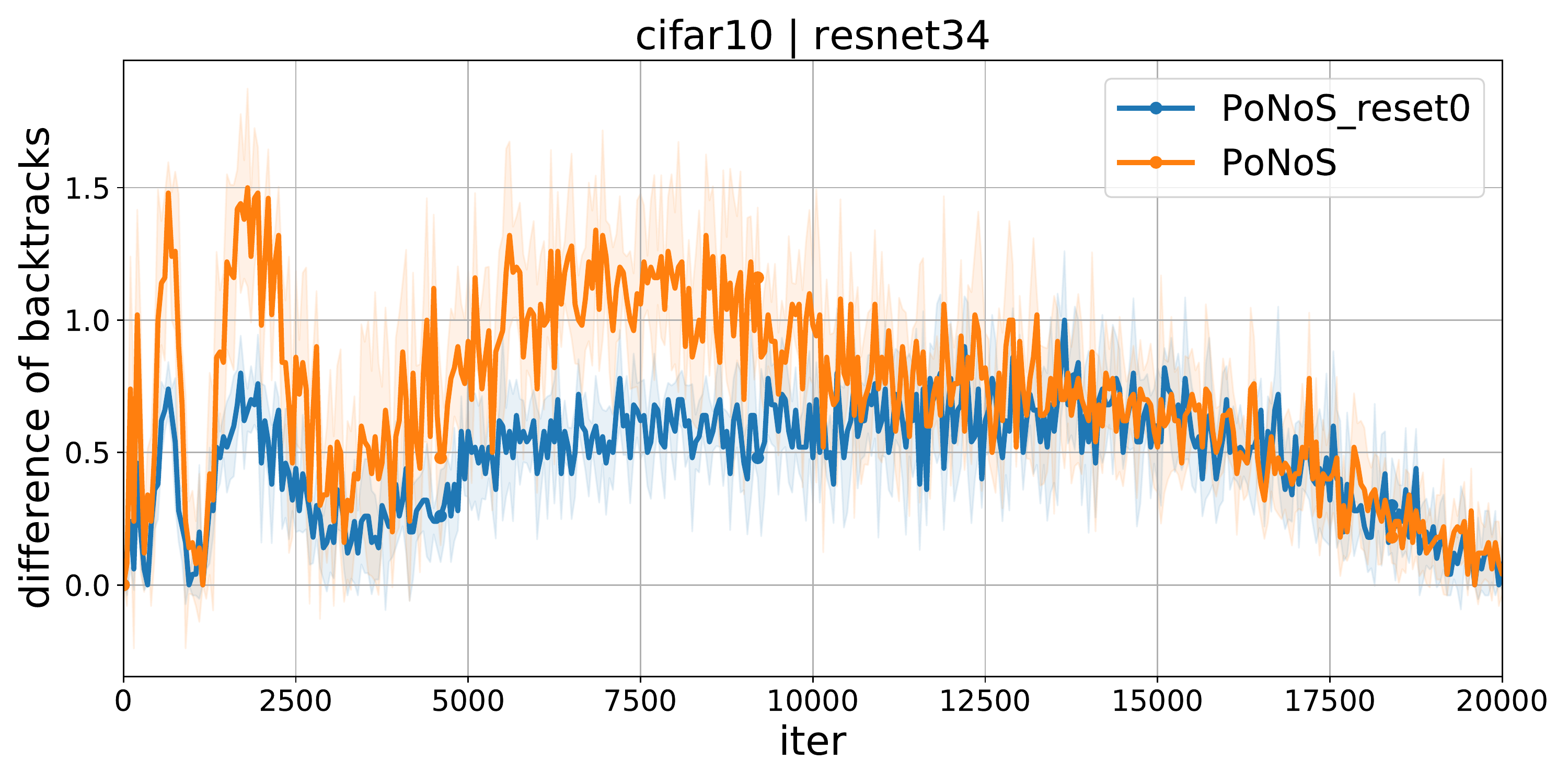}}
		\subfloat{\includegraphics[width=\imgS\linewidth]{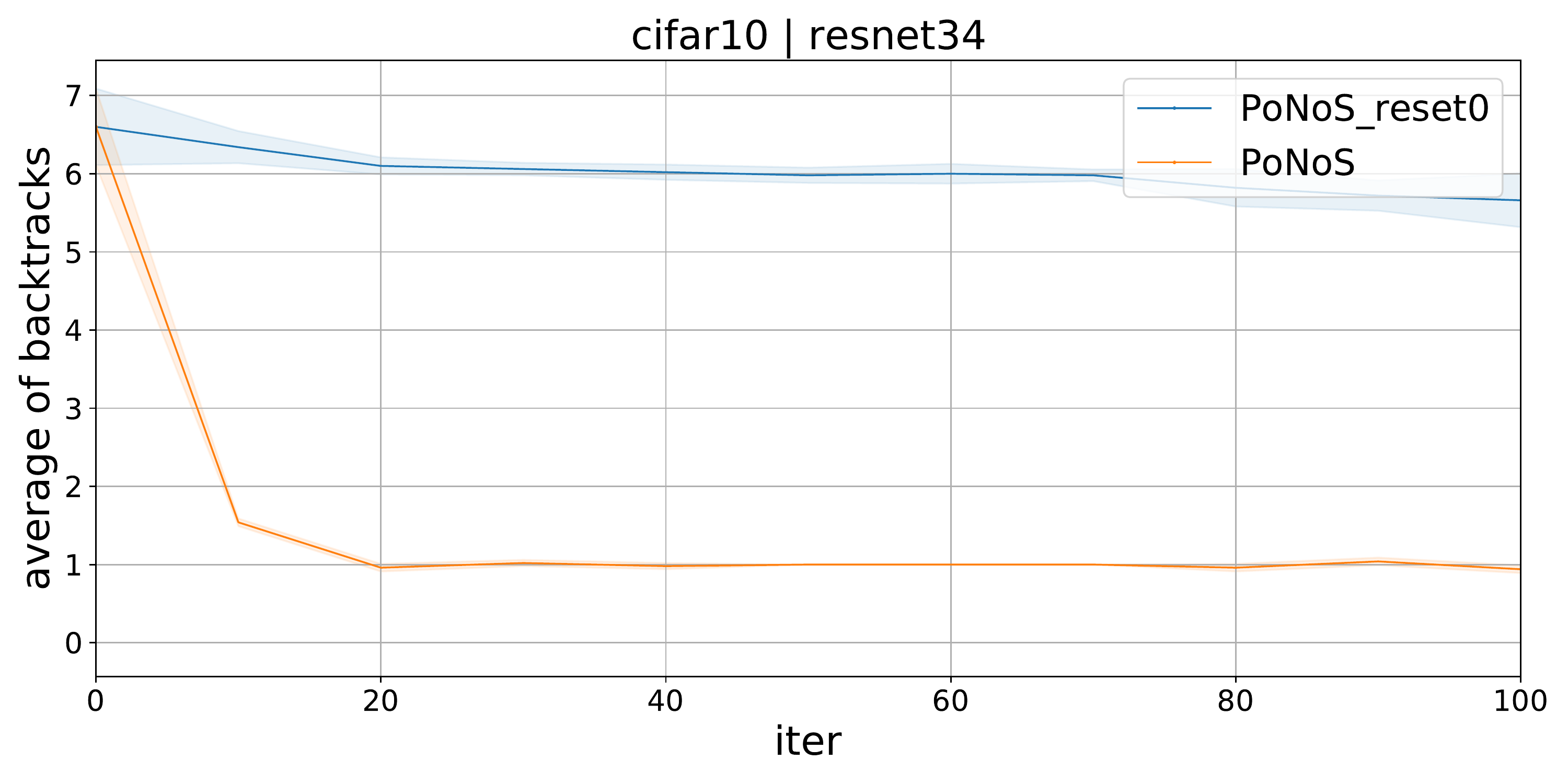}}\\
		\renewcommand{\model}{cifar10_densenet}
		\renewcommand{\modelname}{cifar10\_densenet}
		\subfloat{\includegraphics[width=\imgS\linewidth]{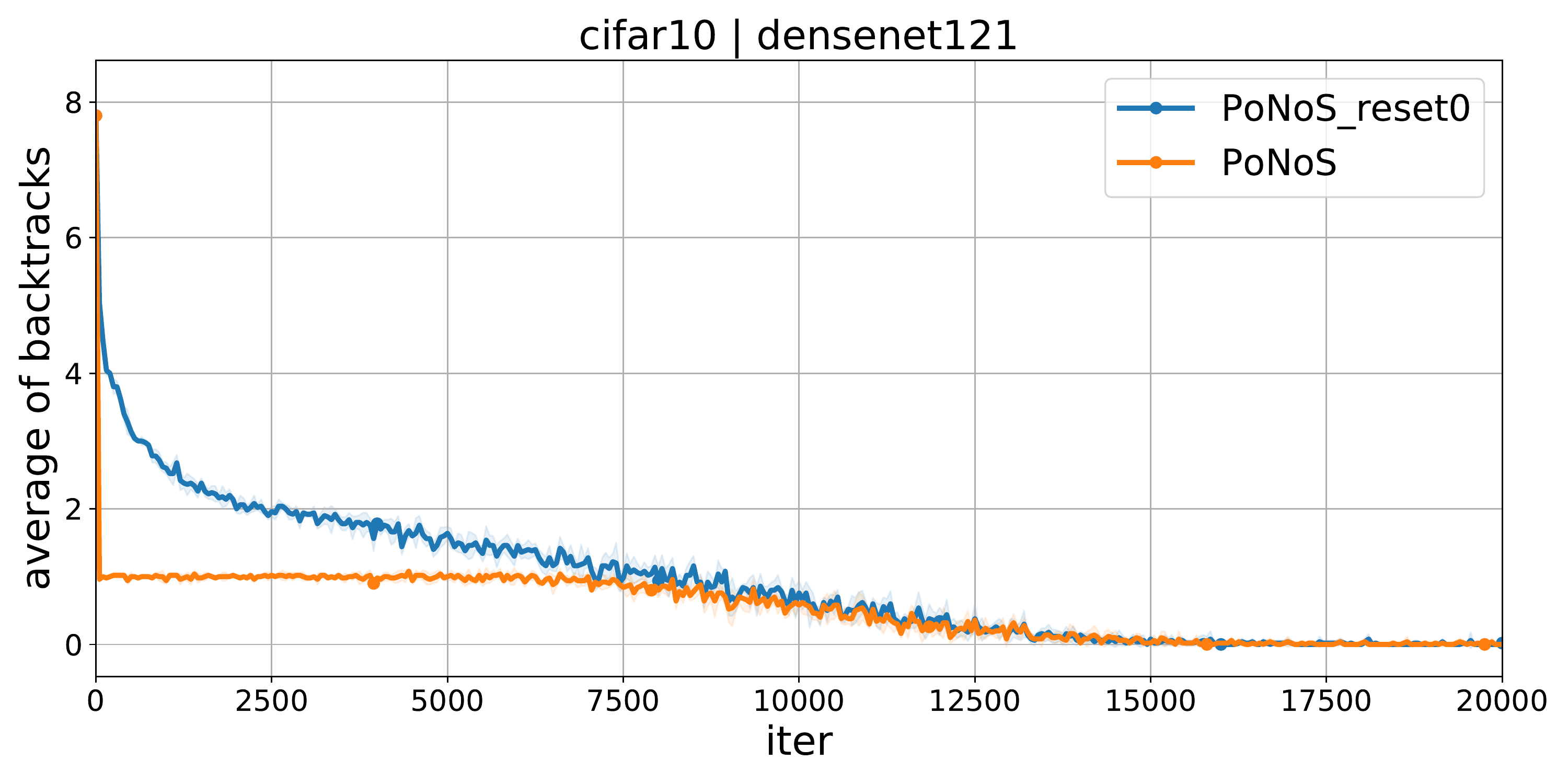}}
		\subfloat{\includegraphics[width=\imgS\linewidth]{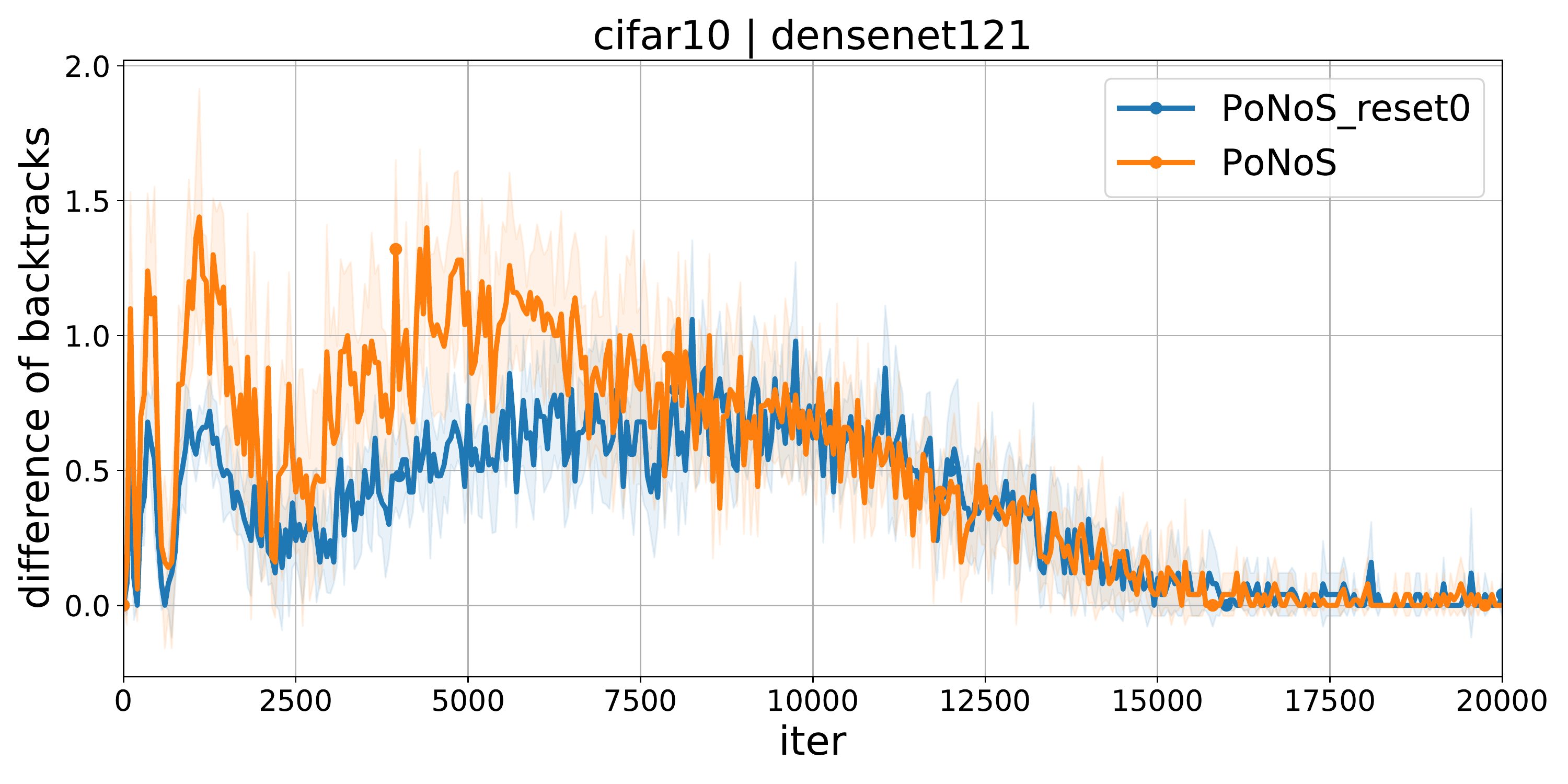}}
		\subfloat{\includegraphics[width=\imgS\linewidth]{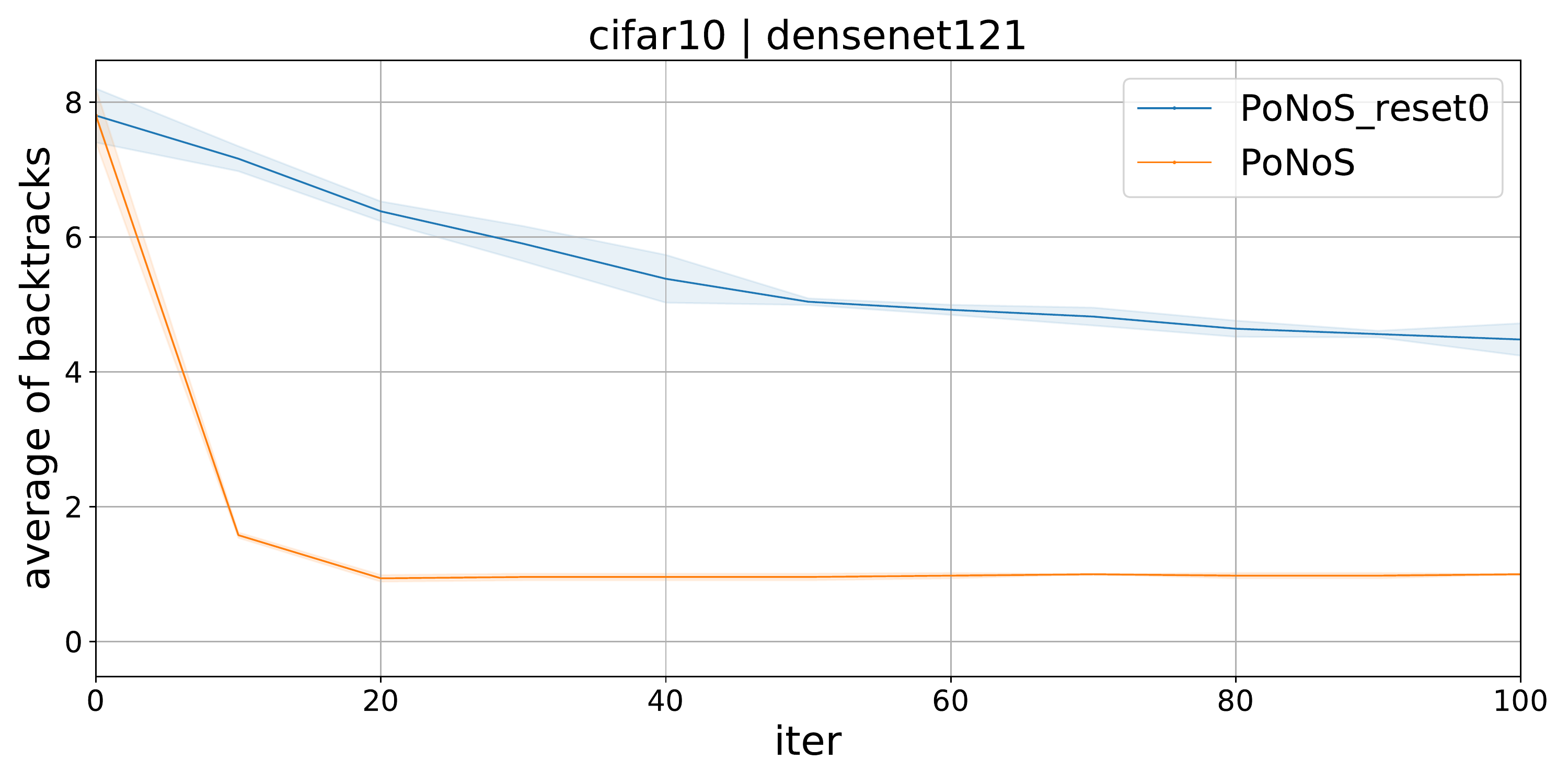}}\\
		\renewcommand{\model}{cifar100_resnet}
		\renewcommand{\modelname}{cifar100\_resnet}
		\subfloat{\includegraphics[width=\imgS\linewidth]{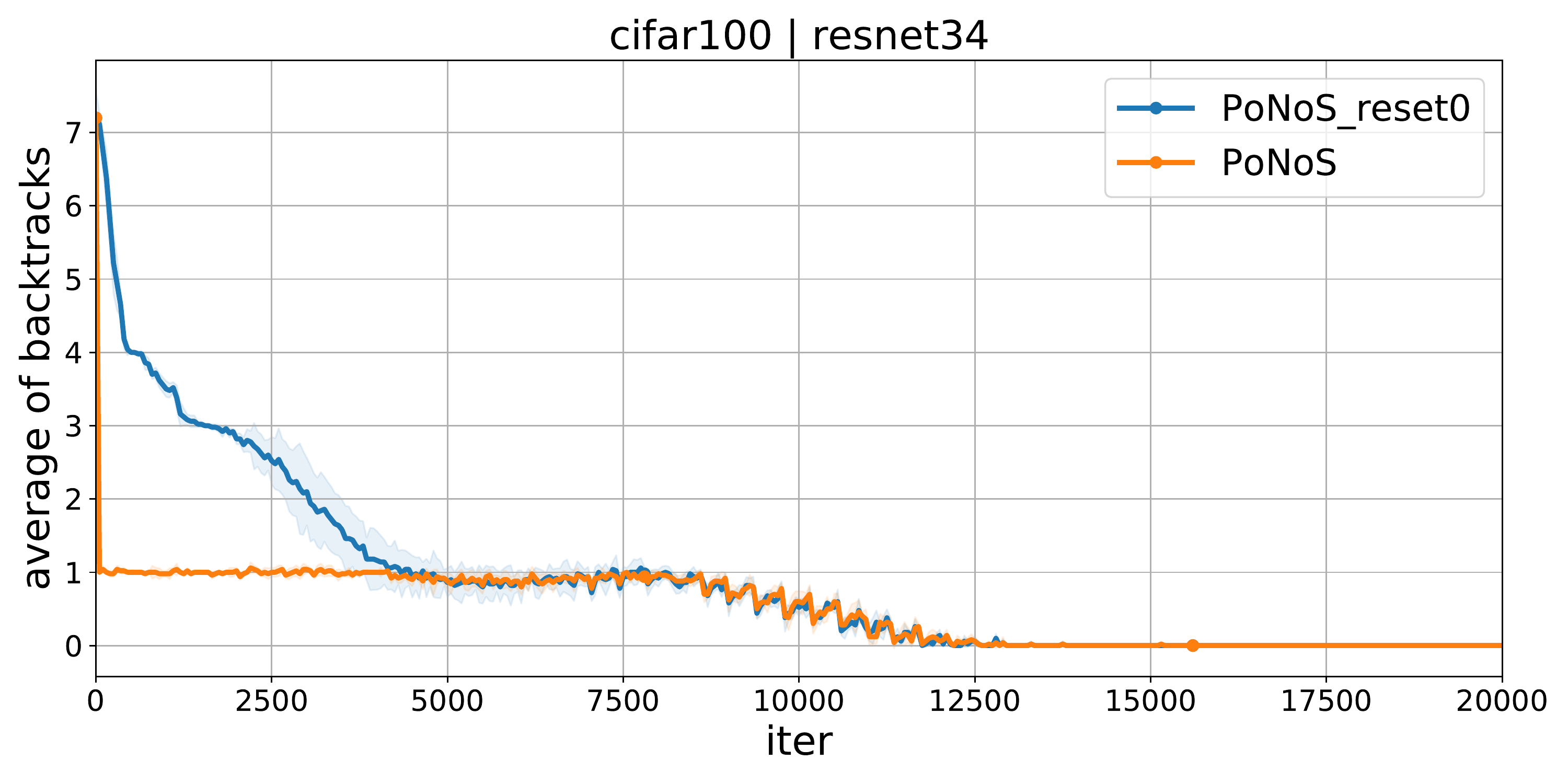}}
		\subfloat{\includegraphics[width=\imgS\linewidth]{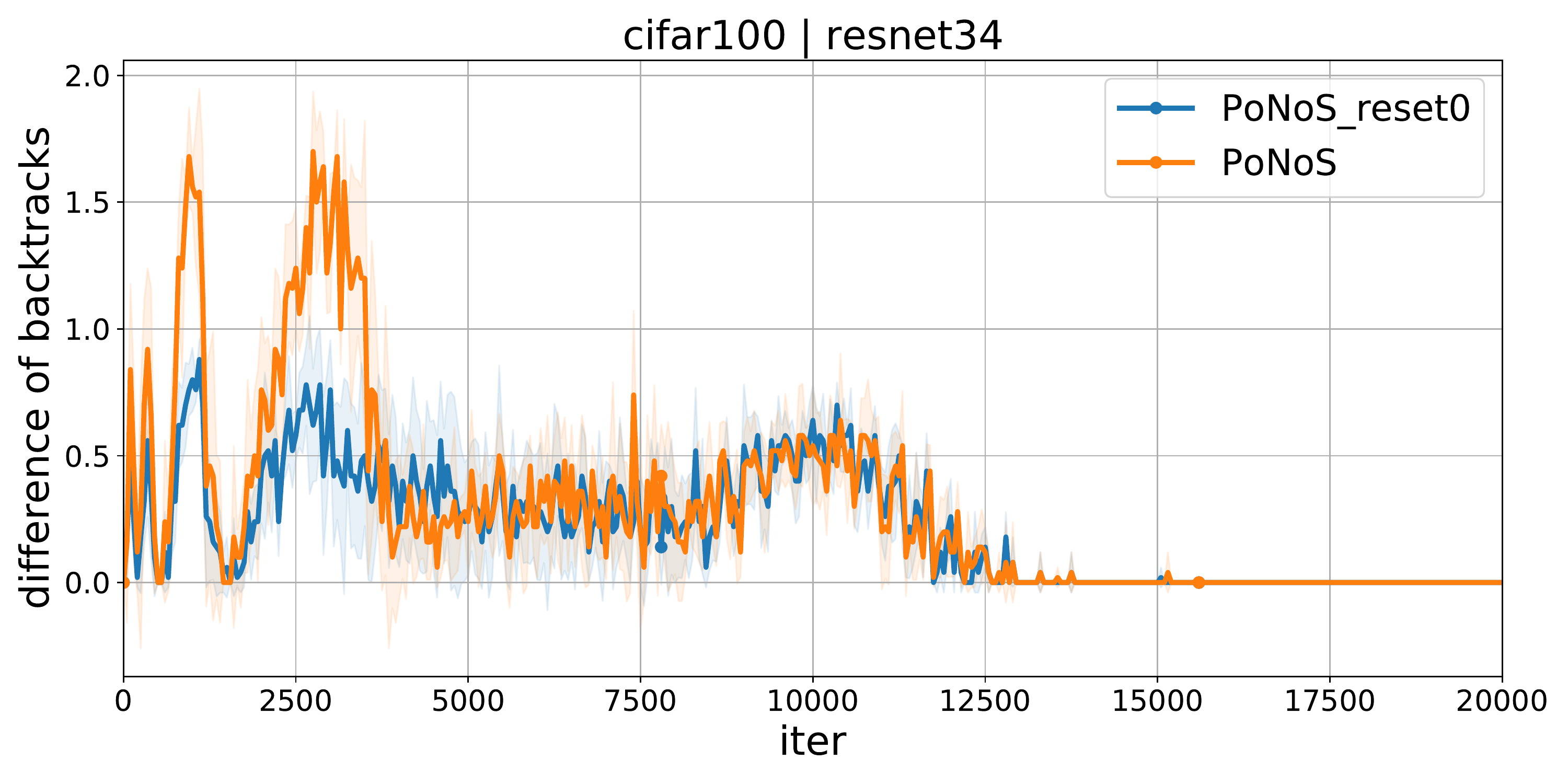}}
		\subfloat{\includegraphics[width=\imgS\linewidth]{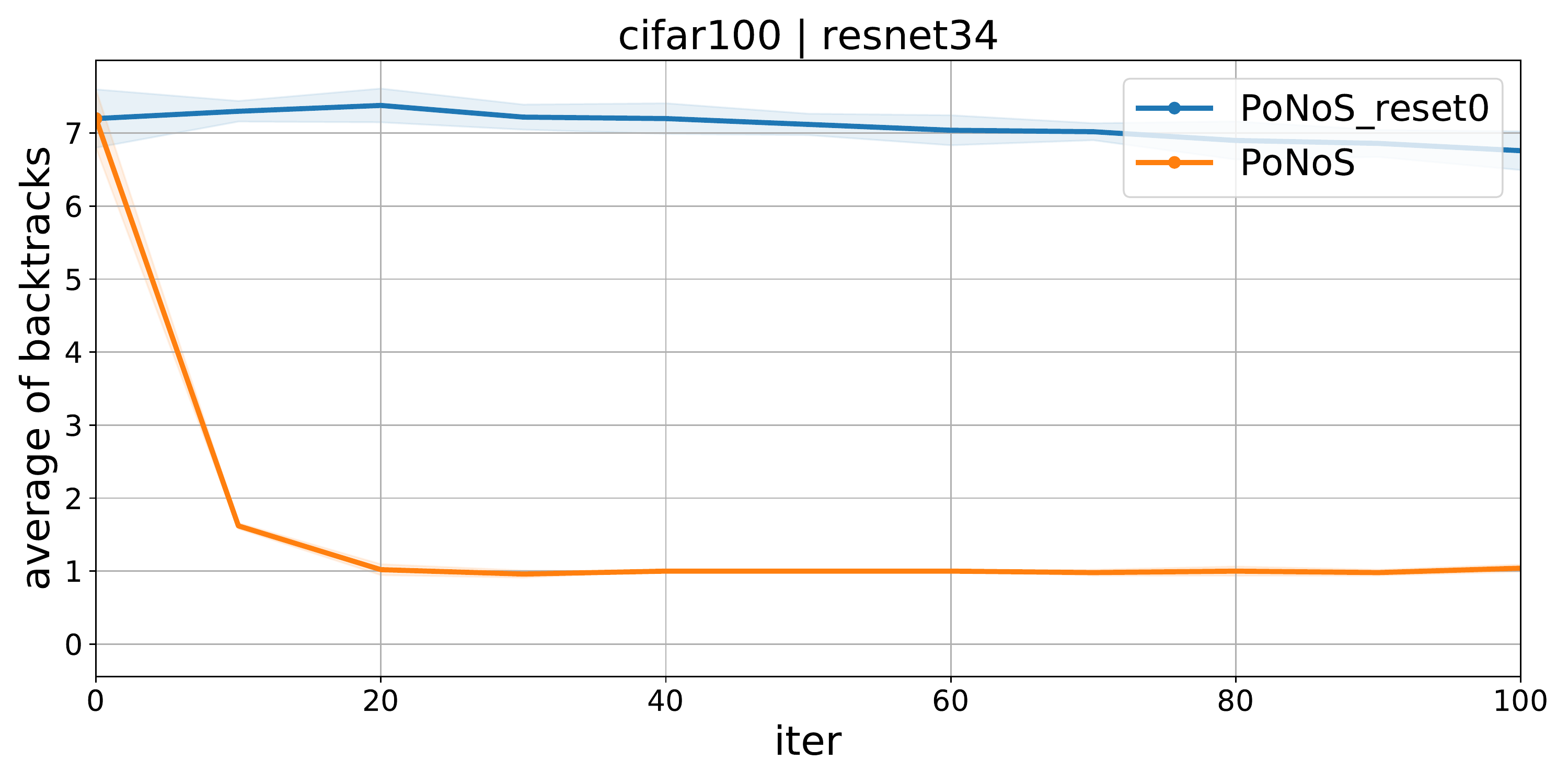}}\\
		\renewcommand{\model}{cifar100_densenet}
		\renewcommand{\modelname}{cifar100\_densenet}
		\subfloat{\includegraphics[width=\imgS\linewidth]{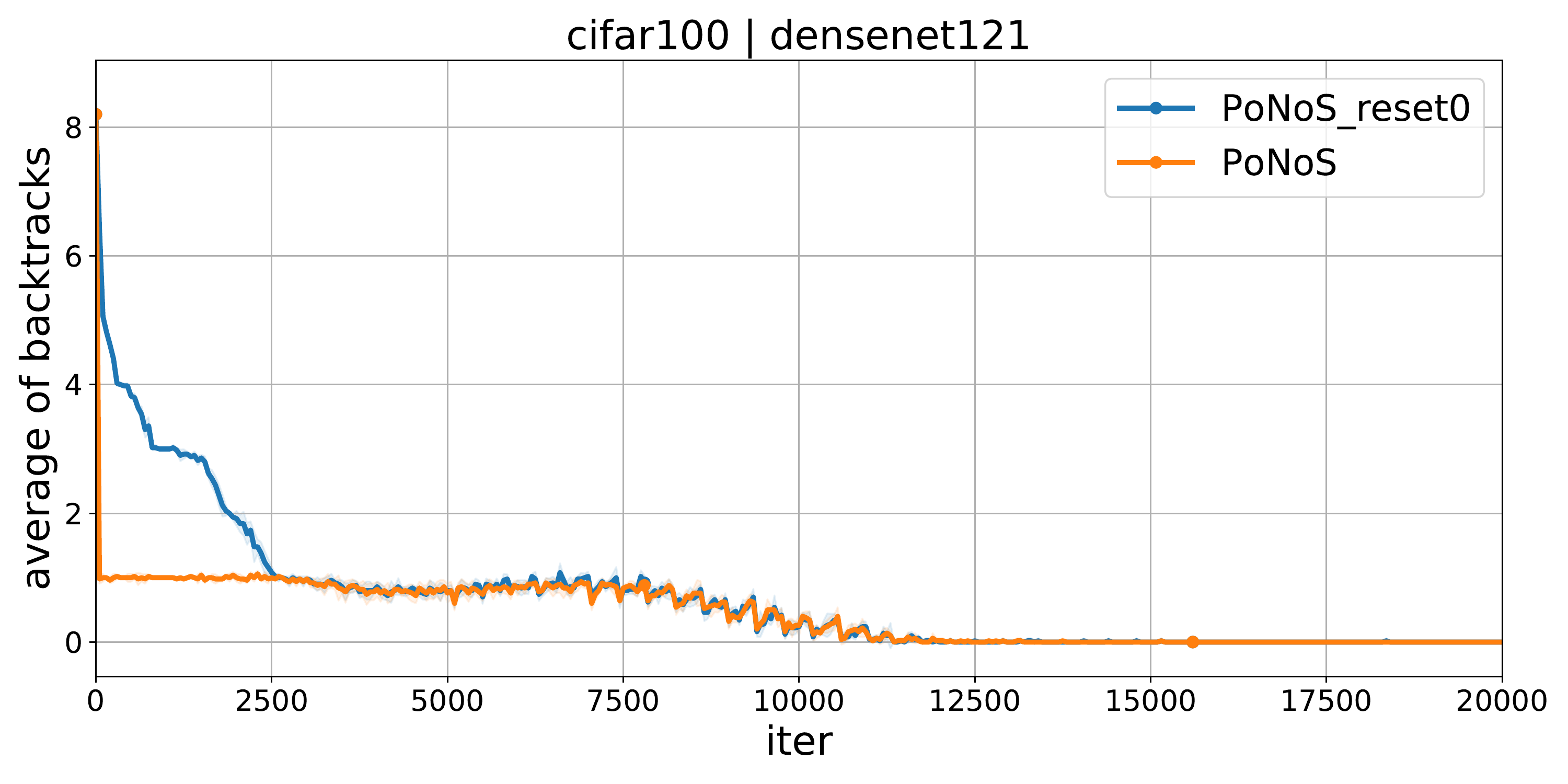}}
		\subfloat{\includegraphics[width=\imgS\linewidth]{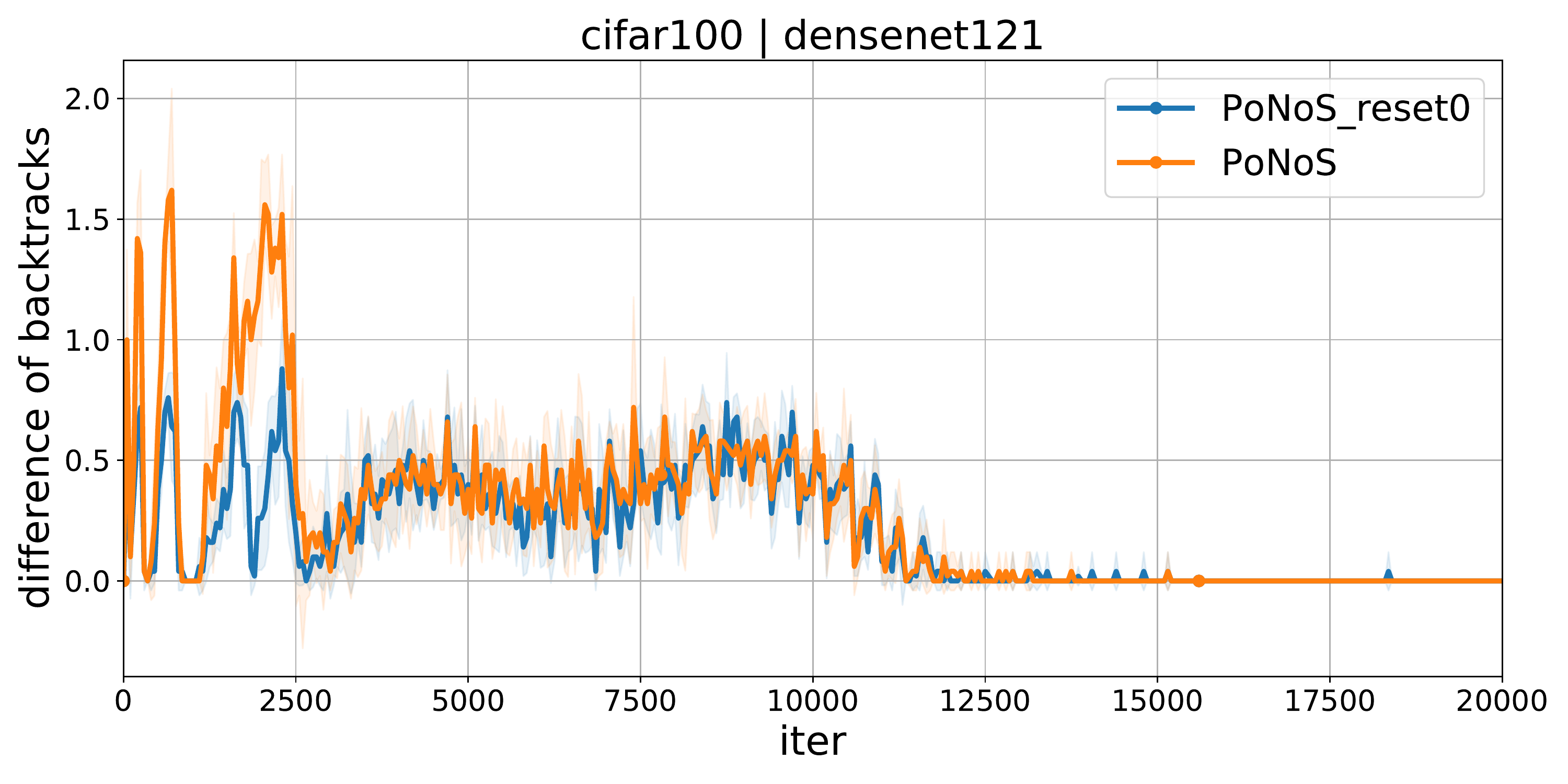}}
		\subfloat{\includegraphics[width=\imgS\linewidth]{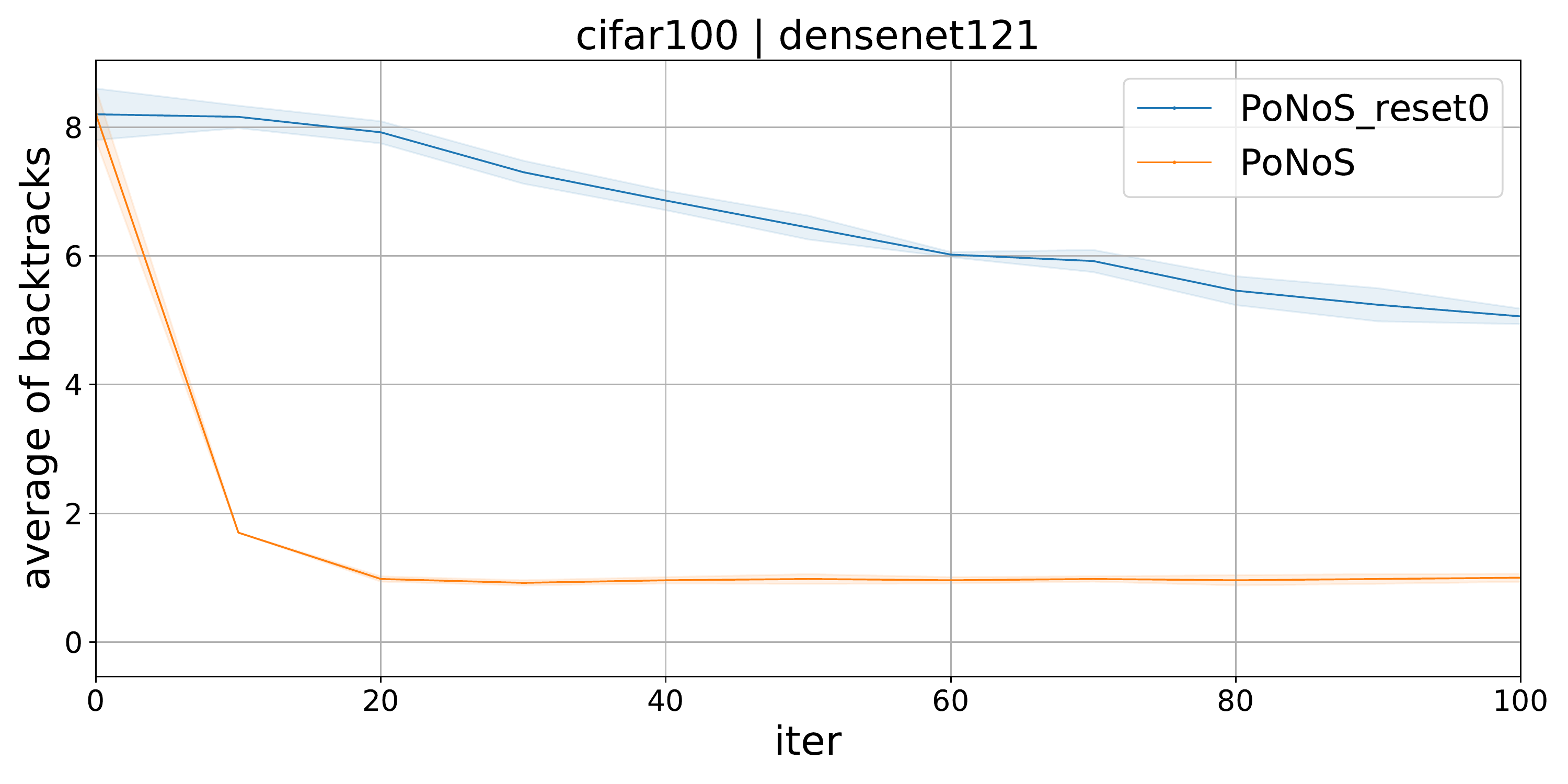}}\\
		\renewcommand{\model}{fashion_effb1}
		\renewcommand{\modelname}{fashion\_effb1}
		\subfloat{\includegraphics[width=\imgS\linewidth]{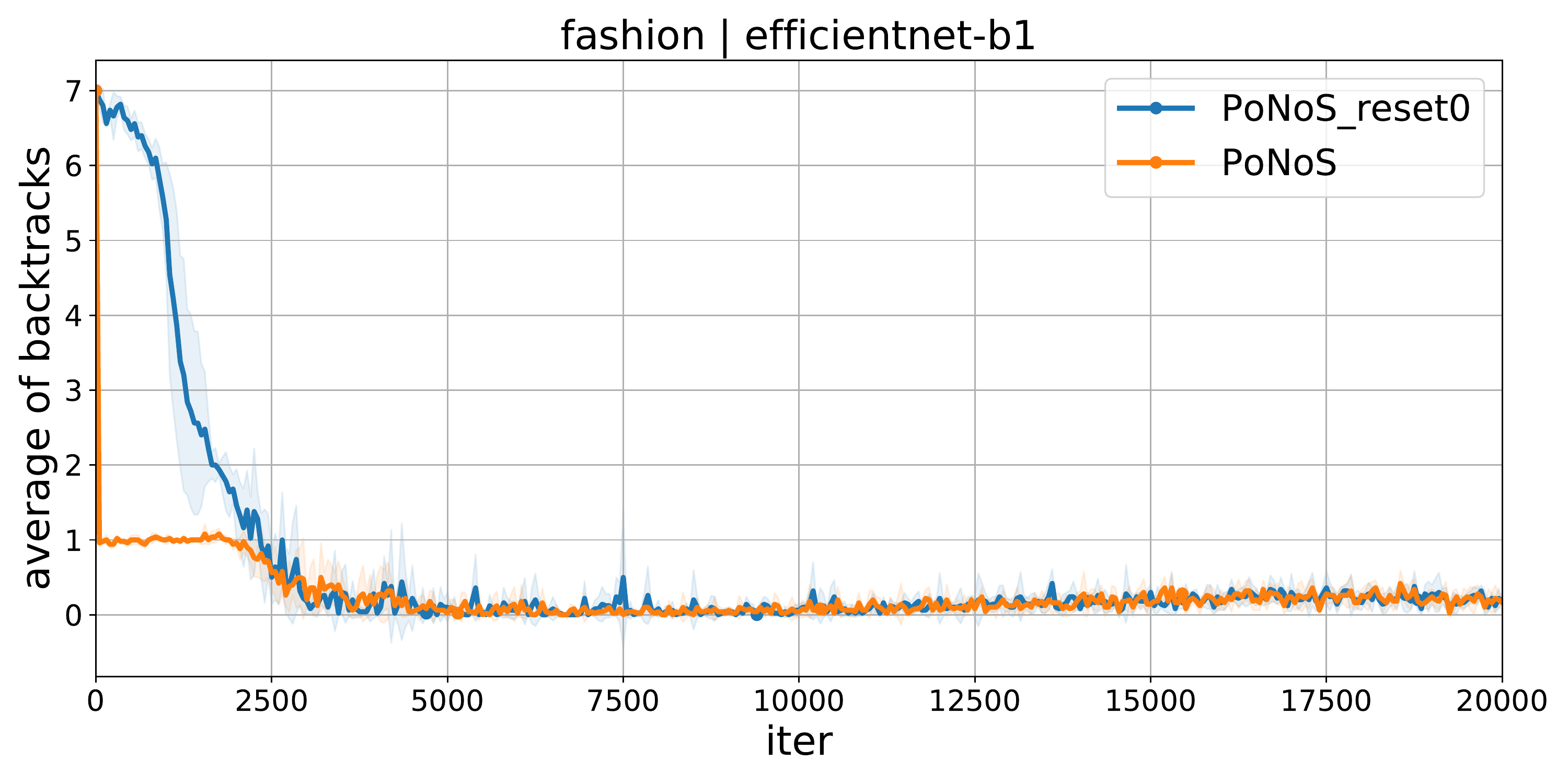}}
		\subfloat{\includegraphics[width=\imgS\linewidth]{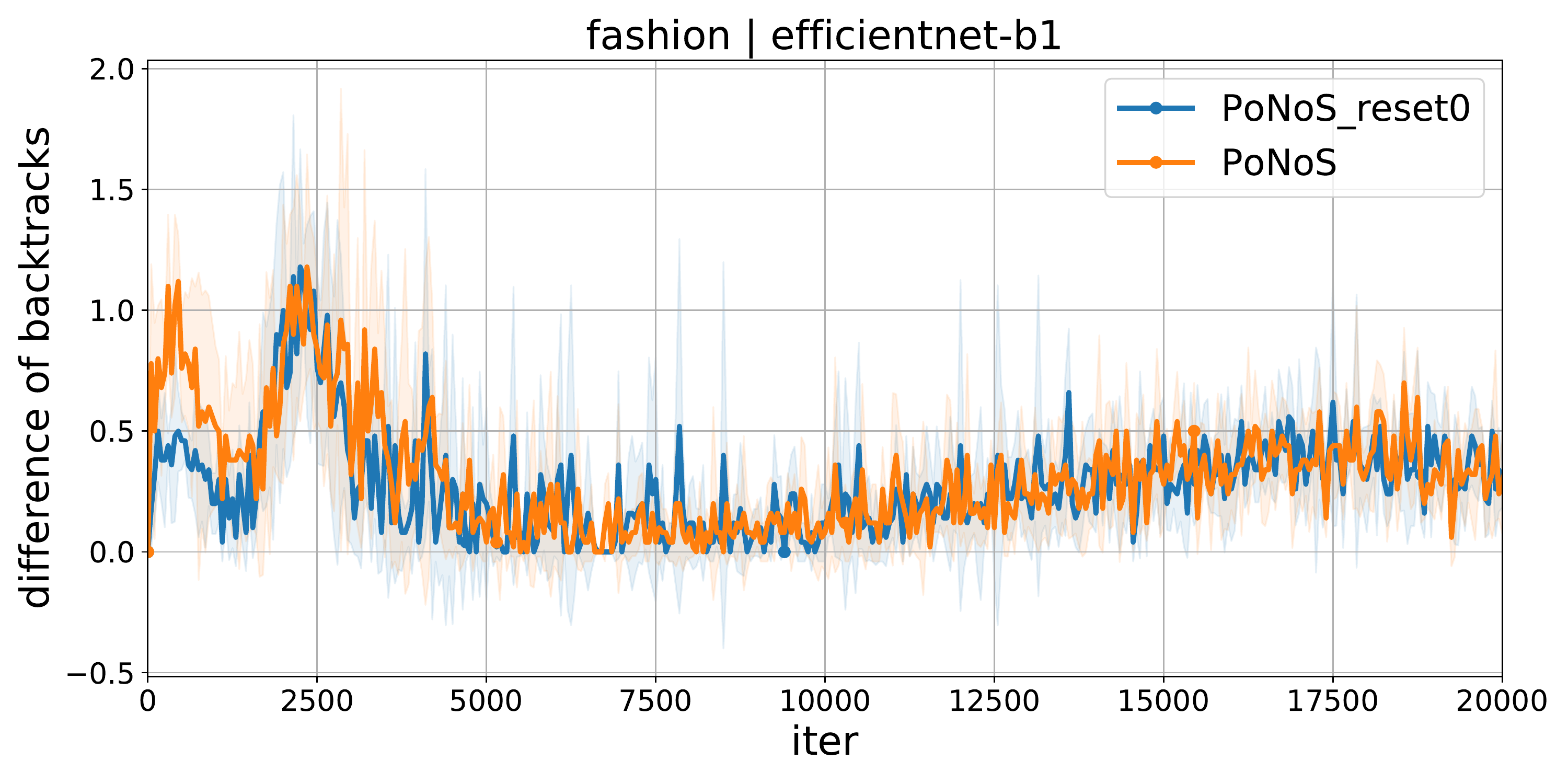}}
		\subfloat{\includegraphics[width=\imgS\linewidth]{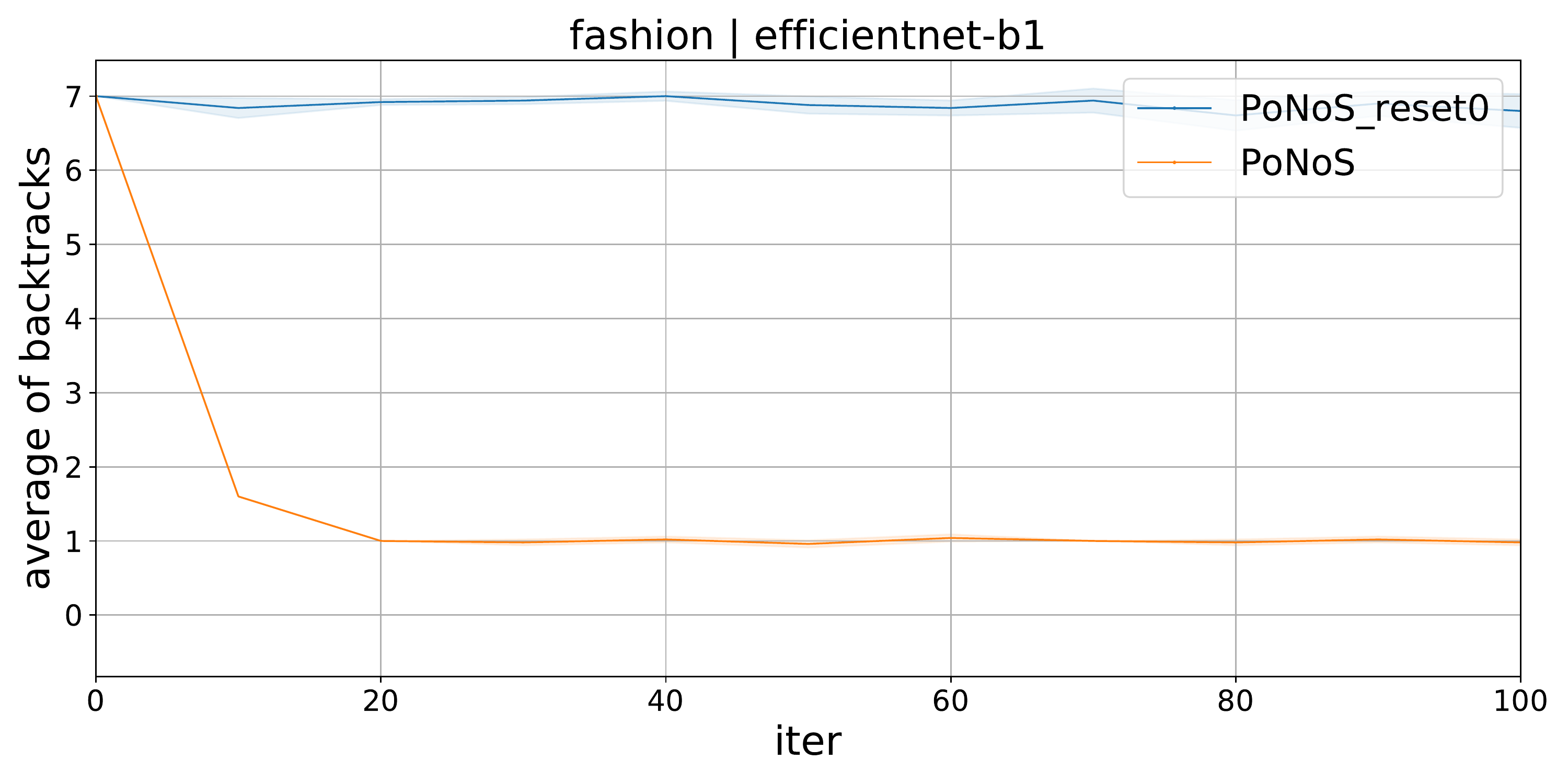}}\\
		\renewcommand{\model}{svhn_wrn}
		\renewcommand{\modelname}{svhn\_wrn}
		\subfloat{\includegraphics[width=\imgS\linewidth]{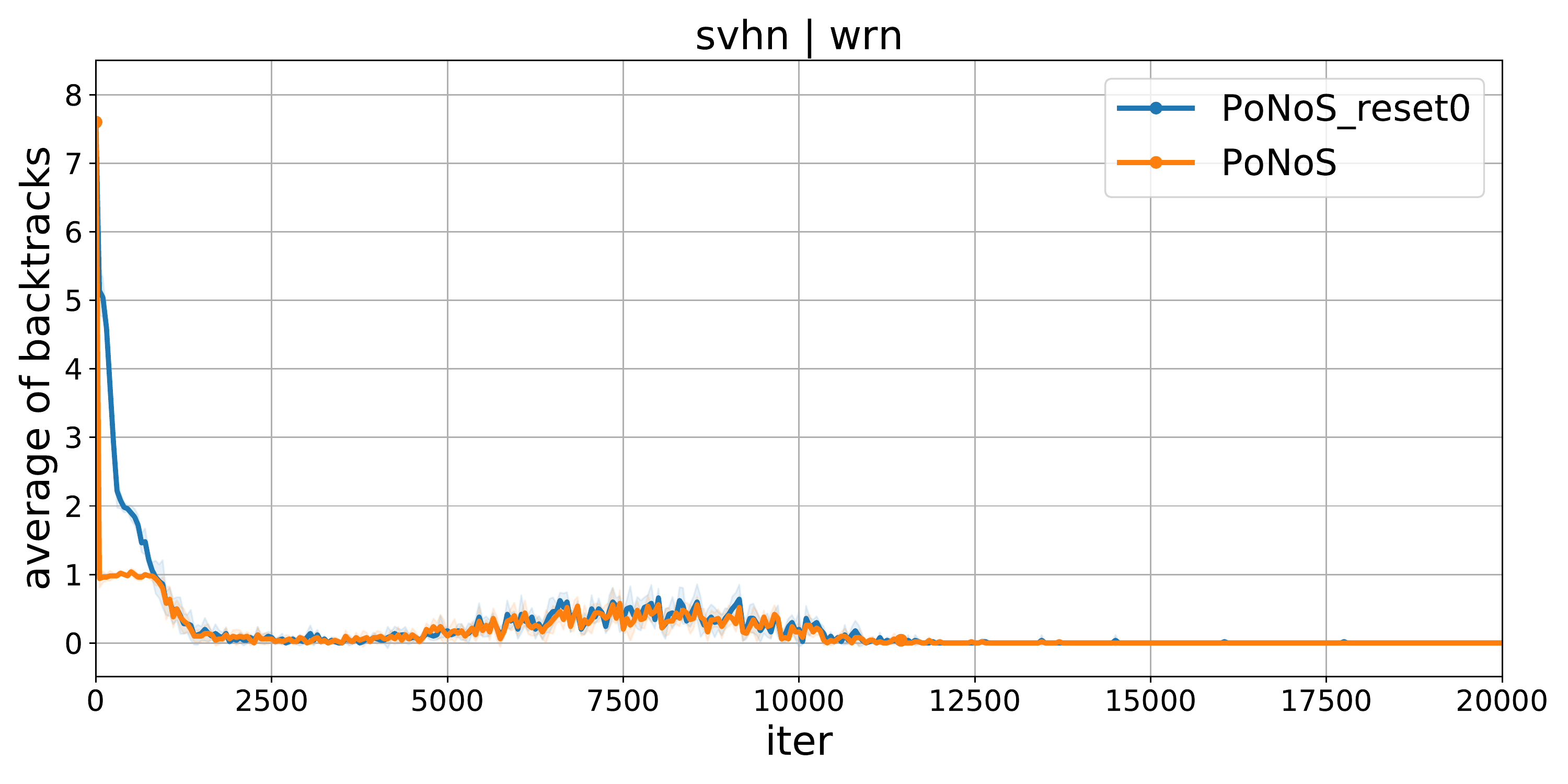}}
		\subfloat{\includegraphics[width=\imgS\linewidth]{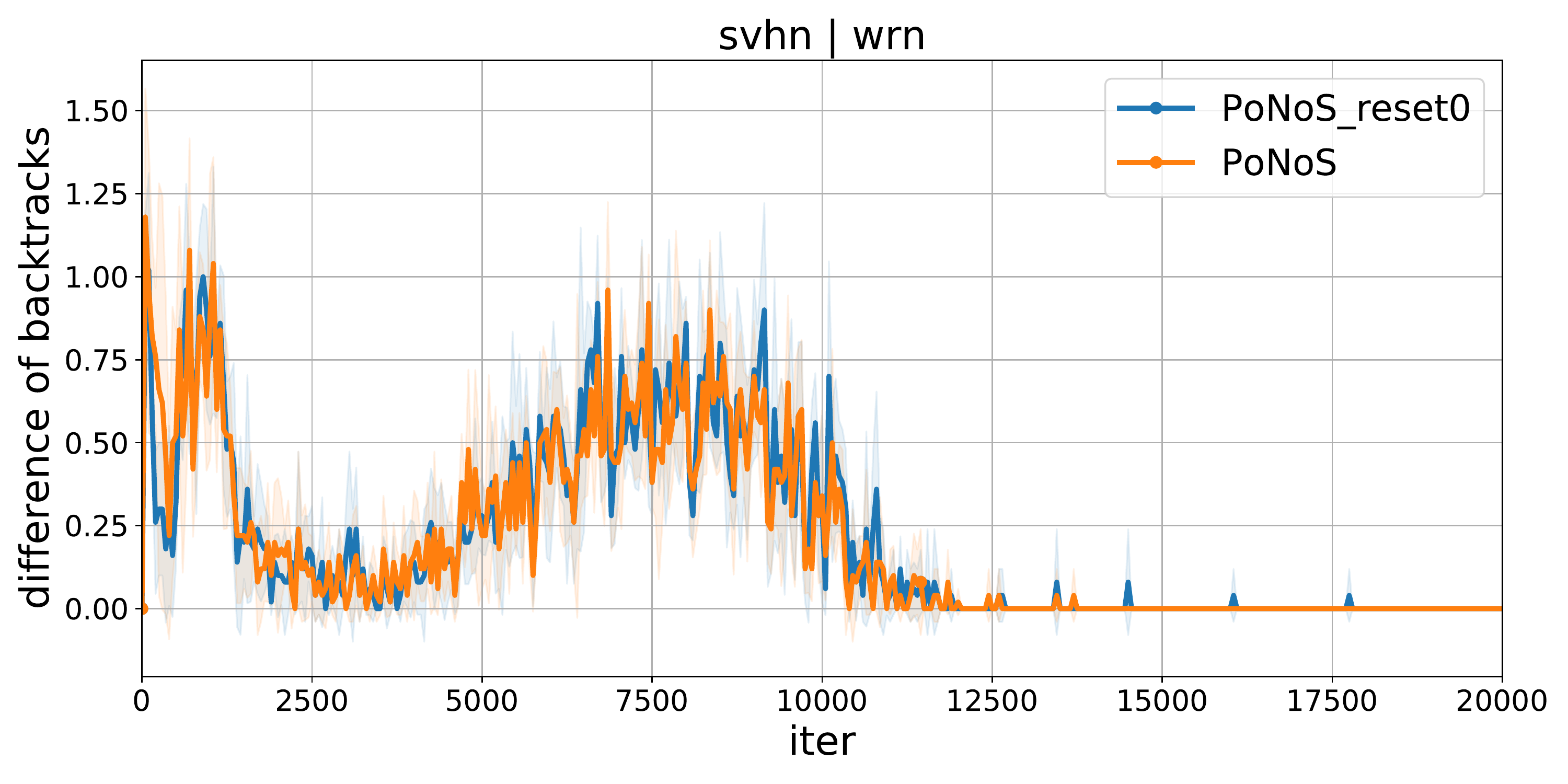}}
		\subfloat{\includegraphics[width=\imgS\linewidth]{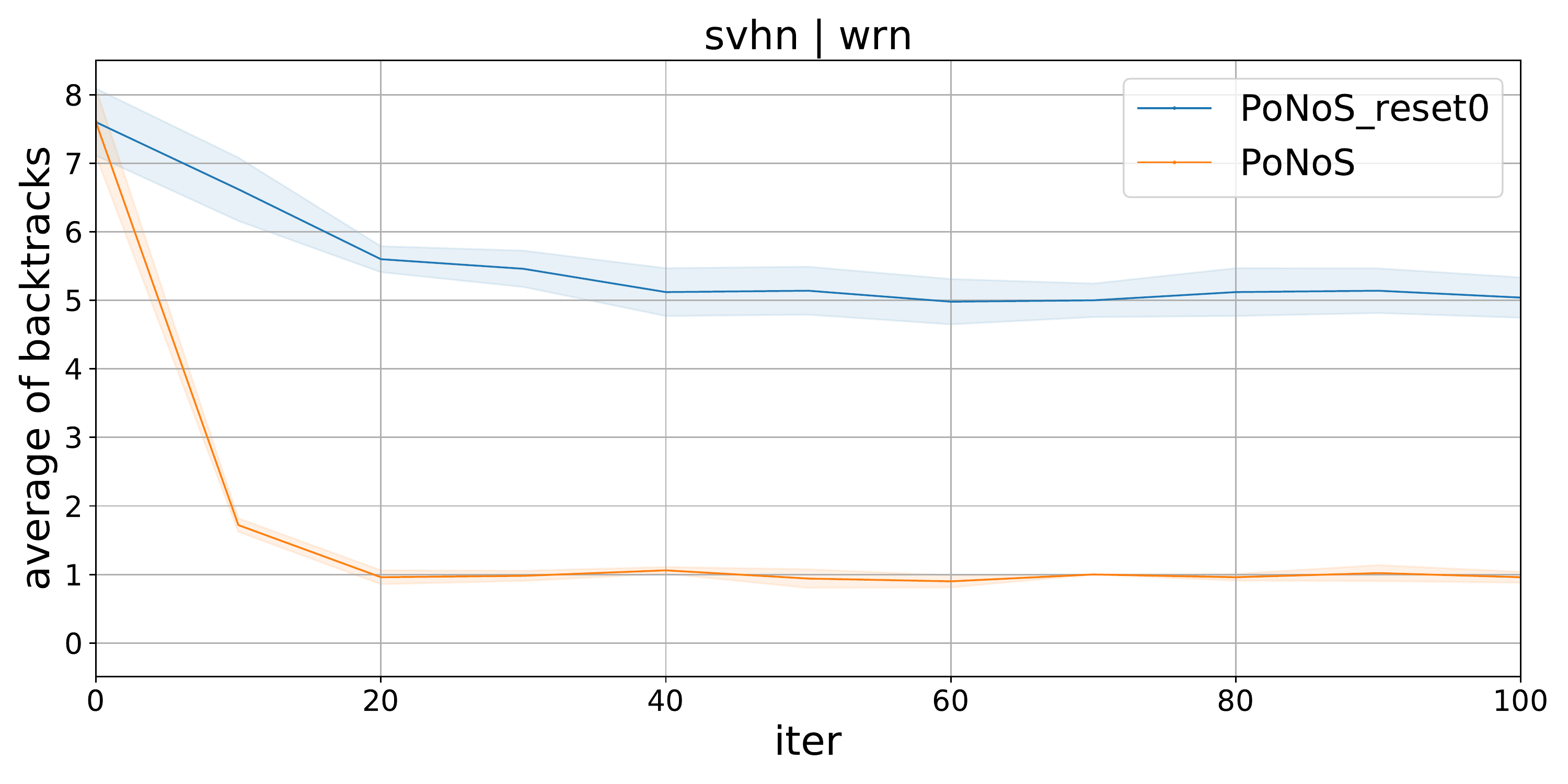}}
		\caption{Comparison on the number of backtracks between the proposed method with (PoNoS) or without the new resetting technique (PoNoS\_reset0). Each row focus on a dataset/model combination. First column: average number of backtracks in the first 200000 iterations. Second column: average difference between two consecutive amount of backtracks in the first 200000 iterations. Third column: average number of backtracks in the first 1000 iterations.}\label{fig:f_eval}
	\end{minipage}
\end{figure}
	\subsection{Study on the Choice of $\etamax$}\label{sec:supp_etamax}
	In this subsection, we report an ablation study on $\etamax$ for PoNoS, SLS and SPS. In Figure \ref{fig:etamax}, we compare
\begin{itemize}
	\item SPS|10: SPS with $\etamax=10$. This setting corresponds to SPS.
	\item SPS|100:SPS with $\etamax=100$. 
	\item SLS|10: SLS with $\etamax=10$. This setting corresponds to SLS.
	\item SLS|100: SLS with $\etamax=100$. 
	\item PoNoS|10: PoNoS\_reset0 with $\etamax=10$. This setting corresponds to PoNoS\_reset0.
	\item PoNoS|100: PoNoS\_reset0 with $\etamax=100$. 
\end{itemize}
We report PoNoS\_reset0 (without \eqref{eq:new_etak}) instead of PoNoS because we want to observe the difference in the amount of backtracks.
To directly observe the effect of changing $\etamax$, we report the average initial step size within each epoch in the third column of Figure \ref{fig:etamax}. From Figure \ref{fig:etamax} we can observe:
\begin{itemize}
	\item The value of $\etamax=100$ is never reached by the step sizes of the three algorithms. The step size of PoNoS|100, SLS|100 and SPS|100 can be considered unbounded.
	\item The three algorithms are robust to the choice of $\etamax$. In fact, the train loss and test accuracy of PoNoS|100, SLS|100 and SPS|100 are similar to those of PoNoS|10, SLS|10 and SPS|10.
	\item In terms of train loss, there are some small differences on \ress, \denses$ $ and \fashion. On these problems, PoNoS|100, SLS|100 and SPS|100 are sometimes slower during the intermediate phase of the optimization. However, the three unbounded algorithms are able to catch up towards the end.
	\item In terms of test accuracy, both SLS|100 and SPS|100 lose more than $2$ points on \denses$ $ w.r.t. SLS|10 and SPS|10, while PoNoS|100 and SLS|100 lose more than $2$ points on \denses$ $.
	\item PoNoS|100 performs overall better than SLS|100 and SPS|100.
	\item The amount backtracks of PoNoS|100 is sometimes remarkably larger than for PoNoS|10. However, the use of \eqref{eq:new_etak} would reduce this number to (almost) always zero.
\end{itemize}

\renewcommand{\dir}{etamax/}
\renewcommand{\imgS}{0.30}
\begin{figure}[!h] 
	\hspace{-9mm}
	\begin{minipage}{0.9\textwidth}
		\renewcommand{\model}{mlp}
		\renewcommand{\modelname}{mlp}
		\subfloat{\includegraphics[width=\imgS\linewidth]{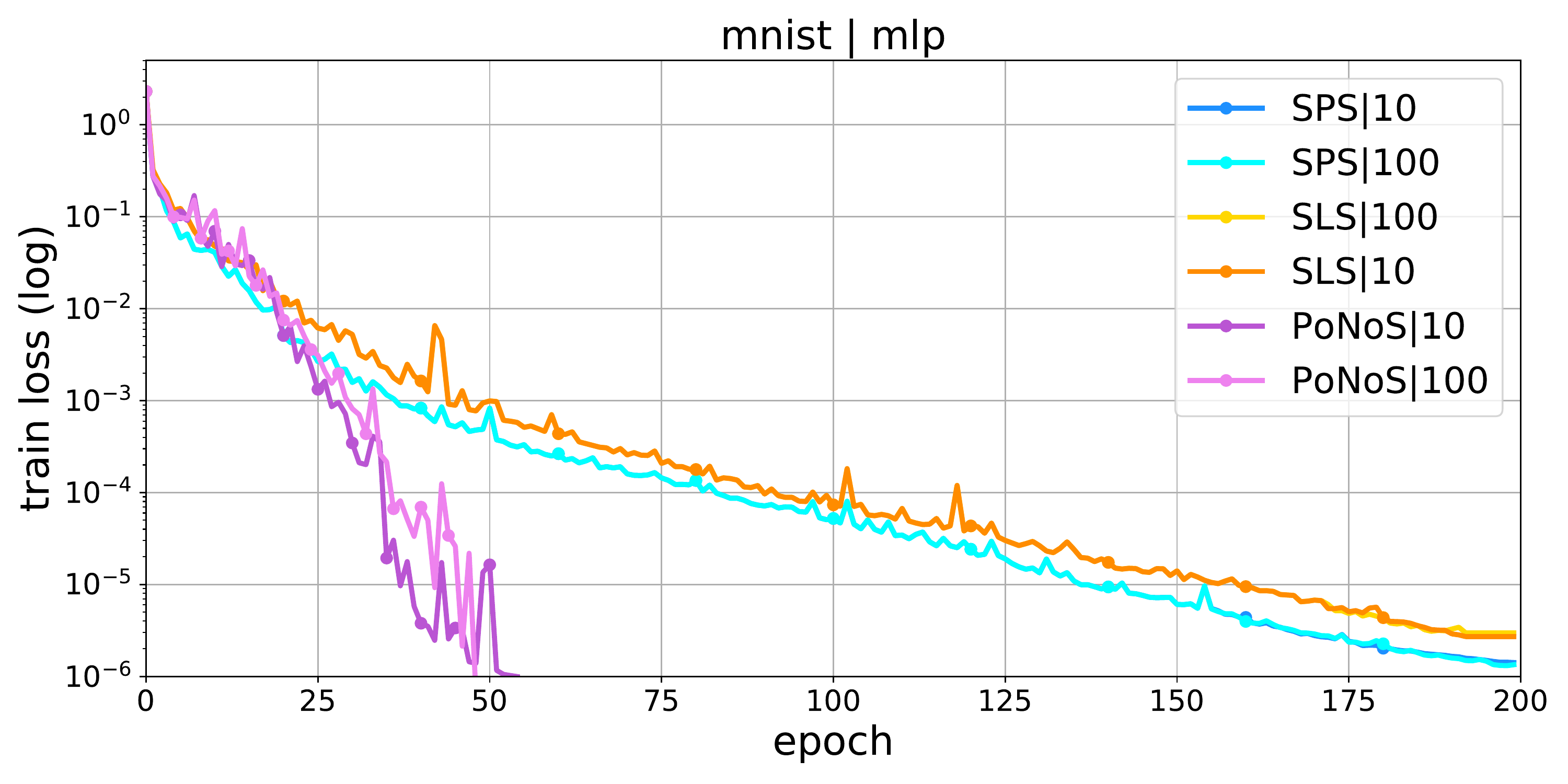}}
		\subfloat{\includegraphics[width=\imgS\linewidth]{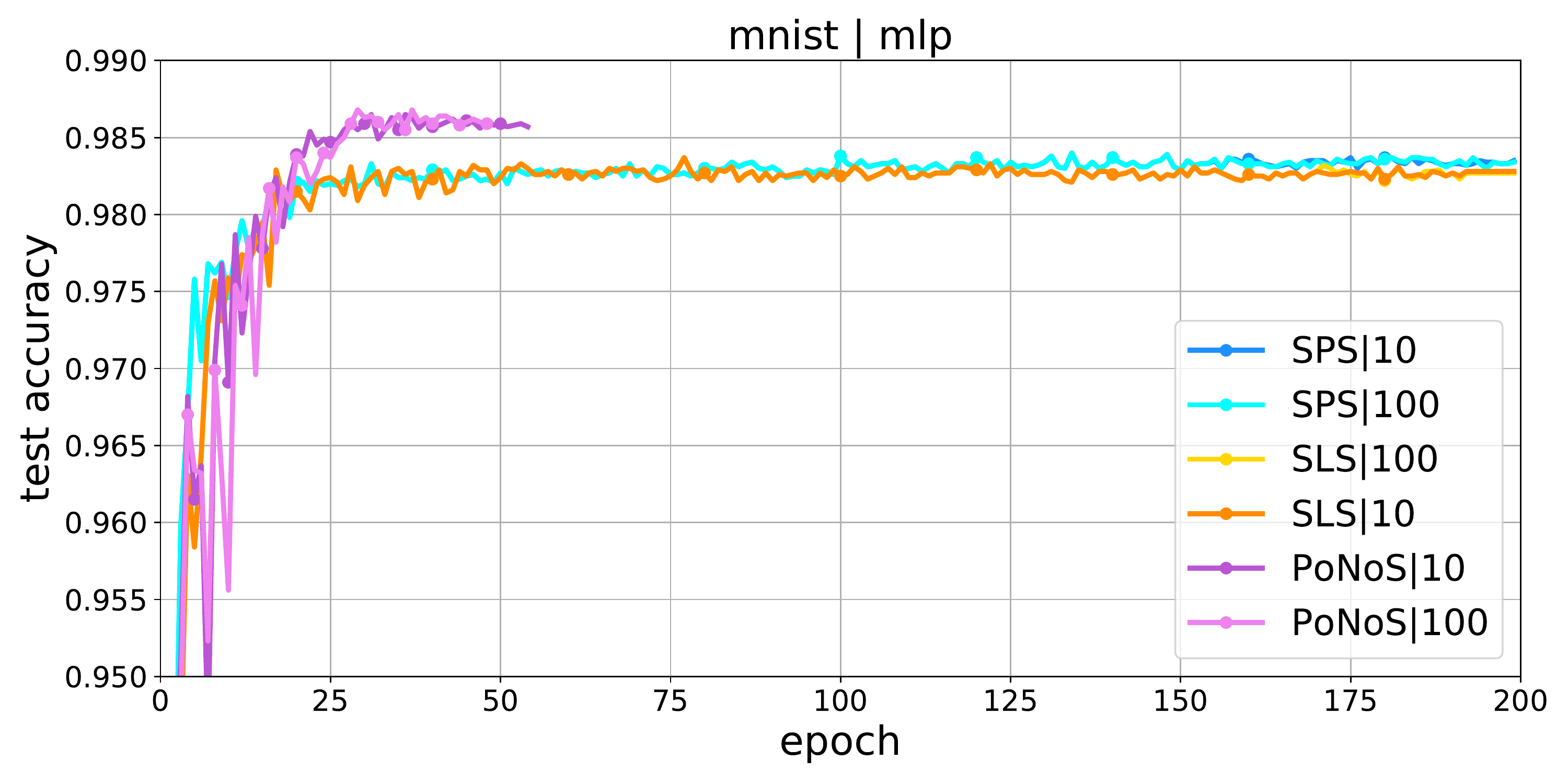}}
		\subfloat{\includegraphics[width=\imgS\linewidth]{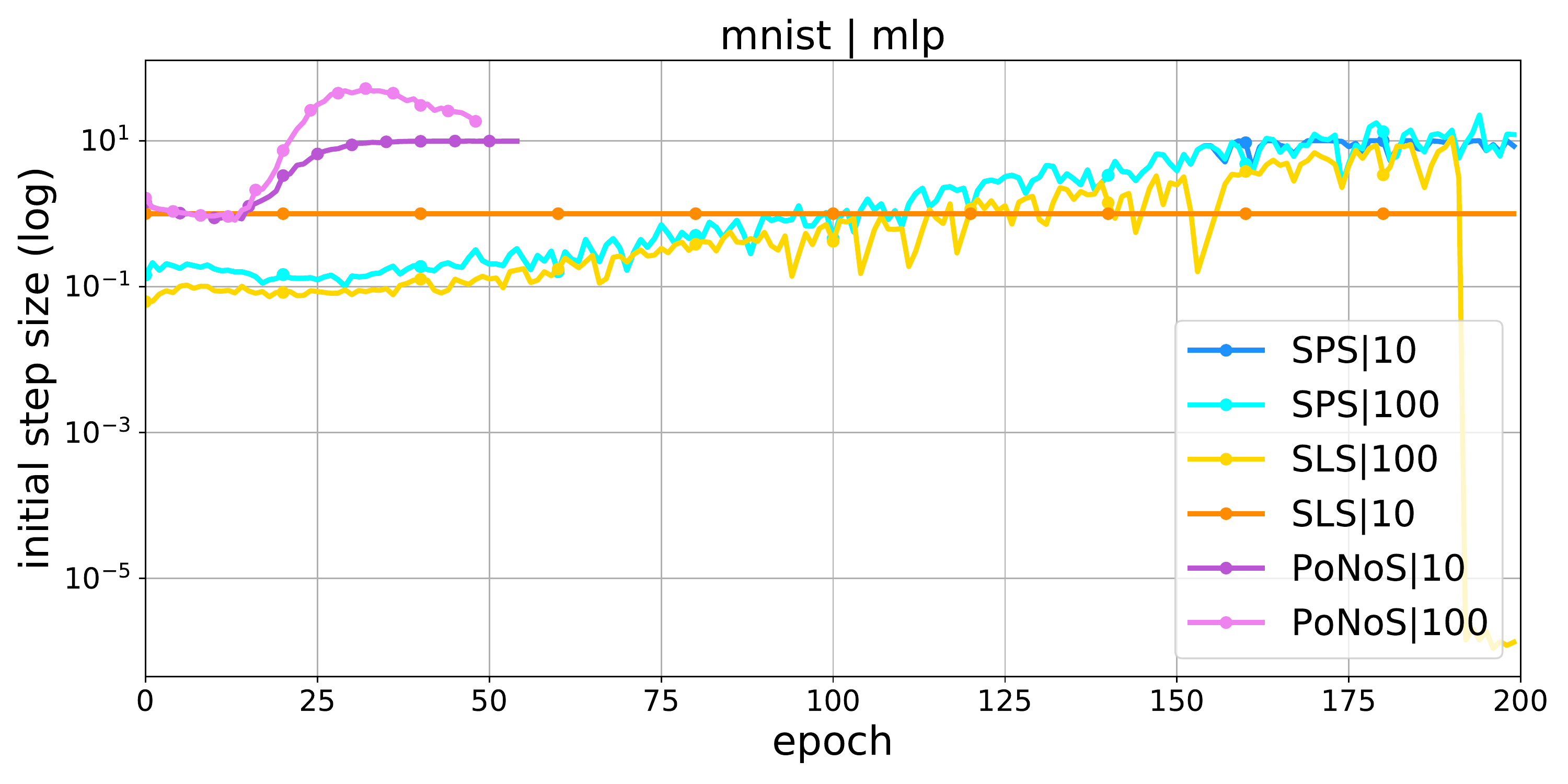}}
		\subfloat{\includegraphics[width=\imgS\linewidth]{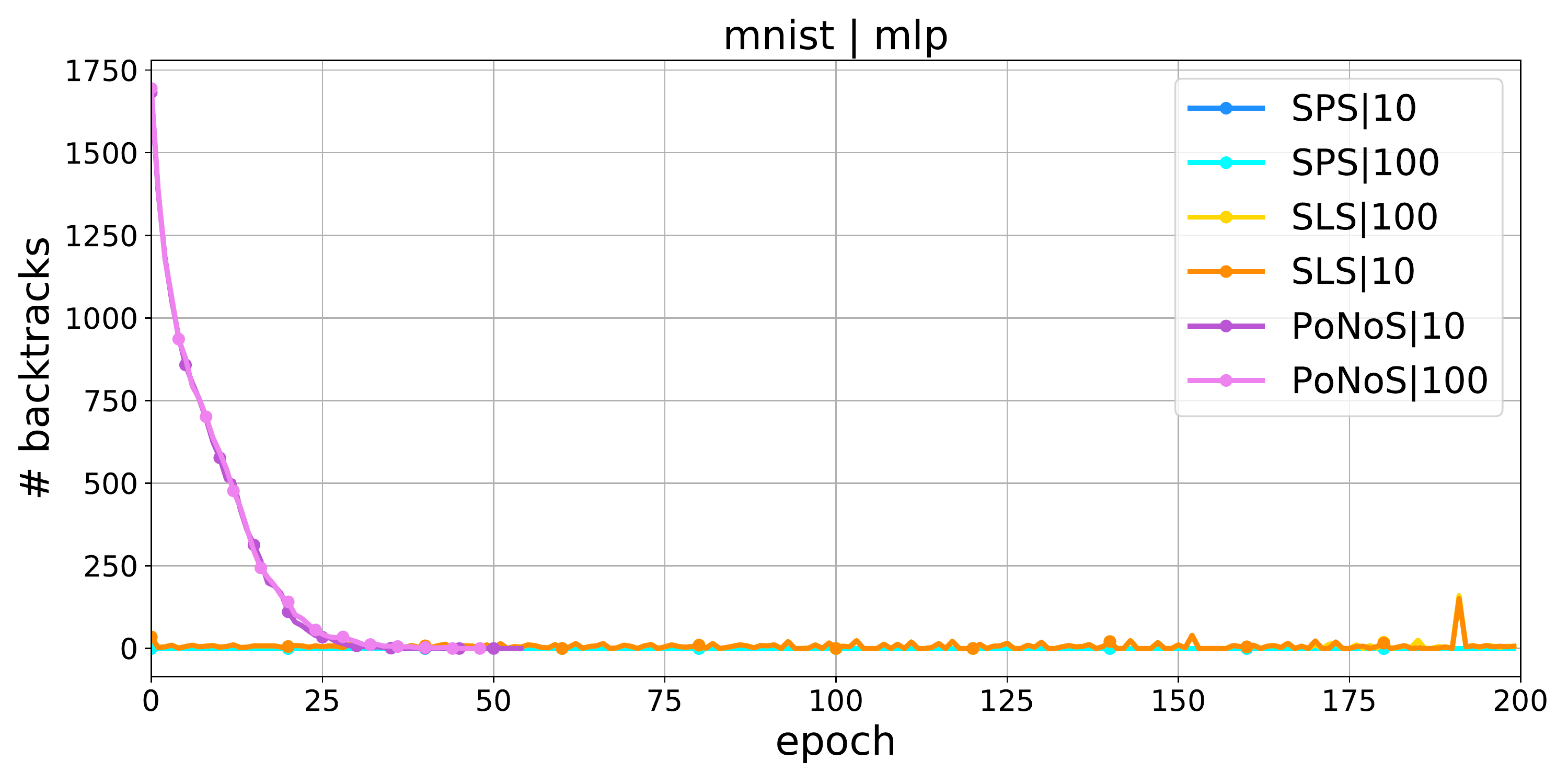}}\\
		\renewcommand{\model}{cifar10_resnet}
		\renewcommand{\modelname}{cifar10\_resnet}
		\subfloat{\includegraphics[width=\imgS\linewidth]{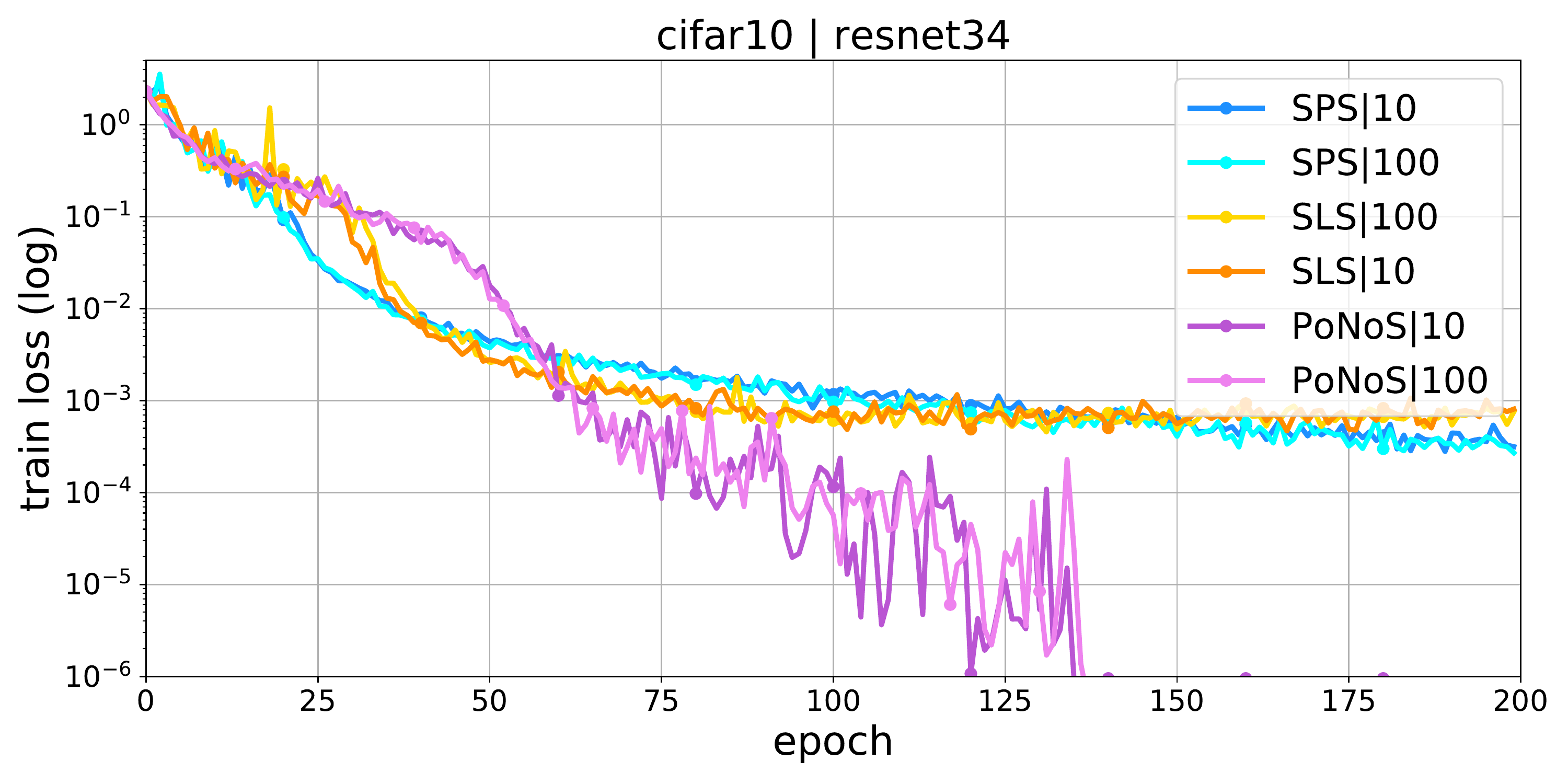}}
		\subfloat{\includegraphics[width=\imgS\linewidth]{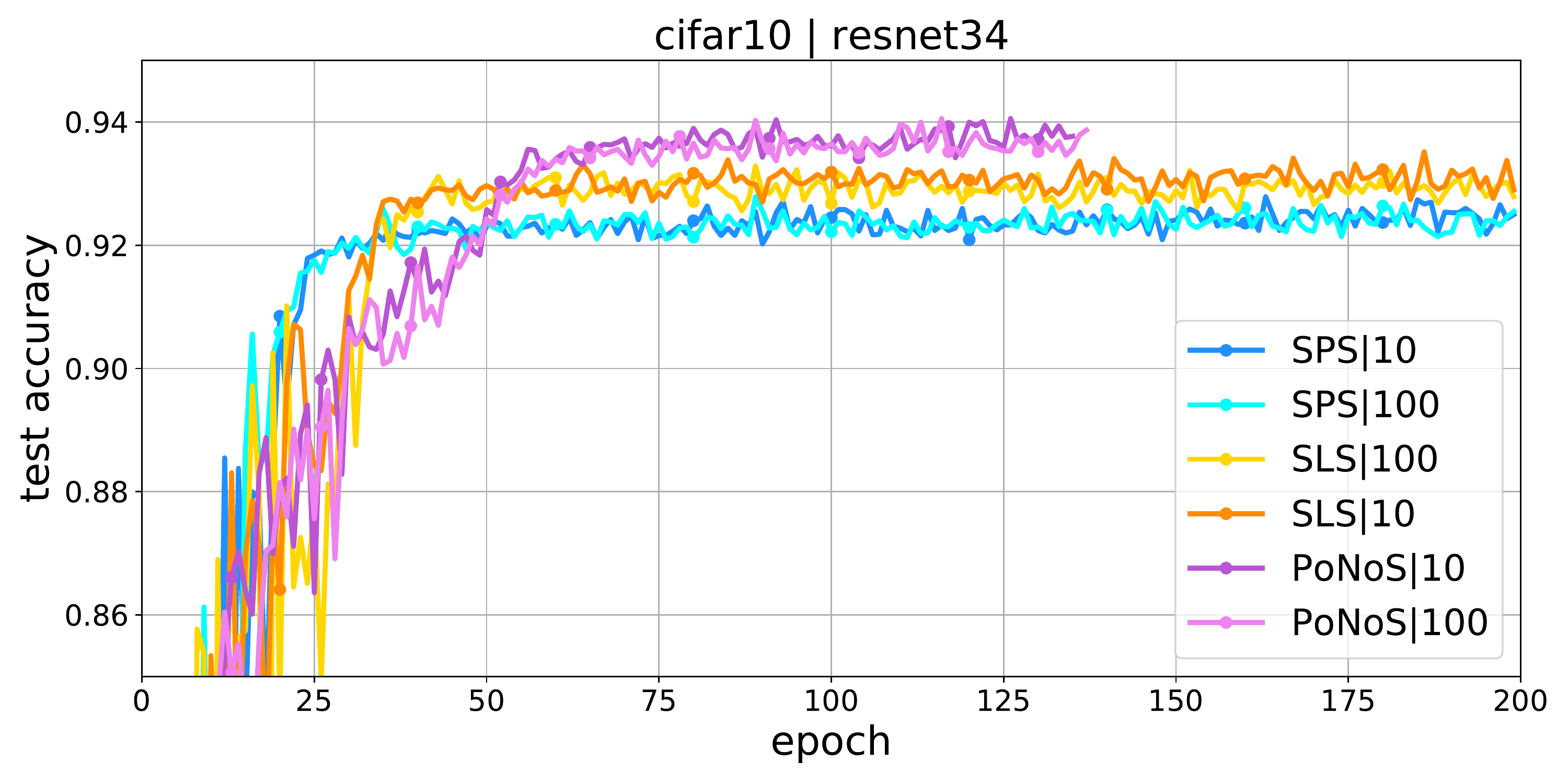}}
		\subfloat{\includegraphics[width=\imgS\linewidth]{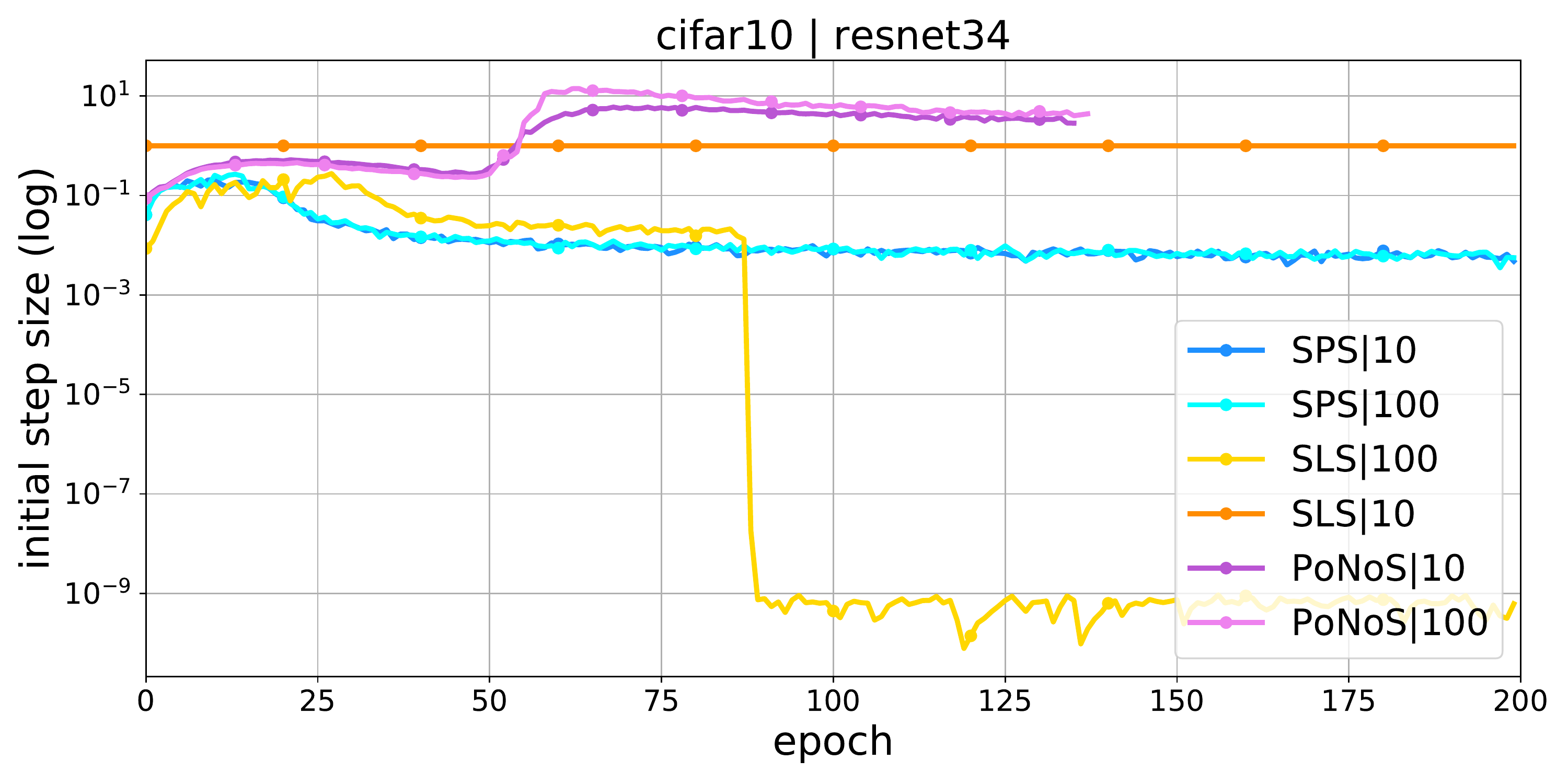}}
		\subfloat{\includegraphics[width=\imgS\linewidth]{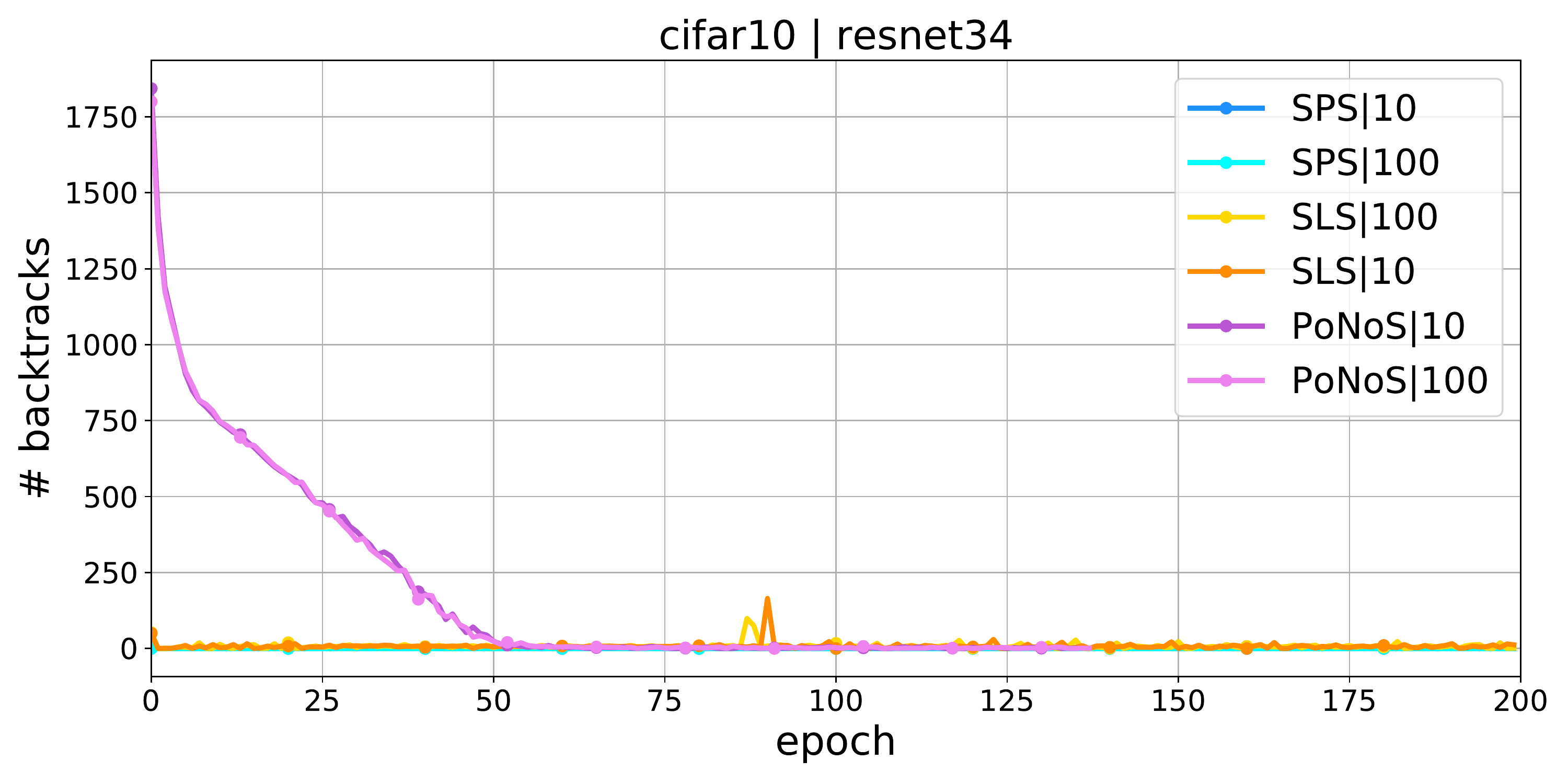}}\\
		\renewcommand{\model}{cifar10_densenet}
		\renewcommand{\modelname}{cifar10\_densenet}
		\subfloat{\includegraphics[width=\imgS\linewidth]{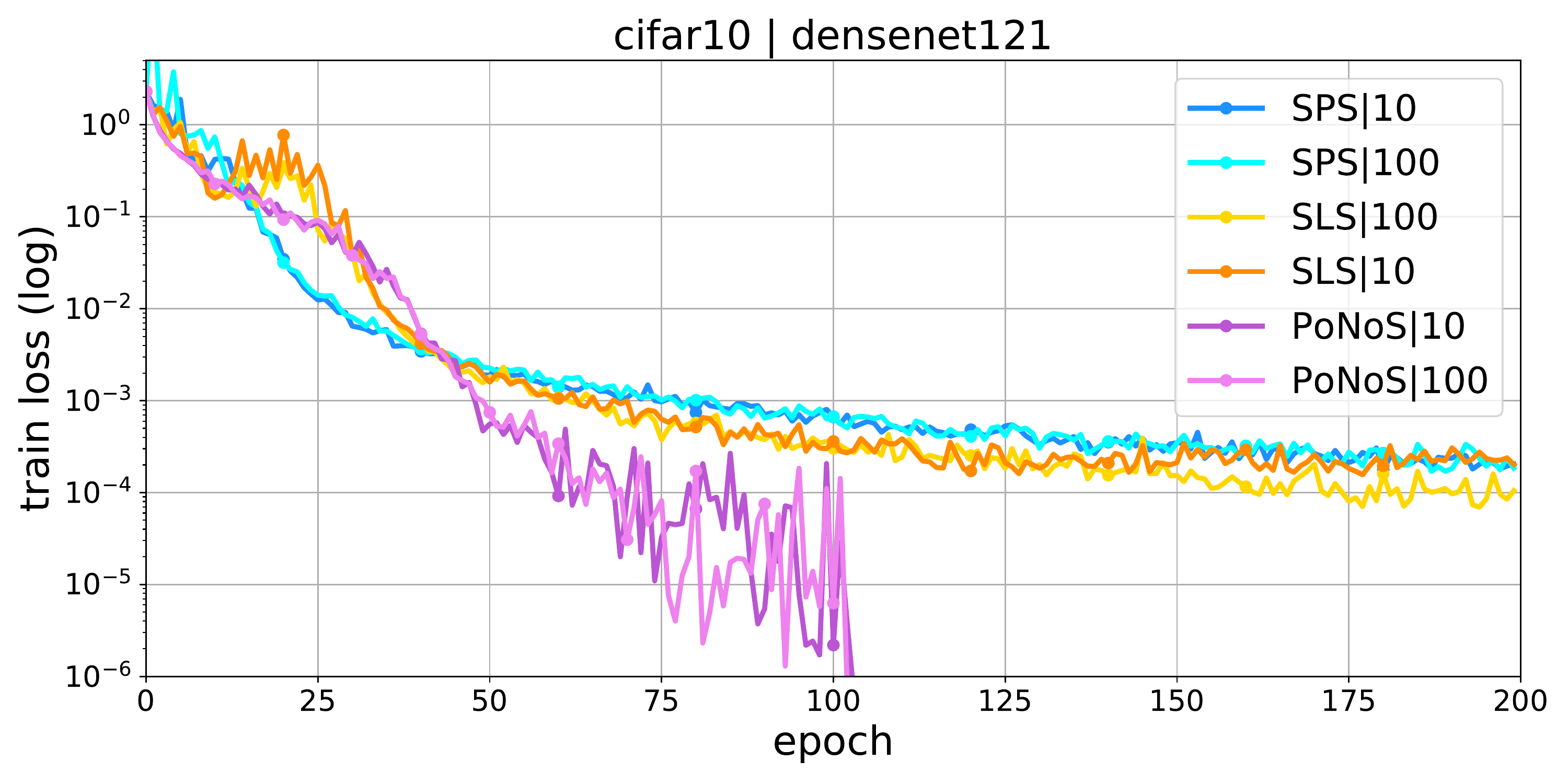}}
		\subfloat{\includegraphics[width=\imgS\linewidth]{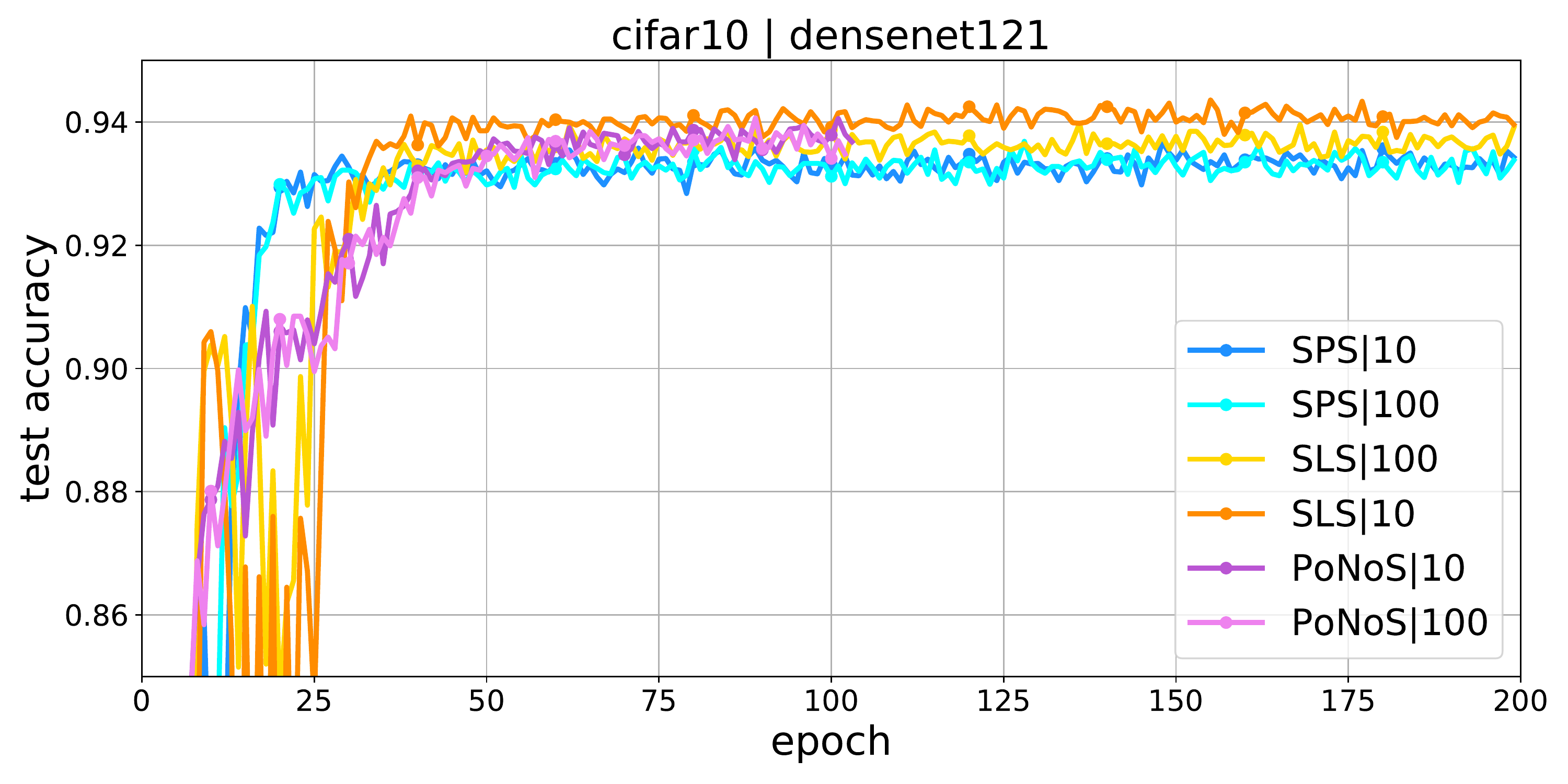}}
		\subfloat{\includegraphics[width=\imgS\linewidth]{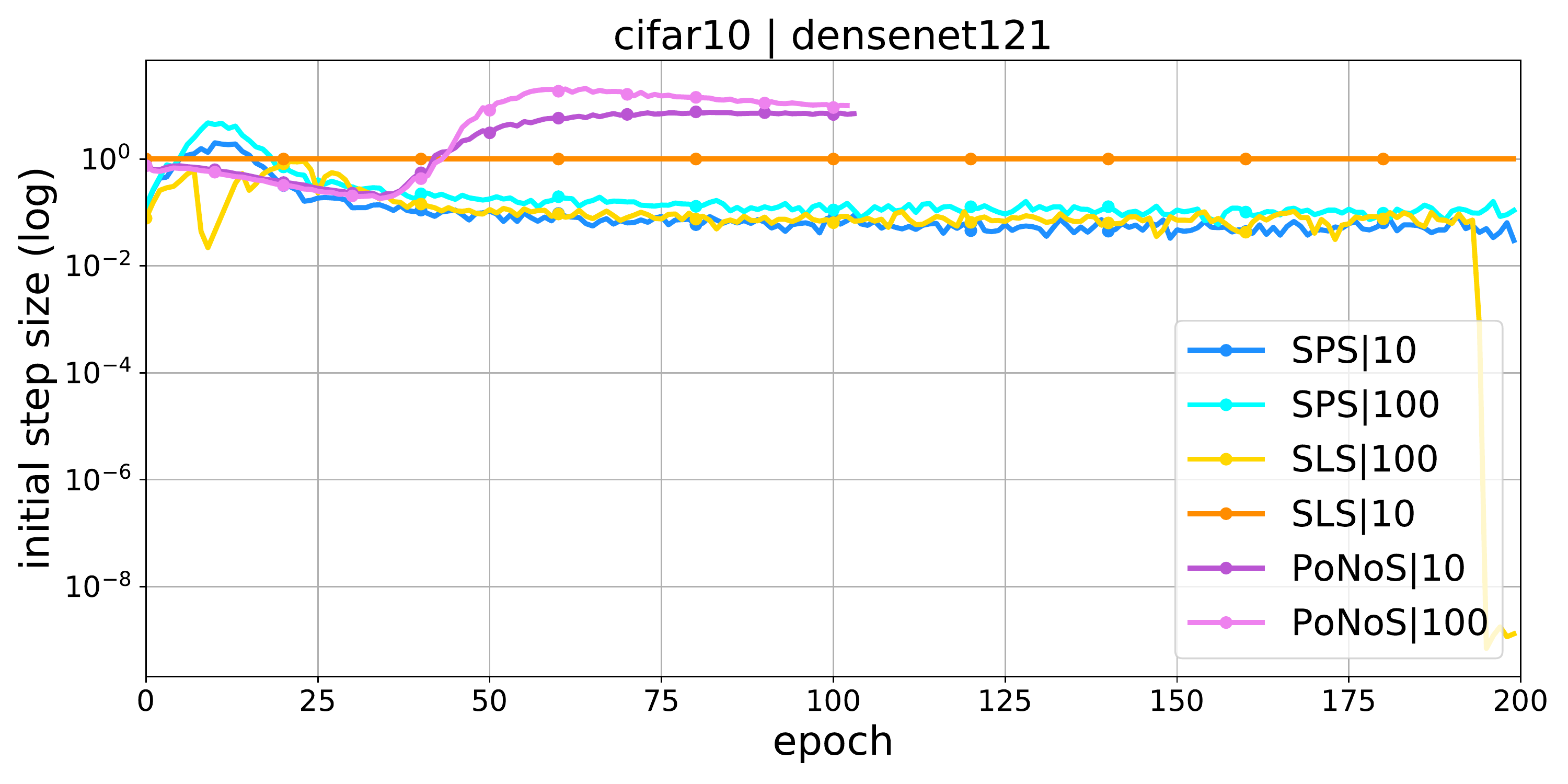}}
		\subfloat{\includegraphics[width=\imgS\linewidth]{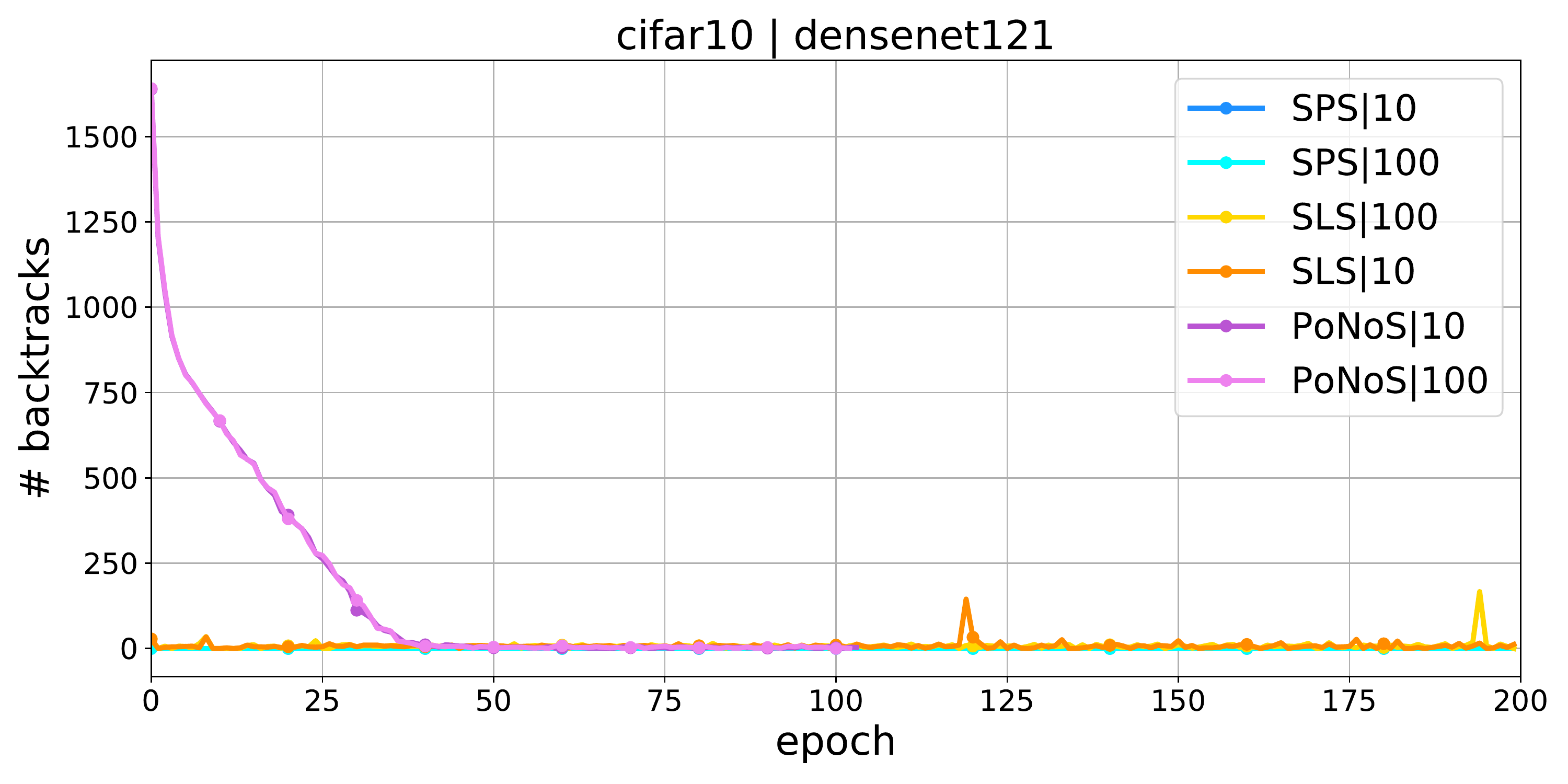}}\\
		\renewcommand{\model}{cifar100_resnet}
		\renewcommand{\modelname}{cifar100\_resnet}
		\subfloat{\includegraphics[width=\imgS\linewidth]{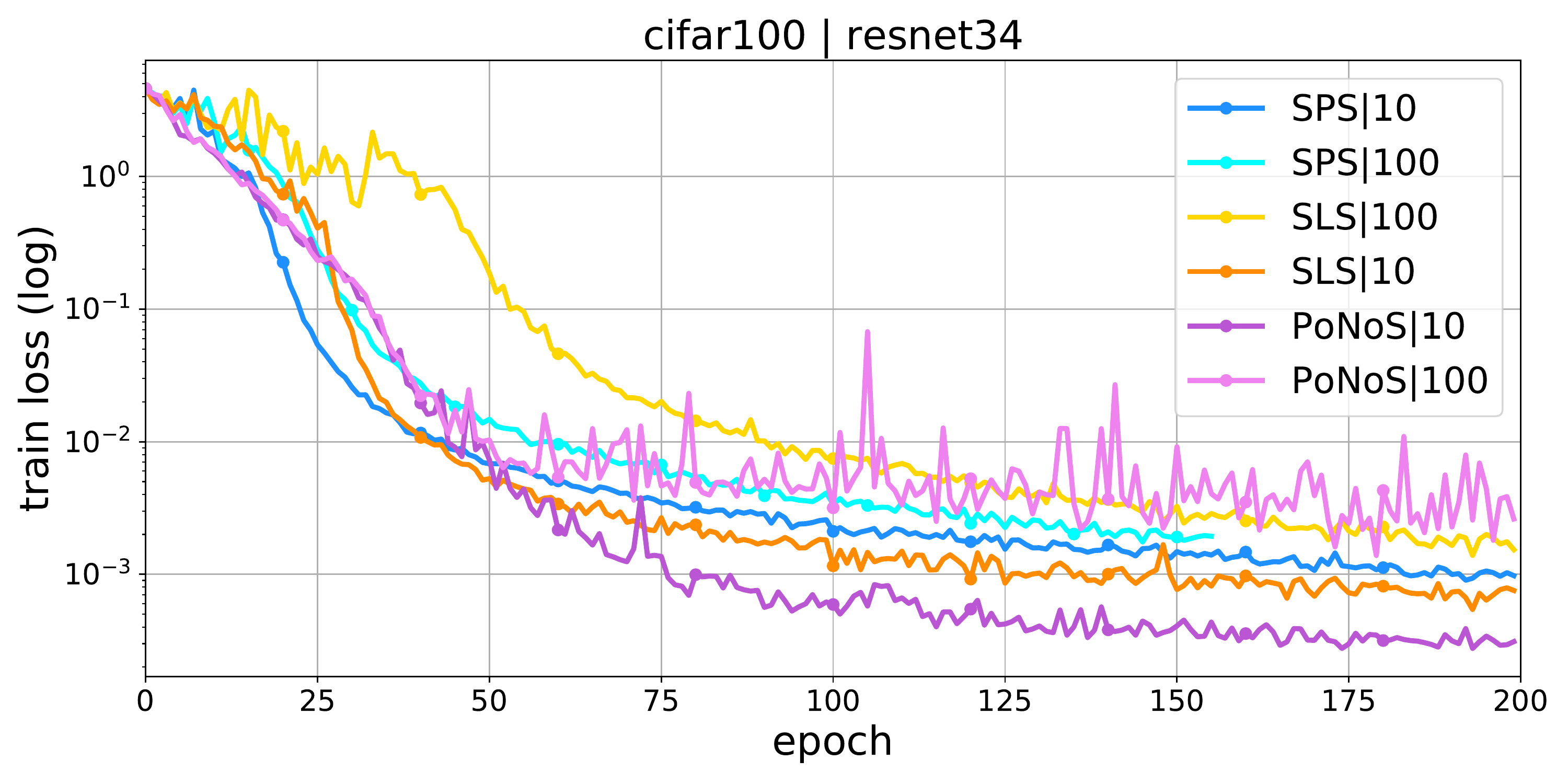}}
		\subfloat{\includegraphics[width=\imgS\linewidth]{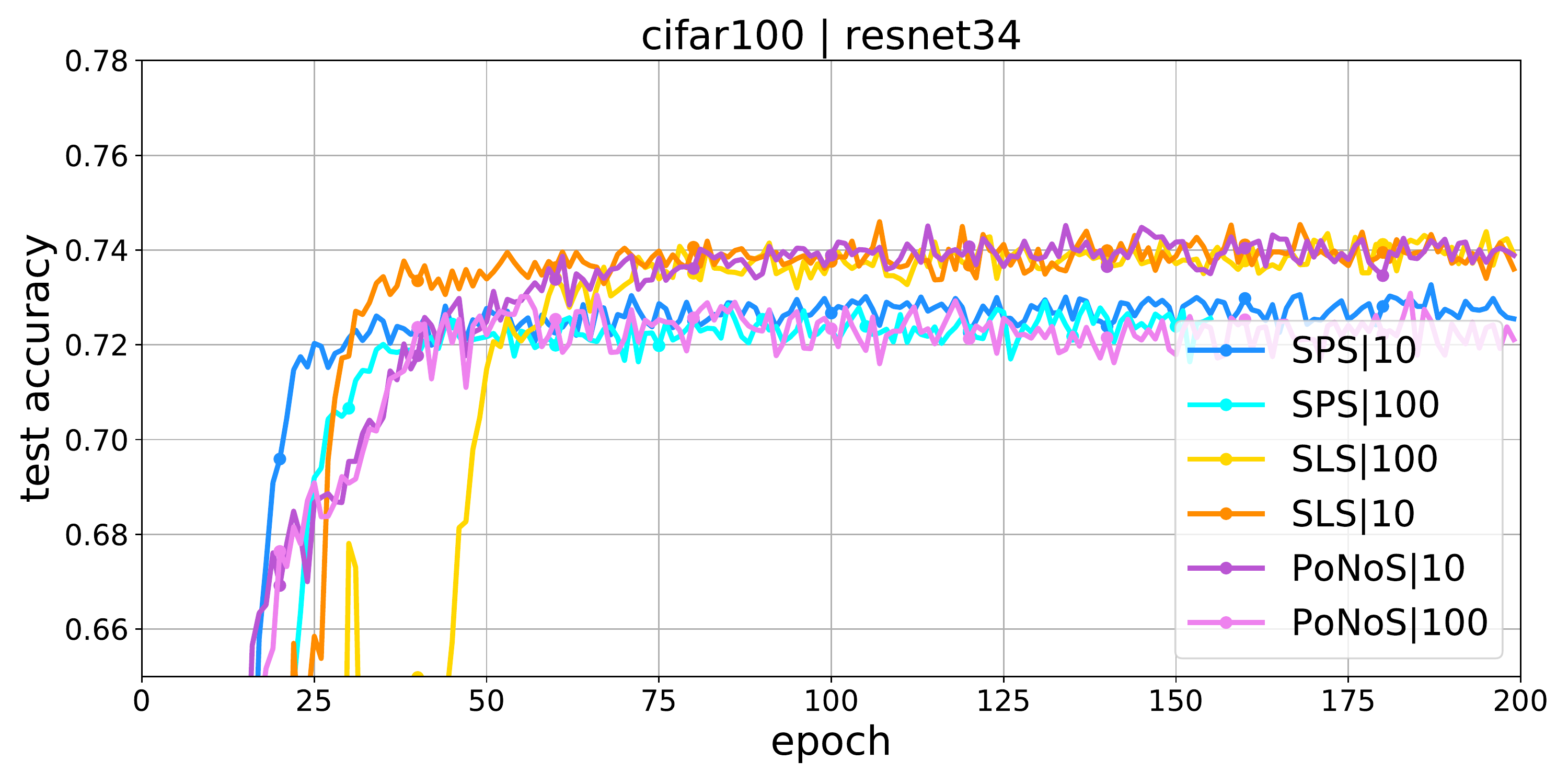}}
		\subfloat{\includegraphics[width=\imgS\linewidth]{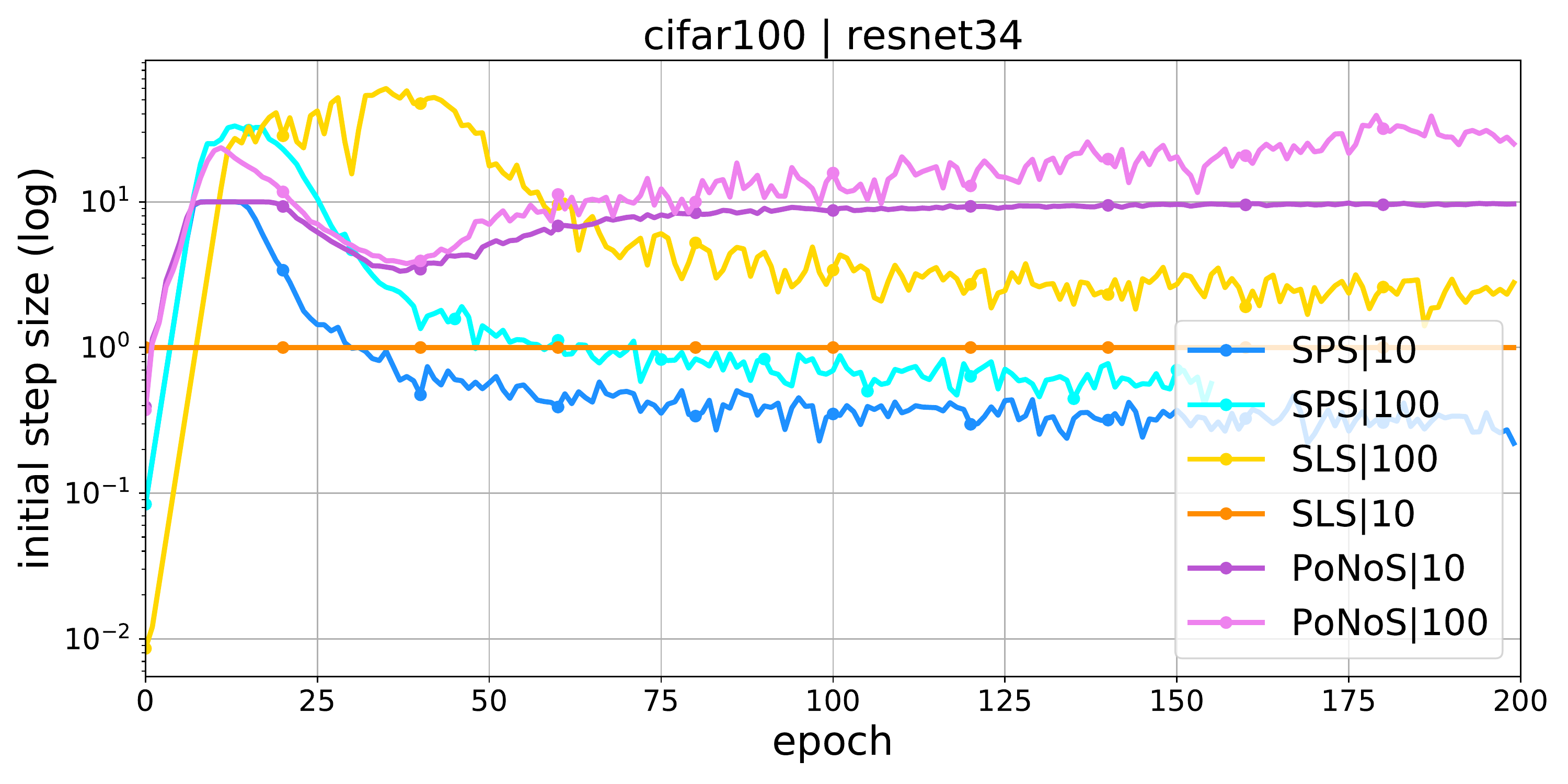}}
		\subfloat{\includegraphics[width=\imgS\linewidth]{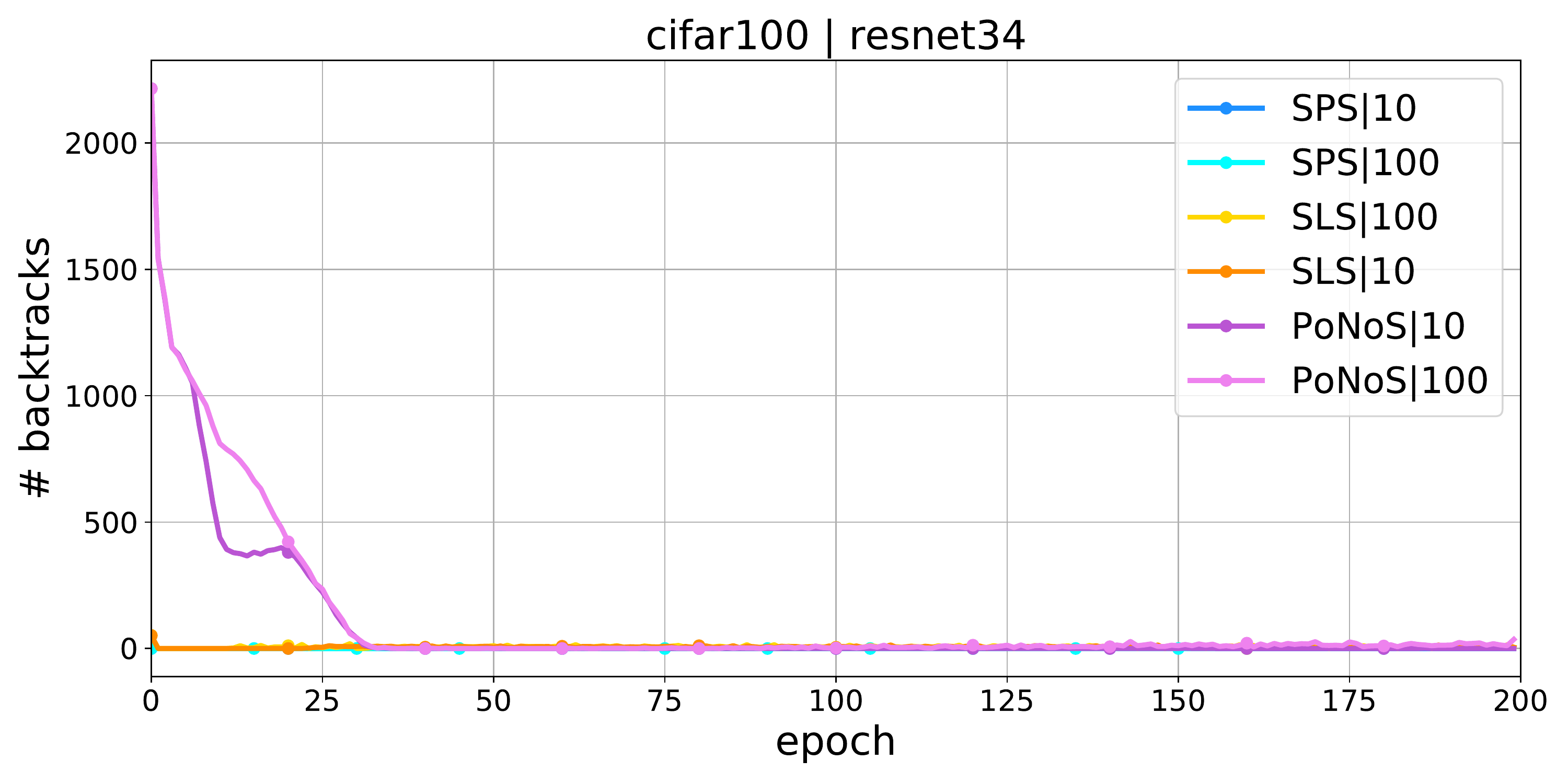}}\\
		\renewcommand{\model}{cifar100_densenet}
		\renewcommand{\modelname}{cifar100\_densenet}
		\subfloat{\includegraphics[width=\imgS\linewidth]{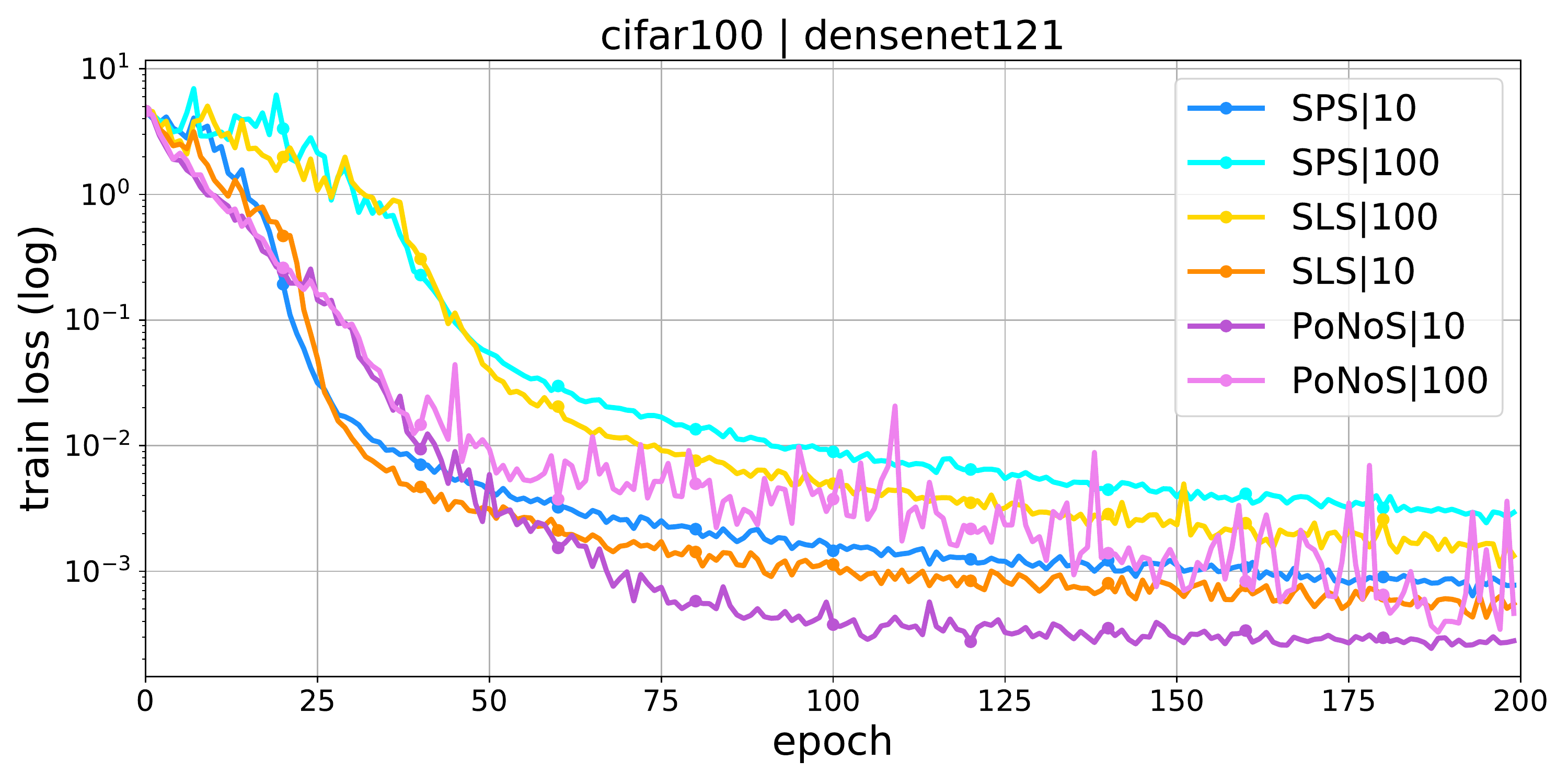}}
		\subfloat{\includegraphics[width=\imgS\linewidth]{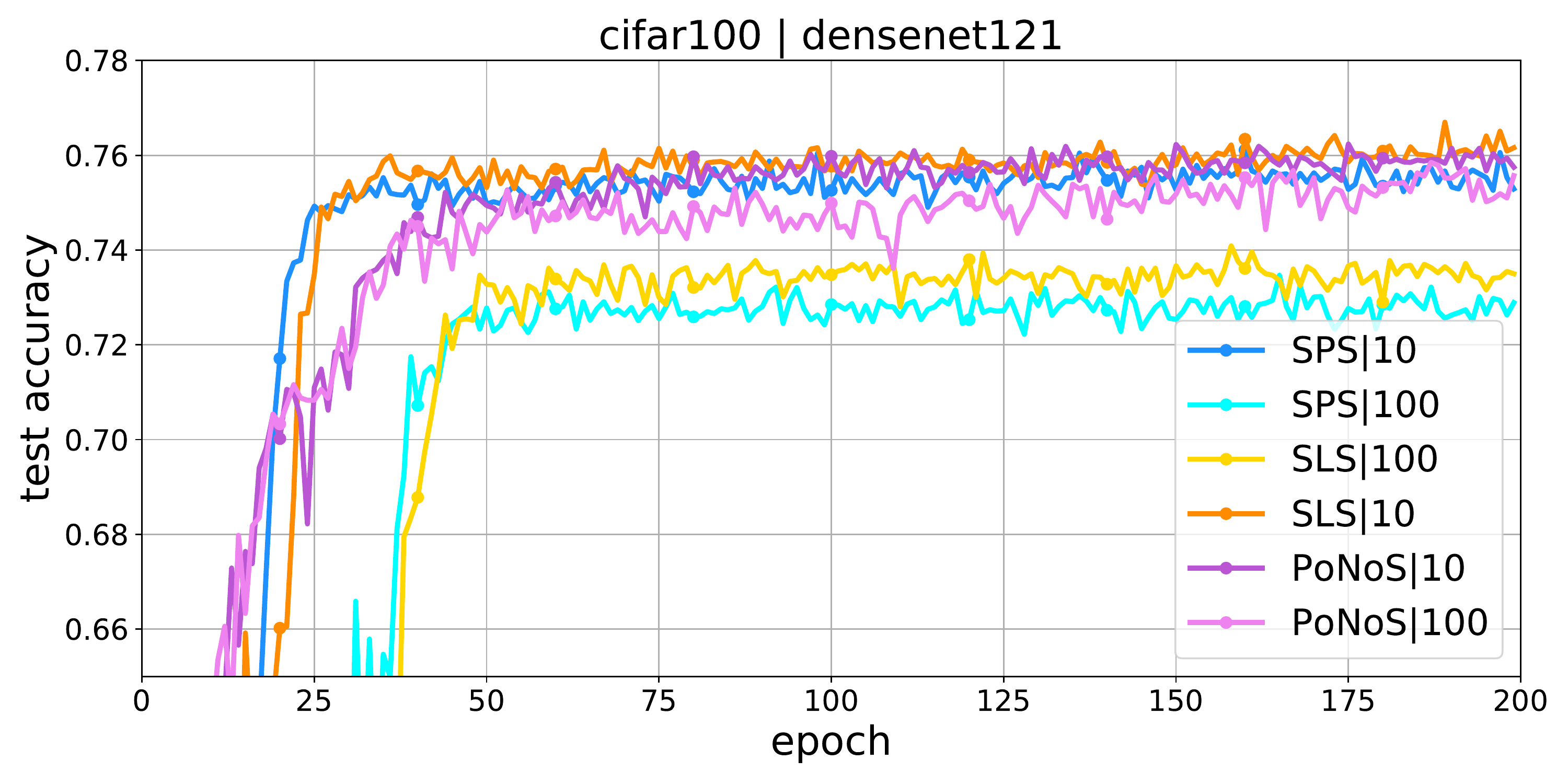}}
		\subfloat{\includegraphics[width=\imgS\linewidth]{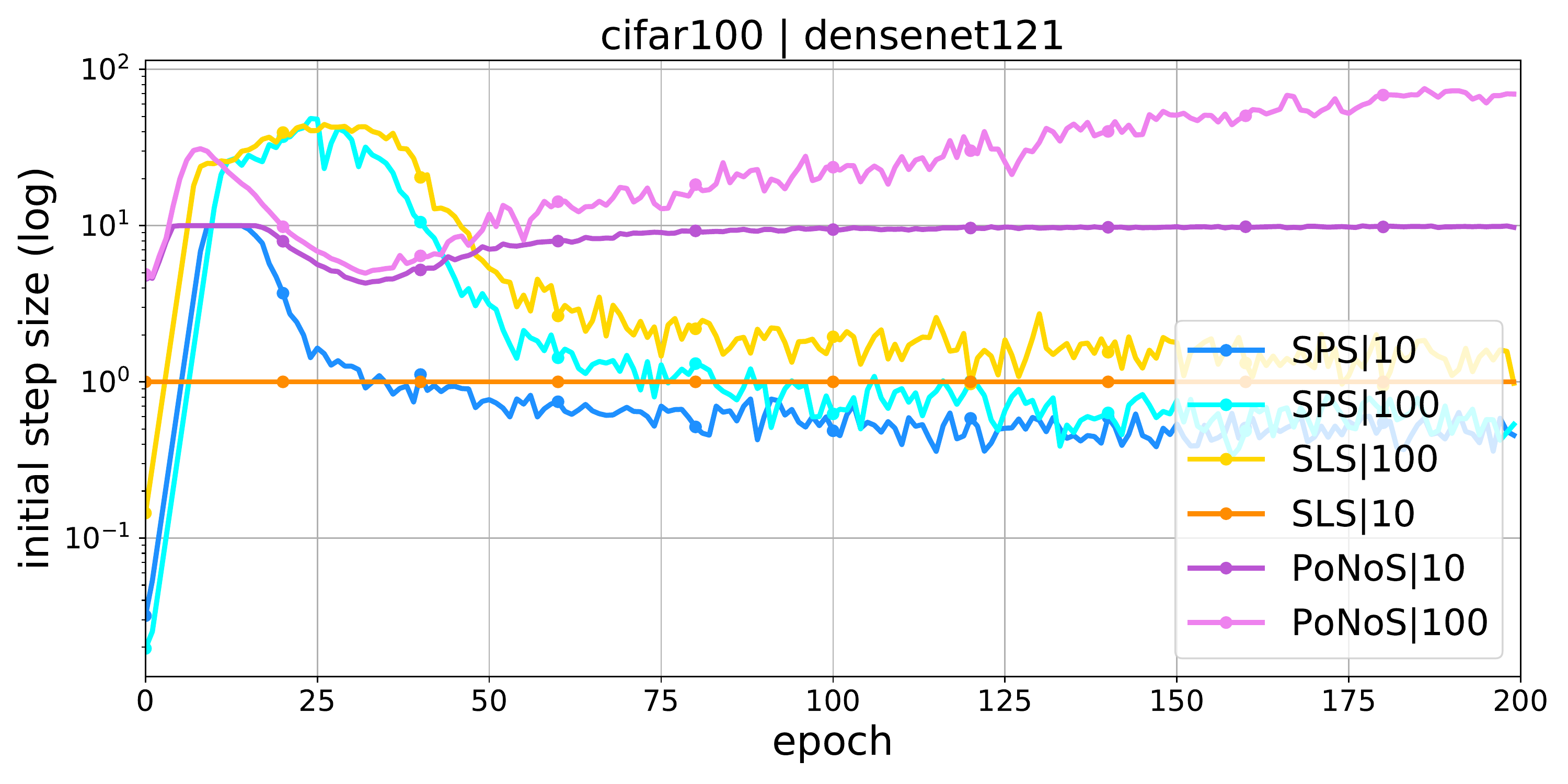}}
		\subfloat{\includegraphics[width=\imgS\linewidth]{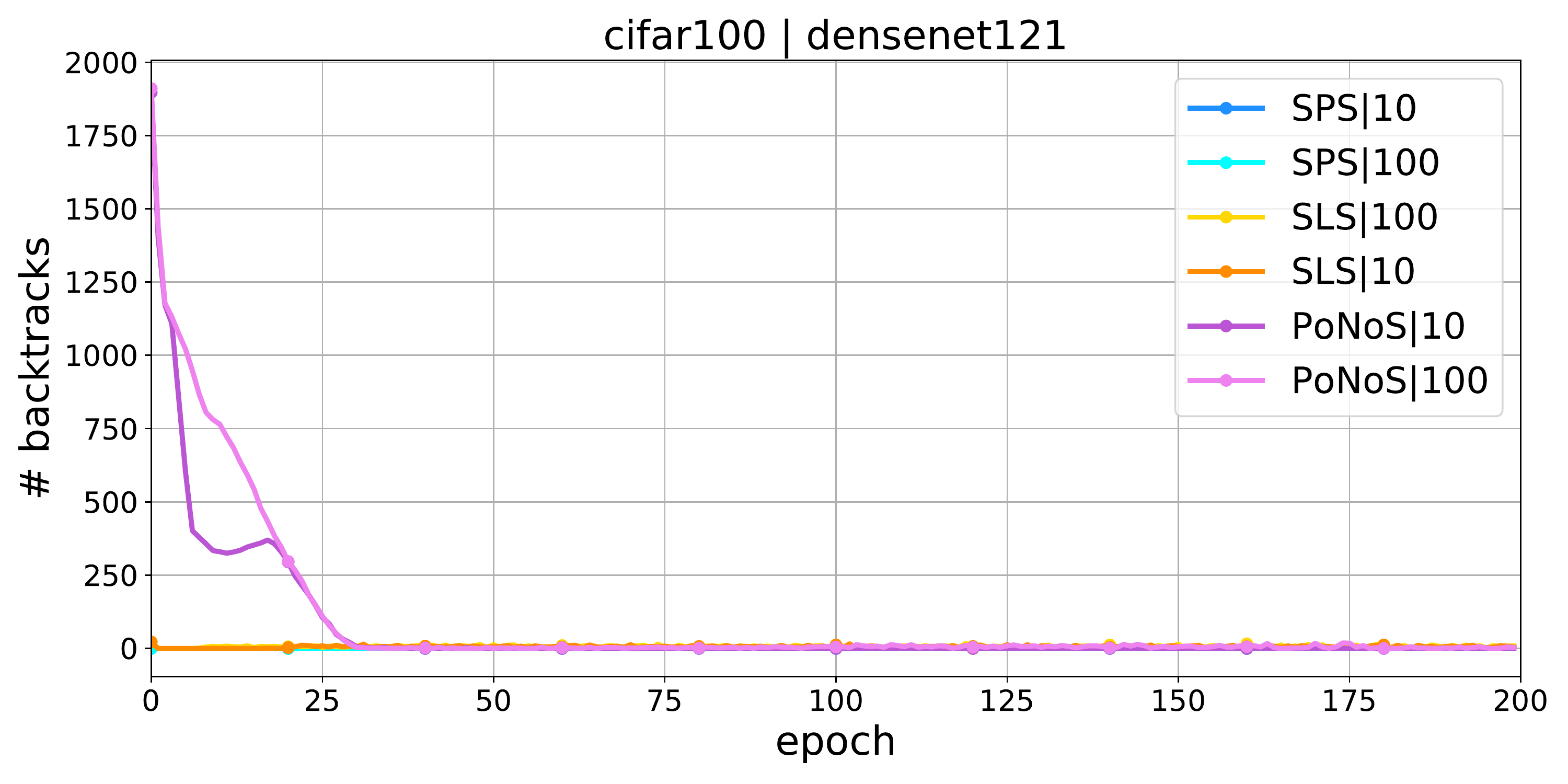}}\\
		\renewcommand{\model}{fashion_effb1}
		\renewcommand{\modelname}{fashion\_effb1}
		\subfloat{\includegraphics[width=\imgS\linewidth]{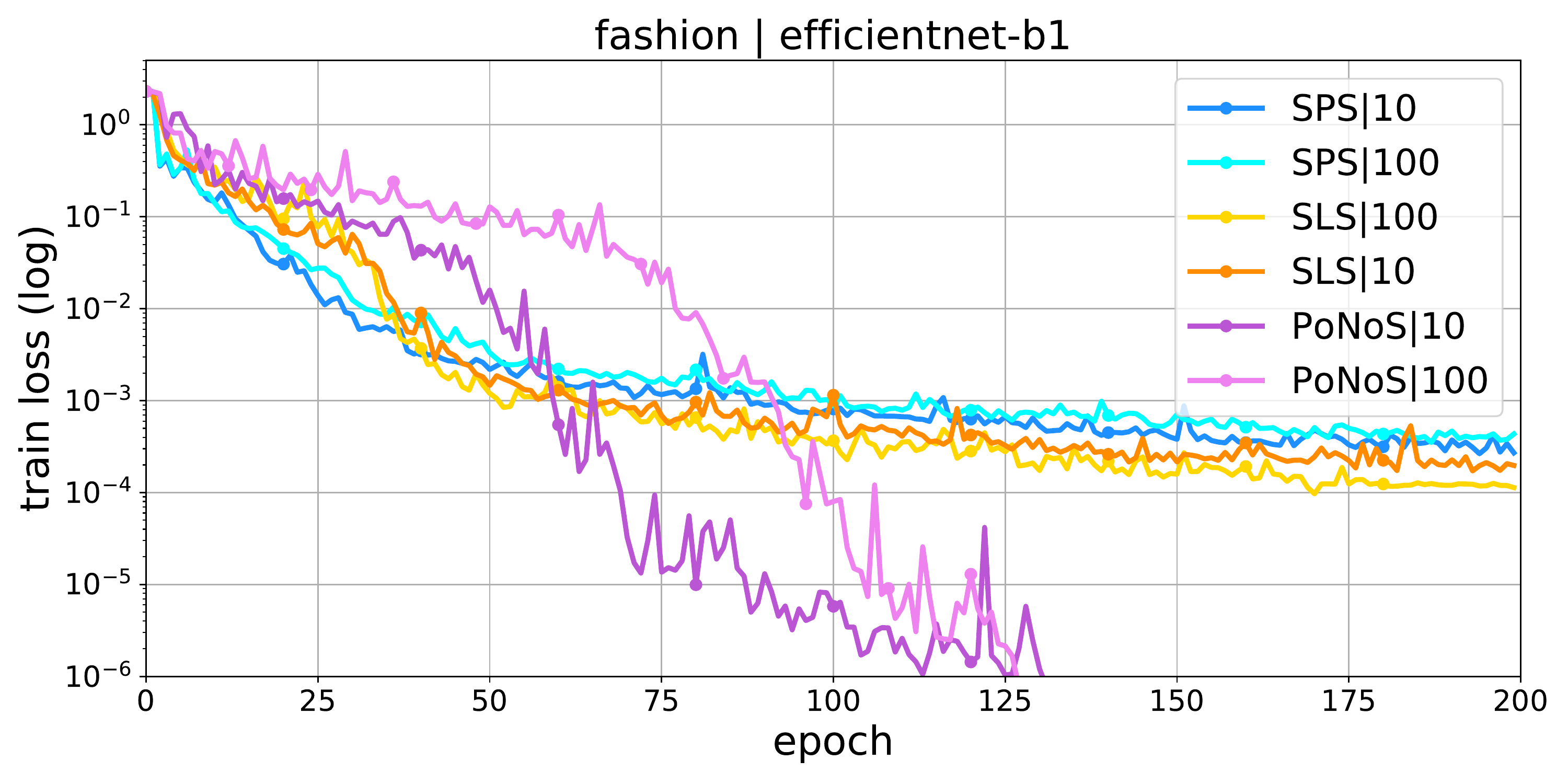}}
		\subfloat{\includegraphics[width=\imgS\linewidth]{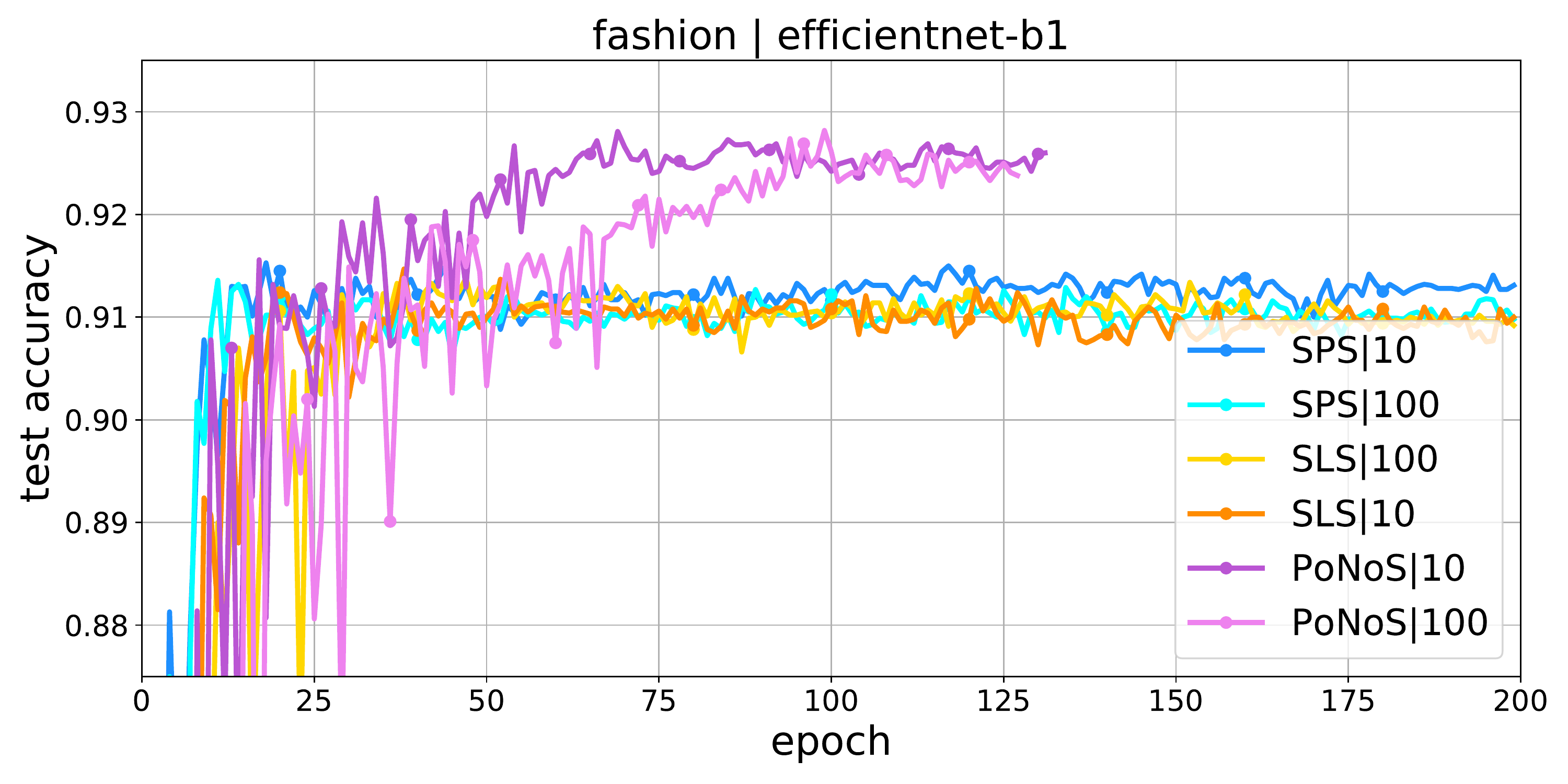}}
		\subfloat{\includegraphics[width=\imgS\linewidth]{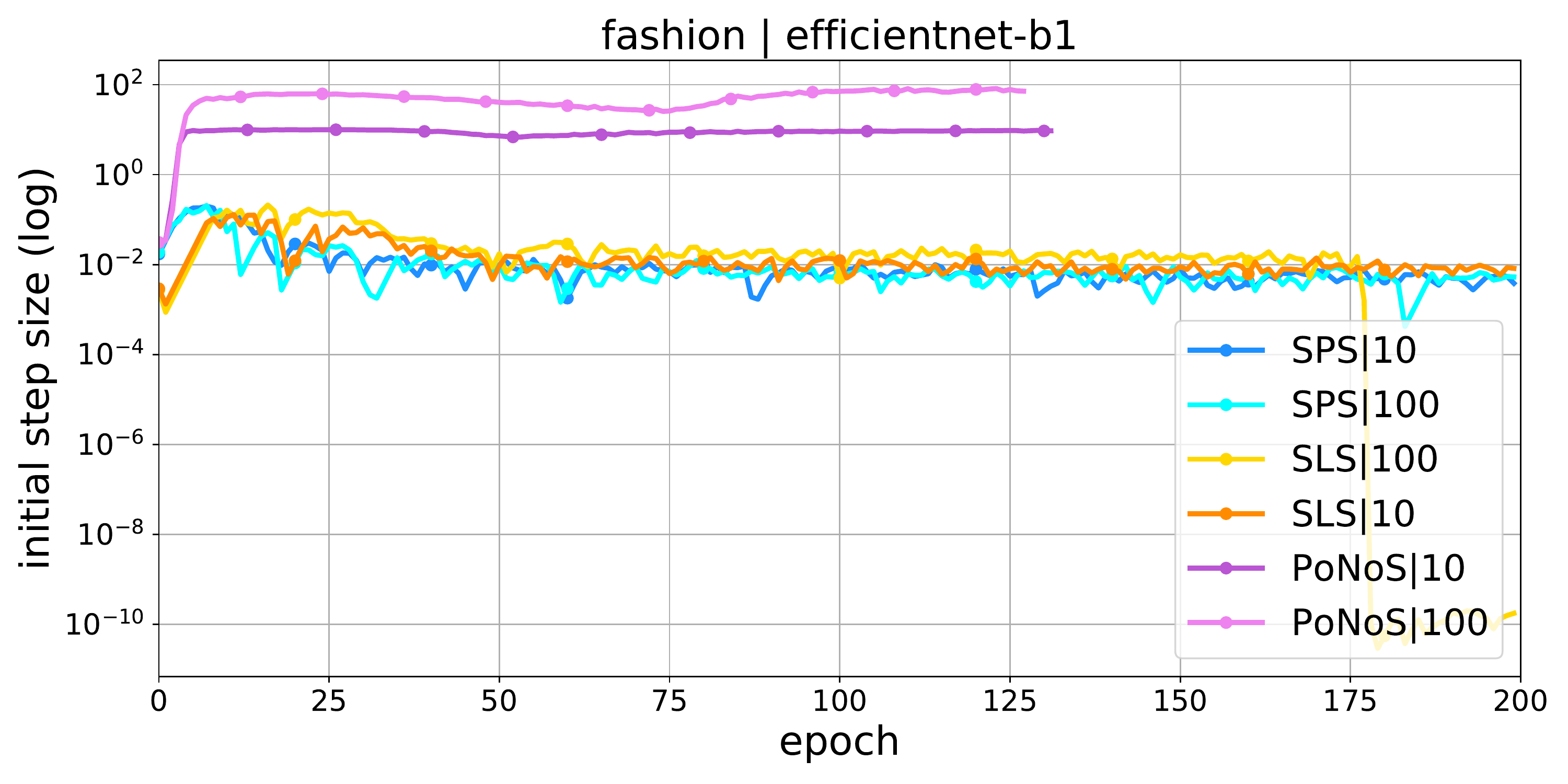}}
		\subfloat{\includegraphics[width=\imgS\linewidth]{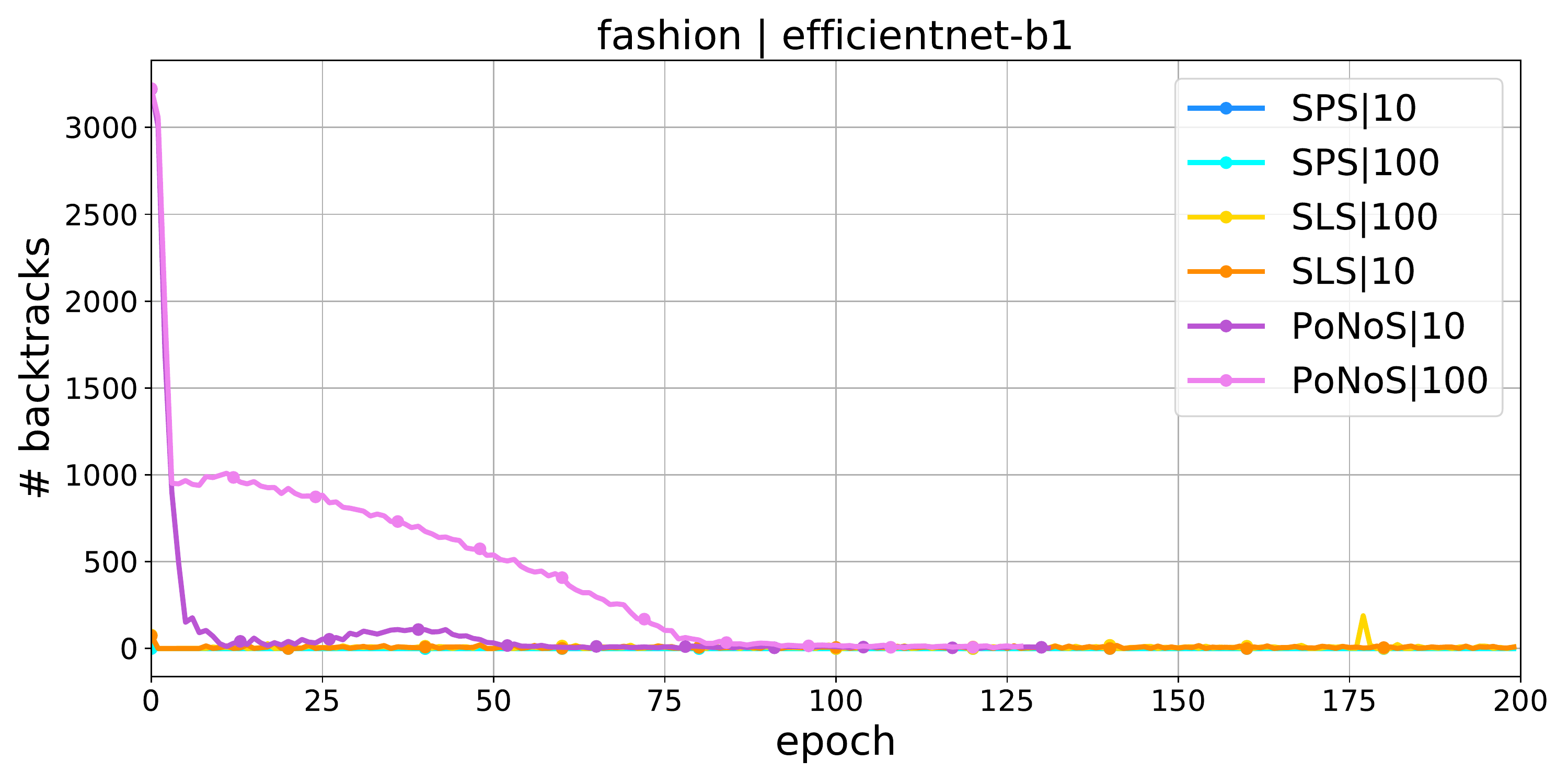}}\\
		\renewcommand{\model}{svhn_wrn}
		\renewcommand{\modelname}{svhn\_wrn}
		\subfloat{\includegraphics[width=\imgS\linewidth]{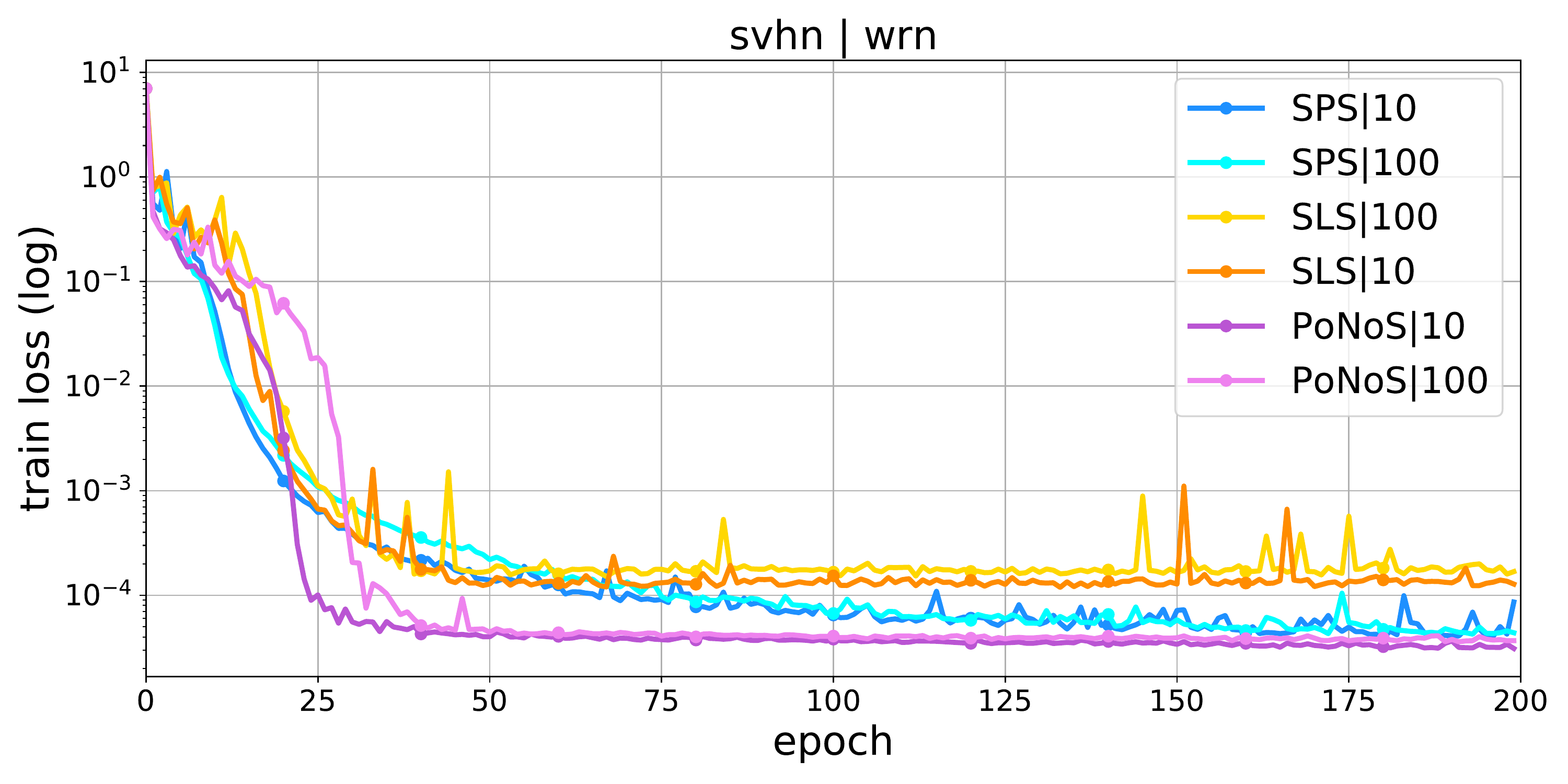}}
		\subfloat{\includegraphics[width=\imgS\linewidth]{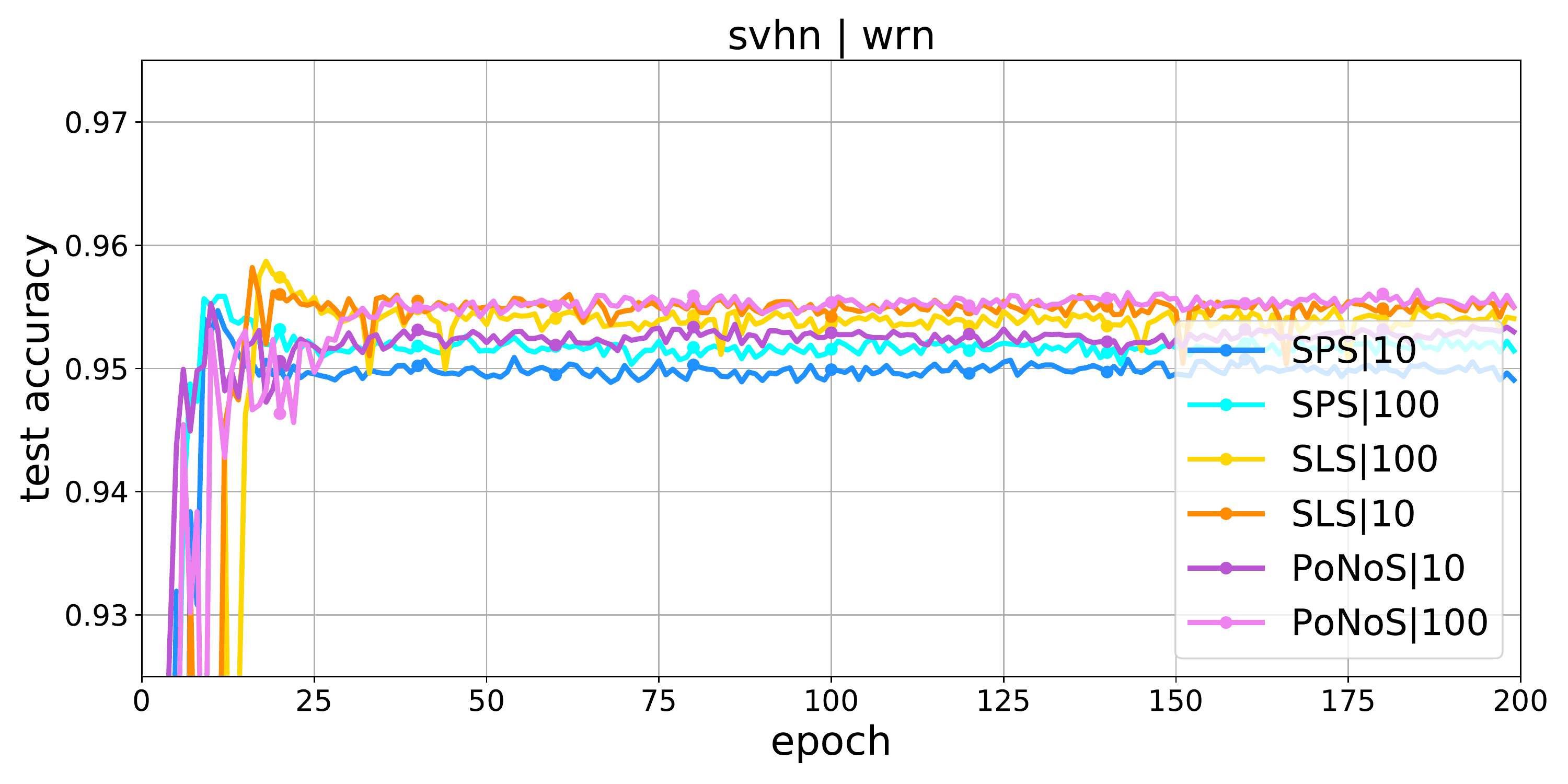}}
		\subfloat{\includegraphics[width=\imgS\linewidth]{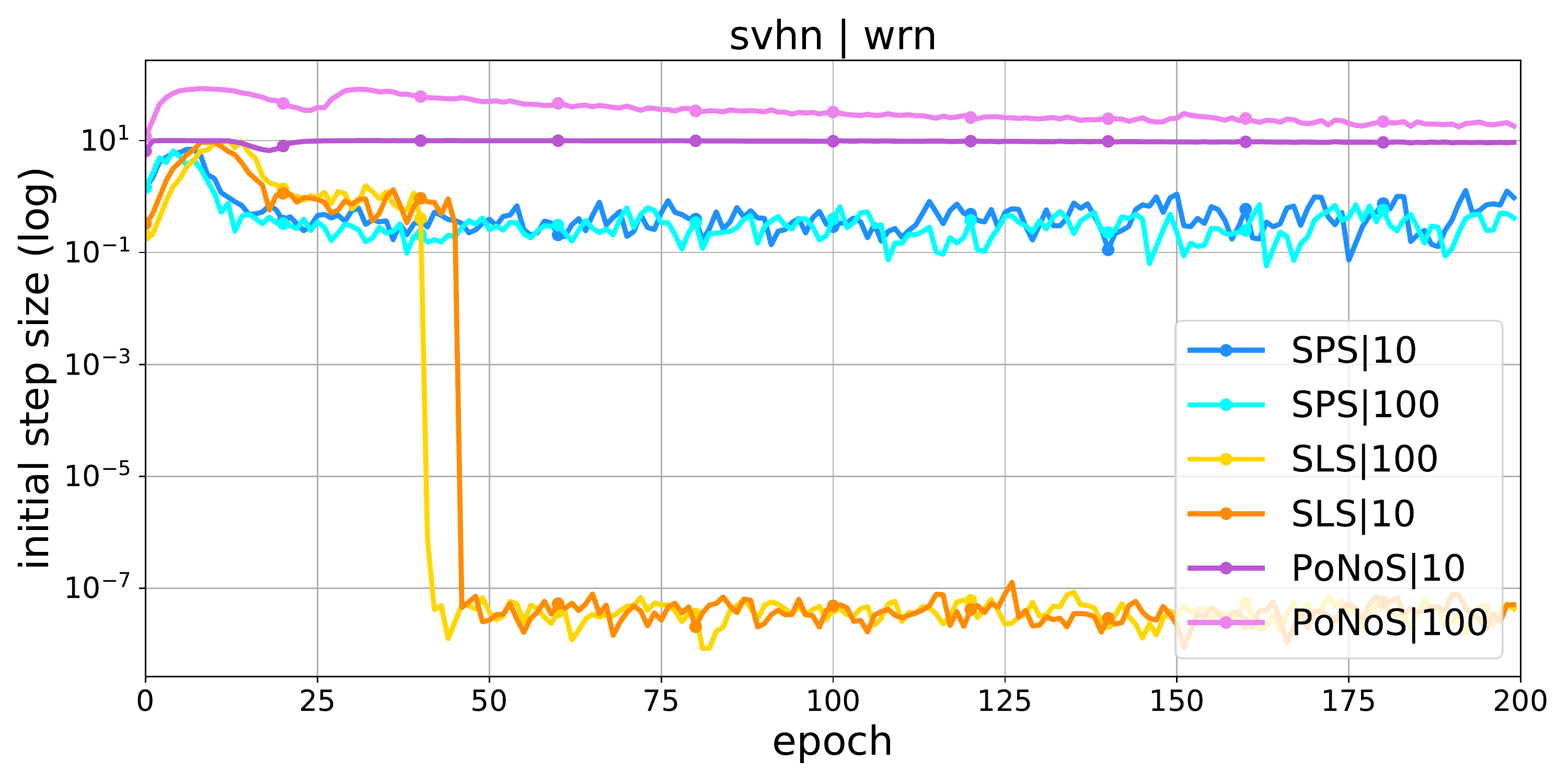}}
		\subfloat{\includegraphics[width=\imgS\linewidth]{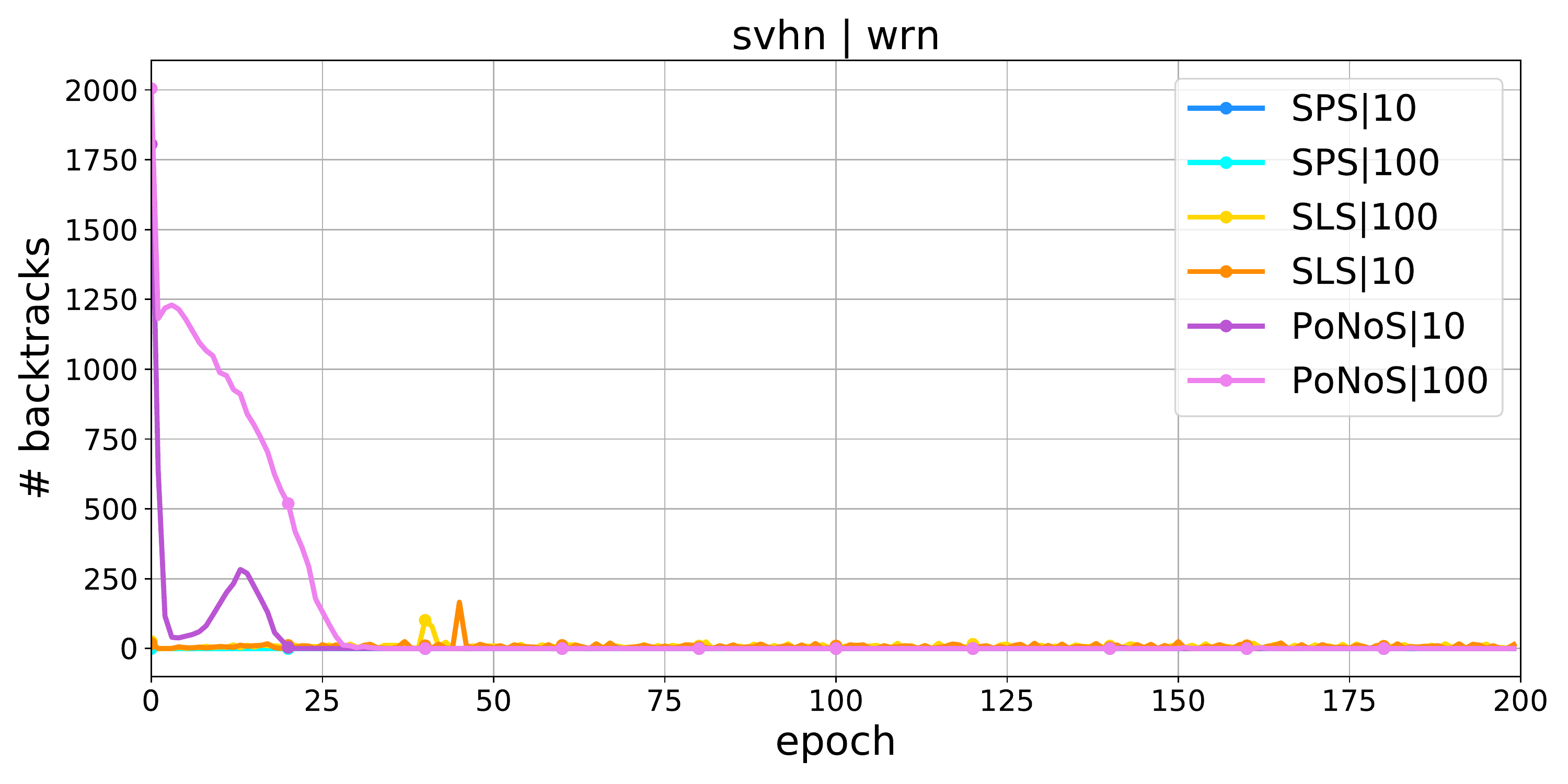}}
		\caption{Comparison between the use of the constant $\etamax=10$ and $\etamax=100$ on SPS, SLS and PoNoS. Each row focus on a dataset/model combination. First column: train loss. Second column: test accuracy. Third column: average initial step size of the epoch. Fourth column: cumulative number of backtracks in the epoch.}\label{fig:etamax}
	\end{minipage}
\end{figure}
	\subsection{Study on the Choice of $c_p$: Doubling the Legacy Value}\label{sec:supp_c_p}
	In this subsection, we report an ablation study on $c_p$ in \eqref{eq:polyak} for PoNoS and SPS. PoNoS utilize the Polyak step size as an initial guess for the line search method and not as a direct learning rate for SGD as in SPS \citep{loizou21a}. We suggest doubling the value employed in \citet{loizou21a} by employing $c_p=0.1$ in \eqref{eq:polyak} instead of the default value of $c_p=0.2$. Given the fact that $\delta=0.5$, the line search will half the step size when needed, achieving only in this case the same effect of reducing $c_p$ back to $c_p=0.2$. In Figure \ref{fig:c_p} we compare
\begin{itemize}
	\item SPS|0.1: SPS with $c_p=0.1$. 
	\item SPS|0.2: SPS with $c_p=0.2$. This setting corresponds to SPS.
	\item PoNoS|0.1: SLS with $c_p=0.1$. This setting corresponds to PoNoS.
	\item PoNoS|0.2: SLS with $c_p=0.2$. 
\end{itemize}
From Figure \ref{fig:c_p} we can observe that the differences between $c_p=0.1$ and $c_p=0.2$ are not remarkable. However, PoNoS|0.1 always achieves slightly better performance than PoNoS|0.2 both in terms of train loss and test accuracy. On the other hand, the test accuracy of SPS|0.1 is not always as high as that achieved by SPS|0.2. To conclude, we can observe from the third column of Figure \ref{fig:c_p} that the step size of PoNoS|0.1 is not always larger than that of PoNoS|0.2. In fact, this does not happen in the initial phase of the optimization, but in the intermediate phase and only on some problems. Given the improved performance of PoNoS|0.1 over PoNoS|0.2, one could argue that the line search finds the regions in which a larger step is beneficial. 
\renewcommand{\dir}{c_p/}
\renewcommand{\imgS}{0.37}
\begin{figure}[!h] 
	\hspace{-3mm}
	\begin{minipage}{0.9\textwidth}
		\renewcommand{\model}{mlp}
		\renewcommand{\modelname}{mlp}
		\subfloat{\includegraphics[width=\imgS\linewidth]{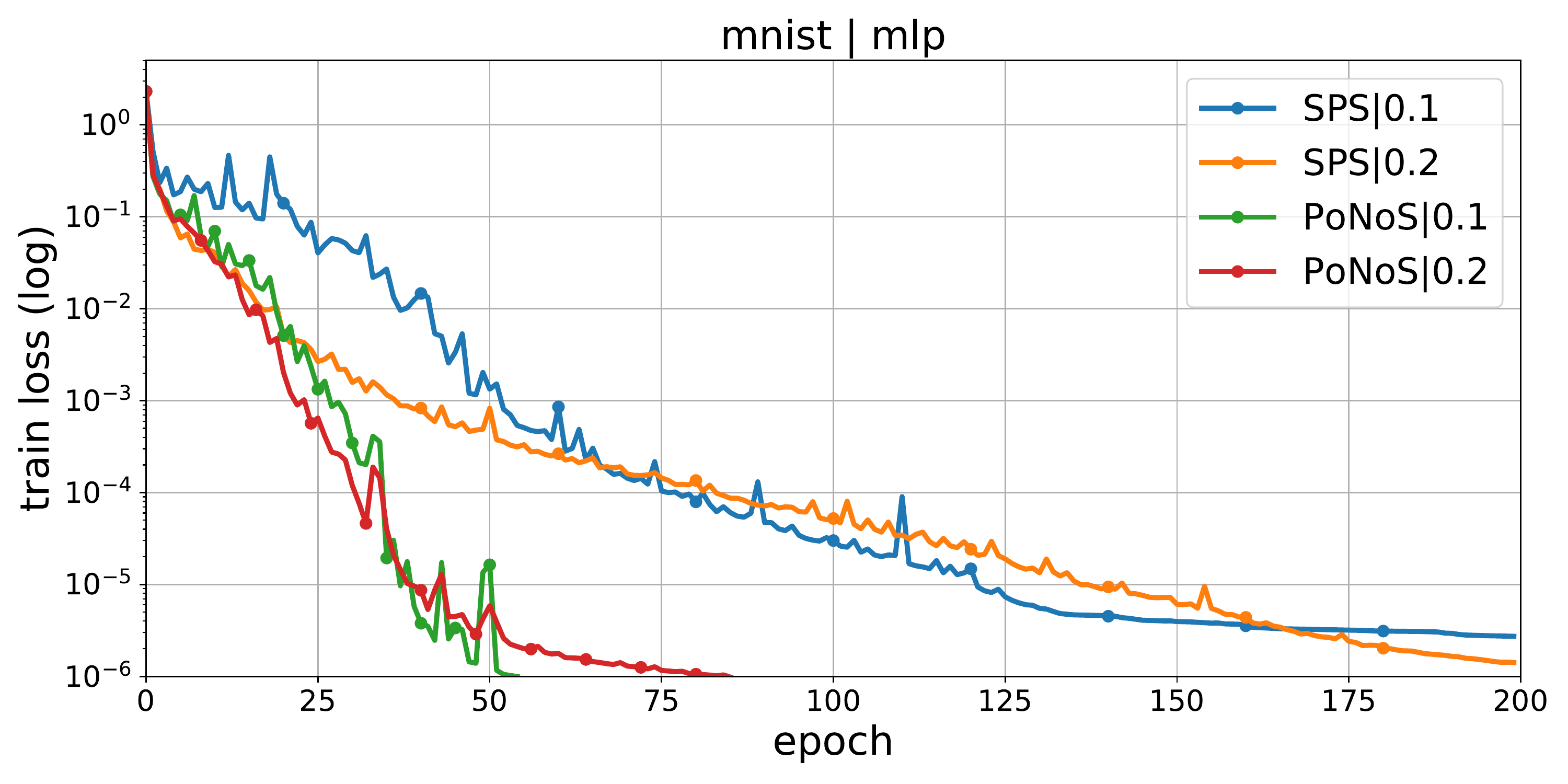}}
		\subfloat{\includegraphics[width=\imgS\linewidth]{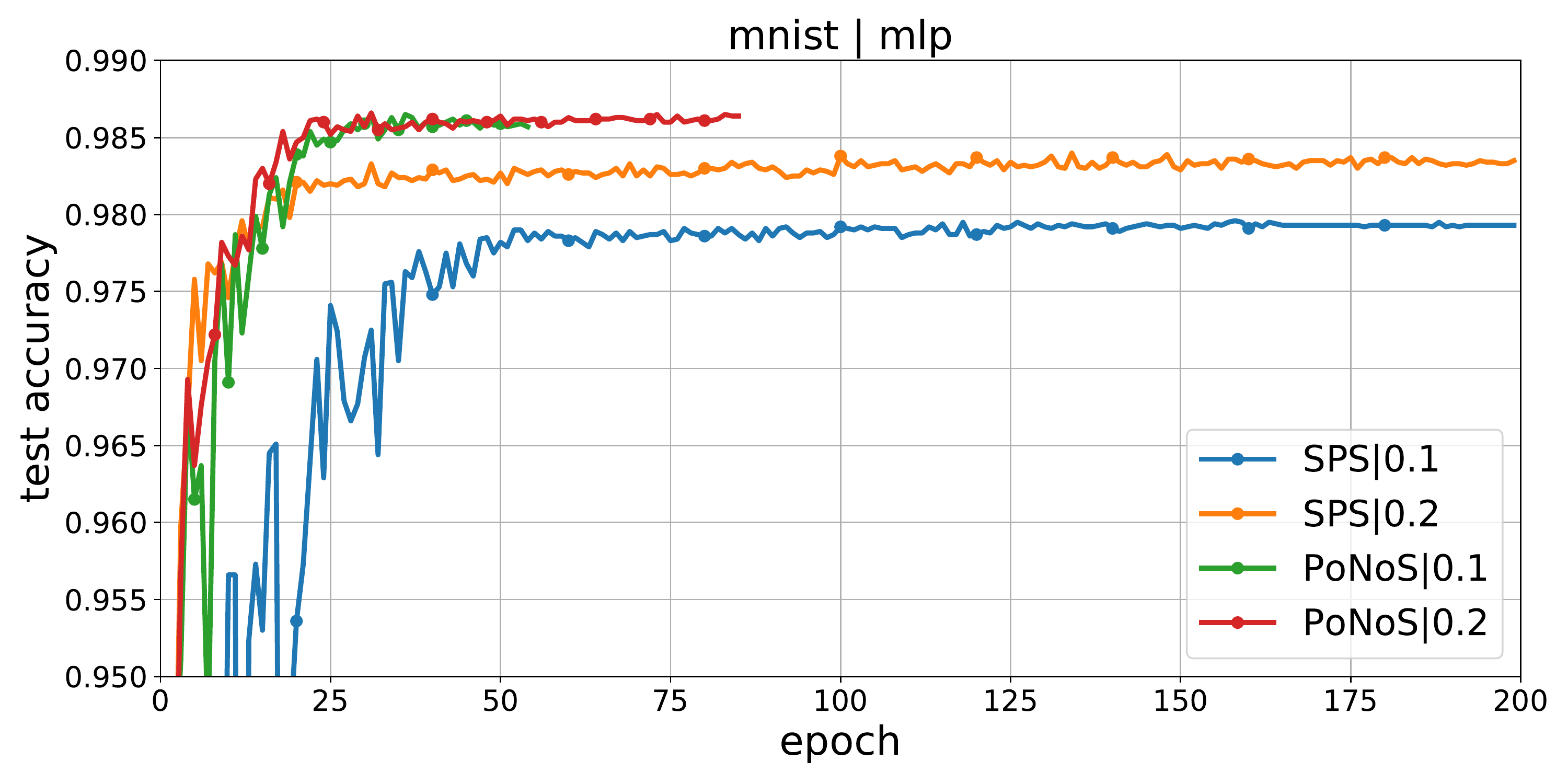}}
		\subfloat{\includegraphics[width=\imgS\linewidth]{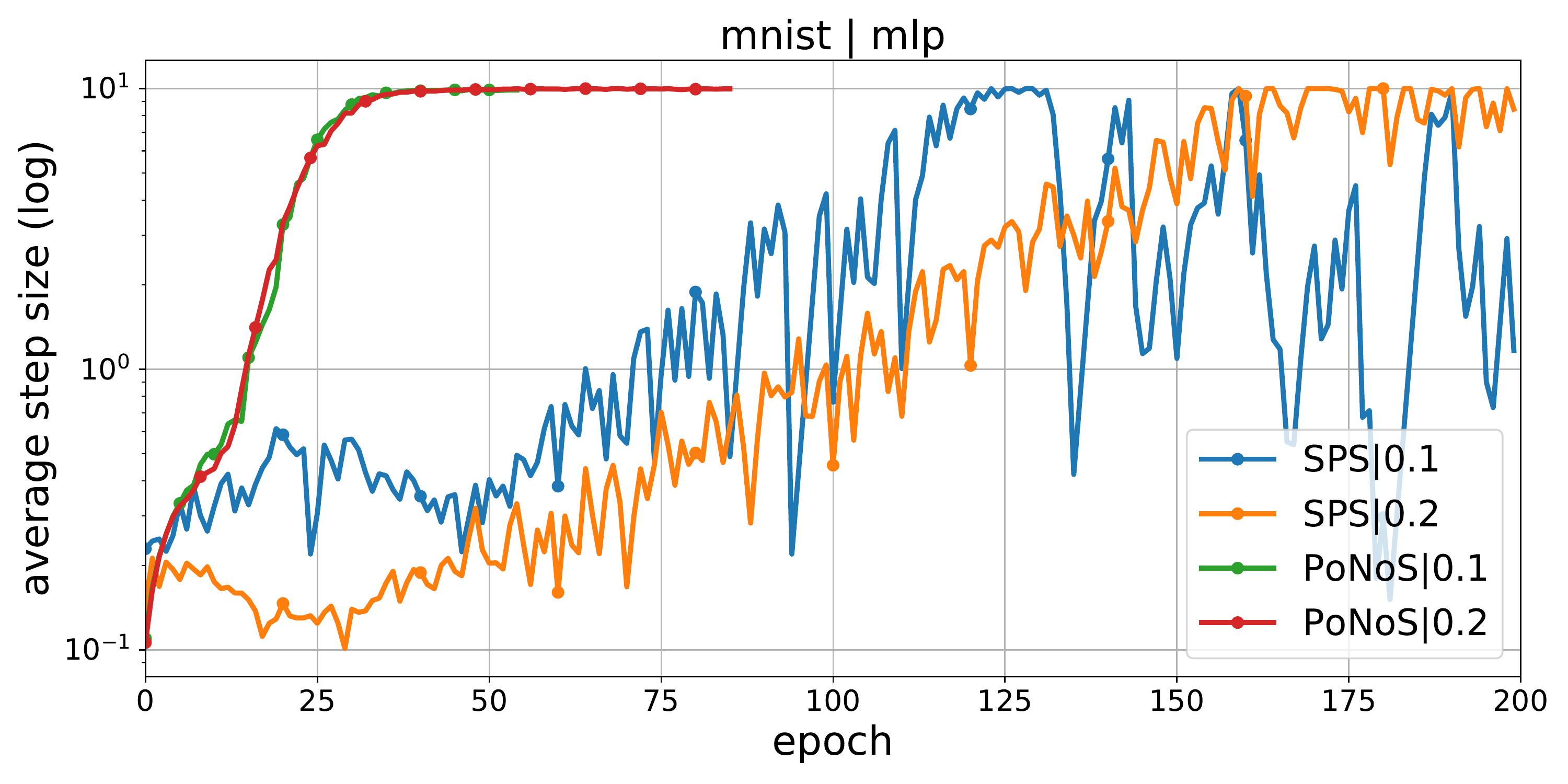}}\\
		\renewcommand{\model}{cifar10_resnet}
		\renewcommand{\modelname}{cifar10\_resnet}
		\subfloat{\includegraphics[width=\imgS\linewidth]{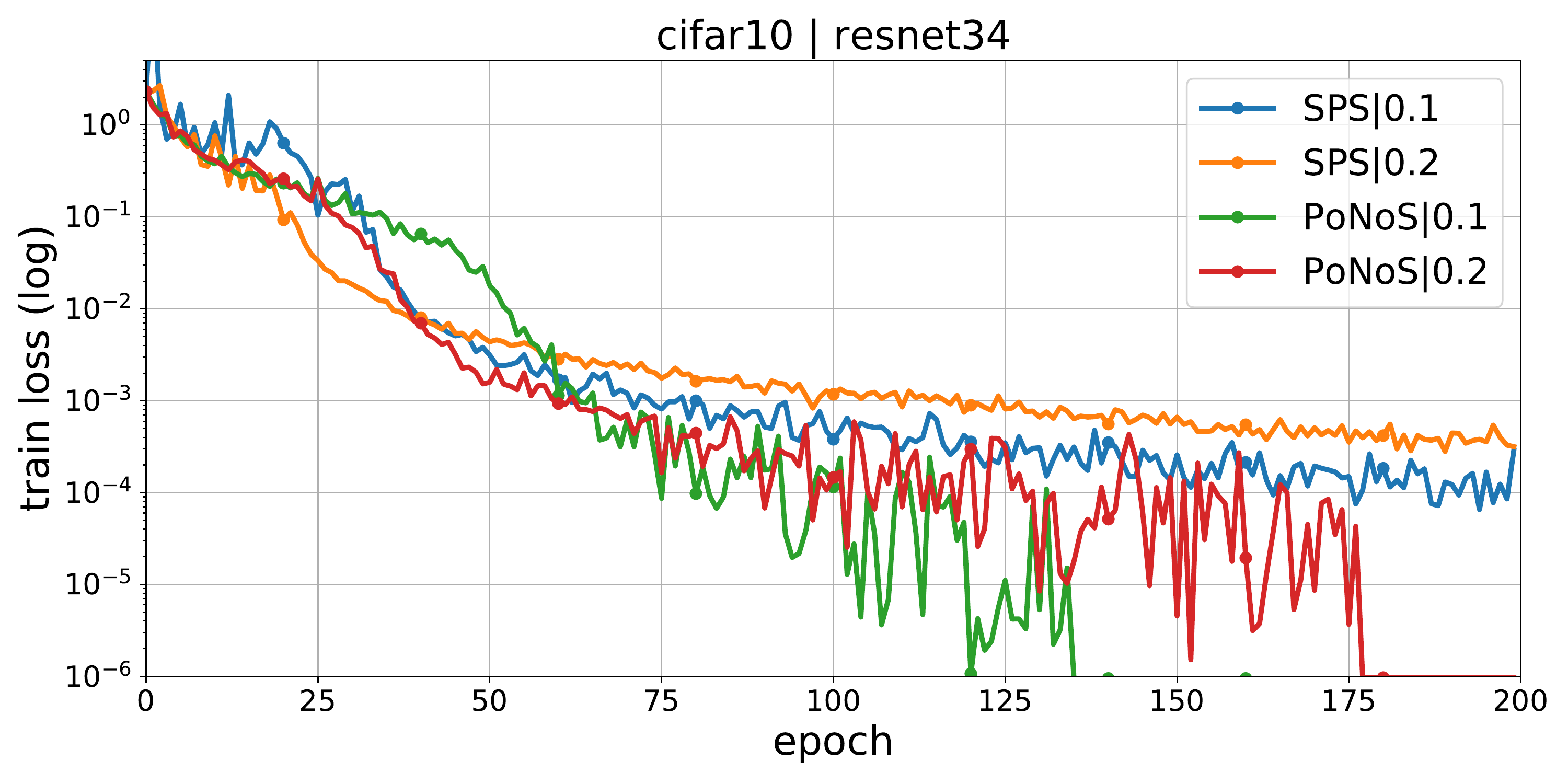}}
		\subfloat{\includegraphics[width=\imgS\linewidth]{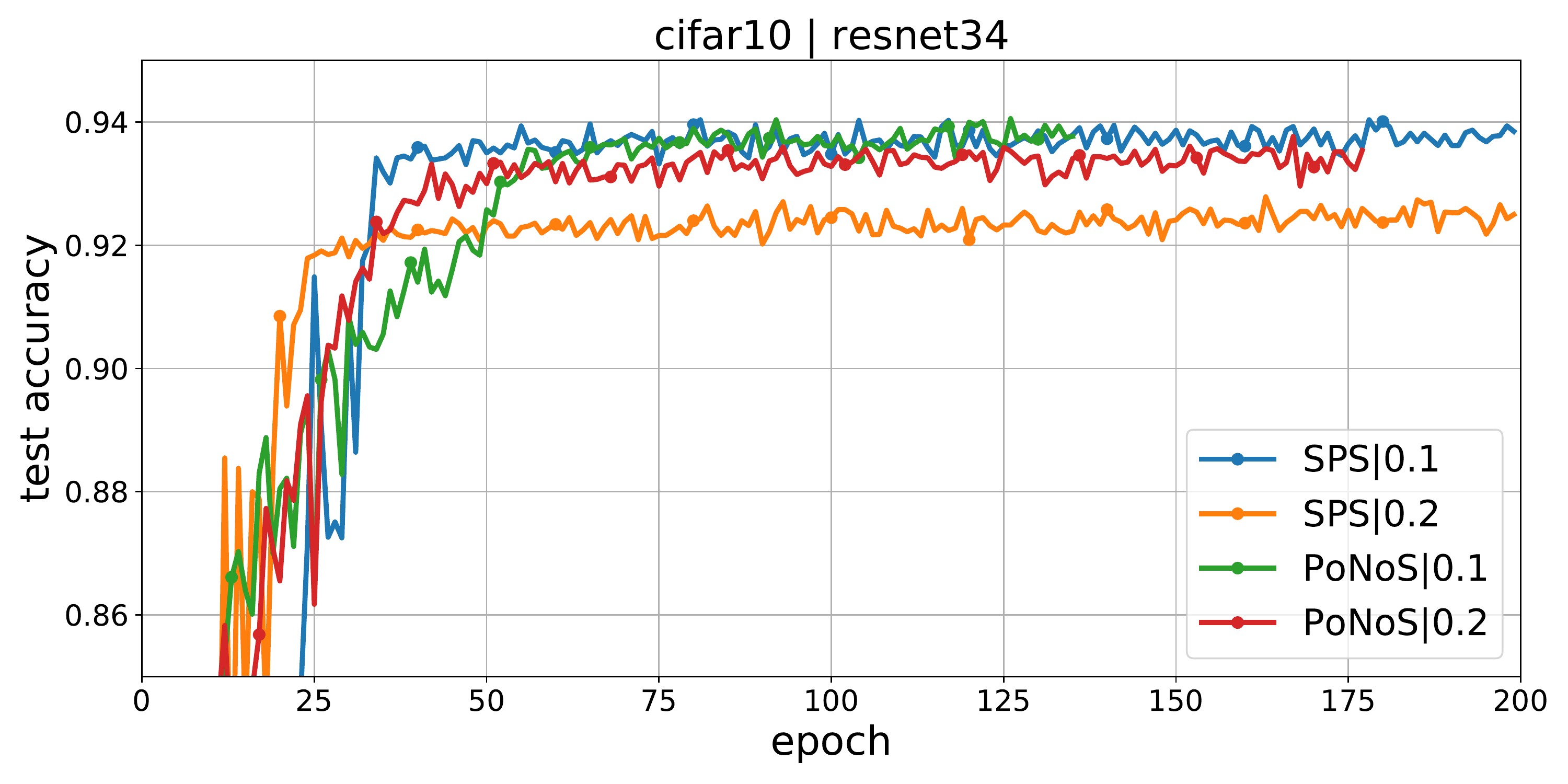}}
		\subfloat{\includegraphics[width=\imgS\linewidth]{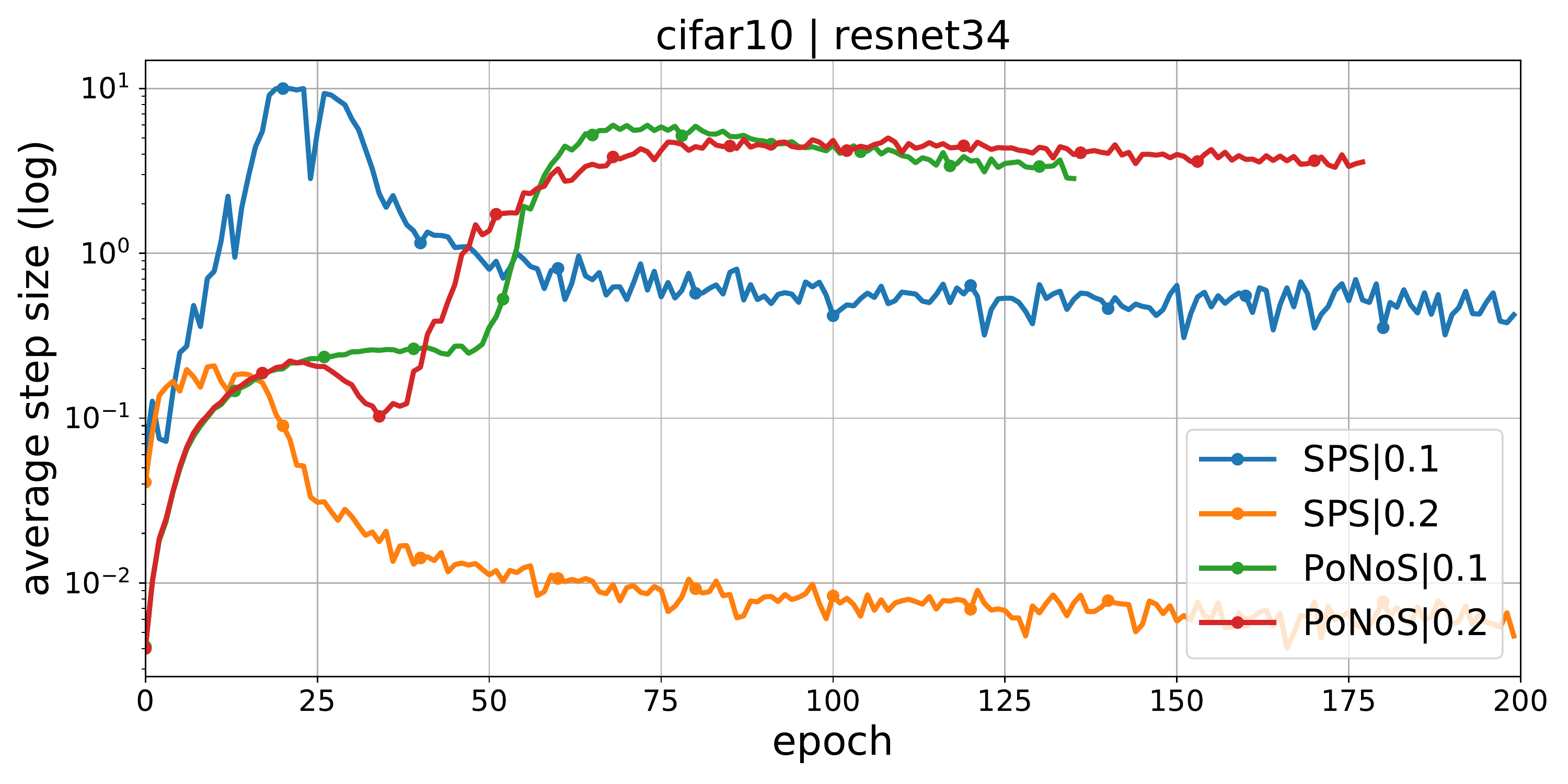}}\\
		\renewcommand{\model}{cifar10_densenet}
		\renewcommand{\modelname}{cifar10\_densenet}
		\subfloat{\includegraphics[width=\imgS\linewidth]{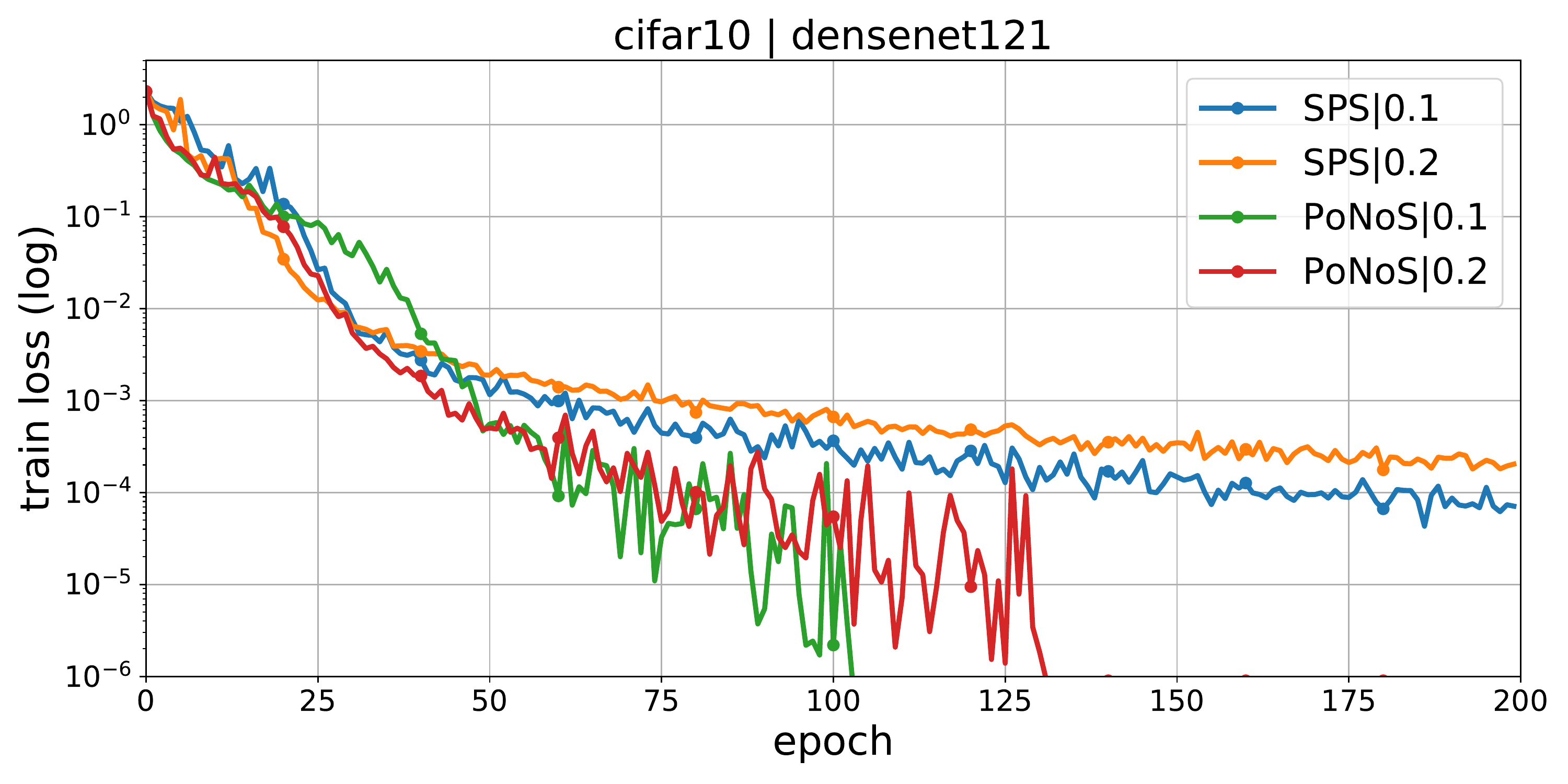}}
		\subfloat{\includegraphics[width=\imgS\linewidth]{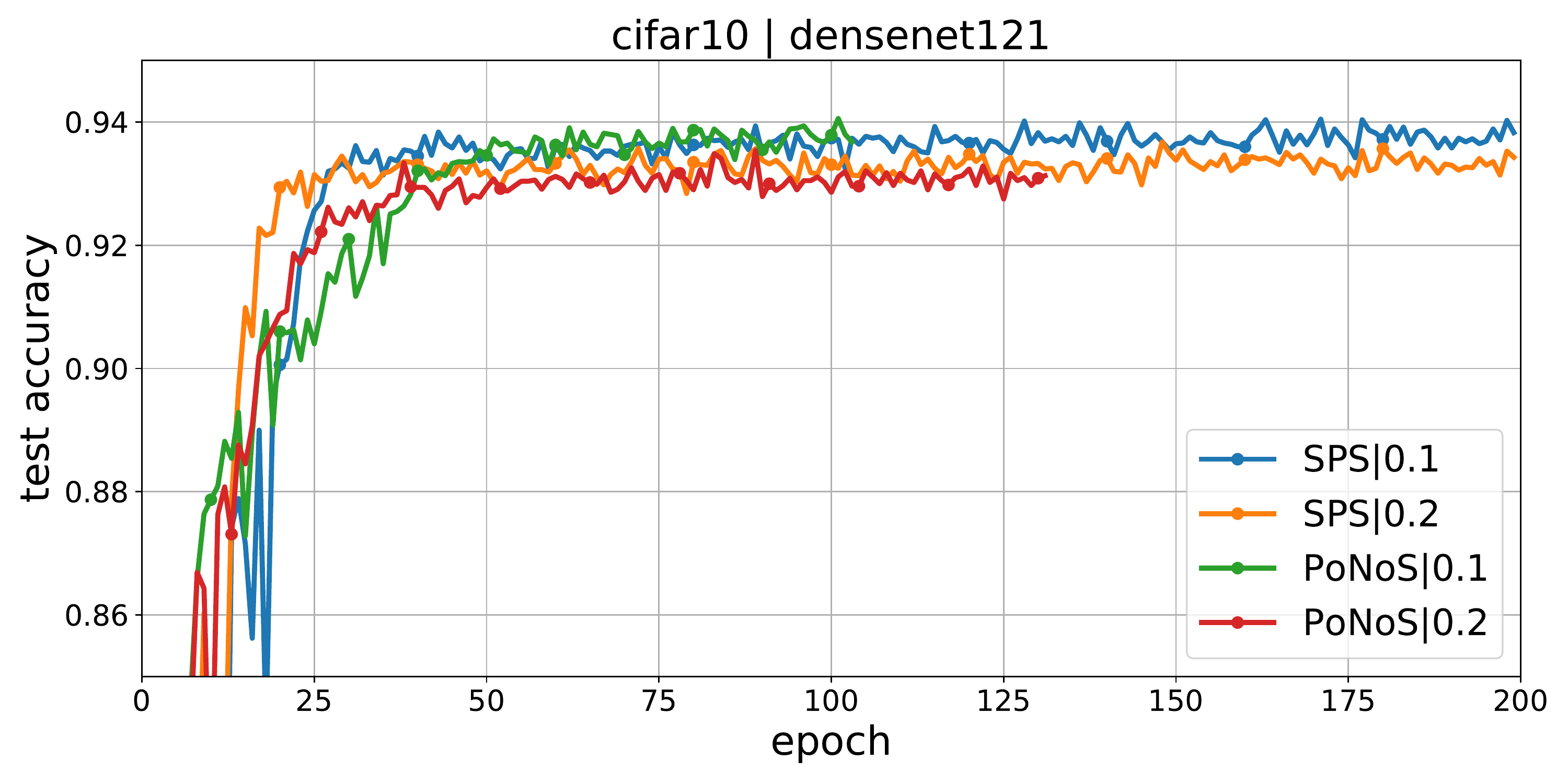}}
		\subfloat{\includegraphics[width=\imgS\linewidth]{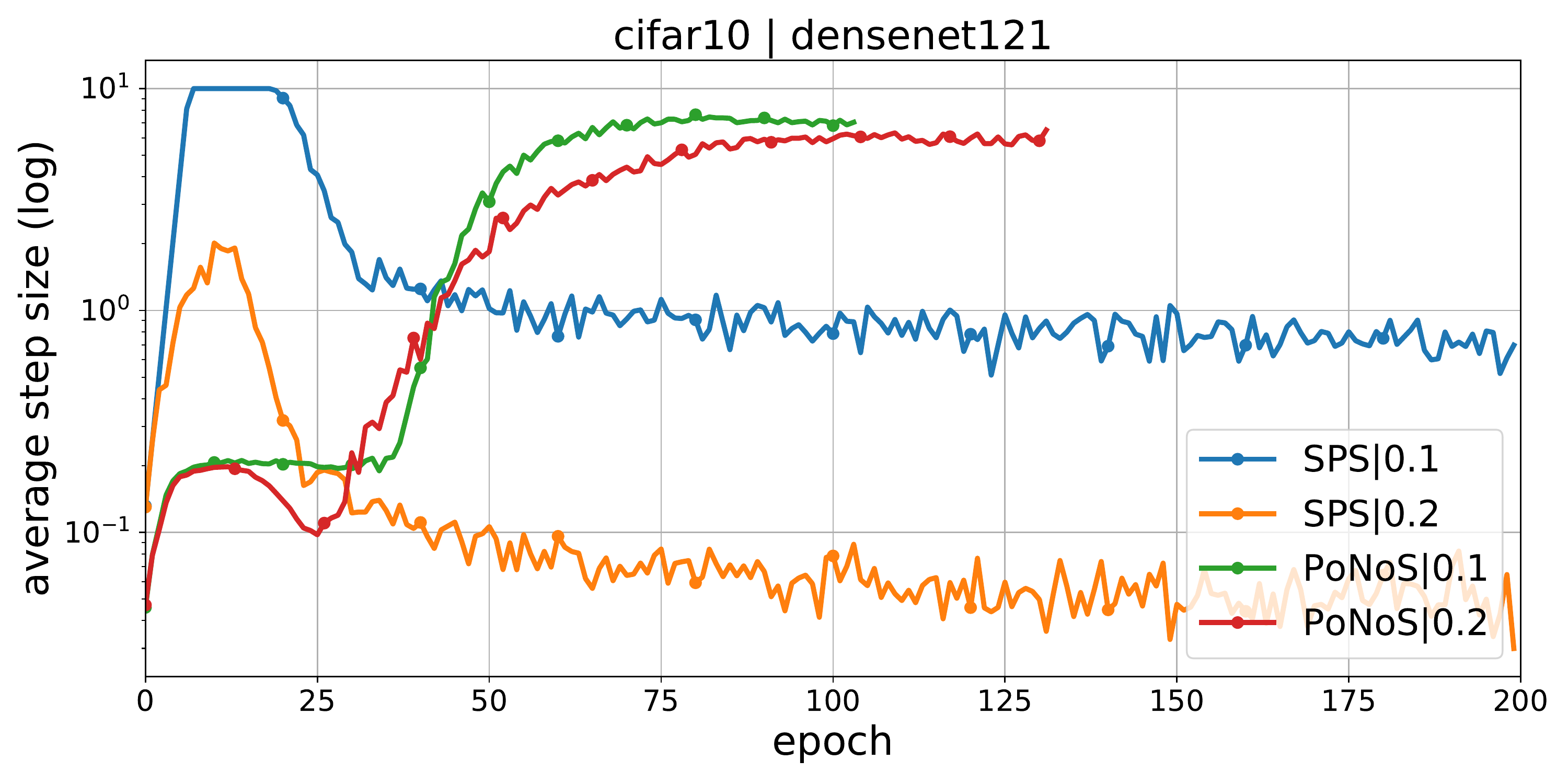}}\\
		\renewcommand{\model}{cifar100_resnet}
		\renewcommand{\modelname}{cifar100\_resnet}
		\subfloat{\includegraphics[width=\imgS\linewidth]{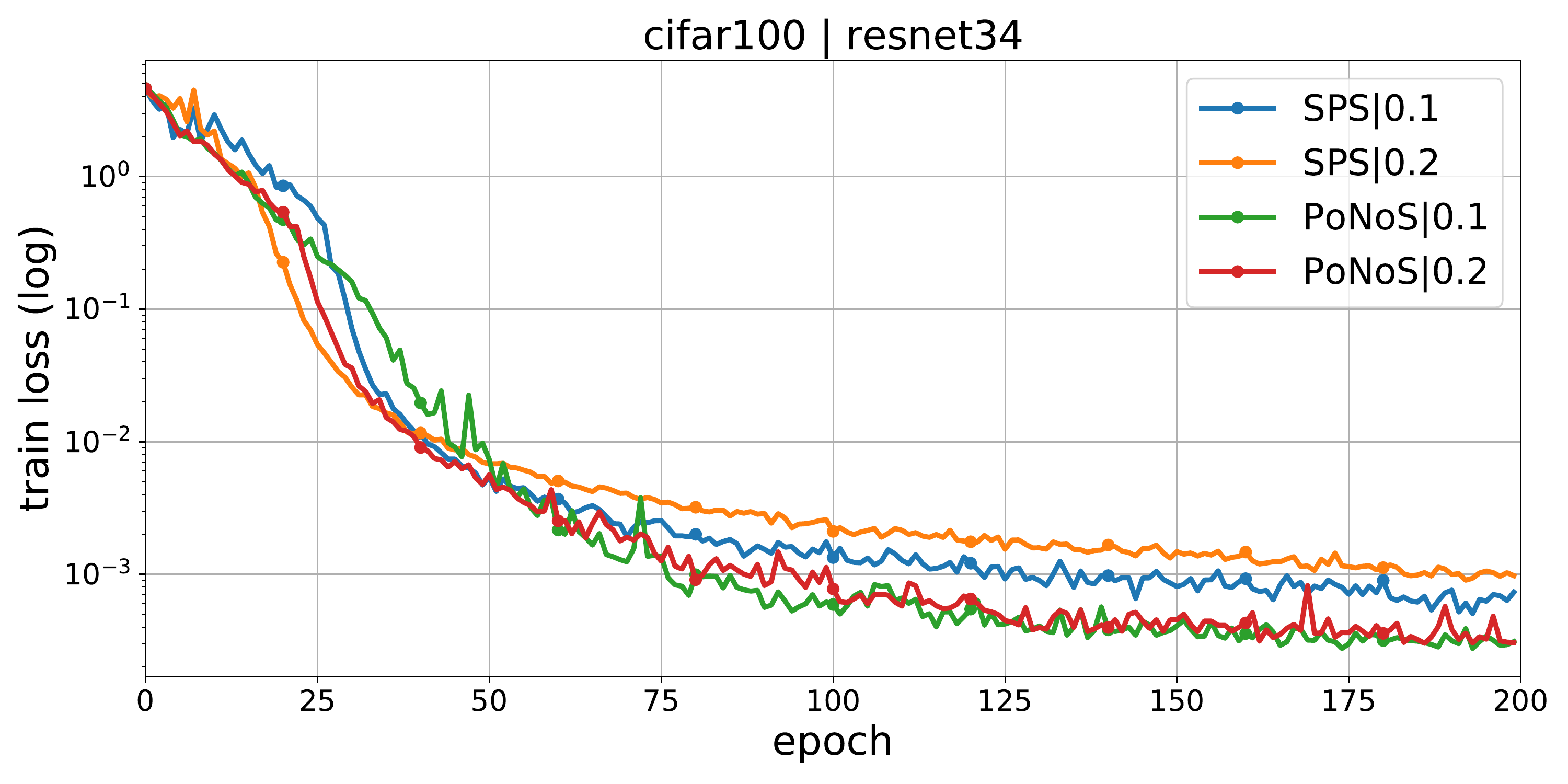}}
		\subfloat{\includegraphics[width=\imgS\linewidth]{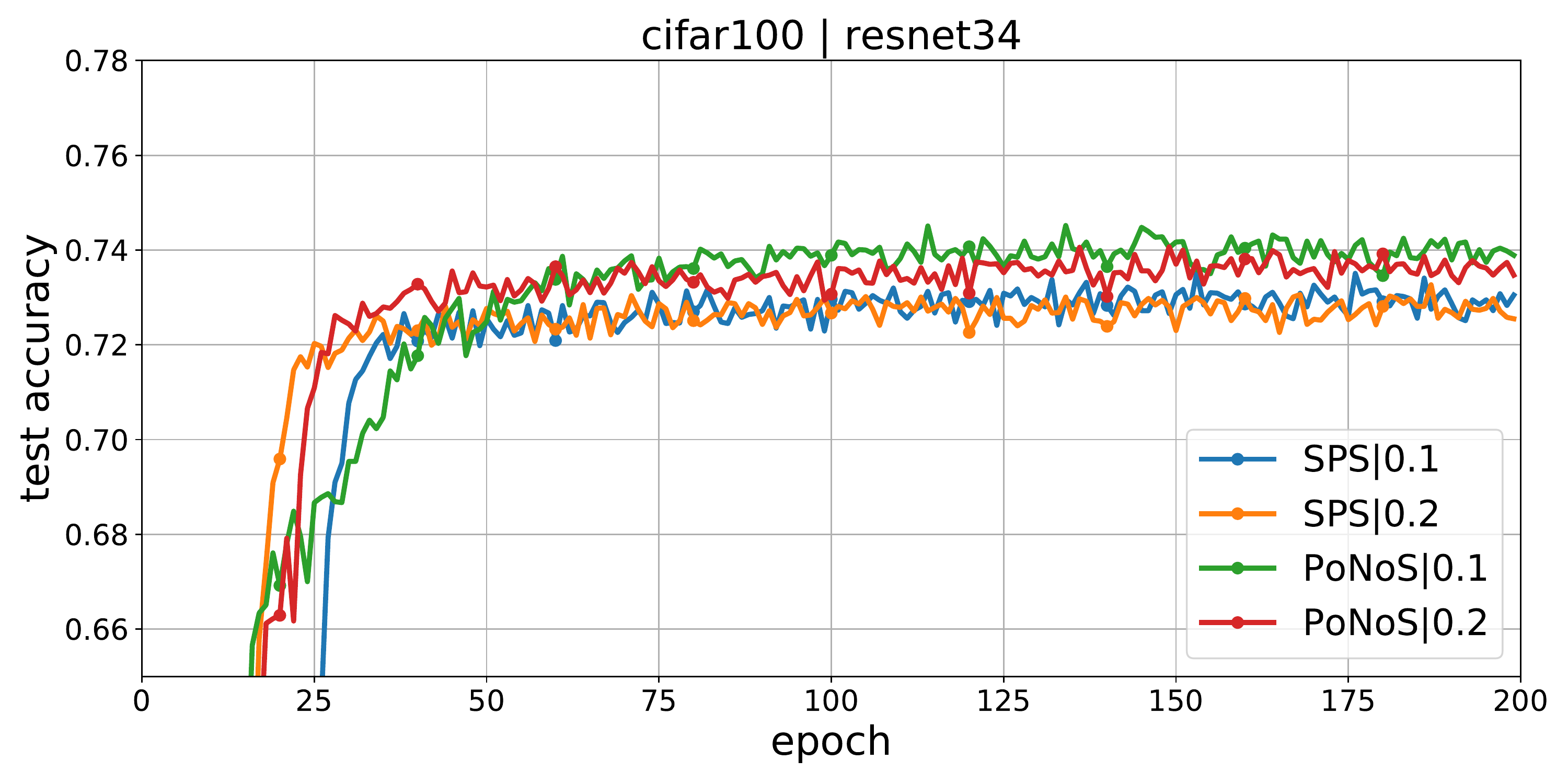}}
		\subfloat{\includegraphics[width=\imgS\linewidth]{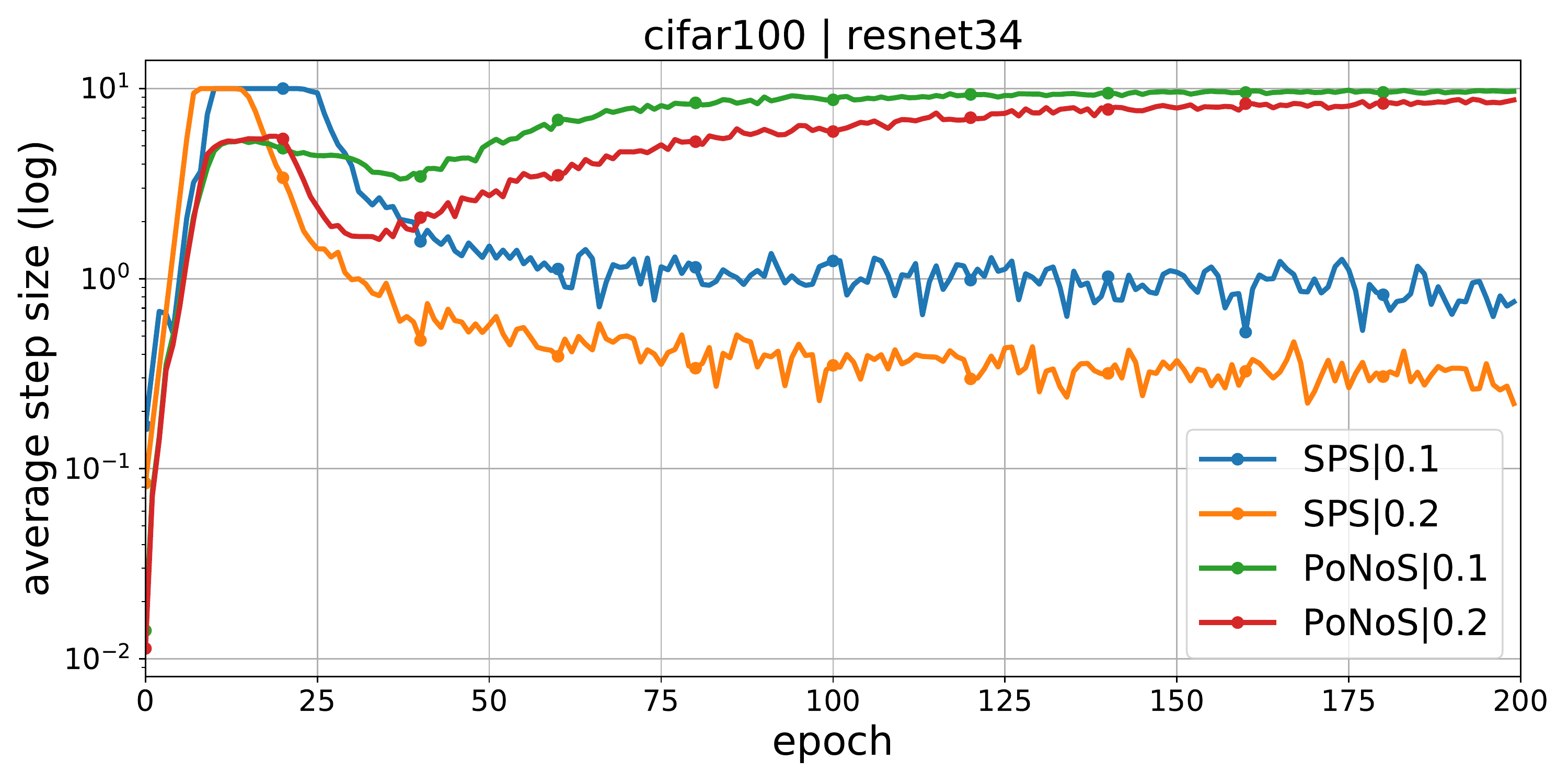}}\\
		\renewcommand{\model}{cifar100_densenet}
		\renewcommand{\modelname}{cifar100\_densenet}
		\subfloat{\includegraphics[width=\imgS\linewidth]{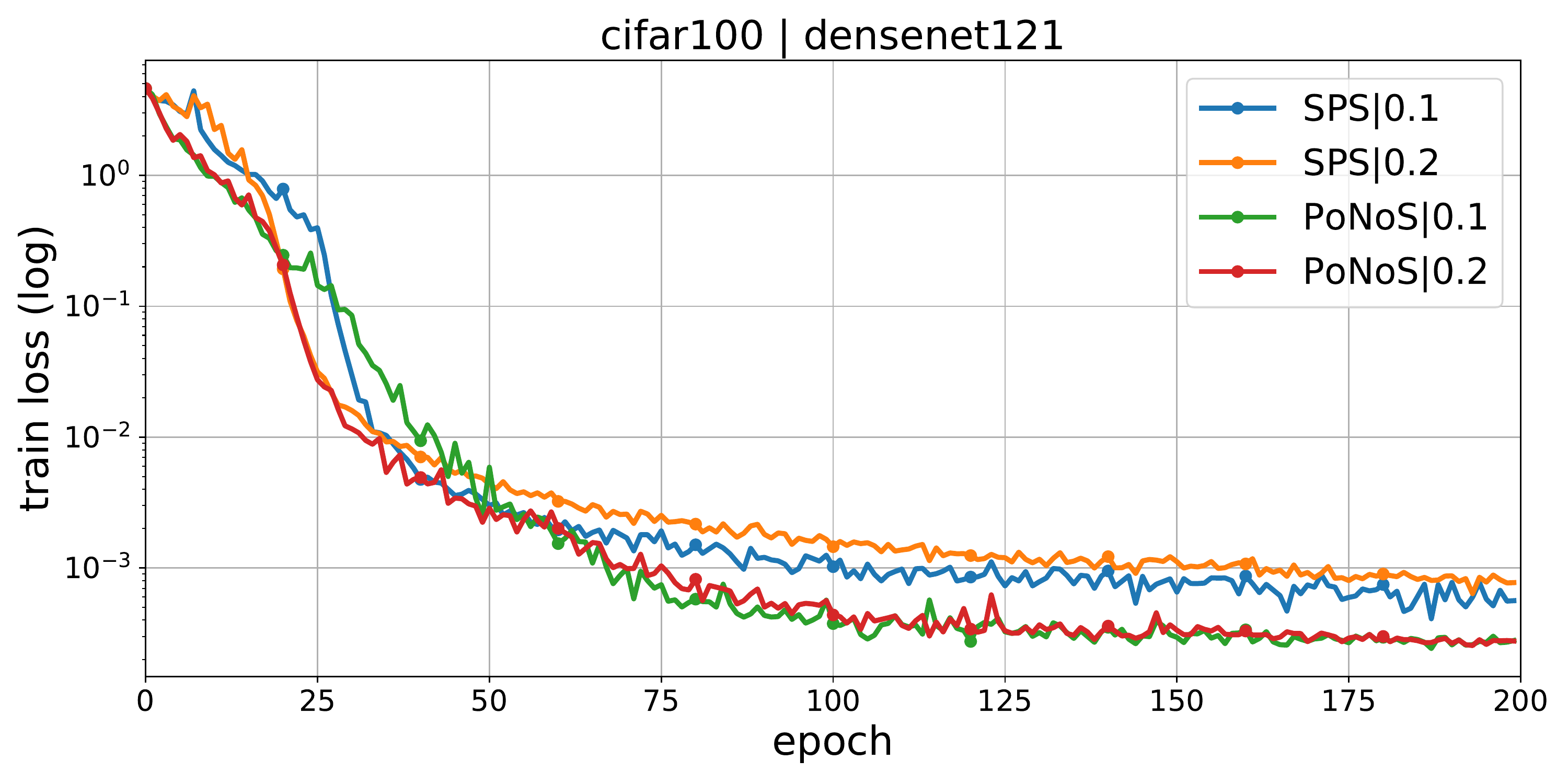}}
		\subfloat{\includegraphics[width=\imgS\linewidth]{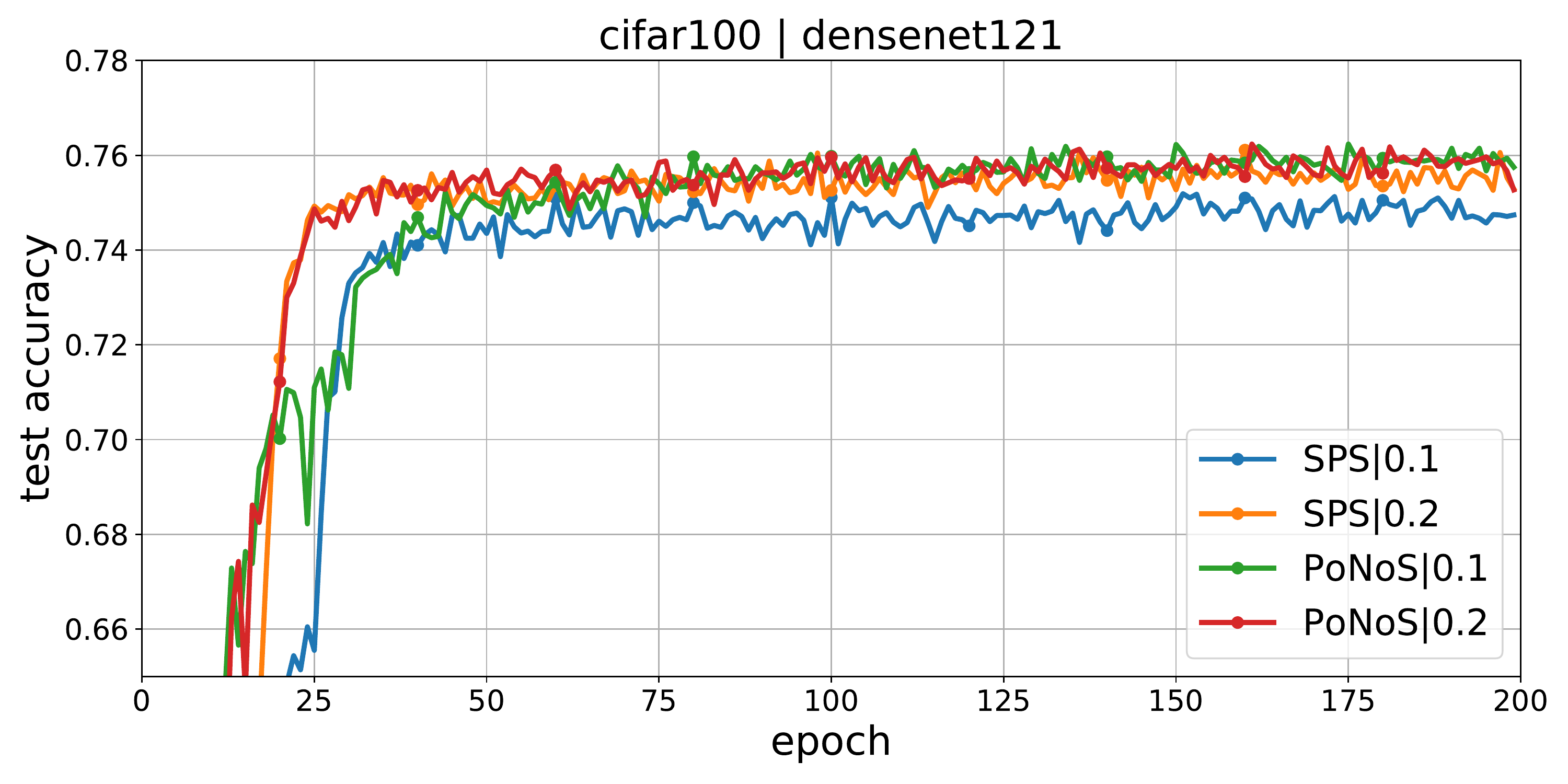}}
		\subfloat{\includegraphics[width=\imgS\linewidth]{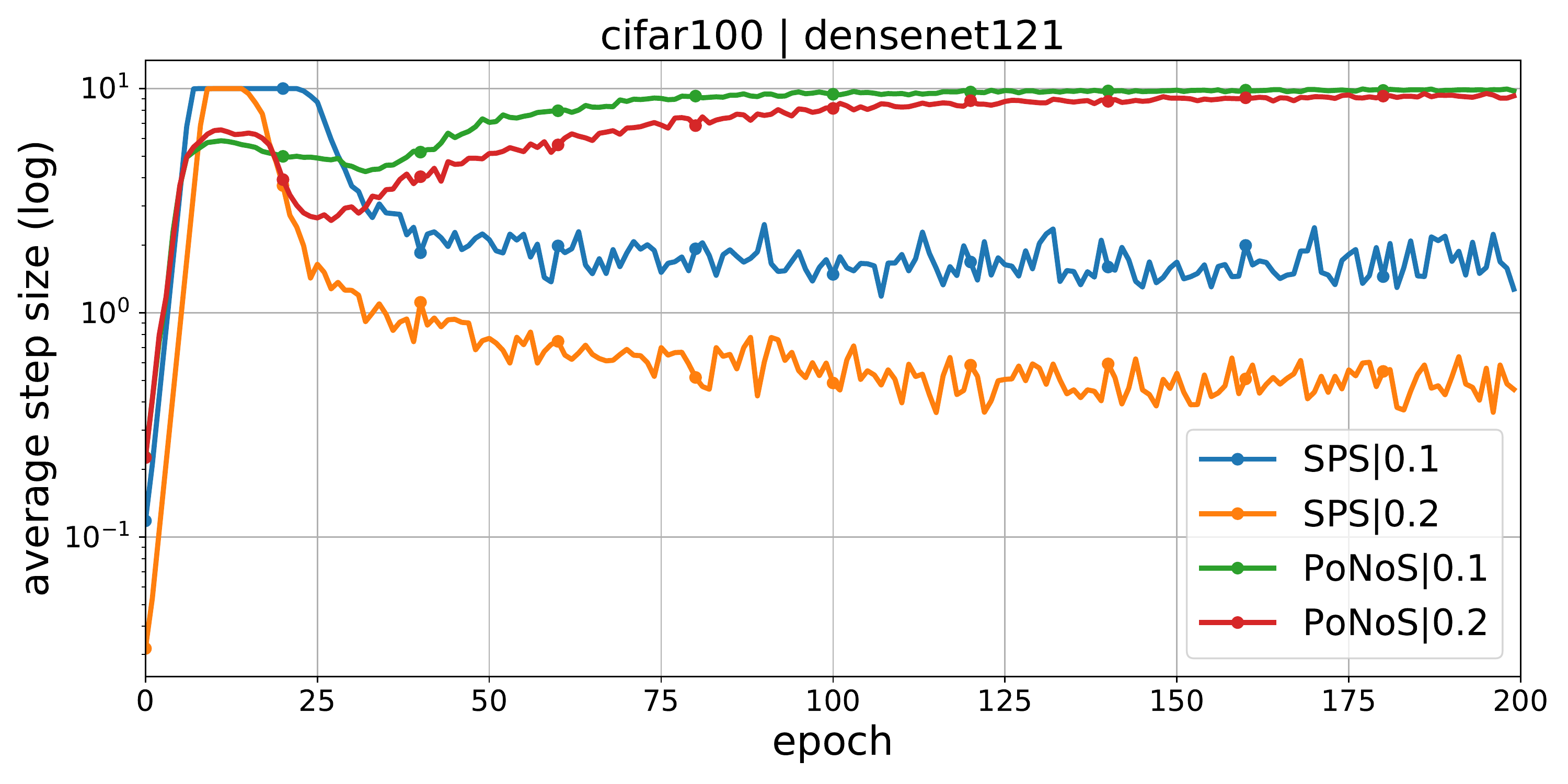}}\\
		\renewcommand{\model}{fashion_effb1}
		\renewcommand{\modelname}{fashion\_effb1}
		\subfloat{\includegraphics[width=\imgS\linewidth]{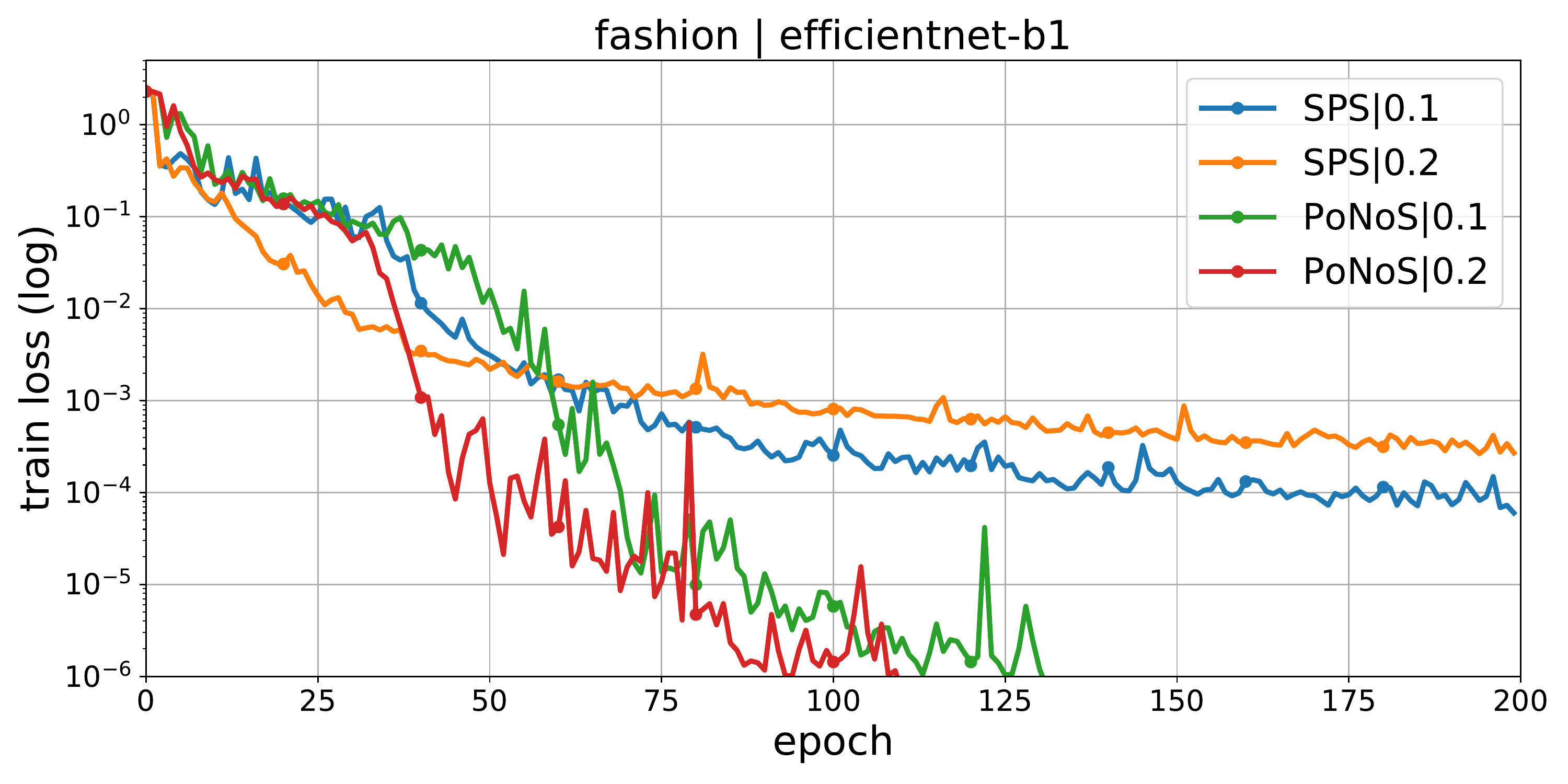}}
		\subfloat{\includegraphics[width=\imgS\linewidth]{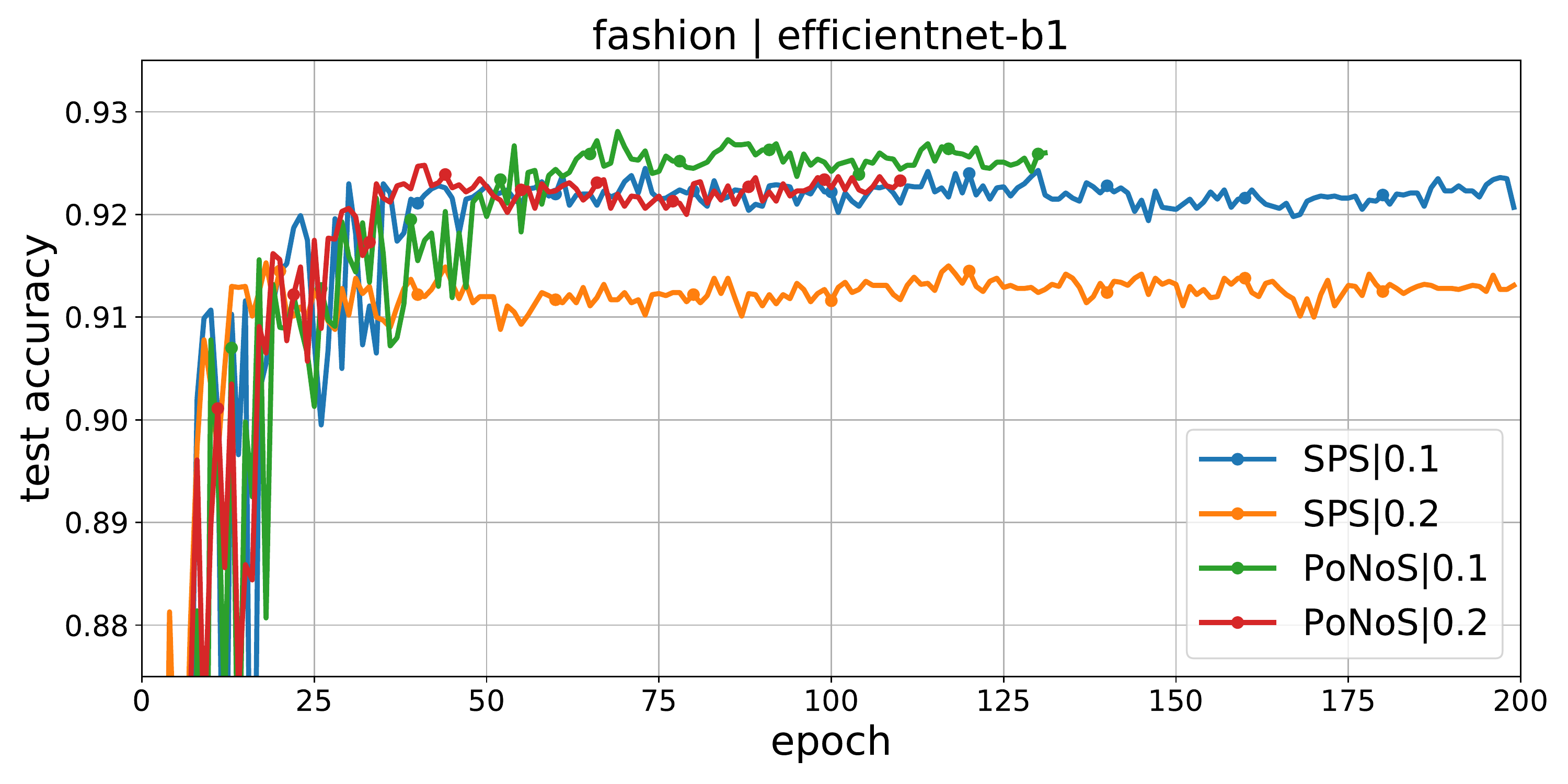}}
		\subfloat{\includegraphics[width=\imgS\linewidth]{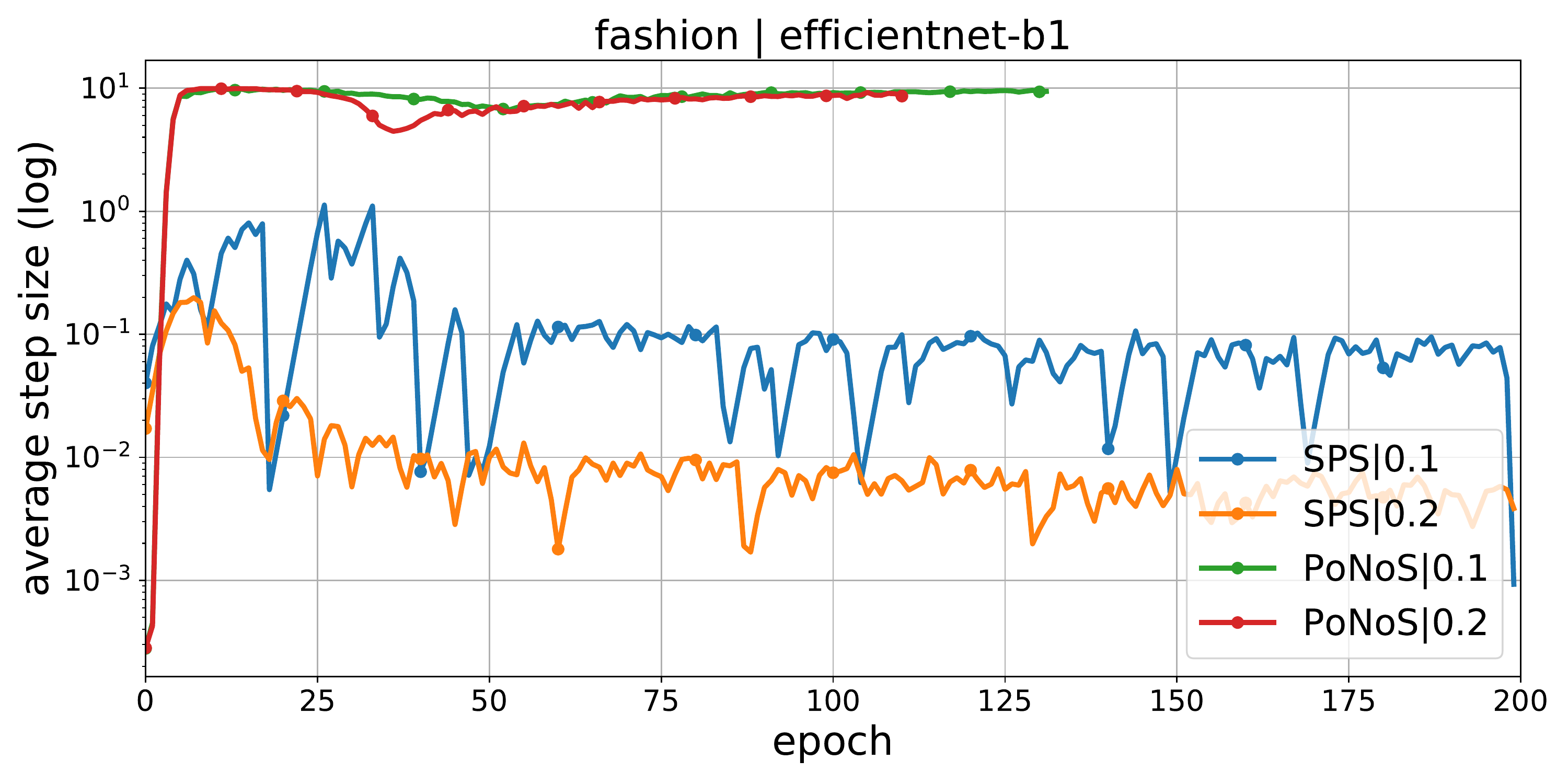}}\\
		\renewcommand{\model}{svhn_wrn}
		\renewcommand{\modelname}{svhn\_wrn}
		\subfloat{\includegraphics[width=\imgS\linewidth]{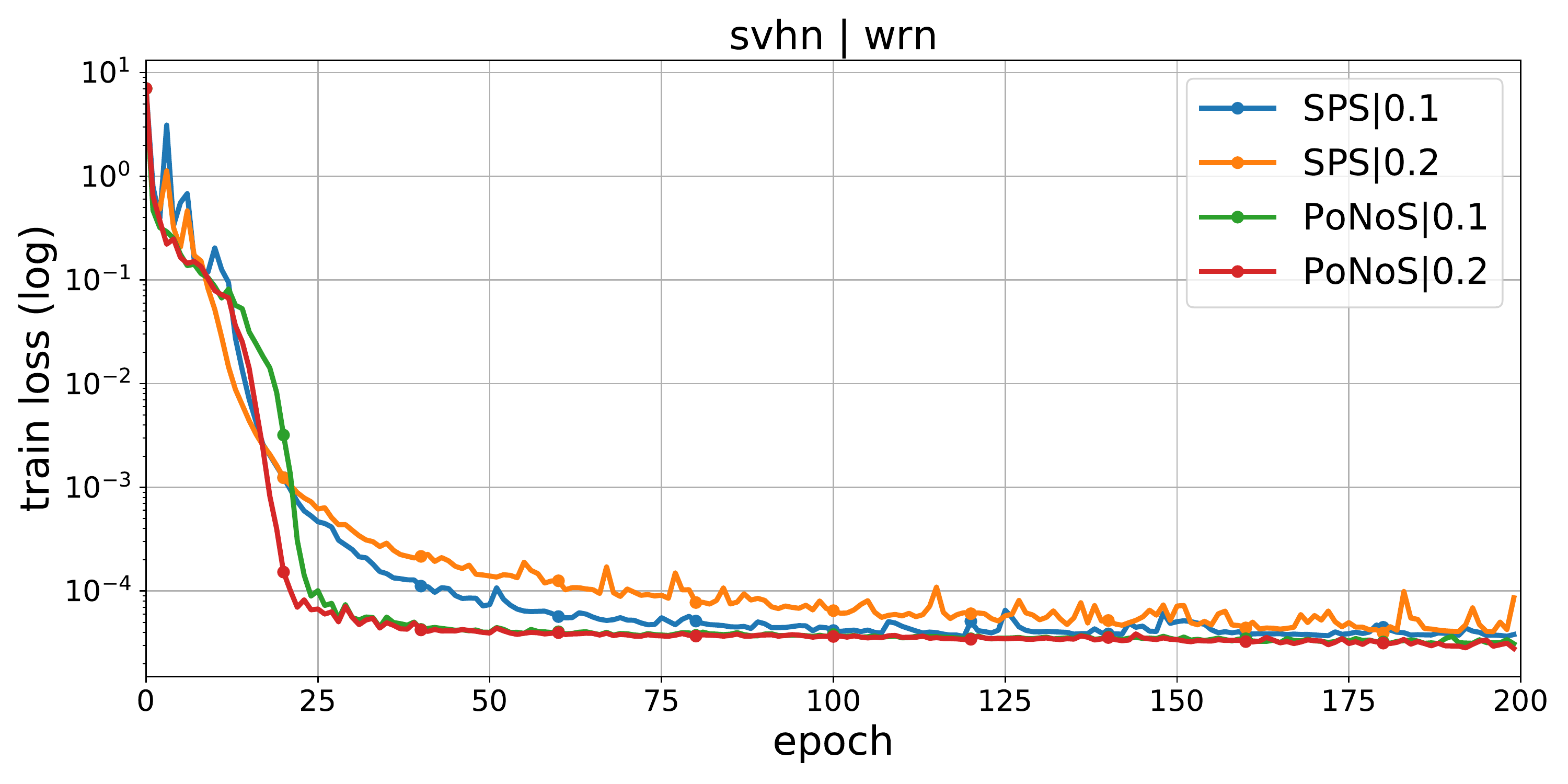}}
		\subfloat{\includegraphics[width=\imgS\linewidth]{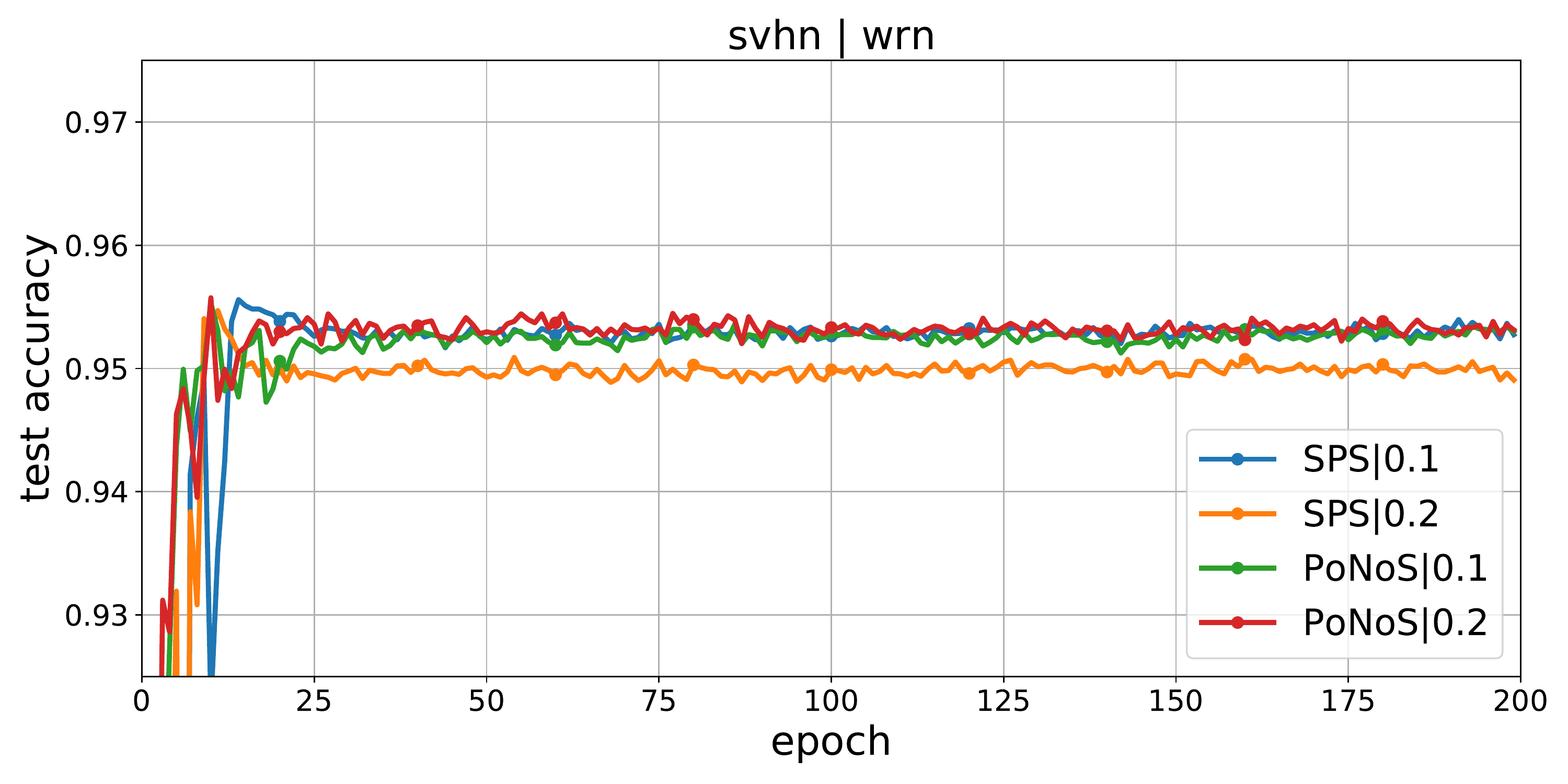}}
		\subfloat{\includegraphics[width=\imgS\linewidth]{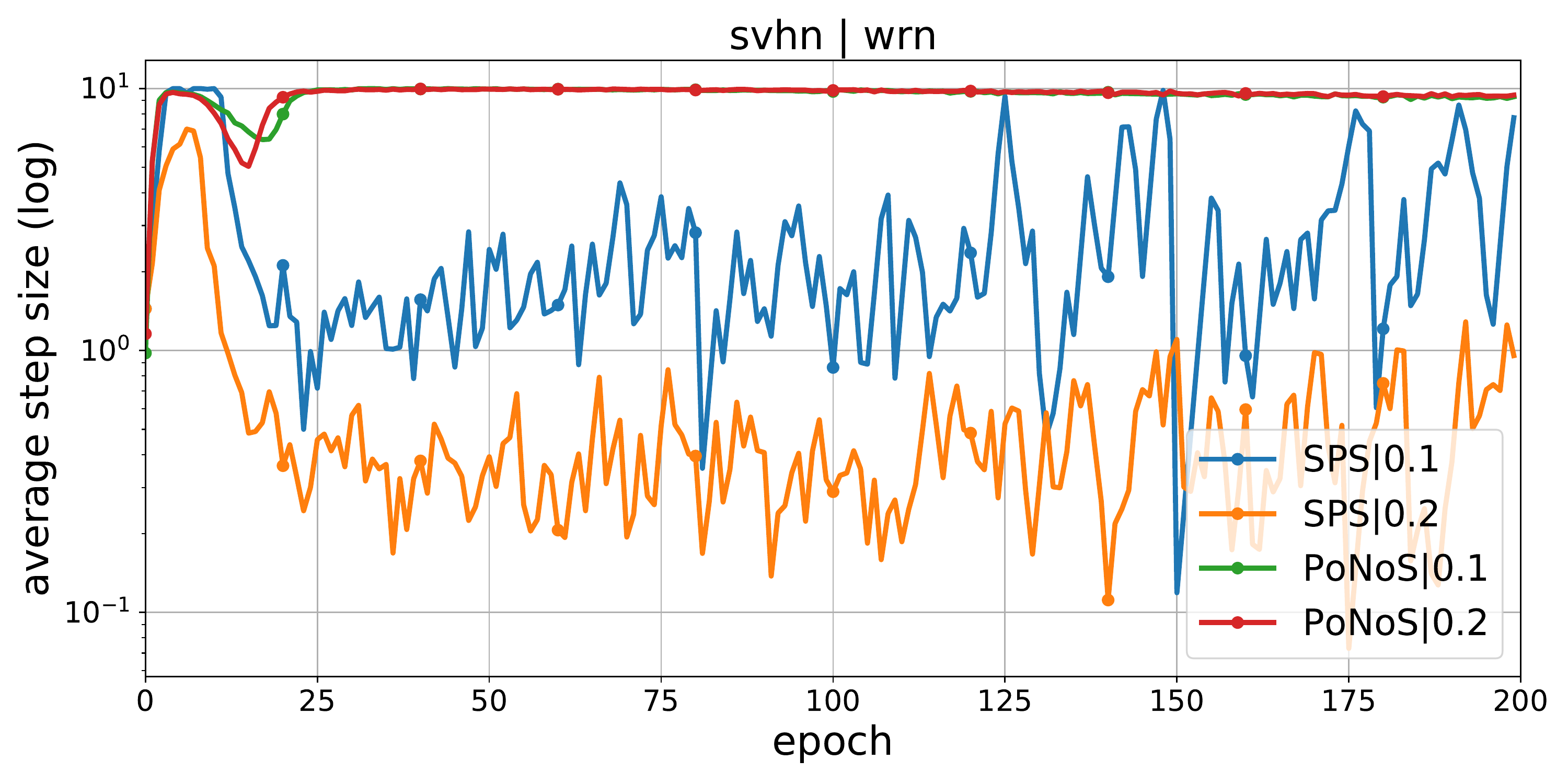}}
		\caption{Comparison between the use of the constant $c_p=0.1$ and $c_p=0.2$ on PoNoS and SPS. Each row focus on a dataset/model combination. First column: train loss. Second column: test accuracy. Third column: average step size of the epoch.}\label{fig:c_p}
	\end{minipage}
\end{figure}
	\subsection{Profiling PoNoS}\label{sec:profiling}
	In this subsection, we profile a single mini-batch iteration of PoNoS. More precisely, in Figure \ref{fig:profiling} we report the time (s) employed by the different operations required to perform an iteration of PoNoS. Averaging along the epoch (over $\frac{M}{b}$ values), in Figure \ref{fig:profiling} we compare 
\begin{itemize}
\item 1st\_fwd: average time employed to perform the first forward pass;
\item 2nd\_fwd: average time employed to perform the second forward pass;
\item extra\_fwd: average time employed to perform all the extra forward passes beyond the second. Notice that the total amount of extra forward passes could be more or less than $\frac{M}{b}$. Despite the actual amount of extra forward passes, we still divide the sum by $\frac{M}{b}$;
\item backward: average time employed to perform the backward pass;
\item batch\_load: average time employed to load a mini-batch of instances into the GPU memory;
\end{itemize}
From Figure \ref{fig:profiling}, we can observe the following.
\begin{itemize}
	\item 1st\_fwd and 2nd\_fwd employ roughly the same time. 2nd\_fwd is always slightly faster than 1st\_fwd, however, the difference is not substantial. 
	\item A backward pass employs roughly twice the time of a forward pass.
	\item The time employed to load the mini-batch in the GPU memory is negligible on some networks (\dense, \denses and \fashion). On the other networks, this time is between one-half and 4 times that of a forward pass. 
	\item As previously shown in Figure \ref{fig:f_eval}, PoNoS transitions from a phase in which it employs 1 backtrack to 0 (on median). In Figure \ref{fig:profiling}, this same behavior can be also observed in the time employed by all the forward passes beyond the second.
\end{itemize}
To conclude, the time employed by a single forward pass may reach up to one-third of the computation required for SGD (i.e., batch\_load + 1st\_fwd + backward). Following these calculations and referring to the two phases of Figure \ref{fig:time} and Figure \ref{fig:profiling}, one iteration of PoNoS only costs $\frac{5}{3}$ that of SGD in the first phase and $\frac{4}{3}$ in the second.

\renewcommand{\dir}{fwd_times/}
\renewcommand{\imgS}{0.42}
\begin{figure}[!h] 
	\begin{minipage}{1\textwidth}
		\centering
		\renewcommand{\model}{mlp}
		\renewcommand{\modelname}{mlp}
		\subfloat{\includegraphics[width=\imgS\linewidth]{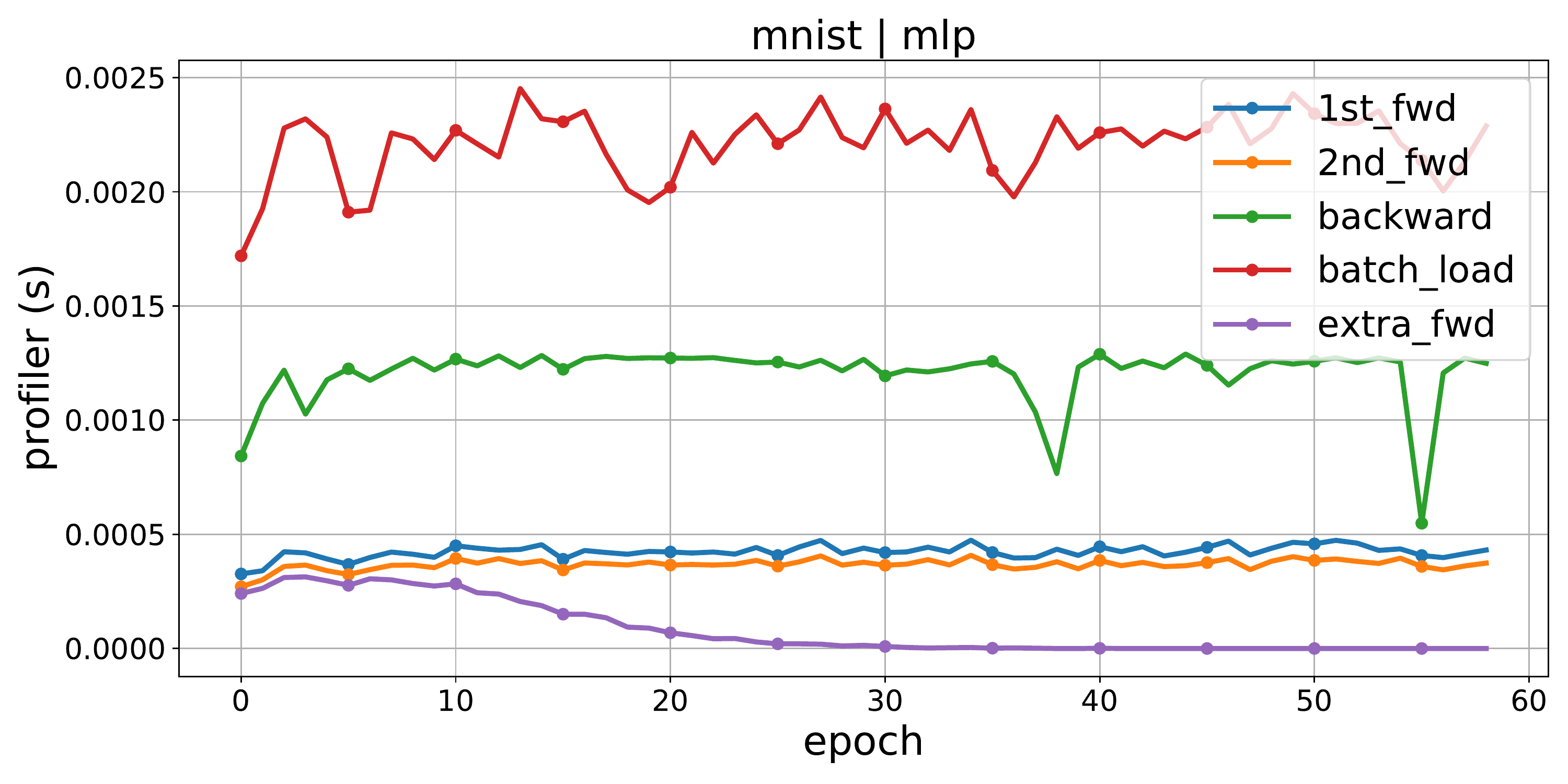}}
		\renewcommand{\model}{cifar10_resnet}
		\renewcommand{\modelname}{cifar10\_resnet}
		\subfloat{\includegraphics[width=\imgS\linewidth]{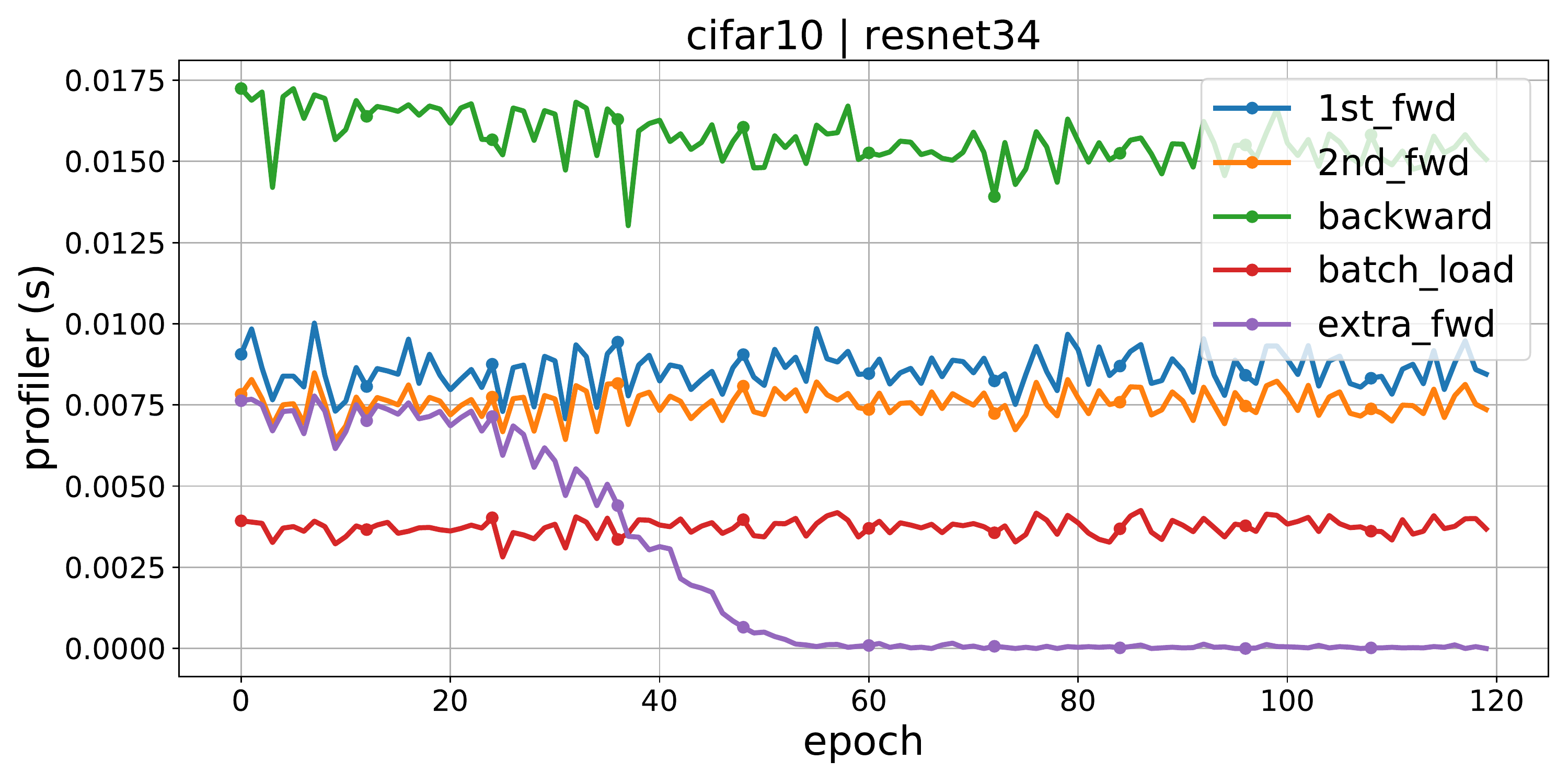}}
		\renewcommand{\model}{cifar10_densenet}
		\renewcommand{\modelname}{cifar10\_densenet}
		\subfloat{\includegraphics[width=\imgS\linewidth]{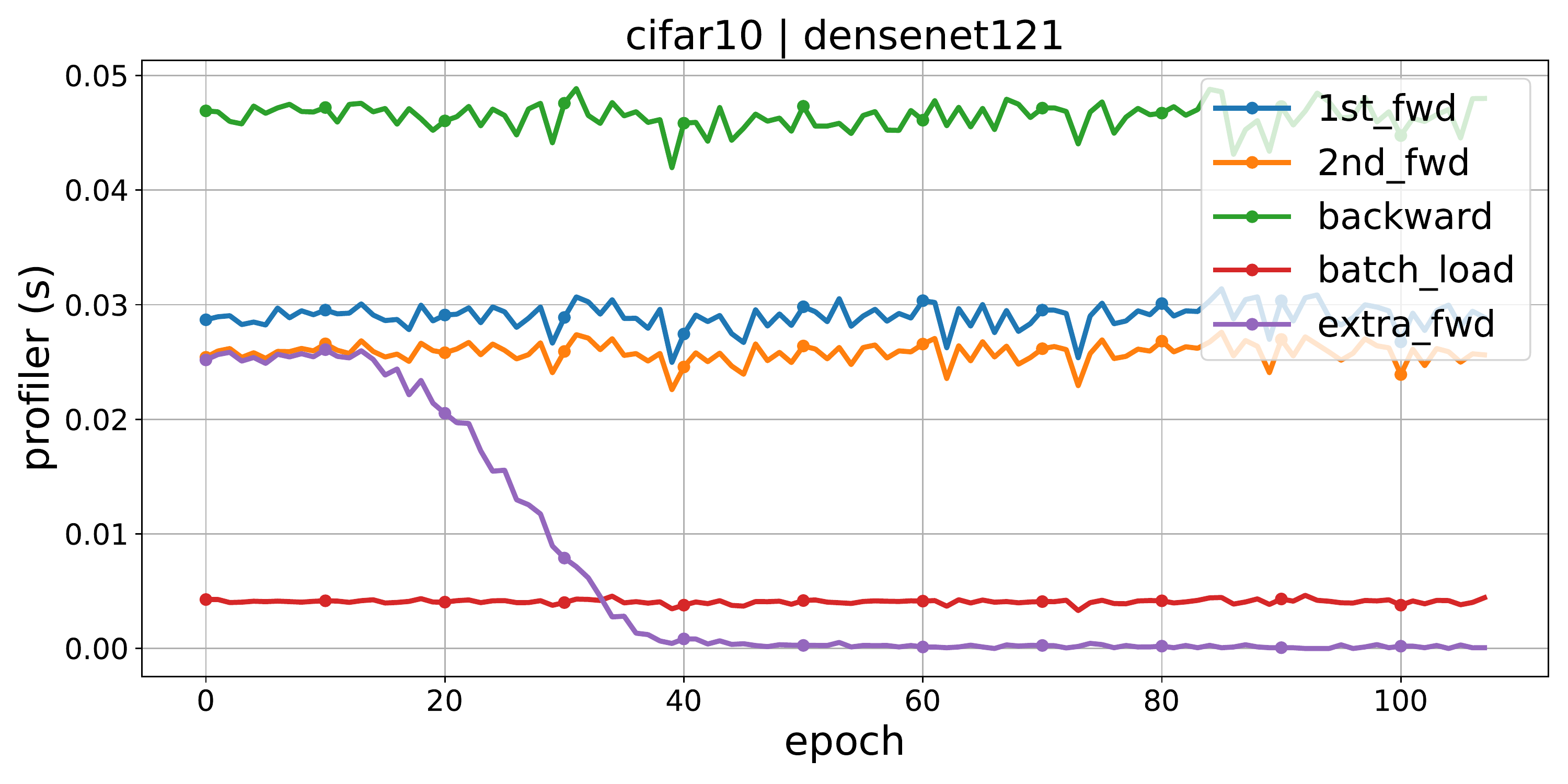}}
		\renewcommand{\model}{cifar100_resnet}
		\renewcommand{\modelname}{cifar100\_resnet}
		\subfloat{\includegraphics[width=\imgS\linewidth]{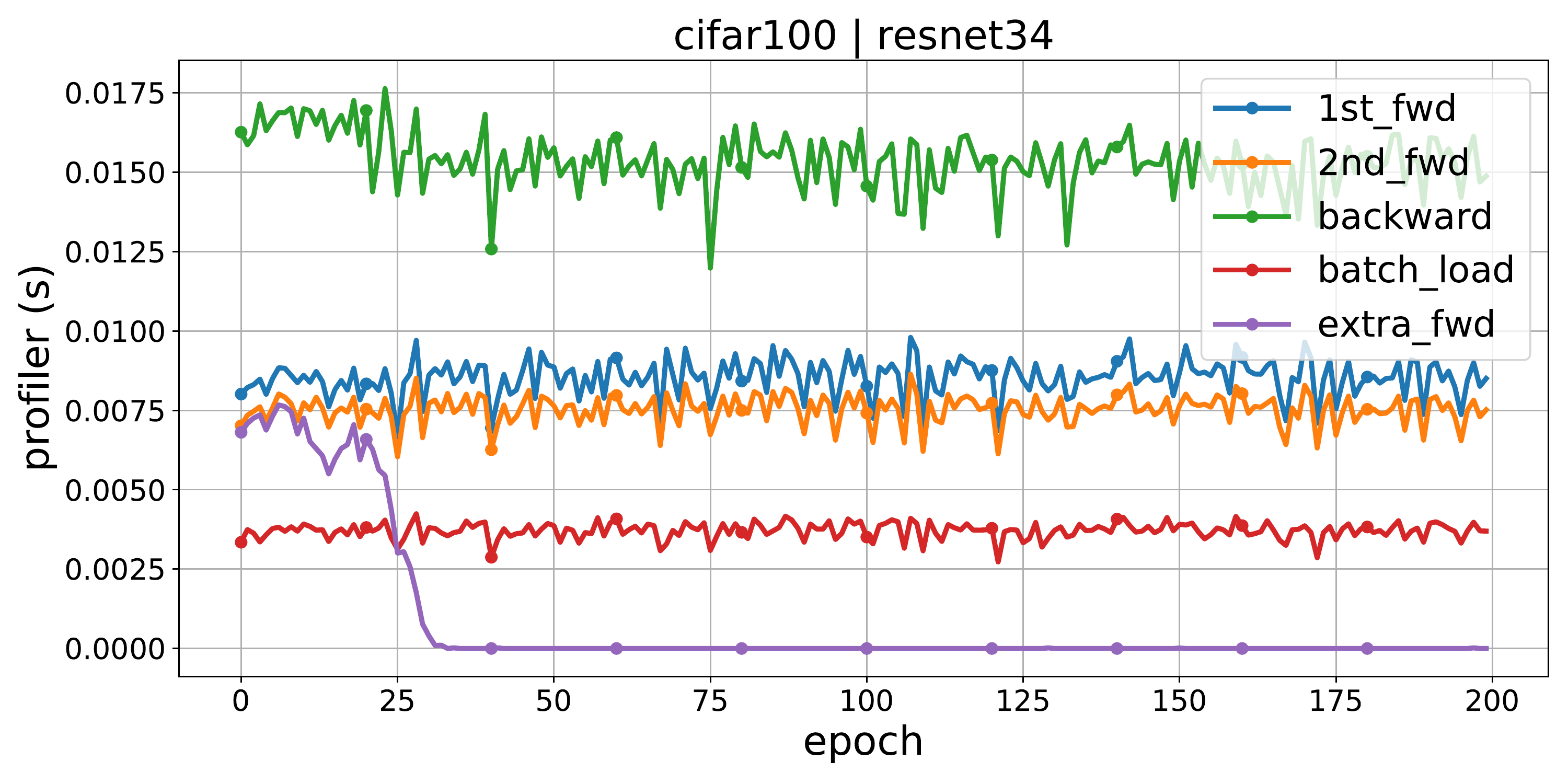}}\\
		\renewcommand{\model}{cifar100_densenet}
		\renewcommand{\modelname}{cifar100\_densenet}
		\subfloat{\includegraphics[width=\imgS\linewidth]{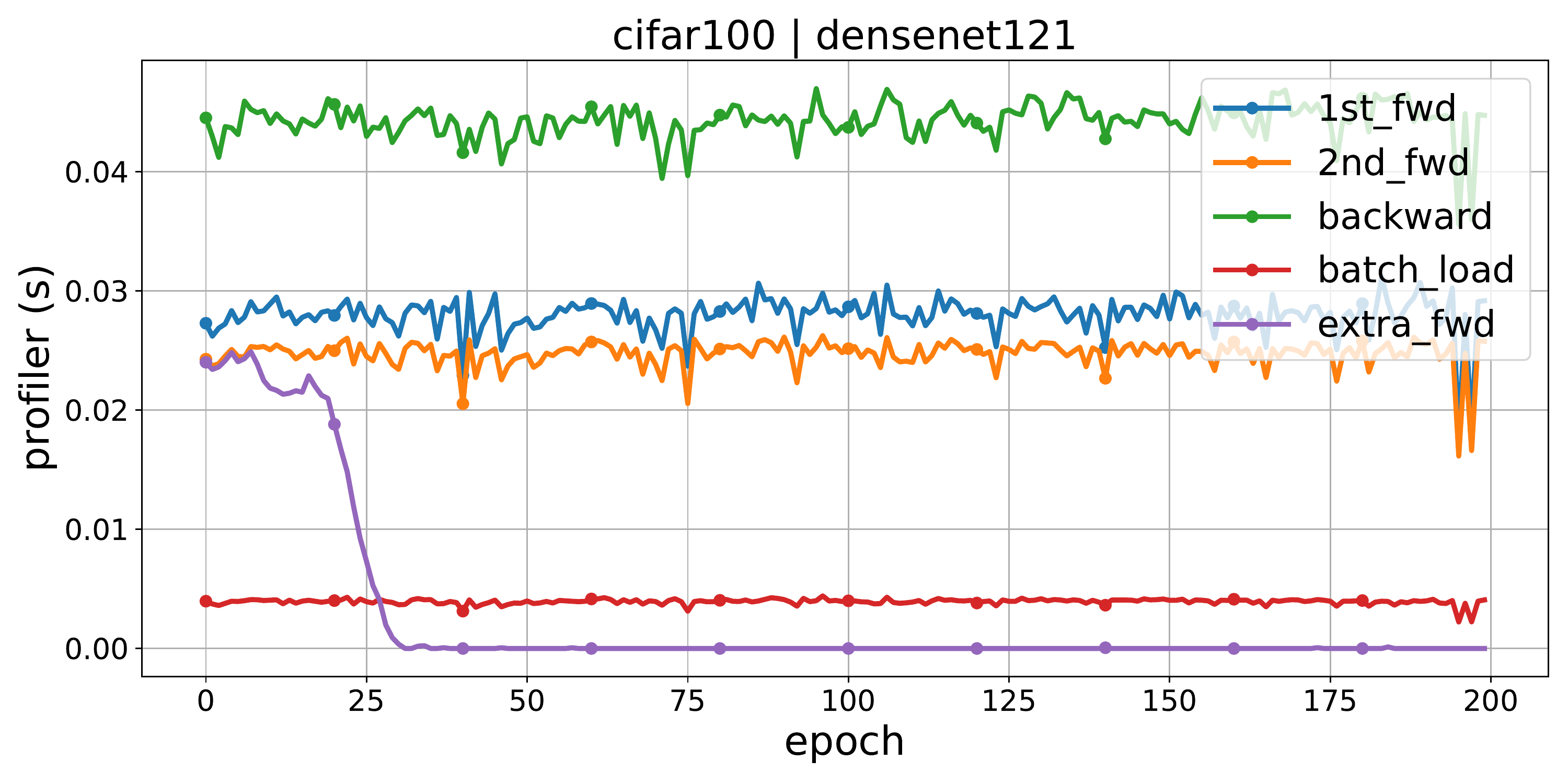}}
		\renewcommand{\model}{fashion_effb1}
		\renewcommand{\modelname}{fashion\_effb1}
		\subfloat{\includegraphics[width=\imgS\linewidth]{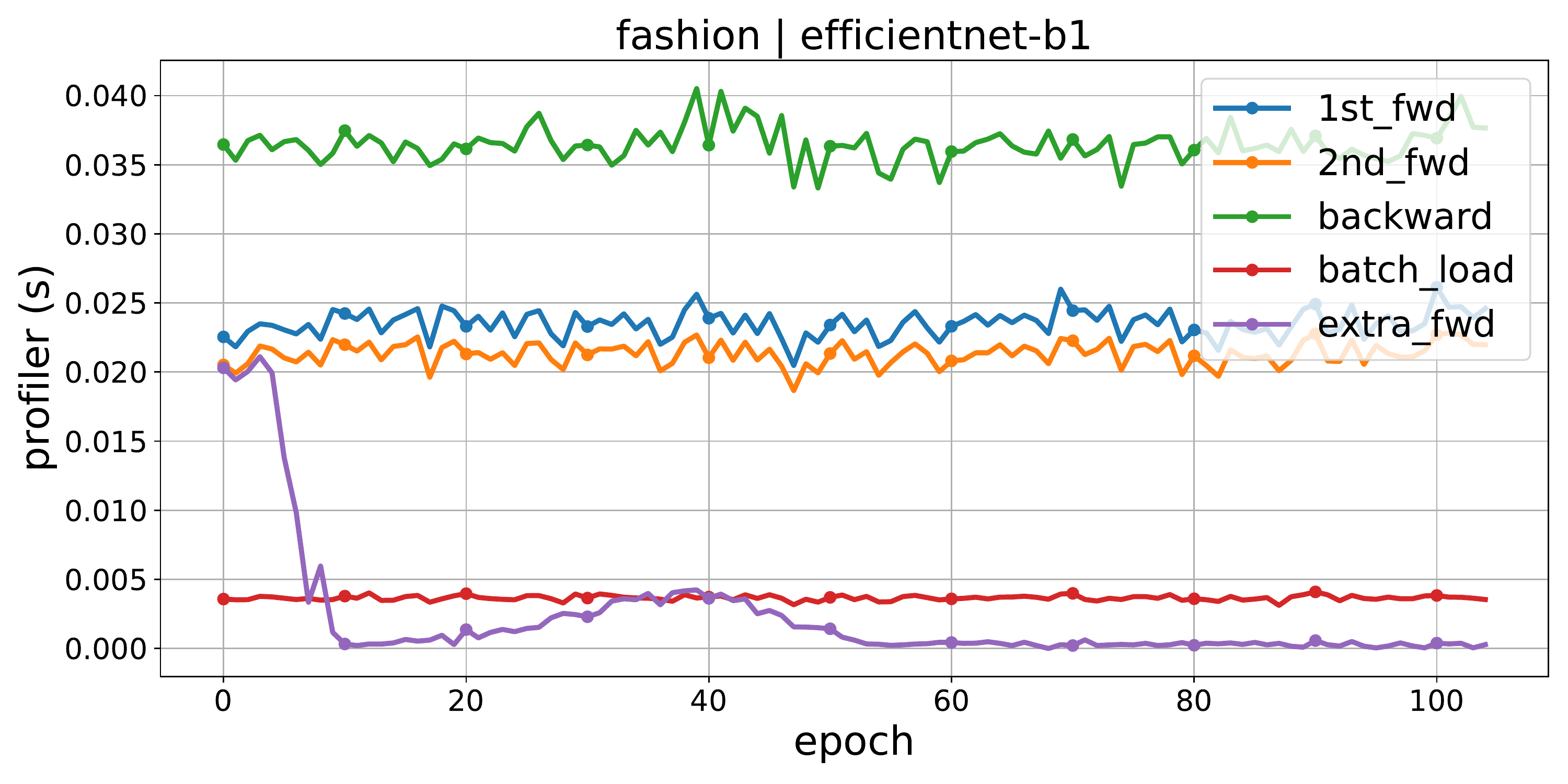}}
		\renewcommand{\model}{svhn_wrn}
		\renewcommand{\modelname}{svhn\_wrn}
		\subfloat{\includegraphics[width=\imgS\linewidth]{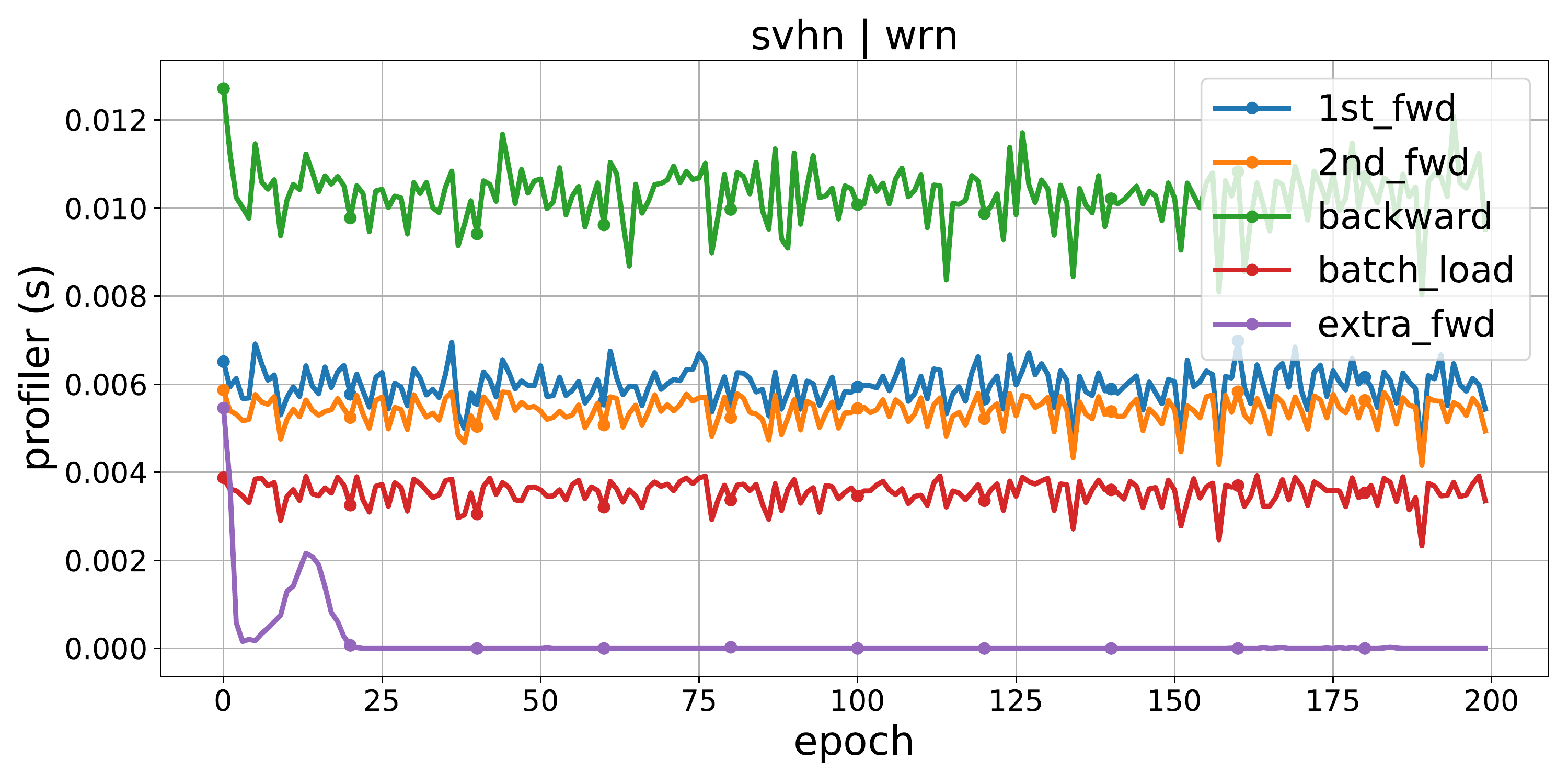}}
		\caption{Profiling of a single iteration of PoNoS. We report the time of the 1st and second forward passes, the time of all the extra forward passes beyond the second, the time of the backward pass and the time to load the mini-batch in the GPU.}\label{fig:profiling}
	\end{minipage}
\end{figure}

\end{document}